\documentclass[10pt]{amsbook}

\usepackage{amsmath}
\usepackage{amssymb, bm}
\usepackage{amscd}
\usepackage{amsthm}
\usepackage{latexsym} 
\usepackage[latin1]{inputenc}

\usepackage[all]{xy}

\newtheorem{theorem}{Theorem}[section]
\newtheorem{corollary}[theorem]{Corollary}
\newtheorem{lemma}[theorem]{Lemma}
\newtheorem{proposition}[theorem]{Proposition}
\newtheorem{scholium}[theorem]{Scholium}

\newcommand{\Proof}[1]{\noindent {\bf Proof{#1}.}}

\numberwithin{equation}{section}

\numberwithin{section}{chapter}

\newcommand{\Hom}{\operatorname{Hom}}

\def\A{\mathbb{A}}
\def\B{\mathbb{B}}
\def\N{\mathbb{N}}
\def\Z{\mathbb{Z}}
\def\PP{{\mathbb{P}}}
\def\Q{\mathbb{Q}}
\def\R{\mathbb{R}}

\def\Rpa{\mathbb{R}_+^\ast}
\def\C{\mathbb{C}}
\def\E{\mathbb{E}}
\def\F{\mathbb{F}}

\def\Fbp{{\mathbb{F}_{\bf p}}}
\def\bp{{\bf p}}
\def\G{\mathbb{G}}
\def\V{\mathbb{V}}

\def\Bn{\mathbb{B}^n}

\def\Bint{{\overset{\circ}{B}}}

\newcommand{\Qbar}{{\overline {\mathbb Q}}}

\newcommand{\Ga}{{{\mathbb G}_a}}

\def\Ahm{\widehat{\mathbb{A}}}

\newcommand{\cf}{{\em cf. }}
\newcommand{\eg}{{\em e.g.}}

\newcommand{\cA}{{\mathcal  A}}
\newcommand{\cB}{{\mathcal  B}}
\newcommand{\cC}{{\mathcal  C}}

\newcommand{\cE}{{\mathcal  E}}
\newcommand{\cEb}{{\overline{\mathcal  E}}}
\newcommand{\cF}{{\mathcal  F}}
\newcommand{\cFS}{{\mathcal{FS}}}

\newcommand{\cI}{{\mathcal  I}}

\newcommand{\cL}{{\mathcal L}}

\newcommand{\cM}{{\mathcal M}}

\newcommand{\cO}{{\mathcal O}}
\newcommand{\cP}{{\mathcal P}}
\newcommand{\cQ}{{\mathcal Q}}
\newcommand{\cS}{{\mathcal S}}
\newcommand{\cT}{{\mathcal T}}
\newcommand{\cU}{{\mathcal U}}
\newcommand{\cV}{{\mathcal  V}}
\newcommand{\cX}{{\mathcal X}}

\newcommand{\cCh}{{\widehat{\mathcal  C}}}

\newcommand{\cVh}{{\widehat{\mathcal  V}}}
\newcommand{\cVt}{{\widetilde{\mathcal  V}}}

\newcommand{\cEh}{{\widehat{{\mathcal  E}}}}
\newcommand{\cEbh}{{\widehat{\overline{\mathcal  E}}}}
\newcommand{\cEbt}{{\tilde{\overline{\mathcal  E}}}}

\newcommand{\cFh}{{\widehat{{\mathcal  F}}}}

\newcommand{\cLh}{{\widehat{{\mathcal  L}}}}
\newcommand{\cLb}{{\overline{\mathcal  L}}}
\newcommand{\cLbh}{{\widehat{\overline{\mathcal  L}}}}

\newcommand{\OK}{{{\mathcal O}_K}}

\newcommand{\Exph}{\operatorname{\widehat{{Exp}}}}

\newcommand{\Gext}{{\rm{Gext}}}

\newcommand{\Ext}{\operatorname{Ext}}
\newcommand{\Exthun}{\operatorname{\widehat{Ext}}^1}

\newcommand{\Exp}{\operatorname{Exp}}

\newcommand{\Spec}{\operatorname{Spec}}

\newcommand{\coker}{\operatorname{coker}}
\renewcommand{\div}{{\rm div\,}}
\newcommand{\mult}{{\rm mult}}
\newcommand{\pr}{\operatorname{pr}}
\newcommand{\Gr}{{\rm Gr}}

\newcommand{\hilb}{{\rm Hilb}}
\newcommand{\covol}{\operatorname{covol}}

\newcommand{\CP}{\mathbf{CP}}
\newcommand{\CTC}{\mathbf{CTC}}
\newcommand{\Hilbcont}{\mathbf{Hilb}^{\rm cont}}

\newcommand{\Spf}{\operatorname{Spf}}

\newcommand{\oli}{\overline}

\newcommand{\Cb}{{\overline C}}
\newcommand{\Db}{{\overline D}}
\newcommand{\Eb}{{\overline E}}
\newcommand{\Fb}{{\overline F}}
\newcommand{\Gb}{{\overline G}}
\newcommand{\Kb}{{\overline K}}
\newcommand{\Lb}{{\overline L}}

\newcommand{\Sb}{{\overline S}}

\newcommand{\Vb}{{\overline V}}

\newcommand{\cOb}{\overline{\mathcal{O}}}

\newcommand{\Et}{{\tilde E}}
\newcommand{\Ft}{{\tilde F}}

\newcommand{\bB}{{\mathbf B}}
\newcommand{\bC}{{\mathbf C}}
\newcommand{\bD}{{\mathbf D}}

\newcommand{\dega}{\,\widehat{\rm deg }\,}
\newcommand{\degan}{{\widehat{\rm deg }}_n\,}
\newcommand{\mua}{\widehat{\mu }}

\newcommand{\rk}{{\rm rk \,}}
\newcommand{\im}{{\rm im\,}}

\newcommand{\fm}{{\mathfrak m}}
\newcommand{\fp}{{\mathfrak p}}

\newcommand{\an}{{\rm an}}

\newcommand{\hD}{{\widehat{D}}}
\newcommand{\hE}{{\widehat{E}}}
\newcommand{\hF}{{\widehat{F}}}
\newcommand{\hG}{{\widehat{G}}}

\newcommand{\hR}{{\hat{R}}}

\newcommand{\hV}{{\widehat{V}}}

\newcommand{\Ah}{{\widehat{A}}}
\newcommand{\Bh}{{\widehat{B}}}
\newcommand{\Ch}{{\widehat{C}}}

\newcommand{\Eh}{{\widehat{E}}}
\newcommand{\Fh}{{\widehat{F}}}
\newcommand{\Gh}{{\widehat{G}}}
\newcommand{\Hh}{{\widehat{H}}}
\newcommand{\Kh}{{\widehat{K}}}
\newcommand{\Lh}{{\widehat{L}}}

\newcommand{\Vh}{{\widehat{V}}}

\newcommand{\Xh}{{\widehat{X}}}
\newcommand{\cOh}{{\widehat{\mathcal O}}}
\newcommand{\Abh}{{\widehat{\overline{A}}}}
\newcommand{\Bbh}{{\widehat{\overline{B}}}}
\newcommand{\Cbh}{{\widehat{\overline{C}}}}

\newcommand{\Ebh}{{\widehat{\overline{E}}}}
\newcommand{\Fbh}{{\widehat{\overline{F}}}}
\newcommand{\Gbh}{{\widehat{\overline{G}}}}
\newcommand{\Hbh}{{\widehat{\overline{H}}}}
\newcommand{\Kbh}{{\widehat{\overline{K}}}}

\newcommand{\Vbh}{{\hat{\overline{V}}}}

\newcommand{\hot}{{h^0_{\theta}}}
\newcommand{\lhot}{{\underline{h}^0_{\theta}}}
\newcommand{\uhot}{{\overline{h}^0_{\theta}}}

\newcommand{\hon}{{h^0_{\rm{Ar}}}} 
\newcommand{\honm}{{h^0_{\rm{Ar}-}}} 
\newcommand{\hont}{{\tilde{h}^0_{\rm{Ar}}}}
\newcommand{\hut}{{h^1_{\theta}}}
\newcommand{\lhut}{{\underline{h}^1_{\theta}}}
\newcommand{\uhut}{{\overline{h}^1_{\theta}}}

\newcommand{\phit}{{\tilde{\phi}}}

\newcommand{\psit}{{\tilde{\psi}}}

\newcommand{\hphi}{{\hat{\phi}}}

\newcommand{\hpsi}{{\hat{\psi}}}

\newcommand{\lth}{{\tau}}

    \renewcommand{\Re}{\,{\rm Re}\,}
    \renewcommand{\Im}{\,{\rm Im}\,}

\newcommand{\ra}{\rightarrow}
\newcommand{\lrasim}{\stackrel{\sim}{\longrightarrow}}
\newcommand{\lra}{\longrightarrow}
\newcommand{\hra}{\hookrightarrow}
\newcommand{\lmt}{\longmapsto}
\newcommand{\hlra}{{\lhook\joinrel\longrightarrow}}

\newcommand{\Sum}{{\mathbf{Sum}}}

\newcommand{\sKC}{{\sigma: K \hookrightarrow \mathbb{C}}}

\renewcommand{\phi}{\varphi}

\renewcommand{\epsilon}{\varepsilon}

\newcommand{\Vt}{{\tilde{V}}}
\newcommand{\Vo}{{\mathring{V}}}
\newcommand{\bV}{{\partial V}}

\newcommand{\cEt}{{\tilde{\mathcal E}}}

\newcommand{\proVectOK}{{{\rm pro\overline{Vect}}^{\rm cont}_\OK}}

\newcommand{\proVectZ}{{{\rm pro\overline{Vect}}^{\rm cont}_\Z}}

\begin{document}

\frontmatter

\title[Theta Invariants and Infinite Dimensional Hermitian Vector Bundles]{ Theta Invariants of Euclidean Lattices and \\ Infinite Dimensional Hermitian Vector Bundles \\ over Arithmetic Curves}
\author{Jean-Beno\^{\i}t Bost}
\address{J.-B. Bost, D{é}partement de Math{é}matiques, Universit{é}
Paris-Sud,
B{â}timent 425, 91405 Orsay cedex, France}
\email{jean-benoit.bost@math.u-psud.fr}

\date{\today}

\begin{abstract}
In this monograph, we lay some foundations of  a theory of  infinite dimensional Euclidean lattices --- and more generally, of infinite dimensional Hermitian vector bundles over some ``arithmetic curve" ${\rm Spec}\,\mathcal{O}_K$ attached to the ring of integers $\mathcal{O}_K$ of some number field $K$ --- with a view towards applications to transcendence theory and Diophantine geometry.

In the first chapters of this monograph, we study the properties of the invariant $h^0_\theta(\overline{E})$ attached to some Euclidean lattice $\overline{E}:= (E, \Vert.\Vert)$, defined  by the expression
$$h^0_\theta(\overline{E}) := \log \sum_{v \in E} e^{- \pi \Vert v \Vert^2},$$
and, more generally, attached to some  finite rank Hermitian vector bundle $\Eb$ over  an arithmetic curve. 

Then we construct categories of infinite dimensional Hermitian vector bundles and we show that it is possible to associate generalized $\theta$-invariants to these objects, so that they satisfy suitable subadditivity and summability properties.

In the last chapter, we present a first application of this formalism to Diophantine geometry: we show how it allows one to establish some algebraicity criterion \emph{\`a la} Chudnovsky concerning formal curves over number fields embedded in some projective space, by arguments that are direct counterparts of  classical algebraization proofs in complex analytic and formal geometry.
\end{abstract}

\maketitle

\setcounter{page}{4}

\tableofcontents

\medskip

\chapter*{Introduction}

\medskip

\section{Pro-Euclidean lattices and $\theta$-invariants}\label{proEuclidean-intro}
In this monograph, we lay some foundations for a theory of \emph{infinite dimensional Euclidean lattices} --- and more generally, of infinite dimensional Hermitian vector bundles over some ``arithmetic curve" $\Spec \OK$ attached to the ring of integers $\OK$ of some number field $K$ --- with a view towards applications to transcendence theory and Diophantine geometry.

\subsection{}\label{debut} Recall that an \emph{Euclidean lattice} is the data
$$\Eb:= (E, \Vert.\Vert)$$
of some free $\Z$-module of finite rank $E$ and of some Euclidean norm $\Vert. \Vert$ on the real vector space $E_\R:= E\otimes_\Z \R$.

The infinite dimensional generalizations of Euclidean lattices we will be mainly interested in are the \emph{pro-Euclidean lattices}. They naturally occur as projective limits of countable systems of Euclidean lattices of the form:
\begin{equation*}\Eb_{\bullet} : 
\Eb_0 \stackrel{q_0}{\longleftarrow}\Eb_1 \stackrel{q_1}{\longleftarrow}\dots \stackrel{q_{i-1}}{\longleftarrow}\Eb_i \stackrel{q_i}{\longleftarrow} \Eb_{i+1} \stackrel{q_{i+1}}{\longleftarrow} \dots \; .
\end{equation*}
Here, for every $i \in \N$, we denote by $\Eb_i$ some Euclidean lattice $(E_i, \Vert.\Vert_i)$ and by $q_i$ a surjective morphism of $\Z$-modules
$$q_i : E_{i+1} \lra E_i$$
such that the norm $\Vert.\Vert_i$ on $V_{i,\R}$ coincides with the the quotient norm deduced from the norm $\Vert.\Vert_{i+1}$ on $E_{i+1,\R}$ by means of the surjective $\R$-linear map
$$q_{i,\R}:= q_i \otimes Id_\R: E_{i+1, \R} \lra E_{i,\R}.$$

A \emph{pro-Euclidean lattice} may actually be defined ``directly", without explicit mention of projective systems of Euclidean lattices, as a triple
$$\Ebh := (\Eh, E_\R^\hilb, \Vert.\Vert)$$ consisting in the following data:
\begin{enumerate}
\item an abelian topological group $\Eh$, isomorphic to $\Z^n$ (for some $n \in \N$) equipped with the discrete topology, or to $\Z^\N$ equipped with the product topology of the discrete topology on each factor $\Z$;
\item a dense real vector subspace $E_\R^\hilb$ of the topological real vector space 
$\Eh_\R := \Eh \widehat{\otimes}_\Z \R$, defined as the completed tensor product\footnote{Any isomorphism of topological groups $\phi: \hE \lrasim \Z^N$ of $\Eh$ by $\R$, with $N \in \N$ or $N =\N$, induces an isomorphism $\phi_\R:= \phi \widehat{\otimes}_\Z  \R: \Eh_\R \lra \R^N$ of topological real vector spaces, where the topology on $\R^N$ is defined as the product of the usual topology on each factor $\R$.} of $\Eh$ by $\R$;
\item a norm $\Vert. \Vert$ on $E_\R^\hilb$ that makes $(E_\R^\hilb, \Vert.\Vert)$ a real separable Hilbert space; this Hilbert space topology on 
$E_\R^\hilb$ is moreover required to be finer than the topology induced by the topology of  $\Eh_\R$.
\end{enumerate}

To any projective system $\Eb_{\bullet}$ as above, one  attaches such data by defining $\Eh$ as the pro-discrete $\Z$-module
$$\Eh := \varprojlim_i E_i$$ 
and $(E_\R^\hilb, \Vert.\Vert)$ as the projective limit, in the category of real normed vector spaces, of the projective system:
$$
(E_{0,\R}, \Vert.\Vert_0) \stackrel{q_{0,\R}}{\longleftarrow}(E_{1,\R}, \Vert.\Vert_1) \stackrel{q_{1,\R}}{\longleftarrow}\dots \stackrel{q_{i-1,\R}}{\longleftarrow}(E_{i,\R}, \Vert.\Vert_i) \stackrel{q_{i,\R}}{\longleftarrow} (E_{i+1,\R},\Vert.\Vert_{i+1}) \stackrel{q_{i+1,\R}}{\longleftarrow} \dots \; .
$$
Indeed $\Eh$ may be identified with the closed submodule of $\prod_{i \in \N} E_i$ consisting of families $(v_i)_{i \in \N}$ such that
\begin{equation}\label{projvi}
\mbox{ for any $i \in \N,\; v_i = q_i(v_{i+1}).$}
\end{equation}
Similarly the topological real vector space $\Eh_\R$ may be identified with the closed subspace of $\prod_{i \in \N} E_{i,\R}$ consisting of families $(v_i)_{i \in \N}$ satisfying (\ref{projvi}), and the real Hilbert space $E_\R^\hilb$ with the subspace of $\prod_{i \in \N} E_{i,\R}$ consisting of families $(v_i)_{i \in \N}$  such that, besides  (\ref{projvi}), the following holds:
\begin{equation*}
\Vert (v_i)_{i \in \N} \Vert := \lim_{i \ra + \infty} \Vert v_i \Vert_i < + \infty.
\end{equation*}
(For any $(v_i)_{i \in \N}$ in $\prod_{i \in \N} E_{i,\R}$ satisfying (\ref{projvi}), the limit $\lim_{i \ra + \infty} \Vert v_i \Vert_i$ exists in $[0, +\infty]$ since $(\Vert v_i \Vert_i )_{i \in \N}$ is a non-decreasing sequence in $\R_+$.)

These descriptions of $\Eh,$ $\Eh_\R$ and $E_\R^\hilb$ make clear the inclusions
$$\Eh \,\hlra \Eh_\R \longleftarrow\joinrel\rhook E_\R^\hilb$$
and show that the triple $(\Eh, E_\R^\hilb, \Vert.\Vert)$ actually satisfies the conditions (1)--(3) above and defines a pro-Euclidean lattice that we shall denote $\varprojlim_i \Eb_i.$  Moreover, any pro-Euclidean lattice $\Ebh$ as defined above may be obtained by this construction from a suitable projective system $\Eh_\bullet$. 

\subsection{}\label{origins} It turns out that meaningful invariants may be attached to the so-defined pro-Euclidean lattices. In this monograph, we shall focus on their $\theta$-\emph{invariants}.

Recall that, in the classical analogy between number fields and function fields, an Euclidean lattice $\Eb:= (E, \Vert.\Vert)$ appears as the counterpart of a vector bundle on some smooth projective curve $C$ over some base field $k$. Then the arithmetic counterpart of the dimension
\begin{equation}\label{hogeo}
h^0(C,V) := \dim_k \Gamma (C,V)
\end{equation}
 of the space of sections of $V$ is the non-negative real number 
 \begin{equation}\label{hoar}
\hot(\Eb) := \log \sum_{v \in E} e^{- \pi \Vert v \Vert^2}.
\end{equation}

The correspondence between the invariants (\ref{hogeo}) and (\ref{hoar}), in the geometric and arithmetic situations, goes back at least\footnote{Hilbert's formulation of his ``twelfth problem", together with Hecke's work on quadratic reciprocity over general number fields, might have been an earlier hint toward this correspondence; see \cite{Hilbert1900} p. 278-279, and \cite{Hecke19}, \cite{Hecke23} Kapitel VIII. } to the work of F. K. Schmidt. 

Indeed, Schmidt's proof  in \cite{Schmidt31}) of the functional equation for the zeta function of the function field $k(C)$ of some smooth projective curve $C$ over a finite field $k$ crucially relies on the Riemann-Roch formula for line bundles over $C$. Compared with the earlier proof by Hecke (\cite{Hecke17}) of the functional equation for the Dedekind zeta function of an arbitrary number field $K$ --- where the Poisson formula for Euclidean lattices of rank $[K:\Q]$ plays a key role --- Schmidt's arguments make clear the correspondence between (\ref{hogeo}) and (\ref{hoar}), and also the fact that, using the  definition (\ref{hoar}), Poisson formula takes a  form similar to the Riemann-Roch formula, namely:
$$\hot(\Eb) - \hot(\Eb^\vee) = \dega \Eb.$$
Here $\Eb^\vee$ denotes the Euclidean lattice dual of $\Eb$ and $\dega \Eb$ the Arakelov degree of $\Eb,$ defined in terms of its of covolume
${\rm covol}(\Eb)$ as
$$\dega \Eb = -\log {\rm covol}(\Eb).$$

The analogy between the invariants (\ref{hogeo}) and (\ref{hoar}) respectively attached to a vector bundle over a projective curve and to an Euclidean lattice, and between the Riemann-Roch and Poisson formulas, belongs to the ``well-known facts" of algebraic number theory: it was familiar to mathematicians like Artin and Hasse shortly after 1930\footnote{See for instance the entry, dated February 1934, in Hasse's mathematical diary (\cite{Hasse2012}, Section 7.28) where Hasse discusses, following Artin, an approach of the functional equation of the Dedekind zeta function of a number field that emphasizes this analogy.},  and is alluded to in  classical references such as Tate's thesis (\cite{Tate67}, Section 4.2).  This analogy is the matter of several entries in Quillen's mathematical diary \cite{QuillenNotebooks} (see notably on 24/12/1971, 26/04/1973, and 01/04/1983), and is also alluded to in \cite{Morishita95}, Section 1.4. It has been reexamined and developed, in relation with Arakelov geometry, in the works of  Roessler (\cite{Roessler93}), van der Geer and Schoof (\cite{vanderGeerSchoof2000}) and Groenewegen (\cite{Groenewegen2001}).

At this point, it  may be worth to emphasize that many contributions to Arakelov geometry consider another invariant of an Euclidean lattice $\Eb:=(E, \Vert.\Vert)$ as the counterpart of the dimension (\ref{hogeo}) of the space of sections of vector bundles over curves --- namely the real number
$$\hon(\Eb) := \log \vert \{ v \in E \mid \Vert v \Vert \leq 1 \}\vert.$$

This definition already appears, in substance, in Weil's famous note \cite{Weil39} (which, somewhat surprisingly, does not allude to the analogy between (\ref{hogeo}) and (\ref{hoar})), and explicitly in the first expositions  of Arakelov geometry, notably in \cite{Szpiro85} and \cite{Manin85} (see also \cite{GilletMazurSoule1991}).

A central theme of this monograph is that the $\theta$-invariant $\hot(\Eb)$ attached to an Euclidean lattice $\Eb$ is a more flexible and better behaved variant of $\hon(\Eb)$.

\subsection{} In this monograph, to a pro-Euclidean lattice 
$$\Ebh:=(\Eh, E_\R^\hilb, \Vert.\Vert),$$
we attach
some infinite dimensional generalizations of the invariant $\hot(\Eb)$ previously defined for (finite dimensional) Euclidean lattices. 

Notably we consider the invariant in $[0, + \infty]$ :
\begin{equation}\label{lhotintro}\lhot(\Ebh) := \log \sum_{v \in \Eh \cap E_\R^\hilb} e^{- \pi \Vert v \Vert^2}
\end{equation}
and we investigate some classes of pro-Euclidean Euclidean lattices for which this invariant is finite and may be computed as the limit
$$ 
\lim_{i \ra + \infty} \hot(\Eb_i)
$$ 
of the $\theta$-invariants of the Euclidean lattices in a projective system $\Eb_{\bullet}$ that defines $\Ebh$ as in \ref{debut}. 

As a preliminary to our study of infinite dimensional Euclidean lattices, we also establish various properties of the invariant $\hot(\Eb)$ in the finite dimensional case. A key point  is the  subadditivity of $\hot$, already noticed by Quillen and Groenewegen. This subadditivity property asserts that, 
for any admissible\footnote{By definition, this means that $0 \lra E \stackrel{i}{\lra} F \stackrel{p}{\lra} G \lra 0$ is a short exact sequence of $\Z$-modules  and that $i_\R$ and the transpose of $p_\R$ are isometries with respect to the Euclidean norms on $E_\R,$ $F_\R$, and $G_\R$ defining the Euclidean lattices $\Eb$, $\Fb$ and $\Gb$.} short exact sequence of Euclidean lattices
$$0 \lra \Eb \stackrel{i}{\lra} \Fb \stackrel{p}{\lra} \Gb \lra 0,$$
the following inequality holds:
\begin{equation}\label{keysubadd}
\hot(\Fb) \leq \hot(\Eb) + \hot(\Gb).
\end{equation}

A similar subadditivity is clearly satisfied by the integer valued invariant $h^0(C,.)$ defined by (\ref{hogeo}) in the geometric case. It is  remarkable that it holds \emph{ne varietur} in the arithmetic case --- in the form  (\ref{keysubadd}), that involves  no error term depending on the ranks of $\Eb$, $\Fb$ and $\Gb$.    This ``unexpectedly good" behavior of the invariant $\hot$ is crucial to extend it to the infinite dimensional situation. 

\subsection{} The formalism of pro-Euclidean lattices --- and more generally of pro-Hermitian vector bundles over arithmetic curves --- and of their $\theta$-invariants developed in this monograph turns out to be a convenient tool for Diophantine geometry and transcendence theory. This formalism is indeed especially well adapted to investigate situations that combine formal geometry over the integers and complex analytic geometry.

In the last chapter of this monograph, we present a simple illustration of this general principle: we use this formalism to establish  an algebraicity criterion \emph{\`a la} Chudnovsky concerning formal curves over number fields embedded in some projective space, by arguments that are direct counterparts of  classical algebraization proofs in complex analytic and formal geometry.

\section{Contents} We now proceed to a synopsis of the main results of this monograph. More detailed presentations of these results are  given  in the first paragraphs of each chapter. 

\subsection{} Chapter  \ref{sec:hermitvect} collects  some basic facts concerning Hermitian vector bundles over an arithmetic curve $\Spec \OK$ (when $K= \Q$ and $\OK = \Z,$ they are simply the Euclidean lattices previously considered in this introduction). This section is mainly intended for later reference, 
and could be skipped by readers   familiar with Arakelov geometry.

The $\theta$-invariants of Hermitian vector bundles over arithmetic curves are studied in Chapters \ref{thetaone} and ~\ref{thetatwo}.

In Chapter \ref{thetaone}, we define the invariants $\hot(\Eb)$ and $\hut(\Eb)$ attached to some Hermitian vector bundle over a general arithmetic curve $\Spec \OK$ and we present some of their basic properties, notably the Poisson-Riemann-Roch formula and their monotonicity properties. We also establish upper bounds for the $\theta$-invariants of Hermitian line bundles over $\Spec \OK$, and we discuss the subadditivity property (\ref{keysubadd}) and its variants. This chapter pursue the previous work on $\theta$-invariants in  \cite{vanderGeerSchoof2000} and \cite{Groenewegen2001}, but our presentation is self-contained and does not require any familiarity with these articles.

In Chapter \ref{thetatwo}, we focus  on the situation where  $\OK = \Z$, and we establish diverse properties of the $\theta$-invariants 
 of some Euclidean lattice $\Eb$ that  compare them 
 to diverse invariants of $\Eb$ classically considered in geometry of numbers.
Notably we establish the following  estimates comparing the two ``arithmetic avatars" $\hot(\Eb)$ and $\hon(\Eb)$ of the ``geometric" invariant  $h^0(C,V)$ mentioned in Paragraph \ref{origins} above:  \begin{equation}\label{comparingavatars}
-\pi \leq \hot(\Eb) -\hon(\Eb) \leq 
 (\rk E/2) . \log \rk E - \log(1-1/2\pi),
\end{equation}
where $\rk E$  denotes the rank of $\Eb$.

Our proof of (\ref{comparingavatars}) uses  some estimates established by Banaszczyk in his work \cite{Banaszczyk93} devoted to the derivation of ``transference inequalities", relating invariants of  Euclidean lattices and of the dual lattices, with essentially optimal constants. Banaszczyk's techniques  rely on the properties of the measure
\begin{equation}\label{Banamu}
\gamma_\Eb := \sum_{v \in E} e^{- \pi \Vert v\Vert^2} \delta_v
\end{equation}
supported by the lattice $E$ inside $E_\R$ and of its Fourier transform, which is a real analytic $E^\vee$-periodic function on $E_\R^\vee$. Although these techniques have been developed without reference to the analogy between number fields and function fields or to the invariants $\hot(\Eb)$ or $\hut(\Eb):= \hot(\Eb^\vee)$, they constitute a remarkably useful tool to investigate the properties of the $\theta$-invariants. 

In a related vein to the comparison estimates (\ref{comparingavatars}), we introduce the following extension of $\hon(\Eb)$, defined for every $t \in \R^\ast_+$:
$$\hon(\Eb, t) :=   \log \vert \{ v \in E \mid \Vert v \Vert^2 \leq t \}\vert,$$
and we prove that it admits an asymptotic version 
$$\hont(\Eb, t) := \hont(\Eb, t) := \lim_{n \rightarrow + \infty} \frac{1}{n} \; \hon(\Eb^{\oplus n}, nt)$$
--- the limit exists in $\R_+$ --- that, remarkably enough, basically coincides with the Legendre transform of the fonction $\log \theta_\Eb$, where
$$\theta_{\Eb}(t):= \sum_{v \in E} e^{-\pi t \Vert v \Vert ^2}.$$
Namely, we establish the dual relations: 
\begin{equation}\label{Leg-hontIntro}
\mbox{for every $x \in \Rpa,$ } \hont(\Eb, x) = \inf_{\beta > 0} (\pi \beta x + \log \theta_\Eb(\beta))
\end{equation}
and
\begin{equation}\label{Leg-logthetaIntro}
\mbox{ for every $\beta \in \Rpa$, }\log\theta_\Eb(\beta) = \sup_{x > 0} (\hont(\Eb, x) -\pi \beta x).
\end{equation}
These relations appear as a special case of a version of Cram\'er's theory of large deviations valid over some measure space of infinite total mass. This theory is presented in Appendix \ref{Append:LD} and turns out to be closely related to the thermodynamic formalism.

\subsection{} The  remaining chapters are mainly devoted to the construction of categories of Hermitian vector bundles of (possibly) infinite rank   and to the definition and the investigation of their $\theta$-invariants. 

Chapter \ref{CpCtc} introduces some categories $CT_A$ and $\CTC_A$ of (topological) $A$-modules over some Dedekind ring $A$. When $A$ is the ring of integers $\OK$ of some number field $K$,  objects of the  categories $CT_A$ and $\CTC_A$  will form the  ``algebraic constituents" of the infinite dimensional Hermitian vector bundles over $\Spec \OK$ that are the subject of this monograph.

Diverse proofs of Chapter \ref{CpCtc} rely on some ``automatic continuity" results, discussed in Appendix \ref{prodiscretemod},  concerning $R$-linear maps $\phi: R^\N \lra R$, for suitable rings $R$, when $R^\N \simeq \varprojlim_n R^n$ is equipped with the pro-discrete topology.

The  infinite dimensional Hermitian vector bundles and their morphisms  are formally introduced  and studied in
Chapter \ref{infiniteHermitian}. In Chapter \ref{thetainfinite}, we define and establish the basic properties of their $\theta$-invariants.

In Chapter \ref{SectionSum}, we establish the main technical results of this monograph concerning  $\theta$-invariants of pro-Hermitian vector bundles. A special case of our results may be summarized as follows. 

If $\Eb:= (E, \Vert. \Vert)$ denotes some Euclidean lattice, then, for every real number $\lambda$, we may consider the Euclidean lattice 
$$\Eb \otimes \cOb(\lambda) := (E, e^{-\lambda} \Vert.\Vert)$$
deduced from $\Eb$ by scaling its metric by a factor $e^{-\lambda}$. 

Let us consider a projective system $\Eb_{\bullet}$ of Euclidean lattices and the attached pro-Euclidean lattice $\varprojlim_i \Eb_i$, as in \ref{debut} above. This pro-Euclidean lattice may be thought as being constructed from $\Eb_0$ as ``successive extensions" by the Euclidean lattices
$$\Sb_i:= (\ker q_i, \Vert.\Vert_{i\mid \ker q_{i,\R}})$$
which are defined as the kernels of the morphisms $q_i,$ $i \in \N$. We shall show that, \emph{if, for some $\epsilon \in \R^\ast_+,$ the summability condition
$$\sum_{i \in \N} \hot(\Sb_i \otimes \cOb(\epsilon)) < +\infty$$
is satisfied, the invariant $\lhot(\varprojlim_i \Eb_i)$} (defined by (\ref{lhotintro})) \emph{is finite and is given as the limit:}
\begin{equation}\label{limintro} 
\lhot(\varprojlim_i \Eb_i) = \lim_{i \ra \infty} \hot(\Eb_i).
\end{equation}

The proofs of this finiteness  and of the equality (\ref{limintro}) involve  measure theoretic arguments on the Polish space defined by $\varprojlim_i E_i$ equipped with the pro-discrete topology, concerning the convergence properties of the sequence of Banaszczyk's
measures 
$$\gamma_{\Eb_i} := \sum_{v \in E_i} e^{- \pi \Vert v\Vert_i^2} \delta_v.$$ Various facts concerning measure theory on Polish spaces constructed as projective limits of countable systems of countable discrete sets are presented in Appendix \ref{measurespro} in a form suited to the derivation of the results in Chapter  \ref{SectionSum}.

Our results  allow us to define natural classes of pro-Hermitian vector bundles over an arithmetic curve $\Spec \OK$  with finite and well-behaved $\theta$-invariants: the \emph{strongly summable} and  the $\emph{$\theta$-finite}$ pro-Hermitian vector bundles.

In Chapter \ref{subaddinfinite}, we investigate  short exact sequences of pro-Hermitian vector bundles. We extend the subadditivity properties  of $\hot$ to this infinite dimensional setting, and give some applications to strongly summable and $\theta$-finite pro-Hermitian vector bundles. We also show, equipped with the class of these short exact sequences, the  additive category $\proVectOK$ of pro-Hermitian vector bundles over $\Spec \OK$ becomes an exact category, in the sense of Quillen.

In Chapter \ref{geoman}, we discuss the counterparts of the main results of this monograph over function fields in one variable. In this simpler framework, most of the arguments previously used to investigate pro-Hermitian vector bundles over arithmetic curves --- notably the ones involving some measure theory --- become drastically simpler. Hopefully,   this should shed some light on the basic principles that underly the proofs in the previous chapters, and provide some additional motivation for our earlier constructions.

As already indicated, Chapter \ref{Epi} is devoted to an application to Diophantine geometry of the formalism of pro-Hermitian vector bundles over arithmetic curves and of their $\theta$-invariants developed so far. This final chapter is of a rather different nature than the previous ones. It is indeed longer, and has been conceived for a geometrically-minded reader with some interest in the diverse aspects of the analogy between number fields and function fields. As such, it should be of a broader appeal. 

Chapter \ref{Epi} is readable with only some superficial knowledge of Chapter \ref{thetaone} and Chapters \ref{CpCtc} to \ref{SectionSum}. Its introductory section (Section \ref{EpiIntro}) should be accessible at this stage, and we advise  the readers willing to get an idea of the possible application of the formalism developed in this monograph to read this section before proceeding further.

These chapters are completed by six appendices (Appendices \ref{Append:LD}  to \ref{ApJohn}), that can be read independently, and are devoted to diverse (more or less) classical results 
used in this monograph for which we could not locate in the literature references exactly suited  to our needs.\footnote{In these Appendices, only the extended version of  Cram\'er's theorem, in relation to the thermodynamic formalism, in Appendix \ref{Append:LD} and the vanishing criterion for sections of line bundles on compact K\"ahler manifolds in Proposition \ref{holbmultK} in Appendix \ref{ApUpBd} might claim to some originality.} 

\bigskip

The strict logical dependencies between the chapters and the appendices are indicated in the ``Leitfaden" below. 

Let us emphasize that the order of the chapters reflects some further conceptual dependencies. For instance, the investigation of the $\theta$-invariants in the finite rank situation in Chapter \ref{thetatwo} by means of Banaszczyk's techniques appears as an important motivation for considering Banaszczyk's measures when studying $\theta$-invariants in the  infinite dimensional setting of Chapter  \ref{SectionSum}. The study of infinite rank vector bundles over projective curves in Chapter \ref{geoman} is of interest for the perpective  it gives on the constructions in Chapters \ref{thetainfinite},  \ref{SectionSum}, and   \ref{subaddinfinite}.

\vspace{4cm}

\begin{center}
 {\sc Leitfaden}
\end{center}

\vspace{2cm}

$$
\xymatrix{
  &    &*++[o][F-]{I}  \ar[d] \ar[ddr] &   &    &    &    \\
   &    &   *++[o][F-]{II} \ar[dl] \ar[ddr] &   & *++[o][F-]{IV}\ar[dr] \ar[dl]   &    & *++[o][F-]{B}\ar@{.>}[ll]   \\
*++[o][F-]{A}\ar@{.>}[r]    & *++[o][F-]{III}   &    & *++[o][F-]{V}\ar[d]  &    &*++[o][F-]{IX}   &    \\
     &    &    & *++[o][F-]{VI}\ar[d]  &    &    &    \\
      &    &    & *++[o][F-]{VII} \ar[dr] \ar[dl] &    &    &    *++[o][F-]{C}\ar@{.>}[lll] \\
 *++[o][F-]{D}\ar@{.>}[rr]      &    & *++[o][F-]{VIII}   &   & *++[o][F-]{X}   &    &    *++[o][F-]{E,F}\ar@{.>}[ll]
       }
$$

\newpage
 
\section{Acknowledgements} It is a pleasure to thank Fran\c{c}ois Charles for numerous discussions about the results  in this monograph and their applications, and for his comments on preliminary versions of this work.

I am much indebted to Ga\"etan Chenevier for making me aware of the ``automatic continuity" results discussed in Appendix \ref{prodiscretemod} and I thank Antoine Chambert-Loir for the reference \cite{Specker50}. This allowed me to present the results in Section \ref{StrictCTC} in their natural generality\footnote{An earlier version, relying on the arguments discussed after the proof of Proposition \ref{strictCTCDed3}, had to make a countability assumption on $A$ instead of condition $\mathbf{Ded_3}$.}. I am grateful to Pascal Autissier for useful comments on a previous version of Chapters 2 and 3, and to Ren\'e Schoof for making me aware of the entry in Hasse's Tagebuch mentioned in \ref{origins}, \emph{supra}.

At an early stage of the writing of this monograph, I had the opportunity to present some applications of the formalism it develops to Antoine Chambert-Loir, François Charles and Gerard Freixas, and I would like to thank them  for their encouraging comments. 

 Several years ago, Alain Connes pointed out to me that the invariant $\hot(\Eb)$ of an Euclidean lattice $\Eb = (E, \Vert. \Vert)$ should be expressed in terms of the measure $(p(e))_{e\in E}$ on $E$ of a given energy  $\sum_{e \in E} p(e) \Vert e \Vert^2$ for which the information theoretic entropy $-\sum_{e \in E} p(e) \log p(e)$ is maximal. 
This suggestion, now formalized in Proposition  \ref{maxentropE} \emph{infra}, has been crucial in arousing my interest in the relations between $\theta$-invariants of Euclidean lattices and concepts from statistical thermodynamics, and I warmly thank Alain Connes for sharing his insight.

\section{Notation and conventions} We denote by $\N$ the set of non-negative integers and, for any $k \in \N$, we denote by $\N_{> k}$ (resp., by $\N_{\geq k}$) the set of non-negative integers larger than $k$ (resp., larger than or equal to $k$).

By \emph{countable}, we mean ``of cardinality at most the cardinality of $\N$."

If $M$ is a module over some ring $A$, and if $B$ is some commutative $A$-algebra, we denote by $M_B$ the ``base changed" module $M\otimes_A B$. Similarly, if $\phi: M\lra N$ is a morphism of $A$-modules, we let: $$\phi_B := \phi \otimes_A Id_B : M_B \lra N_B,$$
and if $\cX$ is some $A$-scheme, we let:
$$\cX_B := \cX \times_{\Spec A} \Spec B.$$

\mainmatter

\chapter{Hermitian vector bundles over arithmetic curves}\label{sec:hermitvect}

\medskip

In this chapter, we gather some basic results concerning Hermitian vector bundles over  arithmetic curves (that is, the  affine schemes defined by the rings of integers of number fields). These results are well known, with the  exception of the content of Section \ref{directimageextar}, and appear for instance in \cite{Szpiro85}, \cite{Neukirch92},
\cite{BostGilletSoule94}, \cite{Soule97},
\cite{Bost96Bourbaki}, and 
\cite{BK10}.  

This chapter is primarily intended as a reference chapter. It might be skipped by readers familiar with Arakelov geometry, who could refer to it only when needed.

\bigskip
We denote by $K$ a number field, by $\OK$ its ring of integers, and by 
$$\pi: \Spec \OK \lra \Spec \Z$$
the morphism of schemes from $\Spec \OK$ to $\Spec \Z$, defined by the inclusion morphism 
$\Z \hra \OK.$

\section{Definitions and basic operations}\label{defbasic}

\subsection{Hermitian vector bundles over arithmetic curves}\label{defultrabasic} A \emph{Hermitian vector bundle} over $\Spec \OK$ is a pair $$\Eb = (E, (\Vert.\Vert_\sigma)_\sKC)$$
consisting in  a finitely generated projective $\OK$-module $E$ and in  a family  $(\Vert.\Vert_\sigma)_\sKC$ of Hermitian norms $\Vert .\Vert_\sigma$ on the complex vector spaces 
$$E_\sigma:= E \otimes_{\OK, \sigma} \C$$
defined by means of  the field embeddings $\sKC$. The family $(\Vert.\Vert_\sigma)_\sKC$ is moreover required to be invariant under complex conjugation.\footnote{Namely, for every  embedding $\sKC,$ we may consider the complex conjugate embedding $\overline{\sigma}: K \hra \C$ and the $\C$-antilinear isomorphism $F_\infty:  E_\sigma= E \otimes_{\OK, \sigma}    \lrasim     E_{\overline{\sigma}}= E \otimes_{\OK, \overline{\sigma}}$ defined by $F_\infty(e \otimes \lambda)= e \otimes \overline{\lambda}$. The Hermitian norms $\Vert .\Vert_\sigma$ and $\Vert.\Vert_{\overline{\sigma}}$ have to satisfy:
$\Vert F_\infty(.)\Vert_{\overline{\sigma}} =
\Vert .\Vert_\sigma.$}

When $K=\Q,$ a Hermitian vector bundle $\Eb = (E, \Vert. \Vert)$ over $\Spec \OK = \Spec \Z$ may be identified with a \emph{Euclidean lattice}, defined by a free $\Z$-module of finite rank $E$ and a Euclidean norm $\Vert . \Vert$ on $E_\R := E \otimes_\Z \R$. (Indeed, for any such $E$, the data of some Hermitian norm on $E_\C:= E\otimes_\Z \R$ invariant under complex conjugation and of some Euclidean norm on $E_\R$ are equivalent.)

The \emph{rank} of some Hermitian vector bundle $\Eb$ as above is the rank of the $\OK$-module $E$, or equivalently the dimension of the complex vector spaces $E_\sigma.$ A \emph{Hermitian line bundle} is a Hermitian vector bundle of rank one.  

When confusion may arise, the family of Hermitian norms underlying some Hermitian vector bundle $\Eb$ over $\Spec \OK$ will be denoted by $(\Vert . \Vert_{\Eb, \sigma})_\sKC.$ 

An \emph{isometric isomorphism}, or simply an \emph{isomorphism}, between two Hermitian vector bundles $\Eb$ and $\Fb$ over $\Spec \OK$ is an isomorphism $\psi: E \lrasim F$ between the underlying $\OK$-modules which, after every base change $\sKC$, defines an isometry of complex normed vector spaces  between $(E_\sigma, \Vert . \Vert_{\Eb, \sigma})$ and  $(F_\sigma, \Vert . \Vert_{\Fb, \sigma})$.

\subsection{Pull back. Tensor operations} 

Let $L$ be a number field extension of $K$. The inclusion $\OK \hra \cO_L$ defines a morphism of arithmetic curves 
$$f: \Spec \cO_L \lra \Spec \OK.$$
 The \emph{pull-back} $f^\ast \Eb$ is defined as the Hermitian vector bundle over $\Spec \cO_L$
$$f^\ast \Eb := (f^\ast E, (\Vert. \Vert_{\tau})_{\tau: L \hra \C})$$
where  $f^\ast E := E \otimes_{\OK}\cO_L$ and where, for every field embedding $\tau: L \hra \C,$ of restriction $\tau_{\mid K}=: \sigma$, $\Vert. \Vert_\tau$ denotes the Hermitian norm $\Vert. \Vert_\sigma$ on 
$$(f^\ast E)_\tau := (E\otimes_\OK \cO_L)\otimes_{\OK, \sigma} \C \simeq E \otimes_{\OK, \sigma}\C =: E_\sigma.$$

To any Hermitian vector bundle $\Eb$ over $\Spec \OK$, we may associate its dual Hermitian vector bundle $\Eb^\vee$ and its exterior powers $\wedge^k \Eb,$ $k\in \N.$
They are defined by means of the (compatible) constructions of duals and exterior powers for projective $\OK$-modules and Hermitian complex vector spaces.

Similarly, for any two Hermitian vector bundles $\Eb$ and $\Fb$ over $\Spec \OK,$ we may construct their direct sum $\Eb \oplus \Fb$ and their tensor product $\Eb \otimes \Fb$ as  Hermitian vector bundles over $\Spec \OK.$

These tensor operations are compatible with the pull-back construction defined above. Namely, we have canonical identifications $(f^\ast \Eb)^\vee \simeq  f^\ast (\Eb^\vee)$,  $f^\ast(\wedge^k \Eb) \simeq  \wedge^k(f^\ast \Eb)$,... 

\subsection{The Hermitian line bundles $\cOb(\delta)$}\label{Odelta} For any $\delta \in \R,$ we may consider the Hermitian line bundle over $\Spec \Z$ defined  by
$$\cOb(\delta) := (\Z, \Vert .\Vert_{\cOb(\delta)}),
\mbox{ where }
\Vert 1 \Vert_{\cOb(\delta)}:= e^{-\delta}.$$
It satisfies: $$\dega \cOb(\delta) = \delta.$$
We may also consider its pull-back by the morphism $\pi: \Spec \OK \lra \Spec \Z$:
$$\cOb_{\Spec \OK}(\delta) := \pi^\ast \cOb(\delta).$$

 For any Hermitian vector bundle $\Eb$ over $\Spec \OK,$ the Hermitian vector bundle $\Eb \otimes  \cOb_{\Spec \OK}(\delta)$ may be identified with the Hermitian vector bundle which admits the same underlying $\OK$-module $E$ as $\Eb$, and whose Hermitian structure is defined by the Hermitian norms defining $\Eb$ scaled by the factor $e^{-\delta}.$ 
 
 For simplicity, we shall often write $\Eb \otimes  \cOb(\delta)$ instead of $\Eb \otimes  \cOb_{\Spec \OK}(\delta)$.
 
 The ``trivial" Hermitian line bundle --- namely, $\cOb_{\Spec \OK}(0)$ --- will also be denoted by $\cOb_{\Spec \OK}$ and, when no confusion may arise, simply by $\cOb.$

\section{Direct images. The canonical Hermitian line bundle $\overline{\omega}_{\OK/\Z}$ over $\Spec \OK$}\label{Dircan}

\subsection{} To any Hermitian vector bundle $\Eb = (E, (\Vert.\Vert_\sigma)_\sKC)$ is attached its \emph{direct image}\footnote{The construction of direct images of hermtian vector bundles introduced here makes actually sense in a considerably more general setting; see \cite{BK10}, Section 1.2.1. Notably, it may be extended to any morphism of arithmetic curves $f:\Spec \cO_L \lra \Spec \OK.$} $\pi_\ast \Eb$ over $\Spec \Z$.

Observe that we have an isomorphism of $\C$-algebras:
$$
\begin{array}{rcl}
 \OK \otimes_\Z \C & \lrasim   & \bigoplus_\sKC \C   \\
 \alpha \otimes \lambda & \longmapsto   & (\sigma(\alpha) \lambda)_\sKC.  
\end{array}
$$
Therefore, for any $\OK$-module $M$, if $\pi_\ast M$ denotes the underlying $\Z$-module, we have 
\begin{align*}
 (\pi_\ast M)_\C & := \pi_\ast M \otimes_\Z \C \\
 & \simeq \pi_\ast M \otimes_\OK (\OK \otimes_\Z \C) \\
 & \simeq  \bigoplus_\sKC (M \otimes_{\OK, \sigma} \C).
\end{align*}

Using this observation, the direct image $\pi_\ast \Eb$ may be defined as the Hermitian vector bundle of rank $[K:\Q]. \rk E$ over $\Spec \Z$:
$$\pi_\ast \Eb := (\pi_\ast E, \Vert. \Vert_{\pi_\ast \Eb})$$
where, for any $v = (v_\sigma)_\sKC$ in $(\pi_\ast E)_\C \simeq \bigoplus_\sKC E_\sigma,$
$$\Vert v \Vert_{\pi_\ast \Eb}^2 := \sum_\sKC \Vert v_\sigma \Vert_\sigma^2.$$

Clearly we have:  $$\rk \pi_\ast E = [K:\Q]. \rk E.$$

\subsection{}The compatibility of the direct image operation $\pi_\ast$ and of the duality of Hermitian vector bundles involves 
 the \emph{canonical Hermitian line bundle}
$\overline{\omega}_{\OK/\Z}:=(\omega_{\OK/\Z}, (\Vert . \Vert_{\sigma})_{\sigma: K \hookrightarrow \C})$ over $\Spec \cO_{K}$. It is defined as the  ``canonical module''
$$\omega_{\OK/\Z} := \Hom_{\Z}(\cO_{K}, \Z)$$
--- also known as  the ``inverse of the different'' --- equipped with the Hermitian norms defined by
$$ \Vert {\rm tr}_{K/\Q} \Vert_{\sigma} = 1$$
for any embedding ${\sigma: K \hookrightarrow \C},$ where  ${\rm tr}_{K/\Q}$ denotes the trace map from $K$ to $\Q$ (it is indeed a non-zero element in $\Hom_{\Z}(\cO_{K}, \Z)$).

Indeed, for any Hermitian vector bundle $\Eb$ over $\Spec \cO_{K},$ we have a canonical (!!) isometric isomorphism of Hermitian vector bundles over $\Spec \Z$,
\begin{equation}\label{reldualArakelov}
\pi_{\ast} (\Eb^\vee \otimes \overline{\omega}_{\OK/\Z}) \stackrel{\sim}{\longrightarrow} (\pi_{\ast}\Eb)^\vee.
\end{equation}
 (See for instance \cite{BK10}, Proposition 3.2.2. For any $\xi$ in $E^\vee := \Hom_{\cO_{K}}(E, \cO_{K})$ and any $\lambda$ in $\omega_{\OK/\Z}$, the ``relative duality isomorphism" (\ref{reldualArakelov}) maps $\xi \otimes \lambda$ to $\lambda \circ \xi.$)

\section{Arakelov degree and slopes} 

\subsection{Definition and basic properties} The \emph{Arakelov degree} of some Hermitian line bundle $\Lb := (L, (\Vert. \Vert_\sigma)_\sKC)$ over $\Spec \OK$ is defined by the quality, valid for any $s \in L \setminus \{0\}$:
\begin{align}\dega \Lb & := \log \vert L/\OK s \vert - \sum_\sKC \log \Vert s \Vert_\sigma\\
& = \sum_{\fp} v_\fp (s) \log N\fp - \sum_\sKC \log \Vert s \Vert_\sigma. \label{degaline}
\end{align}
In the last equality, $\fp$ runs over the closed points of $\Spec \OK$ --- that is, on the non-zero prime ideals of $\OK$ --- and $v_\fp(s)$ denotes the $\fp$-adic valuation of $s$, seen as a section of the invertible sheaf over $\Spec \OK$ associated to $L$. Moreover, $N\fp := \vert \OK /\fp \vert$ denotes the norm of $\fp.$

This definition is extended to Hermitian vector bundles of arbitrary rank over $\Spec \OK$ by means of the formula: 
$$\dega \Eb:= \dega \wedge^{\rk E} \Eb.$$

If $\Eb$ is an Hermitian vector bundle over $\Spec \Z$, or in other words, a Euclidean lattice, its Arakelov degree may expressed in term of its covolume\footnote{See for instance Section \ref{Poissonformula}, \emph{infra}.} ${\rm covol}(\Eb);$ namely:
$$\dega \Eb = -\log {\rm covol}(\Eb).$$

For any two Hermitian line bundles $\Lb_1$ and $\Lb_2$ over $\Spec \OK$, the expression (\ref{degaline}) for their Arakelov degree shows that
\begin{equation}\label{degalinetens}
\dega (\Lb_1 \otimes \Lb_2) = \dega \Lb_1 + \dega \Lb_2.
\end{equation}

For any two Hermitian vector bundles $\Eb_1$ and $\Eb_2$ over $\Spec \OK,$ there is a canonical\footnote{up-to a sign.} isomorphism of Hermitian line bundles over $\Spec \OK$:
$$\wedge^{\rk(E_1 \oplus E_2)} (\Eb_1 \oplus \Eb_2) \lrasim \wedge^{\rk E_1} \Eb_1 \otimes \wedge^{\rk E_2} \Eb_2,$$
and the additivity relation (\ref{degalinetens}) applied to $\Lb_i := \wedge^{\rk E_i} \Eb_i$ ($i=1,2$) takes the form:
\begin{equation}\label{degadd}
\dega (\Eb_1 \oplus \Eb_2) = \dega \Eb_1 + \dega \Eb_2.
\end{equation}

Similarly, by reducing to the case of Hermitian line bundles, one shows that, for any Hermitian vector bundle $\Eb$ over $\Spec \OK,$ we have:
\begin{equation*}
\dega \Eb^\vee = - \dega \Eb.
\end{equation*}

\subsection{Arakelov degree and direct image} The very definition of the discriminant $\Delta_K$ of a number field $K$ shows that
$$\covol (\pi_\ast \cOb_{\Spec \OK}) = \vert \Delta_K \vert^{1/2}.$$
In other words,
\begin{equation}\label{degapiO}
\dega \pi_\ast \cOb_{\Spec \OK} = -\frac{1}{2} \log \vert \Delta_K \vert.
\end{equation}

More generally, for any Hermitian vector bundle $\Eb$ over $\Spec \OK,$ we have:
\begin{equation}\label{degapiE}
\dega \pi_\ast \Eb = \dega \Eb -\frac{1}{2} \log \vert \Delta_K \vert. \rk E.
\end{equation}

Observe that the compatibility of the ``relative duality isomorphism" (\ref{reldualArakelov}) and of the ``relative Riemann-Roch formula" (\ref{degapiE}) applied to $\Eb = \cOb_{\Spec \OK}$ and $\Eb = \overline{\omega}_{\OK/\Z}$ shows that
\begin{equation*}
\dega \overline{\omega}_{\OK/\Z} = \log \vert \Delta_K \vert.
\end{equation*}

\subsection{Slopes}\label{Aslopes}

If $\Eb$ is some Hermitian vector bundle over $\Spec \OK$ of positive rank, 
its \emph{slope} $\mua (\Eb)$  is defined as the quotient
$$\mua (\Eb) :=\frac{\degan \Eb}{\rk E}.$$

Its \emph{maximal slope} $\mua_{\rm max}(\Eb)$
(resp., its \emph{minimal slope} $\mua_{\rm min}(\Eb)$) as the maximum
(resp. the minimum)
of the slopes $\mua (\Fb)$, where $\Fb$ is the Hermitian 
vector bundle defined  by  a $\OK$-submodule
(resp. a torsion-free quotient)
of positive rank 
$F$ of $E$, equipped with the restrictions to   
$F_\sigma\subset E_\sigma$ of the Hermitian metrics
$(\|.\|_\sigma)_{\sigma:K\hookrightarrow\C})$ (resp.
with the quotient metrics on the vector spaces $F_\sigma$
of these metrics). 

One easily checks that
\begin{equation}
    \mua_{\rm max}(\Eb)=-\mua_{\rm min}(\check{\Eb}).
    \label{eq:maxmin}
\end{equation}

It is convenient to define the maximum (resp. minimum) slope of Hermitian vector bundles of rank zero to be $-\infty$ (resp. $+\infty$).

\section{Morphisms and extensions of Hermitian vector bundles}

\subsection{The filtration $\Hom_{\OK}^{\leq \bullet}( \Eb, \Fb)$ on $\Hom_{\OK}(E,F)$}

For any two Hermitian vector bundles $\Eb$ and $\Fb$ 
on $\Spec \OK$ and for any $\lambda \in \R_+,$ we may introduce the  subset $\Hom_{\OK}^{\leq \lambda}( \Eb, \Fb)$ of the $\OK$-module $\Hom_{\OK}(E,F)$ consisting in the $\OK$-linear maps 
$$\psi: E \lra F$$
such that, for every $\sigma: K \hlra \C,$ the induced $\C$-linear map
$$ \psi_\sigma: E_{\sigma} \lra F_{\sigma}$$
has an operator  norm $\leq \lambda$ when $E_{\sigma}$ and $F_{\sigma}$ are equipped with the Hermitian norms $\Vert .\Vert_{ \Eb_{\sigma}}$ and $\Vert .\Vert_{ \Fb_{\sigma}}$.

Besides, for any $\lambda \in \R_+, $ transposition defines a bijection:
\begin{equation}\label{transp}
\begin{array}{rcl}
 \Hom_{\OK}^{\leq \lambda}( \Eb, \Fb) &  \lrasim & \Hom_{\OK}^{\leq \lambda}( \Fb^\vee, \Eb^\vee)  \\
  \psi& \longmapsto   & \psi^t .   
\end{array}
\end{equation}

Observe that, if $\Eb_1,$ $\Eb_2$, and $\Eb_3$ are Hermitian vector bundles over $\Spec \OK,$ for any $(\lambda_1, \lambda_2) \in \R^2_+,$ the composition of an element $\psi_1$ in $\Hom_{\OK}^{\leq \lambda_1}( \Eb_2, \Eb_1)$ and of an element 
$\psi_2$ in  $\Hom_{\OK}^{\leq \lambda_2}( \Eb_3, \Eb_2)$ defines an element $\psi_1 \circ \psi_2$ in 
$\Hom_{\OK}^{\leq \lambda_1 \lambda_2}( \Eb_3, \Eb_1)$.

In particular, one endows the class of Hermitian vector bundles over $\Spec \OK$ with a structure of category by defining the  set of morphisms from $\Eb$ to $\Fb$ as  $\Hom_{\OK}^{\leq 1}( \Eb, \Fb)$. The isomorphisms in this category coincide with the ones already introduced in \ref{defultrabasic}.

\subsection{Injective and surjective admissible morphisms} Observe that $\Hom_{\OK}^{\leq 1}( \Eb, \Fb)$ contains the \emph{injective admissible morphisms}, defined as the $\OK$-linear maps $\psi:  E \lra F$ such that $\psi$ is injective with torsion free cokernel and such that, for any $\sKC,$ the $\C$-linear map $\psi_\sigma : E_\sigma \lra F_\sigma$ is an isometry with respect to the Hermitian norms $\Vert .\Vert_{\Eb, \sigma}$ and $\Vert .\Vert_{\Fb, \sigma}$.

The set $\Hom_{\OK}^{\leq 1}( \Eb, \Fb)$ also contains the \emph{surjective admissible morphisms}, defined as the surjective $\OK$-linear maps $\psi:E \lra F$ such that, for any field embedding $\sigma: K \hra \C,$ the norm $\Vert. \Vert_{\Fb_\sigma}$ on $F_\sigma$ is the quotient norm deduced from the norm $\Vert. \Vert_{\Eb_\sigma}$ on $E_\sigma$ by means of the surjective $\C$-linear map $\psi_\sigma: E_\sigma \lra F_\sigma.$

The transposition map (\ref{transp}), with $\lambda = 1,$ exchanges injective and surjective admissible morphisms.

\subsection{Heights of morphisms}\label{htmor}

Recall that, to any non-zero $K$-linear map $\phi: E_K \lra F_K$ is attached its \emph{height with respect to $\Eb$ and $\Fb$}, defined as
\begin{equation}\label{heightmorphism}
{\rm ht} (\Eb, \Fb, \phi) := \sum_{\fp} \log \Vert \phi \Vert_{\fp} + \sum_{\sigma: K \hra \C} \log \Vert \phi \Vert_\sigma.
\end{equation}
(See \cite{Bost01}, Section 4.1.4, where ${\rm ht} (\Eb, \Fb, \phi)$ is noted $h(\Eb, \Fb, \phi)$.) In the right-hand side of (\ref{heightmorphism}), $\fp$ varies over the maximal ideals of $\OK$, and $\Vert \phi \Vert_{\fp}$ denotes the $\fp$-adic norm of $\phi \in \Hom_K(E_K, F_K) \simeq (E^\vee \otimes_{\OK} F)_K$, defined by the equivalence\footnote{We denote by $\cO_{K,\fp}$ the discrete valuation ring defined as the localization of $\OK$ at $\fp$.}:
$$ \Vert \phi \Vert_{\fp} \leq 1 \Longleftrightarrow  \phi \in (E^\vee \otimes_{\OK} F)_{\cO_{K, \fp}}$$
and the normalization condition:
$$\Vert \varpi \phi \Vert_{\fp} = (N\fp)^{-1} \Vert \phi \Vert_\fp,$$
where $\varpi$ denotes a uniformizing element of  $\cO_{K,\fp}$ and $N\fp := \vert \OK / \fp \vert$ the norm of $\fp$. Besides, $\Vert \phi \Vert_\sigma$ denotes the operator norm of $\phi_\sigma : E_\sigma \lra F_\sigma$ defined from the Hermitian norms $\Vert .\Vert_{ \Eb_{\sigma}}$ and $\Vert .\Vert_{ \Fb_{\sigma}}$.

It is convenient to define the height ${\rm ht} (\Eb, \Fb, 0)$ of the zero morphism as $-\infty$.

Clearly, for any $\phi$ in $\Hom_{\OK}^{\leq 1}( \Eb, \Fb)\setminus \{0\},$ all the norms $ \Vert \phi \Vert_{\fp}$ and $\Vert \phi \Vert_\sigma$ belong to $]0,1]$, and therefore
$h(\Eb, \Fb, \phi)$ belongs to $\R_-.$

Besides, when $\Eb$ and $\Fb$ are Hermitian line bundles, then for any non-zero element $\phi$ of the $K$-vector space $\Hom_K(E_K,F_K)$, the very definition (\ref{heightmorphism}) of its height shows that
$$\dega (\Eb^\vee \otimes \Fb)  = - {\rm ht}(\Eb, \Fb, \phi),$$
or equivalently:
\begin{equation}\label{heightdeg} {\rm ht} (\Eb, \Fb, \phi) = \dega \Eb - \dega \Fb.
\end{equation}

This equality may be extended to arbitrary Hermitian vector bundles, in the form of the following ``slope inequalities":

\begin{proposition}\label{slopeineq}
Let $\Eb$ and $\Fb$ be two Hermitian vector bundles, and let $\phi_K: E_K \lra F_K$ be some $K$-linear map. 

1) If $\phi$ is injective, then 
\begin{equation}\label{slope1}
\dega \Eb \leq \rk E \, [\mua_{\rm max} (\Fb) + {\rm ht} (\Eb, \Fb, \phi)].
\end{equation}

2) If $\phi$ is surjective, then
\begin{equation}\label{slope2}
\dega \Fb \geq \rk F \, [\mua_{\rm min} (\Fb) - {\rm ht} (\Eb, \Fb, \phi)].
\end{equation}
 \end{proposition}
 
 The inequality (\ref{slope1}) is a straightforward consequence of  (\ref{heightdeg}) applied to the exterior powers of $\phi$ (see \cite{Bost01}, Proposition 4.5). Then  (\ref{slope2}) follows by duality.
 
 Observe that, in the estimate (\ref{slope1}) (resp. (\ref{slope2})), when $E$ (resp. $F$) has rank zero, we follow the usual convention  $0. (-\infty) = 0$ (resp. $0. (+\infty) = 0$) to define its right-hand side.

\subsection{Admissible short exact sequences of Hermitian vector bundles}\label{AdmShort}
An \emph{admissible short exact sequence} (also called an \emph{admissible extension}) of Hermitian vector bundles over $\Spec \OK$ is a diagram
\begin{equation}\label{shorter0}
\overline{\mathcal E}: 0 \lra \Eb \stackrel{i}{\lra} \Fb \stackrel{p}{\lra} \Gb \lra 0,
\end{equation}
where $\Eb = (E, (\Vert. \Vert_{\Eb, \sigma})_\sKC),$ 
$\Fb = (F, (\Vert. \Vert_{\Fb, \sigma})_\sKC),$
and $\Gb = (G, (\Vert. \Vert_{\Gb, \sigma})_\sKC)$ are Hermitian vector bundles over $\Spec \OK$, and where
$$\mathcal E: 0 \lra E \stackrel{i}{\lra} F \stackrel{p}{\lra} G \lra 0$$
is a short exact sequence of $\OK$-modules such that, for any field embedding $\sKC,$ the short exact sequence of complex vector spaces 
$$0 \lra E_\sigma \stackrel{i_\sigma}{\lra} F_\sigma \stackrel{p_\sigma}{\lra} G_\sigma \lra 0$$
(deduced from $\mathcal{E}$ by the base change $\sigma$) is compatible with the Hermitian norms $\Vert. \Vert_{\Eb, \sigma}$, $\Vert. \Vert_{\Fb, \sigma}$ and $\Vert. \Vert_{\Gb, \sigma}$.\footnote{Namely, $i_\sigma$ is required to be an isometry for the norms $\Vert. \Vert_{\Eb, \sigma}$ and  $\Vert. \Vert_{\Fb, \sigma}$, and the norm $\Vert. \Vert_{\Gb, \sigma}$ is required to be the quotient norm deduced from $\Vert. \Vert_{\Fb, \sigma}$ by means of $p_\sigma.$ Equivalently, $i$ and $p$ are respectively injective and surjective admissible morphisms from $\Eb$ to $\Fb$ and from $\Fb$ to $\Gb$.}

The additivity of the Arakelov degree for direct sums (\ref{degadd}) extends to admissible short exact sequences. Namely, for any admissible short exact sequence  (\ref{shorter0}) of Hermitian vector bundles over $\Spec \OK,$ the following equality holds:
\begin{equation}\label{degaddadm}
\dega \Fb = \dega \Eb + \dega \Gb.
\end{equation}
This follows from the existence of a canonical\footnote{up-to a sign.} isomorphism of Hermitian line bundles
$$\wedge^{\rk F} \Fb \simeq \wedge^{\rk E} \Eb  \otimes \wedge^{\rk G} \Gb$$
attached to the short exact sequence (\ref{shorter0}).
 
An admissible short exact sequence of Hermitian vector bundles (\ref{shorter0}) is said to be \emph{split} when there exists an admissible injective morphism of Hermitian vector bundles $$s: \Gb \lra \Fb$$ such that $p \circ s = Id_G.$ The morphism $s$ is then uniquely determined. 
 
 Indeed, for every embedding $\sigma: K \hra \C,$ the $\C$-linear map $s_\sigma: G_\sigma \lra F_\sigma$ deduced from $s$ by the extension of scalars $\sigma: \OK \lra \C$ necessarily coincides with the \emph{orthogonal splitting} $s^{\perp}_\sigma$ of the short exact sequence of complex vector spaces
\begin{equation}\label{shortersigma}
{\mathcal E}_\sigma: 0 \lra E_\sigma \stackrel{i_\sigma}{\lra} F_\sigma \stackrel{p_\sigma}{\lra} G_\sigma \lra 0.
\end{equation}
attached to  the Hermitian metric $\Vert . \Vert_{\Fb,\sigma}$ on $F_\sigma$.
(Recall that $s^{\perp}_\sigma$ is defined as the (linear) section $s^{\perp}_\sigma$ of $p_\sigma$ which maps $G_\sigma$ isomorphically onto the orthogonal complement $i_\sigma(E_\sigma)^\perp$ of $i_\sigma(E_\sigma)$ in $F_\sigma$ equipped with the Hermitian metric $\Vert . \Vert_{\Fb,\sigma}$.) 


\subsection{Arithmetic extensions}

Let $E$ and $G$ be two finitely generated projective $\OK$-modules. In \cite{BK10} is introduced the \emph{arithmetic extension group} $\Exthun_{\OK}(G,E)$ of $G$ by $E$. In the setting of the present paper, where we work over arithmetic curves of the form $\Spec \OK,$ it may be defined as:
\begin{equation}\label{ExthunDef}
\Exthun_{\OK}(G,E) := \Hom_{\OK}(G,E) \otimes_\Z \R /\Z \lrasim \frac{\left(\bigoplus_{\sigma: K \hra \C} \Hom_\C(G_\sigma, E_\sigma)\right)^{F_\infty}}{\Hom_{\OK}(G,E)}.
\end{equation}
(See \cite{BK10}, Example 2.2.3. As usual, $F_{\infty}$ denotes the complex conjugation.) 

The group $\Exthun_{\OK}(G,E)$ classifies, up to isomorphism, the \emph{arithmetic extensions} of $G$ by $E$. Recall that such an arithmetic extension is a pair $(\cE, (s(\sigma))_\sKC)$ where 
$$\mathcal E: 0 \lra E \stackrel{i}{\lra} F \stackrel{p}{\lra} G \lra 0$$
is an extension (of $\OK$-modules) of $G$ by $E$, and where $(s(\sigma))_\sKC$ is a family, invariant under complex conjugation, of $\C$-linear splittings
$s(\sigma): G_\sigma \lra F_\sigma$
of splitings of the extension of complex vector spaces
$${\mathcal E}_\sigma: 0 \lra E_\sigma \stackrel{i_\sigma}{\lra} F_\sigma \stackrel{p_\sigma}{\lra} G_\sigma \lra 0.
$$

Indeed, observe that, for any arithmetic extension as above, one may choose some $\OK$-linear splitting $s^{\rm{int}} \in \Hom_\OK(G, F)$ of $\mathcal{E}$. Then, for any $\sKC,$
$$p_\sigma \circ (s(\sigma) -  s^{\rm{int}}_\sigma) =0,$$
and therefore there exists a uniquely determined $T_\sigma \in \Hom_\C(G_\sigma, E_\sigma)$ such that
$$s(\sigma)- s^{\rm{int}}_\sigma = i_\sigma \circ T_\sigma.$$ 

Let us choose  an integral splitting of (\ref{shorter0}), that is an $\OK$-linear section 
$s^{\rm{int}}: G \lra F$ of $p: F \lra G.$ 
$$T_\sigma = s_\sigma^\perp - s^{\rm{int}}_\sigma.$$
The family $(T_\sigma)_\sKC$ belongs to $\left(\bigoplus_{\sigma: K \hra \C} \Hom_\C(G_\sigma, E_\sigma)\right)^{F_\infty}$. Moroever its class  $[(T_\sigma)_\sKC]$ 
 modulo $\Hom_\OK (G,E)$ in the extension  group $\Exthun_{\OK}(G,E)$ defined by (\ref{ExthunDef})  does not depend on the choice of the integral splitting $s^{\rm{int}}$. By definition, it is the class $[(\cE, (s(\sigma))_\sKC)]$ of the arithmetic extension $(\cE, (s(\sigma))_\sKC)$.
 

\subsection{Admissible short exact sequences and arithmetic extensions}\label{subsub:admiari}
Any admissible short exact sequence of Hermitian vector bundles over $\Spec \OK$ 
$$\overline{\mathcal E}: 0 \lra \Eb \stackrel{i}{\lra} \Fb \stackrel{p}{\lra} \Gb \lra 0$$
determines an arithmetic extension $(\cE, (s^{\perp}_\sigma)_\sKC)$ of $G$ by $E$, defined by means of the orthogonal splittings with respect to the Hermitian norms $\Vert . \Vert_{\Fb, \sigma}$ of the extensions of complex vector spaces $\cE_\sigma$. 

According to the discussion in \ref{AdmShort}, its class $$
[\overline{\mathcal E}] :=[(\cE, (s^{\perp}_\sigma)_\sKC)]$$ vanishes in  $\Exthun_{\OK}(G,E)$ if and only if the admissible short exact sequence $\Exthun_{\OK}(G,E)$ is split.

Besides, for any two Hermitian vector bundle $\Eb$ and $\Gb$ over $\Spec \OK,$ any class in $\Exthun_{\OK}(G,E)$ may be realized by the previous construction, starting from some suitable admissible short exact sequence $\cE$ as above. 

Indeed, for any $T= (T(\sigma))_\sKC$ in 
$\left(\bigoplus_{\sigma: K \hra \C} \Hom_\C(G_\sigma, E_\sigma)\right)^{F_\infty}$, we may introduce the Hermitian vector bundle over $\Spec \OK$
$$\overline{E\oplus G}^T := (E\oplus G, (\Vert. \Vert_{T(\sigma)})_\sKC)$$
where, for any field embedding $\sKC$ and any $(e,g) \in E_\sigma \oplus G_\sigma,$
$$\Vert(e,g)\Vert_{T(\sigma)}^2 := \Vert e- T_\sigma(g) \Vert_{E_\sigma}^2 + \Vert g \Vert^2_{G_\sigma}.$$
One easily checks that the maps 
$$
\begin{array}{rrcl}
i: & E  & \lra    & E\oplus G   \\
& e  & \longmapsto  & (e,0)   
\end{array}
$$
and
$$\begin{array}{rrcl}
p: & E \oplus G  & \lra   & G   \\
 & (e,g) & \longmapsto  & g  
\end{array}
$$
fit into an admissible short exact sequence of Hermitian vector bundles over $\Spec \OK$
\begin{equation}\label{defcET} 
\overline{\cE} (T) : 0 \lra \Eb \stackrel{i}{\lra} \overline{E\oplus G}^T \stackrel{p}{\lra} \Gb \lra 0,
\end{equation}
and that the associated class in $\Exthun_{\OK}(G,E)$ is $[T]$.

\subsection{Direct images of arithmetic and admissible extensions}\label{directimageextar}

Consider two finitely generated projective $\OK$-modules $E$ and $G$ and the underlying $\Z$-modules $\pi_\ast E$ and $\pi_\ast G$.

There is an obvious inclusion map between modules of $\OK$-linear  and $\Z$-linear morphisms:
$$\Hom_{\OK}(G,E) \hlra \Hom_{\Z} (\pi_\ast G, \pi_\ast E).$$
We will sometimes denotes by $\pi_\ast \phi$ the image in $\Hom_{\Z} (\pi_\ast G, \pi_\ast E)$ of some element $\phi$ of the group $\Hom_{\OK}(G,E).$

In this subsection, we want to investigate the ``derived" analogue of this map $\pi_\ast$, defined between the relevant arithmetic extension groups. 

Consider an arithmetic extension $(\cE, s)$ of $E$ by $G$ over $\Spec \OK,$ defined by an extension 
$$\cE: 0 \lra E \stackrel{i}{\lra} F \stackrel{p}{\lra} G \lra 0,$$
of $\OK$-modules and a family $s= (s_\sigma)_{\sigma: K \hra \C}$, invariant under complex conjugation, of $\C$-linear splittings
$s_\sigma: G_\sigma \lra F_\sigma$ of the extensions of complex vector spaces
$$\cE_\sigma: 0 \lra E_\sigma \stackrel{i_\sigma}{\lra} F_\sigma \stackrel{p_\sigma}{\lra} G_\sigma \lra 0.$$

We may define its direct image by $\pi: \Spec \OK \lra \Spec \Z$ as the arithmetic extension $(\pi_\ast \cE, \pi_\ast s)$ of $\pi_\ast E$  by $\pi_\ast G$ over $\Spec \Z$ defined by the extension of $\Z$-modules
$$\pi_\ast\cE: 0 \lra \pi_\ast E \stackrel{i}{\lra} \pi_\ast  F \stackrel{p}{\lra} \pi_\ast  G \lra 0,$$ equipped with the splitting
$$
\begin{array}{rrcl}
\pi_{\ast}s: & G\otimes_\Z \C \simeq \bigoplus_{\sigma: K \hra \C} G_\sigma & \lra &  E \otimes_\Z \C \simeq \bigoplus_{\sigma: K \hra \C} E_\sigma \\
& (g_\sigma)_{\sigma: K \hra \C} & \longmapsto &  (s_\sigma(g_\sigma))_{\sigma: K \hra \C}.
\end{array}
$$

\begin{proposition}\label{piextun} 1) The above construction defines a morphism of $\Z$-modules:
$$
\begin{array}{rrcl}
\hat{\pi}^1_\ast : & \Exthun_{\OK}(G,E)  & \lra   & \Exthun_{\Z}(\pi_\ast G, \pi_\ast E) \\
&[(\cE, s)] & \longmapsto   & [(\pi_\ast \cE, \pi_\ast s)].  
\end{array}
$$

2) Via the identifications
$$\Exthun_{\OK}(G,E)  \lrasim \frac{\left(\bigoplus_{\sigma: K \hra \C} \Hom_\C(G_\sigma, E_\sigma)\right)^{F_\infty}}{\Hom_{\OK}(G,E)}$$
and
$$\Exthun_{\Z}\,(\pi_\ast G,\pi_\ast E) \lrasim \frac{ \Hom_\C(\pi_\ast G\otimes_\Z \C, \pi_\ast E\otimes_\Z \C )^{F_\infty}}{\Hom_{\Z}(\pi_\ast G, \pi_\ast E)},$$ 
the map $\hat{\pi}^1_\ast $ coincides with the map defined, for any 
$(T(\sigma))_{\sigma: K \hra \C}$ in $(\bigoplus_{\sigma: K \hra \C} \Hom_\C(G_\sigma, E_\sigma))^{F_\infty},$
by
$$ \hat{\pi}^1_\ast [(T(\sigma))_{\sigma: K \hra \C}] = [\oplus_{\sigma} T(\sigma)],$$
where $\oplus_{\sigma} T(\sigma)$ is the element of $\Hom_\C(\pi_\ast G\otimes_\Z \C, \pi_\ast E )^{F_\infty}$ defined by
$$
\begin{array}{rrcl}
\oplus_{\sigma} T(\sigma) : & \pi_\ast G\otimes_\Z \C \simeq \bigoplus_{\sigma: K \hra \C} G_\sigma & \lra & \pi_\ast E\otimes_\Z \C \simeq \bigoplus_{\sigma: K \hra \C} E_\sigma \\
&  (g_\sigma)_{\sigma: K \hra \C} & \longmapsto &  (T_\sigma(g_\sigma))_{\sigma: K \hra \C}.
\end{array}
$$

3) Any admissible short exact sequence of Hermitian vector bundles over $\Spec \OK$
$$\overline{\mathcal E}: 0 \lra \Eb \stackrel{i}{\lra} \Fb \stackrel{p}{\lra} \Gb \lra 0$$
defines, by direct image, an admissible short exact sequence of Hermitian vector bundles over $\Spec \Z$:
$$\pi_\ast\overline{\mathcal E}: 0 \lra \pi_\ast \Eb \stackrel{i}{\lra} \pi_\ast \Fb \stackrel{p}{\lra} \pi_\ast \Gb \lra 0.$$
Moreover this construction is compatible with the direct image map  $\hat{\pi}^1_\ast $ between arithmetic extension groups. Namely, with the above notation, the following equality holds in $\Exthun_{\Z}(\pi_\ast G, \pi_\ast E)$:
$$ \hat{\pi}^1_\ast [ \overline{\mathcal E}] = [ \pi_\ast\overline{\mathcal E}].$$

4) The map $\hat{\pi}^1_\ast$ is injective.
\end{proposition}

\proof Assertions 1), 2), and 3) are left as an easy exercise for the reader. 

To prove the injectivity of $\hat{\pi}^1_\ast$, we use its description in 2), and we are left to show that, with the notation of 2), if $\oplus_{\sigma} T(\sigma)$ belongs to $\Hom _\Z(\pi_\ast G, \pi_\ast E),$ then there exists $T$ in $\Hom_{\OK} (G,E)$ such that, for every $\sigma: K \hra \C,$ $T(\sigma) = T_\sigma.$

In other words, we have to show that, if some $U$ in $\Hom _\Z(\pi_\ast G, \pi_\ast E)$ is such that 
$$U_\C: \pi_\ast G \otimes_\Z \C \lra \pi_\ast E \otimes_\Z \C$$
is block-diagonal with respect to the decompositions  in direct sums 
$$\pi_\ast G\otimes_\Z \C \simeq \bigoplus_{\sigma: K \hra \C} G_\sigma 
\mbox{ and }  \pi_\ast E\otimes_\Z \C \simeq \bigoplus_{\sigma: K \hra \C} E_\sigma, $$
then $U: G \lra E$ is $\OK$-linear. This is indeed the case, since any such $U$ is $\OK$-linear if and only if $U_\C := U \otimes_\Z Id_\C$ is linear as a map of modules over the ring
$\OK \otimes_\Z \C \simeq \bigoplus_{\sigma: K \hra \C} \C$. \qed

\medskip

\chapter{$\theta$-Invariants of Hermitian vector bundles over arithmetic curves}\label{thetaone} 

\medskip

This chapter is devoted to the definitions and to some basic properties of the  $\theta$-invariants of (finite-rank) Hermitian vector bundles over arithmetic curves. Their extensions to infinite-dimen\-sional Hermitian vector bundles established in the latter chapters of this monograph, as well as their Diophantine applications discussed in its sequel, will depend on these properties.

As explained in the Introduction (see \ref{origins}, \emph{supra}), the definitions of the invariants $\hot$ and $\hut$ have their origins in the classical works of Hecke and F. K. Schmidt on the functional equations of number fields and functions fields over finite fields. Formal definitions and some of their basic properties already appear  in \cite{Roessler93}, Section 6.2, and, most importantly from the perspective of this monograph, in the work of van der Geer and Schoof \cite{vanderGeerSchoof2000} and of Groenewegen \cite{Groenewegen2001}. 

In this chapter, we give a self-contained and streamlined presentation of the theory of $\theta$-invariants, adapted to our later use of them. 
We notably emphasize the importance of their  subadditivity properties  
(see notably Lemma \ref{gaussiansumpreimage} and  Section \ref{remsubadd}),  which play a key role in the proof of the main results of this monograph in chapter \ref{SectionSum}.  

Among the  results in this part that, by their novelty, might be  of special interest to readers already familiar with the content of \cite{vanderGeerSchoof2000} and \cite{Groenewegen2001}, we  mention:
\begin{itemize}
\item Proposition \ref{ineqmorthetaequidim} and its corollaries, that provide useful comparison estimates for the $\theta$-invariants of Euclidean lattices with the same underlying $\Q$-vector spaces.
\item The improved bounds on the invariant $\hot(\Lb)$ of an Hermitian line bundle $\Lb$ over an arithmetic curve. These bounds turns out to be remarkably similar to the well-known bounds on the dimension $h^0(C,L)$ of the space of sections of a line bundle $L$ over a projective curve $C$ (see Propositions \ref{linepositive} and \ref{scholiumhot}).
\item The discussion of the subadditivity properties of $\hot$ and $\hut$ in Section \ref{subadd}, notably the results concerning the ``subadditivity measure" $h_\theta (\overline{\cE})$ attached to an admissible short exact sequence of Euclidean lattices $\overline{\cE}$ that  are established in Propositions \ref{Gext} and \ref{GextAverag}.
\end{itemize}

Let us indicate that the content of Section \ref{AmpleTheta} \emph{infra}, which depends only on the results in this chapter, will illustrate how the basic formalism of $\theta$-invariants naturally combines with the classical developments of  Arakelov geometry, concerning higher dimensional projective schemes over $\Spec \Z$. The interested reader may directly go to \ref{AmpleTheta} after this chapter.

\bigskip 

As in chapter \ref{sec:hermitvect}, we denote by $K$ some number field, by $\OK$ its ring of integers and by $\pi$ the morphism of schemes form $\Spec \OK$ to $\Spec \Z.$

\section{The Poisson formula}\label{Poissonformula}

We first review some basic results related to the Poisson formula. This gives us the opportunity to introduce various conventions and notation that will be used in the next chapters.

\subsection{The Poisson formula for Schwartz functions}\label{PoissonSchwartz}
Let $\Vb$ be a Hermitian vector bundle over $\Spec \Z$ (or, equivalently, a Euclidean lattice). We shall denote $\lambda_{\Vb}$ the Lebesgue measure on $V_\R$. It is the unique translation invariant Radon measure on $V_\R$ that satisfies the following normalization condition: for any orthonormal basis $(e_1,\ldots,e_N)$ of $(V_\R, \Vert.\Vert_{\Vb}),$
$$\lambda_{\Vb}\left(\sum_{i=1}^N [0,1[ e_i \right)= 1.$$
This  normalization condition may be equivalently expressed in term of a Gaussian integral:
\begin{equation}\label{gaussint}
\int_{V_\R} e^{-\pi \Vert x \Vert_{\Vb}^2 }\, d\lambda_{\Vb}(x) = 1.
\end{equation}

Then the covolume of $\Vb$ may be defined as
$${\rm covol} (\Vb):= \lambda_{\Vb}\left(\sum_{i=1}^N [0,1[ v_i \right)$$
for any $\Z$-basis $(v_1,\ldots, v_N)$ of $V.$

Any function $f$ in the Schwartz space $\mathcal{S}(V_\R)$ of $V_\R$ admits a Fourier transform $\hat{f}$ in the Schwartz space $\mathcal{S}(V^\vee_\R)$ attached to the dual vector space $V_\R^\vee$, defined by
$$\hat{f}(\xi) := \int_{V_\R} f(x) \, e^{-2\pi i \langle \xi, x \rangle}\, d\lambda_{\Vb}(x), \mbox{ for any $\xi \in V_\R^\vee.$}
$$

With this notation, the Poisson formula for $\Vb$ reads as follows:
\begin{equation}\label{PoissonVvf}
\sum_{v \in V} f(x-v) = (\covol \Vb)^{-1} \sum_{v^\vee \in V^\vee} \hat{f}(v^\vee) \, e^{2\pi i \langle v^\vee, x\rangle}, \mbox{ for any $x \in V_\R.$}
\end{equation}
It is nothing but the Fourier series expansion of the function  $\sum_{v \in V} f(.-v)$, which is $V$-periodic on $V_\R$.

It is convenient to have at one's disposal the following more ``symmetric" form of the Poisson formula:
 \begin{equation}\label{PoissonVvfbis}
\sum_{v \in V} f(x-v)\, e^{2\pi i \langle \xi, v \rangle}= (\covol \Vb)^{-1} \, e^{2\pi i \langle \xi, x \rangle} \sum_{v^\vee \in V^\vee} \hat{f}(\xi - v^\vee) \, e^{- 2\pi i \langle v^\vee, x\rangle}, 
\end{equation}
which is valid for any $(x,\xi) \in V_\R\times V^\vee_\R.$ (The identity (\ref{PoissonVvfbis}) indeed follows from (\ref{PoissonVvf}) applied to the function $e^{-2\pi i \langle \xi, .\rangle} f.$)

\subsection{The Poisson formula for Gaussian functions}

In this monograph, a key role will played by the Poisson formula applied to Gaussian functions. Namely, we apply the above formula to the Gaussian function $f$ defined by 
$$f(x):= e^{-\pi \Vert x \Vert_{\Vb}^2}.$$
Then, for any $\xi \in V_\R^\vee,$ we have:
$$\hat{f}(\xi) = e^{-\pi \Vert \xi \Vert_{\Vb^\vee}^2},$$
and the identity (\ref{PoissonVvf}) and (\ref{PoissonVvfbis}) take the form:
\begin{equation}\label{PoissonGauss}
\sum_{v \in V} e^{-\pi \Vert x-v \Vert_{\Vb}^2} = (\covol \Vb)^{-1} \sum_{v^\vee \in V^\vee} e^{-\pi \Vert v^\vee \Vert_{\Vb^\vee}^2 + 2\pi i \langle v^\vee, x\rangle}
\end{equation}
and 
\begin{equation}\label{PoissonGaussBis}
\sum_{v \in V} e^{-\pi \Vert x-v \Vert_{\Vb}^2 + 2\pi i \langle \xi, v \rangle} = (\covol \Vb)^{-1} e^{2\pi i \langle \xi, x \rangle} \sum_{v^\vee \in V^\vee} e^{-\pi \Vert \xi - v^\vee \Vert_{\Vb^\vee}^2 - 2 \pi i \langle v^\vee, x\rangle}
\end{equation}

In particular, when $x=0,$ the identity (\ref{PoissonGauss}) becomes:
\begin{equation}\label{petitpoisson}
\sum_{v \in V} e^{-\pi \Vert v \Vert_{\Vb}^2} = (\covol \Vb)^{-1}  \sum_{v^\vee \in V^\vee} e^{-\pi \Vert v^\vee \Vert_{\Vb^\vee}^2}.
\end{equation}

Observe also that (\ref{PoissonGauss})  implies that, for any $x \in V_\R,$
\begin{equation}\label{PoissonGaussineq}
\sum_{v \in V} e^{-\pi \Vert x-v \Vert_{\Vb}^2} \leq  
\sum_{v \in V} e^{-\pi \Vert v \Vert_{\Vb}^2},
\end{equation}
and that equality holds in (\ref{PoissonGaussineq}) if and only if $x \in V.$ It also implies that, for any $x \in V_\R,$
\begin{equation}\label{PoissonGaussineqmoins}
\sum_{v \in V} e^{-\pi \Vert x-v \Vert_{\Vb}^2} \geq  
2 (\covol \Vb)^{-1} -\sum_{v \in V} e^{-\pi \Vert v \Vert_{\Vb}^2}.
\end{equation}

\section{The $\theta$-invariants $\hot$ and $\hut$  and the  Poisson-Riemann-Roch formula}

\subsection{The $\theta$-invariants of an Euclidean lattice and the Poisson formula}

For any Hermitian vector bundle $\Eb := (E, \Vert.\Vert)$ over $\Spec \Z$  (that is, for any Euclidean lattice), 
we let :
\begin{equation}\label{defhot}
\hot(\Eb) := \log \sum_{v \in E} e^{- \pi \Vert v \Vert^2_{\Eb}} 
\end{equation}
and
\begin{equation}\label{defhut}
\hut(\Eb) := \hot(\Eb^\vee).
\end{equation}
The Poisson formula (\ref{petitpoisson}) for the Euclidean lattice $\Eb$ reads:
$$\sum_{w \in E^\vee} e^{- \pi \Vert w \Vert^2_{\Eb^\vee}} = {\rm covol}(\Eb).  \sum_{v \in E} e^{- \pi \Vert v \Vert^2_{\Eb}}.$$
It may be rewritten in terms of the $\theta$-invariants $\hot(\Eb)$ and $\hut(\Eb)$ and of the Arakelov degree $ \dega \Eb$, as the following relation: 
\begin{equation}\label{PoissonZ}
\hot(\Eb) - \hut(\Eb) = \dega \Eb.
\end{equation}
Observe the similarity of (\ref{defhut}) and (\ref{PoissonZ}) with the Serre duality and the Riemann-Roch formula for vector bundles over an elliptic curve.

\subsection{The $\theta$-invariants of Hermitian vector bundles over a general arithmetic
curve and the Poisson-Riemann-Roch formula}  We extend the above definitions of $\hot$ and $\hut$ to Hermitian vector bundles over an arbitrary ``arithmetic curve" $\Spec \OK$ as above by defining:
\begin{equation}\label{hithetagen}
h^i_{\theta}(\Eb) := h^i_\theta (\pi_{\ast} \Eb) \mbox{ for $i= 0, 1.$}
\end{equation}

In other words, we have:
$$\hot (\Eb) = \log \sum_{v \in E} e^{- \pi \Vert v \Vert^2_{\pi_\ast \Eb}} $$
and
$$\hut (\Eb) = \hot(\Eb) - \dega \pi_\ast \Eb = \log \left[ {\rm covol }(\pi_\ast \Eb) \sum_{v \in E} e^{- \pi \Vert v \Vert^2_{\pi_\ast \Eb}}\right].$$

Observe  that, as a consequence of the ``relative duality isomorphism" (\ref{reldualArakelov}), we have:
\begin{equation}\label{Serredual}\hut(\Eb) = \hot(\Eb^\vee \otimes \overline{\omega}_{\OK/\Z}).
\end{equation}
This ``Serre duality formula'' could have been used as an alternative definition of $\hut(\Eb).$

Observe finally that, as a consequence of the Poisson formula (\ref{PoissonZ}) for Euclidean lattices and of the expression (\ref{degapiE}) for the Arakelov degree of a direct image, we obtain the general version of the \emph{Poisson-Riemann-Roch formula}, where we denote by  $\Delta_{K}$  the discriminant of the number field $K$:
\begin{equation}\label{Poisson}
\hot(\Eb) - \hut(\Eb) = \dega \Eb - \frac{1}{2}  \log \vert \Delta_K \vert. \rk E.
\end{equation}

\section{Positivity and monotonicity}

\subsection{Some elementary estimates} The following observation is a straightforward consequence from the definitions (\ref{defhot}) and (\ref{defhut}) of the $\theta$-invariants, but plays an important conceptual role in this monograph and its applications:

\begin{proposition}\label{hot positive} For any Hermitian vector bundle $\Eb$ over $\Spec \OK$, the real numbers
$\hot(\Eb)$ and $\hut(\Eb)$ are \emph{non-negative}, and are positive if $\Eb$ has positive rank. \qed
\end{proposition}

Together with the ``Poisson-Riemann-Roch formula" (\ref{Poisson}), this non-negativity shows  that, for any Hermitian vector bundle, the following avatar of the ``Riemann inequality'' is satisfied:
\begin{equation}\label{ThetaRIK}
\hot(\Eb) \geq \dega \Eb - \frac{1}{2} \log \vert \Delta_{K} \vert   \cdot \rk E.
\end{equation}
In particular, for any Euclidean lattice $\Eb,$ the following lower bound is satisfied:
\begin{equation}\label{ThetaRI}
\hot (\Eb) \geq \dega \Eb.
\end{equation}

\begin{proposition}\label{ineqmortheta}
Let $\Eb$ and $\Fb$ be two Hermitian vector bundles over $\Spec \OK,$ and let $\phi : E \lra F$ be a map in $\Hom_{\OK}^{\leq 1}(\Eb, \Fb).$ Let $\phi_K : E_K \lra F_K$ denote the induced morphism of $K$-vector spaces.

1) If $\phi_K$ is injective (or equivalently, if $\phi$ is injective), then
\begin{equation}\label{monothot}
\hot(\Eb) \leq \hot(\Fb).
\end{equation}

2) If $\phi_K$ is surjective, then
\begin{equation}\label{monothut}
\hut(\Eb) \geq \hut(\Fb).
\end{equation}

Moreover, equality holds in either (\ref{monothot}) or (\ref{monothut}) if and only if $\phi$ is an isometric isomorphism from $\Eb$ onto $\Fb$.
\end{proposition}

\begin{proof} 
1) Let us assume that $\phi_K$ is injective. 

By successively using that the maps $\phi_\sigma$ have operator norm at most $1$ and the injectivity of $\phi,$ we obtain:
\begin{align}
 \sum_{v \in E} e^{- \pi \Vert v \Vert^2_{\pi_\ast \Eb}} & \leq \sum_{v \in E} e^{- \pi \Vert \phi(v)\Vert^2_{\pi\ast\Fb}}  \label{phidecreases} \\
 & \leq \sum_{w \in F} e^{-\pi \Vert w \Vert^2_{\pi_\ast \Fb}}. \label{phiinjective}
\end{align}
This establishes (\ref{monothot}).

Equality holds in (\ref{phidecreases}) if and only if, for every field embedding $\sKC$ and every $v\in E,$ $$\Vert \phi_\sigma(v)\Vert^2_{\Fb, \sigma}
= \Vert v \Vert^2_{\Eb, \sigma}.$$

Since the family of Hermitian norms $(\Vert. \Vert_{\Eb, \sigma})_\sKC$ and $(\Vert. \Vert_{\Fb, \sigma})_\sKC$ 	are invariant under complex conjugation, this holds precisely when the $\phi_\sigma$ are isometries.

Moreover equality holds in (\ref{phiinjective}) if and only if $\phi(E) = F$.

This shows that equality holds in (\ref{monothot}) if and only if $\phi$ is an isometric isomorphism.

2) When $\phi_K$ is surjective, we consider the morphism
$$\pi_\ast \phi \in \Hom_\Z^{\leq 1} (\pi_\ast \Eb, \pi_\ast\Fb)$$
and its transpose
$$^t (\pi_\ast \phi) \in \Hom_\Z^{\leq 1} ((\pi_\ast\Fb)^\vee,(\pi_\ast \Eb)^\vee).$$

It is injective and, according to 1), 
\begin{equation}\label{monohutbis}
\hot((\pi_\ast\Fb)^\vee) \leq \hot((\pi_\ast\Eb)^\vee).
\end{equation}
This establishes (\ref{monothut}). 

Moreover equality holds in (\ref{monothut}), or equivalently in (\ref{monohutbis}), if and only if $^t (\pi_\ast \phi)$ is an isometric isomorphism from $(\pi_\ast\Fb)^\vee$ onto $(\pi_\ast \Eb)^\vee$. This is easily seen to be equivalent to the fact that $\phi$ itself is an isometric isomorphism.
\end{proof}

\subsection{Comparing the $\theta$-invariants of generically isomorphic Hermitian vector bundles} 

Combined to the Poisson-Riemann-Roch formula and to the basic properties of the height of morphisms (\cf \ref{htmor}, \emph{supra}), the simple estimates established in Proposition \ref{ineqmortheta} may be extended to more general situations where one deals with two Hermitian vector bundles $\Eb$ and $\Fb$ over $\Spec \OK$ such that $E_K$ and $F_K$ are isomorphic 

\begin{proposition}\label{ineqmorthetaequidim}
Let $\Eb$ and $\Fb$ be two Hermitian vector bundles over $\Spec \OK$  of the same rank $n$ and let $\phi : E \lra F$ be a map in $\Hom_{\OK}^{\leq 1}(\Eb, \Fb)$.

If $\phi$ is injective, or equivalently if   $$\det \phi_K : \wedge^n E_K \lra \wedge^n F_K$$ is not zero, then the following inequalities hold:
\begin{equation}\label{equidimeutile}
\hot(\Eb) \leq \hot(\Fb) \leq \hot(\Eb) - {\rm ht} (\wedge^n \Eb, \wedge^n \Fb, \det \phi_K).
\end{equation}
and
$$ \hut(\Fb) \leq \hut(\Eb)\leq  \hut(\Fb)- {\rm ht} (\wedge^n \Eb, \wedge^n \Fb, \det \phi_K).$$ 
\end{proposition}

\begin{proof} The inequalities
\begin{equation}\label{AA}\hot(\Eb) \leq \hot(\Fb)
\end{equation}
and
\begin{equation}\label{BB}
\hut(\Fb) \leq \hut(\Eb)
\end{equation}
are special cases of Proposition \ref{ineqmortheta}.

Besides, according to (\ref{heightdeg}), 
$${\rm ht} (\wedge^n \Eb, \wedge^n \Fb, \det \phi_K) = \dega \wedge^n \Eb - \dega \wedge^n \Fb = \dega \Eb - \dega \Fb.$$
The inequality
$$ \hot(\Fb) \leq \hot(\Eb) - {\rm ht} (\wedge^n \Eb, \wedge^n \Fb, \det \phi_K)$$
may therefore be written
$$\hot(\Fb) - \dega \Fb \leq \hot(\Eb) - \dega \Eb.$$
Taking the Poisson-Riemann-Roch formula (\ref{Poisson}) for $\Eb$ and $\Fb$ into account, it reduces to (\ref{BB}).

Similarly,  the inequality 
$$\hut(\Eb)\leq  \hut(\Fb)- {\rm ht} (\wedge^n \Eb, \wedge^n \Fb, \det \phi_K)$$
follows from (\ref{AA}).
\end{proof}

\begin{corollary} For any Hermitian vector bundle $\Fb:=(F,(\Vert.\Vert_\sigma)_{\sigma:K\hra \C})$ over $\Spec \OK$ and every $\OK$-submodule $E$ of $F$ such that $E_K = F_K$ --- \emph{ or equivalently, such that the quotient $F/E$ is finite} --- the $\theta$-invariants of the Hermitian vector bundle $\Eb := (E, (\Vert.\Vert_\sigma)_{\sigma:K\hra \C})$ satisfy:
$$\hot(\Fb) - \log \vert F/E \vert \leq \hot(\Eb) \leq \hot( \Fb)\phantom{+ +\log \vert F/E\vert}$$
and 
$$\phantom{+  \log \vert F/E } \hut(\Fb) \leq \hut(\Eb) \leq \hut( \Fb) + \log \vert F/E\vert.$$
\end{corollary}
\begin{proof} Indeed, if $\phi: E \lra F$ denotes the inclusion morphism, we have:
$$ - {\rm ht} (\wedge^n \Eb, \wedge^n \Fb, \det \phi_K) = - \dega \Eb + \dega \Fb = \log \vert F/E \vert.$$
\end{proof}

Recall that, by definition, an \emph{Arakelov divisor} $$\Db:= (\sum_i n_i \fp_i, (\delta_\sigma)_{\sigma:K \hra \C})$$ over $\Spec \OK$ consists in a divisor $\sum_i n_i \fp_i$ in $\Spec \OK$ and in a family, invariant under complex conjugation, of $[K:\Q]$  real numbers $(\delta_\sigma)_{\sigma: K \hra \C}$. 
It is \emph{effective} if the divisor $D$ and the $\delta_\sigma$'s are non-negative.

To $\Db$ is attached a Hermitian line bundle
$$\cO(\Db) := (\OK(D), (\Vert . \Vert_\sigma)_{\sigma: K \hra \C})$$
where
$\OK(D) := \prod_i {\fp_i}^{-n_i}$
 and where the metric $\Vert. \Vert_\sigma$ on $\OK(D)_\sigma\simeq \C$ is defined by $\Vert 1 \Vert_\sigma := e^{- \delta_\sigma}.$
 The Arakelov degree of $\Db$ is defined as:
 $$\dega \Db := \dega  \cO(\Db) = \sum_i n_i \log N{\fp_i} + \sum_\sigma \delta_\sigma.$$

\begin{corollary}\label{cor:hotED}
 Let $\Eb$ be  a Hermitian vector bundle over $\Spec \OK$.
 
 1) For any effective Arakelov divisor $\Db$ over $\Spec \OK$, we have
 \begin{equation}\label{hotED}
 0 \leq \hot(\Eb) - \hot(\Eb \otimes \cO(-\Db)) \leq \rk E. \,\dega \Db.
 \end{equation}
 
 2) For any $\delta$ in $\R_+,$
 \begin{equation}\label{hotEdelta}
 0 \leq \hot(\Eb) - \hot(\Eb \otimes \cOb(-\delta)) \leq \rk E. [K:\Q]. \delta.
 \end{equation}
\end{corollary}

\begin{proof} 1) Let $$\Eb':= \Eb \otimes \cO(-\Db).$$
As $\Db$ is effective, the module 
$$E' = E \otimes \OK(-D) = \prod_i \fp_i^{n_i} E$$
is a $\OK$-submodule of $E$, of the same rank $\rk E$ as $E$, and the inclusion morphism $\phi: E' \hra E$ has Hermitian operator norms
$\Vert \phi_\sigma \Vert  = e^{-\delta_\sigma} \leq 1.$ 

Moreover, the norms of $\det \phi_K$ are easily computed:
\begin{align*}
 \Vert \det \phi_K \Vert_\fp & = 1 \mbox{ if $\fp \notin \{\fp_i\}_i$} \\
 &= (N\fp_i)^{-\rk E. n_i} \mbox{ if $\fp=\fp_i$.}
\end{align*}
and, for every embedding $\sKC$,
$$\Vert \det \phi_K \Vert_\sigma = e^{- \rk E. \delta_\sigma}.$$
 Therefore
 $$h(\wedge^{\rk E} \Eb', \wedge^{\rk E} \Eb, \det \phi_K) = -\rk E \; (\sum_i n_i \log N\fp_i + \sum_{\sKC} \delta_\sigma) = -\rk E . \dega \overline{D}.$$
 and (\ref{hotED}) follows from (\ref{equidimeutile}) applied to the morphism $\phi$ in $\Hom_\OK^{\leq 1}(\Eb',\Eb)$.

2) The Hermitian line bundle $\cOb_{\Spec \OK}(\delta)$ is of the form $\cO(\overline{D})$ for some effective Arakelov divisor $\overline{D}$ of Arakelov degree $\dega \overline{D} = [K:\Q].\delta$ (namely $\overline{D}:= (0, (\delta_\sigma)_{\sKC}$ where $\delta_\sigma := \delta$ for every $\sKC$). Therefore (\ref{hotEdelta}) follows from (\ref{hotED}).
\end{proof}

\section{The functions $\tau$ and $\eta$}\label{taueta}

To express the $\theta$-invariants of  Euclidean lattices of rank one, it will be convenient to introduce the two functions
$$\tau: \R^\ast_+ \lra \R^\ast_+ \mbox{ and } \eta: \R \lra \R^\ast_+$$
defined as follows.

For any $x\in \R_+^\ast,$ we let:
$$\lth(x):= \log \sum_{n \in \Z} e^{-\pi x n^2}.$$ 
This definition may also be written:
\begin{equation*}
\tau(x) = \hot(\cOb(-\frac{1}{2} \log x))
\end{equation*}
and that the Poisson formula (\ref{petitpoisson}) (or equivalently the Poisson-Riemann-formula (\ref{PoissonZ})) applied to $\cOb(-\frac{1}{2} \log x)$ becomes:
\begin{equation}\label{Poissontau}
\tau(x) = \tau(x^{-1}) -\frac{1}{2} \log x.
\end{equation}

Observe also that:
\begin{equation}\label{lthinfinity}
\lth(x) = 2 e^{- \pi x} + O(e^{- 2 \pi x}) \mbox{ when $x\longrightarrow + \infty.$}
\end{equation}
Together with the Poisson formula (\ref{Poissontau}), this implies:
\begin{equation}\label{lthzero}
\lth(x) = - \frac{1}{2} \log x + 2 e^{- \pi / x} + O(e^{- 2 \pi / x}) \mbox{ when $x\longrightarrow 0_+.$}
\end{equation}

For any $t\in \R,$ we also let 
\begin{equation*}
\eta(t) := \tau(e^{2\vert t \vert}).
\end{equation*}
It is a continuous non-increasing function of $\vert t\vert$. In particular, its maximal value is
$$\max_{t \in \R} \eta (t) =\eta(0) = \tau (1).$$

We shall denote this number by $\eta$. In other words,
$$\eta = \hot(\cOb) =\log \omega,$$
where $$\omega := \sum_{n \in \Z} e^{-\pi n^2}.$$
The number $\eta$ is denoted by $\eta(\Q)$ in \cite{vanderGeerSchoof2000}. As explained in \emph{loc. cit.}, proof of Proposition 4, the positive real number $\omega$ may be expressed as
$$\omega = \frac{\pi^{1/4}}{\Gamma (3/4)}.$$
Numerically, one finds:
$$\omega = 1.0864348...  \,\mbox{ and } \,\eta = 0.0829015... \, .$$

Moreover, according to (\ref{Poissontau}), we have:
\begin{equation}
\eta(t) := \tau(e^{-2t}) - t^+,
\end{equation}
or equivalently:
$$\hot(\cOb(t)) = t^+ + \eta(t).$$
Actually, when $\vert t \vert$ goes to $+ \infty,$ $\eta(t)$ converges very fast to zero. Indeed, according to (\ref{lthinfinity},) we have:
\begin{equation}
\eta(t)= 2 e^{-\pi e^{2\vert t \vert}} + O(e^{-2\pi e^{2\vert t \vert}})   \mbox{ when $\vert t \vert \lra + \infty.$}
\end{equation}

Observe that this implies the existence of some constant $c \in \R^\ast_+$ such that, for any $t \in \R,$
\begin{equation}
\eta(t) \leq c e^{-\pi e^{2\vert t \vert}},
\end{equation}
or equivalently, such that for every $x \in [1, + \infty[$
\begin{equation}
\tau(x) \leq c e^{-\pi x}.
\end{equation}
An elementary computation shows that this holds as soon as 
$$c \geq 2 (1 + \frac{e^{-3\pi}}{1 -e^{-5\pi}}) = 2.0002...$$
For instance, we may choose $c=3.$

\section{Additivity. The $\theta$-invariants of direct sums of Hermitian line bundles over $\Spec \Z$}\label{directsumZ1} As a direct consequence of their definition, the invariants $\hot$ and $\hut$ also satisfy the following additivity property:

\begin{proposition}\label{prop:htadd} For any two Hermitian vector bundles $\Eb$ and $\Fb$ over $\Spec \OK$,
\begin{equation}\label{htadd}
h^i_\theta (\Eb \oplus \Fb) = h^i_\theta(\Eb) + h^i_\theta(\Fb),   \,\,\, i=0,1.
\end{equation}\qed

\end{proposition}

If $n$ denotes a positive integer and $\bm{\lambda} := (\lambda_1, \ldots,\lambda_n)$ an element of $\R^{\ast n}_+,$ we may consider the Hermitian vector bundle $\Vb_{\bm{\lambda}}$ over $\Spec \Z$ attached to the lattice $\Z^n$ in $\R^n$ equipped with the Euclidean norm $\Vert . \Vert_{\bm{\lambda}}$ defined by
$$ \Vert (x_1, \ldots, x_n)\Vert^2_{\bm{\lambda}} := \sum_{i=1}^n \lambda_i \, x_i^2.$$
Clearly: 
$$\Vb_{\bm{\lambda}}\simeq \bigoplus_{1\leq i \leq n} \cOb(-\frac{1}{2} \log \lambda_i).$$

The Arakelov degree and the $\theta$-invariants of $V_{\bm{\lambda}}$ are easily computed, by using their additivity (\ref{degaddadm}) and (\ref{htadd}) and the Poisson-Riemann-Roch formula (\ref{PoissonZ}):

\begin{proposition}\label{sumline} For any positive integer $n$ and any $\bm{\lambda} \in \R^{\ast n}_+,$ we have:
$$\dega \Vb_{\bm{\lambda}}  = - \frac{1}{2} \sum_{i=1}^n \log \lambda_i, \,\,\, \hot(\Vb_{\bm{\lambda}}) = \sum_{i=1}^n \tau(\lambda_i),\,\, 
\mbox{ and }\,\, \hut(\Vb_{\bm{\lambda}}) = \sum_{i=1}^n (\tau(\lambda_i) + \frac{1}{2} \log \lambda_i).$$ \qed
\end{proposition}

These formula may be rewritten as follows, in terms of the function $\eta$: 

\begin{proposition} For any positive integer $n$ and any Hermitian line bundles $\Lb_1,\ldots,\Lb_n$ over $\Spec \Z,$ we have: 
$$\dega\bigoplus_{i=1}^n \Lb_i = \sum_{i=1}^n \dega \Lb_i,$$
\begin{equation}\label{hotdegplus}\hot(\bigoplus_{i=1}^n \Lb_i) = \sum_{i=1}^n( \dega^+ \Lb_i + \eta( \dega \Lb_i))
\end{equation}
and \begin{equation}\label{hutdegminus}\hut(\bigoplus_{i=1}^n \Lb_i) = \sum_{i=1}^n ( \dega^- \Lb_i + \eta( \dega \Lb_i)). 
\end{equation} \qed
\end{proposition} 

We have denoted by $\dega^+ \Lb_i$ and $\dega^- \Lb_i$ the positive and negative parts of the real  number $\dega \Lb_i.$
  
  Observe the similarity of (\ref{hotdegplus}) and (\ref{hutdegminus}) with the expressions for the dimensions of the coherent cohomology groups of a direct sum of line bundles over an elliptic curve.

\section{The theta function $\theta_{\Eb}$ and the first minimum $\lambda_1(\Eb)$}
In this section, we consider Hermitian vector bundles over $\Spec \Z$ --- that is Euclidean lattices --- and we discuss some relations between their $\theta$-invariants and the   invariants classically associated to them in geometry of numbers, such as their first minimum or their number of lattice points in the unit ball.

Most of the contents of this section is due to   Groenewegen (\cite{Groenewegen2001}). We shall pursue our study of the relations between  $\theta$-invariants and  more classical invariants attached to Euclidean lattices in Section \ref{Banasc2} below, where we shall reformulate Banaszczyk's results (\cite{Banaszczyk93}). Notably, in \ref{bisrepetita}, we shall compare Groenewegen's and Banaszczyk's results relating the first minimum of some Euclidean lattice with its $\theta$-invariants.

\subsection{The theta function $\theta_{\Eb}$} For  any Hermitian vector bundle $\Eb:= (E, \Vert . \Vert)$
 over $\Spec \Z$, it will be convenient to consider the associated theta function, defined for
 any $t$ in $\R_{+}^\ast$ by
$$\theta_{\Eb}(t):= \sum_{v \in E} e^{-\pi t \Vert v \Vert ^2}.$$

We have, by the very definition of $\hot(\Eb)$:
$$\hot(\Eb) = \log \theta_{\Eb}(1).$$

Moreover, for any $\delta$ in $\R$, if we denote  $\cOb(\delta)$ the Hermitian line bundle over $\Spec \Z$ of Arakelov degree $\dega \cOb(\delta) = \delta$ (introduced in \ref{Odelta}, \emph{supra}), then 
$$\theta_{\Eb \otimes \cOb (\delta)}(t)= \sum_{v \in E} e^{-\pi t e^{-2\delta} \Vert v \Vert ^2}
= \theta_{\Eb}(e^{-2\delta} t).$$
Consequently
$$\hot( \Eb \otimes \cOb (\delta)) = \log \theta_{\Eb}(e^{-2\delta}),$$
and, for any $t \in \R^\ast_+,$
$$\log \theta_{\Eb}(t) = h^0( \Eb \otimes \cOb(- (\log t)/2 )).$$

In other words, $\log \theta_{\Eb}$ is a kind of ``arithmetic Hilbert function" for the Hermitian vector bundle $\Eb$ over $\Spec \Z.$ 

Observe also that the classical functional equation for the theta function may be written :
\begin{equation}\label{FE}
\theta_{\Eb}(t)= t^{-\frac{1}{2} \rk E} (\covol \Eb)^{-1} \theta_{\Eb^\vee}(t^{-1}),
\end{equation}
or equivalently,
\begin{equation}\label{logFE}
\log \theta_{\Eb}(t)= -\frac{1}{2} \rk E. \log t + \dega \Eb + \log \theta_{\Eb^\vee}(t^{-1}).
\end{equation}
It is nothing but the Poisson-Riemann-Roch formula \ref{PoissonZ} for the Hermitian vector bundle $\Eb \otimes \cOb(-(\log t)/2).$

The theta function $\theta_\Eb$ may  be related to the ``naive" counting function
$$N_\Eb: \R_+ \lra \N$$
which controls the number of points in the Euclidean lattice $\Eb$ in balls centered at the origin. Indeed, if we define, for every $x \in \R_+,$
$$N_\Eb (x) := \vert \{ v \in E \mid \Vert v \Vert \leq x \} \vert,$$
then the theta function $\theta_\Eb$ is basically the Laplace transform of $N_\Eb (\sqrt{.})$:

\begin{proposition}\label{thetaLaplace}
With the above notation, for every $t\in \R^\ast_+,$ we have:
$$\theta_\Eb(t) = \pi t \int_0^{+\infty} N_{\Eb}(\sqrt{x}) e^{- \pi tx} dx.$$
 \end{proposition}

\begin{proof} Observe that, for any $x \in \R_+,$
$$\sum_{v\in E} \mathbf{1}_{[\Vert v \Vert^2, + \infty [}(x) = N_\Eb (x).$$
Consequently we have:
\begin{align*}
 \theta_\Eb(t) & := \sum_{v \in E} e^{-\pi t \Vert v \Vert^2} \\
 & = \sum_{v\in E} \pi t \int_0^{+\infty}  \mathbf{1}_{[\Vert v \Vert^2, + \infty [}(x) e^{-\pi t x} dx \\
 & = \pi t \int_0^{+\infty} N_{\Eb}(\sqrt{x}) e^{- \pi tx} dx.
\end{align*}
 (The last equality follows, for instance, from Lebesgue monotone convergence theorem.)
\end{proof}

\subsection{First minima of Euclidean lattices and $\theta$-invariants}

For any Hermitian vector bundle of positive rank $\Eb:= (E, \Vert . \Vert)$  over $\Spec \Z,$ we may consider the first of its successive minima:
$$\lambda_{1}(\Eb):= \min \left\{ \Vert v \Vert, v \in E \setminus \{0\} \right\},$$ and the number of its ``short vectors":
$$\nu := \left\vert \{ v \in E \mid \Vert v \Vert = \lambda_{1}(\Eb) \} \right\vert.$$ 

These quantities are easily seen to control the asymptotic behaviour of $\theta_{\Eb}(t)$ when $t$ goes to $+ \infty,$ or equivalently the one of $\hot(\Eb \otimes \cOb(\delta))$ when $\delta$ goes to $-\infty.$
Indeed, when $t$ goes to $+\infty,$
$$\theta_{\Eb}(t) = 1 + \nu e^{-\pi \lambda_{1}(\Eb)^2 t} + O (e^{- \pi \lambda'^2 t}),$$
for some $\lambda' > \lambda_{1}(\Eb).$
Consequently, for some positive $\epsilon,$ 
$$\log \theta_{\Eb}(t) =  \nu e^{-\pi  \lambda_{1}(\Eb)^2 t}(1+ O (e^{- \epsilon t})) \mbox{ when $t$ goes to $+\infty,$} $$
and
\begin{equation}\label{hotasympt}\hot(\Eb \otimes \cOb(\delta))= \nu e^{-\pi \lambda_{1}(\Eb)^2 e^{-2\delta}} (1+ O (e^{- \epsilon e^{-2\delta}}))  \mbox{ when $\delta$ goes to $-\infty$}.
\end{equation}

It is actually possible to derive an explicit upper bound on $\hot(\Eb)$ in terms of $ \lambda_{1}(\Eb)$, namely:

\begin{proposition}\label{Grbd}
 For any Euclidean lattice $\Eb$ of positive rank $n$ and of first minimum $\lambda := \lambda_{1}(\Eb),$ the following estimates hold:
 \begin{equation}\label{eqGrbd1}
\hot(\Eb) \leq e^{\hot(\Eb)} -1 \leq C(n, \lambda),
\end{equation}
where 
\begin{equation}\label{Cn}
C(n,\lambda) := 3^n (\pi \lambda^2)^{-n/2} \int_{\pi \lambda^2}^{+ \infty} u^{n/2} e^{-u} du.
\end{equation}

Moreover, if $\lambda > (n/2\pi)^{1/2},$
\begin{equation}\label{eqGrbd2}
C(n,\lambda) \leq 3^n \left(1-\frac{n}{2\pi \lambda^2}\right)^{-1} e^{-\pi \lambda^2}.
\end{equation}
\end{proposition}
\qed

This Proposition is a slightly improved version of Proposition 4.4 in \cite{Groenewegen2001}, and we will follow Groenewegen's arguments --- based on the expression of the theta function $\theta_\Eb$ as a Laplace transform of $N_\Eb(\sqrt{.}$) --- with minor modifications.

Observe that the estimates (\ref{eqGrbd1}) and (\ref{eqGrbd2}) applied to $\Eb \otimes \cOb(\delta)$, compared with the asymptotic expression (\ref{hotasympt}) when $\delta$ goes to $-\infty$ show that 
$\nu \leq 3^n.$ Actually, it is classically known\footnote{since each ``short vector" of $\Eb$ is also a facet vector --- that is, a vector $v$ in $\E \setminus\{0\}$ such that the Voronoi domains of $0$ and $v$ have a non-empty intersection of dimension $n-1$ --- and,  as shown by Minkowski,  the number of facet vectors of a Euclidean lattice of rank $n$ is at most $2(2^n - 1).$} that $\nu \leq 2(2^n - 1).$ 

Observe also that, if we define $\tilde{\lambda}$ by
$$\lambda_1 (\Eb) = \sqrt{\frac{n}{2\pi}} \tilde{\lambda},$$
the upper bounds (\ref{eqGrbd1})-(\ref{eqGrbd2}) may be written:
$$\hot(\Eb) \leq (1-\tilde{\lambda}^{-2})^{-1} \left(3 e^{- \tilde{\lambda}^2/2}\right)^n \quad \mbox{when $\tilde{\lambda}>1.$}$$

\begin{lemma}[\cf \cite{Groenewegen2001}, Lemma 4.2]\label{LemmaGroe} With the notation of Proposition \ref{Grbd}, for any $x$ in $\R_+,$ we have:
\begin{equation}\label{UrNlambda}
N_\Eb (x) \leq \left(\frac{2x}{\lambda} + 1\right)^n.
\end{equation}
Consequently, for any $x$ in $[\lambda, +\infty[,$
\begin{equation}\label{Nlambda}
N_\Eb (x) \leq \left(\frac{3x}{\lambda}\right)^n.
\end{equation}
\end{lemma}

\begin{proof} For any $P \in E_\R$ and any $r \in \R_+,$ let us denote by $\Bint(P, r)$ the open ball of center $P$ and radius $r$ in the normed vector space $(E_\R, \Vert. \Vert).$ 	

For any two points $v$ and $w$ of $E,$ we have:
$$ v \neq w \Longrightarrow \Bint\left(v, \lambda/2 \right) \cap \Bint\left(w, \lambda/2\right) = \emptyset.$$
Consequently
$$\sum_{v \in E, \Vert v \Vert \leq x} \lambda (\Bint\left(v, \lambda/2\right))  = \lambda_\Eb \left(\bigcup_{v \in E, \Vert v \Vert \leq x} \Bint\left(v, \lambda/2\right)\right) \leq \lambda (\Bint\left(0, x + \lambda/2\right)).$$
If $v_n$ denotes the volume of the $n$-dimensional unit ball, this estimate may be written:
$$v_n N_\Eb(x) (\lambda/2)^n \leq v_n (x + \lambda/2)^n.$$
This is clearly equivalent to (\ref{UrNlambda}).
\end{proof}

\begin{lemma}\label{incompGamma} For any positive integer $n$ and any $\beta \in ]n/2, + \infty[,$ we have
$$\int_\beta^{+ \infty} u^{n/2}\, e^{-u}\, du \leq (1- \frac{n}{2\beta})^{-1}\,\beta^{n/2}\, e^{-\beta}.$$ 
 \end{lemma}
 
 \begin{proof} The derivative $\frac{n}{2u} - 1$ of $\frac{n}{2} \log u -u$ is bounded from above by $\frac{n}{2\beta} - 1$ when $u$ belongs to $[\beta, + \infty[.$ Therefore, for any $u$ in this interval:
 $$\frac{n}{2} \log u -u \leq (\frac{n}{2} \log \beta -\beta)  + (u-\beta) (\frac{n}{2\beta} - 1).$$
 
 Consequently, we have:
 \begin{align*}
 \int_\beta^{+ \infty} u^{n/2}\, e^{-u}\, du & = \int_\beta^{+ \infty} \exp(\frac{n}{2} \log u -u) \, du \\
 & \leq \exp (\frac{n}{2} \log \beta -\beta) \, \int_\beta^{+\infty} e^{-(1-n/2\beta)(u-\beta)}\, du = \exp (\frac{n}{2} \log \beta -\beta)\, (1-n/2\beta)^{-1}.
\end{align*}
 \end{proof}

\proof[Proof of Proposition \ref{Grbd}]
 From the definition of $\hot(\Eb),$ Proposition \ref{thetaLaplace}, and the fact that $$N_\Eb(x) = 1 \mbox{ if $x \in [0, \lambda[$,}$$
 we get:
\begin{align}
 e^{\hot(\Eb)} -1 = \theta_\Eb (1) - 1 & = \pi \int_0^{+\infty} [N_\Eb(\sqrt{x}) -1] e^{-\pi x}\, dx \notag \\
 & \leq \int_{\pi \lambda^2}^{+\infty} N_\Eb(\sqrt{u/\pi}) e^{-u}\, du. \label{peutetreutile}
\end{align}

Moreover, the upper bound (\ref{Nlambda}) on $N_\Eb$ over $[\lambda, +\infty[$ shows that
$$\int_{\pi \lambda^2}^{+\infty} N_\Eb(\sqrt{u/\pi}) e^{-u}\, du \leq \left(\frac{3}{\lambda \sqrt{\pi}}\right)^n  \int_{\pi \lambda^2}^{+ \infty} u^{n/2} e^{-u} du.$$

This establishes (\ref{eqGrbd1}). Finally (\ref{eqGrbd2}) follows from the integral estimate in Lemma  \ref{incompGamma}, applied with $\beta= \pi \lambda^2.$
\qed

By duality, the upper-bound on $\hot(\Eb)$ in terms of $\lambda_1(\Eb)$ in Proposition \ref{Grbd} may be used to derive an upper bound on $\hut(\Eb):= \hot(\Eb^\vee)$ in terms of the last of its successive minima, namely
$$\lambda_{\rk E} (\Eb) := \min \{ r \in \R_+ \mid \oli{B}(0,r) \cap E\mbox{ contains a basis of $E_\R$} \}.$$

Indeed, for any Euclidean lattice $\Eb$ of positive rank, we have:
$$\lambda_1(\Eb^\vee). \lambda_{\rk E}(\Eb) \geq 1.$$
(Consider a ``short vector" $\xi$ of $E^\vee$, such that $\Vert \xi \Vert_{\Eb^\vee} = \lambda_1(\Eb_\vee)$, and some element $v$ in $$\oli{B}(0, \lambda_{\rk E} (\Eb)) \cap E$$ such that $<\xi, v> \neq 0$ and observe that,  
since $<\xi, v>$ is an integer,
$$ 1 \leq \vert <\xi, v> \vert \leq  \Vert \xi\Vert_{\Eb^\vee} \Vert v\Vert_{\Eb} \leq \lambda_1(\Eb^\vee). \lambda_{\rk E}(\Eb).)$$  
 We therefore immediately derive from Proposition \ref{Grbd}:
 
\begin{corollary}\label{Grbdhut}
For an Euclidean lattice of positive rank $\Eb$ and of last minimum $\lambda_n(\Eb)$, the following estimate holds:
\begin{equation}\label{Grbdhut1}
\hut(\Eb) \leq e^{\hut(\Eb)} -1 \leq C(n, \lambda_n(\Eb)^{-1}),
\end{equation}
where $C(n, .)$ is defined by (\ref{Cn}).

Moreover, if 
$$\tilde{\lambda}_n := \sqrt{\frac{n}{2\pi}}\, \lambda_n(\Eb) < 1,$$
then
\begin{equation}\label{Grbdhut2}
C(n, \lambda_n(\Eb)^{-1}) \leq (1 -\tilde{\lambda}_n^2)^{-1} \left(3\, e^{-\tilde{\lambda}_n^{-2}/2}\right)^n.
\end{equation}
 \qed
\end{corollary}
\section{Application to Hermitian line bundles}

Let $\Lb$ be a Hermitian line bundle over $\Spec \OK.$ 

\subsection{The first minimum of the direct image of a Hermitian line bundle} The direct image $\pi_\ast \Lb$ is an Euclidean lattice of rank $n :=[K:\Q]$ over $\Spec \Z$. Its first minimum $\lambda_1(\pi_\ast \Lb)$ admits a simple lower bound in terms of the \emph{normalized Arakelov degree}
$$\degan \Lb := \frac{1}{[K:\Q]} \dega \Lb$$
of $\Lb$. Indeed, as shown in \cite{Groenewegen2001}, Lemma 7.1, 
or \cite{BK10}, Proposition 3.3.1, 
we have:
\begin{equation}\label{lambdadegan}
\frac{\lambda_1(\Eb)^2}{n} \geq e^{- 2 \degan \Lb}.
\end{equation}

In particular, we have:
\begin{equation}\label{lambdadegnaimpl}
\degan \Lb < \frac{1}{2} \log (2\pi) \Longrightarrow  \frac{\lambda_1(\Eb)^2}{n} > \frac{1}{2\pi}.
\end{equation}
Therefore, for any Hermitian line bundle satisfying $\degan \Lb < (1/2)  \log (2\pi)$, we may apply Proposition \ref{Grbd} and derive an upper bound on $\hot(\Lb) := \hot(\pi_\ast \Lb).$

\subsection{Hermitian line bundles of negative degree}

For instance, when $\dega \Lb \leq 0,$ from (\ref{lambdadegan}) we obtain: 
$$\lambda_1(\Eb)^2 \geq n,$$
and the upper bounds (\ref{eqGrbd1}) and (\ref{eqGrbd2}) become:
\begin{equation}
e^{\hot(\Lb)} -1 \leq 3^n (1-1/2\pi)^{-1} e^{-\pi  \lambda_1(\Eb)^2 }.
\end{equation}

By using (\ref{lambdadegan}) again, we finally obtain the following minor variant of \cite{Groenewegen2001}, Proposition 7.2 (which, in a slightly less precise form, already appears as Corollary 1 to Proposition 2 in \cite{vanderGeerSchoof2000}):

\begin{proposition}\label{linenegative} For any Hermitian line bundle $\Lb$ over $\Spec \OK$ such that $\dega \Lb \leq 0,$
we have:
\begin{equation}\label{eqlinenegative}\hot(\Lb) \leq 3^{[K:\Q]} (1-1/2\pi)^{-1} \exp\left(-\pi [K:\Q] e^{- 2 \degan \Lb}\right).\end{equation}
\end{proposition}

The right-hand side of (\ref{eqlinenegative}) is always $\leq 1,$ and actually goes to zero when $[K:\Q]$ or $-\dega \Lb$ goes to $+\infty.$ Indeed, from (\ref{eqlinenegative}), we immediately obtain:
\begin{equation}\label{eqlinenegativebis}\hot(\Lb) \leq c^{[K:\Q]} \exp\left(-\pi [K:\Q]( e^{- 2 \degan \Lb}- 1)\right),
\end{equation}
where 
$$c:= \frac{3 e^{-\pi}}{1 - 1/2\pi}= 0.154180...$$
\subsection{Hermitian line bundles of positive degree}

\begin{proposition}\label{linepositive}
There exists $c \in ]0,1[$ such that, for any number field $K$
and any Hermitian line bundle $\Lb$ over $\Spec \OK$ such that $\dega \Lb \geq 0,$ we have: 
\begin{equation}\label{eqlinepositive}\hot(\Lb) \leq c^{[K:\Q]} + \dega \Lb.
\end{equation}
\end{proposition}

The proof will show that we may choose the same constant $c$ as in (\ref{eqlinenegativebis}).

Clearly Proposition \ref{linepositive} establishes the validity, for any Hermitian line bundle of non-negative degree over $\Spec \OK$, of the  following  inequality, familiar in a geometric context:
\begin{equation}\label{familiarpositive}
\hot(\Lb) \leq 1 + \dega \Lb.
\end{equation}

Proposition \ref{linepositive} improves on \cite{Groenewegen2001}, Proposition 7.3, where a similar upper bound on $\hot(\Lb)$ is established, with a constant of the order of $(1/2) [K:\Q] \log [K:\Q]$ instead of $c^{[K:\Q]}$.

\proof[Proof of Proposition \ref{linepositive}] When $\dega \Lb =0,$ the estimate (\ref{eqlinepositive}) takes the form
$$ \hot(\Lb) \leq c^{[K:\Q]} $$
and 
follows from 
(\ref{eqlinenegativebis}). 
To derive the general validity of  (\ref{eqlinepositive}) from this special case, we may use the inequality (\ref{hotED}) established in Corollary \ref{cor:hotED}. 

Indeed, if $\Lb$ is a Hermitian line bundle over $\Spec \OK$ of non-negative Arakelov degree and if 
$$\delta = \degan \Lb:= [K:\Q]^{-1} \dega \Lb,$$ then the  Hermitian line bundle $\cOb(\delta)$ satisfies 
$\dega \cOb(\delta) = \dega \Lb.$
Therefore $\dega (\Lb \otimes \cOb(-\delta) )= 0$
and, according to the special case above,
\begin{equation}\label{eqlinepositive1}
\hot( \Lb \otimes \cOb(-\delta)) \leq c^{[K:\Q]}.
\end{equation}

Besides, according to (\ref{hotEdelta}),
\begin{equation}\label{eqlinepositive2}
\hot(\Lb) \leq \hot( \Lb \otimes \cOb(-\delta)) + [K:\Q] \delta.
\end{equation}  

The inequality (\ref{eqlinepositive}) follows from (\ref{eqlinepositive1}) and (\ref{eqlinepositive2}).
\qed

\subsection{A scholium}
For later reference, we spell out the following straightforward consequences of the upper-bounds on $\hot(\Lb)$ established in the previous paragraphs:
\begin{proposition}\label{scholiumhot}
 For any Hermitian line bundle $\Lb$ over $\Spec \OK$ and any $t\in \R$ such that 
 $\dega \Lb \leq t,$
 we have :
 $$\hot(\Lb) \leq 1 + t \,\,\,\mbox{ if $t\geq 0,$}$$
 and 
 $$\hot(\Lb) \leq \exp\left( -\pi [K:\Q] (e^{-2t/[K:\Q]}-1)\right) \leq \exp( 2\pi t) \,\,\,\mbox{ if $t\leq 0.$}$$
 \qed
\end{proposition}

These estimate may be seen as an arithmetic counterpart of the basic ``Riemann inequality": 
\begin{equation}\label{RiemannGeom}
h^0(C,L) := \dim_k \Gamma(C, L) \leq (1+ \deg_C L)^+,
\end{equation}
valid for any line bundle $L$ over some smooth, projective, geometrically connected curve $C$ over some field $k$, of degree 
$$\deg_C L := \deg_k c_1(L).[C],$$

\section{Subadditivity of $\hot$ and $\hut$}\label{subadd}

\subsection{The basic subadditivity property}

\begin{proposition}\label{ineqshorttheta}
For any admissible short exact sequence of Hermitian vector bundles over the arithmetic curve $\Spec \OK$
\begin{equation}\label{shorther}
\overline{\cE} : 0 \lra \Eb \stackrel{i}{\lra} \Fb \stackrel{p}{\lra} \Gb \lra 0,
\end{equation}
the following inequality holds:
\begin{equation}\label{keyineq}
h_\theta (\overline{\cE}) := \hot(\Eb) -\hot(\Fb) + \hot(\Gb) \geq 0.
\end{equation}

Moreover equality holds in (\ref{keyineq}) if and only if the admissible short exact sequence (\ref{shorther}) is split.
\end{proposition}

Observe that the additivity of the Arakelov degree in short exact sequences (\ref{degaddadm}), together with the Poisson-Riemann-Roch formula (\ref{Poisson}) shows that
\begin{equation}\label{hothutshort}
\hut(\Eb) -\hut(\Fb) + \hut(\Gb) = \hot(\Eb) -\hot(\Fb) + \hot(\Gb).
\end{equation}
In particular, Proposition \ref{ineqshorttheta} also holds with $\hot$ replaced by $\hut.$ 

\subsection{Proof of Proposition \ref{ineqshorttheta}}

When $\OK = \Z,$ Proposition \ref{ineqshorttheta} appears in Quillen Notebooks \cite{QuillenNotebooks} in the entry of 26/04/1973, and is established as Lemma 5.3 in \cite{Groenewegen2001}. 

The inequality (\ref{keyineq}) actually follows from the following lemma, of independent interest, which may be seen as a ``pointwise version" of  (\ref{keyineq}) and will play a key role in deriving the main results of this monograph (notably  to establish Lemma \ref{submeasure}, crucial to the proof of Theorem \ref{maintheorem} \emph{infra}):

\begin{lemma}\label{gaussiansumpreimage} Consider an admissible short exact sequence of Hermitian vector bundles over $\Spec \Z$,
$$0 \lra \Eb \stackrel{i}{\lra} \Fb \stackrel{p}{\lra} \Gb \lra 0.$$
Then, for any $g \in G,$ its preimage $p^{-1}(g)$ in $F$ satisfies:
\begin{equation}\label{ineqshortthetapartial}
\sum_{f \in p^{-1}(g)} e^{-\pi \Vert f \Vert^2_{\Fb}} \leq e^{-\pi \Vert g \Vert^2_{\Gb}}  \sum_{e\in E} e^{-\pi \Vert e \Vert^2_{\Eb}}.
\end{equation}
\end{lemma}

Indeed, by summing (\ref{ineqshortthetapartial}) over $g \in G$, we get the inequality:
$$\sum_{f \in F} e^{-\pi \Vert f \Vert^2_{\Fb}} \leq \sum_{g \in G} e^{-\pi \Vert g \Vert^2_{\Gb}}  . \sum_{e\in E} e^{-\pi \Vert e \Vert^2_{\Eb}}.$$
Taking the logarithms, this establishes (\ref{keyineq}) when $\OK = \Z.$ The case of a general number field $K$ follows from this special case, applied to the admissible short exact sequence over $\Spec \Z$ deduced from (\ref{shorther}) by taking its direct image on $\Spec \Z.$

\begin{proof}[Proof of Lemma \ref{gaussiansumpreimage}]
Let us denote by $s^{\perp}:G_\R \lra F_\R$  the orthogonal splitting of the surjective linear map $p_\R: F_\R \lra G_\R$. (Its image is the orthogonal complement $(\ker p_\R)^{\perp}$ of $\ker p_\R$ in $F_\R,$ defined by means of the Euclidean structure on $F_\R$ attached to $\Vert.\Vert_{\Fb}$.) 

We may choose an element $f_0$ in $p^{-1}(g)$. Then we have:
$$p^{-1}(g) = f_0 + i(E).$$
The element $f_0 - s^{\perp}(g)$ of $F_\R$ belongs to $\ker p_\R = \im i_\R$, and may be written $i_\R (\delta)$ for some (unique) element $\delta$ of $E_\R$. Then, for any $e\in E,$ we have
$$\Vert f_0 + i(e) \Vert^2_{\Fb} = \Vert s^{\perp} (g) + i(\delta + e)\Vert^2_{\Fb} = \Vert g \Vert^2_{\Gb} + \Vert \delta + e \Vert^2_{\Eb}.$$
Together with (\ref{PoissonGaussineq}) (applied to $\Vb = \Eb$ and $x=\delta$), this shows that
$$\sum_{f \in p^{-1}(g)} e^{-\pi \Vert f \Vert^2_{\Fb}} = \sum_{e \in E} e^{-\pi (\Vert g \Vert^2_{\Gb} + \Vert \delta + e \Vert^2_{\Eb})}
\leq e^{-\pi \Vert g \Vert^2_{\Gb}}  \sum_{e\in E} e^{-\pi \Vert e \Vert^2_{\Eb}}.$$
\end{proof}

To complete the proof of Proposition \ref{ineqshorttheta}, we are left to show that equality holds in (\ref{keyineq}) if and only if the admissible short exact sequence (\ref{shorther}) over $\Spec \OK$ is split.

Observe that, according to Proposition \ref{piextun}, 
$$\pi_\ast\overline{\mathcal E}: 0 \lra \pi_\ast \Eb \stackrel{i}{\lra} \pi_\ast \Fb \stackrel{p}{\lra} \pi_\ast \Gb \lra 0$$
is an admissible extension of Hermitian vector bundles over $\Spec \Z$, and is split (over $\Spec \Z$) if and only if  $\overline{\mathcal E}$ is split (over $\Spec \OK$). Moreover, by the very definition of $\hot,$ we have:
\begin{equation*}
\begin{split}
h_\theta (\overline{\mathcal E}) & = \hot(\Eb) -\hot(\Fb) + \hot(\Gb) \\
 & = \hot(\pi_\ast \Eb) -\hot(\pi_\ast  \Fb) + \hot(\pi_\ast  \Gb) \\
 & = h_\theta (\pi_\ast\overline{\mathcal E}).
\end{split}
\end{equation*}

Therefore, to complete the proof of Proposition \ref{ineqshorttheta}, we may  assume that $\Spec \OK = \Spec \Z.$ 
In this case, it will follow from an analysis of the equality case in the above arguments.

To expound this analysis, it is convenient to introduce the following definition. Under the assumption that $\OK = \Z,$ for every $T$ in 
$$\Hom_\C(G_\C, E_\C)^{F_\infty} \simeq \Hom_\Z(G,E)\otimes \R \simeq \Hom_\R(G_\R, E_\R), $$
we define:
\begin{equation}
\Gext_{\hE, \hG}(T) := \sum_{(e,g) \in E\times G} e^{-\pi (\Vert e -T(g) \Vert_{\Eb}^2 + \Vert g \Vert_{\Gb}^2)}.
\end{equation}
The proof of Proposition \ref{ineqshorttheta} is now completed by the following Proposition, the formulation of which uses the formalism of arithmetic and admissible extensions recalled in Subsections \ref{AdmShort}--\ref{subsub:admiari}:

\begin{proposition}\label{Gext} For every $T \in \Hom_{\R}(G_\R, E_\R)$ the class of which in $\Exthun_\Z (G,E)$ coincides with the class of the admissible extension $\mathcal E$, the following equality holds:
\begin{equation}\label{htG}
\exp(-h_\theta (\overline{\cE})) = \frac{\Gext_{\Eb, \Gb}(T)}{\Gext_{\Eb, \Gb}(0)}.
\end{equation}

Moreover  we have:  
\begin{equation}\label{PoissonhtG}
\Gext_{\Eb, \Gb}(T) = \covol (\Eb) \sum_{(e^\vee,g) \in E^\vee\times G} e^{-\pi (\Vert e^\vee\Vert_{\Eb^\vee}^2 + \Vert g \Vert_{\Gb}^2)} e^{2\pi i \langle e^\vee\!\otimes g, T\rangle}.
\end{equation}
and
\begin{equation}\label{GextIneq}
0 < \Gext_{\Eb, \Gb}(T) \leq \Gext_{\Eb, \Gb}(0),
\end{equation}
and the equality $\Gext_{\Eb, \Gb}(T) = \Gext_{\Eb, \Gb}(0)$ holds if and only if $T$ belongs to $\Hom_\Z(G,E).$
\end{proposition}

In the right side of (\ref{PoissonhtG}), $e^\vee\!\otimes g$ belongs to $E^\vee \otimes G$ and $T$ to 
$$\Hom_{\R}(G_\R, E_\R) \simeq G^\vee_\R \otimes_\R E_\R \simeq (G^\vee \otimes E)_\R,$$ and their pairing $\langle e^\vee\!\otimes g, T\rangle$ is equal to the real number $e^\vee(T(g))$.

\proof From the very definition of $\Gext_{\Eb, \Gb}(T)$, we get:
$$\hot(\overline{E\oplus G}^T) = \log \Gext_{\Eb, \Gb}(T).$$
In particular:
$$ \hot(\Eb) + \hot(\Gb) = \log \Gext_{\Eb, \Gb}(0).$$
The relation (\ref{htG}) follows from these two equalities.

The equality (\ref{PoissonhtG}) follows from the Poisson formula (\ref{PoissonGauss}) applied to $\Vb = \Eb$ and $x=T(g).$

The inequality $ \Gext_{\Eb, \Gb}(T) >0$ is clear, and the inequality $\Gext_{\Eb, \Gb}(T) \leq \Gext_{\Eb, \Gb}(0)$ follows from (\ref{PoissonhtG}). 

Finally the expression (\ref{PoissonhtG}) for  $\Gext_{\Eb, \Gb}(T)$ shows that, for any $T$ in $$\Hom_\R(G_\R, E_\R)\simeq \Hom_\Z(G,E) \otimes_\Z \R,$$ the following conditions are successively equivalent:
\begin{enumerate}
\item $\Gext_{\Eb, \Gb}(T) = \Gext_{\Eb, \Gb}(0);$
\item for any $(e^\vee,g) \in E^\vee\times G, e^{2\pi i \langle e^\vee\!\otimes g, T\rangle} =1;$
\item for any $(e^\vee,g) \in E^\vee\times G, \langle e^\vee\!\otimes g, T\rangle \in \Z;$
\item $T$ belongs to $G^\vee \otimes E \simeq\Hom_\Z(G,E).$
\qed 
\end{enumerate}

\subsection{The average value of $\exp(-h_\theta (\overline{\cE}))$.}

Proposition \ref{Gext} not only makes clear the non-negativity of $h_\theta (\overline{\cE})$. It also allows us to compute its ``geometric average" over the arithmetic extension group  $\Exthun_\Z (G,E)$.

The group
$$\Exthun_\Z (G,E) \simeq \Hom_\Z(G,E) \otimes_\Z \R/ \Z$$
has indeed a natural structure of compact Lie group (it is a ``compact torus" of dimension $\rk E. \rk F$) and, as such, is equipped with a canonical Haar measure, normalized by the condition
$$\int_{\Exthun_\Z (G,E)} d\mu =1.$$ 

\begin{proposition}\label{GextAverag}
 Let $\Eb$ and $\Gb$ be two Hermitian vector bundles over $\Spec \Z$. For any $T \in \Hom_\R(G_\R, E_\R),$ let $[T]$ denote its class in $\Exthun_\Z (G,E)$ and let 
 $$\overline{\cE} (T): \;\; 0 \lra \Eb \stackrel{i}{\lra} \overline{E\oplus G}^T \stackrel{p}{\lra} \Gb \lra 0$$
 be the associated admissible extension of class $[T]$ \emph{(see \ref{subsub:admiari} and (\ref{defcET}) \emph{supra})}.
 
 Then we have:
 \begin{equation}\label{eq:GextAverag}
\int_{[T] \in \Exthun_\Z (G,E)} e^{-h_\theta(\overline{\cE} (T))} \; d\mu([T]) = 1- \left(1- e^{- \hut(\Eb)}\right) \left(1-e^{-\hot(\Gb)}\right).
\end{equation}
\end{proposition}

\begin{proof} According to (\ref{htG}), we have:
\begin{equation}\label{eq:GextAverag1}
\int_{[T] \in \Exthun_\Z (G,E)} e^{-h_\theta(\overline{\cE} (T))} \; d\mu([T]) = {\Gext_{\Eb, \Gb}(0)}^{-1} \int_{[T] \in \Exthun_\Z (G,E)} \Gext_{\Eb, \Gb}(T) \; d\mu([T]).
\end{equation}

Moreover, from (\ref{PoissonhtG}), we get:
\begin{align}\label{eq:GextAverag2}
 \Gext_{\Eb, \Gb}(0) & = \covol (\Eb) \sum_{(e^\vee,g) \in E^\vee\times G} e^{-\pi (\Vert e^\vee\Vert_{\Eb^\vee}^2 + \Vert g \Vert_{\Gb}^2)} \\
& = \covol (\Eb) \,e^{\hut(\Eb)}\, e^{\hot(\Gb)},
\end{align}
and
$$\int_{[T] \in \Exthun_\Z (G,E)} \Gext_{\Eb, \Gb}(T) \; d\mu([T])  = \covol (\Eb) \sum_{\stackrel{(e^\vee,g) \in E^\vee\times G}{e^\vee\!\otimes g =0}} e^{-\pi (\Vert e^\vee\Vert_{\Eb^\vee}^2 + \Vert g \Vert_{\Gb}^2)} $$
\begin{align}\label{eq:GextAverag3}
 \int_{[T] \in \Exthun_\Z (G,E)} \Gext_{\Eb, \Gb}(T) \; d\mu([T]) & = \covol (\Eb) \sum_{\stackrel{(e^\vee,g) \in E^\vee\times G}{e^\vee\!\otimes g =0}} e^{-\pi (\Vert e^\vee\Vert_{\Eb^\vee}^2 + \Vert g \Vert_{\Gb}^2)} \\ 
 &  = \covol (\Eb) \left( 
  \sum_{e^\vee \in E^\vee} e^{-\pi \Vert e^\vee\Vert_{\Eb^\vee}^2}  +
 \sum_{g \in\times G} e^{-\pi \Vert g \Vert_{\Gb}^2}
 - 1
\right)  \\  
 & = \covol (\Eb) \left( e^{\hut(\Eb)} + e^{\hot(\Gb)} - 1 \right).
\end{align}

Formula (\ref{eq:GextAverag}) follows from (\ref{eq:GextAverag1}), (\ref{eq:GextAverag2}) and (\ref{eq:GextAverag3}).
 \end{proof}

\begin{corollary}
 There exists an admissible extension $\overline{\cE}$ of $\Gh$ by $\Eh$ such that
 $$ h_\theta(\overline{\cE}) > - \log \left[1- \left(1- e^{- \hut(\Eb)}\right) \left(1-e^{-\hot(\Gb)}\right)\right].$$
 \qed
\end{corollary}

\medskip

\chapter{Geometry of numbers and $\theta$-invariants}\label{thetatwo}

\medskip

In this chapter, we pursue our study of the invariants $h^i_\theta (\Eb)$ attached to some Hermitian vector bundle $\Eb$ over some arithmetic curve $\Spec \OK$. We focus on the basic situation when $\OK = \Z$ --- that is, when $\Eb$ is an Euclidean lattice --- and we relate the $\theta$-invariants $\hot(\Eb)$ and $\hut(\Eb)$ to various invariants of $\Eb$ classically considered in geometry of numbers.

The results of the present chapter will not be used in the next chapters of this monograph. They are intended to clarify the meaning of the $\theta$-invariants from the perspective of the theory of Euclidean lattices, and their proofs will combine three lines of thought:

(i) The methods  introduced by Banaszczyk in his work \cite{Banaszczyk93} to establish ``transference inequalities", relating invariants of  Euclidean lattices and of their dual lattices, with essentially optimal constants\footnote{See  \cite{Gasbarri99}, Section 3, for an early application of these methods to Arakelov geometry. Banasczyk results have played  an important role in lattice-based cryptography, and we refer the reader to \cite{MicciancioRegev2007} and \cite{TianLiuXu2014} for improvements and applications of the original results in \cite{Banaszczyk93} related to lattice-based cryptography, and for further references.}. 

(ii) An extended version of Cram\'er's theory of large deviations, valid on some measure space of infinite total mass. This theory is presented in Appendix \ref{Append:LD}, and the reader will find a self-contained summary of its results required for our study of Euclidean lattices  in Section \ref{ReforComp}, where they are stated in a form that emphasizes their relations with the formalism of statistical physics.

(iii) Siegel's mean value theorem and its use for showing the existence of Euclidean lattices with remarkable density properties (\emph{\`a la} Minkowski-Hlawka) by probabilistic arguments. 

As explained in the introduction, Banaszczyk's techniques allows us to establish notably that the invariants $\hot(\Eb)$ and $\hon(\Eb)$ of some Euclidean lattices $\Eb$ differ by some error term bounded in terms of the rank of $E$ only (see (\ref{comparingavatars}) and Theorem \ref{propcompthetanaive}). Although these techniques  have been developed independently,  they  constitute a tool of choice in the study of the $\theta$-invariants. Conversely, the use of $\theta$-invariants shed some light on Banaszczyk's arguments and, in Section \ref{Banasc2},  we  give a self-contained  presentation of some  of the most important  results in \cite{Banaszczyk93} from this perspective. 

The large deviation theorems  in Appendix \ref{Append:LD} admit as consequences the duality relations  (\ref{Leg-hontIntro}) and (\ref{Leg-logthetaIntro}) between the asymptotic invariant $\hont(\Eb, t)$
and the function $\log \theta_\Eb.$ Moreover their thermodynamical interpretation leads us to the expression 
\begin{equation}\label{equ:secondlaawlatticeintro}
\hont(\Eb_1 \oplus \Eb_2, t) = \max_{\stackrel{t_1, t_2 >0}{t_1 +t_2 = t}} \left(\hont(\Eb_1, t_1) + \hont(\Eb_2, t_2)\right).
\end{equation}
 for the asymptotic invariant $\hont(\Eb_1 \oplus \Eb_2, t)$ attached to the direct sum $\Eb_1 \oplus \Eb_2$ of two Euclidean lattices
 $\Eb_1$ and $\Eb_2$. Formula  (\ref{equ:secondlaawlatticeintro}) may actually be understood as an avatar of the second law of thermodynamics 
(see Proposition \ref{SecondLawLattice}).

The probabilistic arguments based on Siegel's theorem mean value theorem will notably establish the existence, for any $(n,\delta)$ in $\N_{\geq 2} \times \R$, of some Euclidean lattice $\Eb$ of rank $n$ and degree $\delta$ such that
\begin{equation*}
\hot(\Eb) < \log (1 + e^\delta).
\end{equation*}  
These probabilistic arguments will also show that the constants in diverse comparison estimates established in this section are essentially optimal.

\section{Comparing $\hot$ and $\hon$}\label{compho}

\subsection{The invariants $\hon(\Eb),$ $\honm(\Eb),$ and $h^0_{\rm Bl}(\Eb)$}

Following \cite{GilletSoule1991}, we define:
$$\hon(\Eb) := \log \vert \{ v \in E \mid \Vert v \Vert \leq 1 \}\vert.$$
This definition is motivated by the classical philosophy in Arakelov geometry (see for instance  \cite{Manin85}, Section 2.3) according to which the finite set  $$\{ v \in E \mid \Vert v \Vert \leq 1 \}$$ should be interpreted as the ``global sections" of the Hermitian vector bundle $\Eb$ over $\Spec \Z$ ``completed" by its archimedean place. 

It turns out that the invariant $\hon(\Eb)$, defined as above as the logarithm of the number of points in the unit ball of the Euclidean lattice $\Eb,$ and the $\theta$-invariant $\hot(\Eb)$ coincide up to an error term bounded in function of $\rk E$ only: 

\begin{theorem}\label{propcompthetanaive} 
 For any Euclidean lattice of positive rank $\Eb,$ we have:
\begin{equation}\label{compthetanaive}
\hot(\Eb)  - \frac{1}{2} \rk E. \log \rk E + \log (1 - 1 / 2 \pi) \leq \hon(\Eb) \leq \hot(\Eb) + \pi.
\end{equation}
\end{theorem}

We shall actually establish a slightly stronger version of the first inequality in (\ref{compthetanaive}), namely\footnote{Since $\honm(\Eb) = \hon(\Eb \otimes \cOb(-\epsilon))$ for any small enough $\epsilon$ in $\R^\ast_+$, (\ref{compthetanaivemoins}) actually follows from (\ref{compthetanaive}) applied to $\Eb \otimes \cOb(-\epsilon)$, by taking the limit $\epsilon \rightarrow 0_+$.}:
\begin{equation}\label{compthetanaivemoins}
\honm(\Eb) := \log \vert \{ v \in E \mid \Vert v \Vert < 1 \}\vert  \geq \hot(\Eb) + \log (1 - 1/2\pi) - \frac{1}{2} \rk E. \log \rk E.
\end{equation}

Together with the ``Riemann inequality'' over $\Spec \Z$ (\ref{ThetaRI}), inequality  (\ref{compthetanaivemoins}) entails the following strengthened variant of Minkowski's First Theorem:
$$\honm(\Eb) \geq \log (1 - 1 / 2 \pi) - \frac{1}{2} \rk E. \log \rk E+ \dega \Eb.$$

\begin{corollary}\label{lambdaoneone} For any Euclidean lattice of positive rank $\Eb,$ 
$$\lambda_1(\Eb) \geq 1 \Longrightarrow \hot(\Eb) \leq - \log (1 - 1/{2 \pi}) + \frac{1}{2} \rk E. \log \rk E .$$
 \end{corollary}
 
 \proof The condition $\lambda_1(\Eb) \geq 1$ is equivalent to the equality $\honm(\Eb) = 0.$
 \qed
 
 Observe that, from Proposition \ref{Grbd}, we also obtain  a lower bound on $\hot(\Eb)$ for any Euclidean lattice $\Eb$ with $n:= \rk E >0$ and  $\lambda_1(\Eb) \geq 1$, namely:
 $$\hot(\Eb) \leq \log (1 + C(n,1)) = \log\left(1+ (3/\sqrt{\pi})^n \int_\pi^{+ \infty} u^{n/2} e^{-u}\, du\right).$$
 This estimate is slightly weaker than the one in Corollary \ref{lambdaoneone}. Indeed, when $n$ goes to $+\infty,$ its right hand side is of the form $\frac{1}{2} n.\log n + \alpha.n + o(n)$ with $\alpha >0$. 
 
 We may also introduce a variant \emph{à la Blichfeldt} of $\hon(\Eb),$ namely
 $$h^0_{\rm Bl}(\Eb) := \log \max_{ x \in E_\R} \vert \{ v \in E \mid \Vert v- x \Vert \leq 1 \} \vert$$
 (see Subsection \ref{Blich} \emph{infra}).
 It is straightforward that
 $$\honm (\Eb) \leq \hon(\Eb) \leq h^0_{\rm Bl}(\Eb),$$
 and we shall establish the following strengthening of the second inequality in (\ref{compthetanaive}):
 
 \begin{proposition}\label{propcompthetaBlich}  For any Euclidean lattice of positive rank $\Eb,$ we have:
\begin{equation}\label{compthetaBlich}
h^0_{\rm Bl}(\Eb) \leq \hot(\Eb) + \pi.
\end{equation}
\end{proposition}

Here is an elementary consequence of Theorem \ref{propcompthetanaive} and Proposition \ref{propcompthetaBlich} which does not seem to appear in the literature:

\begin{corollary} For any Euclidean lattice of positive rank $\Eb,$ 
$$h^0_{\rm Bl}(\Eb) \leq  \pi - \log (1 - 1 / 2 \pi) +  \frac{1}{2} \rk E. \log \rk E +  \honm(\Eb).$$ \qed
 \end{corollary}

\subsection{Proof of  Proposition \ref{propcompthetaBlich} and Theorem \ref{propcompthetanaive} }\label{Banasc1}

To prove Proposition \ref{propcompthetaBlich}, namely that 
$$h^0_{\rm Bl} (\Eb) \leq \hot(\Eb) + \pi,$$
we simply    observe that,   from the very definition of $\hot(\Eb)$ and the inequality (\ref{PoissonGaussineq}), we have, for any $x \in E_\R,$
$$\hot (\Eb) = \log \sum_{v \in E} e ^{-\pi \Vert v \Vert^2} \geq \log \sum_{v \in E} e ^{-\pi \Vert v - x\Vert^2} \geq \log \sum_{v \in E, \Vert v - x \Vert \leq 1} e ^{-\pi \Vert v - x\Vert^2}$$
and that the last sum is at least 
$$e^{-\pi}. \vert \{ v \in E \mid \Vert v -x \Vert \leq 1\} \vert.$$

The second inequality in (\ref{compthetanaive}), namely
$\hon(\Eb) \leq \hot(\Eb) + \pi,$
follows from the special case $x=0$ (which does not request the use of  (\ref{PoissonGaussineq})) of the previous argument.

We split the proof of (\ref{compthetanaivemoins}) in a succession of auxiliary statements, of independent interest. The following assertions  are  variants of   results in \cite{Banaszczyk93}, Section 1. 

\begin{lemma}\label{VarBanas}
1) The expression $\log \theta_{\Eb}(t)$ defines a decreasing function of $t$ in $\R_{+}^\ast,$
and the expression 
\begin{equation}\label{incr}\log \theta_{\Eb}(t)  + \frac{1}{2} \rk E. \log t\end{equation}
 an increasing function of $t$ in $\R_{+}^\ast.$

2) We have:
\begin{equation}\label{thetav2}
\sum_{v \in E} \Vert v \Vert^2 e^{-\pi t \Vert v \Vert^2} \leq \frac{\rk E}{2\pi t} \sum_{v\in E}e^{-\pi t \Vert v\Vert^2}.
\end{equation}

3) For any $t$ and $r$ in $\R^\ast_{+},$ we have:
\begin{equation}\label{thetar}
\sum_{v \in E, \Vert v \Vert  <  r}  e^{-\pi t \Vert v \Vert^2} \geq \left( 1 - \frac{\rk E}{2\pi t r^2}\right) \sum_{v\in E}e^{-\pi t \Vert v\Vert^2}.
\end{equation}
\end{lemma}

\proof The first assertion in 1) is clear. According to the Functional Equation (\ref{logFE}), the expression (\ref{incr}) may also be written
$$ \dega \Eb + \log \theta_{\Eb^\vee}(t^{-1}),$$
and consequently defines an increasing function of $t.$ (This fact is also a reformulation of Lemma \ref{cor:hotED} in the special case $K=\Q$.) 
The inequality (\ref{thetav2}) may also be written
$$-  \frac{1}{\pi} \frac{d \theta_{\Eb}(t)}{dt} \leq \frac{\rk E}{2 \pi t} \theta_{\Eb}(t),$$
and simply expresses that the derivative of (\ref{incr}) is non-negative. 

To establish the inequality (\ref{thetar}), we combine 
the straightforward estimate
$$\sum_{v \in E, \Vert v \Vert \geq  r}  e^{-\pi t \Vert v \Vert^2} \leq \frac{1}{r^2}  \sum_{v\in E}\Vert v \Vert^2 e^{-\pi t \Vert v\Vert^2}$$
with (\ref{thetav2}). This yields:
$$ \sum_{v \in E, \Vert v \Vert \geq  r}  e^{-\pi t \Vert v \Vert^2} \leq   \frac{\rk E}{2 \pi t r^2}  \sum_{v \in E}  e^{-\pi t \Vert v \Vert^2},$$
or equivalently, 
$$\sum_{v \in E, \Vert v \Vert <  r}  e^{-\pi t \Vert v \Vert^2} \geq \left( 1 - \frac{\rk E}{2\pi t r^2}\right) \sum_{v\in E}e^{-\pi t \Vert v\Vert^2}.$$ 
\qed 

From (\ref{thetar}) with $r=1$, we obtain
that, for any $t > \rk E /2\pi,$ we have:
$$\honm(\Eb) \geq \log (1 - \rk E /(2 \pi t)) + \log \theta_{\Eb} (t).$$
Using also that, for any $t\geq 1,$
$$\log \theta_{\Eb} (t) \geq \log \theta_{\Eb}(1)  - \frac{1}{2} \rk E. \log t,$$
we finally obtain:
\begin{proposition}\label{thetanaive2}
For any $t \geq \min (1, \rk E/2 \pi),$ we have
$$\honm(\Eb) \geq \log (1 - \rk E /(2 \pi t))  - \frac{1}{2} \rk E. \log t + \hot(\Eb).$$
\end{proposition}

Notably we may choose $t= \rk E,$ and then we obtain\footnote{The ``optimal"   choice of $t$ in terms of $n := \rk E$ would be $t= (n+2)/2\pi.$ This choice leads to the slightly stronger  estimate:
$\honm(\Eb) \geq - \frac{n+2}{2} \log \frac{n+2}{2\pi} -\log \pi + \hot(\Eb).$}:
$$\honm(\Eb) \geq \log (1 - {1}/{2 \pi}) - \frac{1}{2} \rk E. \log \rk E+ \hot(\Eb).$$
This establishes the  inequality  (\ref{compthetanaivemoins}). 

\section{Banaszczyk's estimates and $\theta$-invariants}\label{Banasc2}

In the next paragraphs, we want to discuss in  more details the relation between Banaszczyk's celebrated results in geometry of numbers (\cite{Banaszczyk93}; see also \cite{MicciancioRegev2007} and \cite{TianLiuXu2014} for some recent developments) and the properties of $\theta$-invariants. 

This relation already appeared in the derivation of the estimates relating the invariants $\hot(\Eb)$, $\hon(\Eb)$ and $h^0_{\rm Bl}(\Eb)$ in Subsection \ref{Banasc1}. 

\subsection{Banaszczyk's key estimate}
The starting point of Banaszczyk results is arguably the following estimate, which actually is a simple consequence of the increasing character of the function $\log \theta_{\Eb}(t)  + \frac{1}{2} \rk E. \log t$ established in Lemma \ref{cor:hotED} and Lemma \ref{VarBanas}, 2):

\begin{lemma}\label{lem:Ban}
 Let $\Eb$ be an Euclidean lattice of positive rank $n$, and let $x$ an element of $E_\R.$
 
 For any $r \in \R^+$ and any $t \in ]0,1],$ we have:
 \begin{equation}\label{Banana}
\sum_{v \in E, \Vert v -x \Vert \geq r} e^{-\pi \Vert v -x \Vert^2} \leq t^{-n/2} e^{-\pi (1-t) r^2} \sum_{v \in E} e^{-\pi \Vert v \Vert^2}.\end{equation}
\end{lemma}

\begin{proof} We have:
\begin{align}
 \sum_{v \in E, \Vert v -x \Vert \geq r} e^{-\pi \Vert v -x \Vert^2} & = \sum_{v \in E, \Vert v -x \Vert \geq r} e^{-\pi (1-t) \Vert v -x \Vert^2} e^{-\pi t \Vert v -x \Vert^2} \notag \\
 & \leq e^{-\pi (1-t) r^2} \sum_{v \in E, \Vert v -x \Vert \geq r}  e^{-\pi t \Vert v -x \Vert^2} \notag \\
 & \leq e^{-\pi (1-t) r^2} \sum_{v \in E}  e^{-\pi t \Vert v \Vert^2} \label{expl1}\\
 & \leq e^{-\pi (1-t) r^2} t^{-n/2} \sum_{v \in E}  e^{-\pi  \Vert v \Vert^2}. \label{expl2}
\end{align}
Indeed, the inequality (\ref{expl1}) follows from (\ref{PoissonGaussineq}), and (\ref{expl2}) from the inequality 
$$\log \theta_{\Eb}(t)  + \frac{n}{2} \log t  \leq \log \theta_{\Eb}(1).$$
\end{proof}

Clearly the inequality (\ref{Banana}) is significant only for values of $r$ such that
$$\inf_{t \in ]0,1]} t^{-n/2} e^{-\pi (1-t) r^2} \leq  1.$$
An elementary computation shows that this holds precisely when $r \geq \sqrt{\frac{n}{2\pi}}$, and that, when this holds, if we write
$$r =  \sqrt{\frac{n}{2\pi}} \tilde{r}$$
with $\tilde{r}$ in $[1, +\infty[$, then the minimum of $t^{-n/2} e^{-\pi (1-t) r^2}$ over $]0,1]$ is achieved for 
$t= t_{\rm min}:= \tilde{r}^{-2}$, and takes the value
$$ t_{\rm min}^{-n/2} e^{-\pi (1-t_{\rm min}) r^2} = [\tilde{r} e^{-(1/2) (\tilde{r}^2-1)}]^n.$$
It is straightforward that this minimum $[\tilde{r} e^{-(1/2) (\tilde{r}^2-1)}]^n$ is a decreasing function of $\tilde{r} \in [1, +\infty[$, which takes the value $1$ when $\tilde{r}=1.$

These elementary considerations show that Lemma \ref{lem:Ban} may be reformulated in the following version, better suited to applications:

\begin{proposition}\label{prop:Banamna}  Let $\Eb$ be an Euclidean lattice of positive rank $n$, and let $x$ an element of $E_\R.$ For every element $\tilde{r}$ in $[1, +\infty[$, if we let
$$r = \sqrt{\frac{n}{2\pi}}\tilde{r},
$$ then we have:
\begin{equation}\label{Banamna}
\sum_{v \in E, \Vert v -x \Vert \geq r} e^{-\pi \Vert v -x \Vert^2} \leq
 [\tilde{r} e^{-(1/2) (\tilde{r}^2-1)}]^n
 \sum_{v \in E} e^{-\pi \Vert v \Vert^2}. 
\end{equation}
\qed
\end{proposition}

\subsection{Application: first minimum and $\theta$-invariants}\label{bisrepetita}

The special case of (\ref{Banamna}) where $x=0$ has already the following non-trivial consequence:

\begin{corollary}
Let $\Eb$ be an Euclidean lattice of positive rank $n$, and of  first minimum $\lambda_1(\Eb) \geq \sqrt{n/2\pi}$. 

If we define $\tilde{\lambda}$ in $[1,+\infty[$ by  the relation:
$$\lambda_1(\Eb)= \sqrt{\frac{n}{2\pi}}\tilde{\lambda} ,$$
then the following inequality holds:
\begin{equation}\label{hotlambdaBan}
\hot(\Eb) \leq \log (1-[\tilde{\lambda} e^{-(1/2) (\tilde{\lambda}^2-1)}]^n)^{-1}.
\end{equation}
and
\begin{equation}\label{hotlambdaBanbis}
e^{\hot(\Eb)} - 1 \leq \frac{[\tilde{\lambda} e^{-(1/2) (\tilde{\lambda}^2-1)}]^n}{1-[\tilde{\lambda} e^{-(1/2) (\tilde{\lambda}^2-1)}]^n}.
\end{equation}
\end{corollary}

\begin{proof} By applying (\ref{Banamna}) to $x=0$ and $r=\lambda_1(\Eb),$ we get:
$$  \sum_{v \in E} e^{-\pi \Vert v \Vert^2} - 1 =\sum_{v \in E, \Vert v\Vert \geq \lambda_1(\Eb) } e^{-\pi \Vert v\Vert^2} \leq
 [\tilde{\lambda} e^{-(1/2) (\tilde{\lambda}^2-1)}]^n
 \sum_{v \in E} e^{-\pi \Vert v \Vert^2}. $$
 This establishes the inequality
 $$e^{\hot(\Eb)} - 1 \leq [\tilde{\lambda} e^{-(1/2) (\tilde{\lambda}^2-1)}]^n e^{\hot(\Eb)},$$
 which in turn is equivalent to (\ref{hotlambdaBan}) and (\ref{hotlambdaBanbis}).
 \end{proof}

We may compare the upper-bound  (\ref{hotlambdaBanbis}) on $\hot(\Eb)$ in terms of the first minimum $\lambda_1(\Eb)$, assumed to be $> (n/2\pi)^{1/2}$, obtained by means of Banaszczyk's methods, with the  upper-bound   derived in Proposition \ref{Grbd} by using  Groenewegen's argument, namely:
\begin{equation}\label{GrbdEncore}
e^{\hot(\Eb)} -1 \leq  3^n \left(1-\frac{n}{2\pi \lambda_1(\Eb)^2}\right)^{-1} e^{-\pi \lambda_1(\Eb)^2}.
\end{equation} 

To achieve this comparison, observe that
\begin{equation}
[\tilde{\lambda} e^{-(1/2) (\tilde{\lambda}^2-1)}]^n = \left(\frac{n}{2\pi e}\right)^{-n/2} \lambda_1(\Eb)^n e^{-\pi \lambda_1(\Eb)^2}.
\end{equation}
This shows that, when $\lambda_1(\Eb)$ goes to $+\infty$, Groenewegen's bound (\ref{GrbdEncore}) is better than (\ref{hotlambdaBanbis}) by a factor 
$$3^n \left(\frac{n}{2\pi e}\right)^{n/2} \lambda_1(\Eb)^{-n} = \left(\frac{3}{\sqrt{e}}\right)^n \tilde{\lambda}^{-n}.$$

Besides, for any fixed value of $n$, Groenewegen's bound is also better 
than Banascszyk's when $\lambda_1(\Eb)$ goes to $(n/2\pi)^{1/2}$. However, when $\lambda_1(\Eb)= \sqrt{n/2}$ or equivalently $\tilde{\lambda}=\sqrt{\pi},$ Banascszyk's bound improves on
Groenewegen's one by a factor 
$$(1-\pi^{-1})^{-1} \left(\frac{\pi e}{9}\right )^{n/2}$$
when $n$ goes to infinity. (Observe that $\pi e/9 =  0.9488... < 1$.)

\subsection{Covering radius and Banaszczyk's transference estimate}

At this stage, we can  easily recover   Banaszczyk's transference estimate   relating the first minimum and the covering radius of some Euclidean lattice and of its dual.
 
Recall that the \emph{covering radius}  of an Euclidean lattice $\Eb$ of positive rank is defined
as the positive real number
$$\rho(\Eb) := \max_{x \in E_\R} \min_{v \in E} \Vert v-x \Vert = \inf \left\{ r \in \R^+ \mid \bigcup_{v \in E} \Bint(v,r) = E_\R \right\}$$
and that Banaszczyk's transference estimate is the second estimate in the following proposition:
\begin{proposition}[\cite{Banaszczyk93}, Theorem 2.2] For any Euclidean lattice $\Eb$ of positive rank $n,$ we have:
$$1/2 \leq  \rho(\Eb)\, \lambda_1(\Eb^\vee) \leq n/2.$$
\end{proposition}

The first inequality $1/2 \leq  \rho(\Eb). \lambda_1(\Eb^\vee)$ is elementary\footnote{Simply consider a ``short vector" $\xi$ of $\Eb^\vee$, an element $v\in E_\R$ such that $\xi(x) = 1/2$ and $v\in E$ such that $\Vert v -x\Vert \leq \rho(\Eb)$, and observe that $ \rho(\Eb). \lambda_1(\Eb^\vee) \geq \Vert \xi \Vert . \Vert v -x\Vert \geq \vert \xi(v)-\xi(x) \vert \geq 1/2.$}. To establish the second one, we first derive another corollary of Proposition \ref{prop:Banamna}:

\begin{corollary}\label{corrho} Let $\Eb$ be an Euclidean lattice of positive rank $n$, and of  covering radius $\rho(\Eb) \geq  \sqrt{n/2\pi}$.

 If we define $\tilde{\rho}$ in $[1,+\infty[$ by  the relation
$$\rho(\Eb) =  \sqrt{\frac{n}{2\pi}}\tilde{\rho},$$
then there exists $x$ in $E_\R$ such that
\begin{equation}\label{boundlambda}
\frac{\sum_{v \in E} e^{-\pi \Vert v -x \Vert^2}}{\sum_{v \in E} e^{-\pi \Vert v \Vert^2}} \leq [\tilde{\rho} e^{-(\tilde{\rho}^2-1)/2}]^n.
\end{equation}

\begin{proof} By the very definition of $\rho(\Eb)$, there exists $x$ in $E_\R$ such that $\Vert v -x \Vert \geq \rho(\Eb)$ for every $v$ in $E.$
For this choice of $x$, (\ref{boundlambda}) follows from (\ref{Banamna}) applied with $r= \rho(\Eb)$.
\end{proof}

\end{corollary}

The following lemma is a straightforward reformulation of already established properties of $\hot.$  
\begin{lemma}\label{lemlambdavee}
Let $\Eb$ be an Euclidean lattice of positive rank $n$. 

1) For every $x$ in $E_\R$, we have:
\begin{equation}\label{xvhotdual}
\frac{\sum_{v \in E} e^{-\pi \Vert v -x \Vert^2}}{\sum_{v \in E} e^{-\pi \Vert v \Vert^2}} \geq 2 e^{-\hot(\Eb^\vee)} -1.
\end{equation} 

2) Assume  that the first minimum of the dual lattice $\Eb^\vee$ satisfies $\lambda_1(\Eb^\vee) \geq \sqrt{n/2\pi}$. 
If we define $\tilde{\lambda}^\vee$ in $[1,+\infty[$ by  the relation:
$$\lambda_1(\Eb^\vee)=  \sqrt{\frac{n}{2\pi}}\tilde{\lambda}^\vee,$$
then we have: 
\begin{equation}\label{hotduallambda}
2 e^{-\hot(\Eb^\vee)} -1\geq  1 -2 [\tilde{\lambda}^\vee e^{-(\tilde{\lambda}^{\vee 2}-1)/2}]^n.
\end{equation} 
\end{lemma}

\begin{proof} 1) According to (\ref{PoissonGaussineqmoins}),
$$\frac{\sum_{v \in E} e^{-\pi \Vert v -x \Vert^2}}{\sum_{v \in E} e^{-\pi \Vert v \Vert^2}} \geq \frac{2}{ (\covol
\Eb)\, \sum_{v \in E} e^{-\pi \Vert v \Vert^2} }-1.$$
We conclude the proof of (\ref{xvhotdual}) by using the Poisson formula (\ref{petitpoisson}) and the definition of $\hot(\Eb^\vee)$.

2) To prove (\ref{hotduallambda}), we just  apply the bound (\ref{hotlambdaBan}) to the dual lattice $\Eb^\vee$.
\end{proof}

From Corollary  \ref{corrho} and Lemma \ref{lemlambdavee}, we easily derive Banaszczyk's transference estimate in the following more precise form:

\begin{proposition}\label{propBanasprecise} Let $\psi: [1, +\infty[ \lrasim ]0, 1]$ be the decreasing homeorphism defined by
$\psi(t) := t e^{-(t^2-1)/2},$
and for every positive integer $n$, let $t_n:= \psi^{-1}(3^{-1/n}).$

Then we have:
\begin{equation}\label{tnasymp}
t_n = 1 + \sqrt{\log 3 / n} + O(1/n) \mbox{ when $n \lra +\infty.$}
\end{equation}
and, for every positive integer $n$: 
\begin{equation}\label{tnbound}
t_n \leq 1 + \sqrt{\log 3 / n}.
\end{equation}

Moreover, for any Euclidean lattice $\Eb$ of positive rank $n$, we have:
\begin{equation}\label{Banasprecise}
\rho(\Eb)\, \lambda_1(\Eb^\vee) \leq \frac{t_n^2 \, n}{2\pi}.
\end{equation}
\end{proposition}

Observe that, according to (\ref{tnbound}), for any $n\geq 3,$ we have:
$$t_n \leq 1 + \sqrt{(\log 3)/3} = 1.605...\leq \sqrt{\pi}= 1.772...$$

Consequently Proposition \ref{propBanasprecise} implies Banaszczyk's inequality $\rho(\Eb)\,\lambda_1(\Eb^\vee) \leq n/2$ when $n\geq 3$. This inequality is trivial when $n=1$. When $n=2,$ it follows from elementary considerations involving reduced bases\footnote{Indeed to establish Banaszczyk's inequality for Euclidean lattices of rank 2, it is enough to prove it for the Euclidean lattices $\Eb_\tau$, defined as $\Z + \tau \Z$ equipped with the usual complex absolute value $\vert. \vert,$ when $\tau$ is an element of the upper half-plane in the usual ``reduction domain" defined by $\tau \geq 1$ and $\vert \Re \tau \vert \leq 1/2$.  For such lattices, Banaszczyk's inequality takes the form: $\rho(\Eb_\tau^\vee) \leq  \Im \tau.$ This last estimate  follows from the observation that, for any $z \in \C$ such that $\vert \Im z \vert \leq \Im \tau/2$, if  we denote the integer  closest to $\Re z$ by $k,$ we have:  $\vert z - k \vert \leq \sqrt{1+ (\Im \tau)^2}/2 \leq \Im \tau.$} of $\Eb$ and $\Eb^\vee$.

\proof[Proof of Proposition \ref{propBanasprecise}.] We leave the derivation of  (\ref{tnasymp}) and (\ref{tnbound}) as elementary exercises.

Let consider  a Euclidean  lattice $\Eb$  of positive rank $n$ and let us define $\tilde{\rho}$ and $\tilde{\lambda}^\vee$ as in Corollary  \ref{corrho} and Lemma \ref{lemlambdavee}. From the estimates (\ref{boundlambda}), (\ref{xvhotdual}) and (\ref{hotduallambda}), it follows that 
$$\mbox{if $\tilde{\rho}$ and $\tilde{\lambda}^\vee$ are $\geq 1$, then }
\psi(\tilde{\rho})^n + 2 \psi(\tilde{\lambda}^\vee)^n \leq 1.$$

Consequently, if for some $t \in \R^\ast_+,$ we have $\tilde{\rho}=\tilde{\lambda}^\vee = t,$ then $3 \psi(t)^n \leq 1$ if $t\geq 1$, and therefore $t\leq t_n$ and 
$$\rho(\Eb) = \lambda_1(\Eb^\vee) \leq t_n \sqrt{n/2\pi}.$$

This establishes (\ref{Banasprecise}) when $\rho(\Eb) = \lambda_1(\Eb^\vee).$ 

To derive the general validity of (\ref{Banasprecise}) from this special case, simply observe that replacing the Euclidean lattice $\Eb$ by $\Eb \otimes \cOb(\delta)$ for some $\delta \in \R$ (that is, scaling the metric of $\Eb$ by a positive factor $e^{-\delta}$) does not change the product $\rho(\Eb)\, \lambda_1(\Eb^\vee)$ and that, by a suitable choice of $\delta$, the condition $$\rho(\Eb\otimes \cOb(\delta)) = \lambda_1((\Eb\otimes \cOb(\delta))^\vee)$$ may be achieved. 
Indeed, from the very definitions of the covering radius and of the first minimum, we obtain:
$$\rho(\Eb\otimes \cOb(\delta)) = e^{-\delta}  \rho(\Eb)$$ and $$\lambda_1((\Eb\otimes \cOb(\delta))^\vee)= e^\delta \lambda_1(\Eb^\vee).$$
\qed

\section{Subadditive invariants of Euclidean lattices}\label{remsubadd}
\subsection{Alternating inequalities}\label{alternating}
With the notation of Proposition \ref{ineqshorttheta}, one has the following series of alternating inequalities:
\begin{align}
&\hot(\Eb) \geq 0, \label{morse1}  \\
\label{morse2}
&\hot(\Eb) - \hot(\Fb) \leq 0, \\
\label{morse3}
&\hot(\Eb) -\hot(\Fb) + \hot(\Gb) \geq 0, \\
\label{morse4}
&\hot(\Eb) -\hot(\Fb) + \hot(\Gb) - \hut (\Eb) \leq 0, \\
\label{morse5}
&\hot(\Eb) -\hot(\Fb) + \hot(\Gb) - \hut (\Eb) + \hut(\Fb) \geq 0, \\
\label{morse6}
&\hot(\Eb) -\hot(\Fb) + \hot(\Gb) - \hut (\Eb) + \hut(\Fb) - \hut(\Gb) = 0.
\end{align}
These are precisely the inequalities we should obtain if the $h^k_\theta$'s were some dimensions of the spaces in a long exact cohomology sequence derived from the admissible short exact sequence
$$0 \lra \Eb \stackrel{i}{\lra} \Fb \stackrel{p}{\lra} \Gb \lra 0,$$
which would vanish in cohomological degree $k>1.$

Indeed, the inequalities (\ref{morse1})-(\ref{morse3}) have already been established. 
  As  already observed (see (\ref{hothutshort}) \emph{supra}), thanks to the Poisson-Riemann-Roch formula (\ref{PoissonZ}), the equality (\ref{morse6}) may be written
$$\dega \Eb - \dega \Fb + \dega \Gb = 0$$
and precisely expresses the additivity (\ref{degaddadm}) of the Arakelov degree in admissible short exact sequences.
Taking (\ref{morse6}) into account,  inequalities (\ref{morse4}) and (\ref{morse5}) are equivalent to the relations
$
 - \hut (\Fb) + \hut (\Gb) \leq 0
\mbox{
and }
\hut(\Gb) \geq 0,
$
which follow from Proposition \ref{ineqmortheta} applied to $\phi = p$ and from Lemma  \ref{hot positive}.

Observe that, using again the Poisson-Riemann-Roch formula (\ref{Poisson}), the inequality (\ref{morse4}) may be also written as
\begin{equation}\label{morse4'}
\hot(\Gb) \leq \hot(\Fb) - \dega \Eb + \frac{1}{2} \log \vert \Delta_K \vert . \rk E
\end{equation}

\subsection{Blichfeldt pairs}\label{Blich} Let us indicate that that the occurrence in geometry of numbers of  ``alternating inequalities", 
similar to the ones satisfied by dimensions of cohomology groups, has been observed by Gillet, Mazur, and Soul\'e (\cite{GilletMazurSoule1991}) in the context of the classical theorem of Blichfeldt.

In \emph{loc. cit.}, instead of Euclidean lattices $\Eb:=(E, \Vert.\Vert)$, defined by finitely generated free $\Z$-module $E$ and a Euclidean norm $\Vert. \Vert$ on $E_\R,$ the authors deal with so-called \emph{Blichfeldt pairs} $$\cE:=(E, \bB),$$ defined by a $\Z$-module $E$ as above and a bounded Lebesgue measurable subset $\bB$ in $E_\R$. They define 
$$h^0(\cE) := \log \max_{v \in E_\R} \vert E \cap (v+ \bB)\vert$$
and
$$h^1(\cE):= h^0(\cE) - \log \mu(\bB)$$ 
where $\mu$ is the Haar measure on $E_\R$ which gives a fundamental domain for $E$ measure equal to one. They define an \emph{exact sequence of Blichfeldt pairs}
$$0 \lra (E, \bB) \stackrel{i}{\lra} (F, \bC) \stackrel{p}{\lra} (G, \bD) \lra 0$$
as an exact sequence of $\Z$-modules
$$0 \lra E  \stackrel{i}{\lra} F \stackrel{p}{\lra} G \ra 0$$
  such that $p_\R(\bC)$ and $\bD$ (resp., for any $x \in \bC$, $p_\R^{-1}(p_\R(x))\cap \bC$ and $i_\R(\bB)$) coincide up to translation in $G_\R$ (resp., in $F_\R$). Then they
establish the validity of (\ref{morse1})-(\ref{morse6}), with $h^k$ instead of $h^k_\theta$, for any short exact sequence of Blichfeldt pairs as above.

It is also possible to define the direct sum of two Blichfeldt pairs $\cE_1:=(E_1, \bB_1)$ and $\cE_2:=(E_2, \bB_2)$ as 
$$\cE_1 \oplus \cE_2 := (E_1 \oplus E_2, \cB_1 \times \cB_2).$$
Then the following additivity relations are easily established:
$$h^k (\cE_1 \oplus \cE_2) = h^k(\cE_1) + h^k(\cE_2), \mbox{ for $k=0,1.$}$$ 

To any Euclidean lattice $\Eb= (E,\Vert .\Vert)$, we may attach a natural Blichfeldt pair, namely
$$\cE:= (E, \bB(\Eb_\R)),$$
where
$$\bB(\Eb_\R):= \{ v \in E_\R \mid \Vert v \Vert \leq 1 \}.$$ 
Observe that, with the notation of Section \ref{compho}, we have:
$$h^0(\cE) = h^0_{\rm Bl}(\Eb).$$

However  this construction of Blichfeldt pairs from Euclidean lattices is not compatible with short exact sequences or products of Euclidean lattices and Blichfeldt pairs. Actually, for any two Euclidean lattices $\Eb_1$ and $\Eb_2$ of positive ranks, the Blichfeldt pair associated to their direct sum $\Eb_1 \oplus \Eb_2$ cannot be expressed as the direct sum of two Blichfeldt pairs.

This lack of compatibility prevents one to associate 
invariants $h^0$ and $h^1$ to Euclidean lattices, so that they would satisfy the alternating inequalities (\ref{morse1})-(\ref{morse6}), by reducing to the construction in \cite{GilletMazurSoule1991}.

\subsection{Concerning the subadditivity of $\hon$.} At this point, it may be worth to 
 emphasize  that the subadditivity property (\ref{keyineq}) (or equivalently (\ref{morse4})) satisfied by $\hot$ does \emph{not} hold when $\hot$ is replaced by $\hon$ (contrary to what is claimed in  \cite{GilletSoule1991}, p. 356, Proposition 7, (i); see \cite{GilletSoule2009}).

As shown by the following proposition, counterexamples may be obtained with $\Eb$ a small perturbation of the hexagonal rank-two
lattice\footnote{Recall that $A_2$ is defined as the lattice $\Z + \Z e^{2\pi i/3}$ inside $\C$ (equipped with its usual absolute value), or equivalently as the lattice $\Z^2$ inside $\R^2$ equipped with the Euclidean norm $\Vert.\Vert_{A_2}$ defined by $\Vert(x,y)\Vert^2_{A_2} := x^2-xy+ y^2.$ } $A_2$:

\begin{proposition}\label{A2lambda}
 For any $\lambda \in ]0, 4[,$
  let $\Eb_\lambda$  be the Euclidean lattice defined by $E_\lambda= \Z^2$ inside $E_{\lambda, \R}= \R^2$ equipped with the Euclidean norm $\Vert.\Vert_\lambda$ such that
 \begin{equation}\label{deflambda}
 \Vert(x,y)\Vert^2_{\lambda} := \lambda(x^2-xy)+ y^2,
 \end{equation}  
 and let $\Fb_\lambda$ be the sub-Euclidean lattice of $\Eb_\lambda$ defined by the $\Z$-submodule $F_\lambda := \Z\oplus\{0\}$ of $\Z^2.$
 
 Then, for any $\lambda \in ]0,4[,$ we have:
 \begin{equation}\label{Flambda}
\hon(\Fb_\lambda) = 0 \Longleftrightarrow \lambda > 1,
\end{equation}
 \begin{equation}\label{EFlambda}
\hon(\Eb_\lambda/\Fb_\lambda) \leq \log 3  \Longleftrightarrow \lambda < 3,
\end{equation}
and
 \begin{equation}\label{log5}
\hon(\Eb_\lambda) \geq \log 5.
\end{equation}
\end{proposition}

Indeed, the Euclidean lattice $\Eb_1$ is nothing but the hexagonal lattice $A_2$, and (\ref{Flambda})--(\ref{log5}) show that, for any $\lambda \in ]1,3[,$
$$\hon(\Eb_\lambda) > \hon(\Fb_\lambda) + \hon(\Eb_\lambda/\Fb_\lambda).$$

\begin{proof}[Proof of Proposition \ref{A2lambda}] Observe that the definition (\ref{deflambda}) of $\Vert.\Vert_\lambda$ may also be written
\begin{equation}\label{deflambdabis}
\Vert(x,y)\Vert^2_{\lambda} = \lambda(x-y/2)^2+ (1-\lambda/4) y^2.
\end{equation}

The equivalence (\ref{Flambda}) follows from the fact that $\Fb_\lambda$ may be identified with the lattice $\Z$ inside $\R$ equipped with the norm $\Vert.\Vert$ such that $\Vert 1\Vert^2 = \lambda.$

To establish (\ref{EFlambda}), observe that, as shown by (\ref{deflambdabis}),  the orthogonal complement of $F_{\lambda,\R}$ in the Euclidean vector space $(E_{\lambda,\R}, \Vert.\Vert_\lambda)$ is the real line $\R(1/2, 1)$. Consequently,  the norm in $\Eb_\lambda/\Fb_\lambda$ of the class $[(0,1)]$ of $(0,1)$ is given by
$$\Vert [(0,1)]\Vert^2_{\Eb_\lambda/\Fb_\lambda}= \Vert (1/2,1) \Vert^2_{\Eb_\lambda}= 1- \lambda/4.$$
Therefore $\hon(\Eb_\lambda/\Fb_\lambda) \leq \log 3$ if and only if $\sqrt{1-\lambda/4} > 1/2$, that is, if and only if $\lambda < 3.$

Finally, the unit ball of $\Eb_\lambda$ always contains the five lattice points $(0,0), (0,1), (0,-1), (1,1),$ and $(-1,-1).$ This proves (\ref{log5}).
\end{proof}

\section{The asymptotic invariant $\hont(\Eb, t)$}\label{asymptoticho}

In a vein related to the discussion of subadditive invariants of Euclidean lattices in the previous section --- notably to the the discussion of Blichfeldt pairs in  \ref{Blich} --- it may be worth mentioning that $\hon$ satisfies a \emph{superadditivity} property, that turns out to lead to another interpretation of the $\theta$-invariants of Euclidean lattices and to relate $\hon$ and $\hot$ to the thermodynamic formalism.

\subsection{The invariants $\hon(\Eb, t)$ and $\hont(\Eb, t)$} To formulate the superadditivity of $\hon,$ it is convenient to introduce a simple generalization of this invariant. Namely, for any Euclidean lattice $\Eb =(E, \Vert.\Vert)$ and any positive real number $t,$ we let:
\begin{equation}\label{hon,t}
\begin{split}
 \hon(\Eb, t) & := \log \left\vert \left\{ v \in E \mid \Vert v \Vert^2 \leq t \right\} \right\vert \\
 & = \hon(\Eb \otimes \cOb((\log t)/2)).
\end{split}
\end{equation}

The following observation is straightforward:

\begin{lemma}\label{honsuper} For any two Euclidean lattices $\Eb_1= (E_1, \Vert.\Vert_1)$ and $\Eb_2 =(E_2,\Vert.\Vert_2)$, and any two positive real numbers $t_1$ and $t_2,$ we have:
\begin{equation}\label{eq:honsuper}
\hon(\Eb_1, t_1)  + \hon(\Eb_2, t_2) \leq \hon(\Eb_1\oplus \Eb_2, t_1+t_2).
\end{equation}
\end{lemma}
Indeed, it follows from the inclusion:
$$ \left\{ v_1 \in E_1 \mid \Vert v_1 \Vert_1^2 \leq t_1 \right\}
\times \left\{ v_2 \in E_2 \mid \Vert v_2 \Vert_2^2 \leq t_2 \right\} \subset 
\left\{ v \in E_1\oplus E_2 \mid \Vert v \Vert_{\Eb_1 \oplus \Eb_2}^2 \leq t_1 + t_2 \right\}.$$
\qed

In particular, for any Euclidean lattice $\Eb,$ the sequence 
$(\hon(\Eb^{\oplus n}, nt))_{ n \geq 1}$ is superadditive; namely, it satisfies, for every $(n_1,n_2) \in \N_{\geq 1}^2$: 
\begin{equation}\label{subhonn}
 \hon(\Eb^{\oplus n_1}, n_1t) + \hon(\Eb^{\oplus n_2}, n_2t) \leq \hon(\Eb^{\oplus (n_1 + n_2)}, (n_1+ n_2)t).
\end{equation}

Besides, this sequence growths at most linearly with $n$:

\begin{lemma}\label{On}
 For any Euclidean lattice $\Eb$, when $n$ goes to $+ \infty,$
 \begin{equation}\label{eq:On}\hon(\Eb^{\oplus n}, nt) = O(n).\end{equation}
\end{lemma}
\begin{proof} For any Euclidean lattice $\Vb := (V, \Vert .\Vert)$ and for any $(P, r) \in V_\R \times \R_+$, we shall  denote by $\Bint_\Vb(P, r)$ the open ball of center $P$ and radius $r$ in the normed vector space $(V_\R, \Vert. \Vert).$ 

We shall also denote by $v_n$ the volume of the $n$-dimensional unit ball. Recall that $$v_n= \pi^{n/2}/\Gamma((n/2) +1)$$ and that  consequently, when $n$ goes to $+\infty,$
\begin{equation}\label{vnasympt}
\log v_n = - (n/2) \log n + O(n).
\end{equation}

Let $\lambda$ be the first minimum of $\Eb.$ Observe that, for any two points $v$ and $w$ of $E^{\oplus n}$ :
$$ v \neq w \Longrightarrow \left(v + \Bint_\Eb\left(0, \lambda/2 \right)^n\right) \cap \left(w +\Bint_\Eb\left(0, \lambda/2\right)^n\right) = \emptyset.$$
(Compare with the proof of Lemma \ref{LemmaGroe}.)

Besides, if 
$\cE(n,t) := \left\{ v \in E^{\oplus n}\mid \Vert v \Vert^2_{\Eb^{\oplus n}} \leq nt\right\},$
then we have:
\begin{equation}\label{inclball}
\bigcup_{v \in \cE(n,t)} \left(v + \Bint_\Eb\left(0, \lambda/2 \right)^n\right)  \subset \bigcup_{v \in \cE(n,t)} \Bint_{\Eb^{\oplus n}}(v, \lambda \sqrt{n}/2) \subset \Bint_{\Eb^{\oplus n}}(0, \sqrt{n} (\sqrt{t} + \lambda/2)).
\end{equation}

If $e:= \rk E,$ we finally obtain, by considering the Lebesgue measures of the first and last sets in (\ref{inclball}):\begin{equation}\label{presque}
v_{e}^n (\lambda/2)^n \vert \cE(n,t) \vert  \leq v_{ne} \, [\sqrt{n} (\sqrt{t} + \lambda/2)]^{ne}.
\end{equation}

Finally, when $n$ goes to $+\infty$, from (\ref{presque}) and (\ref{vnasympt}), we obtain:
$$ \hon(\Eb^{\oplus n}, nt) = \log \vert \cE(n,t) \vert \leq ne \log \sqrt{n}  + \log v_{ne} + O(n) = O(n).$$
\end{proof}

Recall that, according to a well-known observation that goes back to Fekete \cite{Fekete23}, superadditive sequences of real numbers have a simple asymptotic behaviour:

\begin{lemma}\label{lemFekete} Let $(a_n)_{n \in \N_{\geq 1}}$ be a sequence of real numbers that is superadditive \emph{(that is, such that  
$a_{n_1 +n_2} \geq a_{n_1} + a_{n_2}$ for any $(n_1,n_2) \in \N_{\geq 1}^2$).}

 Then the sequence $(a_n/n)_{n \in \N_{\geq 1}}$ admits a limit in $]-\infty, +\infty]$. Moreover:
\begin{equation*}
\lim_{n \rightarrow +\infty} a_n/n = \sup_{n \in \N_{\geq 1}} a_n/n.
\end{equation*}
\end{lemma}
\qed

Together with the superadditivity property (\ref{subhonn}) and with the bound (\ref{eq:On}), Fekete's Lemma establishes the following

\begin{proposition}\label{defhont}
 For any Euclidean lattice $\Eb$ and any $t \in \R^\ast_+,$ the limit
 \begin{equation*}
\hont(\Eb, t) := \lim_{n \rightarrow + \infty} \frac{1}{n} \; \hon(\Eb^{\oplus n}, nt)
\end{equation*}
exists in $\R_+$. Moreover, we have:
 \begin{equation*}
\hont(\Eb, t) := \sup_{n \in \N_{\geq 1}} \frac{1}{n} \; \hon(\Eb^{\oplus n}, nt).
\end{equation*}
\end{proposition} \qed 

The  invariant $\hont(\Eb, t)$ defined in Proposition  \ref{defhont} is an ``asymptotic version" of the more naive invariant $\hon(\Eb, t)$. By its very definition, its satisfies, for every positive integer $k$,
$$\hont(\Eb^{\oplus k}, kt) = k\, \hont(\Eb, t),$$
and inherits the superaddivity property of $\hon$ stated in Lemma \ref{honsuper}; namely, with the notation of \emph{loc. cit.}, we have:
\begin{equation}\label{eq:hontsuper}
\hont(\Eb_1, t_1)  + \hont(\Eb_2, t_2) \leq \hont(\Eb_1\oplus \Eb_2, t_1+t_2).
\end{equation}

\subsection{The invariant $\hont(\Eb, t)$ and the Legendre transform of $\log \theta_\Eb$} The invariant $\hont(\Eb, t)$ attached to some Euclidean lattice $\Eb$, as a function of $t \in \Rpa$, turns out to be simply related to the  theta function $\theta_\Eb$ of $\Eb$, and consequently to enjoy various properties --- notably, it is  a real analytic function --- that are not obvious on its original definition.

To express this relation, recall that, provided $\Eb$ has positive rank, the function
$$\log \theta_\Eb : \Rpa \lra \Rpa$$
is a decreasing real analytic diffeomorphism, that moreover is \emph{convex}. Actually the function
$$U_\Eb := -(\log \theta_\Eb)'$$
satisfies, for every $\beta \in \Rpa$,
$$U_\Eb(\beta) =  \frac{\sum_{v \in E} \pi \Vert v\Vert^2 \, e^{-\beta \pi \Vert v \Vert^2}}{\sum_{v \in E} e^{-\beta\pi\Vert v \Vert^2}},$$
 and is easily seen to establish a decreasing real analytic diffeomorphism
\begin{equation*}
U_\Eb : \Rpa \lrasim \Rpa.
\end{equation*}

\begin{theorem}\label{thermolattice}
 For every Euclidean lattice $\Eb$ of positive rank, the function $\hont(\Eb, .)$ is real analytic, increasing, concave and surjective from $\Rpa$ to $\Rpa$.
 
Moreover, if we let:
\begin{equation}\label{SEb}
S_\Eb (x) := \hont(\Eb, x/\pi) \;\;\; (x \in \Rpa),
\end{equation}
then the functions $-S_\Eb(-.)$ and $\log \theta_\Eb$ are Legendre transforms of each other.

Namely, for every $x \in \Rpa,$
\begin{equation}\label{Leg-hont}
\hont(\Eb, x) = \inf_{\beta > 0} (\pi \beta x + \log \theta_\Eb(\beta))
\end{equation}
and, for every $\beta \in \Rpa,$
\begin{equation}\label{Leg-logtheta}
\log\theta_\Eb(\beta) = \sup_{x > 0} (\hont(\Eb, x) -\pi \beta x).
\end{equation}

Moreover the derivative $S'_\Eb$ establishes a real analytic decreasing diffeomorphism 
\begin{equation*}
S'_\Eb : \Rpa \lrasim \Rpa.
\end{equation*}
inverse of $U_\Eb$, and for every $x \in \Rpa,$ the infimum in the right-hand side of (\ref{Leg-hont}) is attained for a unique value $\beta,$ namely  for $$\beta = S'_\Eb(\pi x) = \pi^{-1} \hont(\Eb, .)'(x).$$ 

Dually, for every $\beta \in \Rpa,$ the supremum in the right-hand side of (\ref{Leg-logtheta}) is attained for a unique value of $x$, namely for $x=\pi^{-1} U_\Eb(\beta)$.
\end{theorem}

When $\Eb$ is the ``trivial'' Euclidean lattice of rank one $\cOb(0):= (\Z, \vert. \vert)$, Theorem \ref{thermolattice} may be deduced   from  results of Mazo and Odlyzko (\cite{MazoOdlyzko90}, Theorem 1). 

In  its general formulation above, Theorem \ref{thermolattice} is  a  consequence of the extension of Cram\'er's theory of large deviations presented in Appendix \ref{Append:LD}, concerning an arbitrary measure space $(\cE, \cT, \mu)$ and a non-negative measurable function $H$ on $\cE$.  Theorem \ref{thermolattice} gathers the results of Subsection \ref{AppendixScholium} (notably Theorem \ref{thermo} and Corollary \ref{gentildual}) in the situation where $$(\cE, \cT, \mu) =(E, \cP(E), \sum_{v \in E} \delta _v)$$ --- that is, the set $E$ underlying the Euclidean lattice $\Eb$ equipped with the counting measure --- and where
$$H:= \pi \Vert . \Vert^2.$$ 

Indeed, in this situation, the function $\Psi$ introduced in  Subsection \ref{AppendixScholium} is nothing else than $\log \theta_\Eb$ and the function $S$ is simply the function $S_\Eb$ defined by (\ref{SEb}).
One may also observe that, specialized to this situation, the arguments in Appendix \ref{Append:LD} actually provide an alternative proof of the finiteness of $\hont(\Eb,t)$, that  relies  on the properties of the theta functions of lattices and avoids the estimates in the proof of Lemma  \ref{On}.

\subsection{Euclidean lattices and thermodynamic formalism}

The above construction, of data  $(\cE, \cT, \mu)$ and $H$ of the type considered in Appendix \ref{Append:LD}  from Euclidean lattices, is clearly compatible with finite products: the data associated to the direct sum $\bigoplus_{i \in I} \Eb_i$ of a finite family $(\Eb_i)_{i \in I}$ of Euclidean lattices may be identified with the ``product", in the sense of Subsection \ref{prodtherm}, of the data associated to each of the $\Eb_i$.

This observation allows us to apply Proposition \ref{Prop:secondlaw} to analyze the invariants $\hont$ of a direct sum of Euclidean lattices. Notably, we immediately obtain the following more precise form of the superadditivity (\ref{eq:hontsuper}) of $\hont$:

\begin{proposition}\label{SecondLawLattice}
For any two Euclidean lattices of positive rank $\Eb_1$ and $\Eb_2$ and any $t \in \Rpa,$
\begin{equation}\label{equ:secondlaawlattice}
\hont(\Eb_1 \oplus \Eb_2, t) = \max_{\stackrel{t_1, t_2 >0}{t_1 +t_2 = t}} \left(\hont(\Eb_1, t_1) + \hont(\Eb_2, t_2)\right).
\end{equation}
Moreover the maximum is attained for a unique pair $(t_1, t_2)$, namely for $(\pi^{-1} U_{\Eb_1}(\beta), \pi^{-1} U_{\Eb_2}(\beta))$  where $\beta:= S'_{\Eb_1 \oplus \Eb_2}(\pi t).$ \qed
 \end{proposition}
 
 This proposition can be understood as an expression of the second law of thermodynamics in the context of  Euclidean lattices. 
 
 As a last element pleading for an interpretation of the $\theta$-invariants of Euclidean lattices in terms of  statistical thermodynamics, let us briefly translate to the framework of this section the results in the last subsection \ref{maxentrop} of Appendix \ref{Append:LD}.
 
 Let us consider denote an Euclidean lattice of positive rank $\Eb:= (E, \Vert. \Vert)$  and let us denote by
 $$\cC:= \{ p \in [0,1]^E \mid \sum_{e \in E} p(e) =1\}$$ the (compact convex) space of probability measures on $E$.

 We may consider the functions $\epsilon$ (``energy") and $I$ (``information theoretic entropy") from $\cC$ to $[0, +\infty]$ defined as follows:
 \begin{equation*}
\epsilon(p) := \sum_{e \in E} p(e)  \pi \Vert e \Vert^2
\end{equation*}
and
\begin{equation*}
I(p):= -\sum_{e \in E} p(e) \log p(e).
\end{equation*}

Then Proposition \ref{prop:maxentrop},  applied to 
$(\cE, \cT, \mu) =(E, \cP(E), \sum_{v \in E} \delta _v) \mbox{ and } H= \pi \Vert . \Vert^2,$ becomes the following statement:
\begin{proposition}\label{maxentropE} Let $u$ and $\beta$ be two positive real numbers such that 
$u = U_\Eb(\beta).$
For any $p$ in $\cC$ such that $\epsilon(p) = u,$ we have:
\begin{equation}\label{eq:maxentropE}
I(p) \leq S_\Eb(\beta).
\end{equation}

Moreover the equality is achieved in (\ref{eq:maxentropE}) for a unique  $p$ in $\epsilon^{-1}(u),$ namely for the measure $p_\beta$ defined by
$$p_\beta(e) := \theta_\Eb (\beta)^{-1} e^{-\pi \beta \Vert e \Vert^2}.
 $$ \qed
 \end{proposition}

\subsection{Duality and further comparison estimates}
In this subsection, we denote by $\Eb$ an Euclidean lattice of positive rank, and  we derive additional estimates comparing its invariants $\hon(\Eb,.)$, $\hont(\Eb,.)$ and $\hot(\Eb).$

Ultimately, these estimates will appear as consequences of (i) the ``thermodynamic" formalism of the previous paragraphs and (ii) the Poisson formula (\ref{logFE}), which relates the theta functions $\theta_\Eb$ and $\theta_{\Eb^\vee}$ of $\Eb$ and of its dual Euclidean lattice $\Eb^\vee$.

Indeed, from (\ref{logFE}), we immediately get:

\begin{proposition}
 For any $\beta \in \Rpa,$ we have:
 \begin{equation}
\log \theta_\Eb(\beta)  - \log \theta_{\Eb^\vee}(\beta^{-1}) = -(\rk E/2)\log \beta + \dega \Eb,
\end{equation}
\begin{equation}
\beta\,U_\Eb (\beta) + \beta^{-1} U_{\Eb^\vee}(\beta^{-1}) = \rk E /2,
\end{equation}
and
\begin{equation}\label{Ubound}
0\leq U_\Eb (\beta) \leq \rk E/(2\beta).
\end{equation}
\qed
\end{proposition}

The expression $\rk E/(2\beta)$ which appears in the upper-bound  on $U_\Eb(\beta)$ in (\ref{Ubound}) coincides with   the function $U(\beta)$ in the ``Maxwellian" situation discussed in Subsection \ref{Maxwell}. This upper-bound was also the key point behind the estimates \emph{à la Banaszczyk} derived in Lemma \ref{VarBanas}.

For any integer $n\geq 1,$ we let:
\begin{equation*}
C(n) := - \sup_{t >1}\;  [\log (1-t^{-1}) - (n/2) \log t]. 
\end{equation*}
One easily shows that
$$C(n) = \log (n/2) + (1+n/2) \log(1+2/n)$$
and that
$$1 \leq C(n) - \log (n/2) \leq (3/2) \log 3.$$

\begin{theorem}\label{comp3} For any $x\in \Rpa,$ we have:
\begin{equation}\label{comp31}
\log\theta_\Eb(\rk E/ (2\pi x)) \leq \hon(\Eb, x) + C(\rk E),
\end{equation}
\begin{equation}\label{comp32}
\log\theta_\Eb(\rk E/ (2\pi x)) \leq \hont(\Eb, x),
\end{equation}
and
\begin{equation}\label{comp33}
\hont(\Eb, x) \leq \log\theta_\Eb(\rk E/ (2\pi x)) + \rk E/2.
\end{equation}
\end{theorem}

The following consequence of Theorem \ref{comp3}, which involves only the ``elementary" invariants $\hon(\Eb, t)$ and $\hont(\Eb, t)$, seems worth being mentioned: 
\begin{corollary}
 For any $x\in \Rpa,$
 \begin{equation}
0 \leq \hont(\Eb, x) - \hon(\Eb, x) \leq C(\rk E) + \rk E/2.
\end{equation}
\qed
\end{corollary}

\begin{proof}[Proof of Theorem \ref{comp3}]

Let us start with a straightforward consequence of Lemma \ref{VarBanas} (after the change of notation: $x= r^2$ and $\beta = t$): 
\begin{lemma}\label{VarVarBanas}
 For any $(x, \beta) \in \R_+^{\ast 2}$ such that $\beta x > \rk E/(2\pi),$
 \begin{equation}
\hot(\Eb, x) \geq \log (1- \rk E/(2\pi \beta x)) + \log \theta_{\Eb}(\beta).
\end{equation} \qed
\end{lemma}
 From Lemma \ref{VarVarBanas}, we easily deduce:
\begin{lemma}
 For any $(x, \beta, \beta') \in \R_+^{\ast 3}$ such that $\beta x \geq \rk E/(2\pi)$ and $\beta' > \beta,$
 \begin{equation}\label{VarVarBanasBis}
\hot(\Eb, x) \geq \log (1- \rk E/(2\pi \beta' x))  - (\rk E/2) \log(\beta'/\beta) + \log \theta_{\Eb}(\beta).
\end{equation}
\end{lemma}

\begin{proof}
 Lemma \ref{VarVarBanas} implies that
 $$\hot(\Eb, x) \geq \log (1- \rk E/(2\pi \beta' x)) + \log \theta_{\Eb}(\beta').$$
 
 Besides, we have:
\begin{equation}\label{betabeta}
\log \theta_\Eb(\beta') \geq \log \theta_\Eb(\beta) - (\rk E/2) \log(\beta'/\beta).
\end{equation}
This follows for instance from the upper-bound in (\ref{Ubound}) on 
$U_\Eb(\beta) := - (\log \theta_\Eb)'(\beta).$
 \end{proof}
 
 To prove the inequality (\ref{comp31}), we apply Lemma \ref{VarVarBanasBis} with $\beta := \rk E/(2\pi x)$ and we observe that
 \begin{multline*}
 \sup_{\beta'\in ]\beta, +\infty[} [ \log (1- \rk E/(2\pi \beta' x))  - (\rk E/2) \log(\beta'/\beta)] \\=
 \sup_{\beta'\in ]\beta, +\infty[} [ \log (1- (\beta'/\beta)^{-1})  - (\rk E/2) \log(\beta'/\beta)]
 = - C(\rk E).
 \end{multline*}

For any positive integer $n,$ from (\ref{comp31}) applied to $E^{\oplus n}$ and to $nx$ instead of $\Eb$ and $x$, we get:
$$
n\log\theta_\Eb(\rk E/ (2\pi x)) \leq \hont(\Eb^{\oplus n}, nx) + C(n\, \rk E).$$
Multiplying by $1/n$ and letting $n$ go to $+\infty,$ we obtain (\ref{comp32}), since $C(n\, \rk E) = o(n).$

An alternative proof of (\ref{comp32}) consists in observing that, as an easy consequence of (\ref{betabeta}) (which holds when $\beta' \geq \beta$), the following inequality holds for any $(\beta, \beta') \in \R_+^{\ast 2}$:
$$(\rk E/2) (\beta'/\beta) + \log \theta_\Eb(\beta') \geq \log \theta_\Eb(\beta).$$ Then (\ref{comp32}) follows from the expression (\ref{Leg-hont}) of $\hont(\Eb,.)$ in terms of the Legendre transform of $\log \theta_\Eb$.

Finally (\ref{comp31}) follows from (\ref{Leg-hont}) or (\ref{Leg-logtheta}), by chosing $x$ and $\beta$ related by $\beta x = \rk E/(2\pi).$
  \end{proof}
  
The estimates in Theorem \ref{comp3} show that the expressions $\hon(\Eb, \rk E/2\pi)$ and $\hont(\Eb, \rk E/2\pi)$ satisfy:
\begin{equation}\label{compthetamoinsnaive1}
-C(\rk E) \leq \hon(\Eb, \rk E/ 2 \pi) - \hot(\Eb)  \leq \rk E/2
\end{equation}
and 
\begin{equation}\label{compthetamoinsnaive2}
0 \leq   \hont(\Eb, \rk E/ 2 \pi) - \hot(\Eb)\leq \rk E/2.
\end{equation}

These comparison estimates, relating Arakelov and $\theta$-invariants of Euclidean lattices, should be compared with the comparison estimate (\ref{compthetanaive}) in Section \ref{compho}. The error term, of the order of $(1/2) \rk E. \log \rk E$  in (\ref{compthetanaive}), is replaced by  $\rk E/2$ in (\ref{compthetamoinsnaive1}) and (\ref{compthetamoinsnaive2}). (These error terms will be shown to be basically optimal when $\rk E$ goes to $+\infty$ in the next section; see paragraph \ref{ConstComp} \emph{infra}.)

These remarks plead for  considering the positive real numbers  $$\hon(\Eb, \rk E/2\pi) \mbox{ and } \hont(\Eb, \rk E/2\pi)$$ attached to some Euclidean lattice of positive rank $\Eb$  as variants of $$\hon(\Eb) = \hon(\Eb, 1)$$ that are ``better behaved" than $\hon(\Eb)$ itself.

\section{Some consequences of Siegel's mean value theorem}

\subsection{Siegel's mean value theorem over  $SL_n(\R)/SL_n(\Z)$ and over $\cL(n, \delta)$}

In this paragraph, we denote by $n$ an integer $\geq 2$.

The Lie group $SL_n(\R)$ is unimodular and its discrete subgroup $SL_n(\Z)$ has a finite covolume. We shall denote by $\mu_n$ the Haar measure on $SL_n(\R)$ which satisfies the following normalization condition: \emph{the measure induced by $\mu_n$ on the quotient $SL_n(\R)/SL_n(\Z)$} --- that we shall still denote by $\mu_n$ --- \emph{is a probability measure}. In other words, 
\begin{equation*}
\int_{SL_n(\R)/SL_n(\Z)} d\mu_n = 1.
\end{equation*}

For any Borel function 
$\phi : \R^n \lra [0, +\infty]$
and any $g \in SL_n(\R),$ we may consider the sum
$$\Sigma(\phi) (g) := \sum_{v \in \Z^n\setminus\{0\}} \phi(g.v).$$
Clearly, for any $\gamma \in SL_n(\Z),$ we have 
$$\Sigma(\phi)(g. \gamma) = \Sigma(\phi)(g)$$
and the function
$$\Sigma(\phi) : SL_n(\R)/SL_n(\Z) \lra [0, +\infty]$$
so defined is a Borel function.

In its most basic form, Siegel's mean value theorem is the following statement (\cite{Siegel45}; see also \cite{Weil46} and \cite{MacbeathRogers58} for other derivations, and \cite{Weil82}, Chapter III, for a ``modern" presentation).

\begin{theorem}\label{SMVFOr} For any Borel function $\phi : \R^n \lra [0, +\infty]$ as above, the following equality holds:
\begin{equation}\label{eq:SMVFOr}
\int_{SL_n(\R)/SL_n(\Z)} \Sigma(\phi)(g)\,  d\mu_n(g) = \int_{\R^n} \phi(v) \,d\lambda_n(v), 
\end{equation}
where we denote by $\lambda_n$ the Lebesgue measure on $\R^n.$ \qed
\end{theorem}

 For any $g \in SL_n(\R)$ and any $\delta \in \R$, the lattice $e^{-\delta/n} g(\Z^n)$ in $\R^n$ equipped with  the standard Euclidean norm $\Vert.\Vert_n$  
 (defined by $\Vert(x_1, \ldots, x_n)\Vert_n^2= x_1^2 + \dots + x_n^2$) becomes an Euclidean lattice 
 $( e^{-\delta/n} g(\Z^n), \Vert.\Vert_n)$ of covolume $e^{-\delta}$, or equivalently, of Arakelov degree $\delta$.
 
 We shall denote by $\cL(n,\delta)$ the set of isomorphism classes of Euclidean lattices of rank $n$ and Arakleov degree $\delta$. This set may be endowed with a natural locally compact topology (actually, with a structure of ``orbifold") by means of the identification of the set 
 $$\cL(n) := \coprod_{\delta \in \R} \cL(n,\delta)$$
 of isomorphism classes of Euclidean lattices of rank $n$ with the double coset space $$O_n(\R)\setminus GL_n(\R)/GL_n(\Z).$$
 In concrete terms, the natural topology and Borel structures on $\cL(n,\delta)$ are the quotients of the ones of $SL_n(\R)/SL_n(\Z)$ by the surjective map
 \begin{equation*}
\begin{array}{rrcl}
 \pi_{n,\delta}: & SL_n(\R)/SL_n(\Z) & \lra  & \cL_{n,\delta}   \\
& [g] & \longmapsto  & [( e^{-\delta/n} g(\Z^n), \Vert.\Vert_n)].  
\end{array}
\end{equation*}

We shall denote by
\begin{equation*}
\mu_{n,\delta}:= \pi_{n,\delta \ast} \mu_n
\end{equation*}
the Borel  measure on  $\cL_{n,\delta}$ deduced from the measure $\mu_n$ on $SL_n(\R)/SL_n(\Z)$ by the parametrization $\pi_{n,\delta}$ of $\cL_{n,\delta}$. Like $\mu_n,$ it is a probability measure:
\begin{equation*}
\int_{\cL(n,\delta)} d\mu_{n, \delta} = 1.
\end{equation*}

Applied to a radial function $\phi: \R^n \lra [0,+\infty],$ Siegel's mean value formula (\ref{eq:SMVFOr}) ``descends" through $\pi_{n,\delta}$.

Namely, to any Borel function   
$$\rho : \Rpa \lra [0, +\infty]$$
we may attach the Borel function 
$$\Sigma_{n}(\rho): \cL(n) \lra [0, +\infty]$$
which maps the isomorphism class of some Euclidean lattice $\Eb := (E, \Vert .\Vert)$ of rank $n$ to 
$$\Sigma_n (\rho)(\Eb) :=\sum_{v \in E\setminus\{0\} } \rho(\Vert v \Vert).$$ 

For any $\delta \in \R$ and any $g\in SL_n(\R),$ we have:
\begin{equation*}
\Sigma_n (\rho)(\pi_{n,\delta}([g])) = \sum_{w\in \Z^k\setminus \{0\}} \rho(e^{- \delta/n} \Vert g.w \Vert) = \Sigma(\phi_\delta)([g])
\end{equation*}
where $\phi_\delta$ is the function from $\R^n$ to $[0, +\infty]$ defined by:
$$\phi_\delta(v) := \rho( e^{-\delta/n} \Vert v\Vert).$$ 

Clearly, we have
$$\int_{\R^n} \phi_{\delta}(v) \, d\lambda_n(v) = e^\delta \int_{\R^n} \rho(\Vert v \Vert) \, d\lambda_n(v),$$
and Siegel's mean value formula (\ref{eq:SMVFOr}) applied to $\phi_{\delta}$ becomes:
\begin{theorem}\label{SMVFL} For any $(n,\delta) \in \N_{\geq 2} \times \R$ and for any Borel function $\rho : \R^+ \lra [0, +\infty]$, the following equality holds:
\begin{equation}\label{SW2}
\int_{\cL(n,\delta)} \Sigma_n(\rho) \, d\mu_{n,\delta} = e^\delta \int_{\R^n} \rho(\Vert v \Vert_n) \, d\lambda_n(v).
\end{equation} \qed 
\end{theorem}

The last integral may also be written
$$\int_{\R^n} \rho(\Vert v \Vert_n) \, d\lambda_n(v) = n v_n \int_0^{+\infty} \rho(r) r^{n-1}\, dr $$
where, as previously in this monograph, $v_n$ denotes the volume of the $n$-dimensional ball:
$$v_n := \lambda_n(\{ v \in \R^n \mid \Vert v \Vert_n < 1 \}) = \frac{\pi^{n/2}}{\Gamma(1 + n/2)}.$$

Theorems \ref{SMVFOr} and \ref{SMVFL} are classically used to establish the existence of Euclidean lattices satisfying suitable conditions --- for instance, of lattices of large enough density --- by ``probabilistic arguments", based on the observation that a positive measurable function on some probability space assume values greater or equal to its mean value on some subset of positive measure. 
We refer the reader to \cite{Siegel45} for a concise discussion of the existence of ``dense lattices" as a consequence of Theorem \ref{SMVFOr}  and for references to related earlier work of Minkowski and Hlawka.  

For later reference, we state a formal version of the above observation as the following lemma: 
\begin{lemma}\label{probamethod}
 Let $\phi$ and $\psi$ be two Borel functions from $\cL(n,\delta)$ to $[0,+\infty]$.
 
1)  If the integrals 
 $$I_\phi:= \int_{\cL(n,\delta)} \phi(x) \, d\mu_{n,\delta}(x) \mbox{ and } I_\psi:= \int_{\cL(n,\delta)} \psi(x) \, d\mu_{n,\delta}(x)$$
 are finite and positive, then the Borel subsets
 $$\cE_{\leq} := \{ x \in \cL(n,\delta) \mid  \phi(x)/I_\phi \leq \psi(x)/I_\psi \}$$
 and
  $$\cE_{\geq} := \{ x \in \cL(n,\delta) \mid  \phi(x)/I_\phi \geq \psi(x)/I_\psi \}$$
  have positive $\mu_{n,\delta}$-measures, and therefore are non-empty.
  
  
2)  In particular, 
   if the integral $I_\phi$ is finite, there exists $x$ in $\cL(n,\delta)$ such that $$\phi(x) \leq I_\phi.$$
   If moreover $\phi$ is continous and non constant, then $I_\phi$ is positive and the image $\phi(\cL(n,\delta))$   contains an open neighborhood of $I_\phi$ in $\Rpa$.
    \qed
\end{lemma}

\subsection{Applications to $\hon$ and $\hot$} Let us start by recovering a simple version of the classical results of Minkowski-Hlawka-Siegel alluded to above. We will express it in terms of the invariant $\hon(.,t)$, in a form convenient for later references and for comparison with similar results concerning the invariant $\hot$.

According to the very definition of $\hon(.,t)$, 
we have
$$e^{\hon(\Eb, t)} -1 = \left\vert\{v \in E\setminus\{0\} \mid \Vert v \Vert^2 \leq t \} \right\vert = \sum_{v E\setminus\{0\}} {\bm 1}_{[0, t^{1/2}]} (\Vert v\Vert) = \Sigma_n({\bm 1}_{[0, t^{1/2}]})([\Eb]).$$
Besides, for $\rho = {\bm 1}_{[0, t^{1/2}]},$ the computation of the integral in the right-hand side of Siegel's mean value formula (\ref{SW2}) is straightforward --- indeed,
$$\int_{\R^n} {\bm 1}_{[0, t^{1/2}]}(\Vert v \Vert) d\lambda_n(v) = v_n. t^{n/2}$$
--- and formula (\ref{SW2}) takes the following form:
\begin{proposition} For any $(n,\delta)$ in $\N_{\geq 2} \times \R$ and any $t\in \Rpa$, the following relations hold:
\begin{equation}\label{Averagehon}
\int_{[\Eb] \in \cL(n,\delta)} e^{\hon(\Eb, t)} d\mu_{n,\delta}([\Eb]) = 1 + v_n t^{n/2}\, e^\delta.
\end{equation}
\end{proposition}
 
 In particular, when $t=1,$ we obtain:
 \begin{equation*}
\int_{[\Eb] \in \cL(n,\delta)} e^{\hon(\Eb)} d\mu_{n,\delta}([\Eb]) = 1 + v_n \, e^\delta.
\end{equation*}

Observe that $e^{\hon(.,t)}-1$ takes its values in $\N$ and therefore vanishes where it is $<1.$ Therefore,  if we apply Lemma \ref{probamethod}, part 2), to the function $\phi:= e^{\hon(.,t)}-1$, we obtain that, for any $(n,\delta)$ in $\N_{\geq 2} \times \R$ and any $t\in \Rpa$ such that 
 $$ v_n t^{n/2}\, e^\delta < 1,$$
 there exists some Euclidean lattice $\Eb$ of rank $n$ and Arakelov degree $\delta$ such that
 $\hon(\Eb, t) = 0$, or equivalently, such that $\lambda_1(\Eb) > t^{1/2}.$ 
 
 In other words, we have established the following variant of a classical result of Minkowski:
 \begin{corollary}
 For any $(n,\delta)$ in $\N_{\geq 2} \times \R$,  we have:
 \begin{equation}\label{MinkHlawka}
\sup_{[\Eb] \in \cL(n,\delta)} \lambda_1(\Eb) \geq e^{-\delta /n} \, v_n^{-1/n}.
\end{equation}
\end{corollary}
This estimates has to be compared with the upper bound 
 \begin{equation}
\sup_{[\Eb] \in \cL(n,\delta)} \lambda_1(\Eb) \leq  2 e^{-\delta /n} \, v_n^{-1/n}
\end{equation}
that follows from the so-called ``Minkowski First Theorem". 

Let us also recall that, when $n$ goes to infinity, 
$$v_n^{-1/n} \sim \sqrt{n/(2\pi e)}$$ 
and that the positive real number $$\gamma_n := \sup_{[\Eb] \in \cL(n,0)} \lambda_1(\Eb)^2$$ is classically known as the \emph{Hermite constant} in dimension $n$. Thus the lower-bound (\ref{MinkHlawka}), when expressed in terms of Hermite constants, takes the following asymptotic form:
$$\liminf_{n \ra +\infty} \gamma_n / n \geq 1/(2\pi e),$$
well-known in the study of sphere packings (see \cite{ConwaySloane1993}, notably Chapter 1, for additional informations and references).

To compute the average value on $\cL(n,\delta)$ of the $\theta$-invariants, we apply Siegel's mean value formula (\ref{SW2}) to the Gaussian function
$$\rho( x) := e^{- \pi x^2}.$$
For this choice of $\rho,$ the integral in the right-hand side of (\ref{SW2}) is simply:
$$\int_{\R^n} e^{-\pi \Vert v \Vert^2} d\lambda_n(v) = 1,$$
and Siegel's mean value formula takes the following form:

\begin{proposition} For any $(n,\delta)$ in $\N_{\geq 2} \times \R$, 
\begin{equation}\label{Averagehot}
\int_{[\Eb] \in \cL(n,\delta)} e^{\hot(\Eb)} d\mu_{n,\delta}([\Eb]) = 1 + e^\delta
\end{equation} \qed
\end{proposition}

This expression for the mean value of $e^{\hot(\Eb)}$ has to be compared with the lower bound 
$$\hot(\Eb) \geq \delta$$
valid over $\cL(n,\delta)$ (see (\ref{ThetaRI})).

The function on $\cL(n,\delta)$ defined by $\hot$ is clearly continuous. Moreover it is non-constant (this follows for instance from its expression for direct sums of rank-one Euclidean lattices in Proposition \ref{sumline}). Therefore we may apply the last assertion of Lemma \ref{probamethod} to the function $e^{\hot}$ and we obtain, from the value of its integral computed in (\ref{Averagehot}):
\begin{corollary}\label{smallhot}
For any $(n,\delta)$ in $\N_{\geq 2} \times \R$, there exists an Euclidean lattice $\Eb$ of rank $n$ and degree $\delta$ such that
\begin{equation}\label{Ezerodelta}
\hot(\Eb) < \log (1 + e^\delta).
\end{equation} 
\qed
 \end{corollary}

According to the Poisson-Riemann-Roch formula, for any $[\Eb]$ in $\cL(n,\delta),$ we have:
$$\hot(\Eb) - \hut(\Eb) = \delta.$$
Therefore the equality (\ref{Averagehot}) may be also written:
\begin{equation}\label{Averagehut}
\int_{[\Eb] \in \cL(n,\delta)} e^{\hut(\Eb)} d\mu_{n,\delta}([\Eb]) = 1 + e^{-\delta},
\end{equation}
and the condition (\ref{Ezerodelta}) is equivalent to:
$$\hut(\Eb) < \log (1 + e^{-\delta}).$$

To put the conclusion of Corollary \ref{smallhot} in perspective,  we may consider the ``obvious" Euclidean lattice of rank $n$ and Arakelov degree $\delta$, namely $\cOb(\delta/n)^{\oplus n}$, for any $(n,\delta) \in \N_{\geq 1} \times \R$. If we define, for every $t \in \Rpa$,
$$\theta(t) :=\theta_{\cOb(0)} (t) = \sum_{k \in \Z} e^{- \pi k^2 t},$$
its $\theta$-invariant is:
$$\hot(\cOb(\delta/n)^{\oplus n}) = n \log \theta(e^{-2\delta/n}).$$
When $\delta$ is fixed and $n$ goes to infinity, this expression is equivalent to 
$$n \log\theta(0) =n \hot(\cOb) = n \eta.$$
This demonstrates that the existence of  a (class of) Euclidean lattice in $\cL(n,\delta)$ satisfying (\ref{Ezerodelta}) is not ``obvious" when $n$ is large.

We may also apply the first part of Lemma \ref{probamethod} to the functions $e^{\hon(. t)}$ and $e^{\hot}$. Taking into account the expressions (\ref{Averagehon}) and (\ref{Averagehot}) for their integrals, we obtain:

\begin{corollary}\label{comparehonhot} For any $(n, \delta, t) \in \N_{\geq 2} \times \R \times \Rpa,$ there exist Euclidean lattices $\Eb_+$ and $\Eb_-$, of rank $n$ and Arakelov degree $\delta$, such that
\begin{equation*}
\hon(\Eb_+, t) - \hot(\Eb_+) \geq \log \frac{1 + v_n t^{n/2} e^\delta}{1+ e^\delta}
\end{equation*}
and
\begin{equation*}
\hon(\Eb_-, t) - \hot(\Eb_-) \leq \log \frac{1 + v_n t^{n/2} e^\delta}{1+ e^\delta}.
\end{equation*} \qed
 \end{corollary}
 
\subsection{Constants in comparison estimates}\label{ConstComp} From Corollary \ref{comparehonhot}, one easily derives that the additive constants in diverse estimates relating the invariants $\hon(\Eb,t)$ and $\hot(\Eb)$ 
established in the previous sections are ``of  the correct order of growth" when the rank of the Euclidean lattice $\Eb$ goes to $+\infty.$ 

For instance, consider the first inequality in (\ref{compthetanaive}). It asserts that, for any Euclidean lattice $\Eb$ of rank $n\geq 1,$ 
\begin{equation}\label{AOptim1}
\hon(\Eb) - \hot(\Eb) \geq -(n/2). \log n + \log (1-1/2\pi).
\end{equation}

According to Corollary \ref{comparehonhot} applied with $t=1$ and $\delta =0$, for every $n \in \N_{\geq 2},$ there exists an Euclidean lattice of rank $n$ such that
$$\covol (\Eb_-) =1$$
and 
$$\hon(\Eb_-) - \hot(\Eb_-) \leq \log \frac{1 + v_n}{2}.$$
 Besides, when $n$ goes to $+ \infty,$ 
 $$\log \frac{1 + v_n}{2} =  -(n/2). \log n + O(n).$$
 
 This shows that the ``best constant" in the right-hand side of  (\ref{AOptim1}) --- even if one considers Euclidean lattices of covolume $1$ only --- is equivalent to $-(n/2). \log n $ when $n$ goes to $+\infty$.
 
 Consider now the estimates, valid for any Euclidean lattice $\Eb$ of rank $n \geq 1,$
 \begin{equation}\label{AOptim2}
\hon(\Eb, n/2 \pi) -\hot(\Eb)  \leq \hont(\Eb, n/2 \pi) -\hot(\Eb) \leq n/2,
\end{equation}
already considered in (\ref{compthetamoinsnaive1}) and (\ref{compthetamoinsnaive2}). (These estimates follow from the definition of $\hont$ and from (\ref{comp33}).)

According to Corollary \ref{comparehonhot} applied with $t=n/2\pi$, for any $n \in \N_{\geq 2}$ and any $\delta \in \R,$ there exists an Euclidean lattice $\Eb_+$ of rank $n$ and Arakelov degree $\delta$ such that
$$ 
\hon(\Eb_+, n/2\pi) - \hot(\Eb_+) \geq \log \frac{1 + v_n (n/2\pi)^{n/2} e^\delta}{1+ e^\delta}.$$
Besides, 
$$\lim_{\delta \ra + \infty} \log \frac{1 + v_n (n/2\pi)^{n/2} e^\delta}{1+ e^\delta} = \log [v_n (n/2\pi)^{n/2}],$$
and, when $n$ goes to $+\infty,$
$$\log [v_n (n/2\pi)^{n/2}] = n/2 + O(\log n).$$

This shows notably that the ``best constant" in the right-hand side of (\ref{AOptim2}) is equivalent to $n/2$ when $n$ goes to $+\infty$.

\medskip

\chapter[Countably generated projective modules over Dedekind rings]{Countably generated projective modules and linearly compact Tate spaces over Dedekind rings}\label{CpCtc}

\medskip

In this chapter, we denote by $A$ a Dedekind ring (in the sense of Bourbaki, \cite{BourbakiAC7}, VII.2.1; in other words, $A$ is either a field, or a Noetherian integrally closed domain of dimension 1) and  we introduce some categories of (topological) modules $\CP_A$ and $\CTC_A$ attached to $A$. When the Dedekind ring $A$ is the ring    
$\OK$ of integers   in some number field $K,$ the modules in these categories will occur in the following chapters as the $\OK$-modules underlying the ``infinite dimensional Hermitian vector bundles" over $\Spec \OK$ investigated in this monograph.

The objects in the dual categories $\CP_A$ and $\CTC_A$ are easily described. Namely, an object of $\CP_A$ is an $A$-module which is, either finitely generated and projective, or isomorphic to $A^{(\N)}$. An object of $\CTC_A$ is a topological $A$-module\footnote{By a \emph{topological $A$-module}, we  mean a topological $A$-module over the ring $A$ equipped with the discrete topology.} which is, either a finitely generated projective $A$-module equipped with the discrete topology, or isomorphic to $A^\N$ equipped with the prodcut of the discrete topology on every factor $A$.

Handling the morphisms in these categories requires more care. The \emph{strict} morphisms in $\CTC_A$  play an especially important role, as shown in Section \ref{StrictCTC}, and diverse ``pathologies" concerning the morphisms in $\CP_A$ and $\CTC_A$ occur naturally, as demonstrated by the examples in Section \ref{Path}.

\section{Countably generated projective $A$-modules}

\subsection{The category $\CP_A$}\label{CatCP}

The following proposition is a simple  consequence of the fact that, over a Dedekind ring, a finitely generated module is projective when it is torsion free.

\begin{proposition}\label{CPAdef} For any $A$-module $M$, the following conditions are equivalent:

(1) The $A$-module $M$ is countably generated and projective.

(2) The $A$-module $M$ is isomorphic to a direct summand of $A^{(\N)}.$

(3) The $A$-module $M$ is isomorphic to some $A$-submodule of $A^{(\N)}$.

(4) There exists a family $(M_i)_{i \in \N}$ of $A$-submodules of $M$ such that:
\begin{enumerate}
\item[(i)] for any $i \in \N,$ $M_i$ is a finitely generated torsion free $A$-module;
\item[(ii)] for any $i \in \N,$ $M_i$ is a saturated $A$-submodule of $M_{i+1};$
\item[(iii)] $M = \bigcup_{i \in \N} M_i.$
\end{enumerate}

(5) The $A$-module $M$ is a countable direct sum of finitely generated projective $A$-modules. 

\end{proposition}

\begin{proof}[Proof of Proposition \ref{CPAdef}] The implications
$(5) \Rightarrow (1) \Rightarrow (2) \Rightarrow (3)$ are clear.

When $(3)$ holds, we may consider the filtration $(N_i)_{i \in \N}$ of $N:= A^{(\N)}$ defined by
$$N_i := \{ (a_k)_{k\in \N} \in A^{(\N)} \mid \forall k \in \N_{\geq i}, a_k =0 \}.$$
Then the filtration $(M_i)_{i \in \N}$ of $M$ defined by $M_i := M \cap N_i$ satisfies $(4)$. 

When $(4)$ holds, for every $i \in \N,$ the quotient $M_{i+1}/M_i$ is a  finitely generated $A$-module, which is torsion free, hence projective. Therefore the short exact sequence  of $A$-modules
$$0 \lra M_i \lra M_{i+1} \lra M_{i+1}/M_i \lra 0$$
is split, and there exists a (necessarily finitely generated and projective) $A$-submodule $P_i$ of $M_{i+1}$ such that $M_{i+1} = M_i \oplus P_i.$ Then we obtain the following decomposition of $M$:
$$M = M_0 \oplus \bigoplus_{i \in \N} P_i.$$
This displays $M$ as a countable direct sum of finitely generated projective $A$-modules. 
\end{proof}

We define the $A$-linear category $\CP_A$ of  \emph{{\bf c}ountably generated {\bf p}rojective $A$-modules}  as the category whose objects are $A$-modules satisfying the equivalent conditions in Proposition \ref{CPAdef}, and whose morphisms are $A$-linear maps.

For any object $M$ of $\CP_A$, we denote by $\cF(M)$ the family of finitely generated $A$-submodules of $M$, and by $\cFS(M)$ the family of saturated finitely generated $A$-submodules of $M$. We shall also denote by ${\rm co}\cF(M)$ (resp. by ${\rm co}\cF\cS(M)$) the family of $A$-submodules $M'$ of $M$ such that $M/M'$ is a finitely generated $A$-module (resp. a finitely generated torsion free $A$-module).

Observe that, according to the equivalences of conditions (1) and (3) in Proposition \ref{CPAdef}, any $A$-submodule $M'$ of an object $M$ of $\CP_A$ is again an object of $\CP_A.$ However, the quotient $A$-module $M/M'$ --- even if assumed torsion-free --- is not always an object in $\CP_A$ (see for instance, when $A=\Z,$ the constructions in Proposition \ref{padicex}, in paragraph \ref{Path} \emph{infra}, notably the short exact sequence \ref{sesalpha1}). 

The $A$-linear category $\CP_A$ admits obvious finite direct sums, that are also finite direct products, and is actually an additive category. It also admits countable direct sums.

If $B$ denotes a Dedekind ring which is an $A$-algebra, the tensor product defines an additive functor:
$$. \otimes_A B : \CP_A \lra \CP_B.$$

\subsection{A theorem of Kaplansky}

As any finitely generated projective $A$-module is a (finite) direct sum of invertible $A$-modules, Condition (5) is equivalent to $M$ being a countable direct sum of invertible $A$-modules. 
Actually, when $M$ has infinite rank,  this observation admits the following strengthening, proved by Kaplansky (\cite{Kaplansky52}, Theorem 2) in a more general setting: 

\begin{proposition}\label{Kplfree} If some $A$-module $M$ satisfies the conditions in Proposition \ref{CPAdef} and has infinite rank (or equivalently, is not finitely generated), then it is free, hence isomorphic to $A^{(\N)}$.
 \end{proposition}

When $A$ is principal (\eg, when $A = \Z$, a case of special interest in this monograph), this is straightforward.  The part of Kaplansky's argument in \emph{loc. cit.} relevant to the derivation of Proposition \ref{Kplfree} for a general Dedekind ring $A$ may be summarized as follows. 

Firstly one shows that, \emph{for any element $m$ of some projective countably generated $A$-module $M$ of finite rank, there exists a direct summand $P$ in $M$, free and of finite rank, which contains $m$}. 

To achieve this, observe  that $M$ may written as an infinite countable direct sum $\bigoplus_{i \in \N} I_i$ of invertible submodules $I_i$ of $M$, and recall that, for any two invertible $A$-modules $I$ and $J,$ the $A$-modules $I \oplus J$ and $A \oplus (I\otimes J)$ are isomorphic. This last fact implies that, for any $n \in \N,$ if we define $J_n := \bigotimes_{0\leq i \leq n} I_i,$ then the $A$-module $\bigoplus_{0\leq i \leq n} I_i \oplus J^\vee$ is free of rank $n+1$,  and that there  exists some isomorphism of $A$-modules:
$$\phi : I_{n+1} \oplus I_{n+2} \lrasim J^\vee \oplus (J\otimes I_{n+1} \otimes I_{n+2}).$$
 Therefore, if $n$ is chosen so large that $\bigoplus_{0\leq i \leq n} I_i$ contains $m$, then the submodule 
 $P:= \bigoplus_{0\leq i \leq n} I_i \oplus \phi^{-1}( J^\vee \oplus \{0\})$
 of $M$ is a free direct summand, of rank $n+1$, and contains $m$.

Secondly one considers a countable family of generators $(m_i)_{i \in \N_{>0}}$ of $M$, and by means of the above fact, one constructs inductively projective $A$-submodules $(P_i)_{i\in \N_{>0}}$ and $(M^i)_{i\in \N}$ of $M$, such that the $P_i$ are free of finite rank and the $M^i$ are countably generated, and such that the following conditions are satisfied:
 \begin{enumerate}
\item $M^0=M;$
\item for any $i \in \N_{>0},$ $M^{i-1} = P_i \oplus M^i$ and $m_i \in P_i$. 
\end{enumerate}
Thus we obtain a decomposition   $M = \bigoplus_{i \in \N_{>0}} P_i,$
which shows that $M$ is a countable direct sum of free modules of finites ranks, and completes the proof.

\section{Linearly compact Tate spaces with countable basis}\label{CTCbasics}

\subsection{Basic definitions}

We define the $A$-linear category $\CTC_A$ of \emph{linearly {\bf c}ompact {\bf T}ate spaces  with {\bf c}ountable basis} over $A$ as follows.\footnote{The terminology of \emph{Tate space} is borrowed from Drinfeld \cite{Drinfeld2006}.}

An object $N$ of $\CTC_A$ is a topological module over the ring $A$ equipped with the discrete topology which satisfies the following two conditions: 

 $\mathbf{CTC_1 :}$ \emph{The topology of $N$ is Hausdorff and complete}.

$\mathbf{CTC_2:}$  \emph{Their exists a countable basis of neighborhoods $U$ of $0$ in $N$ consisting in $A$-submodules of $N$ such that $N/U$ is a finitely generated projective $A$-module.} 
Morphisms in the category $\CTC_A$ are $A$-linear continuous maps: for any two objects $N1$ and $N_2$ in $\CTC_A$, we let
$$\Hom_{\CTC_A}(N_1, N_2) := \Hom_A^{\rm cont}(N_1,N_2).$$

\medskip

For any subset $U$ of some $A$-module $N$, we may consider  the  condition appearing in $\mathbf{CTC_2}$:
\begin{description}
\item[$\mathbf{C}_U$] {\emph{
 $U$ is a 
$A$-submodule of $N$, and $N/U$ is a finitely generated projective $A$-module.}} 
\end{description}
Then we have: 

\begin{lemma}\label{lemmCU}
Let $N$ be an object of $\CTC_A$. A subset $U$ of $N$ is an neighborhood of $0$ and satisfies Condition $\mathbf{C}_U$ if and only if $U$ is an open saturated submodule of $N.$
\end{lemma}

\begin{proof} The necessity is clear. Conversely, if $U$ is an open saturated  submodule of $N$, then it contains a neighborhood $U_0$ of $0$ which satisfies $\mathbf{C}_{U_0}$. Then $U/U_0$ is a saturated submodule of the finitely generated projective $A$-module $N/U_0$, and consequently 
$$N/U \simeq (N/U_0)/(U/U_0)$$
also is a finitely generated projective $A$-module. 
\end{proof}

For any object $N$ of $\CTC_A$, we shall denote the family of open saturated submodules of $N$ by $\cU(N).$ It is stable under finite intersection. 

Any finitely generated projective $A$-module, equipped with the discrete topology, becomes an object of $\CTC_A$. In this way, the category of finitely generated projective $A$-modules and $A$-linear maps appears as a full subcategory of $\CTC_A$.

According to the countability assumption in $\mathbf{CTC_2}$, for any object $N$ of $\CTC_A$, there exists a ``non-increasing" sequence
\begin{equation*}\label{Ui}
U_0 \hookleftarrow U_1 \hookleftarrow U_2 \hookleftarrow \dots
\end{equation*}
of submodules in $\cU(N)$ which constitute a basis of neighborhoods of $0$ in $N$.
We shall call any such sequence $(U_i)_{i \in \N}$ in $\cU(\Eh)^\N$  a \emph{filtration defining the topology of $\Eh$}, or shortly a \emph{defining filtration} in $\cU(\Eh)^\N$.

From any  defining filtration $(U_i)_{i \in \N}$ in $\cU(\Eh)^\N$, 
we may construct a countable projective system of finitely generated projective $A$-modules:
$$\hE/U_0 \longleftarrow \hE/U_1 \longleftarrow \hE/U_2 \longleftarrow \dots,$$
and we may consider the canonical morphism $\hE \lra  \varprojlim_{i} \hE/U_i$, 
defined by the quotient maps $\hE \lra \hE/U_i$. According to Condition $\mathbf{CTC_1}$,  this morphism  is bijective, and actually becomes an isomorphism of topological $A$-modules
 \begin{equation}\label{canprojE}
\hE \simeq 
\varprojlim_{i} \hE/U_i,
\end{equation} when  $\varprojlim_{i} \hE/U_i$ is equipped with the projective limit topology deduced from the discrete topology on the finitely generated projective modules $\hE/U_i$. 

Conversely, for any projective system
\begin{equation}\label{projE}
E_0 \stackrel{q_0}{\longleftarrow}E_1 \stackrel{q_1}{\longleftarrow} E_2 \stackrel{q_2}{\longleftarrow}\dots
\end{equation}
of surjective morphisms between finitely generated projective $A$-modules, the projective limit
$$\hE := \varprojlim_{i} E_i,$$ equiped with its natural prodiscrete topology, defines an object of $\CTC_A.$

Moreover, if $\hE := \varprojlim_{i} E_i$ and $\hF := \varprojlim_{j} F_j$ are two objects of $\CTC_A,$ realized as limits of projective systems of finitely generated projective $A$-modules as above, we have a canonical identification:
\begin{equation*}
\Hom_{\CTC_A}(\hE, \hF)  := \Hom_A^{\rm cont} (\varprojlim_{i} E_i, \varprojlim_{j} F_j) \simeq  \varprojlim_{j} \varinjlim_{i} \Hom_A (E_i,F_j).
\end{equation*}

If $B$ denotes a Dedekind ring which is an $A$-algebra, the \emph{completed} tensor product defines an additive functor:
$$.\, \widehat{\otimes}_A B : \CTC_A \lra \CTC_B.$$

Observe that, for any projective system (\ref{projE}) of surjective morphisms of finitely projective $A$-modules, we get, by extending the scalars from $A$ to $B$, a projective system of surjective morphisms of finitely projective $B$-modules
 \begin{equation}\label{projEB}
E_{0,B} \stackrel{q_{0,B}}{\longleftarrow}E_{1,B} \stackrel{q_{1,B}}{\longleftarrow} E_{2,B} \stackrel{q_{2,B}}{\longleftarrow}\dots
\end{equation}
Its projective limit ``is" the object $\Eh \,\widehat{\otimes}_A B $ of $\CTC_B$ deduced from the projective limit $\Eb$ of (\ref{projE}) by the completed tensor product functor. Indeed, we have a canonical isomorphism of prodiscrete $B$-modules:
$$(\varprojlim_{i} E_i) \hat{\otimes}_A B \lrasim \varprojlim_{i} E_{i,B}.$$

The following proposition is included for later reference in Section \ref{provectsmoothcurve}.

\begin{proposition}\label{Utens} Withe the above notation, the completed tensor product defines a map 
\begin{equation}\label{UAB}
.\,\widehat{\otimes}_A B : \cU(\Eh) \lra \cU(\Eh \,\widehat{\otimes}_A B),
\end{equation}
which is injective when $B$ is flat (or equivalently torsion free) over $A$, and bijective when $B$ is a localization of $A$.

\begin{proof} This easily follows from the special case where $\Eb$ is a finitely generated projective $A$-module $M$. Then $\Eh \,\widehat{\otimes}_A B$ is the tensor product $M\otimes_A B$, and $\cU(\Eh)$ (resp. $\cU(\Eh \,\widehat{\otimes}_A B)$) is the set of saturated $A$-submodules of $M$ (resp., of saturated $B$-submodules of $M\otimes_A B$). We leave the details to the interested reader.
\end{proof}

\end{proposition}

\subsection{Subobjects and countable products}

\begin{proposition}\label{subCTC} Let $N$ be an object of $\CTC_A$. Any closed $A$-submodule $N'$ of $N$, equipped with the induced topology, is an object of $\CTC_A$.
\end{proposition}

\begin{proof}

Equipped with the induced topology, $N'$ is clearly Hausdorff and complete. Therefore the topological $A$-module $N'$ satisfies $\mathbf{CTC_1}$.

Moreover, if some neighborhood $U$ of $0$ in $N$ satisfies $\mathbf{C}_U$, then $U' := U \cap N'$ is a neighborhood of $0$ in $N'$ which is clearly an $A$-submodule of $N'$. Moreover the inclusion $N' \hra N$ defines an injective morphism of $A$-modules  $N'/U' \ra N/U$; therefore, $N'/U'$ --- as any submodule of a finitely generated projective module over a Dedekind ring --- is also a finitely generated projective $A$-module. In other words, $U'$ satisfies  $\mathbf{C}_{U'}$.

This immediately implies that condition $\mathbf{CTC_2}$ also is inherited by $N'$.
\end{proof}

Observe that, with the notation of Lemma \ref{subCTC}, the  topological $A$-module $N/N'$, even if assumed torsion-free, may not be an object of $\CTC_A$. 

For instance, when $A=\Z,$ the short exact sequences\footnote{Observe that these are actually strict short exact sequences of topological abelian groups.}
 (\ref{sesbeta2}) and (\ref{sesbeta2'}) in Proposition \ref{padicex} \emph{infra} and its proof display the ring of $p$-adic integers $\Z_p$, equipped with its $p$-adic topology, as a quotient of $\Z^\N$ by a closed submodule. Actually, one may easily show that \emph{the topological $\Z$-modules that may be realized has quotient of an object of $\CTC_\Z$ by a closed subobject are precisely the commutative Polish topological groups $G$ admitting a basis of neighborhoods of $0$ which are open subgroups $U$ such that $G/U$ is finitely generated.}
 
The $A$-linear category $\CTC_A$ admits obvious finite direct sums, that are also finite direct products, and is actually an additive category. It also admits countable direct products.

It also admits \emph{countable products}. Indeed, if $(N_i)_{i \in I}$ is a countable family of objects in  $\CTC_A,$ the $A$-module 
$$N := \prod_{i \in I} N_i,$$
equipped with the product topology, is easily seen to define an object in $\CTC_A.$ The projection maps ${\rm pr}_i: N \lra N_i$ are morphisms in $\CTC_A$, and $(n, (p_i)_{i \in I})$ is a product of the $N_i 's $ in the category $\CTC_A.$

More generally, any projective system of surjective open morphisms in $\CTC_A$
$$M_0 \twoheadleftarrow M_1 \twoheadleftarrow M_2 \twoheadleftarrow M_3 \twoheadleftarrow \dots$$
admits a (projective) limit $\varprojlim_i M_i$ in $\CTC_A$, defined by the $A$-module projective limit of the $M_i$'s, equipped with the projective limit of their discrete topology.

The following proposition shows that, up to isomorphism, every object of $\CTC_A$ is a product of finitely generated projective $A$-modules (equipped with the discrete topology). Its easy proof is left to the reader.

\begin{proposition}\label{prodfin}
 Consider a projective system of surjective morphisms of finitely generated projective $A$-modules:
 \begin{equation}
E_0 \stackrel{q_0}{\longleftarrow}E_1 \stackrel{q_1}{\longleftarrow} E_2 \stackrel{q_2}{\longleftarrow}\dots
\end{equation} 

For every $i \in \N,$ there exists an $A$-linear section $\sigma_i: E_i \lra E_{i+1}$ of $q_i$, and if we define
$$
\begin{array}{ccll}
S_i  & :=   & \ker q_{i-1} & \mbox{ if $i \geq 1,$}  \\
  & :=  & E_0  &  \mbox{ if $i= 0,$}
 \end{array}
$$
then the $A$-modules $S_i$ are finitely generated and projective, and for every $i \in \N_{>0},$ we have:
\begin{equation}\label{Esigma}
E_i = S_i \oplus \sigma_{i-1} (E_{i-1}).
\end{equation}

 The direct sum decompositions (\ref{Esigma}) determine a family of isomorphisms of $A$-modules
 $$\iota_n: E_n \lrasim \bigoplus_{0\leq i \leq n} S_i$$
 such that 
 $$\iota_n \circ q_n \circ \iota^{-1}_{n+1} : \bigoplus_{0\leq i \leq n+1} S_i \lra \bigoplus_{0\leq i \leq n} S_i$$
 is the projection map on the first $n+1$-th factors, and consequently an isomorphism in $\CTC_A$:
 $$\iota: \varprojlim_n E_n \lrasim \prod_{i \in \N} S_i.$$
\end{proposition}

\subsection{Continuity of  morphisms of $A$-modules between objects of $\CTC_A$.}
In this paragraph, we want to indicate that, for a large class of Dedekind rings $A$, any morphism of $A$-modules between two objects $N_1$ and $N_2$ in $\CTC_A$ is automatically continuous. In other words, for these rings, the forgetful functor from the category $\CTC_A$ to the category of $A$-modules is fully faithful. 

This will follow from the variant of results of  Specker (\cite{Specker50}) and Enochs (\cite{Enochs64}) discussed in
Appendix \ref{prodiscretemod}.

Observe that, for any Dedekind ring $A$, precisely one of the following three conditions is satisfied:
{\it 

$\mathbf {Ded_1 :}$  $A$ is a field;

$\mathbf {Ded_2 :}$  $A$ is a complete discrete valuation ring;

$\mathbf {Ded_3 :}$  there exists some non-zero prime ideal $\fp$ of $A$ such that $A$ is not $\fp$-adically complete.\footnote{It is straightforward that, when $\mathbf {Ded_3}$ holds, the ring $A$ is not $\fp$-adically complete for \emph{every} non-zero prime ideal $\fp$ of $A$.}
}

For instance, any countable Dedekind ring $A$ which is not a field satisfies   $\mathbf {Ded_3}$, since the cardinality of a complete discrete valuation ring is at least the cardinality of the continuum. 

When the Dedekind ring $A$ satisfies $\mathbf {Ded_1}$ or $\mathbf {Ded_2}$, there exists many ``linear forms" in $\Hom_A(A^\N, A)$ which are not continuous when $A^\N$ (resp. $A$) is equipped with its natural prodiscrete (resp. discrete) topology, or equivalently, there exists some $A$-linear map $\xi: A^\N \lra A$ which is not of the form
\begin{equation}\label{xixin}
\xi((x_n)_{n \in \N}) = \sum_{n \in \N} \xi_n x_n
\end{equation}
for some $(\xi_n)_{n \in \N}$ in $A^{(\N)}$.

Indeed, when $A$ is a field, any non-zero linear form on the vector space $A^\N$ that vanishes on its subspace $A^{(\N)}$ is such a map.
When $A$ is  a complete discrete valuation ring, of maximal ideal $\fm$, for any sequence $(\xi_n)_{n \in \N}$ in $A^\N \setminus A^{(\N)}$ that converges to zero in the $\fm$-adic topology, the formula (\ref{xixin}) defines an element $\xi \in \Hom_A(A^\N, A)$ of the required type.

In contrast, for Dedekind rings satisfying $\mathbf {Ded_3}$, the following continuity results holds:

\begin{proposition}\label{prop:HomHom}
If the Dedekind ring $A$ satisfies $\mathbf {Ded_3}$, then for any two objects $N$ and $N'$ in $\CTC_A$, we have:
\begin{equation}\label{HomHom}
\Hom_A(N, N') = \Hom_A^{\rm cont}(N, N').
\end{equation}

\end{proposition}
\begin{proof} When $N' = A,$ this follows from 
 Corollary \ref{contuitautombis} applied to $R:= A$ and to $M:= N$. (Indeed, we may choose as $\fm$ any non-zero prime ideal $\fp$ of $A$: the local ring $A_{(\fp)}$ is not complete, since $A$ satisfies $\mathbf {Ded_3}$.)

 The validity of (\ref{HomHom}) when $N'= A$ implies its validity when $N'= A^{\oplus n}$ for some $n \in \N$, hence its validity for any finitely generated projective $A$-module $N'$ equipped with the discrete topology. (Indeed an such module may be realized as a submodule --- actually as a direct summand --- of some free $A$-module $A^{\oplus n}$.)
 
 To complete the proof of Proposition \ref{prop:HomHom}, observe that any object $N'$ of $\CTC_A$ may be realized as the projective 
  limit $N' = \varprojlim_i N'_i$ of some projective system of finitely generated projective $A$-modules
  $N'_0 \leftarrow N'_1 \leftarrow N'_2 \leftarrow \dots$, and that we have natural identifications:
   $$\Hom_A(N, N') \simeq \varprojlim_i \Hom_A(N, N'_i)$$
   and 
   $$\Hom^{\rm cont}_A(N, N') \simeq \varprojlim_i \Hom^{\rm cont}_A(N, N'_i).$$
   The validity of (\ref{HomHom}) consequently follows from the already established equalities:
   $$\Hom_A(N, N'_i) = \Hom_A^{\rm cont}(N, N'_i).$$
 \end{proof}

\subsection{The topology on objects in $\CTC_A$ when $A$ is a topological ring}

Let us assume that, besides its discrete topology, the ring $A$ is equipped with a topology that makes $A$ a topological ring. We shall denote $A^\an$ the topological ring defined by $A$ equipped with this finer topology.

For instance, $A$ may be a discrete valuation ring and $A^\an$ the ring $A$ equipped with the topology associated to its discrete valuation, or a local field and  $A^\an$ the field $A$ equipped with its ``usual" locally compact topology.

In this situation, besides its topology of pro-discrete $A$-module, any object $N$ of $\CTC_A$ is canonically endowed with a finer topology, which makes it a topological $A^\an$-module $N^\an$.

Indeed, any finitely generated projective $A$-module $P$ is equipped with a canonical topology of $A^\an$-module: if $P$ is embedded as a direct summand in the direct sum $A^{\oplus n}$ of a finite number of copies of $A$, this topology is the one induced by the product topology on $(A^{\an}) n$. 
We shall denote by $P^\an$ the so-defined topological $A^\an$-module. Observe that any $A$-linear morphism $\phi: P_1 \lra P_2$ between finitely generated projective $A$-modules defines a continuous morphism $\phi: P^\an_1 \lra P^\an_2$ of topological $A^\an$-modules.   

By definition, if $N$ denotes an object of $\CTC_A$, we have a canonical isomorphism of topological $A$-modules
\begin{equation}\label{Niso}
N \lrasim \varprojlim_{U \in \cU(N)} N/U,
\end{equation}
where $\varprojlim_{U \in \cU(N)} N/U$ is equipped with the pro-discrete topology. We make $N$ a topological $A^\an$-module $N^\an$ by declaring (\ref{Niso}) to be an isomorphism of topological $A^\an$-modules
\begin{equation}\label{Nisoan}
N^\an\lrasim \varprojlim_{U \in \cU(N)} (N/U)^\an,
\end{equation}
where $\varprojlim_{U \in \cU(N)} (N/U)^\an$ is equipped with the projective limit topology deduced from the canonical topology of $A^\an$-module on each finitely generated projective $A$-module $N/U$.

Any $U \in \cU(N)$ is an object of $\CTC_A$ and as such is equipped with the ``analytic" topology $U^\an$. One easily see that $U$ is actually closed in $N^\an$ and that the topology of $U^\an$ is the topology induced by the topology of $N^\an$.

Any morphism $f:N' \lra N$ in $\CTC_A$ defines a continuous morphism of topological $A^\an$-modules $f:N'^\an \lra N^\an$. However the so defined injective map $$\Hom_A^{\rm cont}(N',N) \hlra \Hom_{A^\an}^{\rm cont}(N'^\an,N^\an)$$
is not surjective in general.

When $A^\an$ is the field $\R$ (resp. $\C$) equipped with its usual ``analytic" topology, the topological $A^\an$-module $N^\an$ is a topological vector space over $\R$ (resp. over $\C$), isomorphic to $\R^n$ (resp. to $\C^n$) if $n:= \dim_{A^\an} N$ is finite, and to $\R^\N$ (resp. to $\C^\N$) if $n$ is infinite. In particular, it is a Fréchet locally convex vector space over $\R$ (resp. over $\C$).

\section{The duality between $\CP_A$ and $\CTC_A$}\label{dualCPCTC}

In this section, we discuss the (anti)equivalence of $A$-linear categories between  $\CP_A$ and $\CTC_A$ defined by duality, and some of its consequences. 
\subsection{The duality functors}$\, $

(i) To any $A$-module $M$, we may attach its dual topological $A$-module, namely  the $A$-module
\begin{equation}\label{dualobj}
M^\vee := \Hom_A(M,A)
\end{equation}
equipped with the topology of pointwise convergence. If $\alpha: M_2 \lra M_1$ is a morphism of $A$-modules, its transpose
\begin{equation}\label{dualmor}\alpha^\vee := . \circ \alpha : M_1^\vee \lra M_2^\vee
\end{equation} 
is clearly $A$-linear and continuous.

For any family $(M_i)_{i\in I}$ of $A$-modules, there is a canonical identification of topological $A$-modules:
\begin{equation*}
(\,\bigoplus_{i \in I} M_i\,)^\vee \lrasim \prod_{i \in I} M_i^\vee.
\end{equation*}

Besides, if $M$ is a finitely generated projective $A$-module, then its dual $M^\vee$ also is finitely generated and projective, and the topology of $M^\vee$ is discrete. (Indeed, this holds for any finitely generated free $A$-module $A^{\oplus n}$, and consequently for any direct summand of such an $A$-module.) 

As any object in $\CP_A$ is a countable direct sum of finitely generated projective $A$-modules, these observations show that 
the constructions (\ref{dualobj}) and (\ref{dualmor}) define a contravariant $A$-linear duality functor:
\begin{equation}\label{dual1}.^\vee : \CP_A \lra \CTC_A.
 \end{equation}

 \medskip
 
 (ii) Conversely, to any topological $A$-module $N$, we may attach its topological dual, namely the $A$-module
\begin{equation}\label{dualobjbis}
N^\vee := \Hom_A^{\rm cont}(N, A)
\end{equation}
consisting of \emph{continuous} $A$-linear maps from $N$ to $A$.  Then any continuous $A$-linear morphism $\beta : N_1 \lra N_2$ of topological $A$-modules defines, by transposition, a morphism of $A$-modules:
\begin{equation}\label{dualmorbis}
\beta^\vee := . \circ \beta : N_2^\vee \lra N_1^\vee.
\end{equation}

For any projective system of topological $A$-modules
 \begin{equation}\label{projNp}
N_0 \stackrel{p_0}{\longleftarrow}N_1 \stackrel{p_1}{\longleftarrow} N_2 \stackrel{p_2}{\longleftarrow}\dots,
\end{equation}
we may form the projective limit $\varprojlim_i N_i$, as a topological $A$-module, and we may also consider the dual inductive system of $A$-modules:
 \begin{equation*}\label{}
N_0^\vee \stackrel{p_0^\vee}{\longrightarrow}N_1^\vee \stackrel{p_1^\vee}{\longrightarrow} N_2^\vee \stackrel{p_2^\vee}{\longrightarrow}\dots  \,\, .
\end{equation*}
Then there is a canonical identification of $A$-modules:
\begin{equation*}
(\,\,\varprojlim_i N_i \, )^\vee \simeq \varinjlim_i N_i^\vee.
\end{equation*}

Besides, if $N$ is a finitely generated projective $A$-module, equipped with the discrete topology, then
$$N^\vee = \Hom_A^{\rm cont}(N, A) = \Hom_A (N, A)$$
is also a finitely generated projective $A$-module. Moreover, if $q: N' \ra N$ is a surjective morphism of finitely generated projective $A$-modules, then $p^\vee: N^\vee \ra N'^\vee$ is injective and its image is a direct summand in $N'^\vee$.

As any object in $\CTC_A$ is (up to isomorphism) the projective limit of some countable projective system $E_0 \leftarrow E_1 \leftarrow E_2 \leftarrow \dots$ of surjective morphisms between finitely generated projective $A$-modules, these observations show that the constructions (\ref{dualobjbis}) and (\ref{dualmorbis}) define a contravariant $A$-linear duality functor:
\begin{equation}\label{dual2}.^\vee : \CTC_A \lra \CP_A.
 \end{equation}

\subsection{Duality as an adjoint equivalence}\label{dualadjointA}

 Observe that, for any object $M$ of $\CP_A$ and any object $N$ of $\CTC_A,$
the $A$-modules $\Hom_A^{\rm cont}(M, N^\vee)$ and $\Hom_A(N, M^\vee)$ may both be identified with the module of $A$-bilinear maps $M \times N \lra A$ which are continuous in the first variable. In this way, we obtain a systems of bijections
\begin{equation}\label{adjtdual}
\Hom_{\CTC_A} (M, N^\vee) := \Hom_A^{\rm cont} (M, N^\vee)  \simeq \Hom_A(N, M^\vee) = \Hom_{\CP_A^{\rm op}} (M^\vee, N).
\end{equation}

\begin{proposition}\label{prop:dual} The bijections  (\ref{adjtdual}) defines an adjunction of functors: %
 $$.^\vee :  \CTC_A  \leftrightarrows \CP_A^{\rm op} : .^\vee$$
 
It is actually an adjoint equivalence, 
 whose unit and counit are the natural isomorphism $\eta: I_{\CTC_A} \simeq .^{\vee \vee}$ and  $\epsilon: .^{\vee \vee} \simeq I_{\CP_A^{\rm op}}$ defined by the biduality isomorphisms
 $$
\begin{array}{rrcl}
 \epsilon_M:& M  & \lrasim   & \Hom_A^{\rm cont}(\Hom_A (M,A),A)  \\
& m & \longmapsto   & (\xi  \longmapsto \xi(m))
\end{array}
$$
and $$
\begin{array}{rrcl}
 \eta_N:& N  & \lrasim   & \Hom_A(\Hom_A^{\rm cont} (N,A),A)  \\
& n & \longmapsto   & (\zeta  \longmapsto \zeta(n))   
\end{array}
$$
associated to any object $M$ of $\CP_A$ and to any object $N$ of $\CTC_A$.
 \end{proposition}
 
 \begin{proof} We only sketch the proof and leave the details to the readers.
 
 The naturality with respect to $M$ (in $\CTC_A$) and to $N$ (in $\CP_A^{\rm op}$) of the bijections (\ref{adjtdual}) is straightforward, and the  expressions for $\epsilon_M$ and $\eta_N$ as well.
 
 To complete the proof, we are left to establish that, for any $M$ (resp. $N$) in $\CP_A$ (resp. in $\CTC_A$), $\epsilon_M$ (resp. $\eta_N$) is an isomorphism in $\CP_A$ (resp. in $\CTC_A$). The compatibility of the duality functors  with countable direct sums (in $\CP_A$) and countable products (in $\CTC_A$), and the fact that any object in $\CP_A$ (resp. in $\CTC_A$) is a countable direct sum (resp. a countable product) of finitely generated projective $A$-modules, allows us to reduce the proof to the special case when
$M$ (resp. $N$) is a finitely generated projective $A$-module. Then it is straightforward. 
 \end{proof}

\begin{corollary}
The duality functors (\ref{dual1}) and (\ref{dual2}) are (anti)equivalences of categories.  \qed 
\end{corollary}

 Let $B$ be a Dedekind ring which is an $A$-algebra. The reader may easily establish the compatibility between the duality functors between $\CP_A$ and $\CTC_A$ and  between $\CP_B$ and $\CTC_B$, and of the ``base change" functors $.\otimes_A B : \CP_A \lra \CP_B$ and
 $. \,\hat{\otimes}_A B: \CTC_A \lra \CTC_B.$
 
 \subsection{Applications}\label{appdualCPCTC} Diverse results about the category $\CP_A$ may be transferred by duality to the category $\CTC_A$.
 \medskip
 
 (i) For instance, the morphisms $\phi$ in $\Hom_A(A^{(\N)}, A^{(\N)})$ are in bijection with  ``infinite matrices" $(\phi_{ij})_{(i,j) \in \N\times \N}$ in $A^{\N \times \N}$ which admit only a finite of non-zero entries in each column, by the usual formula, valid for any $x=(x_i)_{i \in \N}$ and $y=(y_j)_{j \in \N}$ in $A^{(\N)}$:
 \begin{equation}\label{matphi}
y = \phi(x) \Longleftrightarrow y_i = \sum_{j \in \N} \phi_{ij} x_j.
\end{equation}

By duality, this implies that the morphisms $\phi$ in $\Hom_A^{\rm cont}(A^\N, A^\N)$ are bijection with matrices  $(\phi_{ij})_{(i,j) \in \N\times \N}$ in $A^{\N \times \N}$ which contains only a finite of non-zero entries in each row, still by means of formula (\ref{matphi}).
\medskip
 
 (ii) From Kaplansky's result concerning the freeness of modules of infinite rank in $\CP_A$ (Proposition \ref{Kplfree}), we get the following refinement of the description of objects of $\CP_A$ as countable products of finitely generated projective $A$-modules in Proposition \ref{prodfin}:
 \begin{proposition}\label{Kapdual}
 Let $N$ be an object in $\CTC_A$.
 Either the topology of $N$ is discrete and the $A$-module $N$ is finitely generated and projective, or $N$ is isomorphic to $A^{\N}$ as a topological $A$-module. \qed
 \end{proposition}

(iii) When $A$ is a field $k$, the objects of $\CTC_A$ are precisely the \emph{linearly compact $k$-vector spaces}, introduced by Lefschetz and Chevalley (\cite{Lefschetz42}, Chapter II, §6), that admit a countable basis of neighborhoods of zero. In this case, the category $\CP_A$ is the category of $k$-vector spaces of countable dimension. Diverse results from linear algebra over the field $k$ may be transported, by duality, to results concerning $\CTC_A$.

 In this way, from the basic facts  about $k$-vector spaces of countable dimension and their $k$-linear maps, we derive the following results (which go back to Toeplitz and K\"othe; see notably \cite{Toeplitz1909} and \cite{Koethe49};  see also \cite{DieudonneCrelle50}):

\begin{proposition}\label{CTCk} Let $k$ be a field.

1) In the category $\CTC_k,$ any object is isomorphic, either to $k^n$ for some non-negative integer $n$, or to $k^\N.$

2) For any morphism $\phi: N_1 \lra N_2$ in $\CTC_k,$ there exists objects $K,$ $N,$ and $C$ 
 and  isomorphisms 
 $$ u: N_1 \lrasim K \oplus N  \mbox{ and } v : N_2 \lrasim C \oplus N $$
 in $\CTC_k$ such that the morphism
 $\tilde{\phi} := v \phi u^{-1} : K \oplus N \lra C \oplus N$
  is the ``bloc diagonal" morphism $0 \oplus Id_N$,  which sends   $(k,n) \in K \oplus N$ to $(0, n)$ in $C \oplus N.$ \qed

\end{proposition}

\section{Strict morphisms, exactness and duality}\label{StrictCTC} 

\subsection{Strict morphisms in $\CTC_A$}$\,$

We shall say that a morphism 
\begin{equation}\label{phinn}
\phi : N_1 \lra N_2
\end{equation}
in $\CTC_A$ is \emph{strict} if $\phi$ is a strict morphism of topological groups, namely if the induced map
$$
\begin{array}{rrcl}
\tilde{\phi}: & N_1/\ker \phi  & \lra    & \im \phi   \\
& [x]  & \longmapsto  & \phi(x)  \\
  \end{array}
$$
is a homeomorphism, when $N_1/\ker \phi$ (resp. $\im \phi$) is equipped with the topology quotient of the topology of $N_1$ (resp., with the topology induced by the topology of $N_2$).

As $N_1$ and therefore its quotient $N_1/\ker \phi$ are complete, if the morphism (\ref{phinn}) is strict, its image $\im \phi$ also is complete, hence closed in $N_2$, and therefore defines an object of $\CTC_A$ by Proposition \ref{subCTC}.

Accordingly, the strict morphisms in $\Hom_A^{\rm cont}(N_1,N_2)$ are precisely the maps $\phi : N_1 \lra N_2$ for which there exists an object $I$ in $\CTC_A$, defined by some closed $A$-submodule of $N_2,$ and an \emph{open surjective} morphism $\phi_0$ in $\Hom_A^{\rm cont}(N_1, I)$ such that $\phi$ admits the factorization
$$\phi: N_1 \stackrel{\phi_0}{\twoheadrightarrow} I \hra N_2.$$

\subsection{Finite rank morphisms}
 We want to show that any morphism in $\CTC_A$ whose image is finitely generated is strict. This will follow from the following result, of independent interest:

\begin{proposition}\label{injPN} Let $P$ be a finitely generated projective $A$-module, and let $N$ be an object of $\CTC_A$. Let $K$ denote the field of fractions of $A$.

For any morphism of $A$-modules 
$\phi: P \lra N,$
the following conditions are equivalent:

(1) $\phi$ is injective; 

(2) the $K$-linear map $\phi_K : P_K := P\otimes_A K \lra N_K := N \hat{\otimes}_A K$ is injective;

(3) there exists $U$ in $\cU(N)$ such that the composite morphism of $A$-modules
$$P\stackrel{\phi}{\lra} M \twoheadrightarrow N/U$$
is injective.
\end{proposition}

\begin{proof} The implications $(3) \Rightarrow (2) \Rightarrow (1)$ are straightforward.

To prove the implication $(1) \Rightarrow (3),$ observe that the inverse images $\phi^{-1}(V)$ of the submodules $V$ of $N$ in $\cU (N)$ constitute a family of saturated submodules of $E$ stable under finite intersection. As $P$ has finite rank, we may consider $U$ in $\cU (N)$ such that $\phi^{-1}(U)$ has minimal rank. Then, for any $V \in \cU(N),$ the intersection $\phi^{-1}(U) \cap \phi^{-1}(V)$ is saturated in $N$, of rank at most the rank of $\phi^{-1}(U)$, and therefore coincides with $\phi^{-1}(U).$
This shows that 
\begin{equation}
\phi^{-1}(U) = \bigcap_{V \in \cU(N)} \phi^{-1}(V). 
\end{equation}

When (1) is satisfied, 
$$\bigcap_{V \in \cU(N)} \phi^{-1}(V)
= \phi^{-1}(\bigcap_{V \in \cU(N)} V) = \phi^{-1}(0) = \{0\},$$
and therefore $\phi^{-1}(U)=\{0\},$ and (3) is satisfied.
\end{proof}

\begin{corollary}\label{SatCTC}
 If $N$ is an object of $\CTC_A$ and if $P$ is a finitely generated $A$-submodule of $N$, then $P$ is a finitely generated projective $A$-module and the topology of $N$ induces  the discrete topology on $P$.
 
 Moreover the saturation 
 $$\tilde{P} := \{ n \in N \mid  \exists  \alpha \in A\setminus\{0\}, \alpha n \in P\}$$
 of $P$ in $N$ is also a finitely generated $A$-module.
\end{corollary}
\begin{proof}  The projectivity is clear of $P$ is clear. Moreover, according to Proposition \ref{injPN}, we may choose an element $U$ of $\cU(N)$ such that $U \cap P = 0$. Since $U$ is open in $N$, this shows that the topology on $P$ induced by the one of $M$ is the discrete topology.

 To prove that $\tilde{P}$ is finitely generated observe that the composite map $$U \hlra N \lra N/P$$ is injective and fits into an exact sequence of $A$-modules:
 $$0 \lra U \lra N/P \stackrel{q}{\lra} N/(U+P)  \lra  0.$$
As $U$ is torsion free, the map $q$ defines by restriction an injection of torsion submodules:
$$q: (N/P)_{\rm tor} \hlra (N/(U+P))_{\rm tor}.$$

As $N/U$ --- hence $N/(U+P)$ --- is a finitely generated $A$-module, this proves that $(N/P)_{\rm tor}$ also is finitely generated. Finally the short exact sequence 
$$0 \lra P \lra \tilde{P} \lra (N/P)_{\rm tor} \lra 0$$ shows that $\tilde{P}$ is finitely generated.
\end{proof}

\begin{corollary} If the image 
of some morphism 
in $\CTC_A$ is a
finitely generated  $A$-module, then this morphism 
is strict. \qed  
\end{corollary}

The following proposition completes Proposition \ref{injPN} in the situation when $\phi$ as a saturated image. 
\begin{proposition}\label{finitesaturated}
Let $P$ be a finitely generated projective $A$-module, and let $N$ be an object of $\CTC_A$. 

For any  morphism of $A$-modules 
$\phi: P \lra N,$
the following conditions are equivalent:

(1) the morphism $\phi$ is injective and its image $\phi(P)$ is a saturated $A$-submodule of $N$. 

(2) there exists $U$ in $\cU(N)$  such that the composite morphism of $A$-modules
$$P\stackrel{\phi}{\lra} M \twoheadrightarrow N/U$$
is  injective and its image is a saturated $A$-submodule of $N/U$.
\end{proposition}

\begin{proof}
 The implication $(2) \Rightarrow (1)$ is straightforward. 
 
 To establish the converse implication, let us assume that $\phi$ is injective with saturated image, and  consider a defining sequence $(U_i)_{i \in \N}$ in $\cU(N)$. We define
 $$N_i:= M/U_i$$
 and we  denote by $$p_i: N \lra N_i$$
 the canonical quotient map and by
 $$\phi_i := p_i \circ \phi : P \lra N_i$$
 its composition with $\phi.$
 
 Let $\bp$ be a non-zero prime ideal of $A$, and let $\Fbp := A/\bp$ be its residue field.
 
 The $\Fbp$-vector space 
 $N_{\Fbp} := N/\bp N$ may be identified with the tensor product $N\otimes_A \Fbp$. Equipped with the topology quotient of the topology of $N$, it coincides with the object of $\CTC_{\Fbp}$ defined as the completed tensor product $N \hat{\otimes}_A \Fbp,$ and also with the projective limit $\varprojlim_i N_{i,\Fbp}$ of the finite dimensional $\Fbp$-vector spaces $N_{i, \Fbp}$ equipped with the discrete topology. 
 
 The $\Fbp$-linear map 
 $$\phi_{\Fbp}: P_\Fbp \lra N_\Fbp$$
 is injective, since $\phi$ is injective and its image is saturated in $N$. The map $\phi_\Fbp$ is also defined by the compatible system of $\Fbp$-linear maps
 $$\phi_{i,\Fbp}: P_\Fbp \lra N_{i,\Fbp}.$$ 
 Therefore
 $$\bigcap_{i \in \N} \ker \phi_{i,\Fbp} = \ker \phi_i = \{0\}.$$
 This shows the existence of some $\iota(\bp)$ in $\N$ such that $\ker \phi_{i,\Fbp} = \{0\}$ for every $i \in \N_{\geq \iota(\bp)}.$
 
 Besides, according to Proposition \ref{injPN}, there exists $\iota_0$ in $\N$ such that $\phi_{\iota_0}$ (and consequently $\phi_i$ for any $i \in \N_{\geq \iota_0}$) is injective. The set $\cT$ of non-zero prime ideals in $A$ such that 
 $$\phi_{\iota_0,\Fbp}: P_\Fbp \lra N_{\iota_0,\Fbp}$$  
 is non-injective is finite. It is indeed defined by the vanishing of the  non-zero section $\wedge^{\rk P} \phi_{\iota_0}$ of $(\wedge^{\rk P} P)^\vee \otimes \wedge^{\rk P} N_{\iota_0}$ over $\Spec A$. 
 
 Finally, if we let
 $$\iota_1 := \max (\iota_0, \max_{\bp \in \cT} \iota(\bp)),$$
 then, for any $i \in \N_{\geq \iota_1}$ and any non-zero prime ideal $\bp$ of $A$, $\phi_{i,\Fbp}$ is injective, and therefore $\phi_i$ is injective with a saturated image.
\end{proof}

When Condition (2) in Proposition \ref{finitesaturated}  is satisfied by some $U$ in $\cU(N)$, it is clearly also satisfied by any element  $U'$ in $\cU(N)$ contained in $U$. This implies:

\begin{corollary}\label{quotfinCTC} Let $N$ be an object of $\CTC_A$ and let $P$ be a finitely generated $A$-submodule of $N$.
The quotient topological $A$-module $N/P$ is an object of $\CTC_A$ if (and only if) $P$ is saturated in $N$. \qed 
 \end{corollary}

\subsection{Strict short exact sequences and duality}\label{stricshortdual}

We shall say that a diagram
\begin{equation}\label{shortstrict}
0 \lra N_1\stackrel{f}{\lra} N_2 \stackrel{g}{\lra} N_3 \lra 0
\end{equation}
of object and morphisms in $\CTC_A$ is a \emph{strict short exact sequence} if it is a short exact sequence of $A$-modules and if the morphisms $f$ and $g$ are strict --- in other words, if $f$ establishes an isomorphism of topological $A$-modules from $N_1$ onto $\ker g$ and if $g$ is a surjective open map. 

This last condition on $g$ is satisfied notably if $g$ admits a section $h$ (that is, a right inverse) in $\Hom_A^{\rm cont} (N_3, N_1)$. When this holds, the strict short exact sequence (\ref{shortstrict}) is said to be \emph{split}, and $h$ is called a \emph{splitting} of (\ref{shortstrict}).

We shall say that a closed $A$-submodule $N'$ of some object $N$ in $\CTC_A$ is \emph{supplemented in $\CTC_A$} when there exists a closed $A$-submodule $N''$ of $N$ such that the sum map
$$
\begin{array}{rcl}
N' \oplus N''  & \lra   & N  \\
 (n', n'') & \longmapsto  & n' + n''  
\end{array}
$$
is an isomorphism of topological $A$-modules (in $\CTC_A$, by  Proposition \ref{subCTC}). 

Observe that the strict short exact sequence (\ref{shortstrict}) is split precisely when the closed submodule $f(N_1)$
of $N_2$ is supplemented in $\CTC_A$. Indeed the splittings $h$ in of (\ref{shortstrict}) are in bijections with the ``topological supplements" $N''$ in $\CTC_A$ of 
$f(N_1)$ by the map which sends $h$ to its image $h(N_3)$.

Let $M_1,$ $M_2,$ and $M_3$ be objects in $\CP_A$, and let $N_1:=M_1^\vee,$ $N_2:=M_2^\vee,$ and $N_3:=M_3^\vee$ be the dual objects in $\CTC_A$.

Let $\cS$ (resp. $\cT$) be the subset of $\Hom_A(M_3,M_2) \times \Hom_A(M_2,M_1)$ (resp. of $\Hom^{\rm cont}_A(N_1,N_2) \times \Hom^{\rm cont}_A(N_2,N_3)$) consisting of the pairs of morphisms $(i,p)$ (resp. $(j,q)$) such that the diagram
\begin{equation}\label{Mip}
0 \lra M_3\stackrel{i}{\lra} M_2 \stackrel{p}{\lra} M_1 \lra 0
\end{equation}
is an exact sequence of $A$-modules (resp. such that 
\begin{equation}\label{Njq}
0 \lra N_1\stackrel{j}{\lra} N_2 \stackrel{q}{\lra} N_3 \lra 0
\end{equation}
is a strict short exact sequence in $\CTC_A$).

For each $i \in \{1,2,3\},$ the topological dual $N_i^\vee$ of $N_i$  will be identified with $M_i$ by the biduality isomorphisms $\epsilon_{M_i}^{-1}$
of Proposition \ref{prop:dual}.

\begin{proposition}\label{STsplit}
 With the above notation, one defines a bijection $\delta :\cS \lrasim \cT$ by the formula:
 \begin{equation}\label{defdelta}
\delta(i,p) := (p^\vee, i^\vee).
\end{equation}
The inverse bijection $\delta^{-1}$ is given by 
\begin{equation}\label{defdeltainv}
\delta^{-1}(j,q) = (q^\vee, j^\vee).
\end{equation}

Moreover, for any $(i,p)$ in $\cS$ (resp., for any $(j,q)$ in $\cT$), the short exact sequence of $A$-module (\ref{Mip}) (resp. the strict short exact sequence in $\CTC_A$ (\ref{Njq})) is split.
 \end{proposition}

\begin{proof} 1) For any $(i,p)$ in $\cS,$ the associated short exact sequence of $A$-modules (\ref{Mip}) is split, since $M_1$ is projective. This implies that the diagram 
$$0 \lra M^\vee_1\stackrel{p^\vee}{\lra} M^\vee_2 \stackrel{i^\vee}{\lra} M^\vee_3 \lra 0$$
deduced from (\ref{Mip}) by duality is a split strict short exact sequence in $\CTC_A$.

This shows in particular that $(p^\vee, i^\vee)$ belongs to $\cT$.

2) Consider an element $(j,q)$ of $\cT$. 

If we apply the functor $\Hom_A^{\rm cont}(., A)$ to the strict short exact sequence (\ref{Njq}),
we get the following exact sequence of $A$-modules:
$$0 \lra N^\vee_3\stackrel{q^\vee}{\lra} N^\vee_2 \stackrel{j^\vee}{\lra} N^\vee_1.$$

Indeed the injectivity of $q^\vee$ follows from the surjectivity of $Q$, and the equality $\im q^\vee = \ker j^\vee$ means that the continuous $A$-linear maps from $N_2$ to $A$ which vanishes on $j(N_1)$ are in bijection --- by means of factorization through $q$ --- which the $A$-linear forms from $N_3$ to $A$: this follows from the fact that $q$ is continuous and open, of kernel $j(N_1)$.

The morphism $j^\vee: N^\vee_2 \lra N^\vee_1$ may be factorized as
\begin{equation}\label{factorjvee}
j^\vee = \iota \circ [j^\vee] : 
N^\vee_2 \stackrel{[j^\vee]}{\lra} \im j^\vee \stackrel{\iota}{\hlra} N^\vee_1,
\end{equation}
where $\iota$ denotes the inclusion morphism. Moreover
 the $A$-module $\im j^\vee,$ as any submodule of $N_1^\vee$, is  an object of $\CP_A$.

Finally the diagram 
\begin{equation}\label{dualpro}
0 \lra N^\vee_3\stackrel{q^\vee}{\lra} N^\vee_2 \stackrel{[j^\vee]}{\lra} \im j^\vee \lra 0
\end{equation}
is a short exact sequence in $\CP_A.$ 

According to part 1) of the proof, the short exact sequence 
(\ref{dualpro}) is split and determines by duality a split short exact sequence in $\CTC_A$. Moreover the factorization (\ref{factorjvee}) shows that $j = [j^\vee]^\vee \circ \iota^\vee$. Therefore the following diagram in $\CTC_A$ is commutative, and its lines are short strict exact sequences:
\begin{equation}\label{grodiag}
\begin{CD}
 0 @>>> (\im j^\vee)^\vee @>{[j^\vee]^\vee}>> N_2 @>q>> N_3 @>>> 0 \\
 @.            @AA{\iota^\vee}A                                    @|               @|           @. \\
 0 @>>>     N_1                   @>j>>                    N_2  @>q>>    N_3 @>>>0.
 \end{CD}
\end{equation}
This  implies that $\iota^\vee$ is an isomorphism in $\CTC_A$. In particular, the second line in  (\ref{grodiag}), like the first one, is a split short exact sequence in $\CTC_A$. 

Moreover, since the duality functor $.^\vee$ is an equivalence of category from $\CP_A$ to $\CTC_A$, this also proves that $\iota$ is an isomorphism in $\CP_A$. This establishes the equality $\im j^\vee = N_1^\vee$ and the exactness of 
$$0 \lra N^\vee_3\stackrel{q^\vee}{\lra} N^\vee_2 \stackrel{j^\vee}{\lra} N_1^\vee \lra 0.
$$
This shows that $(q^\vee, j^\vee)$ belongs to $\cS$. 

The fact that the maps between $\cS$ and $\cT$ defined by (\ref{defdelta}) and \ref{defdeltainv}) are inverse to each other is a straightforward consequence of the ``biduality" established in Proposition \ref{prop:dual}. \end{proof}

For later reference, we spell out some consequences of the results on short exact sequences in $\CP_A$ and $\CTC_A$ established in the previous proposition: 

\begin{proposition}\label{EF} Let $M$ and $M'$ be two objects in $\CP_A$ and let $N:= M^\vee$ and $N':=M'^\vee$ be  their duals in $\CTC_A$.

 Let $\alpha: M \lra M'$ be a morphism of $A$-modules and let 
$\beta := \alpha^\vee : N' \lra N$ 
denote the dual morphism, in $ \Hom_A^{\rm cont} (N', N)$.

1) The following two conditions are equivalent:

$\mathbf{E_1:}$The morphism $\beta$ is surjective and strict.

$\mathbf{E_2:}$The morphism $\alpha$ is injective and its cokernel is a projective $A$-module.

When these conditions are realized, $\im \alpha$ is a direct summand of $M'$ and  $\ker \beta$ is a closed submodule of $N$ supplemented in $\CTC_A$. 

2) The following two conditions are equivalent:

$\mathbf{F_1 :}$ The morphism $\beta$ is injective and strict, and the topological $A$-module $N/ \im \beta$ is an object of $\CTC_A.$

$\mathbf{F_2 :}$ The morphism $\alpha$ is surjective.

When these conditions are realized, $\ker \alpha$ is a direct summand of $M$ and $\im \beta$ is a closed submodule of $N$ supplemented in $\CTC_A$. \qed

\end{proposition}

\begin{corollary}\label{closedsupplCTC}
 A closed $A$-submodule $N'$ of some object $N$ in $\CTC_A$ is supplemented in $\CTC_A$ if and only if the quotient topological $A$-module $N/N'$ is an object of $\CTC_A$.
\qed \end{corollary}

\begin{corollary}\label{FSU}
Let $M$ be an object of $\CP_A$ and let $N:= M^\vee$ be the dual object in $\CTC_A$. 

Any submodule of $M$ (resp. of $N$) in $\cF\cS(M)$ (resp. in $\cU(N)$) is supplemented in $M$ (resp. in $N$). 

Moreover there is an inclusion reversing bijection
$$.^\perp : \cF\cS(M) \lrasim \cU(N).$$
It sends a module $M'$ in $\cF\cS(M)$ to
$$M'^\perp := \{ \xi \in N \mid \forall m \in M', \xi(m)=0 \}.$$
The inverse bijection sends an element $U$ in $\cU(N)$ to
$$U^\perp :=  \{ m \in M \mid \forall \xi \in U, \xi(m)=0 \}.$$

Moreover, when $U= M'^\perp,$ there exists a unique isomorphism of $A$-modules $I: N/U \lrasim M'^\vee$ such that, if we denote by $i_{M'}: M' \hookrightarrow M$ the inclusion morphism, the following diagram is commutative:
\begin{equation}\label{Iperp}
\begin{CD}
N @> = >> M ^\vee \\
@V{p_U}VV      @VV{i_{M'}^\vee}V \\
N/U @>{\stackrel{I}{\sim}}>> M'^\vee.
\end{CD}
\end{equation}
\qed \end{corollary}

\subsection{Strict morphisms and Conditions $\mathbf{Ded_1}, \mathbf{Ded_2}, \mathbf{Ded_3}$}\label{strictDed}

The significance of being strict for a morphism in $\CTC_A$ turns out to depend on which of the conditions $\mathbf{Ded_1}$, $\mathbf{Ded_2}$, or $\mathbf{Ded_3}$ the base ring $A$ satisfies.

In this subsection, we present diverse results that illustrate this point.  We will return on constructions of non-strict morphisms in the next section, devoted to examples.

\begin{proposition}\label{strictCTCDed1}
 If the Dedekind ring $A$ satisfies $\mathbf{Ded_1}$ --- \emph{that is, if $A$ is a field} --- then any morphism in $\CTC_A$ is strict.
 \end{proposition}
 
 \begin{proof} This follows from the description of morphisms in $\CTC_k$ in Proposition \ref{CTCk}. Indeed, with the notation of this Proposition, $\tilde{\phi} = 0 \oplus Id_N$ is clearly a strict morphism (of kernel $K \oplus \{0\}$ and of image $\{0\} \oplus N$) and consequently
$\phi = v^{-1} \tilde{\phi} u$ is also strict.
\end{proof}
 
  In Paragraph \ref{Pathnonstrict} \emph{infra},
we shall see that if $A$ satisfies $\mathbf{Ded_2}$ --- that is, if $A$ is a complete discrete valuation ring --- then there exists bijective morphisms in $\CTC_A$ that are not strict.

\begin{proposition}\label{strictCTCDed3}
When the Dedeking ring $A$ satisfies $\mathbf{Ded_3}$, 
 a morphism $\phi: N_1 \lra N_2$ in $\CTC_A$ is strict if and only if its image $\im \phi$ is closed in $N_2$. 


Moreover,  for any two objects $N_1$ and $N_2$ of $\CTC_A$, an $A$-linear map $\phi: N_1 \lra N_2$ is continuous (hence defines a morphism in $\CTC_A$) if and only if its graph is closed in $N_1 \oplus N_2$.
\end{proposition}

\begin{proof}
1) We have already observed that the image of any strict morphism is closed (for an arbitrary Dedekind ring $A$). 

Conversely, if a morphism $\phi: N_1 \lra N_2$ in $\CTC_A$ has its image $\im \phi$ closed in $N_2$, then we may form the following diagram in $\CTC_A$, which is an exact sequence of $A$-modules:
\begin{equation}\label{suitepastop}
0 \lra \ker \phi \stackrel{j}{\lra} N_1 \stackrel{\phi}{\lra} \im \phi \lra 0,
\end{equation}
where $j$ denote the inclusion morphism. By applying the functor $\Hom_A(., A)$ --- which coincides with $\Hom_A^{\rm cont}(.,A)$ on $\CTC_A$ when $A$ satisfies $\mathbf{Ded_3}$, as shown in Proposition \ref{prop:HomHom} --- to the exact sequence (\ref{suitepastop}), we obtain an exact sequence of $A$-modules:
\begin{equation*}
0 \lra (\im \phi)^\vee \stackrel{\phi^\vee}{\lra} N_1^\vee \stackrel{j^\vee}{\lra} (\ker \phi)^\vee.
\end{equation*}
Since any $A$-submodule of some object in $\CP_A$ is again an object in $\CP_A,$ we finally obtain the following exact sequence in $\CP_A$:
\begin{equation*}
0 \lra (\im \phi)^\vee \stackrel{\phi^\vee}{\lra} N_1^\vee \stackrel{j^\vee}{\lra} \im j^\vee \lra 0.
\end{equation*}
The equivalence of Conditions $\mathbf E_1$ and $\mathbf E_2$ in Proposition \ref{EF} finally shows that the morphism $\phi: N_1 \lra \im \phi$, already known to be surjective, is also strict. This shows that $\phi: N_1 \lra N_2$ is strict.

2) For any map $\phi: N_1 \lra N_2,$ its graph
$${\rm Gr}\, \phi := \{(n, \phi(n)), n \in \N \}$$ is also the inverse image of the diagonal $\Delta_{N_2}$ in $N_2 \times N_2$ by the map $(\phi, Id_{N_2}): N_1 \times N_2 \lra N_2 \times N_2$.

When $\phi$ is continuous, $(\phi, Id_{N_2})$ also is continuous, and 
$${\rm Gr}\, \phi = (\phi, Id_{N_2})^{-1}(\Delta_{N_2})$$
is closed, since the topology of $N_2$ is Hausdorff and accordingly $\Delta_{N_2}$ is closed in $N_2 \lra N_2 \times N_2$.

Conversely, if $\phi$ is a continuous $A$-linear map and if its graph  ${\rm Gr}\, \phi$ is closed in $N_1 \times N_2$, then ${\rm Gr}\, \phi$ defines an object of $\CTC_A$ by Proposition \ref{subCTC}, and the two projections
$$\pr_i : {\rm Gr}\, \phi \lra N_i, \,\,\, i=1,2$$
are morphisms in $\CTC_A$.

By construction, the map $\pr_1 : {\rm Gr}\, \phi \lra N_1$ is bijective. According to part 1) of the proof, it is therefore an isomorphism in $\CTC_A$, and therefore $\phi = \pr_2 \circ \pr_1^{-1}$ is finally a morphism in $\CTC_A$.
\end{proof}

 When the ring $A$ is countable --- notably when $A$ is the ring of integers of some number field --- Proposition \ref{strictCTCDed3} admits an alternative proof that does not rely on the ``automatic continuity" of morphisms of $A$-modules established for general Dedekind rings satisfying $\mathbf{Ded_3}$.
 
 Indeed, for any object $M$ of $\CTC_A,$ the topology  of $M$ may be defined by a complete metric (this follows from the countability assumption in Condition $\mathbf{CTC_2}$). If moreover $A$ is countable, the topological space $M$ is separable\footnote{This follows from Propositions \ref{prodfin} or \ref{Kapdual}, which also show that, conversely, if $M\neq 0$ and if the topological space $M$ is separable, $A$ is necessarily countable.} and therefore the additive group $(M,+)$ is a Polish topological group.  
 
 Remarkably, the Open Mapping and the Closed Graph Theorems  are valid for continuous morphisms of Polish topological groups, by a
classical theorem of Banach,\footnote{See  \cite{Banach31}. This theorem now appears as a special case of the ``théorème du graphe souslinien", concerning morphisms of Polish topological groups,  presented for instance in \cite{BourbakiTGII74}, Chapitre IX, §6, no. 8, Théorème 4.} and  immediately yield the  conclusions of Proposition \ref{strictCTCDed3} if $A$ is countable.

\begin{corollary} When the Dedekind ring $A$ satisfies $\mathbf{Ded_3}$, any morphism $\phi: N \lra N'$ in $\CTC_A$ whose cokernel $\coker \phi := N'/\phi(N)$ is a finitely generated $A$-module is an open map, and therefore a strict morphism.
 \end{corollary}

\begin{proof} Let $(n'_1, \ldots, n'_k)$ be a finite family of elements of $N_2$ the classes of which generate the $A$-module $\coker \phi$, and let
$$\tilde{\phi}: N \oplus A^{\oplus k} \lra N'$$
the continuous morphism of $A$-modules defined by the formula:
$$\tilde{\phi}(n, a_1,\cdots, a_k) := \phi(n) + a_1 n'_1 + \cdots + a_k n'_k$$
for any $n \in N$ and any $(a_1,\cdots, a_k) \in A^k.$

By construction, the morphism $\tilde{\phi}$ is surjective, and therefore, according to Proposition \ref{strictCTCDed3}, it is strict. In particular, it is an open map. Consequently, the map $\phi: N \lra N'$ --- which is the composition of $\tilde{\phi}$ and of the inclusion map $N \hlra N \oplus A^{\oplus k}$, that is clearly open --- is also open.
 \end{proof}

Let us finally indicate that, for any Dedekind ring $A$ that satisfies  $\mathbf{Ded_3}$, there exist continuous endomorphisms of the $A$-module $A^\N$ which are injective, with dense image, but are not strict (see for instance Proposition \ref{nonstrictDed3}, \emph{infra}).

\subsection{Strict injective morphisms and extensions of scalars}

The following proposition is stated for further reference. Its proof is straightforward and left to the reader. 

\begin{proposition}\label{strictinjective}
 Let $f:N' \lra N$ be a strict injective morphism in $\CTC_A$, and let $(U_i)_{i\in \N}$ be a defining sequence in $\cU(N).$
 
 1) The sequence $(U'_i)_{i \in \N} := (f^{-1}(U_i))_{i \in \N}$ is a defining sequence in $\cU(N').$ Moreover, for every $i \in \N,$ the morphism of (finitely generated projective) $A$-modules 
 $$f_i: N'_i:= N'/U'_i \lra N_i := N/U_i,$$
 defined by $f_i(x' + U'_i) = f(x') + U_i$ for any $x' \in M'$, is injective.
 
 When the topological $A$-modules $N$ and $N'$ are identified with the projective limits $\varprojlim_i N_i$ and $\varprojlim_i N'_i$, the morphism $f:N' \lra N$ gets identified with the morphism $\varprojlim_i f_i.$ 
 
 2) For any field extension $L$ of the fraction field $K$ of $A$, the topological $L$-modules $M'_L$ and $M_L$ may be identified with 
 $\varprojlim_i N_{i,L}$ and $\varprojlim_i N'_{i,L}$  and the morphism $f_L: N'_L \lra N_l$ with the morphism $\varprojlim_i f_{i,L}.$
 
 In particular, like the morphisms $f_{i,L}$, the morphism $f_L$ is injective. \qed
\end{proposition}
 
Let us indicate that, with the notation of Proposition \ref{strictinjective}, if the injective morphism $f: N' \lra N$ is not assumed to be strict, its base change $f_K$ may be non-injective. (For instance, the morphism $\beta_1: \Z^\N \lra \Z^\N$ in $\CTC_\Z$ considered in Proposition \ref{padicex}, 1), \emph{infra} is injective, but $\beta_{1\Q}: \Q^\N \lra \Q^\N$ admits the line $\Q.(p^{-k})_{k \in \N}$ as kernel.)
 
\section{Localization and descent properties}\label{LocDes} 

\subsection{}In this monograph, we shall not investigate systematically the descent properties satisfied by the categories $\CP_A$ and $\CTC_A$ or by the diverse ``infinite-dimensional vector bundles" defined in terms of them, and we shall content ourselves with a few observations in this paragraph and in paragraph \ref{desprovectpro} \emph{infra}.

The duality between the categories $\CP_.$ and $\CTC_.$ is easily seen to be compatible with the base change functors
$$\cdot \otimes_A B : \CP_A \lra \CP_B
\quad \text{and} \quad  \cdot \; \widehat{\otimes}_A B : \CTC_A \lra \CTC_B$$
associated to some ring morphism $A \lra B$ between Dedekind rings.

This allows one, when studying the localization or descent properties of these categories, to concentrate on the descent properties of the categories $CT_.$ of countably generated projective modules over Dedekind rings.

The descent properties of projectivity have actually been investigated in a general setting by Gruson and Raynaud in\cite{GrusonRaynaud71}. They notably prove that, \emph{for any commutative ring $R$ and any faithfully flat commutative $R$-algebra $R'$, some $R$-module $M$ is projective if} (and only if) \emph{the $R'$-module $M \otimes_R R'$ is projective}. (See \emph{loc. cit.}, Seconde partie, § 3.1. The importance of the criteria of projectivity in \cite{GrusonRaynaud71} for the study of ``infinite-dimensional vector bundles in algebraic geometry" has been emphasized by Drinfeld in \cite{Drinfeld2006}.)

In the next paragraphs, we simply spell out two simple consequences of this projectivity criterion, combined with basic faithfully flat descent. As this monograph concentrates on ``pro-vector bundles" rather than on ``ind-vector bundles", we formulate them in terms of the categories $\CTC_.$.

\subsection{}\label{LocDesUV} Let $U$ and $V$ be two open non-empty subsets of $\Spec A$ such that
$$\Spec A = U \cup V.$$
In other words, $U$ and $V$ are the complements
$$U:= \Spec A \setminus F \quad \text{and} \quad V:= \Spec A \setminus G$$
of two disjoint finite subsets $F$ and $G$ of the set $(\Spec A)_0$ of the non-zero prime ideals of $A$. Actually, $U,$ $V$, and $U \cap V$ define affine open subschemes of $\Spec A$.  If $K$ denotes the fraction field of $A$, and, for any $\fp \in (\Spec A)_0$, $v_\fp$ denotes the $\fp$-adic valuation of $K$, we have :
$$A_U := \Gamma(U, \cO_{\Spec A}) = \lbrace x \in K \mid \forall \fp \in (\Spec A)_0\setminus F, v_\fp (x) \geq 0 \rbrace,$$
and similar descriptions of $A_V := \Gamma(V, \cO_{\Spec A})$ and $A_{U\cap V} := \Gamma(U\cap V, \cO_{\Spec A})$, with $F$ replaced by $G$ and by $F \cup G$, respectively.

With this notation, \emph{the datum of an object $N$ in $\CTC_A$ is equivalent to the data of some objects $N_U$ of $\CTC_{A_U}$ and $N_V$ of $\CTC_{A_V}$ and of some ``glueing isomorphism" in} $\CTC_{A_{U \cap V}}$:
$$N_U \widehat{\otimes}_{A_U} A_{U\cap V} \lrasim N_V\widehat{\otimes}_{A_V} A_{U \cap V}.$$
Indeed, the equivalence maps $N$ to the topological modules $N_U := N \widehat{\otimes}_{A} A_U$ and $N_V := N \widehat{\otimes}_{A} A_V$ and to the glueing isomorphism deduced from the canonical isomorphisms of $A$-algebras:
$$A_U\otimes_{A_U} A_{U\cap V} \simeq A_{U\cap V}  \simeq A_V\otimes_{A_V} A_{U\cap V}.$$

\subsection{}\label{LocDesUp}
Let $U := \Spec A \setminus F$ be a non-empty open subscheme of $\Spec A.$ For any $\fp$ in $F,$ we may consider the complete discrete valuation ring $\hat{A}_\fp$ defined as the $\fp$-adic completion of $A$; its fraction field $\hat{K}_\fp$ is the completion of $K$ with respect to the $\fp$-adic valuation $v_\fp$.

Then \emph{the datum of an object $N$ in $\CTC_A$ is equivalent to the data of some object $N_U$ of $\CTC_{A_U}$ and, for every $\fp \in F,$ of an object $N_\fp$ in $\CTC_{\hat{A}_\fp}$ and of some glueing isomorphism in $\CTC_{\hat{K}_\fp}$ :}
$$N_U \widehat{\otimes}_{A_U} \hat{K}_\fp \lrasim N_\fp \widehat{\otimes}_{\hat{A}_\fp} \hat{K}_\fp.$$
The equivalence maps $N$ to the topological modules $N_U := N \widehat{\otimes}_{A} A_U$ and $N_\fp:= N \widehat{\otimes}_{A} \hat{A}_\fp$ and to the glueing morphisms deduced from the canonical isomorphism of $A$-algebras:
$$A_U \otimes_{A_U} \hat{K}_\fp \simeq \hat{K}_\fp \simeq \hat{A}_\fp \otimes_{\hat{A}_\fp} \hat{K}_\fp.$$

\section{Examples}
\label{Path} 

In this section, we discuss some simple examples of objects and morphisms in the categories $\CP_A$ and $\CTC_A$ when $A$ is the ring $\Z_p$ or $\Z$. 

These examples should  make clear that non-strict morphisms in the categories $\CTC_{\Z_p}$ and $\CTC_\Z$ are not ``pathologies" but  occur ``naturally", and that
 exact sequences in the categories $\CP_\Z$ and $\CTC_\Z$ and their duality properties must be handled with some care.

We denote by $p$  a prime number and by $\vert.\vert_p$ the $p$-adic norm of $\Z_p$
on  the ring $\Z_p$ of $p$-adic integers, defined by $\vert p^n u\vert_p := p^{-n}$ for every $n\in \N$ and every $u \in \Z_p^\times.$

\subsection{Subobjects and duality in $\CP_\Z$ and $\CTC_\Z$} 
Besides Condition $\mathbf{E_2}$, we may consider the following condition:

$\mathbf{E'_2: }$\emph{The morphism $\alpha$ is injective and its image is saturated
in $M'$.}

\noindent Similarly, besides Condition $\mathbf{F_1}$, we may consider the following condition:

$\mathbf{F'_1:}$ \emph{The morphism $\beta$ is injective and its image is closed and saturated
in $N_2$.}

\noindent Clearly, Condition $\mathbf{E_2}$ implies Condition $\mathbf{E'_2}$, and Condition $\mathbf{F_1}$ implies Condition $\mathbf{F'_1}$. 

However, the converse implications do not hold in general. This is demonstrated, when $A= \Z$, by  Proposition \ref{padicex} below.

Let $(\epsilon_n)_{n \in \N}$ denote the canonical basis of $\Z^{(\N)}$, defined by
$\epsilon_n(k) := \delta_{nk}$
for every $(n,k) \in \N^2.$ We shall identify the dual in $\CTC_\Z$ of the object $M:= \Z^{(\N)}$ of $\CP_\Z$ with the $\Z$-module $N:= \Z^{\N}$ equipped with the product topology of the discrete topology on $\Z$, by means of the map:
\begin{equation}\label{dualfree}
\begin{array}{rcl}
\Hom_\Z(\Z^{(\N)}, \Z)  & \lrasim   & \Z^\N   \\
\xi   & \longmapsto  & (\xi(\epsilon_n))_{n \in \N}.   
\end{array}
\end{equation}

\begin{proposition}\label{padicex}

1) If we define two morphisms of $\Z$-modules 
 $\alpha_{1}: \Z^{(\N)} \lra \Z^{(\N)}$
 and 
 $\pi_1: \Z^{(\N)} \lra \Z[1/p]$
 by
$$
 \alpha_1(\epsilon_n) := \epsilon_n -p \epsilon_{n+1} \mbox{ and } 
 \pi_1(\epsilon_n) := \frac{1}{p^{n}}$$ for every $n \in \N$,
 then $\alpha_1$ and $\pi_1$ fit into the following short exact sequence:
 \begin{equation}\label{sesalpha1}
0 \lra \Z^{(\N)} \stackrel{\alpha_1}{\lra} \Z^{(\N)} \stackrel{\pi_1}{\lra} \Z[1/p] \lra 0.
\end{equation}

 Moreover the dual morphism
 $\beta_1:= \alpha_1^\vee \in \Hom_\Z^{\rm cont}(\Z^\N, \Z^\N)$
is injective, and if we define a $\Z$-linear map
$\sigma_1: \Z^\N \lra \Z_p/\Z$
by $$\sigma_1((y_k)_{k\in \N}):= \left[\sum_{k \in \N} p^k y_k\right],$$
then the following diagram is a short exact sequence:
 \begin{equation}\label{sesbeta1}
0 \lra \Z^\N \stackrel{\beta_1}{\lra} \Z^\N \stackrel{\sigma_1}{\lra} \Z_p/\Z \lra 0.
\end{equation}

 2)  If we define two morphisms of $\Z$-modules 
 $\alpha_{2}: \Z^{(\N)} \lra \Z^{(\N)}$
 and 
 $\pi_2: \Z^{(\N)} \lra \Z[1/p]/\Z$
 by
 \begin{align*}
 \alpha_2(\epsilon_n) & := -p \epsilon_0   & \text{when $n=0$,} &\\
& :=  \epsilon_{n-1} -p \epsilon_n &  \text{when  $n\geq 1$,}& 
\end{align*}
 and by
 $$\pi_2(\epsilon_n) := \left[\frac{1}{p^{n+1}}\right] \mbox{ for every $n \in \N$,}$$
 then $\alpha_2$ and $\pi_2$ fit into the following short exact sequence:
  \begin{equation}\label{sesalpha2}
0 \lra \Z^{(\N)} \stackrel{\alpha_2}{\lra} \Z^{(\N)} \stackrel{\pi_2}{\lra} \Z[1/p]/\Z \lra 0.
\end{equation}
 
 Moreover the dual morphism
 $\beta_2:= \alpha_2^\vee$ in $\Hom_\Z^{\rm cont}(\Z^\N, \Z^\N)$
is injective, and if we define a continuous $\Z$-linear map
$\sigma_2: \Z^\N \lra \Z_p$
by $$\sigma_2((y_k)_{k\in \N}):= \sum_{k \in \N} p^k y_k,$$
then the following diagram is a short exact sequence:
 \begin{equation}\label{sesbeta2}
0 \lra \Z^\N \stackrel{\beta_2}{\lra} \Z^\N \stackrel{\sigma_2}{\lra} \Z_p \lra 0.
\end{equation}
\end{proposition} 

Observe that $\alpha_1$ satisfies $\mathbf{E'_2}$ and not $\mathbf{E_2}$, since $\beta_1$ does not satisfy $\mathbf{E_1}$, and that $\beta_2$ satisfies $\mathbf{F'_1}$ and not $\mathbf{F_1}$, since $\alpha_2$  does not satisfies $\mathbf{F_2}$.

As already mentioned, Proposition \ref{padicex} also demonstrates how the categories $\CP_\Z$ and $\CTC_\Z$ are ``badly behaved" with respect to quotients. 

\begin{proof} 1) We use the identification of $\Z$-modules
$$
\begin{array}{rcl}
 \Z^{(\N)} & \lrasim   & \Z[X]   \\
 \epsilon_n & \longmapsto   & X^n  
\end{array}
\mbox{ and }
\begin{array}{rcl}
 \Z^\N & \lrasim  & \Z[[X]]   \\
  (y_k)_{k \in \N} & \longmapsto  & \sum_{k \in \N} y_k X^k.  
\end{array}
$$
The isomorphism (\ref{dualfree}) becomes the isomorphism
$$\Hom_\Z(\Z[X], \Z) \lrasim \Z[[X]]$$
induced by the residue pairing
\begin{equation}\label{residuepairing}
\begin{array}{rcl}
  \Z[[X]] \times \Z[X] & \lra  & \Z  \\
(f,P)  & \longmapsto   & {\rm Res}_{X=0} [f(X) P(1/X) dX/X].   
\end{array}
\end{equation}

The diagram (\ref{sesalpha1}) may be written 
\begin{equation}\label{sesalpha1'}
0 \lra \Z[X] \stackrel{\alpha_1}{\lra} \Z[X] \stackrel{\pi_1}{\lra} \Z[1/p] \lra 0
\end{equation}
where 
\begin{equation}\label{defalpha1}
\alpha_1(P):= (1-pX) P
\end{equation}
and 
$$\pi_1(P) = P(1/p).$$
The exactness of (\ref{sesalpha1'}), hence of (\ref{sesalpha1}), follows from the basic properties of polynomials with integer coefficients.

Moreover the expressions (\ref{defalpha1}) for $\alpha_1$ and (\ref{residuepairing}) for the duality pairing between $\Z[[X]]$ and $\Z[X]$ show that the dual morphism
$$\beta_1:= \alpha_1^\vee: \Z[[X]] \lra \Z[[X]]$$
is given by
\begin{equation}\label{defbeta1}
\beta_1(f) = (1-p/X) f + pf(0)/X.
\end{equation}

The vanishing of $\beta_1(f)$ therefore implies the equality in $\Q[[X]]$:
$$f(X) = f(0) (1-p^{-1} X)^{-1} = f(0) \sum_{k\in \N} p^{-k} X^k.$$
Clearly, the only $f$ in $\Z[[X]]$ satisfying this condition is $f=0,$
and accordingly $\beta_1$ is injective.

To complete the proof of the exactness of (\ref{sesbeta1}), we shall use the following easy lemma:

\begin{lemma}\label{evalpinv}
The evaluation morphism 
\begin{equation}\label{defevalp}
\begin{array}{rrcl}
\eta : &\Z[[X]]  & \lra   & \Z_p   \\
& f & \lmt  & f(p)   
\end{array}
\end{equation}
is surjective. Its kernel is $(X-p)\, \Z[[X]].$
\end{lemma}

\begin{proof}[Proof of Lemma \ref{evalpinv}]
 The surjectivity of $\eta$ is clear, and the inclusion $(X-p) \Z[[X]] \subset \ker \eta$ also.
 
 To establish the converse inclusion, observe that the kernel of the evaluation map
 \begin{equation*}
\begin{array}{rrcl}
\eta_{\Z_p} : &\Z_p[[X]]  & \lra   & \Z_p   \\
& f & \lmt  & f(p)   
\end{array}
\end{equation*}
is $(X-p) \Z_p[[X]]$, say by the Weierstrass preparation theorem. 

Therefore any element $f$ of $\Z[[X]]$ in the kernel of $\eta$ may be written
\begin{equation}\label{fpg}
f = (X-p) g
\end{equation}
for some $g$ in $\Z_p[[X]].$
The polynomial $(X-p)= -p(1-X/p)$ is a unit in $\Z[1/p][[X]]$, and the equation (\ref{fpg}) shows that the series $g$, seen  as an element of $\Q_p[[X]]$, actually belongs to 
$\Z[1/p][[X]]$. Consequently $g$ belongs to
$\Z_p[[X]] \cap \Z[1/p][[X]] = \Z[[X]].$ 
\end{proof}

The morphism $\sigma_1: \Z[[X]] \lra \Z_p/\Z$ maps $g\in \Z[[X]]$ to the class of $g(p)$ in $\Z_p/\Z$. It is clearly surjective, and to establish the exactness of (\ref{sesbeta1}), we are indeed left to show that, for any $g \in \Z[[X]],$  the following two conditions are equivalent:

(i) there exists $f \in \Z[[X]]$ such that
$$ g = (1-p/X) f + pf(0)/X;$$

(ii) the element $g(p)$ of $\Z_p$ belongs to $\Z.$

When (i) holds, then $g(p)=f(0)$ and therefore (ii) also holds. Conversely, when (ii) holds, then $g-g(p)$ is an element of $\Z[[X]]$ in the kernel of the evaluation morphism (\ref{defevalp}) and, for some $h \in \Z[[X]],$ we have
$$g-g(p) = (X-p) h.$$ Therefore condition (i) is satisfied by 
$f:= Xh + g(p).$

2) Similarly the exact sequence (\ref{sesalpha2}) and (\ref{sesbeta2}) may be written as
  \begin{equation}\label{sesalpha2'}
0 \lra \Z[X] \stackrel{\alpha_2}{\lra} \Z[X] \stackrel{\pi_2}{\lra} \Z[1/p]/\Z \lra 0,
\end{equation}
where 
$$\alpha_2(P) := \frac{1-pX}{X} P - \frac{P(0)}{X} = \frac{P-P(0)}{X} -p P$$
and 
$$\pi_2(P) = (1/p).P(1/p) \mod \Z,$$
and as
 \begin{equation}\label{sesbeta2'}
0 \lra \Z[[X]]\stackrel{\beta_2}{\lra} \Z[[X]] \stackrel{\sigma_2}{\lra} \Z_p \lra 0,
\end{equation}
where
$$\beta_2(f) = \alpha_2(f) := (X-p) \, f \mbox{ and } \sigma_2 (f) := f(p).$$

We leave the proof of the exactness of (\ref{sesalpha2'}) as an elementary exercise on polynomials with integer coefficients. The exactness of (\ref{sesbeta2'}) is basically the content of Lemma \ref{evalpinv}.
\end{proof}

\subsection{A non-strict bijective morphism in $\CTC_{\Z_p}$}\label{Pathnonstrict}

We  want to point out that, when the base ring $A$ satisfies $\mathbf{Ded_2}$ --- that is, when $A$ a complete discrete valuation ring --- the Open Mapping and the Closed Graph Theorems (as stated in Proposition \ref{strictCTCDed3} for a Dedekind ring $A$ satisfying $\mathbf{Ded_3}$) do not hold. 

For definiteness, let us assume that $A = \Z_p,$ and let $N$ be the object of $\CTC_{\Z_p}$ defined as $N:= \Z_p^\N$ equipped with the product of the discrete topology on each factor $\Z_p.$

\begin{lemma}\label{Graphxi}
 Let $(\xi_n)_{n\in \N}$ be an element of $\Z_p^\N$ such that $\lim_{n \ra + \infty} \vert \xi_n\vert_p = 0.$
 For any $x :=(x_n) \in N= \Z_p^\N,$ the series 
 $\xi(x) := \sum_{n\in \N} \xi_n x_n$ converges in $\Z_p$ equipped with $\vert. \vert_p$, and defines a $\Z_p$-linear map
 $\xi : N \lra \Z_p.$
 
 The map $\xi$ belongs to $\Hom_{\CTC_{\Z_p}}(N, \Z_p)$ if and only if  $(\xi_n)_{n\in \N}$ belongs to $\Z_p^{(\N)}$.
 
 The graph ${\rm Gr}\, \xi$ of $\xi$ is closed in the topological module $N \oplus \Z_p$ of $\CTC_{\Z_p}$.
\end{lemma}
\begin{proof}
 All the assertions are immediate, but possibly the last one. 
 
 To prove that ${\rm Gr}\, \xi$ of $\xi$ is closed in $N \oplus \Z_p = \Z_p^\N \oplus \Z_p$ equipped with the product of the discrete topology on each factor $\Z_p$, consider $x$ in $N$ and $a$ in $\Z_p$ such that $(x,a) \notin {\rm Gr}\, \xi,$
or equivalently, such that $\xi(x) \neq a.$

There exists a positive integer $n_0$ such that, for every $n \in \N_{> n_0},$
$$\vert \xi_n \vert_p < \vert \xi(x) - a \vert_p.$$
Then 
$U := \{0\}^{\{0,\dots,n_0\}} \times \Z_p^{\N_{>n_0}}$
is an open neighborhood of $0$ in $N= \Z_p^\N$ such that, for any $u \in U,$
$$\vert \xi(u) \vert_p < \vert \xi(x) - a \vert_p.$$

Consequently, for any $\tilde{x} \in x + U,$
$$\vert \xi(\tilde{x})-a\vert_p =\vert \xi(x) - a \vert_p \neq 0,$$  
and $(x, a) + U \oplus \{0\}$ is an open neighborhood of $(x,a)$ in $N \oplus \Z_p$ disjoint of ${\rm Gr}\, \xi$.  
\end{proof}

Let us keep the notation of Lemma \ref{Graphxi}.

According to Proposition \ref{subCTC}, the $A$-module ${\rm Gr}\, \xi$, equipped with the topology induced by the one of $N \oplus \Z_p$, becomes an object of $\CTC_{\Z_p}$. Clearly the first projection
$${\rm pr}_{1\mid {\rm Gr}\, \xi} :{\rm Gr}\, \xi \lra N$$
defines a continuous bijective morphism of topological $A$-modules. By construction, it is a homeomorphism --- or equivalently, a strict morphism --- precisely when $\xi: N \lra \Z_p$ is continuous.

This shows that, for any $(\xi_n)_{n \in \N}$ in $\Z_p^\N \setminus \Z_p^{(\N)}$ such that   
$\lim_{n \ra + \infty} \vert \xi_n\vert_p = 0,$ the map ${\rm pr}_{1\mid {\rm Gr}\, \xi}$ is a non-strict bijective morphism in $\Hom_{\CTC_{\Z_p}} ( {{\rm Gr}\, \xi}, N).$

\subsection{Non-strict injective morphisms in $\CTC_A$ when $A$ satisfies $\mathbf{Ded_3}$}

The construction in Proposition \ref{padicex}, 2), may easily be extended to a more general setting and provides, for any Dedekind ring $A$ that satisfies $\mathbf{Ded_3}$, examples of  injective morphisms in $\CTC_A$ that are not strict. 

The ring $A[[X]]$ of formal series with coefficients in $A$ will be equipped with its natural prodiscrete topology (say, defined by its identification with $\varprojlim_n A[X]/(X^n)$). Then it becomes  an object of $\CTC_A$.  

By mimicking the arguments in the proof of Proposition \ref{padicex}, 2), one easily establishes the following proposition. (To  prove its assertion 3), choose some non-zero prime ideal $\fp$ dividing $a$, and consider the evaluation map $g\mapsto g(a)$ from $A$ to the $\fp$-adic completion $\hat{A}_{\fp}$ of $A$. On  the image $\phi_a(A[[X]])$ of $\phi_a,$ this evaluation map  takes values in $A$.) 
We leave the details of its proof to the reader.

\begin{proposition}\label{nonstrictDed3}
For any $a \in A,$ we define a continuous $A$-linear map
$$\phi_a : A[[X]] \lra A[[X]]$$
by letting 
$$\phi_a (f) := (1- a/X)f + a f(0)/X = f - a\,(f-f(0))/X.$$

1) The map $\phi_a$ is injective if and only if $a \notin A^\times.$ 

2) For any $a$ in $A$, the map $\phi_a$ sends $A[X]$ bijectively onto $A[X]$. In particular, its image $\phi_a(A[[X]])$ is dense in $A[[X]]$.

3) If $A$ satisfies $\mathbf{Ded_3}$, then, for any $a \in A \setminus (A^\times \cup \{0\}),$ the map $\phi_a$ is not surjective. \qed
 \end{proposition}

\medskip

\chapter{Ind- and pro-Hermitian vector bundles over arithmetic curves}\label{infiniteHermitian}

\medskip

In this chapter, we introduce diverse categories of infinite dimensional Hermitian vector bundles over an arithmetic curve $\Spec \OK$ defined by the ring of integers $\OK$ of some number field $K$. These categories are constructed from the categories $\CP_A$ and $\CTC_A$ investigated in the previous chapters, specialized to the case of the Dedekind ring $A = \OK$ (of  type $\mathbf{Ded_3}$), by enriching the objects and the morphisms by ``Hermitian data". 

In the applications to Diophantine geometry, we will be mainly interested in \emph{pro-Hermitian} vector bundles: their underlying ``algebraic" objects will be the topological $\OK$-modules in $\CTC_\OK.$ In the basic case where $\OK = \Z,$ these are precisely the pro-Euclidean lattices described in the Introduction (see  \ref{debut}). As already indicated in \emph{loc. cit.}, they admit alternative descriptions, either (i) in terms of objects $\Eh$ of $\CTC_\OK$ and of Hilbert spaces $E_\sigma^\hilb$ densely embedded in the completed tensor products $\Eh_\sigma \simeq \Eh \hat{\otimes}_\sigma \C$ associated to the diverse complex embeddings $\sigma: K \hra \C,$ or (ii) in terms of projective systems    
$$\Eb_{\bullet} : \Eb_0 \stackrel{q_0}{\longleftarrow}\Eb_1 \stackrel{q_1}{\longleftarrow}\dots \stackrel{q_{i-1}}{\longleftarrow}\Eb_i \stackrel{q_i}{\longleftarrow} \Eb_{i+1} \stackrel{q_{i+1}}{\longleftarrow} \dots$$
of surjective admissible morphisms of Hermitian vector bundles over $\Spec \OK.$

The equivalence of these descriptions is elementary, but quite useful in practice. Other constructions described in this chapter are mostly formal, and their details could be skipped at first reading.

The last paragraphs of this chapter are devoted to diverse examples, that are quite elementary but should convey some  feeling of the ``concrete significance" of pro-Hermitian vector bundles and of the technical subtleties one may encounter when handling them.

Notably, in Section \ref{ArHB}, we introduce the ``arithmetic Hardy spaces" $\Hbh(R)$ and the ``arithmetic  
Bergman spaces" $\Bbh(R)$. These pro-Euclidean lattices  constitute the archetypes of the pro-Hermitian vector bundles over arithmetic curves that we shall investigate in the sequel, when applying the formalism developed here to study the interaction of complex analytic geometry and formal geometry (over the integers) in diverse Diophantine settings. 

Besides, in Section \ref{InjSurj} we show, by explicit examples, that the injectivity or surjectivity properties of the morphisms of topological $\OK$-modules and of complex Fr\'echet and Hilbert spaces underlying some morphism of pro-Hermitian vector bundles are in general rather subtly related. 

\bigskip

We denote by $K$ a number field, by $\OK$ its ring of integers, and by $\pi: \Spec \OK \lra \Spec \Z$ the morphism of schemes from $\Spec \OK$ to $\Spec \Z.$ .

\section{Definitions}

\subsection{Ind-Hermitian vector bundles}\label{indherdef}
We define an \emph{ind-Hermitian vector bundle} over the arithmetic curve  $\Spec \OK$ as a pair 
$$\Fb := (F, (\Vert. \Vert_\sigma)_{\sigma: K\hra \C})$$
where $F$ is an object of $\CP_{\OK}$ --- namely a countably generated projective $\OK$-module --- and $(\Vert. \Vert)_{\sigma: K\hra \C}$ is a family of prehilbertian norms on the complex vector spaces
$$F_\sigma := F \otimes_\sigma \C$$
deduced from the $\OK$-module $E$ by the base change $\sigma: \OK \lra \C$. Moreover, the family $(\Vert. \Vert)_{\sigma: K\hra \C}$ is required to be invariant under complex conjugation.\footnote{Namely, for any field embedding $\sigma: K \hra \C$, the norm $\Vert.\Vert_\sigma$ on $F_\sigma$ and the norm $\Vert.\Vert_{\overline{\sigma}}$ on $F_{\overline{\sigma}}$ attached to the complex conjugate embedding $\overline{\sigma}$ coincide through the $\C$-antilinear isomorphism
$$
\begin{array}{rcl}
  F\otimes_\sigma \C & \simeq   & F\otimes_{\overline{\sigma}} \C   \\
e \otimes \lambda &\longmapsto & \overline{ e \otimes \lambda} := e \otimes \overline{\lambda}. 
\end{array}.
$$
In other words, $\Vert \overline{v} \Vert_{\overline{\sigma}} = \Vert v \Vert \mbox{ for every $v \in F_\sigma.$ }$}

An \emph{isometric isomorphism} $\phi: \Fb \lra \Fb'$ between two ind-Hermitian vector bundles
$\Fb := (F, (\Vert. \Vert)_{\sigma: K\hra \C})$
and 
$\Fb' := (F', (\Vert. \Vert')_{\sigma: K\hra \C})$ over $\Spec \OK$ is an isomorphism of $\OK$-modules $\phi: F \lrasim F$ such that, for every embedding $\sigma: K \hra \C,$ the $\C$-linear isomorphism $\phi_\sigma: F_\sigma \lrasim F'_\sigma$ is isometric with respect to the norms $\Vert.\Vert_\sigma$ and $\Vert.\Vert'_\sigma.$

\subsection{Pro-Hermitian vector bundles}\label{proinitial}
By definition,  a \emph{pro-Hermitian vector bundle over $\Spec \OK$} is the data
$$\Ebh:= (\hE, (\Eb_U)_{U \in \cU(\hE)})$$
of an object $\hE$ of $\CTC_{\OK}$ and, for any open saturated $\OK$-submodule $U$ of $\hE$, of a structure of Hermitian vector bundle over $\Spec \OK$
$$\Eb_U := (E_U, (\Vert.\Vert_{U,\sigma})_{\sigma:K \hra \C})$$
on the finitely generated projective $\OK$-module 
$E_U := \hE/U.$

 Moreover, for any two open saturated $\OK$-submodules $U$ and $U'$ of $\hE$ such that $U \subset U',$
the surjective morphism of $\OK$-modules
$p_{U' U}: E_U \lra E_{U'}$
is required to define a surjective \emph{admissible} morphism of Hermitian vector  bundles from $\Eb_U$ onto $\Eb_{U'}$. 

An \emph{isometric isomorphism} $\psi: \Ebh \lra \Ebh'$ between two pro-Hermitian vector bundles $$\Ebh= (\hE, (\Eb_U)_{U \in \cU(\hE)})$$
and 
$$\Ebh':= (\hE', (\Eb'_{U'})_{U' \in \cU(\hE')})$$
is an isomorphism $\psi : \Eh \lrasim \Eh'$ of topological\footnote{According to Proposition \ref{prop:HomHom}, any isomorphism $\psi : \Eh \lrasim \Eh'$ of $\OK$-modules is actually an isomorphism of topological $\OK$-modules.} $\OK$-modules such that, for any $U$ in $\cU(\Eh)$, of image $U':= \psi(U)$ (necessarily in $\cU(\Eh')$), the induced isomorphism of $\OK$-modules 
$$\psi_U: E_U := \Eh/U \lra E'_{U'} := \Eh'/U'$$
defines an isometric isomorphism of Hermitian vector bundles from $\Eb_U$ onto $\Eb'_{U'}.$
 
\subsection{The complex topological vector spaces associated to a pro-Hermitian vector bundle}\label{topologicalcomplexpro}

Let $\Ebh:= (\hE, (\Eb_U)_{U \in \cU(\hE)})$ be a pro-Hermitian vector bundle over $\Spec \OK$, as above, and let $\sigma: K \hlra \C$ be a field embedding.  

\medskip

\subsubsection{}\label{limnaive} We may apply the functor 
$$.\hat{\otimes}_{\OK, \sigma} : \CTC_{\OK} \lra \CTC_\C$$ to the topological $\OK$-module $\hE.$
We thus define the completed  tensor product
$$
\hE_\sigma := \hE \hat{\otimes}_{\OK, \sigma} \C.$$
By its very definition, $\hE_\sigma$ may be identified with an inverse limit of finite dimensional complex vector spaces equipped with the discrete topology:
$$\hE_\sigma \simeq \varprojlim_{U\in \cU(\hE)} E_{U,\sigma}.$$

Besides this ``pro-discrete" topology, which makes it an object of $\CTC_\C,$ the complex vector space $\hE_\sigma$ also admits a canonical separated and locally convex topology : it is defined by  taking the projective limit $\varprojlim_{U\in \cU(\hE)} E_{U,\sigma}$ that defines $\hE_\sigma$ in the category of locally convex complex vector spaces, when each finite dimensional complex vector space $E_{U,\sigma}$ is equipped with its  usual (separated and locally convex) topology. 

Equipped with this topology, $\hE_\sigma$ is a nuclear Fr\'echet space. Actually,  in the category $\CTC_\C$, there exists an isomorphism
$\phi : \Eh_\sigma \lrasim \C^I$
for some countable set $I$ (where $\C^I$ is equipped with the product of the discrete topology on each factor $\C$), and any such isomorphism is an isomorphism of complex locally convex vector spaces,  when $\Eh_\sigma$ (resp., $\C^I$) is equipped with its natural locally convex topology (resp., with the product of the usual topology on each factor $\C$).

From now on, the completed tensor products $\hE_\sigma$ associated to some pro-Hermitian vector bundle $\Ebh$ (or more generally to some object $\Eh$ in $\CTC_\OK$) will be always be endowed with its canonical topology of complex Fr\'echet space. 

We shall denote by
$$p_U: \Eh \lra \Eh/U =: E_U$$
the quotient map, and by $$p_{U,\sigma}: \Eh_\sigma \lra E_{U,\sigma}$$ its ``completed complexification".

\subsubsection{}\label{limHilb}  Observe that the 
Hermitian vector spaces  $\Eb_{U,\sigma}$ and the ``admissible" $\C$-linear maps $p_{U' U,\sigma}: E_{U,\sigma} \lra E_{U',\sigma}$ also constitute a projective system in the category  the objects of which are the complex normed spaces and the morphisms, the continuous linear maps \emph{of operator norm $\leq 1$.} 

This projective system admits a limit in this category, that may be described as follows.  
 
 Its underlying $\C$-vector space is  the subspace
 of 
  \begin{align*}\Eh_\sigma & :=\varprojlim_{U\in \cU(\hE)} E_{U,\sigma} \\
 & =\Big\{(x_U)_{U \in \cU(\hE)}\in \prod_{U \in \cU(\hE)} E_{U,\sigma} \mid \mbox{for any $(U,U') \in \cU(\hE)^2$},   U \subset U' \Longrightarrow p_{U'U}(x_{U})= x_{U'} \Big\}
 \end{align*}
 defined by ``uniformly bounded" elements, namely: 
$$
E_\sigma^{\hilb}  := {\varprojlim_U}^{\hilb} E_{U,\sigma} 
= \Big\{ (x_U)_{U \in \cU(\hE)}\in \varprojlim_U E_{U,\sigma} \mid
\sup_{U\in \cU(\hE)} \Vert x_U \Vert_{\Eb_{U,\sigma}} < +\infty \Big\}. 
$$
Its norm is the norm $\Vert.\Vert_{E_\sigma^{\hilb}}$  defined by the equality 
\begin{equation}\label{defhilpro}
\Vert x  \Vert_{E_\sigma^{\hilb}}
 := \sup_{U\in \cU(\Eb)} \Vert x_U \Vert_{\Eb_{U,\sigma}}
 \end{equation}
  for any element $x=(x_U)_{U \in \cU(\hE)}$ in $E^\hilb_\sigma$. 
 Actually $\Vert x_U \Vert_{\Eb_{U,\sigma}}$ is a non-decreasing function of $U \in \cU(\Eh),$ and therefore we also have:
$$\Vert x  \Vert_{E_\sigma^{\hilb}}
 = \lim_{U\in \cU(\Eb)} \Vert x_U \Vert_{\Eb_{U,\sigma}}.$$ 
 
 The following proposition is a straightforward consequence of the definitions, and its proof is left to the reader:

\begin{proposition} Equipped with the norm $\Vert.\Vert_{E_\sigma^{\hilb}}$, the complex vector space $\hE_\sigma^{\hilb}$ becomes a separable Hilbert space. For any $U \in \cU(\Eh),$ the map
$$p_{U,\sigma \mid E^\hilb_\sigma} : E_\sigma^\hilb  \lra E_{U,\sigma}$$
is a ``co-isometry". In other words, the Hermitian norm $\Vert . \Vert_{\Eb_U, \sigma}$ coincides with the quotient norm deduced from the Hilbertian norm $\Vert.\Vert_{E_\sigma^{\hilb}}$ by means of the surjective $\C$-linear map $p_{U,\sigma \mid E^\hilb_\sigma}.$

Moreover, when  $\hE_\sigma^{\hilb}$ is equipped with its topology of Hilbert space, and $\hE_\sigma$  with its canonical topology of separated locally convex complex vector space,
 the inclusion morphism
$$i_\sigma: E_\sigma^{\hilb} \longrightarrow \hE_\sigma$$
is continuous with dense image.    \qed
\end{proposition}

\medskip 

Finally observe that the constructions \ref{limnaive} and \ref{limHilb} are clearly compatible with isometric isomorphisms of pro-Hermitian vector bundles. 

\subsection{An alternative description of pro-Hermitian vector bundles}\label{subsubs:Alternatpro}
\subsubsection{}  The previous constructions lead us to the following alternative definition of pro-hermi\-tian vector bundles over $\Spec \OK$, which turns out to be more flexible as their initial definition in terms of projective systems of (finite dimensional) Hermitian  vector bundles.
 
 We may define  a pro-Hermitian vector bundle over $\Spec \OK$  as the data
 \begin{equation}\label{alternativepro}
 \Ebh:= (\hE, \left((E_{\sigma}^{\hilb},\Vert. \Vert_\sigma, i_\sigma \right)_{\sigma: K \hra \C}),
 \end{equation}
 where $\hE$ is an object of $\CTC_{\OK}$ and where, for every field embedding $\sigma: K \hra \C$, $E_{\sigma}^{\hilb}$ denotes a complex Hilbert space, $\Vert. \Vert_\sigma$ its norm, and $i_\sigma: E_\sigma^{\hilb} \longrightarrow \hE_\sigma$ a continuous injective $\C$-linear map with dense image. 
 
 The data  $$\left((E_{\sigma}^{\hilb},\Vert. \Vert_\sigma, i_\sigma)\right)_{\sigma: K \hra \C}$$ are required to be com\-pa\-tible with complex conju\-gation. Namely, one requires the existence, for every $\sKC,$ of  a $\C$-antilinear bijective isometry 
 $\gamma_\sigma: E^\hilb_\sigma \lrasim E^\hilb_{\overline{\sigma}}$ such that the following relation hold, where $\overline{\cdot}$ denotes the $\C$-antilinear  isomorphism from $\hE_\sigma:= \hE \hat{\otimes}_{\sigma}\C$ onto $\hE_{\overline{\sigma}}:= \hE\hat{\otimes}_{\overline{\sigma}}\C$ deduced from the complex conjugation on $\C$: 
 \begin{equation}\label{gammaisigma}
 i_{\overline{\sigma}} \circ \gamma_\sigma = \overline{\cdot} \circ i_\sigma.
 \end{equation}
 
When they exist, the maps $\gamma_\sigma$ are uniquely determined by the relations (\ref{gammaisigma}). Moreover, they are easily seen to exist when one is given a pro-Hermitian vector bundle $\Ebh:= (\hE, (\Eb_U)_{U \in \cU(\hE)})$ in the sense of  paragraph \ref{proinitial}, and when $E_\sigma^\hilb$, $\Vert. \Vert_\sigma := \Vert. \Vert_{E^\hilb_\sigma}$, and $i_\sigma$ are defined as in paragraph \ref{limHilb}.

 Conversely, starting from the data (\ref{alternativepro}), for every $U \in \cU(\Eh),$ one defines a structure of Hermitian vector bundle 
 $$\Eb_U := (E_U, (\Vert.\Vert_{U,\sigma})_\sKC)$$
 on the finitely generated projective $\OK$-module $E_U:= \Eh/U$ by defining the norm $\Vert. \Vert_{U,\sigma}$ as the quotient norm deduced from the Hilbert norm $\Vert.\Vert_\sigma$ on $E^\hilb_\sigma$ by requiring the $\C$-linear maps 
 $$p_{U,\sigma}\circ i_\sigma: E_\sigma^\hilb \lra E_{U,\sigma}$$
 to be co-isometry. (Observe that the density of $i_\sigma(E_\sigma^\hilb)$ in $\Eh_\sigma$ precisely means that this map is surjective for any $U \in \cU(\Eh).$)
 
 In this way, one constructs a pro-Hermitian vector bundle $\Ebh:= (\hE, (\Eb_U)_{U \in \cU(\hE)})$ in the sense of  paragraph \ref{proinitial} from the data (\ref{alternativepro}).
 
 The reader will easily check that these two constructions are inverse of each other.
 
\medskip 

\subsubsection{} When dealing with pro-Hermitian vector bundles defined by data of type (\ref{alternativepro}), we shall occasionally write $i^{\Ebh}_\sigma$ instead of $i_\sigma$ to make the dependence on $\Ebh$ explicit, notably when discussing ``concrete" examples of pro-Hermitian vector bundles occurring in Diophantine geometry.    
 
 Conversely, when investigating the general properties of pro-Hermitian vector bundles  shall also sometimes avoid to name explicitly the morphisms $i_\sigma$ and $\gamma_\sigma$, and identify $E^\hilb_\sigma$ to its image by $i_\sigma$. Accordingly, a pro-Hermitian vector bundle will be often denoted  by
 $$\Ebh := (\hE, (E_{\sigma}^{\hilb},\Vert.\Vert_{\sigma})_{\sigma: K \hra \C}).$$

 Having identified $E^\hilb_\sigma$ with some subspace of $\hE_\sigma:= \hE \hat{\otimes}_{\sigma}\C$, we may also extend the Hilbert norm $\Vert. \Vert_\sigma$ on $E^\hilb_\sigma$ to a function
  $$\Vert. \Vert_{\sigma}: \hE_{\sigma} \lra [0, + \infty]$$
  by letting
  $$\Vert v \Vert_{\sigma} := + \infty \mbox{ for any $v\in \hE_{\sigma}\setminus   E^\hilb_\sigma$}.$$ 
  Then the relations
  $$\Vert v \Vert_\sigma = \sup_{U \in \cU(\Eh)} \Vert p_{U,\sigma}(v) \Vert_{\Eb_U, \sigma} = \lim_{U \in \cU(\Eh)} \Vert p_{U,\sigma}(v) \Vert_{\Eb_U, \sigma}$$
  hold for any $v \in \Eh_\sigma.$

  Observe that, for any $R \in \R_{+},$ the ball $\{v \in E^\hilb_\sigma \mid \Vert v \Vert \leq R \}$ is closed in the locally convex $\C$-vector space $\hE_{\sigma}.$ (Indeed, it is convex and compact in the weak topology of $E^\hilb_\sigma$, hence in the weak topology of $\hE_{\sigma}$.)
  
    \subsection{Direct images. Ind- and pro-Euclidean lattices} The construction of the direct image of a Hermitian vector bundle over $\Spec \OK$ by the morphism 
    $\pi: \Spec \OK \lra \Spec \Z$ discussed in Section  \ref{Dircan} above extend to ind- and pro-Hermitian vector bundles.
    
    For instance, for any pro-Hermitian vector bundle $\Ebh$ over $\Spec \OK$, we may define its direct image by $\pi$ as  the pro-Hermitian vector bundle over $\Spec \Z$:
  $$\pi_{\ast} \Et := (\pi_{\ast}\hE, (E_{\C}^\hilb, \Vert.\Vert_{\C}))$$
  where $\pi_{\ast}\hE$ is nothing but $\hE$ considered as a topological $\Z$-module, and where  $E_{\C}^\hilb$ denotes the complex Hilbert space defined as
  $$E_{\C}^\hilb := \bigoplus_{\sigma: K \hra \C} E^\hilb_{\sigma},$$
  equipped with the norm $\Vert.\Vert_{\C}$ such that, for any $(x_{\sigma})_{\sigma:K\hra \C} \in E_{\C}^\hilb,$
  $$\Vert(x_{\sigma})_{\sigma:K\hra \C}\Vert^2_{\C}:= \sum_{\sigma: K \hra \C} \Vert x_{\sigma}\Vert^2_{\sigma}.$$
  (Observe that, since $\hE \hat{\otimes}_{\Z}\C \simeq \bigoplus_{\sigma: K \hra \C} \hE_{\sigma}$, the Hilbert space $E_{\C}^\hilb$ is indeed a dense subspace of $(\pi_{\ast}\hE)_{\C}$.)
  
  This construction of direct images reduces many questions concerning ind- and pro-Hermitian vector bundles over $\Spec \OK$ to questions concerning  ind- and pro-Hermitian vector bundles over the  ``final" arithmetic curve $\Spec \Z$. We will often  say \emph{ind-Euclidean lattice} (resp. \emph{pro-Euclidean lattice})  instead of ind-Hermitian (resp. pro-Hermitian) vector bundle over $\Spec \Z.$
  
  Observe also that a pro-Hermitian vector bundle $\Ebh= (\hE, (E_{\C}^\hilb,\Vert.\Vert_{\C}))$ on $\Spec \Z$ may be equivalently defined as a pair $(\hE, (E_{\R},\Vert. \Vert_{\R}))$ where $(E_{\R},\Vert. \Vert_{\R})$ denotes a real Hilbert space equipped with a continuous $\R$-linear injection with dense image
  $$E_{\R}^\hilb \hlra \hE_{\R}:= \hE \hat{\otimes}_{\Z}\R.$$
  (Indeed the real Hilbert space $(E_{\R},\Vert. \Vert_{\R})$ is deduced from $(E_{\C},\Vert. \Vert_{\C})$ by taking its fixed point under complex conjugation. Conversely, we recover  $(E_{\C},\Vert. \Vert_{\C})$  from  $(E_{\R},\Vert. \Vert_{\R})$ by extending the scalars from $\R$ to $\C$.)
  
Similar remarks concerning ind-Hermitian vector bundles might be developed and will be left to the reader.
We only observe that an ind-Hermitian vector bundle $\Fb$ over $\Spec \Z$ may be defined as a pair $(F, \Vert.\Vert)$ where $F$ is a countable free $\Z$-module and $\Vert. \Vert$ is a prehilbertian norm on the real vector space $F_\R:= F\otimes_\Z \R.$

 \section{Hilbertisable ind- and pro-vector bundles}

In applications, it is convenient to have at one's disposal  weakened variants of the notions of ind- and pro-Hermitian vector bundles, where the (pre-)hilbertian norms that enter into their definitions in (\ref{indherdef}) and in (\ref{subsubs:Alternatpro}) are replaced by equivalence classes of (pre-)hilbertian norms.

Thus we shall define a \emph{Hilbertisable ind-vector bundle over $\Spec \OK$} as the data 
$$\dddot{F}:= (F, (F_{\sigma}^{\rm top})_{\sigma:\hra K})$$
of an object $F$ of $\CP_\OK$ and of structures $F_{\sigma}^{\rm top}$ of complex topological vector spaces on the $\C$-vector spaces $F_\sigma:= F\otimes_\sigma \C$ that may be defined by prehilbertian norms. The conjugations maps $F_\sigma \lrasim F_{\overline{\sigma}}$ are required to be homeomorphisms.

To any ind-Hermitian vector bundle $\Fb := (F, (\Vert. \Vert_\sigma)_{\sigma: K\hra \C})$ is attached its underlying Hilbertisable ind-vector bundle
$\dddot{F}:= (F, (F_{\sigma}^{\rm top})_{\sigma:\hra K})$, where $F_{\sigma}^{\rm top})_{\sigma:\hra K}$ denotes $F_\sigma$ equipped with the norm topology defined by $\Vert. \Vert_\sigma.$

Similarly, we shall  define a \emph{Hilbertisable pro-vector bundle over}  $\Spec \OK$ as a pair
 $$\Et := (\hE, (E_{\sigma}^{\hilb}, i_\sigma)_{\sigma: K \hra \C})$$
 where   $\hE$ is an object of $\CTC_{\OK}$ and where, for every field embedding $\sigma: K \hra \C$, $E_{\sigma}^{\hilb}$ denotes a complex Hilbertisable  vector space (that is, a topological complex vector space, the topology of whcih may be defined by some Hilbert space structure) and $i_\sigma: E_\sigma^{\hilb} \longrightarrow \hE_\sigma$ a continuous injective $\C$-linear map with dense image. These data are required to be compatible with complex conjugation (namely, one requires the existence of $\C$-antilinear isomorphisms  $\gamma_\sigma: E^\hilb_\sigma \lrasim E^\hilb_{\overline{\sigma}}$ that satisfy the relations (\ref{gammaisigma})).
 
 Observe that, by the closed graph theorem, the ``Hilbertisable" topology on $E^\hilb_\sigma$ is the unique topology of Fr\'echet space on the complex vector subspace $E^\hilb_\sigma$ of $\Eh_\sigma$ which makes continuous the injection from $E^\hilb_\sigma$ into $\Eh_\sigma$.

To any pro-Hermitian vector bundle 
$$\Ebh:= (\hE, \left(E_{\sigma}^{\hilb},\Vert. \Vert_\sigma, i_\sigma \right)_{\sigma: K \hra \C})$$
is attached its underlying  Hilbertisable pro-vector bundle:
$$\Et := (\hE, (E_{\sigma}^{\hilb}, i_\sigma)_{\sigma: K \hra \C}).$$

\section{Constructions as inductive and projective limits}

\subsection{Construction of ind-Hermitian vector bundles as inductive limits}\label{Consinductive}

Consider an inductive system 
$$\Fb_{\bullet} : \Fb_0 \stackrel{j_0}{\longrightarrow}\Fb_1 \stackrel{j_1}{\longrightarrow}\dots \stackrel{j_{i-1}}{\longrightarrow}\Fb_i \stackrel{j_i}{\longrightarrow} \Fb_{i+1} \stackrel{j_{i+1}}{\longrightarrow} \dots$$
of injective admissible morphisms of Hermitian vector bundles over $\Spec \OK$.

To $\Fb_{\bullet}$, we may attach an ind-Hermitian vector bundle 
$$\varinjlim_i \Fb_i := (F, (\Vert. \Vert_\sigma)_{\sigma: K \hra \C})$$
defined by the following simple construction.

Its underlying $\OK$-module is the inductive limit
$$F:= \varinjlim_i F_i.$$
By construction, it satisfies Condition (4) in Proposition \ref{CPAdef}, and therefore is indeed an object of $\CP_{\OK}.$

Moreover, for any field embedding $\sigma: K \hra \C,$ the maps $j_{i,\sigma}: F_{i,\sigma} \lra F_{i+1, \sigma}$ are isometric with respect to the Hermitian norms $\Vert.\Vert_{\Fb_{i,\sigma}}$ and $\Vert.\Vert_{\Fb_{i+1,\sigma}}$. Therefore there is a unique norm $\Vert.\Vert_\sigma$  on  $F_\sigma:= \varinjlim_i F_{i,\sigma}$ such that the canonical maps $F_{i,\sigma} \hlra F_\sigma$ are isometric with respect to the norms  $\Vert.\Vert_{\Fb_{i,\sigma}}$ and $\Vert.\Vert_\sigma$. The so-defined norm $\Vert. \Vert_\sigma$, like the norms $\Vert.\Vert_{\Fb_{i,\sigma}}$, is clearly  a prehilbertian norm.

Observe that, up to isometric isomorphism, any ind-Hermitian vector bundle $$\Fb:= (F, (\Vert.\Vert_\sigma)_{\sigma: K \hra \C})$$ over $\Spec \OK$ is the limit of an inductive system $\Fb_{\bullet}$ of Hermitian vector bundles as above. Indeed, we may consider a sequence $(F_i)_{i \in \N}$ of $\OK$-submodules of $F$ satisfying Condition (4) in Proposition \ref{CPAdef} (with $A= \OK$), and endow each $F_i$ with the Hermitian norms restrictions of the given norms $(\Vert.\Vert_\sigma)_\sKC$: the so-defined Hermitian vector bundles $\Fb_i$ define an inductive system $\Fb_\bullet$ the limit of which $\varinjlim_i \Fb_i$ is canonically isomorphic to $\Fb.$

\subsection{Construction of pro-Hermitian vector bundles as projective limits}\label{Consprojective}

Consider a projective system
$$\Eb_{\bullet} : \Eb_0 \stackrel{q_0}{\longleftarrow}\Eb_1 \stackrel{q_1}{\longleftarrow}\dots \stackrel{q_{i-1}}{\longleftarrow}\Eb_i \stackrel{q_i}{\longleftarrow} \Eb_{i+1} \stackrel{q_{i+1}}{\longleftarrow} \dots$$
of surjective admissible morphisms of Hermitian vector bundles over $\Spec \OK.$

To $\Eb_{\bullet}$, we may associate a pro-Hermitian vector bundle $$\varprojlim_i \Eb_i = (\hE, (\Eb_U)_{U \in \cU(\hE)})$$  over $\Spec \OK$ defined as follows.

Its underlying topological $\OK$-module is the object of $\CTC_\OK$ defined as the projective limit
$$\Eh := \varprojlim_i E_i.$$
Let us consider the kernels $U_i := \ker p_i$ of the canonical projections $p_i: \Eh \lra E_i$. The sequence $(U_i)_{i \in \N}$ is non-increasing and constitutes a basis of neighborhood of $0$ in $\cU(\hE).$ 

For any $U$ in $\cU(\Eh)$, there exists $i \in \N$ such that $U$ contains $U_i$. Then the quotient map
$$p_{UU_i}: E_i := \Eh/U_i \lra E_U := \Eh/U_i$$
is surjective. Consequently, for any embedding $\sigma: K \hra \C,$ the $\C$-linear map 
$$p_{UU_i,\sigma}: E_{i, \sigma} \lra E_{U,\sigma}$$
also is surjective, and  $E_{U,\sigma}$ may be endowed with the Hermitian norm $\Vert. \Vert_{E_{U,\sigma}}$ defined as the quotient norm, defined by means of $p_{UU_i,\sigma}$, of the Hermitian norm $\Vert. \Vert_{\Eb_i, \sigma}$ on $E_{i,\sigma}$. By construction, the Hermitian vector bundle over $\Spec \OK$
$$\Eb_U := (E_U, (\Vert. \Vert_{E_{U,\sigma}})_{\sigma: K \hra \C})$$
is such that $p_{UU_i}: E_i \lra E_U$ becomes a surjective admissible morphism from $\Eb_i$ to $\Eb_U$.

Using the fact that the morphisms $q_i: \Eb_{i+1} \lra \Eb_i$, and therefore their compositions
$$p_{U_i U_{i'}} = q_i \circ \cdots \circ q_{i' -1}: \Eb_{i'} \lra \Eb_{i'-1} \lra \cdots \lra \Eb_i,$$  are surjective admissible, one easily checks that the construction of the Hermitian structure on $\Eb_U$ does not depend on the choice of the open saturated submodule $U_i$ contained in $U$.

Finally, for any $U$ and $U'$ in $\cU(\Eh)$ and any $i \in \N$ such that $U_i \subset U \subset U'$, the commutativity of the diagram 
$$
\xymatrix{ E_i \ar[r]^{p_{U U_i}} \ar[dr]_{p_{U'U_i}} & E_U\ar[d]^{p_{U'U}}\\
& E_{U'}
}
$$  
and the fact that $p_{UU_i}$ (resp. $p_{U'U_i}$) is a surjective admissible surjective morphism from $\Eb_i$ to $\Eb_U$ (resp. from $\Eb_i$ to $\Eb_{U'}$) implies that $p_{U'U}$ is a surjective admissible morphism from $\Eb_U$ to $\Eb_{U'}$. 

Observe that, up to isometric isomorphism, any pro-Hermitian vector bundle $$\Ebh:= (\Eh, (\Eb_U)_{U \in \cU(\Eh)})$$ over $\Spec \OK$ is the limit of some projective system of Hermitian vector bundle $\Eb_{\bullet}$ as above. Indeed, we may choose a decreasing sequence $(U_i)_{i\in \N}$ in $\cU(\Eh))$ which constitutes a basis of neighborhoods of $0$ in $\Eh,$ and consider the projective system
defined by the Hermitian vector bundles $\Eb_i := \Eb_{U_i}$.

\section{Morphisms between ind- and pro-Hermitian vector bundles over $\OK$}

For any two normed complex vector spaces $(V, \Vert. \Vert)$ and $(V',\Vert.\Vert')$ and for any $\lambda$ in $\R_+,$ we may consider the set of $\C$-linear maps of operator norm at most $\lambda$ from $(V, \Vert. \Vert)$ to $(V',\Vert.\Vert')$:
$$\Hom_\C^{\leq \lambda}((V,\Vert.\Vert)), (V',\Vert.\Vert')) :=
\{ T \in \Hom_\C(V,V') \mid \Arrowvert T \Arrowvert := \sup_{v \in V, \Vert v \Vert \leq 1} \Vert Tv \Vert' \leq \lambda 
\}.
$$

Their union
$$ \Hom_\C^{\rm cont}((V,\Vert.\Vert)), (V',\Vert.\Vert')) :=
\bigcup_{\lambda \in \R^+} \Hom_\C^{\leq \lambda}((V,\Vert.\Vert)), (V',\Vert.\Vert'))$$
is the $\C$-vector space of continuous linear maps from $(V, \Vert. \Vert)$ to $(V',\Vert.\Vert')$.

\subsection{Categories of ind-Hermitian vector bundles}\label{CatInd}  Let $\Fb_1$ and $\Fb_2$ be two ind-Hermitian vector bundles over $\OK$. For any $\lambda$ in $\R_+,$ we define
$\Hom_{\OK}^{\leq \lambda} (\Fb_2, \Fb_1)$ as the subset of $\Hom_{\OK}(F_2, F_1)$ consisting of the $\OK$-linear maps
$$\psi: F_2 \lra F_1$$
such that, for every embedding  $\sigma: K \hlra \C,$ the induced $\C$-linear map
$$ \psi_\sigma: F_{2,\sigma} \lra F_{1,\sigma}$$
is continuous, of operator norm $\leq \lambda$, when $F_{2,\sigma}$ and $F_{1,\sigma}$ are equipped with the pre-hilbertian norms $\Vert .\Vert_{ \Fb_{2,\sigma}}$ and $\Vert .\Vert_{ \Fb_{1,\sigma}}$.

Clearly, if $\Fb_1$, $\Fb_2$, and $\Fb_3$ are ind-Hermitian vector bundles over $\OK$, and if $\lambda$ and $\mu$ are two elements of $\R_+,$ the composition of an element $\psi$  in $\Hom_{\OK}^{\leq \lambda} (\Fb_2, \Fb_1)$ and of an element $\psi'$ in  $\Hom_{\OK}^{\leq \mu} (\Fb_3, \Fb_2)$ defines an element $\psi \circ \psi'$ in 
$\Hom_{\OK}^{\leq \lambda \mu} (\Fb_3, \Fb_1)$.

Consequently, it is possible to define a category whose objects are the ind-Hermitian vector bundles over $\OK$ by either of the following constructions:

(a) by defining the morphisms from $\Fb_2$ to $\Fb_1$ to be
\begin{multline*}
\Hom_{\OK}^{\rm cont}(\Fb_2,\Fb_1)   := \bigcup_{\lambda \in \R_+} \Hom_{\OK}^{\leq \lambda} (\Fb_2, \Fb_1)  \\
                                                             = \left\{
\psi \in \Hom_{\OK}(F_2,F_1) \mid 
\mbox{for every $\sigma: K \hra \C, \psi_\sigma \in \Hom^{\rm cont}_\C ( (F_{2,\sigma},\Vert.\Vert_{\Fb_2, \sigma}), 
(F_{1,\sigma},\Vert.\Vert_{\Fb_1, \sigma}))$}
\right\}.
\end{multline*}
The so-defined category ${\rm ind\overline{Vect}}^{\rm cont}(\OK)$ is clearly $\OK$-linear.

(b) by defining the morphisms from $\Fb_2$ to $\Fb_1$ to be $\Hom_{\OK}^{\leq 1} (\Fb_2, \Fb_1)$. The so-defined category will be denoted by ${\rm ind\overline{Vect}}^{\leq 1}(\OK)$. 

Observe that an isomorphism $\psi: \Fb_2 \lrasim \Fb_1$ in ${\rm ind\overline{Vect}}^{\leq 1}(\OK)$ (resp. in ${\rm ind\overline{Vect}}^{\rm cont}(\OK)$) is an isomorphism of $\OK$-modules $\psi: F_2 \lrasim F_2$ such that the $\C$-linear isomorphisms $ \psi_\sigma: F_{2,\sigma} \lrasim F_{1,\sigma}$ is an isometry (resp. a homeomorphism) between the normed vector spaces $(F_2, \Vert .\Vert_{ \Fb_{2,\sigma}})$  and $(F_1, \Vert .\Vert_{ \Fb_{1,\sigma}})$. In particular, isometric isomorphisms of ind-Hermitian vector bundles (as defined in paragraph \ref{indherdef}) are exactly the isomorphisms in ${\rm ind\overline{Vect}}^{\leq 1}(\OK)$.

The inductive limit $\varinjlim_i \Fb_i$ of an inductive system $\Fb_{\bullet}$ of injective admissible morphisms of Hermitian vector bundles, as considered in paragraph \ref{Consinductive}, together with the obvious inclusion maps $\Fb_k \lra \varinjlim_i \Fb_i$, is easily checked to be a inductive limit of $\Fb_{\bullet}$ in the category ${\rm ind\overline{Vect}}^{\leq 1}(\OK)$.

\medskip  We may also introduce an $\OK$-linear  category
${\rm ind\widetilde{Vect}}(\OK)$, whose objects are the Hilbertisable ind-vector bundles over $\Spec \OK$ :  in the category ${\rm ind\widetilde{Vect}}(\OK)$, the set of morphisms between two Hilbertisable ind-vector bundles $\dddot{F}_2:= (F_2, (F_{2,\sigma}^{\rm top})_{\sigma:\hra K})$ and $\dddot{F}_1:= (F_1, (F_{1,\sigma}^{\rm top})_{\sigma:\hra K})$ over $\Spec \OK$ is defined as the the $\OK$-module
\begin{multline*}
\Hom_{\OK}^{\rm cont}(\dddot{F}_2,\dddot{F}_1) 
                                                             \\ = \left\{
\psi \in \Hom_{\OK}(F_2,F_1) \mid 
\mbox{for every $\sigma: K \hra \C, \psi_\sigma: F_{2,\sigma}^{\rm top} \lra F_{1,\sigma}^{\rm top}$ is continuous}
\right\}.
\end{multline*}

Finally there is a natural forgetful functor from ${\rm ind\overline{Vect}}^{\rm cont}(\OK)$ to ${\rm ind\widetilde{Vect}}(\OK)$, which maps an ind-Hermitian vector bundle $\Fb$ to the associated prehilbertisable ind-vector bundle $\dddot{F}$, and is the identity on morphisms. It is easily seen to be an equivalence of category. 

\subsection{Categories of pro-Hermitian vector bundles} 

For any two pro-Hermitian vector bundles $\Ebh_1$ and $\Ebh_2$ over $\Spec \OK$ and for any $\lambda$ in $\R_+,$ we define:
\begin{equation}\label{defhom1}
 \Hom_{\OK}^{\leq \lambda}(\Ebh_1, \Ebh_2) := \varprojlim_{U_2} \varinjlim_{U_1}
\Hom_{\OK}^{\leq \lambda} (\Eb_{1,U_1}, \Eb_{2,U_2}).
\end{equation}
(In the inductive (resp. projective) limit, $U_1$ (resp. $U_2$) varies in the filtered set $\cU(\hE_1)$, ordered by $\supseteq$.)

We also define the set of \emph{continuous $\OK$-morphisms from $\Ebh_1$ and $\Ebh_2$} as the $\OK$-module
\begin{equation}\label{defhom2}
\Hom_{\OK}^{\rm cont} (\Ebh_1, \Ebh_2) := \bigcup_{\lambda \in \R_+} \Hom_{\OK}^{\leq \lambda}(\Ebh_1, \Ebh_2).
\end{equation}

Observe that an element $\tilde{\phi}$ in  $\Hom_{\OK}^{\rm cont}(\Ebh_1, \Ebh_2)$ is uniquely determined by its image $\hat{\phi}$ in 
the $\OK$-module $$\varprojlim_{U_2} \varinjlim_{U_1}
\Hom_{\OK} (E_{1,U_1}, E_{2,U_2})\simeq \Hom^{\rm cont}_{\OK} (\Eh_1, \Eh_2)$$ of $\OK$-linear continous maps from $\Eh_1$ to $\Eh_2$.

The following proposition is a direct consequence of the construction of the Hilbert spaces associated to a pro-Hermitian vector bundle:

\begin{proposition}\label{Altmor} With the above notation, an element $\hat{\phi}$ in $\Hom^{\rm cont}_{\OK} (\Eh_1, \Eh_2)$ may be lifted to an element 
 $\tilde{\phi}$ in $\Hom_{\OK}^{\leq \lambda}(\Ebh_1, \Ebh_2)$ if and only if, for every embedding $\sigma: K \hra \C,$ there exists a continuous $\C$-linear map of operator norm $\leq \lambda$
 $$\phi_\sigma: E_{1,\sigma}^\hilb \lra E_{2,\sigma}^\hilb$$
 between the Hilbert spaces $\Eb_{1,\sigma}^{\rm Hilb}$
and $\Eb_{2,\sigma}^{\rm Hilb}$ such that the following diagram is commutative:
\begin{equation}\label{altdiag}\begin{CD}
E^{\rm Hilb}_{1,\sigma} @>\phi_\sigma>> E^{\rm Hilb}_{2,\sigma} \\
@V{i^{\Ebh_1}_\sigma}VV      @VV{i^{\Ebh_2}_\sigma}V \\
\hE_
{1,\sigma} @>{\hat{\phi}_\sigma}>> \hE_{2,\sigma}.
\end{CD}
\end{equation}
When this holds, these morphisms $\phi_\sigma$ are unique. \qed
\end{proposition}

According to Proposition \ref{Altmor}, an element of $\Hom_{\OK}^{\rm cont} (\Ebh_1, \Ebh_2)$ (resp. $\Hom_{\OK}^{\leq \lambda} (\Ebh_1, \Ebh_2)$) may be described as  a pair
\begin{equation}
\label{morpair}
\phit:= (\hphi, (\phi_{\sigma})_{\sigma: K \hra \C}),
\end{equation}
consisting in the following data:
\begin{enumerate}
\item  a continuous morphism of topological $\OK$-modules $$\hphi: \hE_1 \lra \hE_2;$$
\item for every embedding $\sigma: K \hra \C$, 
 a continuous $\C$-linear map continuous map (resp. of operator norm $\leq \lambda$) $$\phi_\sigma: E^{\rm Hilb}_{1,\sigma} \lra E^{\rm Hilb}_{2,\sigma}$$ between the Hilbert spaces $\Eb_{1,\sigma}^{\rm Hilb}$
and $\Eb_{2,\sigma}^{\rm Hilb}$ that is compatible with $\hphi,$ in the sense that the  diagram (\ref{altdiag}) is commutative for every embedding $\sigma: K \hra \C$.
\end{enumerate}

In the following sections, and in the sequel of this monograph, we shall freely use this alternative description of $\Hom_{\OK}^{\rm cont} (\Ebh_1, \Ebh_2)$.

As a special case of this description, observe that, for any Hermitian vector bundle $\Eb$ over $\Spec \OK$ and for any pro-Hermitian vector bundle $\Fbh$ over $\Spec \OK,$ we have a natural identification:
$$\Hom_{\OK}^{\rm cont}(\Eb, \Fbh ) \lrasim \Hom_{\OK} (E, \hF \cap F^{\rm Hilb}_{\C}).$$

For any three pro-Hermitian vector bundles $\Ebh_1,$ $\Ebh_2,$ and $\Ebh_3$ over $\Spec \OK,$ there is a natural $\OK$-bilinear composition map
\begin{equation}\label{comp}
. \circ . : \Hom_\OK^{\rm cont} (\Ebh_2,\Ebh_3) \times  \Hom_\OK^{\rm cont} (\Ebh_1,\Ebh_2) \lra \Hom_\OK^{\rm cont} (\Ebh_1,\Ebh_3),
\end{equation} --- that, for any $(\lambda, \mu) \in \R_+^2,$ maps $\Hom_\OK^{\rm cont} (\Ebh_2,\Ebh_3) \times  \Hom_\OK^{\rm cont} (\Ebh_1,\Ebh_2)$ to $\Hom_\OK^{\rm cont} (\Ebh_1,\Ebh_3)$ --- 
deduced from the composition of morphisms
$$\Hom_\OK^{\leq \lambda} (\Ebh_{2,U_2},\Ebh_{3, U_3}) \times  \Hom_\OK^{\leq \mu} (\Ebh_{1,U_1},\Ebh_{2,U_2}) \lra \Hom_\OK^{\leq  \lambda \mu} (\Ebh_{1, U_1},\Ebh_{3, U_3})$$
by passage to the projective and inductive limits involved in the definition (see (\ref{defhom1})) of continuous $\OK$-morphisms of pro-Hermitian vector bundles.     

In terms of the description of the morphisms of pro-Hermitian vector bundles as pairs of the form (\ref{morpair}), this composition law may be described as follows. 
If $\phit:= (\hphi, (\phi_{\sigma})_{\sigma: K \hra \C})$ (resp. $\psit:= (\hpsi, (\psi_{\sigma})_{\sigma: K \hra \C})$) is an element of $\Hom_\OK^{\rm cont} (\Ebh_1,\Ebh_2)$ (resp. of $\Hom_\OK^{\rm cont} (\Ebh_2,\Ebh_3)$), the composition law (\ref{comp}) maps $(\psit,\phit)$ to 
$$\psit \circ \phit := (\hpsi \circ \hphi, (\psi_{\sigma}\circ \phi_{\sigma})_{\sigma: K \hra \C}).$$

By defining the morphisms from $\Ebh_1$ to $\Ebh_2$ to be $\Hom_\OK^{\rm cont} (\Ebh_1,\Ebh_2)$ and the composition of morphisms as above, the class of pro-Hermitian vector bundles over $\Spec \OK$ becomes
 an $\OK$-linear category. We shall denote it by
${\rm pro\overline{Vect}}^{\rm cont}(\OK)$.

We may define another category, the objects of which are again the  pro-Hermitian vector bundles over $\Spec\OK$, by defining the morphisms the morphisms from $\Ebh_1$ to $\Ebh_2$ to be $\Hom_\OK^{\leq 1} (\Ebh_1,\Ebh_2)$, and by defining the composition as above. We shall denote this category by 
${\rm pro\overline{Vect}}^{\leq 1}(\OK)$.

Finally we may formulate some observations concerning these categories of pro-Hermitian vector bundles similar to the ones concerning ind-Hermitian vector bundles at the end of paragraph \ref{CatInd}:

(i) An isomorphism $\phit: \Ebh_1 \lrasim \Ebh_2$ in ${\rm pro\overline{Vect}}^{\rm cont}(\OK)$ (resp. in ${\rm pro\overline{Vect}}^{\leq 1}(\OK)$) is the data of an isomorphism 
$\hphi: \Eh_1 \lrasim \Eh_2$ of topological $\OK$-modules and of $\C$-linear homeomorphisms $\phi_\sigma: E_{1,\sigma}^\hilb \lrasim E_{2,\sigma}^\hilb$ compatibel with $\hphi$.

An isomorphism $\phit: \Ebh_1 \lrasim \Ebh_2$  in ${\rm pro\overline{Vect}}^{\leq 1}(\OK)$ is precisely an isometric isomorphism of pro-Hermitian vector bundles.

(ii) The projective limit $\varprojlim_i \Eb_i$ of a projective system $\Eb_{\bullet}$ of surjective admissible morphisms of Hermitian vector bundles (as considered in paragraph \ref{Consprojective}), equipped with the projection maps $\varprojlim_i \Eb_i \lra \Eb_k,$  is easily checked to be a projective limit of $\Eb_{\bullet}$ in ${\rm pro\overline{Vect}}^{\leq 1}(\OK)$. 

(iii) We may introduce the $\OK$-linear category ${\rm pro\widetilde{Vect}}(\OK)$, whose objects are the Hilbertisable pro-vector bundles over $\Spec \OK$, and a forgetful functor
$$ {\rm pro\overline{Vect}}^{\rm cont}(\OK) \lra {\rm pro\widetilde{Vect}}(\OK)$$
--- it sends an object $\Ebh$ of  ${\rm pro\overline{Vect}}^{\rm cont}(\OK)$ to the underlying object $\Et$ of ${\rm pro\widetilde{Vect}}(\OK)$ --- that is actually an equivalence of category.

\section{The duality between ind- and pro-Hermitian vector bundles}\label{dualindpro} In this section, we construct a duality between the categories  ${\rm ind\overline{Vect}}^{\leq 1}(\OK)$ and  ${\rm pro\overline{Vect}}^{\leq 1}(\OK)$ and between the categories
${\rm pro\overline{Vect}}^{\rm cont}(\OK)$ and ${\rm pro\overline{Vect}}^{\rm cont}(\OK)$ by combining the duality between $\CP_{\OK}$ and $\CTC_{OK}$ described in Section \ref{dualCPCTC} and the classical duality theory of (pre-)Hilbert spaces.

\subsection{The duality functors} $\,$

(i) Let $\Fb:=(F,(\Vert.\Vert_\sigma)_{\sKC})$ be an ind-Hermitian vector bundle over $\Spec \OK$. To $\Fb$, we may attach a dual pro-Hermitian vector bundle 
\begin{equation}\label{Fveedef}
\Fb^\vee:= (F^\vee, (F^{\vee\hilb}_\sigma, \Vert.\Vert^\vee_{\sigma})_\sKC)
\end{equation}
over $\Spec \OK,$ defined as follows. 

Its underlying topological $\OK$-module is the $\OK$-module 
$$
F^\vee := \Hom_{\OK}(F, \OK)
$$
equipped with the topology of pointwise convergence. In other words, it is the dual in $\CTC_\OK$ of the object $F$ of $\CP_\OK.$

For any embedding $\sKC,$ we get canonical isomorphisms of complex vector spaces
\begin{equation}\label{IdFsigma}
F^\vee_\sigma := F^\vee \hat{\otimes}_{\OK, \sigma} \C \simeq \Hom_{\OK, \sigma}(F, \C) \simeq \Hom_\C(F_\sigma, \C).
\end{equation}
Actually the Fréchet topology on $F^\vee_\sigma$ (deduced  from the structure of topological $\OK$-module on $F^\vee$, as explained in \ref{topologicalcomplexpro} (i)) coincides with the locally convex topology on $\Hom_\C(F_\sigma, \C)$ defined by the pointwise convergence on $F_\sigma.$

By means of the identifications (\ref{IdFsigma}), we may introduce the vector subspace 
$$F_\sigma^{\vee \hilb} := \Hom_\C^{\rm cont}((F_\sigma, \Vert. \Vert_\sigma), \C)$$
of $F_\sigma^\vee$ consisting in the linear forms on $F_\sigma$ continuous with respect to the Hermitian  norm $\Vert. \Vert_\sigma$. Equipped with its operator norm, defined by 
$$\Vert \xi \Vert^\vee_\sigma := \sup_{f \in F_\sigma, \Vert f \Vert_\sigma \leq 1} \vert \xi(f)\vert,$$
this vector space becomes a separable Hilbert space $(F^{\vee\hilb}_\sigma, \Vert.\Vert^\vee_{\sigma}),$ and the inclusion $$F^{\vee\hilb}_\sigma \hlra F^{\vee}_\sigma$$  is easily seen to be continuous and to have a dense image.\footnote{Actually, from any countable $\C$-basis of $F_\sigma$, by orthonormalization we get a countable orthonormal basis of $F_\sigma$. By means of any such orthonormal basis, we get compatible isomorphisms: $F_\sigma \lrasim \C^{(\N)}$, $F^\vee_\sigma \lrasim \C^{\N},$
and $F_{\sigma}^{\vee \hilb} \lrasim \ell^2 (\N).$ }

This construction is compatible with complex conjugation, and the right hand of  (\ref{Fveedef}) actually defines a pro-Hermitian vector bundle over $\Spec \OK$, in terms of the alternative approach to pro-Hermitian vector bundles presented in paragraph
\ref{subsubs:Alternatpro}.

Let $\lambda$ be a positive real number, and let $\Fb_1:= (F_1, (\Vert.\Vert_{1,\sigma})_{\sKC})$ and
$\Fb_2:= (F_2, (\Vert.\Vert_{2,\sigma})_{\sKC})$  be two ind-Hermitian vector bundles over $\Spec \OK.$ Consider a morphism  $\alpha$ in $\Hom_{\OK}^{\leq \lambda}(\Fb_1, \Fb_2)$. By definition, $\alpha$ is an element of $\Hom_\OK(F_1, F_2)$ such that, for any $\sKC,$ the $\C$-linear map 
$\alpha_\sigma: F_{1,\sigma} \lra F_{2,\sigma}$ is continuous of operator norm $\leq \lambda$ from
  $(F_{1\sigma}, \Vert. \Vert_{1,\sigma})$ to $(F_{2\sigma}, \Vert. \Vert_{2,\sigma})$.
  
  The dual (or adjoint) morphism
  $$\alpha^\vee := . \circ \alpha: F_2^\vee := \Hom_\OK(F_2, \OK) \lra F_1^\vee := \Hom_\OK(F_1, \OK)$$
  becomes, after ``completed base change" under the  embedding $\sigma: \OK \hra \C,$ the $\C$-linear map
  $$\alpha^\vee_\sigma := . \circ \alpha_\sigma : F^\vee_{2,\sigma} \simeq \Hom_\C(F_{2,\sigma}, \C) \lra
  F^\vee_{1,\sigma} \simeq \Hom_\C(F_{1,\sigma}, \C).$$ 
  It is a continuous $\C$-linear map between the Fréchet spaces $F^\vee_{2,\sigma}$ and $F^\vee_{1,\sigma}$. Moreover, it sends 
  $F_{2,\sigma}^{\vee \hilb} := \Hom_\C^{\rm cont}((F_{2,\sigma}, \Vert. \Vert_{2,\sigma}), \C)$
to $F_{1,\sigma}^{\vee \hilb} := \Hom_\C^{\rm cont}((F_{1,\sigma}, \Vert. \Vert_{1,\sigma}), \C),$
 and defines a continuous $\C$-linear map of operator norm $\leq \lambda$ from
 $(F^{\vee\hilb}_{2,\sigma}, \Vert.\Vert^\vee_{2,\sigma})$ to $(F^{\vee\hilb}_{1,\sigma}, \Vert.\Vert^\vee_{1,\sigma})$.
  
 In conclusion, the dual map $\alpha^\vee$  defines a morphism in $\Hom_{\OK}^{\leq \lambda} (\Fb_2^\vee, \Fb_1^\vee)$. Moreover, the so-defined maps
 $$
\begin{array}{rcl}
\Hom_{\OK}^{\leq \lambda}(\Fb_1, \Fb_2)  &  \lra  &  \Hom_{\OK}^{\leq \lambda} (\Fb_2^\vee, \Fb_1^\vee) \\
 \alpha & \longmapsto   & \alpha^\vee.   
  \end{array}
$$
define an $\OK$-linear map
$$.^\vee: \Hom_{\OK}^{\rm cont}(\Fb_1, \Fb_2)    \lra    \Hom_{\OK}^{\rm cont} (\Fb_2^\vee, \Fb_1^\vee).$$
This construction is clearly functorial and defines two contravariant duality functors,
$$.^\vee: 
{\rm ind\overline{Vect}}^{\leq 1}(\OK) \lra {\rm pro\overline{Vect}}^{\leq 1}(\OK)$$
and 
$$.^\vee: 
{\rm ind\overline{Vect}}^{\rm cont}(\OK) \lra {\rm pro\overline{Vect}}^{\rm cont}(\OK),$$
the second of which is $\OK$-linear.

Observe that any inductive system
$$\Fb_{\bullet} : \Fb_0 \stackrel{j_0}{\longrightarrow}\Fb_1 \stackrel{j_1}{\longrightarrow}\dots \stackrel{j_{i-1}}{\longrightarrow}\Fb_i \stackrel{j_i}{\longrightarrow} \Fb_{i+1} \stackrel{j_{i+1}}{\longrightarrow} \dots$$
of injective admissible morphisms of Hermitian vector bundles over $\Spec \OK$ determines by duality a projective system
$$\Fb^\vee_{\bullet} : \Fb_0^\vee \stackrel{j_0^\vee}{\longleftarrow}\Fb_1^\vee \stackrel{j_1^\vee}{\longleftarrow}\dots \stackrel{j_{i-1}^\vee}{\longleftarrow}\Fb_i^\vee \stackrel{j_i^\vee}{\longleftarrow} \Fb_{i+1}^\vee \stackrel{j_{i+1}^\vee}{\longleftarrow} \dots$$
of surjective admissible morphisms of Hermitian vector bundles over $\Spec \OK$. The injection morphisms   $j_k : \Fb_k \lra \varinjlim_i \Fb_i$ define, by duality, morphisms in ${\rm pro\overline{Vect}}^{\leq 1}(\OK)$:
$$j_k^\vee : (\varinjlim_i \Fb_i)^\vee \lra \Fb_k^\vee,$$
which in turn define a morphism  from $(\varinjlim_i \Fb_i)^\vee$ to the projective limit $\varprojlim_i \Fb_i^\vee$ of $\Fb_{\bullet}^\vee$, that is easily seen to be an isometric isomorphism of pro-Hermitian vector bundles over $\Spec \OK$:
\begin{equation}\label{dualprojlim}
(\varinjlim_i \Fb_i)^\vee \lrasim \varprojlim_i \Fb_i^\vee.
\end{equation}
 
 \medskip

(ii)  Conversely, to any pro-Hermitian vector bundle over $\Spec \OK,$  $$\Ebh := (\hE, (E_{\sigma}^{\hilb},\Vert.\Vert_{\sigma})_{\sigma: K \hra \C}),$$ 
we may attach a dual ind-Hermitian vector bundle 
\begin{equation}\label{Ebhveedef}
\Ebh^\vee := (\Eh^\vee, (\Vert.\Vert_\sigma^\vee)_\sKC)
\end{equation}
defined as follows.

Its underlying projective $\OK$-module is the dual 
$$\Eh^\vee := \Hom_\OK^{\rm cont}(\Eh, \OK)$$
in $\CP_\OK$ of the object $\Eh$ of $\CTC_\OK.$ Actually, as a consequence of Proposition \ref{prop:HomHom}, $\Eh^\vee$ coincides with the algebraic dual $\Hom_\OK(\Eh, \OK)$ of the $\OK$-module $\Eh.$

For any embedding $\sKC,$ we have canonical isomorphisms:
\begin{equation*}
(\Eh^\vee)_\sigma := \Eh^\vee \otimes_{\OK,\sigma} \C \simeq (\varinjlim_{U \in \cU(\Eh)} E_U^\vee) \otimes_{\OK,\sigma} \C
\simeq \varinjlim_{U \in \cU(\Eh)} E_{U,\sigma}^\vee \simeq \Hom_\C^{\rm cont}(\Eh_\sigma, \C).
\end{equation*}
Since the injection $i_\sigma: E^\hilb_\sigma \hlra \Eh_\sigma$ is continuous with dense image, its transpose defines an injective map (also with dense image)
$$^t i_\sigma:  (\Eh^\vee)_\sigma \hlra (E_\sigma^\hilb)^\vee$$
from $(\Eh^\vee)_\sigma$ to the dual Hilbert space  $(E_\sigma^\hilb)^\vee$, and we define the norm $\Vert. \Vert_\sigma^\vee$ as the restriction to $(\Eh^\vee)_\sigma$ of the Hilbert norm on $(\Eh^\vee)_\sigma$  dual of the  norm $\Vert. \Vert_\sigma$ on $E^\hilb_\sigma$. (In other words, the norm $\Vert \xi \Vert^\vee_\sigma$ of some linear form $\xi$ in $\Hom^{\rm cont}_\C(\Eh_\sigma, \C)$ is the operator norm, with respect to $\Vert.\Vert_\sigma$, of its restriction to $E_\sigma^\hilb.$)

It is straightforward that the construction of the norms  $(\Vert. \Vert_\sigma^\vee)_\sKC$ on the complex vector spaces $((\Eh^\vee)_\sigma)_\sKC$ is compatible with complex conjugation. Consequently the right hand-side of   (\ref{Ebhveedef}) indeed defines some ind-Hermitian vector bundle over $\Spec \OK$.

Let $\lambda$ be a positive real number, and let $$\Ebh_1 := (\hE_1, (E_{1,\sigma}^{\hilb},\Vert.\Vert_{1,\sigma})_{\sigma: K \hra \C})$$ and $$\Ebh_2 := (\hE_2, (E_{2,\sigma}^{\hilb},\Vert.\Vert_{2,\sigma})_{\sigma: K \hra \C})$$ be two pro-Hermitian vector bundles over $\Spec \OK.$ Consider a morphism $\beta$ in $\Hom^{\leq \lambda}_\OK(\Ebh_1, \Ebh_2)$. By definition, $\beta$ is an element in $\Hom^{\rm cont}_\OK (\Eh_1, \Eh_2)$ such that, for any embedding $\sKC,$ the continuous linear map of Fréchet spaces
\begin{equation}\label{betasigma}
\beta_\sigma: \Eh_{1,\sigma} \lra \Eh_{2,\sigma}
\end{equation}
maps $E_{1,\sigma}^\hilb$ to $E_{2,\sigma}^\hilb$, with an operator norm (with respect to the Hilbert norms $\Vert.\Vert_{1,\sigma}$ and $\Vert.\Vert_{2,\sigma}$) at most $\lambda$.

The dual morphism
$$\beta^\vee: . \circ \beta : \Eh_2^\vee := \Hom_{\OK}^{\rm cont}(\Eh_2, \OK) \lra \Eh_1^\vee := \Hom_{\OK}^{\rm cont} (\Eh_1, \OK),$$
 after the base change $\sigma: \OK \hlra \C,$ becomes the transpose of the map (\ref{betasigma}), namely
$$\beta^\vee_\sigma = . \circ \beta_\sigma : (\Eh^\vee_2)_\sigma \simeq \Hom_\C^{\rm cont}(\Eh_{2,\sigma}, \C) \lra 
(\Eh^\vee_1)_\sigma \simeq \Hom_\C^{\rm cont}(\Eh_{1,\sigma}, \C),$$
and therefore satisfies, for any $\xi \in (\Eh^\vee_2)_\sigma$:
$$\Vert \beta^\vee_\sigma (\xi)\Vert^\vee_{1,\sigma} = \Vert \xi \circ \beta_\sigma \Vert^\vee_{1,\sigma} \leq \lambda \Vert \xi \Vert^\vee_{2,\sigma}.$$
This shows that $\beta^\vee$ belongs to $\Hom^{\leq \lambda}_\OK(\Ebh_2^\vee, \Ebh_1^\vee).$

The so-defined maps 
$$
\begin{array}{rcl}
\Hom_{\OK}^{\leq \lambda}(\Ebh_1, \Ebh_2)  &  \lra  &  \Hom_{\OK}^{\leq \lambda} (\Ebh_2^\vee, \Ebh_1^\vee) \\
 \beta & \longmapsto   & \beta^\vee.   
  \end{array}
$$
define an $\OK$-linear map
$$.^\vee: \Hom_{\OK}^{\rm cont}(\Ebh_1, \Ebh_2)    \lra    \Hom_{\OK}^{\rm cont} (\Ebh_2^\vee, \Ebh_1^\vee).$$
This construction is clearly functorial and defines two contravariant duality functors,
\begin{equation}\label{dualproind1}
^\vee: 
{\rm pro\overline{Vect}}^{\leq 1}(\OK) \lra {\rm ind\overline{Vect}}^{\leq 1}(\OK)
\end{equation}
and 
$$.^\vee: 
{\rm pro\overline{Vect}}^{\rm cont}(\OK) \lra {\rm ind\overline{Vect}}^{\rm cont}(\OK),$$
the second of which is $\OK$-linear.

The duality functor (\ref{dualproind1}) is compatible with the construction of pro-Hermitian vector bundles as limits of projective systems of surjective admissible maps of Hermitian vector bundles over $\Spec \OK$ discussed in  paragraph \ref{Consprojective}.
Namely, if $$\Eb_{\bullet} : \Eb_0 \stackrel{q_0}{\longleftarrow}\Eb_1 \stackrel{q_1}{\longleftarrow}\dots \stackrel{q_{i-1}}{\longleftarrow}\Eb_i \stackrel{q_i}{\longleftarrow} \Eb_{i+1} \stackrel{q_{i+1}}{\longleftarrow} \dots$$
is such a projective system, we may form the dual inductive system
$$\Eb_{\bullet}^\vee : \Eb_0^\vee \stackrel{q_0^\vee}{\longrightarrow}\Eb_1^\vee \stackrel{q_1^\vee}{\longrightarrow}\dots \stackrel{q_{i-1}^\vee}{\longrightarrow}\Eb_i^\vee \stackrel{q_i^\vee}{\longrightarrow} \Eb_{i+1}^\vee \stackrel{q_{i+1}^\vee}{\longrightarrow} \dots$$
of injective admissible morphisms of Hermitian vector bundles over $\Spec \OK.$
The projections $p_k: \varprojlim_i \Eb_i \lra \Eb_k$ define, by duality, morphisms in ${\rm ind\overline{Vect}}^{\leq 1}(\OK)$,
$$p_k^\vee : \Eb_k^\vee \lra  (\varprojlim_i \Eb_i)^\vee,$$
that in turn define a morphism from the inductive limit of $\Eb_{\bullet}^\vee$ to $(\varprojlim_i \Eb_i)^\vee$. This morphism is easily seen to be an isometric isomorphism of ind-Hermitian vector bundles over $\Spec \OK$:
\begin{equation}\label{dualindlim}
\varinjlim_i \Eb_i^\vee \lrasim (\varprojlim_i \Eb_i)^\vee.
\end{equation}

\subsection{Duality as  adjoint equivalences}\label{dualadjointArith}

Let $\Ebh := (\Eh, (E^\hilb_\sigma, \Vert.\Vert_{E,\sigma})_\sKC)$ be an object of the category ${\rm pro\overline{Vect}}(\OK)$ and let $\Fb:=(F,(\Vert.\Vert_{F,\sigma})_\sKC$ be an object of ${\rm ind\overline{Vect}}(\OK)$. We may consider their dual objects $\Ebh^\vee = (\Eh^\vee, (\Vert.\Vert_{E,\sigma}^\vee)_\sKC)$ and $\Fb^\vee = (F^\vee, (\Vert.\Vert_{F,\sigma}^\vee)_\sKC$, in ${\rm ind\overline{Vect}}(\OK)$ and ${\rm pro\overline{Vect}}(\OK)$ respectively.

As discussed in \ref{dualadjointA}, we have natural isomorphisms which define the duality between the $\OK$-linear categories $\CTC_\OK$ and $\CP_\OK$:
\begin{equation}\label{Bil}
\Hom^{\rm cont}_\OK (\Eh, F^\vee) \lrasim \Hom_\OK(F, \Eh^\vee).
\end{equation}
Indeed both $\Hom^{\rm cont}_\OK (\Eh, F^\vee)$ and $\Hom_\OK(F, \Eh^\vee)$ may be identified with the $\OK$-module consisting in the $\OK$-bilinear maps 
$$b: \Eh \times F \lra \OK$$
which are continuous in the first variable, or equivalently with the $\OK$-module $\varinjlim_{U \in \cU(\Eh)} E_U^\vee \otimes_\OK F^\vee$.

Similarly, for any field embedding $\sKC,$ we have a duality isomorphism of complex vector spaces:
\begin{equation}\label{Bilsigma}
\Hom^{\rm cont}_\C (\Eh_\sigma, F_\sigma^\vee) \lrasim \Hom_\C (F_\sigma, \Eh_\sigma^\vee).
\end{equation}
Indeed, both $\Hom^{\rm cont}_\C (\Eh_\sigma, F_\sigma^\vee)$ and $\Hom_\C (F_\sigma, \Eh_\sigma^\vee)$ may be identified with the complex vector space consisting in the $\C$-bilinear maps
$$\tilde{b}: \Eh_\sigma \times F_\sigma \lra \C$$
which are continuous in the first variable, or equivalently with the complex vector space $$\varinjlim_{U \in \cU(\Eh)} E_{U,\sigma}^\vee \otimes_\C F_\sigma^\vee.$$

Moreover, for any $\lambda \in \R_+,$ the bijection (\ref{Bilsigma}) defines by restriction a bijection:
\begin{equation}\label{Billambda}
\Hom^{\rm cont}_\C (\Eh_\sigma, F_\sigma^\vee) \cap \Hom^{\leq \lambda}_\C ((E^\hilb_\sigma, \Vert.\Vert_{E,\sigma}), (F_\sigma^\vee, \Vert.\Vert^\vee_{F,\sigma}))
 \lrasim \Hom_\C^{\leq \lambda} ((F_\sigma, \Vert.\Vert_{F,\sigma}), (\Eh_\sigma^\vee, \Vert.\Vert^\vee_{E,\sigma})).
\end{equation}
Indeed, the subsets of $\Hom^{\rm cont}_\C (\Eh_\sigma, F_\sigma^\vee)$ and of $\Hom_\C (F_\sigma, \Eh_\sigma^\vee)$ defined by both sides of (\ref{Billambda}) may be identified with the spaces of $\C$-bilinear maps $\tilde{b}$ as above such that, when considered as bilinear maps on $E^\hilb_\sigma \times F_\sigma,$ their $\varepsilon$-norm
$$\Vert \tilde{b}\Vert_\varepsilon := \sup \left\{ \vert \tilde{b}(x,y) \vert; (x,y) \in E^\hilb_\sigma \times F_\sigma, \Vert x\Vert_{E,\sigma} \leq 1, \Vert y \Vert_{F,\sigma} \leq 1 \right\}$$
is at most $\lambda$.

Besides, the identifications (\ref{Bil}) and (\ref{Bilsigma}) are compatible with ``extension of scalars through $\sKC$."
In other words, (\ref{Bil}) and (\ref{Bilsigma}) fits into a commutative diagram:
\begin{equation*}
\begin{CD}
\Hom^{\rm cont}_\OK (\Eh, F^\vee) @>\sim>>  \Hom_\OK(F, \Eh^\vee) \\ 
@VV{.\otimes_\sigma 1_\C}V      @VV{.\otimes_\sigma 1_\C}V \\
\Hom^{\rm cont}_\C (\Eh_\sigma, F_\sigma^\vee) @>\sim>> \Hom_\C (F_\sigma, \Eh_\sigma^\vee).
\end{CD}
\end{equation*}

Together with the bijections (\ref{Billambda}), this shows that the bijection (\ref{Bil}) defines --- by restriction --- a bijection
\begin{equation}\label{duallambda}
\Hom_\OK^{\leq \lambda}(\Ebh, \Fb^\vee) \lrasim \Hom_\OK^{\leq \lambda}(\Fb, \Ebh^\vee)
\end{equation}
for every $\lambda \in \R_+,$ and consequently an $\OK$-linear isomorphism:
\begin{equation}\label{dualcont}
\Hom_\OK^{\rm cont}(\Ebh, \Fb^\vee) \lrasim \Hom_\OK^{\rm cont}(\Fb, \Ebh^\vee).
\end{equation}

We may finally formulate the duality between the categories of pro- and ind-Hermitian vector bundles over $\Spec \OK$ as the following statement, to be compared with the duality between $\CTC_\OK$ and $\CP_\OK$ that constitutes the special case of Proposition \ref{prop:dual} where $A= \OK$:

\begin{proposition}\label{prop:dualArith} The bijections  (\ref{duallambda}) with $\lambda =1$ and (\ref{dualcont}) define adjunctions of functors: %
 $$.^\vee :  {\rm pro\overline{Vect}}^{\leq 1}(\OK) \leftrightarrows {\rm ind\overline{Vect}}^{\leq 1}(\OK)^{\rm op} : .^\vee$$
 and
 $$.^\vee :  {\rm pro\overline{Vect}}(\OK) \leftrightarrows {\rm ind\overline{Vect}}(\OK)^{\rm op} : .^\vee \;\;\;.$$
 
These are  actually  adjoint equivalences. Their 
  unit and counit are  natural isomorphisms $\eta$ and  $\epsilon$ defined by  isometric isomorphisms
  $$\epsilon_{\Fb} : \Fb \lrasim \Fb^{\vee \vee}
\mbox{ and } \eta_{\Ebh}: \Ebh \lrasim \Ebh^{\vee \vee},$$
associated to any object $\Fb$ in ${\rm ind\overline{Vect}}(\OK)$ and any object $\Ebh$ in ${\rm pro\overline{Vect}}(\OK)$, whose underlying morphisms from $F$ to $F^{\vee \vee}$ and from $\Eh$ to $\Eh^{\vee \vee}$ are the biduality isomorphisms:
$$
\begin{array}{rrcl}
 \epsilon_F:& F  & \lrasim   & \Hom_\OK^{\rm cont}(\Hom_\OK (F,\OK),\OK)  \\
& f & \longmapsto   & (\xi  \longmapsto \xi(f))
\end{array}
$$
and $$
\begin{array}{rrcl}
 \eta_{\Eh}:& \Eh  & \lrasim   & \Hom_\OK(\Hom_\OK^{\rm cont} (\Eh,\OK),\OK)  \\
& e & \longmapsto   & (\zeta  \longmapsto \zeta(e)).  
\end{array}
$$
 \end{proposition}

\begin{proof} This may be deduced from the duality between $\CTC_\OK$ and $\CP_\OK$ established in Proposition \ref{prop:dual}, combined with basic results concerning the duality of (pre-)Hilbert spaces.

At this stage, we may also argue directly as follows. The naturality (in each of the relevant categories) with respect to $\Ebh$ and $\Fb$ of the bijections  (\ref{duallambda}) with $\lambda =1$ and (\ref{dualcont}) is straightforward. We are left to show that the unit  $\epsilon_{\Fb}$ and counit $\eta_{\Ebh}$ of these adjunctions are isometric isomorphisms with underlying isomorphisms the biduality isomorphisms $\epsilon_F$ and $\eta_\Eh$.

This directly follows from (i) the validity of these properties when $\Fb$ and $\Ebh$ are Hermitian vector bundles, (ii) the compatibility  of the duality functors with inductive and projective limits (see (\ref{dualprojlim}) and (\ref{dualindlim})), and (iii) the fact that any ind- (resp. pro-)Hermitian vector bundle over $\Spec \OK$ may be realized as an inductive (resp. projective) limit of an admissible system of Hermitian vector bundles (\cf Subsections \ref{Consinductive} and \ref{Consprojective}).
\end{proof}

\section{Examples -- I. Formal series and holomorphic functions on disks}\label{ArHB}

Let $R$ be a positive real number.

Let us consider the open disc in $\C$ of radius $R$,
$$D(R) := \{ z \in \C \mid \vert z \vert < R \},$$
 and the space $\cO^{\an}(D(R))$ of holomorphic functions on $D(R).$
 
 Equipped with the topology of uniform convergence on compact subsets of $D(R)$, it is a Fréchet space. Moreover, Taylor expansion at $0$ defines an inclusion
 $$
 \begin{array}{rrcl}
  i_R: &\cO^{\an}(D(R)) & \hlra & \C[[X]] \\
  & f &\longmapsto & \sum_{n \in \N} (1/n!) \, f^{(n)}(0)  \, X^n. 
\end{array}
 $$
 When $\C[[X]] \simeq \C^\N$ is equipped with the its natural Fréchet topology, defined by the simple convergence of coefficients, the map $i_R$ is continuous with dense image.
 
 Let $f$ be an element of $\cO^{\an}(D(R))$ and let
 $$f(z) = \sum_{n \in \N} a_n z^n$$
 be its expansion at the origin.
 
 For any $r \in [0,R[,$ we have:
 \begin{equation}\label{circR}
\int_0^1 \vert f(r e^{2\pi i t}) \vert^2 dt = \sum_{n \in \N} r^{2n} \vert a_n\vert^2.
\end{equation}

This relation shows that, for every $r \in ]0,R[,$ one defines a Hermitian norm $\Vert.\Vert_r$ on  $\cO^{\an}(D(R))$ by letting:
$$\Vert f \Vert_r^2 = \int_0^1 \vert f(r e^{2\pi i t}) \vert^2 dt,$$
and that the \emph{Hardy space} 
$$H^2(R):= \left\{ f \in \cO^{\an}(D(R)) \mid \sup_{r\in [0,R[} \Vert f\Vert_r^2 < + \infty \right\}$$ becomes a Hilbert space when equipped with the norm $\Vert.\Vert_R$ defined by:
$$\Vert f \Vert^2_R := \sup_{r \in [0,R[} \Vert f \Vert^2_r = \sum_{n \in \N} R^{2n} \vert a_n \vert^2.$$ 

From (\ref{circR}), we also derive:
\begin{align}\label{L2R}
\Vert f \Vert_{L^2(D(R)} & := \int_{x + iy \in D(R)} \vert f(x+iy)\vert^2 dx \, dy \\
& = \int_0^R \int_0^{2\pi} \vert f(re^{i\theta}) \vert^2 r dr \, d\theta \\
& = \sum_{n \in \N} \pi (n+1)^{-1} R^{2(n+1)}  \vert a_n \vert^2.
\end{align}
and, equipped with the norm $\Vert.\Vert_{L^2(D(R))},$ the \emph{Bergman space}
$$B(R):= \cO^{\an}(D(R)) \cap L^2(D(R)$$
is also a Hilbert space.

Observe that the expressions (\ref{circR}) and (\ref{L2R}) for the norms $\Vert.\Vert_R$ and  $\Vert.\Vert_{L^2(D(R))}$ show that the vector space $\C[T]$ is dense both in $H^2(R)$ and in $B(R).$ This show that the composite injections
$$i_R^{H}: H^2(R) \hlra \cO^{\an}(D(R)) \hlra \C[[X]]$$
and
$$i_R^{B}: B(R) \hlra \cO^{\an}(D(R)) \hlra \C[[X]],$$
that clearly are continuous, have dense images. 

Moreover, $H^2(R)$ and in $B(R)$ are subspaces of $\cO^{\an}(D(R))$ invariant under the operation of complex conjugation $\overline{\cdot}$, defined by
$$\overline{f}(z) := \overline{f(\overline{z})}= \sum_{n \in \N} \overline{a_n} z^n.$$

Finally, we may define the following pro-Hermitian vector bundles over $\Spec \Z$:
\begin{equation}\label{ArHardy}
\Hbh(R):= (\Z[[X]], (H^2(R), \Vert. \Vert_R), i^H_R)
\end{equation}
and 
$$\Bbh(R):= (\Z[[X]], (B(R), \Vert. \Vert_R, i^B_R)),$$
that may be seen as arithmetic avatars 
of the classical Hardy and Bergman  spaces.

Let us emphasize that the isomorphism class in ${\rm pro\overline{Vect}}(\OK)$ (and \emph{a fortiori} in ${\rm pro\overline{Vect}}^{\leq 1}(\OK)$) of $\Hbh(R)$, or of $\Bbh(R)$, varies with $R \in \R^\ast_+.$ (This may be shown be considering their $\theta$-invariants $h^i_\theta (\Hbh(R) \otimes \cOb(\delta))$; see Proposition \ref{hRdelta}, \emph{infra}.)

\section[Examples -- II.  Injectivity and surjectivity of morphisms]{Examples -- II.  Injectivity and surjectivity of morphisms of pro-Hermitian vector bundles}\label{InjSurj}

In this paragraph, we gather some observations and examples that demonstrate that, if $$f : \Ebh \lra \Fbh$$ is a morphism of pro-Hermitian vector bundles over $\Spec \OK,$ the injectivity (resp. surjectivity) properties of the underlying morphisms
$\hat{f}: \Eh \lra \Fh,$
$\hat{f}_\sigma: \Eh_\sigma \lra \Fh_\sigma$ 
and $f_\sigma: E^\hilb_\sigma \lra F^\hilb_\sigma$
 (of topological $\OK$-modules  and of complex Fréchet and Hilbert spaces) are in general loosely related.
 
 \medskip
 \subsection{} Concerning their injectivity, on the positive side, let us observe that each of the following assertions implies the following one:
 
 $(i)$ the morphism $\hat{f}: \Eh \lra \Fh$ of topological $\OK$-module is injective and strict;
 
 $(ii)$ the $\C$-linear map $\hat{f}_\sigma: \Eh_\sigma \lra \Fh_\sigma$ is injective;
 
 $(iii)$ the $\C$-linear map $f_\sigma: E^\hilb_\sigma \lra F^\hilb_\sigma$ is injective.
 
 \noindent Indeed, the implication $(i) \Longrightarrow (ii)$ follows from Proposition \ref{strictinjective}, and the implication $(ii) \Longrightarrow (iii)$ is clear.
 
 Besides, the injectivity $(ii)$ of $\hat{f}_\sigma$ immediately implies the injectivity of $\hat{f}$. 

  \subsection{} Simple examples of morphisms of pro-Hermitian vector bundles which demonstrate that $\hat{f}$ or $f_\sigma$ may be injective, while  $\hat{f}_\sigma$ is not, are easily obtained by the constructions in Proposition  \ref{padicex}, 2),  and  Proposition \ref{nonstrictDed3}.
 
 Indeed, for any $a\in \Z$, the map $\phi_a$ considered in Proposition \ref{nonstrictDed3} with $A=\Z$ --- namely the map from $\Z[[X]]$ to itself defined by the formula
 \begin{equation}\label{phiadef}
\phi_a (f) := (1- a/X)f + a f(0)/X = f - a\,(f-f(0))/X
\end{equation}
--- defines a morphism between ``arithmetic Hardy spaces" (as defined in (\ref{ArHardy}))
 $$\Phi_a: \Hbh(R) \lra \Hbh(R')$$
 for any two positive real  numbers $R$ and $R'$ such that $R' \leq R$. 
By definition, $\widehat{\Phi_a} := \phi_a$, and the  morphisms  
$$\phi_{a,\C}:= \widehat{\Phi_a}_{\C}: \C[[X]] \lra \C[[X]]$$
and 
$$\phi^\hilb_a:= \Phi_{a,\C}: H^2(R) \lra H^2(R')$$ are still defined by formula (\ref{phiadef}), and actually make sense for any $a \in \C$.

According to Proposition \ref{nonstrictDed3}, the map
$$\widehat{\Phi_a} := \phi_a: \Z[[X]] \lra \Z[[X]]$$
is injective if and only if $a \notin \{1,-1\}$, is not surjective, but satisfies $\phi_a(\Z[X]) = \Z[X]$. Moreover $\phi_{a,\C}$ satisfies $\phi_{a,\C}(\C[X])= \C[X]$ and is injective if and only if $a=0$.

These properties are complemented by the following proposition, that we leave as an easy exercise:

\begin{proposition} For any $a \in \C^\ast,$ the map
$$\phi_{a,\C}: \C[[X]] \lra \C[[X]]$$
is surjective and its kernel is the line $\C. \sum_{k \in \N} a^{-k} X^k.$
 
 For any $a\in \C$ and any two positive real numbers $R$ and $R'$ such that $R'\leq R,$ the continuous linear map 
 $$\phi^\hilb_a: H^2(R) \lra H^2(R')$$
 is injective if and only if $\vert a \vert \leq R.$ When $\vert a \vert > R,$ its kernel is the line $\C.(a-X)^{-1}$. Moreover the following conditions are equivalent:
 
 (i) $\phi^\hilb_a$ is onto;
 
 (ii) $\phi^\hilb_a$ is a strict morphism of complex topological vector space;
 
 (iii) $R=R'$ and $\vert a \vert \neq R.$ \qed
\end{proposition}

In particular, we obtain:

\begin{scholium}
 For any $a \in \Z \setminus\{1,0,-1\},$ and any $R \in \R^\ast_+,$ the morphism of pro-Hermitian vector bundles over $\Spec \Z$
 $$\Phi_a: \Hbh(R) \lra \Hbh(R)$$
 is such that $\widehat{\Phi_a}: \Z[[X]] \lra \Z[[X]]$ is injective and not strict, and $\widehat{\Phi_a}_\C: \C[[X]] \lra \C[[X]]$ is strict and not injective. Moreover $\Phi_{a,\C}: H^2(R) \lra H^2(R)$ is an isomorphism (resp. injective with dense image, but not strict; resp. surjective but not injective) if $\vert a \vert < R$ (resp. if $\vert a \vert =R;$ resp. if $\vert a \vert > R$).  \qed
 \end{scholium}
 
 \subsection{} Concerning the surjectivity properties of the underlying morphisms attached to a morphism 
of pro-Hermitian vector bundles $\hat{f}: \Eh \lra \Fh$ over some arithmetic curve $\Spec \OK$, we may consider the following conditions, where we denote by $\sigma$ some field embedding of $K$ into $\C$:

$(i)$ the image $\hat{f}(\Eh)$ of the morphism $\hat{f}: \Eh \lra \Fh$ is dense in $\Fh$;

$(ii)$ the $K$-linear map $\hat{f}_K: \Eh_K \lra \Fh_K$ is surjective;
 
 $(iii)$ the $\C$-linear map $\hat{f}_\sigma: \Eh_\sigma \lra \Fh_\sigma$ is surjective;
 
 $(iv)$ the image $f_\sigma(E_\sigma^\hilb)$ of the continuous $\C$-linear map $f_\sigma: E^\hilb_\sigma \lra F^\hilb_\sigma$ is dense in $F_\sigma^\hilb$.
 
\noindent The properties of the morphisms in $\CTC_k$ when $k$ is a field presented  in Subsections \ref{appdualCPCTC} and \ref{strictDed} (see notably Propositions \ref{CTCk} and \ref{strictCTCDed1}) show that the morphisms $\hat{f}_K$ and $\hat{f}_\sigma$ in $\CTC_K$ and $\CTC_\C$ are necessarily strict, and are therefore surjective if and only if they have a dense image. Proposition \ref{CTCk} also shows that the surjectivity of $\hat{f}_K$ and of $\hat{f}_\sigma$ are equivalent.

 From these observations, one immediately derives the validity of the following implications:
 $$(i) \Longrightarrow (ii) \Longleftrightarrow (iii) \Longleftarrow (iv).$$
 The converse implications $(ii) \Longrightarrow (i)$ and $(iii) \Longrightarrow (iv)$ do not hold in general. This will be demonstrated by the examples in \ref{iv} and \ref{v} below.
 
 \subsection{}\label{iv}  Let $R$ be a positive real number and let $p$ be a prime number. 
 
 Let us consider the morphism
 $$f: \Hbh(R) \lra \Hbh(R)$$
 defined by the morphism of multiplication by $X-p$. Namely,
 $$\hat{f} := (X-p).\;\; : \Z[[X]] \lra \Z[[X]]$$
 is the morphism $\beta_2$ considered in Proposition \ref{padicex}, 2), 
 and the continuous $\C$-linear map
 $$f_\C:= H^2(R) \lra H^2(R)$$
 is the multiplication by the function $(z\mapsto z-p)$.
 
 Then, according to Proposition \ref{padicex}, the cokernel of $\hat{f}$ may be identified, as a topological $\Z$-module, with the $p$-adic integers $\Z_p$. Moreover, the maps $\hat{f}_\Q: \Q[[X]] \lra \Q[[X]]$ and $\hat{f}_\C: \C[[X]] \lra \C[[X]]$ are still defined by the multiplication by $X-p$, and are therefore isomorphisms. 
 
 Finally the injective map $f_\C$ is an isomorphism (resp. has a dense image, but is not surjective; resp. has a closed image of codimension 1) if $R < p$ (resp. if $R=p;$ resp. if $R>p$).
 
  \subsection{}\label{v} We finally construct an example of a morphism $f: \Ebh \lra \Fbh$ of pro-Hermitian vector bundles over $\Spec \OK$ such that $\hat{f}$ (and consequently $\hat{f}_K$ and $\hat{f}_\sigma$ for every embedding $\sigma:K \hra \C$) is an isomorphism, and $f_\sigma$ is an isometry with infinite dimensional cokernel.
 
 For simplicity, we shall assume that $K= \Q$.
 
 Let $I = [a, b]$ be a closed bounded interval in $\R,$ of length $b-a >0.$
 
 To $I$, we may attach the ind-Hermitian vector bundle $\overline{V}_I:=(V_I, \Vert.\Vert_I)$ over $\Spec \Z$ defined by the $\Z$-module 
 $V_I := \Z[X]$ equipped with the $L^2$-norm $\Vert.\Vert_I$ on $V_{I,\C} = \C[X]$ defined by
 $$\Vert P\Vert^2_I := \int_I \vert P(t)\vert^2 \, dt.$$
 The completion of the complex normed space $(V_C, \Vert.\Vert_I)$ is the Hilbert space $L^2(I).$
Consequently, if we consider the pro-Hermitian vector bundle 
$$\Fbh_I := \overline{V}_I^\vee$$
dual to $\overline{V}_I,$ its underlying Hilbert space $F_{I\C}^\hilb$ may also be identified\footnote{The natural isomorphism $L^2(I) \simeq F_{I\C}^\hilb$ is characterized by the fact that the composite map $L^2(I) \simeq F_{I\C}^\hilb \hlra \hat{F}_{I,\C} := \Hom_\C(V_{I,\C}, \C)$ sends a function $\phi \in L^2(I)$ to the linear form $(P\mapsto \in_I \phi(t) P(t) \, dt)$ on  $V_{I,\C} = \C[X]$.} with $L^2(I)$.
 
 Let $I'$ be another closed bounded interval of positive length, and let $\overline{V}_{I'}$ and $$\Fbh_{I'} := \overline{V}_{I'}^\vee$$ be the associated ind- and pro-Hermitian vector bundles over $\Spec \Z.$ Let us moreover assume that 
 $$I' \subset I.$$
 
 The the identity map $Id_{\Z[X]}$ defines a morphism $\rho{I'I}$ in $\Hom_\Z^{\leq 1}(\overline{V}_{I}, \overline{V}_{I'})$. Its adjoint map
 $$\eta_{II'} := \rho_{I'I}^\vee \in \Hom_Z^{\leq 1} (\Fbh_{I'}, \Fbh_{I})$$
 is easily seen to the morphism of pro-Hermitian vector bundles over $\Spec \Z$ defined by the morphism 
 $$\overline{\eta_{II'}} := Id_{\Z[X]^\vee}$$
 of topological $\Z$-modules and the continuous $\C$-linear map ``extension by zero"
 $$\eta_{II',\C} : L^2(I') \lra L^2(I)$$ 
that sends a function $\Psi \in L^2(I')$ to the function $\eta_{II',\C}(\psi)$ such that
$$
\begin{array}{cll}
 \eta_{II',\C}(\psi)(x) & := \psi(x)  & \mbox{ if $x\in I'$}  \\
  & := 0  &  \mbox{ if $x \in I \setminus I'.$}  
\end{array}
$$

This map $\eta_{II',\C}$ is an isometry, and its cokernel may be identified  with $L^2(I \setminus I')$, which is infinite dimensional if $I' \neq I.$ 

\section{Examples -- III. Subgroups of pre-Hilbert spaces and ind-Euclidean lattices}\label{HilbertSubgroups}

Let $(H, \Vert. \Vert)$ be a real pre-Hilbert space, and let $\Gamma$ be a subgroup of $(H,+).$ 

The inclusion morphism $i : \Gamma \lra H$ uniquely extends to a $\R$-linear map 
$i_\R: \Gamma_\R := \Gamma \otimes_\R \R \lra H.$ Its image is 
$$\im i_\R = \sum_{f \in \Gamma} \R.f.$$

\begin{proposition}\label{GammaH1}
 The following two conditions are equivalent:

\noindent  (i) The  map $i_\R$ is injective and the $\Gamma$ is  an object of $\CP_\Z.$

\noindent (ii) The group $\Gamma$ is countable, and the following condition is satisfied:

$({\mathbf F})$ For any finite subset $F$ of $\Gamma$ the abelian group $\sum_{f \in F}  \Z. f$ has finite index in $\Gamma \cap \sum_{f \in F} \R. f.$

\end{proposition}
Recall that the objects of $\CP_\Z$ are precisely the countable free $\Z$-modules. 

Observe also that $i_\R$ is injective if and only if $\Vert i_\R(.)\Vert$ is a prehilbertian norm on  $\Gamma_\R$. Consequently, Condition (i) may be rephrased as:

\emph{(i') The pair $(\Gamma, \Vert i_\R(.)\Vert)$ defines an ind-Euclidean lattice.}

\begin{proof} The direct implication (i) $\Rightarrow$ (ii) is straightforward, since Condition $({\mathbf F})$ is satisfied by the subgroup $\Gamma = \Z^I$ of the real vector space $\Gamma_\R \simeq \R^I,$ for any (countable) set $I$.

To establish the converse implication (ii) $\Rightarrow$ (i), let us assume that (ii) is satisfied, and consider a sequence $(f_n)_{n \in \N} \in \Gamma^\N$ in which any element of $\Gamma$ occurs.

For any $i \in \N,$ let us consider
$$\Gamma_i := \Gamma \cap \sum_{0\leq n \leq i} \R. f_i.$$
It is a torsion free $\Z$-module, which contains $\sum_{0 \leq n \leq i} \Z. f_i$ as a submodule of finite index. Consequently $\Gamma_i$ is a finitely generated torsion free $\Z$-module. Moreover, for any $i \in \N,$ $\Gamma_i$ is clearly a saturated $\Z$-submodule of $\Gamma_{i+1}$, and $\Gamma = \cup_{i \in \N} \Gamma_i.$ 

Therefore, the $\Z$-module $\Gamma$ is countably generated and projective, hence isomorphic to $\Z^{(I)}$, with $I$ at most countable (this follows from the implications (4) $\Rightarrow$ (1) and (3) in Proposition \ref{CPAdef}). We may assume that 
$$I =\{ n \in \N \mid n < N \}$$
with $N:= \vert I \vert \in \N \cup\{+\infty\}.$ 

Let $(e_i)_{i \in I}$ be a $\Z$-basis of $\Gamma.$ It is also a $\R$-basis of $\Gamma_\R$, and to prove the injectivity of $i_\R,$ we are left to show that the family $(e_i)_{i\in I}$, considered as a family of vectors in $H$, is free over $\R.$

To achieve this, observe that, for every $n\in I,$ the $\Z$-module $\sum_{0\leq i < n} \Z. e_i$ is a saturated submodule of $\Gamma$, hence a saturated submodule of $\Gamma \cap \sum_{0\leq i < n} \R. e_i$. Besides,  the validity of $({\mathbf F})$ implies that $\sum_{0\leq i < n} \Z. e_i$ has finite index in $\Gamma \cap \sum_{0\leq i < n} \R. e_i$. Consequently, these two $\Z$-modules coincide:
\begin{equation}\label{sumsum}
\sum_{0\leq i < n} \Z. e_i = \Gamma \cap \sum_{0\leq i < n} \R. e_i.
\end{equation}

If the family $(e_i)_{i\in I}$ were not free over $\R$, there would exist $n\in I$ such that $e_n \in  \sum_{0\leq i < n} \R. e_i.$ This would contradict (\ref{sumsum}).
\end{proof}

\begin{proposition}\label{GammaH2} 
 If $\Gamma$ is a discrete subgroup in the pre-Hilbert space $(H,\Vert.\Vert),$ then, for any finite subset $F$ of $\Gamma,$ the subgroups $\sum_{f \in F}  \Z. f$ and $\Gamma \cap \sum_{f \in F} \R. f$ are lattices in the finite dimensional real vector space  $\sum_{f \in F} \R. f$.
 \end{proposition}

\begin{proof} The subgroups $\sum_{f \in F}  \Z. f$ and $\Gamma \cap \sum_{f \in F} \R. f$ satisfy the inclusion
$$\sum_{f \in F}  \Z. f  \subset \Gamma \cap \sum_{f \in F} \R. f.$$
Besides $\sum_{f \in F}  \Z. f$ generates the real vector space $\sum_{f \in F} \R. f$, and $\Gamma \cap \sum_{f \in F} \R. f$ is discrete in this vector space, since $\Gamma$ is 
discrete in $H$. Therefore both  $\sum_{f \in F}  \Z. f$ and $\Gamma \cap \sum_{f \in F} \R. f$ generate and are discrete in $\sum_{f \in F} \R. f$.
\end{proof}

From Propositions \ref{GammaH1} and \ref{GammaH2}, we immediately derive:

\begin{corollary}\label{corHilbertSubgroups} If the pre-Hilbert space $(H, \Vert.\Vert)$ is separable and if $\Gamma$ is discrete in $H$, then the 
pair $(\Gamma, \Vert i_\R(.)\Vert)$ defines an ind-Euclidean lattice. \qed
\end{corollary}

\medskip

\chapter[$\theta$-Invariants of infinite dimensional Hermitian vector bundles]{$\theta$-Invariants of infinite dimensional Hermitian vector bundles:  definitions and first properties}\label{thetainfinite}

\medskip

In this chapter, we extend the definitions and some of the basic properties $\hot$ and $\hut$ to infinite-dimensional vector bundles over an arithmetic curve $\Spec \OK$. As in the previous chapter, many constructions will be of a rather formal nature, and details could be skipped at first reading.

We should however emphasize that the extension of $\hot$ to general pro-Hermitian vector bundles cannot be completely formal. Indeed, as shown by Proposition  \ref{limsupinf} and the examples in paragraph \ref{examplecountable}, if some pro-Hermitian vector $\Ebh$ is realized as the projective limit $\varprojlim_i \Eb_i$ of some  projective system 
$$\Eb_{\bullet} : \Eb_0 \stackrel{q_0}{\longleftarrow}\Eb_1 \stackrel{q_1}{\longleftarrow}\dots \stackrel{q_{i-1}}{\longleftarrow}\Eb_i \stackrel{q_i}{\longleftarrow} \Eb_{i+1} \stackrel{q_{i+1}}{\longleftarrow} \dots$$
of surjective admissible morphisms of Hermitian vector bundles over $\Spec \OK,$ then the sequence $(\hot(\Eb_i))_{i \in \N}$ is in general not convergent in $[0, +\infty]$; moreover, when it converges, its limit may depend of the chosen projective system $\Eb_{\bullet}$.  This issue motivates our introduction of \emph{two} generalizations, denoted by $\lhot$ and $\uhot$ , of $\hot$ to the infinite dimensional setting.

The last paragraphs of this chapter are again dedicated to concrete examples involving these invariants. Notably we pursue the study of the ``arithmetic Hardy spaces" $\Hbh(R)$ introduced in the previous chapter.

\bigskip

We still denote by $K$ a number field, by $\OK$ its ring of integers, and by $\pi: \Spec \OK \lra \Spec \Z$ the morphism of schemes from $\Spec \OK$ to $\Spec \Z.$ .

\section{Limits of $\theta$-invariants}\label{easytheta} It is straightforward that, 
for any ind-Hermitian vector bundle $\Fb$ over $\Spec \OK,$  the following limit exists in $[0, + \infty]$:
\begin{equation}\label{easyhot}
\hot(\Fb) := \lim_{F' \in \cF(F)} \hot(\Fb').
\end{equation}
Indeed $\hot(\Fb')$ is an increasing function $F' \in \cF(F)$ and this limit  is actually given by the expression:
$$
\hot (\Fb) := \log \sum_{v \in F} e^{- \pi \Vert v \Vert^2_{\pi_\ast \Fb}}$$
(where, by convention $\log(+ \infty) = +\infty$). 

Observe also the equality:
\begin{equation}\label{easyhotsup}
\hot(\Fb) := \sup_{F' \in \cF(F)} \hot(\Fb').
\end{equation}
In particular, if 
$$\Fb_0 \hlra \Fb_1 \hlra \dots \hlra  \Fb_{i} \hlra \Fb_{i+1}\hlra \dots$$
is an inductive system of admissible injections of Hermitian vector bundles,
\begin{equation}\label{easyhotlim}
\hot(\varinjlim_{i} \Fb_i) = \lim_{i \rightarrow + \infty} \hot (\Fb_i).
\end{equation}

Dually, for any pro-Hermitian vector bundle  $\Ebh := (\hE, (\Eb_U)_{U \in \cU(\hE)})$ over $\Spec \OK,$ $\hut(\Eb_U)$ is a decreasing function of $U \in \cU(\hE)$, and we define:
\begin{equation}\label{easyhut}
\hut(\Et):= \lim_{U \in \cU(\hE)} \hut(\Eb_U).
\end{equation}
It is again an element of $[0, + \infty]$, which, according to its very definition satisfies the following duality relation:

\begin{proposition} Let $\Fb$ be an ind-Hermitian vector bundle over $\Spec \OK,$ and let $\Ebh := \Fb^\vee$ denote the dual pro-Hermitian vector bundle. Then the following equality holds in $[0, + \infty].$
$$\hut(\Ebh) = \hot(\Fb).$$ \qed
\end{proposition}

By duality, the relation (\ref{easyhotsup}) becomes:
\begin{equation}\label{easyhutsup}
\hut(\Ebh) := \sup_{U \in \cU(\hE)} \hut(\Eb_U),
\end{equation} 
and (\ref{easyhotlim}) shows that, for any projective system 
$$\Eb_0 \longleftarrow \Eb_1 \longleftarrow \dots \longleftarrow\Eb_i \longleftarrow \Eb_{i+1} \longleftarrow \dots$$
of admissible surjections of Hermitians vector bundles, we have
\begin{equation}\label{easyhutlim}
\hut(\varprojlim_{i} \Eb_i) = \lim_{i \rightarrow + \infty} \hut (\Eb_i).
\end{equation}

The proof of the following proposition is left as an easy exercise:

\begin{proposition}\label{indthetafinite} Let $\Fb:= (F, \Vert. \Vert)$ be an ind-Hermitian vector bundle over $\Spec \Z$, and let 
$$N_{\Fb}: \R_{+} \lra \N_{>0} \cup \{+\infty\} $$
be the non-decreasing function defined by
$$N_{\Fb}(x) := \vert \{ f \in F \mid \Vert f \Vert \leq x \} \vert.$$

If $\hot(\Fb \otimes \cOb(\lambda_{0}))$ is finite for some $\lambda_{0} \in \R$, then $\hot(\Fb\otimes \cOb(\lambda))$ is finite for every $\lambda \in ]-\infty, \lambda_{0}].$
Moreover, the following three conditions are equivalent:

(i) For any $\lambda \in \R,$  $\hot(\Fb\otimes \cOb(\lambda))$ is finite.

(ii) For every $t \in \R^\ast_{+},$ $$\sum_{f \in F} e^{- \pi t \Vert f \Vert^2} < + \infty.$$

(iii) For any $x\in \R_{+},$ $N_{\Fb}(x)$ is finite and, when $x$ goes to $+\infty,$
$$\log N_{\Fb}(x) = o(x^{1/2}).$$
\qed
\end{proposition}

\section{Upper and lower $\theta$-invariants}

In the applications to Diophantine geometry developed in the sequel to this monograph, the significant extension of $\theta$-invariants to infinite rank Hermitian vector bundles is not the ``obvious" ones, simply defined as limits, that we introduced in the previous paragraph \ref{easytheta}. Instead, a key role will be played by some avatar of the invariant $\hot$ attached to a \emph{pro}-Hermitian vector bundle (or dually, by some avatar of  $\hut$ attached to an \emph{ind}-Hermitian vector bundle).

\subsection{Upper limits of $\theta$-invariants} As demonstrated by the next proposition, if we want them to take finite values on some infinite rank Hermitian vector bundles, these generalized invariants cannot be defined by the obvious modifications of the definitions (\ref{easyhut}) and (\ref{easyhot}), where $\hut$ would be replaced by $\hot$ and \emph{vice versa}. 

\begin{proposition}\label{limsupinf} For any pro-Hermitian vector bundle  $\Ebh := (\hE, (\Eb_U)_{U \in \cU(\hE)})$ of infinite rank over $\Spec \OK$,
we have:
\begin{equation}\label{limsupinf1}
\limsup_{U \in \cU(\hE)} \hot(\Eb_U) = + \infty.
\end{equation}
 
 For any ind-Hermitian vector bundle $\Fb$ of infinite rank over $\Spec \OK,$ we have:
 \begin{equation}\label{limsupinf2}
 \limsup_{F' \in \cFS(F)} \hut(\Fb') = + \infty.
 \end{equation}
\end{proposition}

\begin{proof}  The proof will be based on the following classical fact:

\begin{lemma}\label{simple2} Let $\hat{V}$ be a Hermitian vector bundle of rank at least two over $\Spec \OK$. 

The  saturated $\OK$-submodules $W$ of rank $\rk V -1$ in $V$ are in bijection with the $K$-points of $\PP(V_K)$ by the map which associates the  quotient map $V_K \lra V_K/W_K$ to $W$. 

Moreover one defines a height  function $h_{\PP(V_K)}: \PP(V_K)(K) \lra \R$ associated to the line bundle  $\cO(1)$ over the projective space $\PP(V_K)$ by the formula:
$$h_{\PP(V_K)}(W) := \dega \overline{V/W}.$$

In particular, there exists saturated $\OK$-submodules $W$  of rank $\rk V -1$ in $V$ such that $\dega \overline{V/W}$ take  arbitrary large positive values.
 \qed 
\end{lemma}

Consider a pro-Hermitian vector bundle $\Ebh$ of infinite rank over $\Spec \OK$ and $U$ a element of $\cU(\hE)$. We want to prove that there exists $U' \in \cU(\hE)$ contained in $U$ with $\hot(\Eb_{U'})$ arbitrary large.

As $\hE$ has infinite rank, there exists $\tilde{U}$ in $\cU(\hE)$ such that 
$p_{\tilde{U} U}: E_{\tilde{U}} \lra E_U$ has a kernel $V:= \ker p_{\tilde{U} U}$ of rank  at least $2$. Let $W$ be a saturated submodule of rank $\rk V -1$ in $V$. The inverse image 
$$U' := q_{\tilde{U}}^{-1}(W)$$
of the saturated submodule $W$ of $E_{\tilde{U}}$ by the quotient map $q_{\tilde{U}}: \hE \lra \hE_{\tilde{U}}$ is an element of $\cU(\Eb)$  contained in $U$. Moreover $\Eb_{U'}$ may be identified with the quotient $\overline{E_U/W}$, which itself contains $\overline{V/W}$. Therefore:
$$\hot(\Eb_{U'}) = \hot(\overline{E_U/W}) \geq \hot(\overline{V/W}) \geq \dega \overline{V/W}.$$
According to Lemma \ref{simple2}, this Arakelov degree takes an arbitrary large positive value for a suitable choice  of $W$.

This completes the proof of (\ref{limsupinf1}). The formula (\ref{limsupinf2}) follows, by duality, from   (\ref{limsupinf1}) applied to 
$\Ebh := \Fb^\vee.$
\end{proof}

\subsection{Definitions} In spite of the ``divergent behavior" of the $\theta$-invariants highlighted in Proposition \ref{limsupinf}, it is possible to
introduce some ``lower versions" $\underline{h}^i_\theta$ and some ``upper versions" $\overline{h}^i_\theta$ of these invariants, that will belong to $[0, +\infty]$ in general, but will achieve finite values in  significant instances. 

For any pro-Hermitian vector bundle  $\Ebh := (\hE, (\Eb_U)_{U \in \cU(\hE)})$ over $\Spec \OK$, 
we may consider the class $\cL(\Ebh)$ of pairs $(\Eb', \iota)$, where $\Eb'$  is some Hermitian vector bundle over $\Spec \OK$ and $\iota$ an element of $\Hom_\OK^{\leq 1}(\Eb', \Ebh)$ such that $\iota: E' \lra \Eh$ is injective. 

We may also consider the sub-class $\cL_{\rm ad}(\Ebh)$ of $\cL(\Ebh)$ of the pairs $(\Eb', \iota)$ as above such that moreover the image $\iota(E')$ is a saturated $\OK$-submodule of $\Ebh$ and, for every embedding $\sigma: K \hra \C,$ the maps $\iota_\sigma$ is an isometry from $(E'_\sigma, \Vert.\Vert_{\Eb',\sigma})$ to $E_\sigma^\hilb$ equipped with its natural Hilbert norm $\Vert.\Vert_{E_\sigma^{\hilb}}$, defined by (\ref{defhilpro}).

Then we may introduce the following definitions: 
\begin{equation}\label{lhot}
\lhot(\Ebh):= \sup \{\hot(\Eb'); (\Eb', \iota) \in \cL(\Ebh) \} 
\end{equation}
and
\begin{equation}\label{uhot}
\uhot(\Ebh):= \liminf_{U \in \cU(\hE)} \hot(\Eb_U).
\end{equation}

Let $\Fb$ denote some ind-Hermitian vector bundle over $\Spec \OK$. We recall that we denote by ${\rm co}\cF\cS(F)$ the family of $\OK$-submodules $F'$ of $F$ such that $F/F'$ is finitely generated and torsion free, and by $\cF(F)$ (resp. by $\cF\cS(F)$) the family of finitely generated (resp. finitely generated and saturated) $\OK$-submodules of $F$ (\cf Section \ref{CatCP} \emph{supra}). 

We may also consider the following subset of  ${\rm co}\cF\cS(F)$:
\begin{multline*}
\cF\cQ(\Fb) := \\
 \left\{ F' \in {\rm co}\cF\cS(F) \mid \mbox{for every $\sigma: K \hra \C,$ $F'_\sigma$ is closed in the pre-hilbert space $(F_\sigma, \Vert.\Vert_{\Fb, \sigma})$} \right\}.
 \end{multline*}
 For any $F \in \cF\cQ(\Fb),$ the finitely generated projective $\OK$-module $F/F'$ becomes a Hermitian vector bundle over $\Spec \OK$, that we shall denote by $\overline{F/F'}$, when equipped with the quotient norms of the norms $\Vert. \Vert_{\Fb, \sigma}.$ 

Using this notation, we may also define:
\begin{equation}\label{lhut}
\lhut(\Fb) := \sup \{ \hut(\overline{F/F'}); {F' \in \cF\cQ(\Fb)} \}
\end{equation}
and
\begin{equation}\label{uhut}
\uhut(\Fb) := \liminf_{F' \in \cF(F)} \hut(\Fb').
\end{equation}

\subsection{Variants} The following two propositions provide alternative definitions of the invariants $\lhot, \uhot, \lhut,$ and $\uhut$ that we have just introduced.

\begin{proposition}\label{vardef}
 For any pro-Hermitian vector bundle $\Ebh$ over $\Spec \OK,$ we have:
 \begin{equation}\begin{split} \label{lhotvar}
\lhot(\Ebh) & = \sup \left\{\hot(\Eb'); (\Eb', \iota) \in \cL_{\rm ad}(\Ebh) \right\}  \\
                   & = \sup \left\{ \hot(\overline{P}); \mbox{$P$ finitely generated $\OK$-submodule of $\Eh \cap  E^\hilb_\R$} \right\}  \\
                   & = \sup \left\{ \hot(\overline{P}); \mbox{$P$ finitely generated saturated $\OK$-submodule of $\Eh \cap  E^\hilb_\R$} \right\}.
                   \end{split} 
                   \end{equation}
                                      
 Consequently, we also have:                  
\begin{equation}\label{pistarlhot}
\lhot(\Ebh) = \lhot(\pi_\ast\Ebh) = \log \sum_{ v \in \hE \cap E_\R^\hilb} e^{-\pi \Vert v \Vert_{\pi_\ast \Ebh}^2}.
\end{equation}
\end{proposition}

In the second and third lines of (\ref{lhotvar}), we have denoted by $\overline{P}$ the Hermitian vector bundle over $\Spec \OK$ defined by the finitely generated projective $\OK$-module $P$ equipped with the restrictions of the 
 Hilbert norms  $\Vert.\Vert_{E_\sigma^{\hilb}}$ on the Hilbert  spaces $E_\sigma^\hilb$.
Observe also that $\Eh \cap  E^\hilb_\R$ is a saturated $\OK$-submodule of $\Eh$, and that consequently, a submodule of  $\Eh \cap  E^\hilb_\R$ is saturated in $\Eh \cap  E^\hilb_\R$ if and only if it is saturated in $\Eh$.

\begin{proof} The relations (\ref{lhotvar}) directly follows from the definition of $\lhot(\Ebh)$, from the increasing character of the $\theta$-invariant $\hot$ for Hermitian vector bundles (see Proposition \ref{ineqmortheta}, 1)), and from the properties of the saturation of finitely generated $\OK$-submodules of  objects in $\CTC_{\OK}$ (see Corollary \ref{SatCTC}). We leave the details to the reader. 

The equality 
\begin{equation}\label{quasipistarlhot}
\lhot(\Ebh) = \log \sum_{ v \in \hE \cap E_\R^\hilb} e^{-\pi \Vert v \Vert_{\pi_\ast \Ebh}^2}.
\end{equation}
then follows from the second line in (\ref{lhotvar}). Indeed, by the very definition of $\hot,$ we have:
$$\hot(\overline{P}) =  \log \sum_{ v \in P} e^{-\pi \Vert v \Vert_{\pi_\ast \Ebh}^2},$$
and  any finite subset of $\hE \cap E_\R^\hilb$ is contained in some finitely generated $\OK$-submodule of $\hE \cap E_\R^\hilb$.

Applied to the pro-Hermitian vector bundle $\pi_\ast\Ebh$ over $\Spec \Z,$ the equality (\ref{quasipistarlhot}) becomes 
$$\lhot(\pi_\ast\Ebh) = \log \sum_{ v \in \hE \cap E_\R^\hilb} e^{-\pi \Vert v \Vert_{\pi_\ast \Ebh}^2}.$$
This completes the proof of (\ref{pistarlhot}).
\end{proof}

Observe that, with the notation of Proposition \ref{vardef}, we also have, as a straightforward consequence of the definition of $\uhot$:
\begin{equation}\label{pistaruhot}
\uhot(\Ebh) \geq \uhot(\pi_\ast\Ebh). 
\end{equation}

\begin{proposition}
For any ind-Hermitian vector bundle $\Fb$ over $\Spec \OK,$ we have:
\begin{equation}\label{uhutvar}
\uhut(\Fb) = \liminf_{F' \in \cFS(F)} \hut(\Fb').
\end{equation}
\end{proposition}
\begin{proof} The inequality  $$\uhut(\Fb) \leq \liminf_{F' \in \cFS(F)} \hut(\Fb')$$ is clear. 

To prove the converse inequality, simply observe that, for any submodule $F' \in \cF(F)$, its saturation $F'^{\rm sat}$ belongs to $\cFS(F)$ and satisfies $\hut(\overline{F'^{\rm sat}}) \leq \hut (\Fb').$
\end{proof}

The expression (\ref{pistarlhot}) 
 shows that the finiteness of $\lhot(\Ebh)$ implies the countability of $\hE \cap E_\R^\hilb.$ The second half of the following easy proposition  is a refinement of this observation:
\begin{proposition}\label{boundhot} For any ind-Hermitian vector bundle $\Fb$ over  $\Spec \OK$ and any $\lambda \in \R_+,$ we have:
\begin{equation}\label{hotfinind}
\vert \{ x \in F \mid \Vert x \Vert_{\pi_\ast \Fb} \leq \lambda \} \vert \leq \exp (\hot(\Fb) + \pi \lambda^2).
\end{equation}

For any pro-Hermitian vector bundle $\Ebh$ over  $\Spec \OK$ and
for any $\lambda \in \R_+,$ we have:
\begin{equation}\label{hotfinpro}
\vert \{ x \in \hE \cap E_\R^{\hilb} \mid \Vert x \Vert_{\pi_\ast \Ebh} \leq \lambda \} \vert \leq \exp (\lhot(\Ebh) + \pi \lambda^2).
\end{equation}
 
\end{proposition}

\begin{proof} Let us consider a finite subset $A$ of $\{ x \in F \mid \Vert x \Vert_{\pi_\ast \Fb} \leq \lambda \}$. The following inequalities are straightforward:
$$\vert A \vert . \exp (-\pi \lambda^2) \leq \sum_{v\in A} e^{-\pi \Vert v\Vert_{\pi_\ast \Fb}^2} \leq \sum_{v\in F} e^{-\pi \Vert v\Vert_{\pi_\ast \Fb}^2} =: \exp (\hot(\Fb)).$$
Their validity for an arbitrary $A$ as above establishes (\ref{hotfinind}).  

The proof of (\ref{hotfinpro}) is similar to the one of (\ref{hotfinind}).
\end{proof}

\section{Basic properties}

\subsection{Duality}
The newly defined $\theta$-invariants of ind- and pro-Hermitian vector bundles 
still satisfy a duality relation:

\begin{proposition} Let $\Fb$ be an ind-Hermitian vector bundle over $\Spec \OK,$ and let $\Ebh := \Fb^\vee$ denote the dual pro-Hermitian vector bundle. Then the following equalities hold in $[0, + \infty]:$
\begin{equation}\label{dualencore}
\lhut(\Fb) = \lhot(\Ebh) \,\,\,\,\mbox{ and }\,\,\,\, \uhut(\Fb)
 = \uhot(\Ebh).
 \end{equation}
\end{proposition}

\begin{proof} According to the relations (\ref{lhut}) and (\ref{lhotvar}), the first equality may be written: 
\begin{multline}\label{firstdual}
 \sup \left\{ \hut(\overline{F/F'}); {F' \in \cF\cQ(\Fb)} \right\} = \\
 \sup \left\{ \hot(\overline{P}); \mbox{$P$ finitely generated saturated $\OK$-submodule of $\Eh \cap  E^\hilb_\R$} \right\}.
 \end{multline}
 and, according to (\ref{uhutvar}) and (\ref{uhot}), the second one may be written: 
 \begin{equation}\label{seconddual}
 \liminf_{F' \in \cFS(F)} \hut(\Fb') = \liminf_{U \in \cU(\hE)} \hot(\Eb_U).
 \end{equation}
 
 To establish (\ref{seconddual}), recall that, according to Corollary \ref{FSU}, there is an inclusion reversing bijection
$$.^\perp : \cF\cS(F) \lrasim \cU(\Eh).$$
and for any $U \in \cU(\Eh),$ there is a canonical isomorphism of finitely generated projective $\OK$-modules:
\begin{equation}
I_U: E_U = \Eh/U \lrasim (U^\perp)^\vee
\end{equation}(see diagram (\ref{Iperp})).
This isomorphism is easily seen to define an isometric isomorphism between the Hermitian vector bundles $\Eb_U$ and $\overline{U^\perp}^\vee$.
Consequently:
\begin{equation*}
\hot(\Eb_U) = \hot(\overline{U^\perp}^\vee) = \hut(\overline{U^\perp}).
\end{equation*}
The equality (\ref{seconddual}) follows by taking the lower limit over $U$ in the filtering ordered set $\cU(\Eh)$.

The proof of (\ref{firstdual}) is similar, and we shall leave some details to the reader. We may define an inclusion reversing bijection 
$$.^\perp: \cF\cQ(\Fb) \lrasim 
\left\{ \mbox{finitely generated saturated $\OK$-submodule of $\Eh \cap  E^\hilb_\R$} \right\}$$
by sending a submodule $F'$ in $\cF\cQ(\Fb)$ to
$$F'^\perp := \left\{ \xi \in \Eh := F^\vee \mid \xi_{\mid F'} = 0 \right \}.$$
Moreover, for any $F' \in \cF\cQ(\Fb)$, the $\OK$-linear map 
$$
\begin{array}{rcl}
 F'^\perp & \lra   & (F/F')^\vee   \\
 \xi & \longmapsto   & ([x] \mapsto \xi(x))  
\end{array}$$
defines an isomorphism of Hermitian vector bundles over $\Spec \OK$:
$$\Fb' \lrasim \overline{F/F'}^\vee.$$
In particular, we have:
$$\hut(\overline{F/F'}) = \hot(\overline{F/F'}^\vee) = \hot(\overline{F'^\perp}).$$
The equality (\ref{firstdual}) follows by taking the supremum over $F'$ in $\cF\cQ(\Fb)$.
\end{proof}

\begin{corollary} 
For any ind-Hermitian vector bundle $\Fb$ over  $\Spec \OK$, 
\begin{equation}\label{pistarluhut}
\lhut(\Fb) = \lhut(\pi_\ast \Fb) \,\,\,\,\mbox{ and }\,\,\,\, \uhut(\Fb) \geq \uhut(\pi_\ast \Fb).
\end{equation}
\end{corollary}
\begin{proof} This follows from the duality relations (\ref{dualencore}) and from the similar properties (\ref{pistarlhot}) and (\ref{pistaruhot}) satisfied by $\lhot$ and $\uhot$.
\end{proof}

\subsection{Additivity. Comparing $\lhot$ and $\uhot$}

Observe that the expression (\ref{pistarlhot}) for $\lhot(\Ebh)$ implies that $\lhot$ is an additive invariants of pro-Hermitian vector bundles, and dually that $\lhut$ is an additive invariant of ind-Hermitian vector bundles:

\begin{corollary}\label{cor:addlhtheta} For any finite family $(\Ebh_i)_{1 \leq i \leq n}$ (resp. $(\Fb_i)_{1 \leq i \leq n}$) of pro-Hermitian (resp. of ind-Hermitian) vector bundles over $\Spec \OK,$
\begin{equation}\label{addlhtheta}
\lhot(\bigoplus_{1 \leq i \leq n} \Ebh_i) = \sum_{1\leq i \leq n} \lhot(\Ebh_i)  \,\,\,\,(\mbox{resp. } 
\lhut(\bigoplus_{1 \leq i \leq n} \Fb_i) = \sum_{1\leq i \leq n} \lhut(\Fb_i)).
\end{equation}
\end{corollary} \qed

\begin{proposition} For any pro-Hermitian vector bundle $\Ebh$ over  $\Spec \OK$, 
\begin{equation}\label{ineqluhot}
\lhot(\Ebh) \leq \uhot(\Ebh).
\end{equation}

For any ind-Hermitian vector bundle $\Fb$ over  $\Spec \OK$, 
\begin{equation}\label{ineqluhut}
\lhut(\Fb) \leq \uhut(\Fb). 
\end{equation}
\end{proposition}

\begin{proof} To prove (\ref{ineqluhot}), let us consider a finitely generated $\OK$-submodule $P$ of $\Eh \cap E^\hilb_\R$.

 For $U$ small enough in $\cU(\Eh),$ the quotient morphism $p_U: \Eh \lra E_U$ has an injective restriction:
$$p_{U \mid P}: P\, \hlra E_U$$
(see Proposition \ref{injPN}).
Moreover, for any $x$ in $P$ and for any embedding $\sKC,$ the norm $\Vert p_U(x)\Vert_{\Eb_{U,\sigma}}$ is a non-decreasing function of $U \in \cU(\hE)$ and converges to $\Vert x \Vert_{E_\sigma^\hilb}$ when $U$ becomes arbitrarily small. 

This shows that 
$$\hot(\overline{P}) = \lim_{U \in \cU(\hE)} \hot(\overline{p_U(P)}).$$
Together with the obvious inequality
$$ \hot(\overline{p_U(P)}) \leq \hot(\Eb_U),$$
this implies the estimate
$$\hot(\overline{P}) \leq \liminf_{U \in \cU(\hE)} \hot(\Eb_U)$$
and establishes  (\ref{ineqluhot}).

The second inequality (\ref{ineqluhut}) follows from (\ref{ineqluhot}) and from the duality relations (\ref{dualencore}).
\end{proof}

The possible existence of pro-Hermitian vector bundles $\Eb$ such that 
\begin{equation}\label{intriguing}
\lhot(\Ebh) < \uhot(\Ebh)
\end{equation}
is an intriguing issue. A related issue is the additivity of $\uhot$, namely the validity of the first part of (\ref{addlhtheta}) for $\uhot$ instead of $\lhot$.\footnote{With the notation of Corollary \ref{cor:addlhtheta}, the inequality
$\uhot(\bigoplus_{1 \leq i \leq n} \Ebh_i) \leq \sum_{1\leq i \leq n} \uhot(\Ebh_i)$
holds as a straightforward consequence of the definition of $\uhot$ and of the additivity of $\hot$ stated in Proposition \ref{prop:htadd}. The additivity of $\lhot$ shows that examples of pro-Hermitian vector bundles $\Ebh_i$ for which this inequality is strict would lead to examples where (\ref{intriguing}) is strict.}

In the geometric situation  where one deals with pro-vector bundles $\Eh$ over a smooth projective curve $C$ over some field $k$, pro-vector bundles $\Eh$ whose invariants $h^0(C,\Eh)$ and $\overline{h}^0(C,\Eh)$ --- which are the geometric counterparts of the invariants $\lhot(\Eb)$ and $\uhot(\Eb)$
of a pro-Hermitian vector bundle --- satisfy 
$$h^0(C,\Eh)  <\overline{h}^0(C,\Eh)$$
may be constructed when $C$ is an elliptic curve and, at least under some mild technical assumption, on any $C$ of genus $g>1$. 
The existence of such ``wild" pro-vector bundles makes very likely the existence of pro-Hermitian vector bundles satisfying (\ref{intriguing}).  

\subsection{Monotonicity properties}
From now on, we will focus on the properties of the invariants $\lhot$ and $\uhot$ associated to pro-Hermitian vector bundles and we shall leave the formulation of the dual properties of the invariants $\lhut$ and $\uhut$ associated to ind-Hermitian vector bundles to the interested reader. 

\begin{proposition}\label{ineqmorthetapro}
Let $\Ebh$ and $\Fbh$ be two pro-Hermitian vector bundles over $\Spec \OK,$ and let $\phi : \Ebh \lra \Fbh$ be a morphism in $\Hom_{\OK}^{\leq 1}(\Ebh, \Fbh).$ 

 If $\hat{\phi}: \hE \lra \hF$ is injective, then
\begin{equation}\label{ineqlhotpro}
\lhot(\Ebh) \leq \lhot(\Fbh).
\end{equation}

 If $\hat{\phi}: \hE \lra \hF$ is injective and a strict morphism of topological $\OK$-modules, then
\begin{equation}\label{inequhotpro}
\uhot(\Ebh) \leq \uhot(\Fbh).
\end{equation}

 If  $\phi_K$ has a dense image\footnote{namely, if $\phi_K(\hE_K)$ is dense in $\hF_K\simeq\lim_{V\in\cU(\hF)} (F/V)\otimes_{\OK}K$ equipped with its natural pro-discrete topology.},  then
\begin{equation}\label{ineqhutpro}
\hut(\Ebh) \geq \hut(\Fbh).
\end{equation} 
\end{proposition}

\begin{proof} The inequality (\ref{ineqlhotpro}) when the morphism $\hat{\phi}$ is a straightforward consequences of the definitions of $\lhot(\Ebh)$ and $\lhot(\Fbh)$. Indeed for any $(\Ebh', \iota)$ in $\cL(\Ebh)$, the pair $(\Ebh', \phi \circ \iota)$ belongs to $\cL(\Fbh)$.

To complete the proof of the proposition, observe that, for any $V \in \cU(\hF),$ its inverse image $U:= \phi^{-1} (V)$ is an element of $\cU(\hE)$, and there exists a unique morphism of $\OK$-modules
$\phi_V: E_U:=\hE/U \lra F_V :=\hF/V$
such that the following diagram is commutative:
$$\begin{CD}
\hE @>{\hat{\phi}}>> \hF \\
@V{p^{\hE}_U}VV      @VV{p^{\hF}_V}V \\
E_U @>{\phi_V}>> F_V.
\end{CD}
$$
Since $\phi$ belongs to $\Hom_{\OK}^{\leq 1}(\Ebh, \Ft),$ the morphism $\phi_V$ belongs to 
$\Hom_{\OK}^{\leq 1}(\Eb_U, \Fb_V).$ Moreover, by construction, it is injective. Consequently, according to Proposition \ref{ineqmortheta}, 1), we have:
\begin{equation}\label{ineqhotV}
\hot(\Eb_{\hat{\phi}^{-1}(V)}) \leq \hot (\Fb_V).
\end{equation}

To prove (\ref{inequhotpro}), observe that, when $\hat{\phi}: \hE \lra \hF$ is injective and strict,  $\phi^{-1} (V)$ converges to $0$ in $\hE$ if $V$ converges to $0$ in $\hF.$ Consequently:
\begin{equation}\label{liminfUV}
\liminf_{U \in \cU(\hE)} \hot(\Eb_U) \leq 
\liminf_{V \in \cU(\hF)} \hot(\Eb_{\hat{\phi}^{-1}(V)}).
\end{equation}
Then inequality (\ref{inequhotpro}) follows from (\ref{ineqhotV}) and (\ref{liminfUV}).

Besides, when $\phi_K$ has a dense image, the $K$-linear map
$$\phi_{V,K}: E_{\hat{\phi}^{-1}(V),K} \lra F_{V,K}$$
is surjective for every $V$ in $\cU(\Fb)$. Consequently, according to 
Proposition \ref{ineqmortheta}, 2), we have:
\begin{equation}\label{ineqhutV}
\hut(\Eb_{\hat{\phi}^{-1}(V)}) \geq \hut (\Fb_V).
\end{equation}
The inequality 
follows by taking the supremum over $V$ in $\cU(\hF)$, thanks to the expression 
(\ref{easyhutsup}) of $\hut$ as a supremum.
\end{proof}

\section{Examples} 

\subsection{Countable products of Hermitian line bundles over $\Spec \Z$.}\label{examplecountable}

Let $\bm{\lambda} := (\lambda_i)_{i \in \N}$ be an element of $\R_+^{\ast \N}.$ 

For any integer $n \in \N$, we may consider the $n$-tuple 
$$\bm{\lambda}_{<n} := (\lambda_i)_{0 \leq i < n}$$
in $\R_+^n$ and the associated Hermitian vector bundle  $\Vb_{\bm{\lambda}_{<n} }$ of rank $n$ over $\Spec \Z$, as defined in \ref{directsumZ1} above. The morphisms
$$q_n: \Vb_{\bm{\lambda}_{<n+1} } \lra \Vb_{\bm{\lambda}_{<n} }$$
defined by 
$$q_n(x_0,\cdots,x_{n-1}, x_n) := (x_0,\cdots,x_{n-1})$$
are surjective admissible, and we may consider the projective limit 
$$\Vbh_{\bm{\lambda}} := \varprojlim_n \Vb_{\bm{\lambda}_{<n} }$$
of the projective system:
$$
\Vb_{\bm{\lambda}_{<0}}\stackrel{q_0}{\longleftarrow}\Vb_{\bm{\lambda}_{<1}} \stackrel{q_1}{\longleftarrow}\dots \stackrel{q_{n-1}}{\longleftarrow}\Vb_{\bm{\lambda}_{<n}} \stackrel{q_n}{\longleftarrow} \Vb_{\bm{\lambda}_{<n+1}}\stackrel{q_{n+1}}{\longleftarrow} \cdots .$$

The underlying topological $\Z$-module $\hV_{\bm{\lambda}}$ of the pro-Hermitian vector bundle $\tilde{V}_{\bm{\lambda}}$ may be identified with $\Z^{\N}$, equipped with the product topology of the discrete topology on each factor $\Z$.  The locally convex complex vector space $\hV_{\bm{\lambda},\C}$ may be identified with $\C^{\N}$, also equipped with is natural product topology, and the Hilbert space $V^{\hilb}_{\bm{\lambda},\C}$ with the subspace of $\C^{\N}$ consisting in the elements $(x_i)_{i \in \N}$ such that 
$$\sum_{i \in \N} \lambda_i \vert x_i \vert^2 < + \infty.$$
Moreover, for any  element $(x_i)_{i \in \N}$  of  $V^{\hilb}_{\bm{\lambda},\C}$, 
$$\Vert (x_i)_{i \in \N}\Vert^2_{\tilde{V}_{\bm{\lambda},\C}} = \sum_{i \in \N} \lambda_i \vert x_i \vert^2.$$

For any $n \in \N,$ we may also consider the injective isometric morphism
$$i_n: \Vb_{\bm{\lambda}_{<n}} \lra \Vbh_{\bm{\lambda}}$$
defined by:
$$i_n(x_0, \cdots, x_{n-1}) := (x_0, \cdots, x_{n-1}, 0, 0, \cdots).$$

\begin{proposition}\label{hVlambda}
With the above notation, the following equalities hold in $[0,+\infty]$:
\begin{equation}\label{hotlambda}
\lhot(\Vbh_{\bm{\lambda}}) =\uhot (\Vbh_{\bm{\lambda}}) = \sum_{i \in \N} \tau(\lambda_i)
\end{equation}
and 
\begin{equation}\label{hutlambda}
\hut (\Vbh_{\bm{\lambda}}) = \sum_{i \in \N} \tau(\lambda_i^{-1}).
\end{equation}
\end{proposition}

\begin{proof} As shown in Section \ref{directsumZ1}, we have:
$$\hot(\tilde{V}_{\bm{\lambda}_{<n}}) = \sum_{0 \leq i < n} \tau(\lambda_{i})  \mbox{ and } \hut(\Vb_{\bm{\lambda}_{<n}}) = \sum_{0 \leq i < n} \tau(\lambda_{i}^{-1}).$$

Besides, the existence of the isometric injections $i_{n}$ shows that, for every $n\in \N$,
$$\hot(\Vb_{\bm{\lambda}_{<n}}) \leq \lhot(\Vbh_{\bm{\lambda}}),$$
and, by the very definition of $\uhot(\tilde{V}_{\bm{\lambda}})$,
$$\uhot(\tilde{V}_{\bm{\lambda}}) \leq \liminf_{n \ra + \infty} \hot(\tilde{V}_{\bm{\lambda}_{<n}}).$$
Together with the inequality $\lhot (\tilde{V}_{\bm{\lambda}}) \leq \uhot (\tilde{V}_{\bm{\lambda}}),$
this implies (\ref{hotlambda}).

Formula (\ref{hutlambda}) follows  from (\ref{easyhut}) applied to the projective system $(\Vb_{\bm{\lambda}_{<\bullet}}).$ \end{proof}

Observe that, as a straightforward consequence of (\ref{hotlambda}) and (\ref{hutlambda}) and of the asymptotic behavior (\ref{lthinfinity}) of the function $\lth$, we obtain the finiteness conditions:
\begin{equation}\label{hotlambdafin}
\lhot (\Vbh_{\bm{\lambda}}) < + \infty \Longleftrightarrow
\uhot (\Vbh_{\bm{\lambda}}) < + \infty \Longleftrightarrow  \sum_{i \in \N} e^{-\pi \lambda_i} < + \infty
\end{equation}
and 
\begin{equation}\label{hutlambdafin}
\hut (\Vbh_{\bm{\lambda}}) < + \infty \Longleftrightarrow  \sum_{i \in \N} e^{-\pi / \lambda_i} < + \infty.
\end{equation}

The content of Proposition \ref{hVlambda} may be reformulated as follows:
\begin{corollary} For any countable family $(\Lb_i)_{i \in \N}$ of Hermitian line bundles over $\Spec \Z,$  the $\theta$-invariants of the pro-Hermitian vector 
$$\hat{\bigoplus_{i \in \N}} \Lb_i := \varprojlim_n \bigoplus_{0\leq i \leq n} \Lb_i,$$
satisfy:
$$\lhot(\hat{\bigoplus}_{i \in \N} \Lb_i) =\uhot(\hat{\bigoplus}_{i \in \N} \Lb_i) = \sum_{i \in \N} \tau(e^{- 2 \dega \Lb_i}) = \sum_{i \in \N} ( \dega^+ \Lb_i + \eta( \dega \Lb_i))
$$ and $$
\hut(\hat{\bigoplus}_{i \in \N} \Lb_i) = \sum_{i \in \N} \tau(e^{2 \dega \Lb_i}) = \sum_{i \in \N} ( \dega^- \Lb_i + \eta( \dega \Lb_i)).
$$

\end{corollary} \qed

\subsection{Pro-Euclidean lattices defined by formal series and holomorphic functions on disks}\label{thetaHardyar}

The computation of the $\theta$-invariants of countable products of one dimensional Euclidean lattices in the previous paragraph, although quite elementary, allows one to investigate the $\theta$-invariants attached to the arithmetic avatars of the Hardy and Bergman spaces introduced in  paragraph \ref{ArHB}. 

\begin{proposition}\label{hRdelta}
 For every $(R,\delta) \in \R^\ast_+ \times \R,$ the following equalities hold in $[0, +\infty]$:
\begin{equation}\label{thetaHRlambda}\lhot(\Hbh(R)\otimes \cOb(\delta)) = \uhot(\Hbh(R)\otimes \cOb(\delta)) =\sum_{n \in \N} \tau (R^{2n} e^{-2\delta}).
\end{equation}
Moreover
$$h(R,\delta) := \sum_{n \in \N} \tau (R^{2n} e^{-2\delta})$$
is finite if and only if $R> 1$,
and, for any $R>1,$ we have:
\begin{equation}\label{thetaHRlambdaasympt}
h(R,\delta) = \frac{1}{2 \log R}\,\, \delta^2 + O(\delta) \,\, \mbox{ when $\delta \lra +\infty$.}
\end{equation}
\end{proposition}

Similar results hold concerning the Bergman pro-Euclidean lattices $\Bbh(R)$, and are left as exercises for the reader.

\begin{proof} The isomorphism in $\CTC_\Z$
$$
\begin{array}{rcl}
 \Z^\N  & \lrasim  & \Z[[X]]  \\
(a_n)_{n \in \N}  & \longmapsto  & \sum_{n \in \N} a_n X^n.  
\end{array}
$$
``extends" to an isomorphism of pro-Hermitian vector bundles over $\Spec \Z$:
$$  V_{\bm{\lambda}_{H}(R,\delta)} \lrasim \Hbh(R)\otimes \cOb(\delta),$$
where
$$\bm{\lambda}_{H}(R,\delta) := (R^{2n}e^{-2\delta})_{n\in \N}.$$
The relations (\ref{thetaHRlambda}) therefore follows from Proposition \ref{hVlambda}, equation (\ref{hotlambda}).

As we discussed in Paragraph (\ref{directsumZ1}), for any $x \in \R^\ast_+,$ we may express $\tau(x)$ as the sum
$$\tau(x) = (-\frac{1}{2} \log x)^+ + \eta(-\frac{1}{2} \log x),$$
where $\eta$ denotes an even positive continuous function on $\R$ which satisfies
$$\eta(t) \leq 3 e^{- \pi e^{2\vert t \vert }} \mbox{ for any $t\in \R$,}$$
or equivalently,
$$\eta(-\frac{1}{2} \log x ) \leq 3 e^{- \pi \max(\vert x \vert, \vert x \vert^{-1})} \mbox{ for any $x\in \R^\ast_+.$} $$

Consequently we have:
$$h(R, \delta) = \sum_{n \in \N} (-n \log R + \delta)^+ + \sum_{n \in \N} \eta(-n \log R + \delta).$$

When $-\log R$ is positive (rest. vanishes), the first (rest. second)  sum is $+\infty$. Therefore $h(R,\delta) = +\infty$ when $R \in ]0,1].$

When $\log R$ is positive both sums are clearly finite. Actually, we have:
$$\sum_{n \in \N} \eta(-n \log R + \delta) \leq \sum_{n \in \Z} \eta(-n \log R + \delta) \leq 6 \sum_{n\in \N} e^{-\pi R^n} < +\infty.$$
Moreover, when $\delta$ goes to $+\infty,$
$$\sum_{n \in \N} (-n \log R + \delta)^+ = \sum_{\stackrel{n \in \N}{n \leq \delta/\log R}} (-n \log R + \delta) =\frac{1}{2 \log R}\,\, \delta^2 + O(\delta).  $$
This completes the proof of (\ref{thetaHRlambdaasympt}). \end{proof}

\subsection{Pro-Hermitian vector bundles of infinite rank with vanishing $\uhot$.}\label{Provanish}

For any positive integer $a$, we may consider the constant $a$-tuple $a^{\times a} := (a,\ldots,a)$
and, with the notation of Section \ref{directsumZ1}, the Hermitian vector bundle
$$\Eb_a := \Vb_{a^{\times a}},$$
namely the Hermitian vector bundle over $\Spec \Z$ defined by the lattice $\Z^a$ inside $\R^a$ equipped with the Euclidean norm
$$\Vert.\Vert_{\Eb_A}:=\Vert.\Vert_{a^{\times a}}: (x_i)_{1\leq i \leq a} \longmapsto (a \sum_{1\leq i \leq a} x_i^2)^{1/2}.$$

We leave the proof of the next lemma as an exercise for the reader.

\begin{lemma} For any two positive integers $a$ and $b$ such that $a \mid b,$ one defines an admissible surjective morphism of Hermitian vector bundles over $\Spec \Z$,
$$p_{ab}: \Eb_b \longmapsto \Eb_a,$$
by letting $$p_{ab}(x_i)_{1\leq i \leq b} = (y_j)_{1\leq j \leq a},
\mbox{ where } 
y_j := \sum_{i=1+(j-1)d}^{jd} x_i.$$

Moreover, for any three positive integers $a,$ $b,$ and $c$ such that $a \mid  b$ and $b \mid c,$ the following transitivity relation holds:
$$p_{ab} \circ p_{bc} = p_{ac}.$$ \qed

\end{lemma}

Observe also that
$$\hot(\Eb_a) = a \tau(a) \sim 2 a e^{-\pi a} \mbox{ when $a$ goes to $+\infty.$}$$ 

Consequently, for any infinite set $\cA$ of positive integers in which the divisibility relation $\mid$ is filtering, we may consider the pro-Hermitian vector bundle (of infinite rank) over $\Spec \Z$:
$$\tilde{E}_{\cA} :=\varprojlim_{(\cA, \mid)} \Eb_a.$$
Since 
$\lim_{(\cA, \mid)} \hot(\Eb_a) = 0,$
it satisfies
$\uhot(\tilde{E}_{\cA}) = 0.$

This construction may be immediately generalized to produce pro-Hermitian vector bundles of infinite rank with vanishing $\uhot$ over $\Spec \OK$ for any number field $K$.
 
\subsection{A pro-Euclidean lattice with vanishing $\lhot$ and positive $\uhot$.}\label{WildEl}

We finally describe a construction of ind-Euclidean lattices, that are  special instances of Banasczyk's ``exotic groups" (\cite{Banaszczyk91}, Section 5), whose dual pro-Euclidean lattices $\Ebh$ will satisfy $$\lhot(\Ebh) = 0 \quad\mbox{and}\quad \uhot(\Ebh) > 0.$$

Let $(F_\R^\hilb, \Vert.\Vert)$ be some infinite dimensional separable real Hilbert space, and let $(e_n)_{n \in \N}$ be some orthonormal basis of $F_\R^\hilb$. We may construct a sequence $(f_n)_{n \in \N}$ of elements of $\bigoplus_{n \in \N} \Q e_n$ which is dense in $F_\R^\hilb$. Actually, by ``inserting zeroes" in this sequence, we may assume that it satisfies the following condition:
\begin{equation}\label{efn}
\mbox{for any $i \in \N$},\quad  f_i \in \bigoplus_{0 \leq i < n} \R e_i.
\end{equation}

\begin{proposition}[\cite{Banaszczyk91}, Theorem 5.1] \label{BanExo} With the previous notation,
$$F := \bigoplus_{n \in \N} \Z (e_n+ f_n)$$ is a discrete subgroup of $F_\R^\hilb$ equipped with the norm topology. Its $\R$-span $F_\R$ coincides with $\bigoplus_{n \in \N} \R e_n$ and is therefore dense in $F_\R^\hilb$.

Moreover, any continuous linear form $\xi: F_\R^\hilb \lra \R$ such that $\xi(F) \subset \Z$ is zero.

\end{proposition}

\begin{proof}
  Let $f$
   be a non-zero element of $F$. We may write it 
   $$f = \sum_{0\leq i \leq n} k_i (e_i + f_i)$$
   for some $n \in \N$, $(k_0,\ldots,k_n) \in \N^{n+1}$, and $k_n \neq 0.$ Then, according to (\ref{efn}), we have:
   $\langle e_n, f \rangle = k_n$
   and therefore
$\Vert f \Vert \geq \vert k_n \vert \geq 1.$ This establishes that $F$ is discrete subgroup of  $F^\hilb_\R$.
The equality $F_\R = \bigoplus_{n \in \N} \R e_n$ is a straightforward consequence of (\ref{efn}). 

Let us consider $\xi \in F_\R^{\hilb \vee}$ such that $\xi(F) \subset \Z.$ For any $f$ in $F^\hilb_\R$, we may find a subsequence $(f_{n(i)})_{i \in \N}$ of $(f_n)_{n \in \N}$ such that $\lim_{i \ra + \infty} f_{n(i)} = f.$ Then we have:
$$\xi(f) = \lim_{i \ra + \infty} \xi(f_{n(i)}) = \lim_{i \ra + \infty} \xi( e_{n(i)} + f_{n(i)}).$$
Therefore $\xi(f)$, like the $ \xi( e_{n(i)} + f_{n(i)})$'s, is an integer. This shows that $\xi(F^\hilb_\R) \subset \Z$ and implies that $\xi$ is zero.
\end{proof}

The first part of Proposition \ref{BanExo} shows that $\Fb := (F, \Vert. \Vert)$ defines some ind-Euclidean lattice $\Fb$, such that $(F_\R^\hilb, \Vert.\Vert)$ coincides with the Hilbert space completion of $(F_\R, \Vert. \Vert)$. Let $\Ebh := (\Eh, E_\R^\hilb)$ be the  pro-Euclidean lattice dual to $\Fb$. In other words,
$$\Eh := \Hom_\Z(F, \Z) \quad \mbox{ and } \quad E_\R^\hilb := \Hom^{\rm cont}_\R((F_\R, \Vert .\Vert), \R) \simeq F_\R^{\hilb \vee}.$$

The second part of Proposition (\ref{BanExo}) may be rephrased as:
$$\Eh \cap E^\hilb_\R = \{0\},$$
or equivalently as:
$$\lhot(\Ebh) =0.$$

We claim that
\begin{equation}\label{posuhot}
\uhot(\Ebh) \geq \eta := \hot(\Z, \vert.\vert).
\end{equation}
(\cf Section  \ref{taueta} \emph{supra} for more about the positive real number $\eta$.)

This will follow from the next Proposition, where, for any $n \in \N$, we let: 
$$\xi_n := \langle e_n, .\rangle \in E_\R^\hilb.$$

\begin{proposition}\label{gammazerovanishes}  For any non-zero finite dimensional $\Q$-vector subspace $V$ of $$F_\Q = \bigoplus_{n \in \N} \Q (e_n+ f_n),$$ 
the set
\begin{equation}\label{SetV}
\{ n\in \N \mid \xi_{n \mid V} \neq 0 \}
\end{equation}
is finite and non-empty. If $n(V)$ denotes its largest element,  then 
\begin{equation}\label{xinV}
\xi_{n(V)}(V \cap F) \in \Z.
\end{equation}
\end{proposition}

Indeed, for $U \in \cU(\Eh),$ the Euclidean lattice $\Eb_U$ may be identified with the dual of the Euclidean lattice $(V\cap F, \Vert.\Vert)$  for some suitable finite dimensional $\Q$-vector subspace of $F_\Q$. (With the notation of Corollary   \ref{FSU}, $V\cap F = U^\perp.$) With this notation, if $U \neq \Eh,$ then $V$ is non-zero, and $\xi_{n(V)\mid V \cap F}$ defines an element of $E_U$ of norm in $]0,1]$, or equivalently an injective morphism from $(\Z, \vert.\vert)$ to $\Eb_U$ of norm $\leq 1$. This implies that
$$\hot(\Z, \vert.\vert) \leq \hot(\Eb_U)$$ and establishes the lower bound (\ref{posuhot}).

\begin{proof}[Proof of Proposition \ref{gammazerovanishes}] The non-emptiness of  (\ref{SetV}) is clear. Its finiteness follows from the fact that, according to (\ref{efn}), $\xi_n(e_i + f_i) = 0$ if $n>i$.

Actually, (\ref{efn}) shows that, for every $n \in \N,$
$$W_n := \bigoplus_{o \leq i \leq n} \R e_i = \bigoplus_{o \leq i \leq n} \R(e_i+ f_i)$$
and that, over 
$$W_n\cap F = \bigoplus_{o \leq i \leq n} \Z (e_i+ f_i),$$
the linear form $\xi_n$ takes its values in $\Z$. Since
$$V \cap F \subset \bigcap_{i >n(V)} \ker \xi_i \cap F = W_{n(V)} \cap F,$$
this implies (\ref{xinV}).
\end{proof}

\medskip

\chapter[Summable projective systems of Hermitian vector bundles]{Summable projective systems of Hermitian vector bundles and finiteness  of $\theta$-invariants}\label{SectionSum}

\medskip

We still denote by $K$ a number field, by $\OK$ its ring of integers and by $\pi$ the morphism of schemes form $\Spec \OK$ to $\Spec \Z.$

\section{Main theorem}\label{SectionMainTheo} This chapter is mainly devoted to a proof of the following theorem, which shall play a key role in the sequel of this monograph for constructing pro-Hermitian vector bundles with finite and well-behaved invariants $\lhot(\Ebh)$ and $\uhot(\Ebh)$.

\begin{theorem}\label{maintheorem}Let 
$$\Eb_{\bullet} : 
\Eb_0 \stackrel{q_0}{\longleftarrow}\Eb_1 \stackrel{q_1}{\longleftarrow}\dots \stackrel{q_{i-1}}{\longleftarrow}\Eb_i \stackrel{q_i}{\longleftarrow} \Eb_{i+1} \stackrel{q_{i+1}}{\longleftarrow} \dots$$
be a projective system of surjective admissible morphisms of Hermitian vector bundles over $\Spec \OK$, and consider 
the associated pro-Hermitian vector bundle over $\Spec \OK$:
$$\Ebh := \varprojlim_i \Eb_i.$$

For every $i \in \N,$ let us denote by $\overline{\ker q_i}$ the Hermitian vector bundle over $\Spec \OK$ defined as the kernel of $q_i$ equipped with the Hermitian structure induced by the one of $\Eb_{i+1}.$

If there exists $\epsilon \in \R^\ast_+$ such that
\begin{equation}\label{sumepsiloninit}
\sum_{i\in \N} \hot(\overline{\ker q_i}\otimes_{\OK} \cOb_{\Spec \OK}(\epsilon)) < + \infty,
\end{equation}
then the limit $\lim_{i \rightarrow + \infty} \hot(\Eb_i)$ exists in $\R_+$ and
\begin{equation}\label{mainequation}
\lhot(\Ebh) = \uhot(\Ebh) = \uhot(\pi_\ast \Ebh) = \lim_{i \ra + \infty} \hot(\Eb_i).
\end{equation}

Moreover, for any $k \in \N$, the non-negative real number (\ref{mainequation}) admits the following upper bound:
\begin{equation}\label{htupperbound}
\lim_{i \rightarrow + \infty} \hot(\Eb_i) \leq \hot(\Eb_k) + \sum_{j=k}^{+\infty} \hot({\overline{\ker q_i}}).
\end{equation}

\end{theorem}

 When $\Spec \OK = \Spec \Z,$ this theorem will be established in a more precise form as Proposition \ref{propkeylim} and Theorem \ref{propSum}, 2), below. The simple derivation of Theorem \ref{maintheorem} for an arbitrary number field $K$ from its special case where $\Spec \OK = \Spec \Z$ is presented in Section \ref{completemain}.
 
The proof of Proposition \ref{propkeylim} and of Theorem \ref{propSum} are presented in the next sections (\ref{prel} to \ref{ProofII}) of this chapter. 

Section  \ref{sub:conv}  will  be devoted to establishing that Theorem \ref{maintheorem} is basically optimal : we  will show that any pro-Hermitian vector bundle $\Ebh$ such that $\lhot(\Ebh) = \uhot(\Ebh) < + \infty$ may be realized as the projective limit $\varprojlim_i \Eb_i$ of a projective system of Hermitian vector bundles $\Eb_{\bullet}$ satisfying a summability condition similar to (\ref{sumepsiloninit}). This additional result will finally allow us to give sensible definitions of \emph{strongly summable} and of \emph{$\theta$-finite} pro-Hermitian vector bundles in  Section \ref{sub:cons}. 

The proofs in this part will use some basic facts concerning measure theory on Polish spaces constructed as projective limits of countable systems of countable discrete sets. These facts are presented in a form suited to our need in Appendix \ref{measurespro}.

 \section{Preliminaries}\label{prel}
 We begin by introducing the notation and the definitions that will be used in the formulation and in the proofs of Proposition \ref{propkeylim} and Theorem \ref{propSum} in Sections \ref{sec:summableprojective}--\ref{ProofII}.
  
\subsection{Notation}
Let 
$$\Eb_{\bullet} : 
\Eb_0 \stackrel{q_0}{\longleftarrow}\Eb_1 \stackrel{q_1}{\longleftarrow}\dots \stackrel{q_{i-1}}{\longleftarrow}\Eb_i \stackrel{q_i}{\longleftarrow} \Eb_{i+1} \stackrel{q_{i+1}}{\longleftarrow} \dots$$
be a projective system of surjective admissible morphisms of Hermitian vector bundles over $\Spec \Z$, and let
$$\tilde{E} := \varprojlim_i \Eb_i = (\hE, E^\hilb_\R, \Vert.\Vert)$$ be the associated pro-Hermitian vector bundle over $\Spec \Z$.

For every $(i,j) \in \N^2$ such that $i \leq j,$ we let:
$$p_{ij}:= q_i \circ q_{i+1} \circ \cdots \circ q_{j-1}: E_j \lra E_i,$$
and we denote by
$$p_i: \hE := \varprojlim_{k} E_k \lra E_j$$
the $i$-th projection morphism, and by $$U_i := \ker p_i$$ its kernel.  

In this section, we shall apply the basic results concerning measures on countable sets and on their projective limits recalled in Appendix \ref{measurespro} to the projective system of countable sets 
$$E_{\bullet} : 
E_0 \stackrel{q_0}{\longleftarrow}E_1 \stackrel{q_1}{\longleftarrow}\dots \stackrel{q_{i-1}}{\longleftarrow}E_i \stackrel{q_i}{\longleftarrow} E_{i+1} \stackrel{q_{i+1}}{\longleftarrow} \dots$$
underlying $\Eb_{\bullet}$. A key point will be that, to the Hermitian vector bundle $\Eb_i$, is naturally attached a bounded positive measure of total mass $\exp \hot(\Eb_i)$ on the underlying set $E_i.$

Indeed, to any (finite dimensional) Hermitian vector bundle $\Vb:= (V, \Vert. \Vert)$ over $\Spec \Z,$ we may attach the positive measure on the countable set $V$:
$$\gamma_{\Vb} := e^{-\pi \Vert . \Vert_\Vb^2}  \sum_{v \in V} \delta_v = \sum_{v \in V} e^{-\pi \Vert v \Vert_{\Vb}^2}\, \delta_v.$$ 
This measure, which  occurs implicitly in the arguments \emph{à la} Banaszczyk of paragraph \ref{Banasc1} and section \ref{Banasc1}, has a finite total mass:
$$\gamma_{\Vb}(V) = \sum_{v \in V} e^{-\pi \Vert v \Vert_\Vb^2} = \exp(\hot(\Vb)).$$

For every $i\in \N,$ we may consider the continuous function $\Vert p_i\Vert_{\Eb_i}$ on $\hE_\R$. We thus define a non-decreasing sequence of functions on $\hE_\R$, which converges (pointwise) towards the function $\Vert. \Vert$ on $\hE_\R$ defined as the Hilbert norm on $E^\hilb_\R$ and as $+\infty$ on $\hE_\R \setminus  E^\hilb_\R$. In particular, 
$$\Vert . \Vert : \hE_\R \lra [0, +\infty]$$
is lower semi-continuous.

For any $\eta \in \R_+^\ast,$ we may also consider the functions
$$f_{i,\eta} := \exp (-\eta \pi \Vert p_i\Vert_{\Eb_i}^2)$$
and
$$f_{\eta} := \exp (-\eta \pi \Vert . \Vert^2)$$ 
on $\hE_\R$. Then $(f_{i,\eta})_{i \in \N}$ is a non-increasing sequence of continuous functions from $\hE_\R$ to $[0,1]$ and converges pointwise to $f_\eta$, which is therefore upper semi-continuous.

Observe that, for any $x \in \hE_\R$,
$$f_\eta(x)= 
\begin{cases}
 e^{-\eta \pi  \Vert x \Vert^2} & \mbox{if $x\in E^\hilb_\R$} \\
 0 & \mbox{if $x \in \hE_\R \setminus  E^\hilb_\R$}
\end{cases}
$$
and that consequently:
\begin{equation}\label{limfeta}
\lim_{\eta \ra 0_+} f_\eta(x) = 1_{E^\hilb_\R}(x).
\end{equation}

\subsection{Definitions}\label{sub:def}

For every $i \in \N,$ we  consider the Hermitian vector bundle $\overline{\ker q_i}$,  defined as the kernel of $q_i$ equipped with the Hermitian structure induced by the one of $\Eb_{i+1}.$ By construction, its fits into an admissible exact sequence
\begin{equation}\label{adkerq}
\overline{\cS}_i : 
0 \lra \overline{\ker q_i} \hlra \Eb_{i+1} \stackrel{q_i}{\lra} \Eb_i \lra 0.
\end{equation}
We shall say that the projective system $\Eb_{\bullet}$ is \emph{summable} when it satisfies the following condition:
$$\Sum (\Eb_{\bullet}) \, : \,  \sum_{i\in \N} \hot(\overline{\ker q_i}) < + \infty.$$
For any $\lambda \in \R,$ we may ``twist" $\Eb_{\bullet}$ by the Hermitian line bundle $\cOb(\lambda)$ and consider the attached summability condition:
$$\Sum (\Eb_{\bullet}\otimes \cOb(\lambda)) \, : \,  \sum_{i\in \N} \hot(\overline{\ker q_i}\otimes \cOb(\lambda)) < + \infty.$$

Observe that, when this condition is satisfied for some value $\lambda_0$ of $\lambda$, then it is also satisfied for any $\lambda$ in $]-\infty, \lambda_0].$ We shall say that $\Eb_{\bullet}$ is \emph{strongly summable} when $\Sum (\Eb_{\bullet}\otimes \cOb(\epsilon))$ is satisfied for some $\epsilon \in \R^\ast_+.$

For any $i\in \N$ and any $t\in \R^\ast_+,$ we may consider
$$\log \sum_{v \in E_i} e^{-\pi t \Vert v \Vert^2_{\Eb_i}} = \hot(\Eb_i \otimes \cOb (-(\log t)/2))).$$
As a function of $t$, it is convex on $\R^\ast_+.$ We shall say that $M_{\Eb_{\bullet}}(t)$ \emph{is defined} when the limit
$$M_{\Eb_{\bullet}}(t) = \lim_{i \ra +\infty} \log \sum_{v \in E_i} e^{-\pi t \Vert v \Vert^2_{\Eb_i}} = \lim_{i \ra +\infty} \hot(\Eb_i \otimes \cOb (-(\log t)/2)))$$
exists in $\R_+$. 

\subsection{Convexity of $M_{\Eb_{\bullet}}$} The following lemma, which is a straightforward consequence of the previous observations, turns out to play a key role in our  derivation of Theorem   \ref{maintheorem} (through the proof of Corollary \ref{MEcont} \emph{infra}).

\begin{lemma}\label{MEconv} 
 If $M_{\Eb_{\bullet}}(t)$ is defined for every $t$ in some compact interval $[a,b]$ in $\R_+^\ast,$ then the function
$$M_{\Eb_{\bullet}}: [a,b] \lra \R_+$$
is convex and non-increasing. In particular it is continuous on $]a,b]$. \qed
\end{lemma}

\section[Summable projective systems and  associated measures]{Summable projective systems of Hermitian vector bundles and  associated measures}\label{sec:summableprojective}

\subsection{Existence of the limit $\lim_{i \rightarrow + \infty} \hot(\Eb_i)$}

\begin{proposition}\label{propkeylim} If the projective system $\Eb_{\bullet}$ is summable, then the limit
\begin{equation}\label{keylim}
\lim_{i \rightarrow + \infty} \hot(\Eb_i)
\end{equation}
exists in $[0, + \infty[.$ 

Moreover, for any $k \in \N$,
$$\lim_{i \rightarrow + \infty} \hot(\Eb_i) \leq \hot(\Eb_k) + \sum_{j=k}^{+\infty} \hot({\overline{\ker q_i}}).$$
 \end{proposition}

\begin{proof} 
The subadditivity of $\hot$ (Proposition \ref{ineqshorttheta}) applied to the admissible short exact sequence $\overline{\cS}_i$ (see (\ref{adkerq})) shows that, for any $i \in \N$,
\begin{equation}
\hot(\Eb_{i+1}) \leq \hot(\Eb_i) + \hot(\overline{\ker q_i}).
\end{equation}
The sequence
$$(\hot(\Eb_i) - \sum_{0 \leq j < i} \hot(\overline{\ker q_j}))_{i \in \N}$$
is therefore non-increasing and bounded below by $-\Sigma,$
where
$\Sigma := \sum_{0 \leq j < +\infty} \hot(\overline{\ker q_j}),$
and consequently admits a limit $l$ in $[-\Sigma, +\infty[.$ This proves that $(\hot(\Eb_i))_{i\in \N}$ converges to  $l+\Sigma \in \R^+.$

Moreover, for any $k \in \N,$
$$l+\Sigma \leq (\hot(\Eb_k) - \sum_{0 \leq j < k} \hot(\overline{\ker q_j})) + \Sigma = \hot(\Eb_k) + \sum_{k \leq j < + \infty} \hot(\overline{\ker q_j}).$$ 
\end{proof}
 
\begin{corollary} If the projective system $\Eb_{\bullet}$ is summable (resp. strongly summable), then $M_{\Eb_{\bullet}}(t)$ is defined for any $t$ in $[1, +\infty[$ (resp. for any $t$ in some open interval containing $[1, +\infty[$). \qed
 \end{corollary}

\begin{proof} Observe that the validity of $\Sum (\Eb_{\bullet})$ implies the one of $\Sum (\Eb_{\bullet}\otimes \cOb(\lambda))$ for every $\lambda$ in $\R_-$ and that, for any $(t,\lambda) \in \R_+^\ast \times \R,$ $M_{\Eb_{\bullet}\otimes \cOb(\lambda)}(t)$ and $M_{\Eb_{\bullet}}(e^{-2 \lambda}t)$ are defined simultaneously
and are then equal. 

Therefore to establish the Corollary, it is enough to show that $M_{\Eb_{\bullet}}(1)$ is defined when $\Eb_{\bullet}$ is summable: this is precisely the existence of the limit (\ref{keylim}) established in Proposition \ref{propkeylim}.
\end{proof}

 Taking Lemma \ref{MEconv} into account, we finally obtain:
 \begin{corollary}\label{MEcont} When $\Eb_{\bullet}$ is strongly summable, the convex function $M_{\Eb_{\bullet}}$ on $[1, +\infty[$ is continuous at $1$. \qed
\end{corollary}

\subsection{The main technical result}
We may now formulate the main technical result of this chapter. Its proof will be the object of the next two sections, and will notably rely on the  results concerning measures on projective limits on countable sets  established in Appendix \ref{measurespro}.

\begin{theorem}\label{propSum} 1)  If the projective system $\Eb_{\bullet}$ is summable, then the pro-Hermitian vector bundle $\Ebh:= \varprojlim_i \Eb_i$ satisfies:
\begin{equation}\label{hotlim}
\uhot(\Ebh) = \lim_{i \rightarrow + \infty} \hot(\Eb_i).
\end{equation}

Moreover, for any $i\in \N,$ the sequence  
$(p_{ij\ast} \gamma_{\Eb_j})_{j \in \N_{\geq i}}$ converges to some limit $\mu_i$ in $\cM^b_+(E_i)$, and there exists a unique 
 measure $\mu_{\Eb_{\bullet}}$ in $\cM^b_+(\hE)$ such that, for any $i \in \N,$
$$\mu_i = p_{i\ast} \mu_{\Eb_{\bullet}}.$$
The mass of the measure $\mu_{\Eb_{\bullet}}$ is $\geq 1$ and satisfies: 
\begin{equation}\label{uhotmutot}
 \uhot( \Ebh) = \log \mu_{\Eb_{\bullet}} (\hE).
\end{equation}
Moreover, for any $v \in \hE \cap E^\hilb_\R,$ we have:
\begin{equation}\label{muv}
\mu_{\Eb_{\bullet}}(\{v \}) = e^{-\pi \Vert v \Vert^2}.
\end{equation}

2) If the projective system $\Eb_{\bullet}$ is summable and if the convex function $M_{\Eb_{\bullet}}: [1,+\infty[ \ra \R_+$ is continuous at 1 --- \emph{for instance, if  $\Eb_{\bullet}$ is strongly summable} ---
then:
\begin{equation}\label{muE}
\mu_{\Eb_{\bullet}}  
= \sum_{v \in \hE \cap E^\hilb_\R} e^{-\pi \Vert v \Vert^2}\, \delta_v
\end{equation}
and
\begin{equation}\label{hthetamuE}
\uhot(\Ebh) = \lhot(\Ebh)=\log \sum_{v \in \hE \cap E^\hilb_\R} e^{-\pi \Vert v \Vert^2}. 
\end{equation}
\end{theorem}

Observe that, for any pro-Hermitian vector bundle $\Ebh$ such that $\lhot(\Ebh) < + \infty,$ the right-hand side of (\ref{muE}) defines a measure in $\cM^b_+(\hE)$, that we shall denote  by $\mu_{\Ebh}.$ It satisfies
\begin{equation*}
 \lhot( \hE) = \log \mu_{\Ebh} (\hE).
\end{equation*}
and the conclusion of Theorem \ref{propSum}, 2), may be rephrased as the equality:
$$\mu_{\Eb_{\bullet}} = \mu_{\Ebh}.$$

\subsection{Completion of the proof of Theorem \ref{maintheorem}}\label{completemain} Before we proceed to the proof of Theorem \ref{propSum}, we explain how, taking Proposition \ref{propkeylim} and Theorem \ref{propSum}, 2), for granted, one easily establishes   Theorem \ref{maintheorem}.

Let us use the notation of Theorem \ref{maintheorem}.
From Theorem \ref{propSum}, 2), applied to 
  the direct images $\pi_\ast \Eb_{\bullet}$ and $\pi_{\ast}\Ebh \simeq \varprojlim_i \pi_\ast\Eb_i$ of the projective system $\Eb_{\bullet}$ and of its limit $\Ebh$ by the morphism $\pi: \Spec \OK \ra \Spec \Z,$ we obtain the existence of the limit $\lim_{i \rightarrow + \infty} \hot(\Eb_i)$ and the relations:
  \begin{equation}\label{generalmain}
\lhot(\Ebh) =  \uhot(\pi_\ast \Ebh) = \lim_{i \ra + \infty} \hot(\Ebh_i).
\end{equation}
 (Indeed, according to (\ref{hithetagen}) and (\ref{pistarlhot}), we have:
  $$\lhot(\pi_\ast\Eb_i)=\lhot(\Eb_i), \,\,\, \hot(\pi_\ast\overline{\ker q_i}\otimes_\Z \cOb(\epsilon)) = \hot(\overline{\ker q_i}\otimes_{\OK} \cOb_{\Spec \OK}(\epsilon)),$$
 and $$\lhot(\pi_\ast\Ebh)=\lhot(\Ebh).)$$
Moreover, the inequality (\ref{pistaruhot}) and the very definition of $\uhot(\Ebh)$ show that:
 $$\uhot(\pi_\ast\Ebh) \leq 
 \uhot(\Ebh) \leq \lim_{i \rightarrow + \infty} \hot(\Eb_i).$$
 
 Combined with (\ref{generalmain}), this completes the proof of (\ref{mainequation}).  

\section[Proof of Theorem 7.3.4 -- I. The equality $\uhot({\hat{\overline{E}}}) = \lim_{i \rightarrow + \infty} \hot(\Eb_i)$]{Proof of Theorem \ref{propSum} -- I. The equality $\uhot({\hat{\overline{E}}}) = \lim_{i \rightarrow + \infty} \hot(\Eb_i)$}\label{ProofI}

In this paragraph, we consider a \emph{summable} projective  system $(\Eb_{\bullet})$ of surjective admissible morphisms of Hermitian vector bundles over $\Spec \Z.$ According to Proposition \ref{propkeylim}, the limit $$l:=\lim_{i \rightarrow + \infty} \hot(\Eb_i)$$ exists in $\R_+$.
Clearly, $$\liminf_{U \in \cU(\hE)} \hot(\Eb_U) \leq \lim_{i \lra + \infty} \hot(\Eb_i).$$ Therefore, to prove (\ref{hotlim}), it is enough to show that
\begin{equation}\label{geql}
\liminf_{U \in \cU(\hE)} \hot(\Eb_U) \geq l.
\end{equation}

To achieve this, observe that, for any open saturated submodule $U$ of $\hE,$ we may perform the following construction.

If  $i_0$ is a large enough positive integer (depending on $U$), then for any $i\geq i_0,$ the submodule $U$ contains $U_i$, and we may consider the quotient map:
$$p_{U,U_i}: E_i \simeq \hE/ U_i \lra E_U:= \hE/U.$$
Its kernel  defines a Hermitian vector bundle $\Kb^U_i$ which fits into an admissible short exact sequence 
\begin{equation}\label{KUi}
0 \lra \Kb^U_i \lra \Eb_i \lra \Eb_U \lra 0.
\end{equation}
Moreover, the morphism $q_i:=p_{U_i,U_{i+1}}: E_{i+1} \lra E_{i}$ defines by restriction an admissible surjective morphism
$$q^U_i:\Kb^{U}_{i+1} \lra \Kb^{U}_{i},$$
of kernel $$\ker q^U_i= \ker q_i.$$ 
Moreover the natural euclidian structures  on $\ker q^U_i$ and $\ker q_i$ (defined by the Euclidean structures
on $\Kb^U_{i+1}$ and on $\Eb_{i+1}$) clearly coincide: 
$$\overline{\ker q^U_i}= \overline{\ker q_i} \mbox{ for every $i\geq i_0$}.$$

Thus we may consider the summable projective system
$$\Kb^U_{\bullet} : 
 \Kb^{U}_{i_0}\stackrel{q_{i_0}}{\longleftarrow}\Kb^{U}_{i_0+1} \stackrel{q_{i_0+1}}{\longleftarrow}\dots
\stackrel{q_{i-1}}{\longleftarrow}\Kb^{U}_i \stackrel{q_i}{\longleftarrow} \Kb^{U}_{i+1} \stackrel{q_{i+1}}{\longleftarrow} \cdots \,\, .$$
(The pro-Hermitian vector bundle $\tilde{K}^U:= \varprojlim_i \Kb^U_i
= (\widehat{K}^U, K^{U,\hilb}_\R, \Vert.\Vert_{\tilde{K}_U})$ may also be directly defined by
$$ \widehat{K}^U = U, \,\, K^{U,\hilb_\R}:= \ker p_{U,\R\mid E^\hilb_\R}, \mbox{ and } \Vert.\Vert_{\tilde{K}_U} := \Vert . \Vert_{\mid K^{U,\hilb_\R}},$$
where we denote by  $p_U: \hE \lra E_U$ the quotient map from $\hE$ onto $\hE/U$ and  by $p_{U,\R}:\hE_\R \ra E_{U,\R}$ its $\R$-linear continuous extension.)

According to Proposition \ref{propkeylim}, the limit
$$l(U) := \lim_{i \ra + \infty} \hot(\Kb^{U}_{i})$$
exists in $\R_+$ and satisfies 
\begin{equation}\label{lui0}
l(U) \leq \hot(\Kb^{U}_{i_0}) + \sum_{i \geq i_0} \hot(\overline{\ker q_i}).
\end{equation}

\begin{lemma}\label{LemmaU} 1) For any open saturated submodule $U$ of $\hE$, we have:
\begin{equation}\label{KL1}
l \leq \hot(\Eb_U) + l(U).
\end{equation}
2) For any two open saturated submodules $U$ and $U'$ of $\hE$, 
\begin{equation}\label{KL2}
U' \subset U \Longrightarrow l(U') \leq l(U).
\end{equation}
3) For any $k \in \N,$
\begin{equation}\label{KL3}
l(U_k) \leq \sum_{i \geq k} \hot(\Sb_i).
\end{equation}
4) We have : $$\lim_{U \in \cU (\hE)} l(U) =0.$$
 \end{lemma}

\begin{proof} 1) The subadditivity of $\hot$ applied to the admissible short exact sequence (\ref{KUi}) establishes the inequality
$$\hot(\Eb_i) \leq \hot(\Kb^{U}_{i}) + \hot(\Eb_U).$$
This yields (\ref{KL1}) by letting $i$ go to infinity.

2) For any large enough integer $i$, we have 
$K^{U'}_i \subset K^U_i,$
and consequently
$$\hot(\Kb^{U'}_i)\leq \hot(\Kb^{U}_i).$$
This yields (\ref{KL2}) by letting $i$ go to infinity.

3) In the above construction, when $U=U_k$, we may choose $i_0=k$. Then $\Kb^U_{i_0}=0$, and (\ref{KL3}) is nothing but (\ref{lui0}).

4) follows from 2) and 3).
 \end{proof}
 
The estimate (\ref{geql}) directly follows from assertions 1) and 4) in Lemma \ref{LemmaU}.

\section[Proof of Theorem 7.3.4 -- II. Convergence of measures]{Proof of Theorem \ref{propSum} -- II. Convergence of measures}\label{ProofII}

As in the previous paragraph, we consider a summable projective  system $\Eb_{\bullet}$ of Hermitian vector bundles over $\Spec \Z.$ 

\begin{lemma}\label{submeasure} For any $i \in \N,$ the measures $\gamma_{\Eb_i}$ and $q_{i\ast} \gamma_{\Eb_{i+1}}$ on $E_{i}$ satisfy:
$$q_{i\ast} \gamma_{\Eb_{i+1}} \leq e^{\hot(\overline{\ker q_i})} \gamma_{\Eb_i}.$$
\end{lemma}

\begin{proof} This is the content of Lemma \ref{gaussiansumpreimage}, applied to the admissible short exact sequence  $\overline{\cS}_i$ (defined in (\ref{adkerq})). \end{proof}

The existence of the limit measures $\mu_i$ on $E_i$ and of the existence and the unicity of the measure $\mu_{\Eb_{\bullet}}$
now follows from Proposition \ref{limitmuDh} of the Appendix, applied to $D_i = E_i$, $\gamma_i = \gamma_{\Eb_i},$ and $\lambda_i =\hot(\Eb_i).$ Indeed, according to the summability assumption on $\Eb_{\bullet},$ the sequence $(\hot(\Eb_i))_{i\in \N}$ belongs to $l^1(\N).$

According to (\ref{hotlim}), the equality (\ref{uhotmutot}) may also be written 
$$\mu_{\Eb_{\bullet}} (\hE) =  \lim_{i \rightarrow + \infty} e^{\hot(\Eb_i)}.$$
As
$e^{\hot(\Eb_i)} = \gamma_{\Eb_i}(E_i),$
it is nothing but a reformulation of (\ref{mulim}).

The equality (\ref{muv}) will follow from the following:

\begin{lemma}\label{lemmuivi} Let $v$ be an element of $\hE \cap E^\hilb_\R,$ and, for any $i \in \N,$ let $v_i := p_i(v)$ be its image in $E_i$.
Then, for any $i \in \N,$ we have
\begin{equation}\label{muivi}
e^{-\pi \Vert v \Vert^2} \leq \mu_i(\{v_i\}) \leq e^{-\pi \Vert v_i \Vert_{\Eb_i}^2 + \sum_{k\geq i} \hot(\overline{\ker q_i})}.
\end{equation}
\end{lemma}

Indeed, with the notation of  Lemma \ref{lemmuivi}, the set $\{v\}$ may be described as the countable decreasing intersection
$$\{v \} = \bigcap_{i \in \N} p_i^{-1} (v_i),$$
and therefore:
$$\mu_{\Eb_{\bullet}}(\{v\}) = \lim_{i \ra + \infty} \mu_{\Eb_{\bullet}}(p_i^{-1} (v_i)) = \lim_{i \ra + \infty} \mu_i(\{v_i\}).$$

\begin{proof}[Proof of Lemma \ref{lemmuivi}] We have, by the very definition of $\mu_i$:
$$\mu_i(\{v_i\})= \lim_{j \ra +\infty} p_{ij\ast} \gamma_{\Eb_j} (\{v_i\}).$$

For any integer $j \geq i,$ the preimage $p_{ij}^{-1}(v_i)$ contains $v_j$. Therefore
\begin{equation}\label{muivi1}
p_{ij\ast} \gamma_{\Eb_j} (\{v_i\}) = 
 \gamma_{\Eb_j} (p_{ij}^{-1}(v_i)) \geq  \gamma_{\Eb_j}(v_j) = e^{- \pi \Vert v_j \Vert_{\Eb_j}^2}.
 \end{equation}
 
 Besides, the inequality $q_{i\ast} \gamma_{\Eb_{i+1}} \leq e^{\hot(\overline{\ker q_i})} \gamma_{\Eb_i}$ established in Lemma \ref{submeasure} implies that:
 $$p_{ij\ast} \gamma_{\Eb_j} \leq e^{\sum_{i \leq k < j} \hot(\Eb_k)} \gamma_{\Eb_i}.$$
 Therefore,
 \begin{equation}\label{muivi2}
p_{ij\ast} \gamma_{\Eb_j} (\{v_i\}) \leq e^{\sum_{i \leq k < j} \hot(\Eb_k)} \gamma_{\Eb_i} (\{v_i\}) = e^{-\pi \Vert v_i\Vert^2_{\Eb_i} +\sum_{i \leq k < j} \hot(\Eb_k)}.
\end{equation}

The estimates (\ref{muivi}) follow from (\ref{muivi1}) and (\ref{muivi2}) by taking the limit $j\ra +\infty.$
\end{proof}

This completes the proof of part 1) of Theorem \ref{propSum}. The proof of part 2) will be based on the following:

\begin{lemma}
For any $\eta$ in $\R^\ast_+,$
\begin{equation}\label{intfeta}
\int_{\hE} f_\eta(x) d\mu_{\Eb_{\bullet}}(x) \geq \exp({M_{\Eb_{\bullet}}(1+\eta)}).
\end{equation}
\end{lemma}

\begin{proof} For any $(i,\eta) \in \N \times \R_+^\ast,$ we have:
\begin{equation*}
\begin{split}
 \int_{\hE} f_{i,\eta} (x) d\mu_{\Eb_{\bullet}}(x) & = \int_{\hE} e^{-\eta \pi \Vert p_i(x) \Vert^2_{\Eb_i}} d\mu_{\Eb_{\bullet}}(x) \\
 & = \int_{E_i} e^{-\eta \pi \Vert v \Vert^2_{\Eb_i}} d\mu_i(v) \\
 & = \lim_{j \ra + \infty} \int_{E_i} e^{-\eta\pi \Vert v \Vert^2_{\Eb_i}} dp_{ij\ast} \gamma_{\Eb_j}(v) \\
 & = \lim_{j \ra + \infty} \int_{E_j} e^{-\eta\pi \Vert p_{ij}(w) \Vert^2_{\Eb_i}} d\gamma_{\Eb_j}(w).
\end{split}
\end{equation*}
As the linear map $p_{ij,\R}: E_{j,\R} \ra E_{i,\R}$ has an operator norm $\leq 1$ with respect to the Euclidean norms $\Vert.\Vert_{\Eb_j}$ and  $\Vert.\Vert_{\Eb_i}$, we have:
$$\int_{E_j} e^{-\eta\pi \Vert p_{ij}(w) \Vert^2_{\Eb_i}} d\gamma_{\Eb_j}(w) \geq 
\int_{E_j} e^{-\eta\pi \Vert w \Vert^2_{\Eb_j} }d\gamma_{\Eb_j}(w) = \sum_{w \in E_j} e^{-\pi(1+\eta) \Vert w \Vert^2_{\Eb_j}}.$$
 
 By taking the limit when $j \in \N_{\geq i}$ goes to $+\infty,$ we obtain:
 $$\int_{\hE} f_{i,\eta}(x) d\mu_{\Eb_{\bullet}}(x) \geq \exp({M_{\Eb_{\bullet}}(1+\eta)}).$$
 
 Finally, we obtain (\ref{intfeta}) by taking the limit $i\ra + \infty$, since by dominated convergence:
 $$\lim_{i \ra + \infty} \int_{\hE} f_{i,\eta}(x) d\mu_{\Eb_{\bullet}}(x) =  \int_{\hE} f_\eta(x) d\mu_{\Eb_{\bullet}}(x).$$
 \end{proof}

To complete the proof of part 2) of Theorem \ref{propSum}, observe that, when $\eta$ goes to $0$, $f_\eta$ converges pointwise towards $ 1_{E^\hilb_\R}$ (see (\ref{limfeta})), and therefore, by dominated convergence:
$$\lim_{\eta \ra 0+} \int_{\hE} f_\eta(x) d\mu_{\Eb_{\bullet}}(x) = \mu_{\Eb_{\bullet}}(\hE \cap E^\hilb_\R).$$
If we now assume that  $\lim_{t \ra 1+} M_{\Eb_{\bullet}} (t) = M_{\Eb_{\bullet}}(1),$  we therefore deduce from (\ref{intfeta}):
$$\mu_{\Eb_{\bullet}}(\hE \cap E^\hilb_\R) \geq \exp(M_{\Eb_{\bullet}}(1)).$$
The definition of $M_{\Eb_{\bullet}}(1)$, together with (\ref{hotlim}) and (\ref{uhotmutot}), show that
$$\exp(M_{\Eb_{\bullet}}(1))=\mu_{\Eb_{\bullet}}(\hE),$$
and the previous inequality may also be written:
$$\mu_{\Eb_{\bullet}}(\hE \cap E^\hilb_\R) \geq \mu_{\Eb_{\bullet}}(\hE).$$

This shows that 
$$\mu_{\Eb_{\bullet}}(\hE \cap E^\hilb_\R) = \mu_{\Eb_{\bullet}}(\hE),$$
or equivalently:
\begin{equation}\label{munull}
\mu_{\Eb_{\bullet}}(\hE \setminus \hE \cap E^\hilb_\R) = 0.
\end{equation}

Besides, as observed just before Proposition \ref{boundhot}, the finiteness of $\uhot(\Ebh)$, hence of $\lhot(\Ebh)$ implies that the set $\hE \cap E^\hilb_\R$ is countable. Together with (\ref{munull}), this shows that
$$
\mu_{\Eb_{\bullet}} =  \sum_{v \in \hE \cap E^\hilb_\R} \mu_{\Eb_{\bullet}}(\{v\}) \delta_v.$$

Combined with our previous computation (\ref{uhotmutot}) of $\mu_{\Eb_{\bullet}}(\{v\})$, this proves the equality (\ref{muE}):
$$
\mu_{\Eb_{\bullet}} =  \sum_{v \in \hE \cap E^\hilb_\R} e^{-\pi \Vert v \Vert^2}\, \delta_v.$$
In particular
$$\mu_{\Eb_{\bullet}}(\hE) =  \sum_{v \in \hE \cap E^\hilb_\R} e^{-\pi \Vert v \Vert^2}.$$
By taking its logarithm, this relation becomes the equality (\ref{hthetamuE}): $\uhot(\Ebh) = \lhot(\Ebh).$

\section{A converse theorem}\label{sub:conv}

In this section, we establish a 
converse to the results in the preceding paragraphs. Before formulating it, let us formulate some simple observations concerning the invariant $\uhot(\Ebh)$ associated to some pro-Hermitian vector bundle $\Ebh$.

Observe that, if $\Ebh$ is the projective limit of some projective system $\Eb_{\bullet}$ of surjective admissible morphism of Hermitian vector bundles over $\Spec \Z,$ then, according to the very definition of $\uhot(\Ebh),$ we have:
$$\uhot(\Ebh) \leq \liminf_{i \ra + \infty} \hot(\Eb_{i}).$$ 

Moreover, this inequality is a sense optimal. Indeed there exists a decreasing sequence $(U_{i})_{i \in \N}$ in $\cU(\hE)$, which constitutes a  neighborhood  basis of $0$ in $\hE,$ such that
$$\lim_{i \ra +\infty} \hot(\Eb_{U_{i}}) = \uhot(\Ebh).$$
The pro-Hermitian vector bundle $\Ebh \simeq \varprojlim_{i} \Eb_{U_{i}}$ may therefore be realised as the projective limit of some admissible projective system $\Eb_{\bullet}$ such that
$$\lim_{i \ra +\infty} \hot(\Eb_{i}) = \uhot(\Ebh).$$

\begin{theorem}\label{converse} Let $\Ebh:=(\hE, E^\hilb_\R, \Vert.\Vert)$ be a pro-Hermitian vector bundle over $\Spec \Z$ such that
$$\lhot(\Ebh) = \uhot(\Ebh) < + \infty.$$

If $\Eb$ is  the projective limit of some admissible projective system $\Eb_{\bullet}$ such that
$$\lim_{i \ra +\infty} \hot(\Eb_{i}) = \uhot(\Ebh),$$
then there exist an increasing sequence of positive integers $i_{\bullet}:=(i_{k})_{k \in \N}$ such that the admissible projective system
\begin{equation*}
\Eb_{i_\bullet} : \Eb_{i_0} \stackrel{p_{i_0 i_{1}}}{\longleftarrow}\Eb_{i_1} \stackrel{p_{i_1 i_{2}}}{\longleftarrow}\ldots \stackrel{p_{i_{k-1} i_{k}}}{\longleftarrow}\Eb_{i_k} \stackrel{p_{i_k i_{k+1}}}{\longleftarrow} \Eb_{i_{k+1}} \stackrel{p_{i_{k+1} i_{k+2}}}{\longleftarrow} \ldots,
\end{equation*}
is summable.
\end{theorem}

As before, we have denoted by $p_{ij}:\Eb_{i}\ra \Eb_{j}$ the admissible morphisms defining the projective system $\Eb_{\bullet}.$

Theorem \ref{converse} is actually a consequence of the following more precise result:

\begin{lemma}\label{lemconverse} Under the assumptions of Theorem \ref{converse}, for any $\epsilon \in \R_{+}^\ast,$ there exists  $(i(\epsilon),j(\epsilon))$ in $\N^2$ for which  the following condition is satisfied, for any $(i,j) \in \N^2$ such that $i \leq j$:
\begin{equation*}
 i \geq i(\epsilon) \mbox{ and } j \geq j(\epsilon) \Longrightarrow \hot(\overline{\ker p_{ij}}) < \epsilon.
\end{equation*}
 \end{lemma}
 
 Indeed, taking Lemma \ref{lemconverse} for granted, for any sequence $(\eta_{k})_{k \in \N}$ in $\R_{+}^{\ast\N}$ such that 
 $\sum_{k \in \N} \eta_{k} < +\infty,$ we may find an increasing sequence $(i_{k})_{k \in \N}$ of positive integers such that, for any $k\in \N,$ $i_{k} \geq i(\eta_{k})$ and $i_{k+1} \geq j(\eta_{k})$. Then, for any $k \in \N$, 
 $\hot(\overline{\ker p_{i_{k}i_{k+1}}}) < \eta_{k},$
 and the admissible projective system $\Eb_{i_\bullet}$ is summable.

\begin{proof}[Proof of Lemma \ref{lemconverse}] We shall rely on Proposition \ref{limitgammaC} of Appendix \ref{measurespro}, which we shall apply to the projective system of countable sets equipped with a finite measure:
$$(D_{i}, \gamma_{i}) := (E_{i}, \gamma_{\Eb_{i}}).$$

The sequence of total masses 
$$\gamma_{i}(D_{i})= \gamma_{\Eb_{i}}(E_{i}) = \exp \hot(\Eb_{i})$$
is indeed convergent (of limit $\exp \uhot(\Ebh)$) in $\R_{+}$ and, inside $\hD = \hE,$ we may consider the subset
$$\cC := \hE \cap E_{\R}^\hilb.$$
Since $\lhot(\Ebh) < + \infty,$ it is countable. For any $x\in \cC,$ we have 
$$\gamma_{j}(p_{j}(x)) = e^{- \pi \Vert p_{j}(x)\Vert^2_{\Eb_{j}}},$$
which converges to 
$$\gamma(x) := e^{- \pi \Vert x \Vert^2}$$
when $j$ goes to $+\infty$.
Moreover
$$\sum_{x\in \cC}\gamma(x) = \sum_{x \in \hE \cap E_{\R}^\hilb} e^{- \pi \Vert x \Vert^2} = \exp \lhot(\Ebh)$$
is equal, by assumption, to 
$$\exp \uhot(\Ebh) = \lim_{i \ra +\infty} \gamma_{i}(D_{i}).$$

This shows that the hypotheses of Proposition \ref{limitgammaC} are fulfilled. Consequently, the sequence of measures $(\gamma_{\Eb_{i}})_{i \in \N}$ satisfies the condition $\mathbf{Conv}$, and the associated limit measure on $\hE$ is 
$$\sum_{x \in \hE \cap E_{\R}^\hilb} e^{- \pi \Vert x \Vert^2} \delta_{x}=: \mu_{\Ebh}.$$

As before, let us denote by $p_{i}: \hE \ra E_{i}$ the canonical quotient map. The subset $\{0\}$ of $\hE$ may be written as a countable decreasing intersection:
$$\{0\} = \bigcap_{i \in \N} p_{i}^{-1}(\{0\}),$$
and therefore
\begin{equation}\label{muzero}
\mu_{\Ebh}(\{0\})= \lim_{i \ra +\infty}\mu_{\Ebh}(p_{i}^{-1}(0)).
\end{equation}

Besides, since $\mu_{\Ebh}$ is the limit measure associated to $(\gamma_{\Eb_{i}})_{i\in \N},$ we have, for every $i\in \N$:
\begin{equation}\label{limtantetplus}
\begin{split}
\mu_{\Ebh}(p_{i}^{-1}(0))  & = p_{i\ast}\mu_{\Ebh}(\{0\}) = \lim_{j\in \N_{i}, j\ra +\infty} p_{ij\ast}\gamma_{\Eb_{j}}(\{0\}) \\
& =  \lim_{j\in \N_{i}, j\ra +\infty} \gamma_{\Eb_{j}}(p_{ij}^{-1}(0)) =  \lim_{j\in \N_{i}, j\ra +\infty}  \exp \hot(\overline{\ker p_{ij}}).
\end{split}
\end{equation}

Since $\mu_{\Ebh}(\{0\}) =1,$ from (\ref{muzero}) we derive the existence, for every $\epsilon \in \R_{+}^\ast$, of some $i(\epsilon) \in \N$ such that
$$\mu_{\Ebh}(p_{i(\epsilon)}^{-1}(0)) < e^\epsilon.$$
Then (\ref{limtantetplus}) shows the existence of an integer $j(\epsilon) \geq i(\epsilon)$ such that, for any $j\geq j(\epsilon),$
$$\hot(\overline{\ker p_{i(\epsilon)j}}) < \epsilon.$$
The pair $(i(\epsilon), j(\epsilon))$ satisfies the conclusion of Lemma \ref{lemconverse}. Indeed, for any $(i,j) \in \N^2$ such that
$i(\epsilon) \leq i\leq j$ and $j\geq j(\epsilon),$ we have $\ker p_{ij} \subset \ker p_{i(\epsilon)j}$
and therefore  
\begin{equation*}\hot(\overline{\ker p_{ij}}) \leq \hot(\overline{\ker p_{i(\epsilon)j}}). \qedhere
\end{equation*}
\end{proof}

\section{Strongly summable and $\theta$-finite pro-Hermitian vector bundles}\label{sub:cons} In this section, we formulate a few consequences of the main results in this chapter, Theorems  \ref{maintheorem} and \ref{converse}. 

These consequences, and the related definitions that follow them, are presented with a view toward Diophantine applications to be developed in the sequel of this monograph.

\subsection{Strongly summable pro-Hermitian vector bundles} 

\begin{corollary}\label{strongsummableeq} For any pro-Hermitian vector bundle $\Ebh$ over $\Spec \OK,$ the following conditions $\mathbf {StS_1}$ and $\mathbf {StS_2}$ are equivalent:
 
$\mathbf {StS_1 :}$ There exists $\epsilon \in \R^\ast_+$ such that
$$\lhot(\Ebh \otimes \cOb(\epsilon)) = \uhot(\Ebh \otimes \cOb(\epsilon)) < + \infty.$$
 
$\mathbf {StS_2 :}$ There exists $\eta \in \R^\ast_+$ and a projective system of surjective admissible morphisms of Hermitian vector bundles over $\Spec \OK$
$$\Eb_{\bullet} : 
\Eb_0 \stackrel{q_0}{\longleftarrow}\Eb_1 \stackrel{q_1}{\longleftarrow}\dots \stackrel{q_{i-1}}{\longleftarrow}\Eb_i \stackrel{q_i}{\longleftarrow} \Eb_{i+1} \stackrel{q_{i+1}}{\longleftarrow} \dots$$
 such that $\Ebh \simeq  \varprojlim_i \Eb_i$ and such that the projective system $\Eb_{\bullet}\otimes \cOb(\eta)$ is summable.
\end{corollary}

\begin{proof} Theorem \ref{maintheorem} shows that, when Condition (ii) holds for some $\eta \in \R^\ast_+,$ then Condition (i) holds for any $\epsilon \in ]0, \eta[.$

Conversely, assume that Condition (i) is satisfied for some $\epsilon \in \R^\ast_+.$ As already observed, by the very definition of $\lhot(\Ebh \otimes \cOb(\epsilon))$, there exists a projective system  $\Fb_{\bullet}$ of surjective admissible morphisms of Hermitian vector bundles over $\Spec \OK$ such that $\Ebh \otimes \cOb(\epsilon) \simeq  \varprojlim_i \Fb_i$ and 
$$\lim_{i \ra + \infty} \lhot(\Fb_i ) = \lhot(\Ebh \otimes \cOb(\epsilon)).$$
Theorem \ref{converse}, applied to the pro-Hermitian vector bundle $\pi_\ast \Ebh \otimes \cOb(\epsilon)$ and to the  projective system $\pi_\ast \Fb_{\bullet}$ over $\Spec \Z$, shows the existence of some increasing sequence of positive integers $i_{\bullet}$ such that  the projective sysytem $\pi_\ast \Fb_{i_\bullet}$ --- or equivalently $\Fb_{i_\bullet}$ --- is summable.  

Finally Condition (ii) is satisfied by $\eta := \epsilon$ and by
$\Eb_{\bullet} := \Fb_{i_\bullet} \otimes  \cOb(-\epsilon).$
\end{proof}

When the conditions $\mathbf {StS_1}$ and $\mathbf {StS_2}$  in Corollary \ref{strongsummableeq} are satisfied, we shall say that the pro-Hermitian vector bundle $\Ebh$ is \emph{strongly summable.} According to  $\mathbf {StS_2}$, the strongly summable pro-Hermitian vector bundles are precisely those that can be realized as projective limits of the  strongly summable  projective systems of surjective admissible morphisms of Hermitian vector bundles defined in paragraph \ref{sub:def}.  

Clearly, if $\Ebh$ is strongly summable, then, for any $\delta$ in $\R_+,$ $\Ebh \otimes \cOb(-\delta)$ is strongly summable and 
$$\lhot(\Ebh \otimes \cOb(-\delta)) = \uhot(\Ebh \otimes \cOb(-\delta)) < + \infty.$$
Then we shall write $\hot(\Ebh)$ instead of $\lhot(\Ebh)$ or $\uhot(\Ebh)$.

\subsection{$\theta$-finite pro-Hermitian and Hilbertisable vector bundles}

\begin{corollary}\label{thetafinite} Let $\Ebh:=(\Eh, (E_\sigma^\hilb, \Vert . \Vert_{\sigma})_{\sigma: K \hra \C})$ be a pro-Hermitian vector bundle over $\Spec \OK$, and let $\Et:= (\Eh, (E_\sigma^\hilb)_{\sigma: K \hra \C})$ be the associated object in ${\rm pro\widetilde{Vect}}(\OK)$.

The following conditions are equivalent:


$\mathbf{\theta-Fin_1:}$ For any pro-Hermitian vector bundle $\Ebh' :=(\Eh, (E_\sigma^\hilb, \Vert . \Vert'_{\sigma})_{\sigma: K \hra \C})$ admitting $\Et$ as associated object in ${\rm pro\widetilde{Vect}}(\OK)$, 
\begin{equation}\label{Ebh'}
\lhot(\Ebh') =\uhot(\Ebh') < + \infty.
\end{equation}

$\mathbf{\theta-Fin_2:}$ For any $\delta \in \R,$
$$\lhot(\Ebh \otimes \cOb(\delta)) =\uhot(\Ebh \otimes  \cOb(\delta)) < + \infty.$$
\end{corollary}

\begin{proof}
 For any $\delta \in \R,$ the pro-Hermitian vector bundles $\Ebh$ and $\Ebh\otimes \cOb(\delta)$ define the same object in ${\rm pro\widetilde{Vect}}(\OK)$. Therefore $\mathbf{\theta-Fin_1}$ implies $\mathbf{\theta-Fin_2}$.
 
 Conversely, let us assume that $\mathbf{\theta-Fin_2}$ holds and consider a pro-Hermitian vector bundle $\Ebh' :=(\Eh, (E_\sigma^\hilb, \Vert . \Vert'_{\sigma})_{\sigma: K \hra \C})$ admitting $\Et$ as associated object in ${\rm pro\widetilde{Vect}}(\OK)$. There exists $\lambda \in \R^\ast_+$ such that the identity maps on $\hE$ and on the $E_\sigma^\hilb$'s 
   define a morphism in 
 $\Hom^{\leq \lambda}_{\OK}(\Ebh',\Ebh),$ or equivalently in $\Hom^{\leq 1}_{\OK}(\Ebh',\Ebh \otimes \cOb(\log \lambda)).$ According to (ii), $\Ebh \otimes \cOb(\log \lambda)$ is satisfies $\mathbf{StS_1}$, and therefore, by Corollary \ref{strongsummableeq}, it satisfies $\mathbf{StS_2}$. This implies that $\Ebh'$ also satisfies $\mathbf{StS_2}$, and is therefore is strongly summable,  and consequently that (\ref{Ebh'}) holds. This establishes $\mathbf{\theta-Fin_1}$.
\end{proof}

When the equivalent conditions in Corollary \ref{thetafinite} are realized --- that is, when $\Ebh\otimes \cOb (\delta)$ is strongly summable for every $\delta \in \R$ --- we shall say that $\Ebh$ and $\Et$ are \emph{$\theta$-finite.} 

This terminology makes reference to the fact that,  to any $\theta$-finite $\Ebh$ as above, one may associate its theta function $$\theta_{\Ebh}: \R^\ast_+ \lra [1, + \infty[ $$
defined, for any $t\in \R^\ast_+$, by: 
$$\theta_{\Ebh}(t):=\exp \left({\hot(\Ebh \otimes \cOb(-(\log t)/2))}\right).$$
According to (\ref{hthetamuE}), we have:
$$\theta_{\Ebh}(t)= \sum_{v \in \hE \cap E^\hilb_\R} e^{-\pi t \Vert v \Vert^2}, $$
where as usual, $\Vert. \Vert$ denotes the Hilbert norm on $E^\hilb_\C \simeq \bigoplus_{\sigma: K \hra \C} E^{\hilb}_\sigma$ defined by
$$\Vert (v_\sigma)_{\sigma: K \hra \C} \Vert^2 := \sum_{\sigma: K \hra \C} \Vert v_\sigma \Vert^2_\sigma.$$ 

Observe  that, for any $\theta$-finite pro-Hermitian vector bundle $\Ebh,$ the subgroup $\hE \cap E^\hilb_\R$ is discrete in the Hilbert space $(E_\R^\hilb, \Vert. \Vert)$. (This directly follows from Proposition \ref{boundhot}.) Consequently, as discussed in Section \ref{HilbertSubgroups} (see notably Corollary \ref{corHilbertSubgroups}),  it defines an \emph{ind}-Euclidean lattice:   $$
\Ebh^{\rm ind} := (\hE \cap E^\hilb_\R, \Vert. \Vert).$$  This ind-Euclidean lattice is easily seen to satisfy the equivalent conditions in Proposition \ref{indthetafinite}. Actually, for any $\delta \in \R,$ the following equality holds:
$$\hot( \Ebh^{\rm ind}\!\!\otimes \cOb (\delta)) = \hot(\Ebh \otimes \cOb(\delta)) = \log \theta_{\Ebh}(e^{-2\delta}) .$$

\subsection{Examples} From the computations in Subsections \ref{examplecountable}, \ref{thetaHardyar} and \ref{Provanish}, we immediately obtain explicit examples of $\theta$-summable pro-Hermitian vector bundles:

\begin{proposition}
 1) With the notation of paragraph \ref{examplecountable}, for any element $\bm{\lambda} := (\lambda_i)_{i \in \N}$ of $\R_+^{\ast \N},$ the pro-Euclidean lattice $\Vbh_{\bm{\lambda}}$ is $\theta$-summable if and only if, for every $t\in \R_+^\ast,$
 $$\sum_{i \in \N} e^{-t \lambda_i} < + \infty.$$
 
 2) For any positive real number $R$, the pro-Euclidean lattice $\Hbh(R)$ is $\theta$-summable if and only if $R>1.$ \qed
\end{proposition}

Observe also that, for any positive integer $a$ and any $\delta \in \R,$ the Euclidean lattice $\Eb_a$ considered in \ref{Provanish} satisfies:
$$\hot(\Eb_a \otimes \cO(\delta)) = a\, \tau(a e^{-2\delta}).$$
As $\lim_{a \ra +\infty} a \,\tau(a e^{-2\delta}) =0,$  we immediately obtain:

\begin{proposition}
 For any infinite set $\cA$ of positive integers in which the divisibility relation $\mid$ is filtering, the pro-Euclidean lattice $\Et_{\cA}$ defined in \ref{Provanish} is $\theta$-summable and satisfies 
 $\theta_{\Et_{\cA}} = 0.$
\end{proposition}

\subsection{A permanence property} We conclude this section by showing that the property, for a pro-Hermitian vector bundle, of being strongly summable or $\theta$-finite is inherited by its ``closed sub-bundles". 

\begin{proposition}\label{permStTetc}
 Let $f: \Ebh \lra \Fbh$ be a morphism of pro-Hermitian vector bundles over $\Spec \OK$ such that the underlying morphism of topological $\OK$-modules
 $\hat{f}: \Eh \lra \Fh$ is injective and strict.
 
 1) If $f$ belongs to $\Hom_\OK^{\leq 1}(\Ebh, \Fbh)$ and if $\Fbh$ is strongly summable, then $\Ebh$ is strongly summable.
 
 2) If $\Fbh$ is $\theta$-finite, then $\Ebh$ is $\theta$-finite.
 \end{proposition}

\begin{proof} 1) Let us assume that $\Fbh$ is strongly summable, and hence satisfies $\mathbf{StS_2}$. Then we may choose a defining sequence $(V_i)_{i \in \N}$ in $\cU(\Fh)$ and $\epsilon >0$ such that, if we let $\Fb_i:= \Fb_{V_i}$, the admissible surjective morphisms 
$$r_i := p_{V_i V_{i+1}}: \Fb_{i+1} \lra \Fb_i$$
satisfy
$$\sum_{i \in \N} \hot(\overline{\ker r_i}\otimes \overline{\cO(\epsilon)}) < + \infty.$$ 

We may apply Proposition \ref{strictinjective}, part 1), to  the strict injective morphism $\hat{f}: \Eh \lra \Fh$ in $\CTC_{\OK}$ and to the defining sequence $(V_i)_{i \in \N}$. Therefore the sequence $(U_i)_{i \in \N} := (\hat{f}^{-1}(V_i))_{i \in \N}$ is a defining sequence in $\cU(\Eh)$, and we may consider the Hermitian vector bundles $\Eb_i:=\Eb_{U_i}$ and the injective morphisms of finitely generated projective $\OK$-modules
$$f_i := E_i := \Eh/U_i \lra F_i := \Fh/V_i$$
induced by $\hat{f}.$

Let us also assume that $f$ belongs to $\Hom_\OK^{\leq 1}(\Ebh, \Fbh)$. Then, for every $i\in \N,$ $f_i$ belongs to $\Hom_\OK^{\leq 1}(\Eb_i, \Fb_i)$. Indeed the Hermitian norms defining $\Eb_i$ and  $\Fb_i$ are the quotient norms of the Hilbert space norms defining $\Ebh$ and $\Fbh$.

Let us consider the surjective admissible morphisms of Hermitian vector bundles over $\Spec \OK$:
$$q_i := p_{U_i U_{i+1}}: \Eb_{i+1} \lra \Eb_i.$$ 
For any $i \in \N$, the commutativity of the diagram 
 \begin{equation}\label{CDqr}
\begin{CD}
 \Eb_{i+1} @>{q_i}>> \Eb_i \\
 @V{f_{i+1}}VV                @VV{f_i}V              \\
 \Fb_{i+1} @>{r_i}>> \Fb_i
 \end{CD}
\end{equation}
shows that $f_{i+1}$ defines a morphism
$$f_{i+1}: \ker q_i \lra \ker r_i$$
which is injective and belongs to $\Hom_\OK^{\leq 1} (\overline{\ker q_i}, \overline{\ker r_i})$.

Therefore
$$\hot(\overline{\ker q_i}\otimes \overline{\cO(\epsilon)}) \leq  \hot(\overline{\ker r_i}\otimes \overline{\cO(\epsilon)})$$
(by Proposition \ref{ineqmortheta}, 1)) and finally, we obtain:
$$\sum_{i \in \N} \hot(\overline{\ker q_i}\otimes \overline{\cO(\epsilon)}) \leq \sum_{i \in \N} \hot(\overline{\ker r_i}\otimes \overline{\cO(\epsilon)}) < + \infty.$$
Thus $\Ebh$ also satisfies $\mathbf{StS_2}$.

2) For any embedding $\sigma: K \hlra \C,$ let us denote by $\Vert f_\sigma \Vert$ the operator norm of the continuous linear map $f_\sigma$ between the Hilbert spaces $(E_\sigma^\hilb, \Vert.\Vert_{\Ebh, \sigma})$ and $(F_\sigma^\hilb, \Vert.\Vert_{\Fbh, \sigma})$, and let us choose a real number such that
$$\lambda \geq \max_{\sigma: K \hra \C} \log \Vert f_\sigma \Vert.$$ 
 Then $f$ belongs to $\Hom_\OK^{\leq 1}(\Ebh, \Fbh\otimes \cOb(\lambda))$ and therefore, for any $\delta \in \R,$ to $\Hom_\OK^{\leq 1}(\Ebh\otimes\cOb(\delta), \Fbh\otimes\cOb(\lambda+\delta))$.
 
 Let us assume that $\Fbh$ is $\theta$-finite. Then, for any $\delta \in \R,$  $\Fbh\otimes\cOb(\lambda+\delta)$ is strongly summable, and therefore, according to Part 1), $\Ebh \otimes\cOb(\delta)$ is strongly summable. This proves that $\Ebh$ is $\theta$-finite. 
\end{proof}

\medskip

\chapter[Exact sequences of 
 infinite dimensional Hermitian vector bundles]{Exact sequences of 
 infinite dimensional Hermitian vector bundles and 
subadditivity of their $\theta$-invariants}\label{subaddinfinite} 
\medskip

In this  chapter, we investigate the properties of \emph{short exact sequences} and \emph{admissible short exact sequences} of pro- and ind-Hermitian vector bundles.

We define these short exact sequences in Section \ref{shortinfdim}. Then, in Section \ref{SortExTheta}, we extend the subadditivity properties of the $\theta$-invariants, previously considered in Section  \ref{subadd} and Subsection \ref{alternating} for Hermitian vector bundles, to the infinite dimensional setting, and  we 
give some applications of these subadditivity properties to strongly summable and $\theta$-finite pro-Hermitian vector bundles in Section \ref{ShortStrong}.  

In Section \ref{VanCrit}, we apply these properties of $\theta$-invariants of $\theta$-finite pro-Hermitian vector bundles to derive a criterion for the vanishing of connecting maps arising in some systems of exact sequences of pro-Hermitian vector bundles that will play a crucial role in Diophantine applications.

Finally, in Section \ref{proexact} --- that can be read immediately after the  definition of short exact sequences of pro-Hermitian vector bundles introduced in \ref{shortpro} --- we investigate more closely the formal structure of the additive category $\proVectOK$ equipped with the class of short exact sequences introduced in this chapter. We prove notably that it constitutes an idempotent complete \emph{exact category} in the sense of Quillen. 

From the perspective of applications,   the interest of this result stems from that fact that, according to  some  classical constructions of Heller, Deligne, Keller,..., 
it shows that the category $\proVectOK$ admits an associated derived category which basically satisfies the same formal properties as the derived category of an abelian category.  

\bigskip
In this chapter, we focus on short exact sequences in categories of pro-Hermitian vector bundles, and we leave the formulation of the dual results, concerning short exact sequences of ind-Hermitian vector bundles, to the interested reader.

\bigskip

We still denote by $K$ a number field, by $\OK$ its ring of integers, and by $\pi: \Spec \OK \lra \Spec \Z$ the morphism of schemes from $\Spec \OK$ to $\Spec \Z.$ . 

\section{Short exact sequences of infinite dimensional Hermitian vector bundles}\label{shortinfdim}

\subsection{Short exact sequences of pro-Hermitian vector bundles}\label{shortpro} We define a \emph{short exact sequence of pro-Hermitian vector bundles over $\Spec \OK$} as a diagram
\begin{equation}\label{eq:shortpro}
0 \lra \Ebh \stackrel{i}{\lra} \Fbh \stackrel{q}{\lra} \Gbh \lra 0
\end{equation}
in ${\rm pro\overline{Vect}}^{\rm cont}(\OK)$ such that the following two conditions are satisfied:
 
 $\bf ProSE_1 :$ \emph{the diagram of $\OK$-modules
\begin{equation}\label{shorthat}
0 \lra \Eh \stackrel{\hat{i}}{\lra} \Fh \stackrel{\hat{q}}{\lra} \Gh \lra 0
\end{equation}
is a short exact sequence;} 

 $\bf ProSE_2 :$ \emph{for every embedding $\sigma: K\hra\C,$ the complex of $\C$-vector spaces
\begin{equation}\label{shorthilb}
0 \lra E_\sigma^\hilb \stackrel{i_\sigma}{\lra} F_\sigma^\hilb \stackrel{q_\sigma}{\lra} G_\sigma^\hilb \lra 0.
\end{equation}
is a short exact sequence.}

Let us formulate a few observations concerning this definition:

(i) It is straightforward that the diagram (\ref{eq:shortpro}) in ${\rm pro\overline{Vect}}^{\rm cont}(\OK)$ is a short exact sequence of pro-Hermitian vector bundles over $\Spec \OK$ if and only if the diagram  
$$0 \lra \pi_\ast\Ebh \stackrel{i}{\lra} \pi_\ast\Fbh \stackrel{q}{\lra} \pi_\ast\Gbh \lra 0$$ 
is a short exact sequence of pro-Hermitian vector bundles over $\Spec \Z$.

(ii) The exactness of the diagram (\ref{shorthat}) implies it is actually a \emph{split} short exact sequence of topological $\OK$-module. Namely there exists a continuous morphism of topological $\OK$-modules
$$\hat{s}: \Gh \lra \Fh$$
such that
$$\hat{q}\circ \hat{s} = Id_{\hG}.$$
This follows from the results on strict morphisms in $\CTC_A$ established in Section  \ref{StrictCTC}, applied to the ring $A= \OK$: this rings satisfies  $\mathbf{Ded_3}$ and therefore, according to Proposition \ref{strictCTCDed3}, the morphisms $\hat{i}$ and $\hat{q}$ are strict; therefore, by Proposition \ref{STsplit}, the short exact sequence (\ref{shorthat}) is split. 

 In particular, the diagram  (\ref{shorthat}) remains exact --- actually split --- after any ``completed base change". Notably, for every embedding $\sigma: K \hra \C,$ the diagram 
\begin{equation*}\label{shortfrech}
0 \lra \Eh_\sigma\stackrel{\hat{i}_\sigma}{\lra} \Fh_\sigma \stackrel{\hat{q}_\sigma}{\lra} \Gh_\sigma \lra 0
\end{equation*}
is a split short exact sequence of complex Fr\'echet spaces.

(iii) Similarly, according to Banach's open mapping theorem and to basic results on Hilbert spaces, the morphism $i_\sigma$ and $p_\sigma$ in the diagram (\ref{shorthilb}) are strict, and this diagram is actually split in the category of topological vector spaces. 

We shall say that the short exact sequence (\ref{eq:shortpro}) in ${\rm pro\overline{Vect}}^{\rm cont}(\OK)$ is \emph{admissible} when the maps $i_\sigma: E_\sigma^\hilb \lra F_\sigma^\hilb$ (resp.  $q_\sigma: F_\sigma^\hilb \lra G_\sigma^\hilb$) are isometries (resp. co-isometries\footnote{In other words, their adjoints $q_\sigma^\ast$ are isometries.})  from $(E_\sigma^\hilb, \Vert. \Vert_\sigma^\hilb)$ to $(F_\sigma^\hilb, \Vert. \Vert_\sigma^\hilb)$ (resp. from $(F_\sigma^\hilb, \Vert. \Vert_\sigma^\hilb)$ to $(G_\sigma^\hilb, \Vert. \Vert_\sigma^\hilb)$).

For simplicity, we will often call a short exact sequence ({eq:shortpro}) as defined above \emph{a short exact sequence in $\proVectOK$}. This notion is clearly compatible with isomorphisms in $\proVectOK$, and immediately leads to a notion of short exact sequence in $OK.$

\subsection{Short exact sequences of ind-Hermitian vector bundles. Duality} 

We define a \emph{short exact sequence of ind-Hermitian vector bundles over $\Spec \OK$} as a diagram 
\begin{equation}\label{shortind}
0 \lra \Eb \stackrel{i}{\lra} \Fb \stackrel{q}{\lra} \Gb \lra 0
\end{equation}
in ${\rm ind\overline{Vect}}^{\rm cont}(\OK)$ such that the following two conditions are satisfied:

$\bf IndSE_1 :$ \emph{the diagram of $\OK$-modules
\begin{equation}\label{shortshort}
0 \lra E \stackrel{{i}}{\lra} F \stackrel{{q}}{\lra} G \lra 0
\end{equation}
is a short exact sequence;}

$\bf IndSE_2 :$  \emph{for every embedding $\sigma: K\hra\C,$ the diagram of $\C$-vector spaces 
\begin{equation}\label{shortsigma}
0 \lra E_\sigma \stackrel{i_\sigma}{\lra} F_\sigma \stackrel{q_\sigma}{\lra} G_\sigma \lra 0
\end{equation}
is a short exact sequence, and $i_\sigma$ and $q_\sigma$ are strict morphisms of topological vector spaces} when $E_\sigma,$ $F_\sigma,$ and $G_\sigma$ are equipped with the topology defined by the Hermitian norms $\Vert .\Vert_{E,\sigma}$, $\Vert .\Vert_{F,\sigma}$, and 
$\Vert .\Vert_{G,\sigma}$ that define $\Eb := (E, (\Vert .\Vert_{E,\sigma})_{\sKC})$, $\Fb := (F, (\Vert .\Vert_{F,\sigma})_{\sKC})$ and
$\Gb := (G, (\Vert .\Vert_{G,\sigma})_{\sKC})$.

The last conditions means that the norms $\Vert .\Vert_{E,\sigma}$ and $\Vert i_\sigma(.)\Vert_{F,\sigma}$ on $E_\sigma$ are equivalent, and that the norm $\Vert .\Vert_{G,\sigma}$ and the norm quotient of $\Vert .\Vert_{F,\sigma}$ \emph{via} $q_\sigma$ on $G_\sigma$ are equivalent.

We shall say that the short exact sequence (\ref{shortind}) is \emph{admissible} when the norms $\Vert .\Vert_{E,\sigma}$ and $\Vert i_\sigma(.)\Vert_{F,\sigma}$ on $E_\sigma$ coincide, and also  the norm $\Vert .\Vert_{G,\sigma}$ and the norm quotient of $\Vert .\Vert_{F,\sigma}$ on $G_\sigma$.

 These definitions are compatible, via the duality between ${\rm ind\overline{Vect}}^{\rm cont}(\OK)$ and ${\rm pro\overline{Vect}}^{\rm cont}(\OK)$, with the definitions of short exact sequences  and short exact sequences in ${\rm pro\overline{Vect}}^{\rm cont}_\OK$ introduced in the previous paragraph. 
 
 Namely, by combining the basic facts concerning the duality of ind- and pro-Hermitian vector bundles presented in Section  \ref{dualindpro}, the results on exact sequences in $\CP_A$ and $\CTC_A$ established paragraph \ref{stricshortdual} (in the special case where $A= \OK$, and therefore satisfies $\mathbf{Ded_3}$), and the basic theory of Hilbert spaces, one may establish the following proposition, the proof of which is left to the reader:

\begin{proposition}
 Let $\Eb,$ $\Fb$ and $\Gb$ be three ind-Hermitian vector bundles over $\Spec \OK,$ and let $\Eb^\vee$, $\Fb^\vee$ and $\Gb^\vee$ be the dual pro-Hermitian vector bundles over $\Spec \OK$.
 
 For any two morphisms $i : \Eb \lra \Fb$ and $q: \Fb \lra \Gb$ of ind-Hermitian vector bundles over $\Spec \OK,$ the diagram
 $$0 \lra \Eb \stackrel{i}{\lra} \Fb \stackrel{q}{\lra} \Gb \lra 0$$
  is a short exact sequence (resp. an admissible short exact sequence) of ind-Hermitian vector bundles over $\Spec \OK$ if and only if the dual diagram   
  $$0 \lra \Gb \stackrel{q^\vee}{\lra} \Fb^\vee \stackrel{i^\vee}{\lra} \Eb^\vee \lra 0$$
 is a short exact sequence (resp. an admissible short exact sequence) of pro-Hermitian vector bundles over $\Spec \OK$. \qed
\end{proposition}
 
\section{Short exact sequences and $\theta$-invariants of pro-Hermitian vector bundles}\label{SortExTheta}
\subsection{}

This section is devoted to the following extension to infinite dimensional Hermitian vector bundles of the sub-additivity of the $\theta$-invariants established for finite rank Hermitian vector bundles in Section \ref{subadd} and Subsection \ref{alternating}. 

\begin{theorem}\label{subaddpro} Consider an admissible short exact sequence of pro-Hermitian vector bundles over the arithmetic curve $\Spec \OK$:
\begin{equation*}
0 \lra \Ebh \stackrel{i}{\lra} \Fbh \stackrel{q}{\lra} \Gbh \lra 0
\end{equation*}

1) The following inequalities hold in $[0, + \infty]$:
\begin{equation}\label{subaddlhotpro}
\lhot(\Fbh) \leq \lhot(\Ebh) + \lhot(\Gbh),
\end{equation}
\begin{equation}\label{subadduhotpro}
\uhot(\Fbh) \leq \uhot(\Ebh) + \uhot(\Gbh),
\end{equation}
and
\begin{equation}\label{subaddhutpro}
\hut(\Fbh) \leq \hut(\Ebh) + \hut(\Gbh),
\end{equation}

2) If  $\Ebh$ has finite rank \emph{(or equivalently, is defined by some Hermitian vector bundle $\Eb$ over $\Spec \OK$)}, then we also have:
\begin{equation}\label{subaddproEfinitel}
\lhot(\Gbh) \leq \lhot(\Fbh) - \dega \pi_\ast \Eb.
\end{equation}

3) If $\Gbh$ has finite rank \emph{(or equivalently, is defined by some Hermitian vector bundle $\Gb$ over $\Spec \OK$)}, then
\begin{equation}\label{subaddprohut}
\hut(\Ebh) \leq \hut(\Fbh) + \dega \pi_\ast \Gb.
\end{equation}

\end{theorem}

By means of arguments similar to the one in the proof of part 1) of Theorem \ref{permStS} \emph{infra}, one may also show that, in the situation of part 2) of Theorem  \ref{subaddpro} (that is when  $\Ebh$ is some Hermitian vector bundle $\Eb$), we also have:
\begin{equation}\label{subaddproEfiniteu}
\uhot(\Gbh) \leq \lim_{\epsilon \rightarrow 0_+} \uhot(\Fbh\otimes \cOb(\epsilon)) - \dega \pi_\ast \Eb.
\end{equation}
We leave this to the interested reader.

For later reference, let us record some straightforward consequences of Theorem  \ref{subaddpro} and of the estimates previously established in Proposition  \ref{ineqmorthetapro}. 

\begin{corollary}\label{finite rankcor} Let us keep the notation of Theorem \ref{subaddpro}.

1) Let us assume that $\Ebh$ has finite rank. Then $\lhot(\Fbh)  < + \infty$ if and only if $\lhot(\Gbh)  < + \infty.$

2) Let us assume that $\Gbh$ has finite rank. Then $\lhot(\Ebh)  < + \infty$ if and only if $\lhot(\Fbh)  < + \infty,$ and
$\uhot(\Ebh)  < + \infty$ if and only if $\uhot(\Fbh)  < + \infty.$ \qed 
\end{corollary}

\subsection{Proof of Theorem \ref{subaddpro}: I. Preliminary}

Let us consider 
an admissible short exact sequence of pro-Hermitian vector bundles over the arithmetic curve $\Spec \OK$:
\begin{equation}\label{shortencore}
0 \lra \Ebh \stackrel{i}{\lra} \Fbh \stackrel{q}{\lra} \Gbh \lra 0
\end{equation}

By the very definition of a short exact sequence in ${\rm pro\overline{Vect}}^{\rm cont}(\OK)$, from (\ref{shortencore}), we derive a commutative diagram with exact lines, where the vertical arrows denote the inclusion maps defining the pro-Hermitian vector bundles $\Ebh,$ $\Fbh,$ and $\Gbh$:
\begin{equation}\label{bigCD}
\begin{CD}
 0 @>>> \Eh @>\hat{i}>> \Fh @>\hat{q}>> \Gh @>>> 0 \\
 @.          @VVV                @VVV                 @VVV       @. \\
  0 @>>> \Eh_\R @>\hat{i}_\R>> \Fh_\R @>\hat{q}_\R>> \Gh_\R @>>> 0 \\
   @.          @AAA               @AAA                 @AAA       @. \\
  0 @>>> E^\hilb_\R @>{i}_\R>> F^\hilb_\R @>q_\R>> G^\hilb_\R @>>> 0 
\end{CD}
\end{equation}

\begin{lemma}\label{CartfromCD}
 The commutative diagrams
 \begin{equation}\label{smallCD1}
\begin{CD}
 \Eh @>\hat{i}>> \Fh \\
 @VVV                @VVV              \\
 \Eh_\R @>\hat{i}_\R>> \Fh_\R 
 \end{CD}
\end{equation}
and
\begin{equation}\label{smallCD2}
\begin{CD}
 \Eh_\R @>\hat{i}_\R>> \Fh_\R  \\
     @AAA               @AAA                \\
 E^\hilb_\R @>{i}_\R>> F^\hilb_\R 
\end{CD}
\end{equation}
extracted from (\ref{bigCD}) are cartesian squares\footnote{of sets, and also respectively of topological $\Z$-modules and of topological $\R$-vector spaces.}.
\end{lemma}

\begin{proof}
 The cartesian character of (\ref{smallCD1}) (resp.,  of  (\ref{smallCD2}))  follows from the exactness of the first (resp., of the last) two lines in the commutative diagram (\ref{bigCD}) and the injectivity of the vertical map $\Gh \lra \Gh_\R$ (resp. $G^\hilb_\R \lra \Gh_\R$). 
\end{proof}

As already observed in \ref{shortpro}, we may choose a continuous $\OK$-linear splitting 
$$\hat{s}: \Gh \lra \Fh$$
of the map $\hat{p}: \Fh \lra \Gh$.

For any $U \in \cU(\Eh)$ and any $W \in \cU(\Gh),$ the (direct) sum
$$V:= \hat{i}(U) + \hat{s}(W)$$
belongs to $\cU(\Fh).$ The quotients $E_U := \Eh /U,$ $F_V:= \Fh/V$ and $G_W:= \Gh/W$ fit into a commutative diagram of $\OK$-modules

\begin{equation}\label{bigCDUW}
\begin{CD}
 0 @>>> \Eh @>\hat{i}>> \Fh @>\hat{q}>> \Gh @>>> 0 \\
 @.          @V{p_U}VV                @V{p_V}VV                 @V{p_W}VV       @. \\
  0 @>>> E_U @>i_{UW}>> F_V @>q_{UW}>> G_W @>>> 0, \end{CD}
\end{equation}
where the vertical arrows denote the canonical surjections.

Besides, $E_U,$ $F_V$ and $G_W$ are the  finitely generated  projective $\OK$-modules underlying the Hermitian vector bundles $\Eb_U$, 
$\Fb_V$ and $\Gb_W$ defined from the pro-Hermitian vector bundles $\Ebh,$ $\Fbh$ and $\Gbh$. The real vector spaces $E_{U, \R}$, 
$F_{V,\R}$ and $G_{W,\R}$ may be endowed with the Euclidean norms $\Vert.\Vert_{\pi_\ast \Eb_U}$, $\Vert.\Vert_{\pi_\ast \Fb_V}$ and $\Vert.\Vert_{\pi_\ast \Gb_W}$. 

The norm $\Vert.\Vert_{\pi_\ast \Eb_U}$ is characterized by the fact that the composite map
$$p'_{U,\R}: E^\hilb_\R \hlra \Eh_\R \stackrel{p_{U,\R}}\lra  E_{U, \R}$$
--- which is onto since the image of $E^\hilb_\R$ is dense in $\Eh_\R$ ---  is a co-isometry when $E^\hilb_\R$ is equipped with the Hilbert norm attached to $\Ebh$ and $E_{U, \R}$ is equipped with $\Vert.\Vert_{\pi_\ast \Eb_U}$. Similar remarks apply to $\Vert.\Vert_{\pi_\ast \Fb_V}$ and $\Vert.\Vert_{\pi_\ast \Gb_W}$. Finally, from the last two lines of (\ref{bigCD}), we deduce a commutative diagram
\begin{equation}\label{bigCDHilbR}
\begin{CD}
  0 @>>> E_{U,\R} @>i_{UW,\R}>> F_{V,\R} @>q_{UW,\R}>> G_{W,\R} @>>> 0 \\
   @.          @A{p'_{U,\R}}AA               @A{p'_{V,\R}}AA                 @A{p'_{W,\R}}AA       @. \\
  0 @>>> E^\hilb_\R @>{i}_\R>> F^\hilb_\R @>q_\R>> G^\hilb_\R @>>> 0 
\end{CD}
\end{equation}
where the vertical arrows are co-isometries.

\begin{lemma}\label{UWnorms}
Let us keep the previous notation.

1) For any $e \in E_{U,\R}$, we have
\begin{equation}\label{iUWdecreases}
\Vert i_{UW,\R}(e) \Vert_{\pi_\ast \Fb_V} \leq \Vert e \Vert_{\pi_\ast \Eb_U}.
\end{equation}
Moreover $\Vert i_{UW,\R}(e) \Vert_{\pi_\ast \Fb_V}$ is a non-increasing function of $W \in \cU(\Gh)$ and
\begin{equation}\label{iUWlim}
\lim_{W \in \cU(\Gh)} \Vert i_{UW,\R}(e) \Vert_{\pi_\ast \Fb_V} = \Vert e \Vert_{\pi_\ast \Eb_U}. 
\end{equation}

2) The map $q_{UW,\R}: F_{V,\R} \lra G_{W,\R}$ is a co-isometry.
\end{lemma}

The proof of (\ref{iUWlim}) will rely on the following simple proposition concerning orthogonal projections in Hilbert spaces, that we shall leave as an exercise:
\begin{proposition}\label{Hilbertincr}
 Let $(H, \Vert . \Vert)$ be a (real of complex) Hilbert space and let
 $$H_0 \supseteq H_1 \supseteq H_2 \supseteq H_3 \supseteq \ldots \supseteq H_n \supseteq H_{n+1}  \ldots$$
 be a decreasing sequence of closed vector subspaces of $H$.
 
 For any $n \in \N,$ let $$p_n: H \lra H_n^\perp$$ denote the orthogonal projection from $H$ onto the orthogonal complement  $H_n^\perp$ of $H_n$ in $H$, and let $$p_\infty : H \lra (\bigcap_{n \in \N} H_n)^\perp$$ be the orthogonal projection onto the orthogonal complement of $\bigcap_{n \in \N} H_n$. 
 
 Then, for any $x \in H,$
 $$\lim_{n \ra + \infty} \Vert p_n(x) - p_\infty(x) \Vert = 0,$$
 and  $(\Vert p_n(x)\Vert)_{n \in \N}$ is a non-decreasing sequence in $\R_=$ of limit $\Vert p_\infty(x) \Vert$.
 \qed 
\end{proposition}

\begin{proof}[Proof of Lemma \ref{UWnorms}] 
 
1) The estimate (\ref{iUWdecreases}) follows from the fact that, in the commutative diagram (\ref{bigCDHilbR}), the map $i_\R$ is an isometry and $p'_{U,\R}$ and $p'_{V,\R}$ are co-isometries. 

When $W$ decreases to $\{0\}$, $V$ decreases to $\hat{i}(U)$ and $V_\R$ decreases to $\hat{i}(U)_\R.$ More precisely, if $(W_n)_{n\in\N}$ is a defining sequence in $\cU(\Gh)$ and if $V_n := \hat{i}(U) + \hat{s}(W_n),$ then 
$$\bigcap_{n \in \N} V_{n,\R} = \hat{i}(U)_\R.$$ 

This shows that  $\Vert i_{UW,\R}(e) \Vert_{\pi_\ast \Fb_V}$ is a non-increasing function of $W \in \cU(\Gh)$. Moreover the equality (\ref{iUWlim}) follows from Proposition \ref{Hilbertincr} applied to $H= F_\R^\hilb$ and $H_n := V_{n,\R} \cap F_\R^\hilb.$

2) In the commutative diagram (\ref{bigCDHilbR}), the maps $p'_{V,\R}$, $q_\R$ and $p'_{W,\R}$ are co-isometries. Therefore $q_{UW,\R}$ is a co-isometry. \end{proof}

We shall  denote by $\Eb_U^W$  the Hermitian vector bundle over $\Spec \OK$ defined by the $\OK$-modules $E_U$ equipped with the Hermitian structure induced by the one of $\Fb_V$. In other words, the Hermitian structure  on $\Eb_U^W$
is defined in a way that makes the second line of (\ref{bigCDUW}) an admissible short exact sequence of Hermitian vector bundles over $\Spec \OK$:
\begin{equation}\label{CDUWadm}
\begin{CD}
0 @>>> \Eb_U^W @>i_{UW}>> \Fb_V @>q_{UW}>> \Gb_W @>>> 0.
\end{CD}
\end{equation}
(Recall that, according to Lemma \ref{UWnorms}, 2), the map  $q_{UW,\R}$ is a co-isometry.)

\begin{lemma}\label{UWrecap} For any $(U,W)$ in $\cU(\Eh) \times \cU(\Gh),$ we have:
 \begin{equation}\label{UV1}
\hot(\Fb_{\hat{i}(U) + \hat{s}(W)}) \leq \hot(\Eb_U^W) + \hot(\Gb_W).
\end{equation}

 Moreover, for any $U$ in $\cU(\Eh),$
$\hot(\Eb_U^W)$ is a non-decreasing function of $W \in \cU(\Gh)$ and 
\begin{equation}\label{UV3}
\lim_{W \in \cU(\Gh)} \hot(\Eb_U^W) = \hot(\Eb_U).
\end{equation}
\end{lemma}

\begin{proof} The estimate (\ref{UV1}) follows the subadditivity of $\hot$ for admissible short exact sequences of Hermitian vector bundles (see Proposition \ref{ineqshorttheta}) applied to (\ref{CDUWadm}).

The second assertion in Lemma \ref{UWrecap}
follows from part 1) of Lemma \ref{UWnorms}.
\end{proof}

\begin{lemma}\label{UWrecapbis} For any $(U,W)$ in $\cU(\Eh) \times \cU(\Gh),$ we have
 \begin{equation}\label{UV1BIS}
\hut(\Fb_{\hat{i}(U) + \hat{s}(W)}) \leq \hut(\Eb_U^W) + \hut(\Gb_W).
\end{equation}
and 
\begin{equation}\label{UV2BIS}
\hut(\Eb_U^W) \leq \hut(\Eb_U).
\end{equation}
\end{lemma}

\begin{proof} The estimate (\ref{UV1BIS}) follows from the subadditivity of $\hut$ for admissible short exact sequences of Hermitian vector bundles (see Proposition \ref{ineqshorttheta} and equation (\ref{hothutshort}))
applied to (\ref{CDUWadm}). 

The estimate (\ref{iUWdecreases}) shows that the identity map $Id_{E_U}: \Eb_U \lra \Eb_U^W$ is norm decreasing. This implies (\ref{UV2BIS}) (see Proposition \ref{ineqmortheta}, 2)).
\end{proof}

\subsection{Proof of Theorem \ref{subaddpro}: II. Completion of the proof}

We keep the notation of the previous paragraph.

1) (i) Let $P$ be a finitely generated $\OK$-submodule of $\Fh \cap F_\R^\hilb.$

According to Lemma (\ref{CartfromCD}), the inverse images of $P$ by $\hat{i}$ and by $i_\R$ coincide and define a (finitely generated) $\OK$-submodule of $\Eh \cap E_\R^\hilb.$

Similarly, by the commutativity of (\ref{bigCD}), the images of $P$ by $\hat{q}$ and by $q_\R$ coincide and define a (finitely generated) $\OK$-submodule of $\Gh \cap G_\R^\hilb.$

Let $\overline{i^{-1}(P)}$, $\overline{P}$ and $\overline{q(P)}$ the Hermitian vector bundles over $\Spec \OK$ defines by $i^{-1}(P)$, $P$ and $q(P)$ equipped with the restrictions of the Hermitian structures on $\Ebh$, $\Fbh,$ and $\Gbh$. Let $\overline{q(P)}^{\rm quot}$ denotes the Hermitian vector bundle over $\Spec \OK$ defined by $q(P)$ equipped with the Hermitian structures quotient, via the map $q: P \lra q(P)$, of  the ones on $\overline{P}$.

By construction,
\begin{equation}\label{Pqadm}
0 \lra \overline{i^{-1}(P)} \stackrel{i}{\lra} \overline{P} \stackrel{q}{\lra} \overline{q(P)}^{\rm quot} \lra 0
\end{equation}
is an admissible short exact sequence of Hermitian vector bundles over $\Spec \OK$, and accordingly, by Proposition \ref{ineqshorttheta}, we have
\begin{equation}\label{Pquot}
\hot(\overline{P}) \leq \hot(\overline{i^{-1}(P)}) + \hot(\overline{q(P)}^{\rm quot}).
\end{equation}

Besides the map $q: \overline{P} \lra \overline{q(P)}$ has operator norms $\leq 1$, and therefore the identity map $Id_{q(P)}: \overline{q(P)}^{\rm quot} \lra \overline{q(P)}$ also. Accordingly, by Proposition \ref{ineqmortheta}, 1), we have
\begin{equation}\label{qPquot}
\hot(\overline{q(P)}^{\rm quot}) \leq \hot(\overline{q(P)}).
\end{equation}

From (\ref{Pquot}) and (\ref{qPquot}), we derive:
\begin{align*}
\hot(\overline{P}) &  \leq \hot(\overline{i^{-1}(P)}) + \hot(\overline{q(P)})\\
                           &  \leq \lhot(\Ebh) + \lhot(\Gbh).
\end{align*}
This proves (\ref{subaddlhotpro}).

(ii) Let $(U_j)_{j \in \N}$ (resp. $(W_k)_{k \in \N}$) be a defining filtration $\cU(\Eh)$ (resp. in $\cU(\Gh)$) such that
\begin{equation}\label{Uj}
\lim_{j\ra +\infty} \hot(\Eb_{U_j}) = \uhot(\Ebh)
\end{equation}
and
\begin{equation}\label{Wk}
\lim_{k \ra +\infty} \hot(\Eb_{W_k}) = \uhot(\Gbh).
\end{equation}

Let us choose $\epsilon$ a positive real number. There exists $j_0$ (resp. $k_0$) in $\N$ such that, for any integer $j\geq j_0$ (resp. $k\geq k_0$), we have 
\begin{equation}\label{Ujbis}
 \hot(\Eb_{U_j}) \leq \uhot(\Ebh) + \epsilon/3
\end{equation}
and
\begin{equation}\label{Wkbis}
\hot(\Gb_{W_k}) \leq \uhot(\Gbh) + \epsilon/3.
\end{equation}
Moreover, according to (\ref{UV3}), for any $j\in \N,$ there exists $k(j) \in \N$ such that, for any $k \in \N_{\geq k(j)},$
\begin{equation}\label{UjWk}
 \hot(\Eb_{U_j}^{W_k}) \leq \hot(\Ebh_{U_j}) + \epsilon/3.
\end{equation}

Together with (\ref{UV1}), 
these three estimates (\ref{Ujbis})--(\ref{UjWk}) show that, for any $(j,k) \in \N^2,$
$$j \geq j_0 \mbox{ and } k \geq \max(k_0, k(j)) \Longrightarrow
\hot(\Fb_{\hat{i}(U_j) + \hat{s}(W_k)}) \leq \hot(\Eb) + \hot(\Gb)+ \epsilon.$$
Since the submodules $\hat{i}(U_j) + \hat{s}(W_k)$ constitute a basis of neighborhoods of $0$ in $\cU(\Fb)$, this shows that
$$\liminf_{V \in \cU(\Fh)} \hot(\Fb_V) \leq \hot(\Eb) + \hot(\Gb)+ \epsilon.$$
As $\epsilon$ is arbitrary in $\Rpa$, this establishes (\ref{subadduhotpro}).

(iii) Let us recall that $\hut(\Ebh_U)$ (resp. $\hut(\Gbh_W)$) is a non-increasing function of $U$ in $\cU(\Eh)$ (resp. of  $W$ in $\cU(\Gh)$) and that 
\begin{equation}\label{hutEG}
\hut(\Eb) = \lim_{U\in\cU(\Eh)} \hut(\Eb_{U}) \mbox{ (resp. } 
\hut(\Gb) = \lim_{W\in\cU(\Gh)} \hut(\Gb_{W})).
\end{equation}
Similarly $\hut(\Fb_{\hat{i}(U) + \hat{s}(W)})$ is a non-increasing function of $U$ in $\cU(\Eh)$ and  $W$ in $\cU(\Gh),$ and 
\begin{equation}\label{hutF}
\hut(\Fb) = \lim_{U\in\cU(\Eh), V\in\cU(\Gh)}\hut(\Fb_{\hat{i}(U) + \hat{s}(W)}).
\end{equation}

The estimates (\ref{UV1BIS}) and (\ref{UV2BIS}) show that, for any $(U,W)$ in $\cU(\Eh) \times \cU(\Gh),$
\begin{equation}\label{UWeak}
\hut(\Fb_{\hat{i}(U) + \hat{s}(W)}) \leq \hut(\Eb_U) + \hut(\Gb_W).
\end{equation}

The estimate (\ref{subaddhutpro})
$$\hut(\Fbh) \leq \hut(\Ebh) + \hut(\Gbh)$$
directly follows from this estimate, together with the  expressions (\ref{hutEG}) and (\ref{hutF}) for $\hut(\Ebh),$ $\hut(\Gbh)$, and $\hut(\Fbh)$.

2) Let us now assume that $\Ebh$ is a Hermitian vector bundle $\Eb.$

 Then $i(E)$ is contained in $\Fh \cap F_\R^\hilb$. Therefore, with the notation of 1) (i), $\lhot(\Fb)$ is the supremum of the $\hot(\overline{P})$, where $P$ denotes a finitely generated $\OK$-submodule of 
 $\Fh \cap F_\R^\hilb$ containing  $i(P).$

For any such $P$, its image $q(P)$ is contained in $\Gh \cap G_\R^\hilb$ and the admissible short exact sequence (\ref{Pqadm}) takes the form: 
\begin{equation}\label{EPadm}
0 \lra \Eb \stackrel{i}{\lra} \overline{P} \stackrel{q}{\lra} \overline{q(P)} \lra 0.
\end{equation}

Conversely, for any finitely generated $\OK$-submodule $Q$ of $\Gh \cap G_\R^\hilb$, its inverse image by $q$, $P:= q^{-1}(Q),$ is a finitely generated $\OK$-submodule  $\Fh \cap F_\R^\hilb$ containing  $i(P)$ such that 
$Q = q(P).$

The inequality (\ref{morse4}) (or equivalently (\ref{morse4'})) applied to (\ref{EPadm}) takes the form:
$$\hot(\overline{q(P)}) \leq \hot(\overline{P}) - \dega \pi_\ast \Eb.$$
From this inequality, by taking the supremum over the submodules $P$ as above, we obtain  (\ref{subaddproEfinitel}).

3) Let us now assume that $\Gbh$ is a Hermitian vector bundle $\Gb$.

For any $U \in \cU(\Eh),$ from the admissible short exact sequence of pro-Hermitian vector bundles
\begin{equation*}
0 \lra \Ebh \stackrel{i}{\lra} \Fbh \stackrel{q}{\lra} \Gb \lra 0, 
\end{equation*}
we derive an admissible short exact sequence of Hermitian vector bundles:
\begin{equation}\label{shortU}
0 \lra \Eb_U {\lra} \Fbh_{\hat{i}(U)} {\lra} \Gbh \lra 0. 
\end{equation}

From (\ref{shortU}), we derive the relations
$$\hot(\Eb_U) \leq \hot(\Fbh_{\hat{i}(U)})$$
(by Proposition \ref{ineqmortheta}, 1))   and 
$$ \hot(\Fbh_{\hat{i}(U)}) - \hut(\Fbh_{\hat{i}(U)}) =  \hot(\Eb_U) - \hut(\Eb_U) + \dega \pi_\ast \Gb$$
(by the additivity of the Arakelov degree (\ref{degaddadm}) and the Poisson-Riemann-Roch formula (\ref{PoissonZ})).

Thus we obtain:
$$\hut(\Eb_U) \leq \hut(\Fbh_{\hat{i}(U)}) + \dega \pi_\ast \Gb,$$
and the estimate (\ref{subaddprohut}) follows by taking the limit when $U \in \cU(\Eh)$ shrinks to $\{0\}.$

\section{Short exact sequences and strongly summable pro-Hermitian vector bundles}\label{ShortStrong}

\subsection{} In this section, we establish the following permanence properties of strongly summable pro-Hermitian vector bundles, which may be seen as refinements of Corollary  \ref{finite rankcor}:

\begin{theorem}\label{permStS} Consider an admissible short exact sequence of pro-Hermitian vector bundles over the arithmetic curve $\Spec \OK$:
\begin{equation}
0 \lra \Ebh \stackrel{i}{\lra} \Fbh \stackrel{q}{\lra} \Gbh \lra 0.
\end{equation}

 1) When $\Ebh$ has finite rank, $\Fbh$ is strongly summable if and only if $\Gbh$ is strongly summable.
 
 2) When $\Gbh$ has finite rank, $\Fbh$ is strongly summable if and only if $\Ebh$ is strongly summable.
 \end{theorem}
 
\begin{corollary}\label{CorpermitStS} Consider a short exact sequence of pro-Hermitian vector bundles over the arithmetic curve $\Spec \OK$:
\begin{equation}\label{encoreshort}
0 \lra \Ebh \stackrel{i}{\lra} \Fbh \stackrel{q}{\lra} \Gbh \lra 0.
\end{equation}

1) When $\Ebh$ has finite rank, $\Fbh$ is $\theta$-finite  if and only if $\Gbh$ is $\theta$-finite.
 
 2) When $\Gbh$ has finite rank, $\Fbh$ is $\theta$-finite  if and only if $\Ebh$ is $\theta$-finite. 
 \end{corollary}
 
 The fact that $\Ebh$ is $\theta$-finite  when $\Fbh$ is $\theta$-finite is actually a special case of Proposition \ref{permStTetc}, 2).

\begin{proof}[Proof of Corollary \ref{CorpermitStS}] For any choice, compatible with complex conjugation, of Hilbert norms $\Vert . \Vert'_{F, \sigma}$ on the Hilbert spaces $F_\sigma^\hilb$, equivalent to the Hilbert norms $\Vert . \Vert_{F, \sigma}$ defining $\Fbh,$ there exists Hilbert norms, compatible with complex conjugation,  $\Vert . \Vert'_{E, \sigma}$ and $\Vert . \Vert'_{G, \sigma}$ on the Hilbert spaces $E_\sigma^\hilb$ and $G_\sigma^\hilb$, such that the diagram (\ref{encoreshort}) becomes an \emph{admissible} short exact sequence 
\begin{equation}\label{encoreshortprime}
0 \lra \Ebh' \stackrel{i}{\lra} \Fbh' \stackrel{q}{\lra} \Gbh' \lra 0.
\end{equation}
relating the pro-Hermitian vector bundles $$\Ebh':=(\Eh, (E_\sigma^\hilb, \Vert . \Vert'_{E,\sigma})_{\sigma: K \hra \C}),$$ $$\Fbh':=(\Fh, (F_\sigma^\hilb, \Vert . \Vert'_{F,\sigma})_{\sigma: K \hra \C}),$$ and $$\Gbh':=(\Gh, (G_\sigma^\hilb, \Vert . \Vert'_{\sigma})_{G, \sigma: K \hra \C}).$$
Namely, the norms $\Vert . \Vert'_{E, \sigma}$ and $\Vert . \Vert'_{G, \sigma}$ induced and quotient of the norms $\Vert . \Vert'_{F, \sigma}$.

Conversely, for any choice, compatible with complex conjugation,  of Hilbert norms $\Vert . \Vert'_{E, \sigma}$ and $\Vert . \Vert'_{G, \sigma}$ on $E_\sigma^\hilb$ and $G_\sigma^\hilb$,  equivalent to the Hilbert norms $\Vert . \Vert_{E, \sigma}$ and $\Vert . \Vert_{G, \sigma}$ defining $\Ebh$ and $\Gbh$, there exists Hilbert norms $\Vert . \Vert'_{F, \sigma}$ on the Hilbert spaces $F_\sigma^\hilb$, equivalent to the Hilbert norms $\Vert . \Vert_{F, \sigma}$ defining $\Fbh,$ such that (\ref{encoreshortprime}) is an admissible short exact sequence. 

Therefore 1) (resp. 2)) follows from Theorem   \ref{permStS}, 1) (resp. 2))  applied to the admissible short exact sequences  (\ref{encoreshortprime}), together with the characterization of $\theta$-finite pro-Hermitian  by Condition $\mathbf{\theta-Fin_1}$ in Corollary \ref{thetafinite}. 
\end{proof}
 
\subsection{Proof of Theorem \ref{permStS}: I. Preliminary}

The following lemma gathers various facts that will be needed in the proof of part 1) of Theorem \ref{permStS}.

\begin{lemma}\label{PrelimEncore} Consider an admissible short exact sequence of pro-Hermitian vector bundles over $\Spec \OK,$
 \begin{equation*}
0 \lra \Eb \stackrel{i}{\lra} \Fbh \stackrel{q}{\lra} \Gbh \lra 0,
\end{equation*}
with $\Eb$ of finite rank \emph{(or, equivalently, a Hermitian vector bundle over $\Spec \OK$).}

Let $(V_k)_{k \in \N}$ be a defining sequence in $\cU(\Fh)$ and let $(W_k)_{k \in \N}:=(\hat{q}(V_k))_{k \in \N}$ be its image by $\hat{q}.$ 

1) There exists $k_0 \in \N$ such that, for any $k \in \N_{\geq k_0},$
$$V_k \cap \hat{i}(E) =\{0\}$$
and the image of $\hat{i}(E)$ in $\Fh/V_k$ is saturated. 
Moreover, for any such integer $k_0$, the sequence $(W_k)_{k \geq k_0}$ is a defining sequence in $\cU(\Gh)$.

2) For any $k \in \N,$ let us consider the Hermitian vector bundle $\Fb_k:= \Fb_{V_k}$ and  the admissible surjective morphism
$$r_k:=p_{V_k V_{k+1}}: \Fb_{k+1} \lra \Fb_k.$$ Similarly, for any $k \in\N_{\geq k_0},$ let us consider $\Gb_k:= \Gb_{W_k}$ and 
$$s_k:=p_{W_k W_{k+1}}: \Gb_{k+1} \lra \Gb_k.$$

For any $k \in \N_{k \geq k_0},$ the map $\hat{q}$ induces  a map
$$q_k: F_k:= \Fh/V_k \lra G_k:= \Gh/W_k,$$
which fits into an exact sequence:
 \begin{equation*}
0 \lra E \stackrel{i_k}{\lra} F_k \stackrel{q_k}{\lra} G_k \lra 0,
\end{equation*}
where $i_k:= p_{V_k} \circ i.$

Moreover, $q_k$ belongs to $\Hom_\OK^{\leq 1} (\Fb_k, \Gb_k)$.

3) For any $k\in \N_{\geq k_0},$ the map $q_k$ defines an isomorphism
$$\phi_k := q_{k \mid \ker q_k} : \ker r_k \lrasim \ker s_k.$$
Moreover $\phi_k$ belongs to $\Hom_\OK^{\leq 1} (\overline{\ker r_k}, \overline{\ker s_k})$ and, for any $\eta \in \R_+^\ast,$ there exists $k(\eta) \in \N_{\geq k_0}$ such that, for any $k \in \N_{\geq k(\eta)},$ $\phi_k^{-1}$ belongs to  
$\Hom_\OK^{\leq 1} (\overline{\ker s_k}, \overline{\ker r_k}\otimes \overline{\cO(\eta)})$.
\end{lemma}

The proof of the last part of Lemma \ref{PrelimEncore} will rely on the following proposition concerning the geometry of sub-quotients of Hilbert spaces.

\begin{proposition}\label{Hilbertencore}
  Let $(H, \Vert . \Vert)$ be a (real of complex) Hilbert space and let
  $$H_0 \supseteq H_1 \supseteq H_2 \supseteq H_3 \supseteq \ldots \supseteq H_n \supseteq H_{n+1}  \ldots$$
 be a decreasing sequence of closed vector subspaces of $H$ such that 
 $$\bigcap_{n \in \N} H_n = \{0\}.$$
 
 Let $P$ be a \emph{(necessarily closed)} finite dimensional vector subspace of $H$.
 
 1) There exists $n_0 \in \N$ such that, for any $n$ in $\N_{\geq n_0}$,
 $$H_n \cap P = \{0\}.$$
 Then, for any such integer $n \geq \N_{\geq n_0}$, the linear map 
 $$\phi_n : H_n/H_{n+1} \lra (H_n + P)/(H_{n+1} + P),$$
 induced by the inclusion $H_n \hlra H_n+ P$, is bijective.
 
 2) For every $n\in \N$, let us equip $H_n/H_{n+1}$ (resp. $(H_n + P)/(H_{n+1} + P)$) with the norm quotient of the norm $\Vert. \Vert$ on $H_n$ (resp. on $H_n +P$)\footnote{This is indeed a norm since $H_{n+1}$ (resp. $H_{n+1} + P$) is closed in $H_{n}$ (resp. $H_{n} + P$).}. Then, for any $n\in \N$,
 \begin{equation}\label{phin}
\Vert \phi_n \Vert := \sup_{x\in H_n/H_{n+1}, \Vert x \Vert \leq 1} \Vert \phi_n(x) \Vert \leq 1.
\end{equation}
Moreover the sequence $(\Vert \phi_n^{-1} \Vert)_{n \geq n_0}$, defined
by
$$\Vert \phi_n^{-1} \Vert := \sup_{x\in (H_n+P)/(H_{n+1}+P), \Vert x \Vert \leq 1} \Vert \phi_n^{-1}(x) \Vert,$$
satisfies
\begin{equation}\label{phininv}
\lim_{n \lra \infty} \Vert \phi_n^{-1} \Vert =1.
\end{equation}
\end{proposition}

\begin{proof}[Proof of Proposition \ref{Hilbertencore}] Assertion 1) is clear, and the upper bound (\ref{phin}) also.

For every $n\in \N$, let us denote by 
$$\pi_n: H \lra H_n$$
the orthogonal projection onto $H_n$, and let $p: H \lra P$ denote the orthogonal projection onto $P$.

Since $\bigcap_{n \in \N} H_n = \{0\},$ for every $x \in H,$
$$\lim_{n \lra +\infty} \Vert \pi_n(x)\Vert = 0.$$
As $P$ is finite dimensional, the operator norm of the restriction to $P$ of $\pi_n$,  
$$\Vert \pi_{n \mid P} \Vert := \max_{x \in P, \Vert x \Vert \leq 1} \Vert \pi_n(x) \Vert,$$
converges also to $0$ when $n$ goes to $+\infty$.

The restriction to $H_n$ of $p$,
$$p_{\mid H_n}: H_n \lra P$$
is the adjoint of $$\pi_{n \mid P}: P \lra H_n,$$
and therefore its operator norm
$$\Vert p_{\mid H_n} \Vert = \Vert \pi_{n \mid P} \Vert $$
converges also to $0$. This easily implies (\ref{phininv}).
\end{proof}

\begin{proof}[Proof of Lemma \ref{PrelimEncore}] Assertion 1) easily follows from 
Propositions \ref{injPN} and \ref{finitesaturated}.

Assertion 2) is a straightforward consequence of the definitions.

Assertion 3) follows from Proposition \ref{Hilbertencore} applied to $H= F_\sigma^\hilb,$ $H_n = V_{n,\sigma}\cap F_\sigma^\hilb,$ and $P = i_\sigma(E_\sigma)$. 
 \end{proof}

\subsection{Proof of Theorem \ref{permStS}: II. Completion of the proof} 
$\,$

1) Let us assume that $\Ebh$ has finite rank or, in other words, that it is defined by some Hermitian vector bundle $\Eb$ over $\Spec \OK.$ Then the conclusions of Lemma \ref{PrelimEncore} hold, and we shall use its notation.

 From Part 3) of Lemma  \ref{PrelimEncore}, it follows that, for any $\epsilon \in \R$ and any $k \in \N_{k \geq k_0},$
\begin{equation}\label{rksk}
\hot(\overline{\ker r_k} \otimes \overline{\cO(\epsilon)}) \leq \hot(\overline{\ker s_k} \otimes \overline{\cO(\epsilon)})
\end{equation}
and that, for any $(\epsilon, \epsilon') \in \R^2$ and any $k \in \N,$
\begin{equation}\label{skrk}
\epsilon'< \epsilon \mbox{ and } k \geq k(\epsilon -\epsilon') \Longrightarrow 
\hot(\overline{\ker s_k} \otimes \overline{\cO(\epsilon')}) \leq \hot(\overline{\ker r_k} \otimes \overline{\cO(\epsilon)}).
\end{equation}

Let us assume that $\Fbh$ is strongly summable. Then $\Fbh$ satisfies $\mathbf{StS_2}$ and we may choose a defining sequence $(V_k)_{k \in \N}$ in $\cU(\Fh)$ and $\epsilon >0$ such that
$$\sum_{k \in \N} \hot(\overline{\ker r_k} \otimes \overline{\cO(\epsilon)}) < + \infty.$$
For any $\epsilon' \in]0, \epsilon[,$ the upper bound  (\ref{skrk}) shows that
$$\sum_{k \in \N} \hot(\overline{\ker s_k} \otimes \overline{\cO(\epsilon')}) < +\infty.$$
Consequently $\Gbh$ also satisfies $\mathbf{StS_2}$ and is therefore strongly summable.

Conversely, let us assume that $\Gbh$ is strongly summable, and let us choose a defining sequence $(W_k)_{k \in \N}$ in $\cU(\Gh)$ and $\epsilon \in \R^\ast_+$ such that,
if we let $\Gb_k := \Gb_{W_k}$, the admissible surjective morphisms
 $$s_k := p_{W_k W_{k+1}}: \Gb_{k+1} \lra \Gb_k$$
 satisfy:
 $$\sum_{k \in \N} \hot(\overline{\ker s_k}\otimes \overline{\cO(\epsilon)}) < + \infty.$$
 
  There exists a continuous $\OK$-linear splitting $\hat{s}: \Gh \lra \Fh$ of the short exact sequence 
 $$0 \lra \Eh \stackrel{\hat{i}}{\lra} \Fh \stackrel{\hat{q}}{\lra} \Gh \lra 0$$
 (sees \ref{shortpro}, Remark (ii)), and we may define, for every $k \in \N,$ 
 $$V_k := \hat{s}(W_k).$$
 Then $W_k = \hat{q}(V_k),$ and we may apply Lemma \ref{PrelimEncore}. Notably, the estimate (\ref{rksk}) imply that
 $$\sum_{k \in \N} \hot(\overline{\ker r_k} \otimes \overline{\cO(\epsilon)}) \leq \sum_{k \in \N} \hot(\overline{\ker s_k} \otimes \overline{\cO(\epsilon)}) < +\infty.$$
 Therefore satisfies $\mathbf{StS_2}$ and is  strongly summable.
 
 2) The morphism $i$ belongs to $\Hom_\OK^{\leq 1} (\Ebh, \Fbh)$ and $\hat{i}:\Eh \lra \Fh$ is a strict morphism of topological $\OK$-modules. Consequently, according to Proposition \ref{permStTetc}, if $\Fbh$ is strongly summable, then $\Ebh$ is strongly summable.
 
 Conversely, let us assume that $\Ebh$ is strongly summable and that $\Gbh$ has finite rank, or equivalently, that $\Gbh$ is defined by some Hermitian vector bundle $\Gb$ over $\Spec \OK$.
 
 Then $\Ebh$ satisfies $\mathbf{StS_2}$ and we may choose a defining sequence $(U_k)_{k \in \N}$ in $\cU(\Eh)$ and $\epsilon >0$ such that, if we let $\Eb_k := \Eb_{U_k}$, the admissible surjective morphisms
 $$q_k := p_{U_k U_{k+1}}: \Eb_{k+1} \lra \Eb_k$$
 satisfy:
 $$\sum_{k \in \N} \hot(\overline{\ker q_k}\otimes \overline{\cO(\epsilon)}) < + \infty.$$
 
Let us consider the sequence $(V_k)_{k\in \N}:=(\hat{i}(U_k))_{k\in \N}$. As
 $$0 \lra \Eh \stackrel{\hat{i}}{\lra} \Fh \stackrel{\hat{q}}{\lra} G \lra 0$$
 is a strict short exact sequence in $\CTC_{\OK}$ and $G$ is a discrete $\OK$-module, it is a defining sequence in $\cU(\Fh).$ 
 
 Moreover, if we let $\Fb_k := \Fb_{V_k}$ and if we introduce the admissible surjective morphisms
 $$r_k := p_{V_k V_{k+1}}: \Fb_{k+1} \lra \Fb_k,$$
 it is straightforward that the morphisms 
 $$i_k: E_k := \Eh/V_k \lra F_k /U_k$$
 induced by $\hat{i}$  fit into commutative diagrams of $\OK$-module with exact lines:
 \begin{equation*}\label{CDEFGkk+1}
\begin{CD}
 0 @>>> E_{k+1} @>i_{k+1}>> F_{k+1} @>>> G @>>> 0 \\
 @.          @VVq_kV                @VVr_kV                 @VVId_GV       @. \\
0 @>>> E_{k} @>i_{k}>> F_{k} @>>> G @>>> 0 .
\end{CD}
\end{equation*}
Therefore, for any $k \in \N,$ $i_{k+1}$ defines an isomorphism
$$i_{k+1} : \ker q_k \lra \ker r_k,$$
that is easily seen to be an isometric isomorphism of Hermitian vector bundles:
$$i_{k+1}: \overline{\ker q_k} \lrasim \overline{\ker r_k}.$$
Consequently, 
$$\hot(\overline{\ker r_k}\otimes \overline{\cO(\epsilon)}) = \hot(\overline{\ker q_k}\otimes \overline{\cO(\epsilon)}).$$

Finally, 
 $$\sum_{k \in \N} \hot(\overline{\ker r_k}\otimes \overline{\cO(\epsilon)}) = \sum_{k \in \N} \hot(\overline{\ker q_k}\otimes \overline{\cO(\epsilon)}) < + \infty$$
and $\Fbh$ also satisfies $\mathbf{StS_2}$. 

\section{A vanishing criterion}\label{VanCrit}
In this section, we apply the properties of $\theta$-finite pro-Hermitian vector bundles and of their $\theta$-invariants previously established to derive a criterion for the vanishing of the ``connecting maps" arising in some systems of short exact sequences of pro-Hermitian vector bundles.

This criterion has a rather technical formulation. However it will play an important role in the Diophantine applications of the formalism of pro-Hermitian vector bundles. In these applications, the vanishing criterion will be applied to pro-Hermitian vector bundles defined by the coherent cohomology groups of some formal  and complex analytic vector bundles, and the connecting maps (denoted by $\delta_k$ in the Notation below) will be the connecting maps between cohomology groups associated to short exact sequences of formal and complex analytic vector bundles.

\subsection{Notation}\label{NotVan} Let $k_0$ be some integer, and let us consider the following data:
\begin{enumerate}
\item \emph{a sequence $(\Hbh_k)_{k \geq k_0}$ of pro-Hermitian vector bundles over $\Spec \OK$ and a sequence $(p_k)_{k > k_0}$ of morphisms
$$p_k: \Hbh_{k-1} \lra \Hbh_k$$
in $\Hom_\OK^{\leq 1} (\Hbh_{k-1},\Hbh_k)$;
\item a sequence $(\Fb_k)_{k > k_0}$ of \emph{(finite rank)} Hermitian vector bundles and a sequence $(\delta_k)_{k > k_0}$ of ``connecting maps"  $$\delta_k : \Fb_k \lra \Hbh_{k-1}$$ in $\Hom_\OK^{\rm cont}(\Fb_k, \Hbh_{k-1})$.}
\end{enumerate}

We shall assume that, \emph{for every $k \in \N_{> k_0}$, the diagram}
\begin{equation}\label{FHexact}
\Fb_k \stackrel{\delta_k}{\lra} \Hbh_{k-1} \stackrel{p_k}{\lra} \Hbh_k \lra 0
\end{equation}
\emph{is an exact sequence of pro-Hermitian vector bundles over $\Spec \OK$,} in the sense that for any  $k$ in $\N_{> k_0}$, the diagrams
\begin{equation}\label{Exhat}
F_k \stackrel{\hat{\delta}_k}{\lra} \Hh_{k-1} \stackrel{\hat{p}_k}{\lra} \Hh_k \lra 0
\end{equation}
and, for every $\sKC$,
\begin{equation}\label{Exsigma}
F_{k,\sigma} \stackrel{\delta_{k,\sigma}}{\lra} H_{k-1,\sigma} \stackrel{p_{k,\sigma}}{\lra} H_{k,\sigma} \lra 0
\end{equation}
are exact sequences of $\OK$-modules and of $\C$-vector spaces, respectively.

Observe that, for every $k \in \N_{> k_0}$, we may introduce the Hermitian vector bundle $\overline{F_k/\ker \delta_k}$ over $\Spec \OK$, so that the morphism $\delta_k$ factorizes through $\overline{K/\ker \delta_k}$: 
$$\delta_k = \tilde{\delta}_k \circ \pi_k : \Fb_k \stackrel{\pi_k}{\lra}  \overline{F_k/\ker \delta_k} \stackrel{\tilde{\delta_k}}{\lra} \Hbh_{k-1}.$$
The exactness of  (\ref{FHexact}) is then easily seen to be equivalent to the fact that the diagram
\begin{equation}\label{FHexactbis}
0 \lra \overline{F_k/\ker \delta_k} \stackrel{\tilde{\delta}_k}{\lra}  \Hbh_{k-1} \stackrel{p_k}{\lra} \Hbh_k \lra 0
\end{equation}
is  a short exact sequence of pro-Hermitian vector bundles over $\Spec \OK$.

For any $k \in \N_{\geq k_0},$ we shall also consider the morphism
$$q_k : \Hbh_{k_0} \lra \Hbh_k$$
in $\Hom_\OK^{\leq 1} (\Hbh_{k-1},\Hbh_k)$ defined by composing the morphisms $p_j$:
$$
\begin{array}{cll}
  q_k & := id  & \mbox{ if $k=k_0$}  \\
        &:= p_k \circ \cdots \circ p_{k_0 +1}   & \mbox{ if $k>k_0$.}   
\end{array}
$$

For any $k \in \N_{\geq k_0}$ and any embedding $\sigma: K \hra \C,$ the $\C$-linear map
$$q_{k, \sigma} : H^\hilb_{k_0,\sigma} \lra H^\hilb_{k,\sigma}$$
is norm decreasing and surjective. We shall denote by
$$\tilde{q}_{k,\sigma}: H^\hilb_{k_0,\sigma}/\ker q_{k,\sigma} \lra H^\hilb_{k,\sigma}$$
the induced isomorphism of topological vector spaces.

\subsection{The conditions $\mathbf{Fin }$, $\mathbf{Amp }$, $\mathbf{An_1 }$ and $\mathbf{An_2}$} With the above notation, we may introduce the following conditions on the (pro)-Hermitian vector bundles $\Hbh_k$ and $\Fb_k$:

\smallskip

$\mathbf{Fin }$ ($\theta$-finiteness) : \emph{For every $k \in \N_{\geq k_0},$ the pro-Hermitian vector bundle $\Hbh_k$ is $\theta$-finite.} 

According to Corollary \ref{encoreshort} and to the exactness of (\ref{FHexactbis}), this condition is satisfied as soon as $\Hbh_k$ is $\theta$-finite for \emph{some} $k \in \N_{\geq k_0}.$

\smallskip

$\mathbf{Amp }$ (asymptotic ampleness): $$\liminf_{k \ra + \infty} \frac{1}{k} \, \mua_{\rm min}(\Fb_k) > 0.$$
    
Recall that, for every $k\in \N_{k>k_0},$ $p_k$ and $q_k$ are morphisms in ${\rm pro\overline{Vect}}^{\leq 1}(\OK)$. Therefore, for every embedding $\sKC,$ the operator norms $\Vert p_{k,\sigma}\Vert$ (resp. $\Vert q_{k,\sigma}\Vert$) of the continuous $\C$-linear map $p_{k,\sigma}$ (resp. $q_{k,\sigma}$) from the Hilbert space $(H_{k-1,\sigma}^\hilb, \Vert.\Vert_{\Hbh_{k-1}, \sigma})$ (resp. $(H_{k_0,\sigma}^\hilb, \Vert.\Vert_{\Hbh_{k_0}, \sigma})$) to the Hilbert space $(H_{k,\sigma}^\hilb, \Vert.\Vert_{\Hbh_k, \sigma})$ satisfies:
$$\Vert p_{k,\sigma}\Vert \leq 1 \mbox{ (resp. $\Vert q_{k,\sigma}\Vert \leq 1$)}.$$

We shall also consider the following conditions on the operator norms of the morphisms $\delta_{k,\sigma}$ and $\tilde{q}_{k,\sigma}^{-1}$:  

\smallskip

$\mathbf{An_1 :}$ \emph{For every embedding $\sKC,$ the operator norm $\Vert \delta_{k,\sigma}\Vert$ of the  $\C$-linear map $\delta_{k,\sigma}$ from the normed space $(F_{k,\sigma}, \Vert. \Vert_{\Fb_k, \sigma})$ to the Hilbert space $(H^\hilb_{k-1,\sigma}, \Vert.\Vert_{\Hbh_{k-1},\sigma})$ satisfies}
$$\limsup_{k \ra +\infty} \frac{1}{k} \log \Vert \delta_{k,\sigma} \Vert \leq 0.$$

This condition may be rephrased as the existence, for every $\epsilon \in \Rpa,$ of $C_\epsilon \in \Rpa$ such that, for any $k > k_0$ and any embedding $\sKC,$
\begin{equation}\label{An1bis}
\Vert \delta_{k,\sigma} \Vert \leq C_\epsilon e^{\epsilon k}.
\end{equation}

\smallskip

$\mathbf{An_2 :}$ \emph{For every embedding $\sKC,$ the operator norm $\Vert \tilde{q}_{k,\sigma}^{-1}\Vert$ of the  continuous $\C$-linear map $\tilde{q}_{k,\sigma}^{-1}$ from the Hilbert space $(H^\hilb_{k,\sigma}, \Vert.\Vert_{\Hbh_{k},\sigma})$ to the Hilbert space $H^\hilb_{k_0,\sigma}/\ker q_{k,\sigma}$, equipped with the quotient norm deduced from  $\Vert.\Vert_{\Hbh_{k_0},\sigma}$, satisfies}
$$\limsup_{k \ra +\infty} \frac{1}{k} \log \Vert \tilde{q}_{k,\sigma}^{-1} \Vert \leq 0.$$

Similarly to  Condition $\mathbf{An_1},$ Condition $\mathbf{An_2}$
 may be rephrased as the existence, for every $\epsilon \in \Rpa,$ of $C_\epsilon \in \Rpa$ such that, for any $k > k_0$ and any embedding $\sKC,$
\begin{equation}\label{An2bis}
\Vert \tilde{q}_{k,\sigma}^{-1} \Vert \leq C'_\epsilon e^{\epsilon k}.
\end{equation}

\subsection{A vanishing criterion}

\begin{theorem}\label{ConnectV} Let us consider some sequences of (pro)-Hermitian vector bundles $(\Hbh_k)_{k \geq k_0}$ and $(\Fb_k)_{k > k_0}$ of and of morphisms  $(p_k)_{k > k_0}$ and $(\delta_k)_{k > k_0}$ in ${\rm pro\overline{Vect}}^{\rm cont}(\OK)$ as in \ref{NotVan} above.

If the conditions $\mathbf{Fin }$, $\mathbf{Amp }$, $\mathbf{An_1 }$ and $\mathbf{An_2}$ are satisfied, then  there exists $k_1 \in \N_{> k_0}$ such that the connecting map $\delta_k$ vanishes for any $k \in \N_{\geq k_1}$, or equivalently, such that, for any $k \in \N_{\geq k_1}$, 
$$p_k: \Hbh_{k-1} \lra \Hbh_k$$
is an isomorphism in ${\rm pro\overline{Vect}}^{\rm cont}(\OK)$.
 \end{theorem}

\begin{proof} For every $k \in \N_{\geq k_0}$ and every embedding $\sKC$, we may consider the kernel of 
$$\hat{q}_k : \Hh_{k_0} \lra \Hh_k$$
and the one of
$$q_{k,\sigma} : H_{k_0,\sigma} \lra H_{k,\sigma}.$$

They define a non-decreasing sequence of $\OK$-submodules of $\Hh_{k_0}$:
$$\ker \hat{q}_{k_0} = 0 \hra \ker \hat{q}_{k_0 +1} \hra \ldots \hra \ker \hat{q}_k \hra \ker \hat{q}_{k+1} \hra \ldots,$$
and a non-decreasing sequence of $\C$-vector subspaces of $H^\hilb_{k_0,\sigma}$:
$$\ker {q}_{k_0,\sigma} = 0 \hra \ker {q}_{k_0 +1,\sigma} \hra \ldots \hra \ker {q}_{k,\sigma} \hra \ker {q}_{k+1,\sigma} \hra \ldots.$$

The first assertion in the following lemma is a straightforward consequence of the exactness of the diagrams (\ref{Exhat}) and (\ref{Exsigma}). The second one then follows by induction on the integer $k$.

\begin{lemma}\label{ConnectV1}
1) For any $k \in \N_{> k_0}$ and for any embedding $\sKC,$ the maps $\hat{q}_{k-1}$ and $q_{k-1,\sigma}$ define isomorphisms
$$\ker \hat{q}_k/\ker \hat{q}_{k-1} \lrasim \ker \hat{p}_k = \im \hat{\delta}_k$$
and
$$\ker {q}_{k,\sigma}/\ker {q}_{k-1,\sigma} \lrasim \ker {p}_{k,\sigma} = \im {\delta}_{k,\sigma}$$
 of $\OK$-modules and $\C$-vector spaces, respectively.
 
 2) For any $k\in \N_{k\geq k_0},$ the $\OK$ module $\ker \hat{q}_k$ is finitely generated and projective and, for any embedding $\sigma: K \hra \C,$ the complex vector subspace $(\ker \hat{q}_k)_\sigma$ of $\Hh_{k_0,\sigma}$ coincides with $\ker {q}_{k,\sigma}$ (and notably is contained in $H^\hilb_{k_0,\sigma}$).   \qed
\end{lemma}

For any $k\in \N_{\geq k_0},$ we shall define the Hermitian vector bundle $\overline{\ker} q_k$ over $\Spec \OK$ as $\ker \hat{q}_k$ equipped with the Hermitian structure induced by the Hilbert spaces $(H^\hilb_{k_0,\sigma}, \Vert.\Vert_{\Hbh_{k_0},\sigma})$.

 For any $k \in \N_{>k_0},$ we let $n_k := \rk \im \delta_k.$
 
 \begin{lemma}\label{ConnectV2} For any $k \in \N_{\geq n_0},$
 \begin{equation}\label{rkqn}
\rk \ker q_{k+1} - \rk \ker q_k = n_{k+1},
\end{equation}
and
\begin{equation}\label{degqn}
\dega \overline{\ker q_{k+1}} - \dega \overline{\ker q_{k}} \geq n_{k+1} A_{k+1} 
\end{equation}
where
$$A_{k+1} := \mua_{\rm min}(\Fb_{k+1}) - \sum_{\sKC} (\log \Vert \tilde{q}_{k,\sigma}^{-1} \Vert + \log \Vert \delta_{k+1,\sigma} \Vert).$$ 
\end{lemma}
 
 \begin{proof}[Proof of Lemma \ref{ConnectV2}] The relation (\ref{rkqn}) follows from Lemma \ref{ConnectV1}, 1).
 
 For any $k \in \N_{>k_0}$, we may consider the admissible short exact of Hermitian vector bundles over $\Spec \OK$:
 \begin{equation}\label{degq1}
0 \lra \overline{\ker q_k} \hlra \overline{\ker q_{k+1}} \stackrel{q_k}{\lra} \Cb_{k+1} \lra 0,
\end{equation}
where $\Cb_{k+1}$ is defined by the $\OK$-module
$$C_{k+1} := \ker \hat{p}_{k+1} = \im \hat{\delta}_{k+1}$$
equipped with the quotient metrics deduced, via the surjective $\C$-linear maps
$$q_{k,\sigma}: \ker q_{k+1,\sigma} \lra C_{k+1,\sigma},$$
from the Hermitian metrics $\Vert.\Vert_{\overline{\ker q_{k+1}},\sigma}$.  In other words, $C_{k+1}$ is $\im \hat{\delta}_{k+1}$ equipped with the Hermitian metric such that the map
$$\tilde{q}_k^{-1} : (\im \hat{\delta}_{k+1})_\sigma = \im \delta_{k+1, \sigma} \hlra H_{k_0,\sigma}$$
becomes an isometry when $H_{k_0,\sigma}$ is equipped with its Hilbert norm $\Vert.\Vert_{\Hbh_{k_0},\sigma}.$

Consequently the operator norm $\Vert \delta_{k+1, \sigma} \Vert_{\Fb_{k+1}, \Cb_{k+1}}$ of the surjective morphism of Hermitian vector bundles
 $$\delta_{k+1}: \Fb_{k+1} \lra \Cb_{k+1}$$
 satisfies: 
 \begin{equation}\label{degq2}
  \Vert \delta_{k+1, \sigma} \Vert_{\Fb_{k+1}, \Cb_{k+1}} = \Vert\tilde{q}_{k,\sigma}^{-1} \circ \delta_{k+1,\sigma}  \Vert_{\Fb_{k+1}, \Cb_{k+1}} \leq \Vert \tilde{q}_{k,\sigma}^{-1}\Vert \, \Vert \delta_{k+1,\sigma}\Vert. 
 \end{equation}
 
According to the additivity of the Arakelov degree in admissible short exact sequences (see \ref{AdmShort}), from (\ref{degq1}), we derive the equality: 
\begin{equation}\label{degq3}
 \dega \overline{\ker q_{k+1}} - \dega \overline{\ker q_{k}} = \dega \Cb_{k+1}.
 \end{equation}
 
 Moreover, applied to $\delta_{k+1},$ the slope inequality for surjective morphisms of Hermitian vector bundles (Proposition \ref{slopeineq}, 2)) shows that
 \begin{equation}\label{degq4}
\dega \Cb_{k+1} \geq \rk \delta_{k+1} [\mua_{\rm min}(\Fb_{k+1}) - \sum_{\sKC} \log  \Vert \delta_{k+1, \sigma} \Vert_{\Fb_{k+1}, \Cb_{k+1}}].
 \end{equation}
 
 The lower bound (\ref{degqn}) follows from (\ref{degq3}), (\ref{degq4}), and (\ref{degq2}).
 \end{proof}

To complete the proof of Theorem \ref{ConnectV}, observe that the validity of  $\mathbf{Amp }$, $\mathbf{An_1 }$ and $\mathbf{An_2}$ implies that
$$\liminf_{k \ra + \infty} \frac{1}{k} A_k > 0.$$
In other words, there exists $\eta$ in $\R^\ast_+$ and $c$ in $\R_+$ such that, for any $k \in \N_{>k_0}$, 
$$A_k \geq  \eta\, k - c.$$ 

Then, for the relations (\ref{rkqn}) and (\ref{degqn}), we derive that, for every $k \in \N_{> k_0},$
$$\rk \ker q_k = \sum_{k_0< i \leq k} n_i$$
and
$$\dega \overline{\ker q_k} \geq \sum_{k_0< i \leq k} n_i (\eta\, i - c).$$
Therefore, for any $k \in \N_{> k_0},$ we have:
\begin{align}
\hot(\Hbh_{k_0}) & \geq \hot(\overline{\ker q_k}) \label{EV1} \\
 & \geq \dega \pi_\ast \overline{\ker q_k} = \dega \overline{\ker q_k} - \rk q_k .(\log \vert \Delta_K\vert)/2  \label{EV2} \\
 & \geq \sum_{k_0< i \leq k} n_i \left[\eta\, i - c  - (\log \vert \Delta_K \vert)/2\right]. \notag
\end{align}
(The lower bounds (\ref{EV1}) and (\ref{EV2}) follows from the monotonicity of $\hot$, established in Proposition \ref{ineqmorthetapro}, and from ``Riemann inequality" (\ref{ThetaRIK})
applied to $\overline{\ker q_k}$.)

According to $\mathbf{Fin }$, $h:= \hot(\Hbh_{k_0})$ is finite, and, if we let $c':= c + (\log \vert \Delta_K \vert)/2,$ we obtain that, for any $k \in \N_{> k_0},$
\begin{equation}\label{EV3}
\sum_{k_0< i \leq k} n_i (\eta\, i - c') \leq h.
\end{equation}
This implies the vanishing of $n_i$ for $i$ large enough. Actually, if we define 
$$i_0 := \max (k_0, \lceil c'/\eta \rceil),$$
the inequalities (\ref{EV3}) show that
$$\eta \sum_{i > i_0} n_i (i-i_0) \leq h - \sum_{k_0 < i \leq i_0} n_i (\eta \, i - c'),$$
and therefore that $n_i = 0$ if 
$$i > \max(i_0, \eta^{-1}[h - \sum_{k_0 < i \leq i_0} n_i (\eta \, i - c')]).$$
\end{proof}

\subsection{Variants} One might observe that the only point of this section \ref{VanCrit} that relies on the results in the previous sections \ref{shortinfdim}-\ref{ShortStrong} is the observation, after the formulation of Condition $\mathbf{Fin }$, that  this $\theta$-finitenesss condition is satisfied as soon as $\Hbh_k$ is $\theta$-finite for some integer $k\geq k_0.$

Actually, the proof of Theorem \ref{ConnectV} proper only uses the elementary monotonicity properties of the $\theta$-invariants of pro-Hermitian  vector bundles, together with the ``Riemann inequality" for (finite rank) Hermitian vector bundles. Accordingly, Theorem \ref{ConnectV} still holds when Condition $\mathbf{Fin }$ is replaced by the following weaker one:  
$$\mathbf{Fin' } \mbox{ : } \lhot(\Hbh_{k_0}) < + \infty.$$

Conditions $\mathbf{An_1 }$ and $\mathbf{An_1 }$ may also be replaced by the following one:

$\mathbf{An :}$ \emph{For every embedding $\sKC,$ the operator norm $\Vert \tilde{q}_{k,\sigma}^{-1} \circ \delta_{k+1,\sigma}\Vert$ of the  $\C$-linear map $\tilde{q}_{k,\sigma}^{-1} \circ \delta_{k+1,\sigma}$  from the normed space $(F_{k+1,\sigma}, \Vert. \Vert_{\Fb_{k+1}, \sigma})$ to the Hilbert space $H^\hilb_{k_0,\sigma}/\ker q_{k,\sigma}$, equipped with the quotient norm deduced from  $\Vert.\Vert_{\Hbh_{k_0},\sigma}$,  satisfies}
$$\limsup_{k \ra +\infty} \frac{1}{k} \log \Vert \tilde{q}_{k,\sigma}^{-1} \circ \delta_{k+1,\sigma} \Vert \leq 0.$$

\section{The category $\proVectOK$
as an exact category}\label{proexact}

In this last section, we investigate the formal properties of the category $\proVectOK$ equipped with the class of short exact sequences considered in this chapter. We shall prove that, equipped with this class of exact sequences, the additive category  $\proVectOK$ is  an \emph{exact category} in the sense of Quillen \cite{Quillen73}, § 2. 

For the commodity of the reader, we have summarized in Appendix \ref{ExactCat} the basic definitions and results concerning exact categories that we will rely on in this section. We refer the reader to the expository articles of Keller  \cite{Keller1996} and B\"uhler  \cite{Buehler2010} for further details and references.

As the terminology \emph{admissible morphism} is already used in this monograph to specify morphisms that are compatible with Hermitian structures --- as is customary in Arakelov geometry --- and as we have also dealt with \emph{strict morphisms} between topological modules, we will call \emph{allowable} (mono- or epi-) \emph{morphism} what is often called an admissible or a strict (mono- or epi-)morphism in some exact category.

\medskip

Let us denote by $\Hilbcont$ the $\C$-linear additive category of complex separable Hilbert spaces, with morphisms the continuous $\C$-linear maps.

The additive category $\CTC_\OK$ (resp. $\CTC_\C$, resp. $\Hilbcont$) becomes an exact category when one  takes as allowable short exact sequences the class $\cEh$ (resp. $\cEh_\C$, resp. $\cE_\C^\hilb$) of diagrams 
$$0 \lra M_1 \stackrel{f}{\lra} M_2 \stackrel{g}{\lra} M_3 \lra 0$$
in $\CTC_\OK$ (resp. in $\CTC_\C$, resp. in $\Hilbcont$) that are exact sequences of $\OK$-modules (resp. of $\C$-vector spaces).  These exact sequences are precisely the ones that are isomorphic to split short exact sequences in $\CTC_\OK$ (resp. in $\CTC_\C$, resp. in $\Hilbcont$). For $\CTC_\OK$ (resp. for $\CTC_\C$), this follows from Proposition   \ref{strictCTCDed3} (resp. from Proposition \ref{strictCTCDed1}) and Proposition \ref{STsplit}. For $\Hilbcont$, it is a consequence of the Banach open mapping theorem and of the existence of orthogonal supplements of closed subspaces in Hilbert spaces.

Let us denote by 
$$\widehat{\digamma} : \proVectOK \lra \CTC_\OK$$
the forgetful functor that maps an object 
$$\Ebh:= (\hE, \left(E_{\sigma}^{\hilb},\Vert. \Vert_\sigma, i_\sigma \right)_{\sigma: K \hra \C})$$
(resp. a morphism $\phi := (\widehat{\phi}, (\phi_\sigma)_{\sKC})$) of $\proVectOK$ to the object $\Eh$ (resp. the morphism $\widehat{\phi}$)  of $\CTC_\OK$.

For any field embedding $\sKC,$ we shall also denote by 
$$\widehat{\digamma}_\sigma : \proVectOK \lra \CTC_\C \quad \mbox{and} \quad
{\digamma}^\hilb : \proVectOK \lra \Hilbcont$$
the functors that map respectively 
$\Ebh$ and $\phi$ to $\Eh_\sigma$ and $\widehat{\phi}_\sigma$, and to $E_\sigma^\hilb$ and $\phi_\sigma.$

The short exact sequences in $\proVectOK$ as defined in \ref{shortpro} above by Conditions $\bf ProSE_1$ and $\bf ProSE_2$ are precisely 
the diagrams in $\proVectOK$ the images of which under the functors $\widehat{\digamma}$ and ${\digamma}^\hilb$ are allowable short exact sequences in $\cEh$ and in $\cE_\C^\hilb$, and therefore appear as some ``liftings" of the latter. However, in general,  short exact sequences in $\proVectOK$ are not isomorphic to split exact sequences in  $\proVectOK$, and accordingly the structure of exact category they define on $\proVectOK$ is much richer than the structures of exact categories on $\CTC_\OK$ and $\Hilbcont$.

\subsection{Short exact sequences, kernels and cokernels, push-out and pull-back in $\proVectOK$}

Let us begin with some properties of short exact sequences in $\proVectOK$ that are formal consequences of similar properties in the categories $\CTC_\OK$, $\CTC_\C$, and $\Hilbcont.$

\begin{proposition}\label{exactkercokerpopb} 1) For any short exact sequence of pro-Hermitian vector bundles over the arithmetic curve $\Spec \OK,$
\begin{equation}\label{exiq}
 0 \lra \Ebh \stackrel{i}{\lra} \Fbh \stackrel{q}{\lra} \Gbh \lra 0,
\end{equation}
the morphism $i$ is a kernel of $q$, and $q$ is a cokernel of $i$, in the additive category 
${\rm pro\overline{Vect}}^{\rm cont}(\OK)$.

2) For any commutative diagram in $\proVectOK$ of the form
 \begin{equation}\label{poexactbis}
\begin{CD}
 0 @>>> \Ebh @>f>> \Fbh @>g>> \Gbh @>>> 0 \\
 @.          @VV{\phi}V                @VV{\tilde{\phi}}V                 @VV{\rm Id}_{\Gbh}V       @. \\
0 @>>> \Ebh' @>f'>> \Fbh' @>g'>> \Gbh @>>> 0 ,
\end{CD}
\end{equation}
the lines of which are short exact sequences in $\proVectOK$, the left-hand square
 \begin{equation}\label{pobisprime}
\begin{CD}
  \Ebh @>f>> \Fbh  \\
  @VV{\phi}V                @VV{\tilde{\phi}}V         \\
 \Ebh' @>f'>> \Fbh' 
\end{CD}
\end{equation}
is cocartesian in $\proVectOK$.

3) For any commutative diagram in $\proVectOK$ of the form
 \begin{equation*}\label{pbexact}
\begin{CD}
 0 @>>> \Ebh @>f'>> \Fbh' @>g'>> \Gbh' @>>> 0 \\
 @.          @VV{{\rm Id}_\Ebh}V                @VV{\tilde{\psi}}V                 @VV{\psi}V       @. \\
0 @>>> \Ebh @>f>> \Fbh @>g>> \Gbh @>>> 0 ,
\end{CD}
\end{equation*}
the lines of which are short exact sequences in $\proVectOK$, the right-hand square
 \begin{equation*}\label{pbbis}\begin{CD}
   \Fbh' @>g'>> \Gbh'  \\
            @VV{\tilde{\psi}}V                 @VV{\psi}V   \\
\Fbh @>g>> \Gbh 
\end{CD}
\end{equation*}
is cartesian in $\proVectOK$.
\end{proposition}

The proof of Proposition  \ref{exactkercokerpopb} relies on the following observation, that is a straightforward consequence of the definition of the category $\proVectOK$:

\begin{lemma}\label{cartcocart}
 Let us consider a commutative diagram in $\proVectOK$:
 \begin{equation}\label{prosquare}
\begin{CD}
  \Abh @>\alpha_1>> \Bbh_1  \\
  @VV{\alpha_2}V                @VV{\beta_1}V         \\
 \Bbh_2@>{\beta_2}>> \Cbh. 
\end{CD}
\end{equation}

If the commutative diagram 
 \begin{equation}\label{prosquareOK}
\begin{CD}
  \Ah @>\widehat{\alpha}_1>> \Bh_1  \\
  @VV{\widehat{\alpha}_2}V                @VV{\widehat{\beta}_1}V         \\
 \Bh_2@>{\widehat{\beta}_2}>> \Ch. 
\end{CD}
\end{equation}
 and, for every embedding $\sKC,$ the diagrams
 \begin{equation}\label{prosquareC}
\begin{CD}
\Ah_\sigma @>\widehat{\alpha}_{1,\sigma}>> \Bh_{1,\sigma}  \\
  @VV{\widehat{\alpha}_{2,\sigma}}V                @VV{\widehat{\beta}_{1,\sigma}}V         \\
 \Bh_{2,\C}@>{\widehat{\beta}_{2,\sigma}}>> \Ch
\end{CD}
\quad\mbox{ and } \quad
\begin{CD}
 A^\hilb_\sigma @>{\alpha}_{1,\sigma}>> B^\hilb_{1,\sigma}  \\
  @VV{{\alpha}_{2,\sigma}}V                @VV{{\beta}_{1,\sigma}}V         \\
 B^\hilb_{2,\sigma}@>{{\beta}_{2,\sigma}}>> C_\sigma^\hilb. 
\end{CD}
\end{equation}
are cartesian  in $\CTC_\OK$, $\CTC_\C$ and $\Hilbcont$ respectively, then the diagram (\ref{prosquareOK}) is cartesian in $\proVectOK$.

Dually, if the diagrams (\ref{prosquareOK}) and (\ref{prosquareC}) are cocartesian in $\CTC_\OK$, $\CTC_\C$ and $\Hilbcont$ respectively, then the diagram  (\ref{prosquare}) is cocartesian in $\proVectOK$.  \qed
\end{lemma}

 \begin{proof}[Proof of Proposition  \ref{exactkercokerpopb}] Let us prove 1). 
 
 From the exact sequence (\ref{exiq}), we deduce the short exact sequence
 \begin{equation}\label{exiqh}
 0 \lra \Eh \stackrel{\widehat{i}}{\lra} \Fh \stackrel{\widehat{q}}{\lra} \Gh \lra 0
\end{equation}
 in $\CTC_\OK$
 and, for every embedding $\sKC,$ the short exact sequence
 \begin{equation}\label{exiqsigh}
 0 \lra E^\hilb_\sigma \stackrel{{i}_\sigma}{\lra} F^\hilb_\sigma \stackrel{{q}_\sigma}{\lra} G^\hilb_\sigma \lra 0
\end{equation}
in $\Hilbcont$. 

Actually, as already observed, these exact sequences are split in $\CTC_\OK$ and $\Hilbcont$ respectively. This  implies that the diagram
  \begin{equation}\label{exiqsig}
 0 \lra \Eh_\sigma \stackrel{\widehat{i}_\sigma}{\lra} \Fh_\sigma \stackrel{\widehat{q}_\sigma}{\lra} \Gh_\sigma \lra 0
\end{equation}
is a split short exact sequence in $\CTC_\C$. Moreover, $\widehat{i}$ (resp. $\widehat{i}_\sigma$, resp. $i_\sigma$) is a kernel of $\widehat{q}$ (resp. $\widehat{q}_\sigma$, resp. $q_\sigma$)
in $\CTC_\OK$ (resp. $\CTC_\C$, resp. $\Hilbcont$). Dually, $\widehat{q}$ (resp. $\widehat{q}_\sigma$, resp. $q_\sigma$) is a cokernel of $\widehat{i}$ (resp. $\widehat{i}_\sigma$, resp. $i_\sigma$) in $\CTC_\OK$ (resp. $\CTC_\C$, resp. $\Hilbcont$).

As in any additive category, the fact that $i$ is kernel of $q$ in $\proVectOK$ may be rephrased as the fact that the diagram
\begin{equation}\label{kercart}
\begin{CD}
  \Ebh @>i>> \Fbh  \\
  @VVV                @VV{q}V         \\
 0@>>> \Gbh. 
\end{CD}
\end{equation}
is cartesian in $\proVectOK$.
According to Lemma \ref{cartcocart}, this follows from the fact that the diagrams deduced from (\ref{kercart}) by applying the functors $\widehat{\digamma}$, $\widehat{\digamma}_\sigma$ and $\digamma_\sigma$ are cartesian in the additive categories $\CTC_\OK$, $\CTC_\C$ and $\Hilbcont$. These diagrams are indeed cartesian since, as observed above, $\widehat{i}$ (resp. $\widehat{i}_\sigma$, resp. $i_\sigma$) is a kernel of $\widehat{q}$ (resp. $\widehat{q}_\sigma$, resp. $q_\sigma$)
in $\CTC_\OK$ (resp. $\CTC_\C$, resp. $\Hilbcont$).

Dually, we prove that $q$ is a cokernel of $i$ by showing that (\ref{kercart}) is cocartesian; this  follows by Lemma \ref{cartcocart} from the fact that $\widehat{q}$ (resp. $\widehat{q}_\sigma$, resp. $q_\sigma$) is a cokernel of $\widehat{i}$ (resp. $\widehat{i}_\sigma$, resp. $i_\sigma$) in $\CTC_\OK$ (resp. $\CTC_\C$, resp. $\Hilbcont$).

The proof of 2) and 3) are similar.  For instance, to prove that (\ref{pobisprime}) is cartesian, it is enough to show that the diagram deduced from (\ref{pobisprime}) by applying the functors $\widehat{\digamma}$, $\widehat{\digamma}_\sigma$ and $\digamma_\sigma$ are cartesian in the additive categories $\CTC_\OK$, $\CTC_\C$ and $\Hilbcont$. This follows from the commutative diagrams with exact lines in these categories obtained by applying these functors to the diagram (\ref{poexactbis}).
\end{proof}
 
\subsection{Allowable monomorphisms and epimorphisms in $\proVectOK$} 

\begin{proposition}\label{AM} Let $f: \Ebh \lra \Fbh$ be a morphism in $\proVectOK$. The following two conditions on $f$ are equivalent:

(i) The morphism $f := (\hat{f}, (f_\sigma)_{\sKC})$ satisfies the three properties:

 $\bf AM_1 :$ the map $\hat{f}: \Eh \lra \Fh$ is injective and strict\footnote{Recall that a morphism if $\CTC_\OK$ is strict if and only if its image is closed.}, and $\coker \hat{f}$ satisfies $\bf CTC_2$ \emph{(and therefore defines an object of $\CTC_\OK$)};
 
 $\bf AM_2 :$ for every field embedding $\sKC,$ the continuous $\C$-linear map $f_\sigma: E_\sigma^{\hilb} \lra F_\sigma^{\hilb}$ is injective and strict\footnote{Recall that a continuous linear maps between Banach spaces is strict if and only if its image is closed.};
 
 $\bf AM_3 :$ for every embedding $\sKC,$ the natural map
 $$\coker  f_\sigma := F_\sigma^{\hilb}/f_\sigma(E_\sigma^{\hilb}) \lra
 \coker  \hat{f}_\sigma := \Fh_\sigma/\hat{f}_\sigma(\Eh_\sigma)$$
 is injective.
 
 (ii) There exists some pro-Hermitian vector bundle $\Gbh$ and some morphism $g: \Fbh \lra \Gbh$ in $\proVectOK$  such that the diagram
 \begin{equation*}\label{exactfg}
0 \lra \Ebh \stackrel{f}{\lra} \Fbh \stackrel{g}{\lra} \Gbh \lra 0,
\end{equation*}
is a short exact sequence of pro-Hermitian vector bundles. 
\end{proposition}

Recall that, according to Proposition \ref{EF}, 2), Condition $\bf AM_1$ is equivalent to

$\bf AM'_1 :$ \emph{The dual map $\hat{f}^\vee: \Fh^\vee \lra \Eh^\vee$ is surjective.}

\noindent Moreover, when it is satisfied, the short exact sequence in $CTC_\OK$
$$0 \lra \Eh \stackrel{\hat{f}}{\lra} \Fh  \lra \coker \hat{f} \lra 0$$
is split, and the maps
$$\hat{f}_\sigma: \Eh_\sigma \lra \Fh_\sigma$$
are injective. 

Condition $\bf AM_2$ may be rephrased as:

 $\bf AM'_2 :$ \emph{for every field embedding $\sKC,$ the continuous $\C$-linear map $f_\sigma: E_\sigma^{\hilb} \lra F_\sigma^{\hilb}$ is injective with closed image.}
 
Observe also that, when $\bf AM_1$ and $\bf AM_2$ are satisfied, Condition $\bf AM_3$ is equivalent to:

$\bf AM'_3 :$ \emph{for every field embedding $\sKC$, the following diagram is cartesian:}
\begin{equation*}
\begin{CD}
 E^{\hilb}_\sigma @>f_\sigma>> F^{\hilb}_\sigma \\
 @V{i^{\Ebh}_\sigma}VV                @VV{i^{\Fbh}_\sigma}V              \\
 \Eh_\sigma @>\hat{f}_\sigma>> \Fh_\sigma.
 \end{CD}
\end{equation*}

\begin{proposition}\label{AE} Let $g: \Fbh \lra \Gbh$ be a morphism in $\proVectOK$. The following two conditions on $g$ are equivalent:

(i) The morphism $g := (\hat{g}, (g_\sigma)_{\sKC})$ satisfies the three properties:

 $\bf AE_1 :$ the map $\hat{g}: \Fh \lra \Gh$ is surjective; 
  
 $\bf AE_2 :$ for every field embedding $\sKC,$ the  map $g_\sigma: F_\sigma^{\hilb} \lra G_\sigma^{\hilb}$ is surjective;
 
 $\bf AE_3 :$ for every field embedding $\sKC,$ the inclusion map
 $$\beta_\sigma : = i^{\Gh}_{\sigma \mid \ker g_\sigma} : \ker g_\sigma  \hlra \ker \hat{g}_\sigma$$
 has a dense image.
 
 (ii) There exists some pro-Hermitian vector bundle $\Ebh$ and some morphism $f: \Ebh \lra \Fbh$ in $\proVectOK$  such that the diagram
 \begin{equation*}\label{exactfgbis}
0 \lra \Ebh \stackrel{f}{\lra} \Fbh \stackrel{g}{\lra} \Gbh \lra 0,
\end{equation*}
is a short exact sequence of pro-Hermitian vector bundles. 
\end{proposition}

Let us observe that, when Conditions $\bf AE_1$ and $\bf AE_2$ hold, then the morphisms $\hat{f}$ and $f_\sigma$ are automatically strict (by Banach open image theorem), and the Hahn-Banch theorem shows that Condition $\bf AE_3$ is equivalent to

$\bf AE'_3 :$ \emph{for every field embedding $\sKC,$ the map 
$$\beta_\sigma^\vee : (\ker \hat{g}_\sigma)^\vee = \coker \hat{g}_\sigma^\vee \lra
(\ker {g}_\sigma)^\vee = \coker {g}_\sigma^\vee$$
is injective.}

In turn, still assuming $\bf AE_1$ and $\bf AE_2$ satisfied (and therefore the maps $\hat{g}_\sigma^\vee$ and $g_\sigma^\vee$ injective), Condition $\bf AE'_3$ may be rephrased as:

$\bf AE''_3 :$ \emph{for every field embedding $\sKC$, the following diagram is cartesian:}
\begin{equation*}
\begin{CD}
 \Gh^\vee_\sigma @>\hat{g}^\vee_\sigma>> \Fh^\vee_\sigma \\
 @V{i^{\Gbh \vee}_\sigma}VV                @VV{i^{\Fbh \vee}_\sigma}V              \\
 G^{\hilb \vee}_\sigma @>g_\sigma^\vee>> F^{\hilb \vee}_\sigma .
 \end{CD}
\end{equation*}

\medskip
A morphism in ${\rm pro\overline{Vect}}^{\rm cont}(\OK)$ which satisfies the equivalent conditions in Proposition \ref{AM} (resp. in Proposition \ref{AE}) will be called an \emph{allowable monomorphism} (resp. an \emph{allowable epimorphism}).

\begin{proof}[Proof of Proposition \ref{AM}] Let us assume that Condition (i) is satisfied. Then we may consider
$$\Gh := \coker \widehat{f} = \Fh/ \widehat{f}(\Eh).$$
By  $\bf AM_1$, it is an object of $\CTC_\OK$. Moreover, for every embedding $\sKC,$ we have a natural identification
$$\Gh_\sigma \lrasim  \coker \widehat{f}_\sigma:= \Fh_\sigma/ \widehat{f}_\sigma(\Eh_\sigma).$$
(Indeed, the quotient map $\Fh \lra \Gh$ is split surjective by Proposition   \ref{strictCTCDed3}.)

For every embedding $\sKC,$ we may also define $G_\sigma^\hilb$ as the quotient of the Hilbert space $(F^\hilb_\sigma, \Vert. \Vert_{\Fb,\sigma})$ by the subspace $f_\sigma(E_\sigma)$:
$$G_\sigma^\hilb := \coker f_\sigma = F_\sigma^\hilb /f_\sigma(E_\sigma).$$
According to $\bf AM_2$,  $f_\sigma(E_\sigma)$ is closed in $F_\sigma^\hilb$, and therefore, equipped with the quotient norm deduced from $\Vert. \Vert_{\Fbh,\sigma}$, $G_\sigma^\hilb$ becomes a Hilbert space $(G_\sigma^\hilb, \Vert. \Vert_{\Gb,\sigma}).$ 

According to $\bf AM_3$, the natural maps 
$$i_\sigma^{\Gbh} : G_\sigma^\hilb := \coker f_\sigma \lra \Gh_\sigma := \Fh_\sigma/ \widehat{f}_\sigma(\Eh_\sigma)$$
are injective. Moreover, as $i^{\Fbh}_\sigma(F_\sigma^\hilb)$ is dense in $\Fh_\sigma$, its image is dense. Moreover the construction of the Hilbert spaces  $(G_\sigma^\hilb, \Vert. \Vert_{\Gb,\sigma})$ and of the maps $i_\sigma^{\Gbh}$ is clearly compatible with complex conjugation. This shows that
$$\Gbh := (\Gh, (G_\sigma^\hilb, \Vert. \Vert_{\Gb,\sigma}, i_\sigma^{\Gbh})_\sKC)$$
is an object of $\proVectOK$.

The quotient maps
$$\widehat{g}: \Fh \lra \Gh \quad \mbox{ and } \quad {g}_\sigma: \Fh_\sigma^\hilb \lra G_\sigma^\hilb$$
define a morphism in $\proVectOK$
$$g: \Fbh \lra \Gbh$$
and, by construction, the diagram
$$0 \lra \Ebh \stackrel{f}{\lra} \Fbh \stackrel{g}{\lra} \Gbh \lra 0$$
is a short exact sequence of pro-Hermitian vector bundles over $\Spec \OK.$ This proves that Condition (ii) holds.

The converse implication (ii) $\Rightarrow$ (i) is easy and left to the reader.
 \end{proof}

\begin{proof}[Proof of Proposition \ref{AE}] Let us assume that Condition (i) is satisfied. 

According to Proposition \ref{subCTC}, the closed $\OK$-submodule
$\Eh := \ker \widehat{g}$ of $\Fh$ is an object of $\CTC_\OK$. According to $\bf AE_1$, we may form the exact sequence 
$$0 \lra \Eh := \ker \widehat{g} \;\hlra \Fh \stackrel{\widehat{g}}{\lra} \Gh \lra 0.$$
By Proposition  \ref{strictCTCDed3}, it is split in $\CTC_\OK$, and therefore we get a natural identification:
$$\Eh_\sigma \simeq \ker \widehat{g}_\sigma.$$

For every embedding $\sKC,$ we may introduce the Hilbert space $(E_\sigma^\hilb, \Vert.\Vert_{\Ebh, \sigma})$ defined by $\ker g_\sigma$ equipped with the restriction of the Hilbert norm $\Vert.\Vert_{\Fbh, \sigma}$ on $F_\sigma$. According to $\bf AE_2$, it fits into the following short exact sequence in $\Hilbcont$:
$$0 \lra E^\hilb_\sigma \; \hlra F^\hilb_\sigma \stackrel{g_\sigma}{\lra} G^\hilb_\sigma \lra 0.$$

Finally, $\bf AE_3$ shows that 
$$\Ebh := (\Eh, (E^\hilb_\sigma, \Vert.\Vert_{\Ebh, \sigma}, \beta_\sigma)_\sKC)$$
defines an object of $\proVectOK$. It is straightforward that the inclusion maps
$$\Eh \; \hlra \Fh \quad \mbox{ and } \quad E_\sigma^\hilb \; \hlra F_\sigma^\hilb$$
define a morphism
$f: \Ebh \lra \Fbh$
in $\proVectOK$ and that the diagram  
$$0 \lra \Ebh \stackrel{f}{\lra} \Fbh \stackrel{g}{\lra} \Gbh \lra 0$$
is a short exact sequence of pro-Hermitian vector bundles over $\Spec \OK.$ This proves that Condition (ii) holds.

 Here again, the converse implication (ii) $\Rightarrow$ (i) is easy and left to the reader.
\end{proof}

\subsection{The category $\proVectOK$ is an exact category}

\begin{theorem}\label{proVectEx}
 The category $\proVectOK$, equipped with the class $\cE$ of kernel-cokernel pairs $(f,g)$ defined by the short exact sequences
 \begin{equation}\label{shortfg}
0 \lra \Ebh \stackrel{f}{\lra} \Fbh \stackrel{g}{\lra} \Gbh \lra 0
\end{equation}
defined in \ref{shortpro} above is an exact category.
\end{theorem}

We refer the reader to Appendix \ref{ExactCat} for the basic definitions and basic properties of exact categories we will use in this section.

Recall that we have already observed that the pair of morphisms $(f,g)$ attached to short exact sequences (\ref{shortfg}) are kernel-cokernel pairs (Proposition \ref{exactkercokerpopb}). It is also clear that they form a class closed under isomorphisms. Theorem \ref{proVectEx} may therefore be rephrased as the fact  that the allowable monomorphisms and epimorphisms, as defined in the previous subsection, satisfy the axioms $\mathbf{E_0},$ $\mathbf{E_0^{\rm op}}$, 
$\mathbf{E_1},$ $\mathbf{E_1^{\rm op}}$, and
$\mathbf{E_2},$ $\mathbf{E_2^{\rm op}}$ recalled in Appendix \ref{ExactCat}.

The identity morphisms are  both allowable monomorphism and allowable epimorphisms; indeed, for any object $\Ebh$ of $\proVectOK$, the diagrams
$$ 0 \lra \Ebh \stackrel{{\rm Id}_\Ebh}{\lra} \Ebh \lra 0 \lra 0 
\quad \mbox{and} \quad
0 \lra 0 \lra \Ebh \stackrel{{\rm Id}_\Ebh}{\lra} \Ebh \lra 0$$
are short exact sequences in $\proVectOK$. Therefore Axioms $\mathbf{E_0}$ and $\mathbf{E_0^{\rm op}}$ hold.

The validity of Axioms $\mathbf{E_1}$ and $\mathbf{E_1^{\rm op}}$ may be stated as the following proposition:

\begin{proposition} 1) The composition of two allowable monomorphisms in $\proVectOK$ is an allowable monomorphism.

2) The composition of two allowable epimorphisms in $\proVectOK$ is an allowable epimorphism. 
\end{proposition}

\begin{proof}Assertion 1) follows from the observation that Conditions $\bf AM'_1,$ $\bf AM_2$, and $\bf AM'_3$ are stable under composition.

To prove 2), let us consider a diagram
$$\Fbh_1 \stackrel{g}{\lra} \Fbh_2 \stackrel{g'}{\lra} \Fbh_3$$
in $\proVectOK$, with $g$ and $g'$ some allowable epimorphisms.
The conditions $\bf AE_1$ and $\bf AE_2$ are clearly stable under compositions, and therefore the morphism
$$g' \circ g : \Fbh_1 \lra \Fbh_3$$
also satisfy them. Moreover, for every field embedding $\sKC,$ we may consider the commutative diagram: 
\begin{equation*}
\begin{CD} \Fh^\vee_{3,\sigma} @>\widehat{g'}^\vee_\sigma>> \Fh^\vee_{2,\sigma} @>\widehat{g}^\vee_\sigma>> \Fh^\vee_{1,\sigma} \\
 @V{i^{\Fbh_3 \vee}_\sigma}VV                @VV{i^{\Fbh_2 \vee}_\sigma}V           @VV{i^{\Fbh_1 \vee}_\sigma}V       \\
  F^{\hilb \vee}_{3,\sigma} @>g_\sigma^{'\vee}>> F^{\hilb \vee}_{2,\sigma}  @>g_\sigma^\vee>> F^{\hilb \vee}_{1,\sigma}.
 \end{CD}
\end{equation*}
Since $g'$ and $g$ satisfy Condition $\bf AE''_3$, both its left-hand and its right-hand sides are cartesian. This implies that its exterior rectangle is cartesian too.  As $\hat{g}_\sigma\circ \hat{g}_\sigma^{'\vee} = \widehat{(g' \circ g)}^\vee_\sigma$ and ${g}_\sigma\circ {g}_\sigma^{'\vee} = {(g' \circ g)}^\vee_\sigma$, this implies that $g' \circ g$ satisfies $\bf AE''_3$. 
\end{proof}

The validity of Axioms $\mathbf{E_2}$ and $\mathbf{E_2^{\rm op}}$ is a consequence of (actually, equivalent to) the following proposition:

\begin{proposition}\label{E2proVect}
 Let
  \begin{equation}\label{Exfg} 
 0 \lra \Ebh \stackrel{f}{\lra} \Fbh \stackrel{g}{\lra} \Gbh \lra 0
\end{equation}
be a short exact sequence in $\proVectOK$.

1) For any morphism $$\phi: \Ebh \lra \Ebh'$$
in $\proVectOK$ with source $\OK,$ one may form a commutative diagram in $\proVectOK$ 
\begin{equation}\label{popro}
\begin{CD}
 0 @>>> \Ebh @>f>> \Fbh @>g>> \Gbh @>>> 0 \\
 @.          @VV{\phi}V                @VV{\tilde{\phi}}V                 @VV{\rm Id}_{\Gbh}V       @. \\
0 @>>> \Ebh' @>f'>> \Fbh' @>g'>> \Gbh @>>> 0 ,
\end{CD}
\end{equation}
the second line of which is a short exact sequence in $\proVectOK$.

2) For any morphism $$\psi: \Gbh' \lra \Gbh$$ in $\proVectOK$ with range $\Gb,$ one may form a commutative diagram in $\proVectOK$
 \begin{equation}\label{pbpro}
\begin{CD}
 0 @>>> \Ebh @>f'>> \Fbh' @>g'>> \Gbh' @>>> 0 \\
 @.          @VV{{\rm Id}_\Ebh}V                @VV{\tilde{\psi}}V                 @VV{\psi}V       @. \\
0 @>>> \Ebh @>f>> \Fbh @>g>> \Gbh @>>> 0 ,
\end{CD}
\end{equation}
the first line of which is a short exact sequence in $\proVectOK$.
\end{proposition}

Indeed, according to Proposition \ref{exactkercokerpopb}, 2), with the notation of 1), the diagram 
 \begin{equation*}\label{pobis}
\begin{CD}
  \Ebh @>f>> \Fbh  \\
  @VV{\phi}V                @VV{\tilde{\phi}}V         \\
 \Ebh' @>f'>> \Fbh' 
\end{CD}
\end{equation*}
is cocartesian, and therefore \emph{``the" push-out of the allowable monomorphism $f$ along  the morphism $\phi$ exists} --- it is $f'$ --- \emph{ and is an allowable monomorphism}.

Similarly, according to Proposition \ref{exactkercokerpopb}, 3), with the notation of 2), the diagram 
 \begin{equation*}\label{pbbisprime}\begin{CD}
   \Fbh' @>g'>> \Gbh'  \\
            @VV{\tilde{\psi}}V                 @VV{\psi}V   \\
\Fbh @>g>> \Gbh 
\end{CD}
\end{equation*}
is cartesian, and therefore \emph{``the" pull-back of the allowable epimorphism $g$ along the morphism $\psi$ exists} --- it is $g'$ --- \emph{and is an allowable epimorphism.}
 
\begin{proof}[Proof of Proposition \ref{E2proVect}]
Let us prove 2).

 In each of the additive categories $\CTC_\OK$, $\CTC_\C$ and $\Hilbcont$, we may construct the ``pull-back" of the exact sequence deduced from (\ref{Exfg}) by the forgetful functor $\widehat{\digamma}$, $\widehat{\digamma}$ and $\digamma^\hilb$ by the morphisms $\widehat{\psi}$, $\widehat{\psi}_\sigma$ and $\psi_\sigma$ that define $\psi.$ 

Namely, we may define an object $\Fh'$ in $\CTC_\OK$ by letting
$$\Fh' := \Fh \oplus_{\Gh} \Gh' := \ker \left(\Fh \oplus \Gh' \stackrel{\binom{\hat{f}}{-\hat{\psi}}}{\lra} \Gh\right).$$
By restriction to $\Fh \oplus_{\Gh} \Gh'$, the projections from $\Fh \oplus \Gh'$ to $\Fh$ and $\Gh'$ define morphisms
$$\widehat{\tilde{\psi}}: \Fh' \lra \Fh \quad \mbox{ and } \quad \hat{g}': \Fh' \lra \Gh'.$$
Besides, from the strict injection $\hat{f}: \Eh \lra \Fh,$ we may construct another strict injection:
$$\hat{f}' := (\hat{f}, 0): \Eh \lra \Fh' (\hlra \Fh \oplus \Gh').$$ 
Moreover, it is straightforward that the following diagram in $\CTC_\OK$ is commutative
 \begin{equation}\label{pbprohat}
\begin{CD}
 0 @>>> \Eh @>\hat{f}'>> \Fh' @>\hat{g}'>> \Gh' @>>> 0 \\
 @.          @VV{{\rm Id}_\Eh}V                @VV{\widehat{\tilde{\psi}}}V                 @VV{\widehat{\psi}}V       @. \\
0 @>>> \Eh @>\hat{f}>> \Fh @>\hat{g}>> \Gh @>>> 0
\end{CD}
\end{equation}
and that its first line is exact.

Similarly, for very embedding $\sKC,$ we may define
$$\Fh'_\sigma := \Fh_\sigma \oplus_{\Gh_\sigma} \Gh'_\sigma := \ker \left(\Fh_\sigma \oplus \Gh'_\sigma \stackrel{\binom{\hat{f}_\sigma}{-\hat{\psi}_\sigma}}{\lra} \Gh_\sigma\right)$$
and
$$F^{'\hilb}_\sigma := F^\hilb_\sigma \oplus_{G^\hilb_\sigma} G^{'\hilb}_\sigma := \ker \left(F^{\hilb}_\sigma \oplus G^{'\hilb}_\sigma \stackrel{\binom{{f}_\sigma}{-{\psi}_\sigma}}{\lra} G^\hilb_\sigma\right),$$
and form the following commutative diagrams with exact lines in $\CTC_\C$ and $\Hilbcont$: 
 \begin{equation}\label{pbprohatsigma}
\begin{CD}
 0 @>>> \Eh_\sigma @>\hat{f}'_\sigma>> \Fh' @>\hat{g}'_\sigma>> \Gh'_\sigma @>>> 0 \\
 @.          @VV{{\rm Id}_{\Eh_\sigma}}V                @VV{\widehat{\tilde{\psi}}_\sigma}V                 @VV{\widehat{\psi}_\sigma}V       @. \\
0 @>>> \Eh_\sigma @>\hat{f}_\sigma>> \Fh_\sigma @>\hat{g}_\sigma>> \Gh_\sigma @>>> 0
\end{CD}
\end{equation}
and
\begin{equation}\label{pbprohatsigmahilb}
\begin{CD}
 0 @>>> E^\hilb_\sigma @>{f}^{'\hilb}_\sigma>> F^{'\hilb} @>\hat{g}^{'\hilb}_\sigma>> G^{'\hilb}_\sigma @>>> 0 \\
 @.          @VV{{\rm Id}_{E^\hilb_\sigma}}V                @VV{{\tilde{\psi}}^\hilb_\sigma}V                 @VV{{\psi}^\hilb_\sigma}V       @. \\
0 @>>> E^\hilb_\sigma @>{f}^\hilb_\sigma>> F^\hilb_\sigma @>{g}^\hilb_\sigma>> G^\hilb_\sigma @>>> 0.
\end{CD}
\end{equation}
Here again, the morphisms $\widehat{\tilde{\psi}}_\sigma$ and  $\hat{g}'_\sigma$ (resp. ${\tilde{\psi}}_\sigma$ and ${g}'_\sigma)$) are deduced from 
the projections from $\Fh_\sigma \oplus \Gh'_\sigma$
(resp. from $F^\hilb_\sigma \oplus G^{'\hilb}_\sigma$) to
its two factors, and $\hat{f}'_\sigma$ (resp. ${f}^{'\hilb}_\sigma$) is defined as $(\hat{f}_\sigma,0)$ (resp. as $({f}^{\hilb}_\sigma,0)$).

The short exact sequence
  \begin{equation}\label{Exfghat} 
 0 \lra \Eh \stackrel{\hat{f}}{\lra} \Fh \stackrel{\hat{g}}{\lra} \Gh \lra 0
\end{equation}
 may be split in $\CTC_\OK$ (see Proposition \ref{strictCTCDed3}). A choice of splitting for (\ref{Exfghat}) determines an isomorphism
\begin{equation}\label{splithat}
\Fh \simeq \Eh \oplus \Gh,
\end{equation}
a splitting of the short exact sequence in $\CTC_\OK$
  \begin{equation*}\label{Exfghatprime} 
 0 \lra \Eh \stackrel{\hat{f}'}{\lra} \Fh' \stackrel{\hat{g}'}{\lra} \Gh' \lra 0
\end{equation*}
and a corresponding isomorphism:
\begin{equation}\label{splithatprime}
\Fh' \simeq \Eh \oplus \Gh'.
\end{equation}

Moreover, for every embedding $\sKC,$ the splitting of (\ref{Exfghat}) also induces splittings of 
  \begin{equation*}\label{Exfghatsigma} 
 0 \lra \Eh_\sigma \stackrel{\hat{f}_\sigma}{\lra} \Fh_\sigma \stackrel{\hat{g}_\sigma}{\lra} \Gh_\sigma \lra 0
\end{equation*}
and 
  \begin{equation*}\label{Exfghatprimesigma} 
 0 \lra \Eh_\sigma \stackrel{\hat{f}'_\sigma}{\lra} \Fh'_\sigma \stackrel{\hat{g}'_\sigma}{\lra} \Gh'_\sigma \lra 0
\end{equation*}
in $\CTC_\C$, and an associated isomorphism
\begin{equation}\label{splithatprimesigma}
\Fh'_\sigma \simeq \Eh_\sigma \oplus \Gh'_\sigma.
\end{equation}

Besides, the inclusion maps $\Fh \hlra \Fh_\sigma,$ $\Gh \hlra \Gh_\sigma,$ and $\Gh' \hlra \Gh'_\sigma$ define a canonical continous $\OK$-linear map
$$\Fh' := \Fh \oplus_{\Gh} \Gh' \lra \Fh'_\sigma := \Fh_\sigma \oplus_{\Gh_\sigma} \Gh'_\sigma.$$
This map is compatible with the splittings (\ref{splithatprime}) and (\ref{splithatprimesigma}). Therefore it induces a canonical isomorphisms:
$$(\Fh')_\sigma := \Fh' \widehat{\otimes}_\sigma \C \lrasim \Fh'_\sigma.$$

The maps $i_\sigma^{\Fbh},$ $i_\sigma^{\Gbh},$ and $i_\sigma^{\Gbh'}$ satisfy
$$\hat{g}_\sigma \circ i_\sigma^{\Fbh} = i_\sigma^{\Gbh} \circ g_\sigma \quad \mbox{and} \quad 
\widehat{\psi}_\sigma \circ i_\sigma^{\Fbh} = i_\sigma^{\Gbh'} \circ \psi_\sigma,$$
and therefore define a $\C$-linear map:
$$
\begin{array}{rrcl}
 i_\sigma: &  F^{'\hilb}_\sigma := F^\hilb_\sigma \oplus_{G^\hilb_\sigma} G^{'\hilb}_\sigma & \lra   &  \Fh'_\sigma := \Fh_\sigma \oplus_{\Gh_\sigma} \Gh'_\sigma  \\
 & (v, w) & \longmapsto  & (i_\sigma^{\Fbh}(v), i_\sigma^{\Gbh'}(w)).   
\end{array}
$$
Moreover the diagram
\begin{equation}\label{EFpGp}
\begin{CD}
 0 @>>> E^\hilb_\sigma @>{f}^{'\hilb}_\sigma>> F^{'\hilb} @>\hat{g}^{'\hilb}_\sigma>> G^{'\hilb}_\sigma @>>> 0 \\
 @.          @VV{i_\sigma^{\Ebh}}V                @VV{i_\sigma}V                 @VV{i_\sigma^{\Gbh'}}V       @. \\
 0 @>>> \Eh_\sigma @>\hat{f}'_\sigma>> \Fh' @>\hat{g}'_\sigma>> \Gh'_\sigma @>>> 0
 \end{CD}
\end{equation}
is clearly commutative.

\begin{lemma}\label{ibon} The $\C$-linear maps $i_\sigma: F_\sigma^{'\hilb} \lra \Fh_\sigma'$ are continuous and injective, with dense image.
\end{lemma}
\begin{proof}[Proof of Lemma \ref{ibon}] The continuity and the injectivity of $i_\sigma$ are clear. 

To prove that its image is dense, observe that in the commutative diagram (\ref{EFpGp}), the lines are short exact sequences in $\CTC_\C$ and in $\Hilbcont$ respectively, and that the vertical maps are continuous $\C$-linear maps of locally convex complex vector spaces. 

By duality, from (\ref{EFpGp}), we get the following commutative diagram with exact lines:
\begin{equation}\label{EFpGpdual}
\begin{CD}
 0 @<<< E^{\hilb\vee}_\sigma @<{f}^{'\hilb\vee}_\sigma<< F^{'\hilb\vee} @<\hat{g}^{'\hilb\vee}_\sigma<< G^{'\hilb\vee}_\sigma @<<< 0 \\
 @.          @AA{i_\sigma^{\Gbh'\vee}}A                @AA{i_\sigma^\vee}A                 @AA{i_\sigma^{\Ebh\vee}}A       @. \\
 0 @<<< \Eh^\vee_\sigma @<\hat{f}^{'\vee}_\sigma<< \Fh^{'\vee} @<\hat{g}{'\vee}_\sigma<< \Gh^{'\vee}_\sigma @<<< 0. 
\end{CD}
\end{equation}
The  maps  $i_\sigma^{\Gbh'\vee}$ and $i_\sigma^{\Ebh\vee}$ are injective, since $i_\sigma^{\Gbh'}$ and $i_\sigma^{\Ebh}$ have dense images.  The exactness of the lines of (\ref{EFpGpdual}) now imply that $i_\sigma^\vee$ is injective and therefore, by Hahn-Banach theorem, that $i_\sigma$ has a dense image.
\end{proof}

The construction of $F^{'\hilb}_\sigma$, ${\Fh}'_\sigma$ and $i_\sigma$ is clearly compatible with complex conjugation. Together with Lemma \ref{ibon}, this shows that
$$\Fbh' := (\Fh', (F_\sigma^{'\hilb}, i_\sigma)_\sKC)$$
is an object of $\proVectOK.$ It is now straightforward that
$$g' := (\hat{g}', (g'_\sigma)_\sKC) \quad (\mbox{resp. } \tilde{\psi}:= (\widehat{\tilde{\psi}}, (\tilde{\psi}_\sigma)_\sKC))$$
is a morphism from $\Fbh'$ to $\Fbh$ (resp. from $\Fbh'$ to $\Gbh'$) in $\proVectOK$ and that  (\ref{pbpro}) is a commutative diagram in $\proVectOK$, the first line of which is a short exact sequence. 

The proof of 1) is similar: now $\Fbh':= (\Fh', (F^\hilb_\sigma, i_\sigma)_\sKC)$ is constructed from the ``push-out"  $\Fh$ (resp. $F_\sigma^\hilb$ in the categorie $\CTC_\OK$ (resp. $\Hilbcont$) of the exact sequence deduced from from (\ref{Exfg}) by the forgetful functor $\widehat{\digamma}$ (resp. $\digamma^\hilb$) by the morphism $\widehat{\phi}$ (resp. $\psi_\sigma$) underlying $\phi$. The details of the proof are slightly simpler than in the proof of 2) and will be left to the reader.  
\end{proof}  

\subsection{Examples and complements}

\subsubsection{The conditions $\bf AM_3$ and $\bf AE_3$} Let us emphasize the role of the third condition $\bf AM_3$ (resp. $\bf AE_3$) in the characterization of allowable monomorphisms (resp. epimorphisms) of the exact category $\proVectOK$  in Proposition \ref{AM} (resp. in Proposition \ref{AE}): this condition does \emph{not} follow in general from the first two conditions $\bf AM_1$ and  $\bf AM_2$ (resp. $\bf AE_1$ and $\bf AE_2$).

Indeed, as shown in Paragraph \ref{v} \emph{supra}, one may construct a morphism $f: \Ebh \lra \Fbh$ in $\proVectZ$ such that
$$\mbox{$\hat{f}: \Eh \lra \Fh$ is an isomorphism}$$
and
$$\mbox{$f_\C: E_\C^\hilb \lra F_\C^\hilb$ is injective and strict, and not surjective.}$$
Such a morphism satisfies $\bf AM_1$ and  $\bf AM_2$, but not $\bf AM_3$. 

Besides, we may construct a morphism $g: \Fbh \lra \Gbh$ in $\proVectZ$ such that
\begin{equation}\label{gAE1}
\mbox{$\hat{g}: \Fh \lra \Gh$ is surjective, but not injective}
\end{equation}
and 
\begin{equation}\label{gAE2}
\mbox{$g_\C: F^\hilb_\C \lra G^\hilb_\C$ is an isomorphism.}
\end{equation}
Such a morphism satisfies $\bf AE_1$ and  $\bf AE_2$, but not $\bf AE_3$.

For instance, we may choose $R$ in $\R^\ast_+$ and $\alpha$ in $]0, R[$, and take for $\Gbh$ the ``arithmetic Hardy space" defined in Section \ref{ArHB}:
$$\Gbh := \Hbh(R):= (\Z[[X]], H^2(R), j_0)$$
where $j_0: H^2(R) \lra \C[[X]] = \Z[[X]] \widehat{\otimes}_\Z \C$ maps an holomorphic function $\phi \in H^2(R)$ to its Taylor series at the origin, and for $\Fbh$ the pro-Hermitian vector bundles over $\Spec \Z$
$$\Fbh := (\Z[[X]] \oplus \Z[[X]], H^2(R), i)$$
associated to the injective map with dense image
$$i : H^2(R) \lra \C[[X]] \oplus \C[[X]]$$
defined by 
$$i(\phi) := (j_0 (\phi), j_\alpha(\phi)),$$
where $j_\alpha(\phi) := j_0( \phi(. +\alpha))$ denotes the Taylor series  of $\phi$ at $\alpha$.

Then the maps
$$
\begin{array}{rrcl}
 \hat{g}: & \Z[[X]] \oplus \Z[[X]] & \lra   & \Z[[X]]  \\
& (\widehat{\phi}, \widehat{\psi})  &  \lra  & \widehat{\phi}  
\end{array}
$$
and
$$g_\C := Id_{H^2(R)}$$
define a morphism $g: \Fbh \lra \Gbh$ which satisfies  (\ref{gAE1}) and (\ref{gAE2}). 

\subsection{Allowable morphisms in $\proVectOK$}
As in any exact category, an \emph{allowable morphism} in $\proVectOK$ is defined as a morphism that may be written as the composition $m\, e$ of some allowable epimorphism $e$ and  some allowable monomorphism $m$ (see \ref{Amorex}).

It is possible to give a characterization of allowable morphism in $\proVectOK$ similar to the characterizations of allowable monomorphisms and epimorphisms in Propositions \ref{AM} and \ref{AE}.

\begin{proposition}\label{AMor} A morphism 
$$f := (\hat{f}, (f_\sigma)_\sKC) : \Ebh \lra \Fbh$$
in $\proVectOK$ is an allowable morphism if and only if the following conditions are satisfied:

$\bf AMor_1 :$ the map $\hat{f}: \Eh \lra \Fh$ is strict, and $\coker \hat{f}$ satisfies $\bf CTC_2$ \emph{(and therefore defines an object of $\CTC_\OK$)};

$\bf AMor_2 :$ for every field embedding $\sKC,$ the continuous $\C$-linear map $f_\sigma: E_\sigma^{\hilb} \lra F_\sigma^{\hilb}$ is strict;

$\bf AMor_ 3 (= AM_3):$ for every embedding $\sKC,$ the natural map
 $$\coker  f_\sigma := F_\sigma^{\hilb}/f_\sigma(E_\sigma^{\hilb}) \lra
 \coker  \hat{f}_\sigma := \Fh_\sigma/\hat{f}_\sigma(\Eh_\sigma)$$
 is injective.

$\bf AMor_ 4 (= AE_3):$  for every field embedding $\sKC,$ the inclusion map
 $$i^{\Eh}_{\sigma \mid \ker f_\sigma} : \ker f_\sigma  \hlra \ker \hat{f}_\sigma$$
 has a dense image.
 \end{proposition}
The proof of Proposition \ref{AMor} is a variation on the proofs of Propositions \ref{AM} and \ref{AE}, and will be left to the reader. Let us only indicate that he canonical factorization
$$f=  m\, e : \Ebh \stackrel{e}{\lra} \widehat{\overline{I}} \stackrel{m}{\lra} \Fbh$$
of $f$ is constructed as follows: the object $\widehat{\overline{I}}$ and the epimorphism $e$ of $\proVectOK$ are defined as $$\widehat{\overline{I}} := (\im \hat{f}, (\im f_\sigma, i_{\sigma \mid \im f_\sigma}^{\Fbh})_\sKC)$$
and 
$$e:= (\hat{f}: \Eh \ra \im \hat{f}, (f_\sigma: E_\sigma^\hilb \ra \im f_\sigma)_\sKC),$$ and the monomorphism $m$ is defined by the inclusion morphism $\im \hat{f} \hra \Fh$ and $\im f_\sigma \hra F_\sigma^\hilb.$

By applying Corollary \ref{quotfinCTC}, one easily derive from Proposition \ref{AMor}:

\begin{corollary}
 Let $f: \Ebh \lra \Fbh$ be a morphism in $\proVectOK$ such that either $\Ebh$ or $\Fbh$ has finite rank. Then $f$ is an allowable morphism if and only if $\hat{f} (\Eh)$ is a saturated $\OK$-submodule of $\Fh.$  \qed
\end{corollary}

Using that a morphism in $\CTC_\OK$ or $\Hilbcont$ is strict if its image is closed, an easy diagram chasing argument (that will be left to the reader) allows one to derive from Proposition \ref{AMor}:

\begin{corollary}\label{All4} Let us consider a diagram in $\proVectOK$:
$$\Ebh \xrightarrow{u} \Fbh \xrightarrow{v} \Gbh \xrightarrow{w} \Hbh.$$
If $\hat{w}: \Gh \lra \Hh$ is strict and if the diagram
 $$\Eh \xrightarrow{\hat{u}} \Fbh \xrightarrow{\hat{v}} \Gbh \xrightarrow{\hat{w}} \Hbh$$
 is an exact sequence of $\OK$-modules,  and if, for every embedding $\sKC$, the diagram
 $$E^\hilb_\sigma \xrightarrow{u_\sigma} F^\hilb_\sigma \xrightarrow{v_\sigma} G^\hilb_\sigma \xrightarrow{w_\sigma} H^\hilb_\sigma$$
 is an exact sequence of $\C$-vector spaces, then $v$ is an allowable morphism in $\proVectOK$. \qed
\end{corollary}

\subsection{Acyclic complexes and quasi-isomorphisms}

\subsubsection{}The additive category $\proVectOK$ is easily seen to be \emph{idempotent complete} (\cf \ref{DefNot}). Indeed, if $p: \Ebh \lra \Ebh$ is an idempotent endomorphism in $\proVectOK$, then one defines an object $\Kbh$ of $\proVectOK$ by letting:
$$\Kbh := (\ker \hat{p}, (\ker p_\sigma, i^{\Ebh}_{\sigma \mid \ker p_\sigma})_\sKC),$$
and the inclusion morphisms $\ker \hat{p} \hra \Eh$ and $\ker p_\sigma \hra E_\sigma^\hilb$ define a morphism 
$$\iota: \Kbh \lra \Ebh$$
that is easily checked to be a kernel of $p$.

\subsubsection{}  The structure of exact category of $\proVectOK$ allows one to define acyclic complexes in  $\proVectOK$ and quasi-isomophisms among the chain maps in $\mathbf{Ch}(\proVectOK)$, and then to construct the derived category $\bf{D}(\proVectOK)$ (see \ref{acyclicquasiiso} and  \ref{defder}).  Since  $\proVectOK$ is  idempotent complete, they satisfy the good  formal properties summarized in \ref{Karoub}.
 
 The  acyclic complexes  in $\proVectOK$ actually admit a simple description:

\begin{proposition}\label{proacyclic} A complex in $\proVectOK$
 $$\Ebh^\bullet : \quad \ldots \lra \Ebh^{i-1} \xrightarrow{d^{i-1}} \Ebh^{i} \xrightarrow{d^{i}} \Ebh^{i+1}{\lra} \dots$$
 is acyclic if and only if the complex
 $$\Eh^\bullet=\widehat{\digamma}( \Ebh^\bullet) :  \quad \ldots \lra \Eh^{i-1} \xrightarrow{\hat{d}^{i-1}} \Eh^{i} \xrightarrow{\hat{d}^{i}} \Eh^{i+1}{\lra} \dots$$
 and, for every embedding $\sKC,$ the complex
 $$E_\sigma^{\hilb , \bullet}= \digamma_\sigma(\Ebh^\bullet) :  \quad \ldots \lra E_\sigma^{\hilb, i-1} \xrightarrow{d_\sigma^{i-1}} E_\sigma^{\hilb, i} \xrightarrow{d_\sigma^{i}} E_\sigma^{\hilb, i+1}{\lra} \dots$$
 are exact \emph{(as complexes of $\OK$-modules and of $\C$-vector spaces, respectively).} 
\end{proposition}

\begin{proof} According to the description of the canonical factorization of allowable morphisms indicated after Proposition \ref{AMor} above in $\proVectOK$, it is enough to show that the boundary morphisms $d^i$ are allowable morphisms in $\proVectOK$  when $\Eh^\bullet$ and $E_\sigma^{\hilb , \bullet}$ are exact. This directly follows from Corollary \ref{All4}.
 \end{proof}

The additive categories $\CTC_\OK$ and $\Hilbcont$ may be equipped with the ``obvious" structures of exact category alluded to in the introduction to this Section \ref{proexact}. A simpler variant of the last proof shows that the exactness of $\Eh^\bullet$ (resp. of $E_\sigma^{\hilb , \bullet}$) as a complex of $\OK$-modules (resp. of $\C$-vector space) is equivalent to its acyclicity in $\CTC_\OK$ (resp. in $\Hilbcont$).

Applied to the mapping cone of a chain map between two complexes in $\proVectOK$, Proposition \ref{proacyclic} implies:
\begin{corollary} A chain map $f^\bullet: \Ebh^\bullet \lra \Fbh^\bullet$ in $\mathbf{Ch}(\proVectOK)$ is a quasi-isomorphism if and only if the chain map
$$\hat{f}^\bullet = \widehat{\digamma}(f^\bullet): \Eh^\bullet \lra \Fh^\bullet$$
and, for every embedding $\sKC,$ the chain map
$$f_\sigma^\bullet = \digamma_\sigma(f^\bullet): E_\sigma^{\hilb, \bullet} \lra F_\sigma^{\hilb, \bullet}$$
are quasi-isomorphisms of complexes of $\OK$-modules and of $\C$-vector spaces, respectively. \qed
 \end{corollary}

\chapter{Infinite dimensional vector bundles \\ over smooth projective curves}\label{bundlesovercurves}\label{geoman}

\medskip

Let $C$ be smooth, projective, geometrically irreducible curve over some field $k$, and let $K:= k(C)$ be its field of rational functions, and $g$ its genus.

In this chapter, we discuss various results concerning infinite rank vector bundles over $C$, that constitute ``geometric" counterparts, concerning infinite rank vector bundles over the curve $C$, of the main results of this monograph, which concern infinite rank vector bundles over the  ``compactified arithmetic curve"  attached to a number field $K$.

In this context, the real valued invariants $\dega \Eb$ and $\hot(\Eb)$ associated to some Hermitian vector bundles are replaced by their classical integral valued versions $\deg_C E$ and 
$$h^0(C,E) := \dim_k \Gamma(C,E),$$
associated to a vector bundle $E$ over $C$. The ``infinite rank" avatars of $h^0(C,E)$ take their value in $\N \cup \{+\infty\}$, and the measure theoretic arguments on which rely the proofs of the main theorems in Chapter \ref{SectionSum} are replaced by elementary combinatorial arguments, involving algebraic Mittag-Leffler conditions. 
Typically, arguments relying on the fact that a non-increasing sequence in $\R_+$ is convergent are replaced by an appeal to the fact that a non-increasing sequence in $\N$ is eventually constant. 

In this simplified geometric framework, the relation between the invariants $h^0(C,\Eh)$ and $\overline{h}^0(C,\Eh)$ attached to some pro-vector bundle $\Eh$ over $C$ --- which play the role of the invariants $\lhot(\Ebh)$ and $\uhot(\Ebh)$ attached to some pro-Hermitian vector bundle
$\Ebh$ over $\Spec \OK$ --- may be more easily investigated. We notably establish a simple geometric counterpart of the results of Chapter  \ref{SectionSum} (see Proposition \ref{GeomSectionSum}). 

In Section \ref{wildpro} we  construct examples of ``wild" pro-vector bundles over smooth projective curves such that the invariants $h^0(C,\Eh)$ and $\overline{h}^0(C,\Eh)$ do \emph{not} coincide. Such examples exist only when $g>0$, and may be seen as geometric analogues of the pro-Euclidean lattices considered in Paragraph \ref{WildEl}. They are however constructed by a different technique, starting from a sequence $(L_i)_{i \geq 1}$ of line bundles over $C$ and of successive extensions of $L_i$ by $L_{i+1}$.

\section{Pro-vector bundles over smooth curves}\label{provectsmoothcurve}

In this section, we consider some smooth (separated integral) curve $C$ over some base field $k$. (This curve $C$ may possibly be affine, although we are chiefly interested in the projective case.)

\subsection{Definitions}
 Recall that the structure sheaf $\cO_C$ of $C$ becomes a sheaf of topological rings when, for every open subscheme $U$ of $C$, the ring $\cO_C(U)$ is equipped with the discrete topology. Indeed the topology of $C$ is noetherian, hence any open subscheme of $C$ is quasi-compact.

This observation allows us to consider \emph{sheaves of topological $\cO_C$-modules} over $C$ and over its open subschemes.

We shall say that a sheaf of topological $\cO_U$-modules $\Eh$ over 
$C$ is a \emph{pro-vector bundle} over $C$ when it satisfies the following two conditions:

$\mathbf{Pro_1 :}$ \emph{For any affine open subscheme $V$ of $C$, the topological $\cO_C(V)$-module $\Eh(V)$ is an object of $CTC_{\cO_C(V)}$.}

$\mathbf{Pro_2 :}$ \emph{For any two non-empty open affine subschemes $V$ and $V'$ of $C$ such that $V' \subset V,$ the restriction morphism
$$\rho_{V' V}: \Eh(V) \lra \Eh(V')$$
induces an isomorphism}
\begin{equation}\label{EbVVprime}
\Eh(V) \hat{\otimes}_{\cO_C(V)} \cO_C(V') \lrasim \Eh(V').
\end{equation} 

The full subcategory of the $k$-linear category of sheaves of topological $\cO_C$-modules the object of which are the so defined  pro-vector bundles over $C$ will be denoted $\rm{proVect}_C$.

It follows from Proposition \ref{Utens} that, for any two open affine subschemes $V$ and $V'$ of $U$ such that $V' \subset V,$ the isomorphism (\ref{EbVVprime}) determines a canonical bijection:
$$ .\, \widehat{\otimes}_{\cO_C(V)} \cO_C(V'):  \cU(\Eh(V))  \lrasim \cU(\Eh(V'))$$
between the set of open saturated submodules of $\cU(\Eh(V))$ and $\cU(\Eh(V')).$ We shall denote by $\cU(\Eh)$ the limit of the essentially constant inductive system of the $\cU(\Eh(V))$, where $V$ varies over the non-empty affine subschemes of $C$. 

The set $\cU(\Eh)$ may also be identified with the set $\cU(\Eh_{k(C)})$ of the open $k(C)$-vector spaces in the linearly compact $k(C)$-vector space $\Eh(k(C))$ deduced from $\Eh(V)$ by the base change $k(C) \hra \cO_C(V)$ for any non-empty affine subscheme $V$ of $C$.

We shall say that a sequence $(O_i)_{i \in \N}$ in $\cU(\Eh)^\N$ is a \emph{filtration defining the topology of $\Eh$}, or shortly a \emph{defining filtration} in $\cU(\Eh)^\N$, when for some (or equivalently, for every) non-empty affine open subscheme $V$ of $U$, the sequence $(O_i)$, seen as a sequence in $\cU(\Eh(V))$, it is a
nested sequence
\begin{equation*}\label{Oi}
O_0 \hookleftarrow O_1 \hookleftarrow O_2 \hookleftarrow \dots
\end{equation*}
of submodules in $\cU(\Eh(V))$ which constitutes a basis of neighborhoods of $0$ in $\Eh(V)$, or equivalently, is cofinal in  $\cU(\Eh(V))$.

For any closed point $P$ in $C$, we may consider the local ring $\cO_{C,P}$ and its completion $\widehat{\cO}_P$. If denotes some open affine neighborhood of $P$ in $C$, to any pro-Hermitian vector bundle $\Eh$ over $C,$ we may associate the objects
$$\Eh_{\cO_{C,P}} := \Eh(V) \widehat{\otimes}_{\cO(V)} \cO_{C,P}$$
and $$\Eh_{\widehat{\cO}_P} := \Eh(V) \widehat{\otimes}_{\cO(V)} \widehat{\cO}_P$$
in $\CTC_{\cO_{C,P}}$ and $\CTC_{\widehat{\cO}_P}$ respectively.

\subsection{Pro-vector bundles as projective limits of vector bundles}

One easily checks that any element $O$ in $\cU(\Eb)$ defines a subsheaf of (open) $\cO_C$-submodules of $\Eh$: for any non-empty affine open subscheme $V$ of $C$, $O(V)$ is the submodule of $\Eh(V)$ defined by $O$ considered as an element of $\cU(\Eh(V)).$

By the very definition of $\cU(\Eh)$, the quotient topological $\cO_C$-module $\Eh/O$ is a vector bundle, namely a sheaf of discrete $\cO_\C$-modules that is locally free of finite rank. Moreover this construction allows us to identify $\cU(Eh)$ with the set of open subsheaves of $\cO_C$-submodules in $\Eh$ which satisfy this property. 

For any  filtration $(O_i)_{i \in \N}$ in $\cU(\Eh)^\N$ defining the topology of $\Eh,$ we may form the projective system of surjective morphisms
\begin{equation*}
\Eh / O_0 \longleftarrow \Eh / O_1 \longleftarrow \Eh / O_2 \longleftarrow \cdots \longleftarrow \Eh / O_i\longleftarrow \Eh / O_{i+1} \longleftarrow \cdots
\end{equation*}
defined by the canonical quotient maps between the vector bundles $\Eh/O_i$. It is straighforward that $\Eh$ (equipped with the quotient morphisms $\Eh \lra \Eh/O_i$) is ``the"  projective limit of this system in the category of topological $\cO_C$-Modules.

Conversely, if we are given a projective system of surjective morphisms of vector bundles over the curve $C$,
$$E_\bullet: E_0 \stackrel{q_0}{\longleftarrow} E_1 \stackrel{q_1}{\longleftarrow} E_2 \stackrel{q_2}{\longleftarrow} \cdots
\stackrel{q_{i_1}}{\longleftarrow} E_i \stackrel{q_{i-1}}{\longleftarrow} E_{i+1} \stackrel{q_{i+1}}{\longleftarrow} \cdots,$$
we may form its projective limit 
$$\Eh := \varprojlim_k E_k$$
in this category, and one easily checks that it is an object of $\rm{proVect}_C$.
Moreover, for every $i \in \N,$ the kernel 
$$O_i := \ker p_i$$ of the projection
$$p_i: \Eh = \varprojlim_k E_k \lra E_i$$
 belongs to $\cU(\Eh)$ and $(O_i)_{i\in \N}$ is a filtration defining the topology of $\Eh.$
 
 This description of the objects of $\rm{proVect}_C$ in terms of projective systems of vector bundles over $C$ extends to morphisms in $\rm{proVect}_C$. Indeed, if 
 $$E'_\bullet: E'_0 \stackrel{q'_0}{\longleftarrow} E'_1 \stackrel{q'_1}{\longleftarrow} E'_2 \stackrel{q'_2}{\longleftarrow} \cdots
\stackrel{q'_{i_1}}{\longleftarrow} E'_i \stackrel{q'_{i-1}}{\longleftarrow} E'_{i+1} \stackrel{q'_{i+1}}{\longleftarrow} \cdots$$
denotes another projective system of surjective morphisms of vector bundles over $C$, of projective limit
$$\Eh' := \varprojlim_k E'_k,$$
then we have a canonical isomorphism:
$$\Hom_{\rm{proVect}_C} (\Eh, \Eh') \lrasim \varprojlim_{i'}\varinjlim_i(E_i, E'_{i'}).$$

The following lemma is a simple application of the description of pro-vector bundles over $C$ as projective limits of vector bundles.

\begin{lemma}\label{finitesub}
 For any pro-vector bundle $\Eh$ on $C$ and for any finite family $(s_1, \ldots, s_n)$ of elements of $\Gamma(U,\Eh),$ the subsheaf
 $$\cF:= \sum_{k=1}^n \cO_C\, s_k$$
 of $\cO_C$-modules of $\Eh$-generated by $s_1, \ldots s_n$ is locally free. 
\end{lemma}

\begin{proof} We may assume that $\Eh$ is the projective limit $\varprojlim_i E_i$ of a projective system of surjective morphisms of vector bundles $E_\bullet$, and denote by $p_i: \Eh \lra E_i$ the projection morphisms. 

For every $i\in \N,$ the image
$$p_i(\cF)  = \sum_{k=1}^n \cO_C \,p_i(s_k)$$
of $\cF$ in $E_i$ is a torsion free coherent subsheaf of $E_i$. It is therefore locally free of rank at most $n$. Moreover the canonical maps $q_i : E_i \lra E_{i-1}$ define by restrictions some surjective morphisms of vector bundles
$q_{i\mid F_i} : F_i \lra F_{i-1}.$
The sequence of integers $(\rk F_i)_{i \in \N}$ is therefore non-decreasing and bounded. It is thefore eventually constant. 

This shows that there exists $i_0$ in $\N$ such that, for any $i > i_0,$ $q_i$ defines an isomorphism
\begin{equation}\label{isoizero}
q_{i\mid F_i} : F_i \lrasim F_{i-1}.
\end{equation}
For any point $P$ of $C,$ we may find a subset $A$ of $\{1,\ldots, n\}$ such that $(p_{i_0}(s_\alpha))_{\alpha \in A}$ is a basis of the fiber $F_{i_0, P}$. Then, if $V$ is a small enough open neighborhood of $P$ in $C$, the map 
\begin{eqnarray*}
  \cO_V^A & \lra  & F_{i_0 \mid V}  \\
(f^\alpha)_{\alpha \in A}  & \longmapsto & \sum_{\alpha \in A} f^\alpha p_{i_0}(s_\alpha)   
\end{eqnarray*}
is an isomorphism of $\cO_V$-modules. The isomorphisms (\ref{isoizero}) now imply that
\begin{eqnarray*}
  \cO_V^A & \lra  & \cF_{\mid V}  \\
(f^\alpha)_{\alpha \in A}  & \longmapsto & \sum_{\alpha \in A} f^\alpha s_\alpha   
\end{eqnarray*}
is an isomorphism of $\cO_V$-modules.
\end{proof}

\subsection{Descent properties and pro-vector bundles over smooth projective curves}\label{desprovectpro} 

When the curve $C$ is affine, the functor $\Gamma(C,.)$ establishes an equivalence between the category $\rm{proVect}_C$ of pro-vector bundles over $C$ and the category $\CTC_A$, where $A$ denotes the Dedekind ring $A:= \Gamma(C, \cO_C)$. 

When $C$ is projective, the descent properties of the categories $\CTC_.$ discussed in Section \ref{LocDes} allow one to give some ``concrete" descriptions of  $\rm{proVect}_C$. 

 For instance,  if $U$ and $V$ are two open affine (necessarily dense) subscheme of $C$ such that $C = U \cup V,$ then $\rm{proVect}_C$,  by using the descent result in Paragraph \ref{LocDesUV}, one may describe the category  $\rm{proVect}_C$ in terms of the categories $\CTC_{A_U}$, $\CTC_{A_V}$, $\CTC_{A_{U \cap \V}}$ associated to the Dedkind rings $A_U:= \Gamma(U, \cO_C),$ $A_V:= \Gamma(V, \cO_C),$ and $A_{U\cap V}:= \Gamma(U \cap V, \cO_C).$ 
 
 Namely, the datum of some pro-vector bundle $\Eh$ over $C$ is equivalent to the data of the objects
 $N_U := \Gamma(U, \Eh)$  and $N_V:= \Gamma(V, \Eh)$
 in $\CTC_{A_U}$ in $\CTC_{A_V}$ respectively, together with the glueing isomorphism in $\CTC_{A_{U \cap \V}}$:
 $$N_U \widehat{\otimes}_{A_U} A_{U\cap V} \lrasim \Gamma(U\cap V, \Eh) \lrasim N_V \widehat{\otimes}_{A_V} A_{U\cap V}.$$
 
When the smooth  projective curve $C$ is equipped with a non-empty finite set $\Sigma$ of closed points, by means of the descent results in paragraph \ref{LocDesUp}, we may give a description of pro-vector bundles over $C$ that is formally similar to the definition of pro-Hermitian vector bundles over $\Spec \OK$ when $K$ is a number field. 

Indeed, starting from $C$ and $\Sigma,$ we may introduce the Dedekind ring
$$A:= \Gamma(C \setminus \Sigma, \cO_C)$$
of field of fractions $K:= k(C)$, and the completions $\cOh_P$ of the local ring $\cO_{C,P}$ of $C$ at the points $P$ of $\Sigma$. The field of fractions $\Kh_P$ of $\cOh_P$ may also be equipped with the non-archimedean absolute value
$$\vert . \vert_p := e^{-v_P} : \Kh_P \lra \R_+,$$
and may  be identified with the $P$-adic completion of $K$. (The ring $A$ and the normed fields $(\Kh_P, \vert.\Vert_P)$ play the role of the ring of integers of some number field $K$ and of its archimedean completions.)

Using the descent result in \ref{LocDesUp}, one easily shows that the datum of some pro-vector bundle $\Eh$ over $C$ is equivalent to the data of the object
$$N: = \Gamma(C\setminus \Sigma, \Eh)$$
in $CTC_A$, of the objects $\Eh_{\cOh_P}$ in $CTC_{\cOh_P}$ for $P \in \Sigma$, and of the ``glueing isomorphisms"
$$N_{\Kh_P} := N \widehat{\otimes}_A \Kh_P \lrasim \Eh_{\Kh_P} := \Eh_{\cOh_P} \widehat{\otimes}_{\cOh_P} \Kh_P.$$

The topological $A$-module $N$ plays the role of the topological $\OK$-module $\Eh$ underlying a pro-Hermitian vector bundle $\Ebh$ over $\Spec \OK,$ and the inclusions $\Eh_{\cOh_P} \hra   N_{\Kh_P}$ the role of the  inclusions  of the Hilbert spaces $(E_\sigma^\hilb, \Vert .\Vert_\sigma)$ into the spaces  $\Eh_\sigma:= \Eh \widehat{\otimes}_{\OK,\sigma} \C$.

\section{The invariants $h^0(C,\Eh)$ and $\overline{h}^0(C,\Eh)$}

In the following sections of this chapter, we denote by 
$C$ a smooth \emph{projective} geometrically irreducible curve over some base field $k$. We also denote by $K:= k(C)$ be its field of rational functions, and by $g$ its genus.
For any coherent sheaf of $\cO_C$-modules $F$ over $C$, notably for any vector bundle, the dimension $\dim_k \Gamma(C, F)$ of its space of sections will be denoted $h^0(C, F)$.

\subsection{Definitions} We may define  generalizations of $h^0(C, F)$ that make sense for any pro-vector bundle $\Eh$ over $C$, and not only for vector bundles, and constitute the ``function field" counterparts of the invariants $\lhot(\Ebh)$ and $\uhot(\Ebh)$ attached to pro-Hermitian vector bundles investigated in this monograph. 

These invariants will take their value in $\N \cup \{+\infty\}$, and will be defined as follows.

Firstly, we may consider the $k$-vector space $\Gamma(C, \Eh)$ of global sections of $\Eh$ and  its dimension in  $\N \cup \{+\infty\}$:
$$h^0(C, \Eh) := \dim_k \Gamma(C, \Eh).$$
Indeed, it follows from Lemma \ref{finitesub} that $h^0(C, \Eh)$ coincides with the supremum of the integers $h^0(C, E')$ where $E'$ varies over the vector bundles $E'$ over $C$ such that there exists an injective morphism of $\cO_C$-modules $\iota: E' \lra \Eh$, and is therefore the analog of $\lhot(\Ebh)$ defined by (\ref{lhot}).

Secondly, we may consider the lower limit of the dimension of the $k$-vector space $\Gamma(C, \Eh/O)$ of sections of the vector bundle $\Eh/O$, where $O$ belongs to the filtered set $\cU(\Eh)$:
$$\overline{h}^0(C, \Eh) := \liminf_{O \in \cU(\Eh)} h^0(C, \Eh/O).$$
This is a straightforward analogue of the definition (\ref{uhot}) of $\uhot(\Ebh)$.
 
\begin{proposition}
 For any pro-vector bundle $\Eh$ over $C,$ we have:
 \begin{equation}\label{houho}
h^0(C, \Eh) \leq \overline{h}^0(C, \Eh).
\end{equation}
\end{proposition}

\begin{proof} We have to show that, for any finite family $(s_1, \ldots, s_n)$ of $k$-linearly independent elements of $\Gamma(C,\Eh),$
we have:
\begin{equation}\label{ineq:houho}
\overline{h}^0(C, \Eh) \geq n.
\end{equation}

To achieve this, we consider the subsheaf $$\cF:= \sum_{i=1}^n \cO_U s_i$$
 of $\cO_C$-modules of $\Eh$-generated by $s_1, \ldots s_n$. According to Lemma \ref{finitesub}, it is a vector bundle over $C$.
 
 Let us also consider the injection
 $$\iota: \cF \hlra \Eh$$
 and the quotient morphisms
 $$p_U: \Eh \lra E_U,$$
 for $U \in \cU(\Eh) (:=\cU(\Eh_K)).$
 By base change to the generic point $\Spec K$ of $C,$ these maps  become the $K$-linear injection
 $$\iota_K: \cF_K \hlra \Eh_K$$
 and surjections
$$p_{U,K}: \Eh_K \lra E_{U,K} = \Eh_K/U$$
of $K$-vector spaces.

As $\bigcap_{U \in \cU(\Eh_K)} U = \{0\}$ and $\cF_K$ is a finite dimensional $K$-vector space, when $U$ in $\cU(\Eh_K)$ is small enough --- say when $U\subset U_0$ --- we have:
$$\iota_K(\cF_K) \cap U = \{0\}.$$
Then 
$$p_{U,K} \circ \iota_K : \cF_K \lra E_{U,K},$$
and therefore
$$p_U\circ \iota: \cF \lra E_U,$$
is an injective morphism, and 
$$p_U\circ \iota : \Gamma(C,\cF) \lra \Gamma(C, E_U)$$
is an injective $k$-linear map.

This shows that, for any $U$ in $\cU(\Eh_K)$ such that $U\subset U_0$, we have
$$h^0(C, E_U) := \dim_k \Gamma(C, E_U) \geq h^0(C, \cF) := \dim_k \Gamma(C,\cF) \geq n.$$
This establishes (\ref{ineq:houho}).
\end{proof}

\subsection{A geometric analogue of strongly summable pro-Hermitian vector bundles} 

The following proposition is a straightforward consequence of  definitions:

\begin{proposition}\label{lhofinite}
 Let $\Eh$ be a pro-vector bundle over $C$. The following conditions are equivalent:
 
 1) $\overline{h}^0(C, \Eh) < + \infty.$

2) There exists a defining filtration $(O_i)_{i \in \N}$ in $\cU(\Eh)$ such that the sequence $(h^0(C, \Eh/O_i))_{i \in \N}$ is bounded.

3) There exists an integer $n \in \N$ and a defining filtration $(O_i)_{i \in \N}$ in $\cU(\Eh)$  such that, for every $i \in \N,$ 
\begin{equation}\label{3a}
h^0(C, \Eh/O_i) = n
\end{equation}
and, for any $O'$ in $\cU(\Eh)$ contained in $O_0$, $$h^0(C, \Eh/O) \geq n.$$

When these conditions are realized, $\overline{h}^0(C, \Eh) = n.$ \qed
\end{proposition}

\begin{lemma}\label{filtrationlemma}
  Let $\Eh$ be a pro-vector bundle over $C$ such that $\overline{h}^0(C, \Eh) < + \infty, $ and let $(O_i)_{i \in \N}$ be a defining filtration in $\cU(\Eh)$ that satisfies the conditions in Proposition \ref{lhofinite}, 3).
  
 1)  For any $i \in \N,$ the maximal slope\footnote{Recall that the maximal slope $\mu_{\rm max}(F)$ of some vector bundle $F$ on $C$ is defined as the supremum of the slopes $\mu(F'):= \deg_C F' /\rk F'$ of the sub-vector bundles $F'$ of positive rank of $F$. It is $-\infty$ if $\rk F =0.$} of the vector bundle 
  $$T_i := O_0/O_i = \ker (\Eh/O_i \lra \Eh/O_0)$$
  satisfies
  \begin{equation}\label{mumaxTi}
\mu_{\rm max}(T_i) \leq g-1.
\end{equation}

2) If moreover $h^0(C,\Eh) = \overline{h}^0(C, \Eh),$ then, for every large enough $i \in \N,$ the quotient map 
$$q_i: \Eh/O_{i+1} \lra \Eh/O_i$$
induces an isomorphism between spaces of sections:
\begin{equation*}
q_i: \Gamma(C,\Eh/O_{i+1}) \lrasim \Gamma(C,\Eh/O_i).
\end{equation*}
\end{lemma}

\begin{proof} 1) We have to show that, for every non-zero sub-vector bundle $T$ of $T_i,$
$$\mu(T) \leq g-1;$$
or equivalently by Riemann-Roch,
\begin{equation}\label{chineg}
\chi(C, T) := h^0(C,T) - h^1(C,T) \leq 0.
\end{equation}

In the short exact sequence
\begin{equation}\label{shorT}
0 \lra T \lra E_i := \Eh/O_i \lra E_i /T \lra 0,
\end{equation}
the quotient $E_i/T$ may be identified with the vector bundle $\Eh/O$ for some $O$ in $\cU(\Eh)$ satisfying
$$O_i \subset O \subset O_0.$$
Consequently,
\begin{equation}\label{etdetrois}
h^0(C, E_i/T) = h^0(C, \Eh/O) \geq n = h^0(C, E_i). 
\end{equation}

Besides, the long exact sequence of cohomology groups deduced from (\ref{shorT}) shows that
\begin{equation}\label{etdequatre}
h^1(C,E_i) \geq h^1(C,E_i/T)
\end{equation}
and
\begin{equation}\label{etdecinq}
\chi(C,T) = \chi(C,E_i) - \chi(C, E_i/T) = h^0(C,E_i) - h^1(C, E_i) - h^0(C,E_i/T) + h^1(C, E_i/T).
\end{equation}

The inequality (\ref{chineg}) follows from (\ref{etdetrois}), (\ref{etdequatre}), and (\ref{etdecinq}). 

2) The morphisms 
$$p_i: \Gamma(C, \Eh) \lra \Gamma(C, \Eh/O_i)$$ define an isomorphism
$$\Gamma (C, \Eh) \lrasim \varprojlim_i \Gamma(C, \Eh/O_i).$$
Therefore
$$\bigcap_{i \in \N} \ker p_i = \{ 0\}.$$
Since $(\ker p_i)_{i \in \N}$ is a non-increasing sequence of vector subspaces of the finite dimensional $k$-vector space $\Gamma(C, \Eh)$, this shows that, for $i$ large enough, $\ker p_i = \{ 0\}$, and therefore $p_i$ is an injective linear map between two $k$-vector spaces of the same dimension, hence an isomorphism.
 \end{proof}
 
 The following proposition may be seen as an analogue, in the function field case, of Corollary \ref{strongsummableeq}, that led us to define the strongly summable pro-Hermitian vector bundles.

\begin{proposition}\label{GeomSectionSum}
 For any pro-vector bundle $\Eh$ over $C$, the following two conditions are equivalent:
 
 1) $h^0(C,\Eh) = \overline{h}^0(C, \Eh) < + \infty.$
 
 2) There exists a defining filtration $(O_i)_{i \in \N}$ in $\cU(\Eh)$ such that the  quotient morphisms 
 $$q_i: \Eh/O_{i+1} \lra \Eh/O_i$$
 satisfy
 \begin{equation}\label{Gammakerqi}
\Gamma(C, \ker q_i) = 0 \mbox{ for $i$ large enough.}
\end{equation}
 
\end{proposition}

\begin{proof} The implication $1) \Rightarrow 2)$ follows from Lemma \ref{filtrationlemma}, 2).

Conversely, let us assume that 2) is satisfied. Then, for $i$ large enough --- say for $i \geq i_0$ --- the $k$-linear maps
\begin{equation}\label{qiGamma}
q_i: \Gamma(C,\Eh/O_{i+1}) \lra\Gamma(C,\Eh/O_i)
\end{equation}
are injective. Consequently the sequence of non-negative integers $(h^0(C,\Eh/O_i))_{i \geq i_0}$ is non-in\-crea\-sing. Therefore there exists some integer $i_1 \geq i_0$ such that the sequence $(h^0(C,\Eh/O_i))_{i \geq i_1}$ is constant. Then the maps (\ref{qiGamma}) are isomorphisms for $i \geq i_1.$

Therefore the maps in the projective systems
$$\Gamma(C, \Eh/O_{i_1}) \longleftarrow \Gamma(C, \Eh/O_{i_1 +1}) \longleftarrow \cdots$$
are isomorphisms, and therefore define an isomorphism
$$\Gamma(C, \Eh) \lrasim \Gamma(C,\Eh/O_{i_1}).$$
This shows that
$$h^0(C, \Eh) = h^0(C, \Eh/O_{i_1})$$
and that
$$\overline{h}^0(C,\Eh) \leq \liminf_{i \ra +\infty} h^0(C, \Eh/O_{i}) = h^0(C, \Eh/O_{i_1}).$$

Together with the inequality (\ref{houho})
$$h^0(C, \Eh) \leq \overline{h}^0(C, \Eh),$$
this establishes 1).
\end{proof}

\begin{proposition}\label{gzeroeq} When $g=0,$ for any pro-vector bundle $\Eh$ over $C$ such that $\overline{h}^0(C, \Eh) < + \infty,$ we have:
$$h^0(C,\Eh) = \overline{h}^0(C, \Eh).$$
 \end{proposition}
 
 \begin{proof} 
 We choose a defining filtration $(O_i)_{i \in \N}$ in $\cU(\Eh)$ that satisfies Condition 3) of Proposition \ref{lhofinite} and we apply Lemma \ref{filtrationlemma}, 1). The latter shows that the vector bundles $T_i$ defined by the exact sequences
 \begin{equation*}
0 \lra T_i \hlra E_i := \Eh/O_i \stackrel{p_i}{\lra} E_0 := \Eh/O_0 \lra 0
\end{equation*}
satisfy 
$$\mu_{\rm max}(T_i) \leq -1.$$

Therefore $\Gamma(C,T_i) = 0,$ and the map
\begin{equation}\label{pinj}
p_i:\Gamma(C, E_i) \lra \Gamma(C, E_0) 
\end{equation}
is injective. According to (\ref{3a}), the $k$-vector spaces $\Gamma(C, E_i)$ all have the same dimension $n := h^0(C, E_0),$ and therefore the maps (\ref{pinj}) are isomorphisms. This shows that the maps in the projective systems
$$\Gamma(C, \Eh/O_0) \longleftarrow \Gamma(C, \Eh/O_i) \longleftarrow \cdots$$
are isomorphisms, and therefore define an isomorphism
$$\Gamma(C, \Eh) \lrasim \Gamma(C,\Eh/O_0).$$
Finally,
$$h^0(C, \Eh) = \dim_k \Gamma(C, \Eh) = \dim_k \Gamma(C,\Eh/O_0) =\overline{h}^0(C, \Eh).$$
\end{proof}

\section{Successive extensions and  wild pro-vector bundles over projective curves}\label{wildpro}

In this section, we assume that we are given:

\noindent $\bullet$ \emph{a sequence $(L_i)_{i \geq 1}$ of line bundles over $C$;} 

\noindent $\bullet$
\emph{  for every integer $i \geq 2,$ a class $\alpha_i$ in }
$\Ext^1_{\cO_C}(L_{i-1}, L_i) \lrasim H^1(C, L^\vee_{i-1} \otimes L_i).$ 
\subsection{A construction}\label{Construct} 

From the above data,  we may construct a projective system of vector bundles over $C$:
$$ E_\bullet : \;\; E_0 \stackrel{q_0}{\longleftarrow} E_1 \stackrel{q_1}{\longleftarrow} E_2 \stackrel{q_2}{\longleftarrow} \cdots
\stackrel{q_{i-1}}{\longleftarrow} E_i \stackrel{q_{i}}{\longleftarrow} E_{i+1} \stackrel{q_{i+1}}{\longleftarrow} \cdots$$
such that the following two conditions are satisfied:

\noindent  $\mathbf{Cons_1 :}$ \emph{for any $i \in \N,$ the vector bundle $E_i$ has rank $i$ and $q_i$ is surjective;}

\noindent $\mathbf{Cons_2 :}$  \emph{for any integer $i\geq 1,$ there exists an isomorphism
$$j_i: L_i \lrasim \ker q_i;$$
moreover, when  $i\geq 2,$ the class $\tilde{\alpha}_i$ in $\Ext^1_{\cO_C}(E_{i-1}, L_i)$ of the $1$-extension 
$$\cS_i: \;\; 0 \lra L_i \stackrel{j_i}{\lra} E_i \stackrel{q_{i-1}}{\lra} E_{i-1} \lra 0$$
is sent to $\alpha_i$ by the morphism
\begin{equation}\label{morphalpha}
. \circ j_{i-1} : \Ext^1_{\cO_C}(E_{i-1}, L_i)  \lra \Ext^1_{\cO_C}(L_{i-1}, L_i) 
\end{equation}
induced by the morphism} $j_{i-1}: L_{i-1} \lra E_{i-1}.$

Indeed, we may simply construct the vector bundles $E_i$ and the morphisms $q_i$ and $j_i$ inductively as follows:

\noindent $\bullet$ $E_0 := 0;$

\noindent $\bullet$ $E_1 := L_1, j_1 := Id_{L_1}, q_0:= 0.$

\noindent $\bullet$ Let us assume that $E_k$, $j_k$ and $q_{k-1}$ have been constructed for some $k \in \N_{\geq 1}$. Then we may consider the map (\ref{morphalpha}) for $i =k+1$:
$$ . \circ j_{k} : \Ext^1_{\cO_C}(E_{k}, L_{k+1})  \lra \Ext^1_{\cO_C}(L_{k}, L_{k+1}).$$
This map is surjective. (Indeed its cokernel maps injectively into $$\Ext^2_{\cO_C}( \coker j_k, L_{k+1}) \simeq \Ext^2_{\cO_C}( E_{k-1}, L_{k+1}),$$ which vanishes since $\dim C =1.$)
Therefore we may choose a class $\tilde{\alpha}_{k+1}$ in $\Ext^1_{\cO_C}(E_{k}, L_{k+1})$ sent to $\alpha_k$ by $ . \circ j_{k}$. The class $\tilde{\alpha}_{k+1}$ may be realized by some $1$-extension (of $\cO_C$-modules) of $E_k$ by $L_{k+1}$,
$$\cS_{k+1}: \;\; 0 \lra L_{k+1} \stackrel{j_{k+1}}{\lra} E_{k+1} \stackrel{q_{k}}{\lra} E_{k} \lra 0,$$
which defines $E_{k+1}, j_k$ and $q_k$.

Clearly, when the above conditions (i) and (ii) are satisfied, for every $i\geq 1,$ the morphism $j_i$ induces an injective map
$$j_i^0: H^0(C, L_i) \lra H^0(C, E_i).$$
Moreover, for every $i\geq 2,$ we may  consider the long exact sequence of cohomology groups deduced from $\cS_i$:
\begin{multline}\label{LESSi}
 0 \lra H^0(C,L_i) \stackrel{j_i^0}{\lra} H^0(C,E_i) \stackrel{q_{i-1}^0}{\lra} H^0(C,E_{i-1})
 \stackrel{\tilde{\alpha}_i \cup .}{\lra}  H^1(C,L_i) \\ \stackrel{j_i^1}{\lra} H^1(C,E_i) \stackrel{q_{i-1}^1}{\lra} H^1(C,E_{i-1})  \lra 0.
\end{multline}

\subsection{``Wild" pro-vector bundles over projective curves of genus $g \geq1$.} 
\begin{proposition}\label{horresco}
Let $E_\bullet$ be a projective system of vector bundles over $C$ that satisfies Conditions $\mathbf{Cons_1}$ and $\mathbf{Cons_2}$ above.
Let us also assume that, for every $i \geq 2,$ the cup-product by $\alpha_i$ defines an isomorphism
\begin{equation}\label{alphaisom}
\alpha_i \cup . : H^0(C, L_{i-1}) \lrasim H^1(C, L_i).
\end{equation}

1) For every $i\geq 1,$
$${\mathbf A_i :}  \;\; \mbox{ $j_i^0$ is an isomorphism}$$ 
and,  for every $i\geq 2,$ the morphisms in (\ref{LESSi}) satisfy
$${\mathbf B_i :}  \;\; \mbox{ $\tilde{\alpha}_i \cup .$ is an isomorphism, and therefore $q_{i-1}^0$ and $j_i^1$ vanish and $q_{i-1}^1$ is an isomorphism.}$$ 

2) The pro-vector bundle  $\Eh := \varprojlim_i E_i$ over $C$ satisfies:
\begin{equation}\label{howild}
h^0(C, \Eh) = 0
\end{equation}
and
\begin{equation}\label{uhowild}
\overline{h}^0(C, \Eh) = \liminf_{k \lra +\infty} h^0(C, L_k).
\end{equation}
\end{proposition}

\begin{proof}
1)  The validity of  $({\mathbf A_i})_{i\geq 1}$ and of $({\mathbf B_i})_{i\geq 2}$ is established by induction.

As $j_1$ is an isomorphism,  Condition ${\mathbf A_1}$ is satisfied. 
 
 For any $i \geq 2,$ the following diagram is commutative, as a consequence of Condition $\mathbf{Cons_1}$: 
  \begin{equation}\label{triang}
\xymatrix{H^0(C,E_{i-1}) \ar[r]^{\tilde{\alpha}_i \cup .} & H^1(C,L_i) \\
H^0(C,L_{i-1})\ar[u]^{j_{i-1}^0} \ar[ur]^{\alpha_i \cup.} &. }
\end{equation}
As ${\alpha}_i \cup .$ is assumed to be an isomorphism, this establishes the implication ${\mathbf A_{i-1}} \Rightarrow  {\mathbf B_i}$.

Finally, the implication ${\mathbf B_{i}} \Rightarrow  {\mathbf  A_i}$ follows from the exactness of (\ref{LESSi}).

2) All the morphisms 
$q^0_k: H^0(C, E_{k+1}) \lra H^0(C, E_k)$
vanish, and consequently the projective limit
$\Gamma (C, \Eh)$ of the  system 
$$\Gamma(C, E_0) \stackrel{q^0_0}{\longleftarrow} \Gamma(C, E_1) \stackrel{q^0_1}{\longleftarrow} \Gamma (C, E_2) \stackrel{q^0_2}{\longleftarrow} \cdots
\stackrel{q^0_{i_1}}{\longleftarrow} \Gamma(C, E_i) \stackrel{q^0_{i}}{\longleftarrow} \Gamma(C, E_{i+1}) \stackrel{q^0_{i+1}}{\longleftarrow} \cdots$$
is zero. This establishes (\ref{howild}).

Let $(O_n)_{n \in \N}$ be the defining filtration in $\cU(\Eh)$ defined by the kernels of the projection maps $p_n: \Eh \lra E_n.$ Recall that, for any $O$ in $\cU(\Eh)$, we denote by $p_O$ the quotient morphism from $\Eh$ to $E_O$.

  The following lemma is straightforward:
\begin{lemma} Let $O$ be an element of $\cU(\Eh)$. There exists a smallest integer $n\in \N$ such that $O_n \subset O$. 

If  $n\geq 1$ and if $p: E_n \lra E_O$ denotes the unique morphism of vector bundles over $C$
such that the following diagram is commutative:
  \begin{equation}\label{triangbis}
\xymatrix{\hE \ar[d]_{p_n} \ar[r]^{p_O} 
& E_O \\
E_n \ar[ur]_{p} &}
\end{equation}
then the composition $$p \circ j_n : L_n \lra E_O$$
 is a non-zero morphism, and therefore induces an injective map from $H^0(C, L_n)$ to  $H^0(C, E_O)$. In particular,
 \begin{equation}\label{hoineq}
h^0(C, L_n) \leq h^0(C, E_O).
\end{equation}
\qed
\end{lemma}

From (\ref{hoineq}), we derive that
$$\overline{h}^0(C, \Eh) := \liminf_{O \in \cU(Eh)} h^0(C, E_O) \geq \liminf_{n \lra +\infty} h^0(C, L_n).$$

Besides, for every $n\geq 1,$ as the maps $j^0_n$ are isomorphisms, we have:
$$h^0(C, E_n) = h^0(C, L_n).$$
Therefore, 
$$\overline{h}^0(C, \Eh) := \liminf_{O \in \cU(\Eh)} h^0(C, E_O) \leq \liminf_{n \lra +\infty} h^0(C, E_n) = \liminf_{n \lra +\infty} h^0(C, L_n).$$
\end{proof}

Using Proposition \ref{horresco}, one may easily produce ``wild" pro-vector bundles such that
$$h^0(C, \Eh) < \overline{h}^0(C,\Eh)$$ 
when the genus $g$ of $C$ is positive (at least when the base field $k$ is an algebraically closed field, or a finite field of cardinal larger than some function of $g$).

For instance, when $g=1,$ we may perform the above construction with $L_i:= \cO_C$ for every $i \geq 1,$ and with $\alpha_i$ any non-zero element in 
$$H^1(C, L_{i-1}^\vee \otimes L_i) = H^1(C, \cO_C) \simeq k.$$
Indeed, the condition  (\ref{alphaisom}) is then satisfied and therefore, according to  (\ref{howild}) and  (\ref{uhowild}), we get:
\begin{equation}\label{wildconcrete}
h^0(C, \Eh) = 0 \;\mbox{ and }\; \overline{h}^0(C, \Eh) =1.
\end{equation}

More generally, let us assume that $g$ is positive and that there exists some line bundle $L$ of degree $g-1$ over $C$ such that
$h^0(C,L)=1$ (and therefore  $h^1(C,L)=1$). This condition is satisfied for instance when $C$ and the smooth locus of its theta divisor $\Theta$ in ${\rm Pic}^{g-1}_{X/k}$ possess some $k$-rational point.

 If $s$ (resp. $t$) is a non-zero element in $H^0(C,L)$ (resp. in $H^0(C, L^\vee \otimes \omega_C) \simeq H^1(C,L)^\vee$), then their product $s\otimes t$ is a non-zero element in $H^0(C, \omega_C)$. Therefore there exists $\alpha$ in $H^1(C,\cO_C)$ such that the Serre duality pairing $\langle \alpha, s\otimes t \rangle$ is non-zero.  Then the cup-product 
 $$\alpha \cup . : H^0(C, L) \lra H^1(C,L)$$
 is non-zero, hence an isomorphism.
 
 The above construction applied with $L_i := L$ (resp. $\alpha_i := \alpha$) for every $i\geq 1$ (resp. for every $i\geq 2$) produces a pro-vector bundle $\Eh$ over $C$ which again satisfies (\ref{wildconcrete}).
 
 \subsection{A remark concerning direct images}
 Let $\pi: C \lra C'$ be a finite $k$-morphisms between two smooth, projective, geometrically irreducible curves over $k$.
 
 For any pro-vector bundle $\Eh$ over $C,$ its direct image $\pi_\ast \Eh$ defines a pro-vector bundle over $C'$. Clearly,
 $$h^0(C', \pi_\ast \Eh) =  h^0(C, \Eh).$$
 Moreover, one easily see that
 $$\overline{h}^0(C', \pi_\ast \Eh) \leq   \overline{h}^0(C, \Eh).$$
 
The examples of ``wild" pro-vector bundles constructed above demonstrate that the  last inequality may indeed be strict. For instance, if $C$ is an elliptic curve, $C'$ is $\PP_k^1$, and $\Eh$ is a pro-vector bundle over $C$ which satisfies (\ref{wildconcrete}),  
  Proposition \ref{gzeroeq} shows that
  $$\overline{h}^0(C',\pi_\ast \Eh) = h^0(C',\pi_\ast \Eh) = h^0(C, \Eh)  =0,$$
  whereas
  $$\overline{h}^0(C, \Eh) =1.$$

  \section{A vanishing criterion}
  
  In this section, we briefly discuss a geometric analogue of the vanishing criterion established as Theorem \ref{ConnectV}. 
  We hope this will shed some light on the basic principles that underly the proof in Section \ref{VanCrit}.  
  
  Let us consider the following data:
  \begin{enumerate}
  \item \emph{a sequence $(\Hh_k)_{k \geq k_0}$ of pro-vector bundles over $C$ and a sequence $(p_k)_{k > k_0}$ of morphisms
$$p_k: \Hh_{k-1} \lra \Hh_k$$
of topological $\cO_C$-modules;
\item a sequence $(F_k)_{k > k_0}$ of (finite rank)  vector bundles over $C$ and a sequence $(\delta_k)_{k > k_0}$ of maps of $\cO_C$-modules}  $$\delta_k : F_k \lra \Hh_{k-1};$$ 
\end{enumerate}
and let us assume that, \emph{for every $k \in \N_{> k_0}$, the diagram}
\begin{equation*}
F_k \stackrel{\delta_k}{\lra} \Hh_{k-1} \stackrel{p_k}{\lra} \Hh_k \lra 0
\end{equation*}
\emph{is an exact sequence of $\cO_C$-modules.}

\begin{proposition}
 With the above notation, if the following two conditions are satisfied:
 $$\mathbf{Fin :} \quad \mbox{ for every $k \in \N_{\geq k_0},$ }  h^0(C, \Hh_k) < + \infty,$$
and  
 $$\mathbf{Amp :} \quad \liminf_{k \ra + \infty} \frac{1}{k} \, \mu_{\rm min}(F_k) > 0,$$
 then there exists $k_1 \in \N_{> k_0}$ such that the map $\delta_k$ vanishes for any $k \in \N_{\geq k_1}$, or equivalently, such that, for any $k \in \N_{\geq k_1}$, 
$p_k: \Hh_{k-1} \lra \Hh_k$
is an isomorphism in ${\rm proVect}_C$.
\end{proposition}

One easily see that $\ker \delta_k$ is saturated coherent subsheaf of $F_k$, and therefore that
$$\im \delta_k \lrasim F_k /\ker \delta_k$$
is a vector bundle over $C$. 

Since the cohomology groups of $\im \delta_k$ are finite dimensional  $k$-vector spaces, the long exact sequence of cohomology
$$\Gamma(C, \im \delta_k) \lra \Gamma(C, \Hh_{k-1}) \lra \Gamma(C,\Hh_k) \lra H^1(C, \im \delta_k)$$
associated to the short exact sequence of $\cO_C$-modules 
$$0 \lra \ker p_k = \im \delta_k \hlra \Hh_{k-1}  \stackrel{p_k}{\lra} \Hh_k \lra 0 $$ 
 shows that Condition $\mathbf{Fin}$ is satisfied as soon as  $h^0(C, \Hh_k) < + \infty$ for \emph{some} $k \in \N_{\geq k_0}.$
  
  In Condition $\mathbf{Amp}$, $\mu_{\rm min}(F_k)$ denotes the minimal slope of $F_k$, namely the infimum of the slopes $\mu(V) := \deg_C V/ \rk V$ of the vector bundles $V$ of positive rank that arise as quotients of $F_k$. (It is $+ \infty$ if $\rk F_k =0.$)
\begin{proof}
 For any $k \in \N_{\geq k_0},$ we may consider the surjective morphism of topological $\cO_C$-Modules 
 $$q_k : H_{k_0} \lra H_k$$
 defined as the composition of the morphisms $(p_i)_{k_0 < i \leq k},$
 and the non-decreasing sequence of their kernels:
 $$\ker {q}_{k_0} = 0 \hra \ker {q}_{k_0 +1 } \hra \ldots \hra \ker {q}_{k} \hra \ker {q}_{k+1} \hra \ldots.$$
For any $k\in \N_{>k_0},$ we have an exact sequence of $\cO_C$-modules:
$$0 \lra \ker q_{k-1} \hlra \ker q_k \stackrel{q_{k-1}}{\lra} \ker p_k = \im \delta_k \lra 0.$$

As $\ker q_{k_0} = 0,$ this implies inductively that the $\cO_C$-Modules $\ker q_k$ are vector bundles over $C$ and that, if for every $k \in \N_{\geq k_0}$ we let
$$n_k := \rk \im \delta_k,$$
we have:
\begin{equation}\label{rkk}
\rk \ker q_k - \rk \ker q_{k-1} = n_k
\end{equation}
 and
 \begin{equation}\label{degk}
\deg_C \ker q_k - \deg_C \ker q_{k-1} \geq  n_k \, \mu_{\rm min}(F_k).
\end{equation}
(If $F_k =0$, then $n_k =0$ and  $n_k \, \mu_{\rm min}(F_k) =0.$)

The validity of $\mathbf{Amp}$ may be rephrased as the existence of $\eta$ in $\R^\ast_+$ and of $c$ in $\R_+$ such that, for any $k \in \N_{\geq k_0}$,
$$\mu_{\rm min}(F_k) \geq \eta k -c.$$

From this lower bound on $\mu_{\rm min}(F_k)$ combined with the relations (\ref{rkk}) and (\ref{degk}), we derive that, for every  $k \in \N_{\geq k_0}$:
$$\rk \ker q_k = \sum_{k_0< i \leq k} n_i$$
and
$$\deg_C {\ker q_k} \geq \sum_{k_0< i \leq k} n_i (\eta\, i - c).$$
The Riemann inequality now shows that
$$h^0(C, \ker q_k) \geq \deg_C \ker q_k  + (1-g)\,  \rk\ker q_k \geq \sum_{k_0< i \leq k} n_i (\eta\, i - c- g+1).$$

According to $\mathbf{Fin},$ $h^0(C, \Hh_{k_0})$ is finite. The 
estimate
$$\sum_{k_0< i \leq k} n_i (\eta\, i - c- g+1) \leq h^0(C, \ker q_k) \leq h^0(C, \Hh_{k_0}),$$
valid for any $k \geq k_0$, implies that $n_i$ vanishes when $i$ is large enough.
\end{proof}

    \chapter{Epilogue: formal-analytic arithmetic surfaces and algebraization}\label{Epi}

    \medskip
    
  The formalism presented in this monograph has been developed with a view towards applications to Diophantine geometry and transcendence theory. In Diophantine geometry, one often encounters situations that combine \emph{formal geometry over the integers} and \emph{complex analytic geometry}.
 The formalism of pro-vector bundles turns out to be especially well adapted to handle such situations: to such ``formal-analytic data", one may attach natural pro-Hermitian vector bundles 
 $$\Ebh := (\Eh, (E_\sigma^{\hilb}, \Vert.\Vert_\sigma)_\sKC)$$
 whose ``algebraic" part $\Eh$ (resp. ``analytic" part $(E_\sigma^{\hilb}, \Vert.\Vert_\sigma)_\sKC$) encode the formal geometric (resp. the complex analytic) data under investigation --- for instance by considering spaces of sections of suitable ``formal-analytic" vector bundles. The $\theta$-invariants of these pro-Hermitian vector bundles turn out to control significant properties of the ``formal-analytic data" under investigation. 
 
 In this last chapter, we present a simple illustration of this general philosophy.  
 
  \section{An algebraicity criterion for smooth formal curves over $\Q$}\label{EpiIntro}

 \subsection{}    A classic example of the combination of formal geometry over the integers and complex analytic geometry alluded to above is provided by E. Borel's theorem (\cite{Borel94}):

\begin{theorem}\label{thBorel} Let $f \in \Z[[X]]$ be a formal series with integral coefficients. If the complex radius of convergence of $f$ is positive, and if $f$ extends to a meromorphic function on some open disk $\mathring{D}(0, R) := \{ z \in \C \mid \vert z \vert < R \}$ of radius $R >1,$ then $f$ is the expansion of some rational function in $\Q(X).$
 \end{theorem}
 Borel's theorem admits generalizations due to P\'olya, Dwork, and Betrandias, that provide rationality criteria for element of the algebra $K[[X]]$ of formal series in one variable over some number field $K$ (\cite{Polya1928}, \cite{Dwork60}, \cite{amice75}, Chapter 5; see also \cite{Cantor1980}).
 
 In the 1980's, D. V. and G. V. Chudnovsky have established some far reaching extensions of the rationality theorems \emph{à la} Borel, asserting the \emph{algebraicity  of} suitable germs of \emph{formal curves} (\cite{ChudnovskysGroth85}, \cite{ChudnovskysAcad85}). Their work has led to diverse developments, concerning arithmetic algebraicity criteria and their consequences in Diophantine geometry (see  \cite{Andre89}, Chapter VIII, \cite{Andre04}, Section 5, \cite{Graftieaux2001a}, \cite{Graftieaux2001a},  \cite{Bost01}, \cite{Bost04}). We refer notably to \cite{Chambert01} and \cite{Bost06} for surveys and additional references, and to \cite{BostChambert-Loir07} for applications of these algebraicity theorems to rationality results extending the classical theorems à la Borel, valid over the projective line to arbitrary algebraic curves over a number field).
 
\subsection{}\label{SimpleArAlG} Let us state a simple but significant version of the algebraicity criteria alluded to above. 

Let $N$ be a positive integer, and let
$$\phi := (\phi_1, \ldots, \phi_N) \in \Z[[X]]^N$$
be a $N$-tuple of formal series in one variable with integral coefficients.

We let 
$$P:= \phi(0) \in \A^N(\Z)$$
and we assume  that
$$\phi'(0) := (\phi'_1(0), \ldots, \phi'_N(0)) \neq 0.$$
Then, as an element of $\Q[[X]],$ the series $\phi$ defines an isomorphism of formal schemes 
\begin{equation}\label{defChat}\hat{\phi}_\Q := {\rm Spf}\, \Q[[X]] \lrasim \hat{C} \hra \hat{\A}^N_{\Q, P}
\end{equation}
 between the ``formal affine curve over $\Q$"  ${\rm Spf}\, \Q[[X]]$ and a smooth formal curve $\hat{C}$ in the formal completion $\hat{\A}^N_{\Q, P}$ of $\hat{\A}^N_{\Q}$ at the point $P.$ By construction, the tangent space
 $T_P {\hat{C}_\Q}$ admits $\phi'(0)$ as a basis vector.
 
 Besides,  let us consider a compact connected Riemann surface $V$ with (non-empty) boundary $\bV$, and $O$ a point in $\Vo := V \setminus \bV.$
 
 Recall that the \emph{Green's function} $g_{V,O}$ of $O$ in $V$ is the unique continuous function on $V \setminus \{O\}$ which vanishes on $\bV$, is harmonic on $V \setminus \{ O\}$, and admits a logarithmic singularity at $O$; namely, if $z$ denotes some local analytic coordinate on some open neighborhood $U$ of $O$ in $V$, 
 \begin{equation}\label{greengros}
g_{V,O}(P) = \log \vert z(P) - z(O) \vert^{-1} + O(1) \quad\mbox{ when $P$ goes to $O$}.
\end{equation}
Actually the function $g_{V,O} - \log \vert z - z(O) \vert^{-1}$ ---  \emph{a priori} defined and harmonic on $U \setminus \{O\}$ --- extends to some harmonic function on $U$. The \emph{capacitary metric} $\Vert.\Vert^{\rm cap}_{V,O}$ on $T_O V$ is defined by the relation:
 \begin{equation}\label{greenfin}
g_{V,O}(P) = \log \vert z(P) - z(O) \vert^{-1} - \log \Vert \partial / \partial z \Vert^{\rm cap}_{V,O} + o(1) \quad\mbox{ when $P$ goes to $O$}.
\end{equation}

Finally, let us consider a map $$\gamma : V \lra \PP^N(\C)$$   that is $\C$-analytic (up to the boundary of $V$\footnote{This assumption could be replaced by the mere analyticity of $\gamma$ on $\mathring{V}$. This follows from the derivation of Theorem  \ref{ArAlgPM} from Theorem \ref{ArAlgK} in paragraph \ref{DiophantAlg}, together with the relation (\ref{capepsilon}), \emph{infra}.
}) and which defines an isomorphism 
\begin{equation}\label{jform}
\hat{\gamma}_O : \Vh_0 \lrasim \hat{C}_\C
\end{equation}
between the formal completion of $V$ at $O$ and the smooth formal complex curve $\hat{C}_\C$, deduced from $\hat{C}$ by extending the scalars from $\Q$ to $\C$. In particular,
$$\gamma(O) = P \quad\mbox{and}\quad D\gamma(O) : T_OV \lrasim T_P \hat{C}_\C.$$

\begin{theorem}\label{ArAlgPM}
With the above notation, if 
\begin{equation}\label{ArAmplNaive}
\Vert D\gamma(O)^{-1} \phi'(0) \Vert^{\rm cap}_{V,O} < 1,
\end{equation}
the formal curve $\hat{C}$ --- \emph{and consequently the image of $\gamma$} --- is algebraic; namely there exists a closed integral subscheme $X$ of $\PP^N_\Q$ of dimension $1$ such that 
\begin{equation}\label{inclalg}
\hat{C} \subset \Xh_P \quad\mbox{and}\quad \gamma(V) \subset X(\C).
\end{equation}
 \end{theorem}
 
 In other words, the formal curve $\hat{C}$ is a branch of $X$ through $P$.\footnote{This may be  conveniently expressed in terms of the normalization $\nu: \tilde{X} \lra X$ of $X$. Indeed,  the conclusion (\ref{inclalg}) of Theorem
 \ref{ArAmplNaive} and the universal property of the normalization, in formal and in analytic geometry, imply that their exists a (unique) point $\tilde{P} \in \tilde{X} (\Q)$ such that $\nu(\tilde{P}) = P$ and $\nu$ defines an isomorphism of formal completions $\hat{\nu}_{\tilde{P}} : \widehat{\tilde{X}}_{\tilde{P}} \lrasim \hat{C}_\Q (\hra \Xh_P);$ moreover, there exists a unique holomorphic map $\tilde{\gamma}: V \lra \tilde{X}(\C)$ such that $\gamma = \nu_\C \circ\tilde{\gamma}$.} 
 
 This theorem admits Borel's original theorem as a straightforward consequence (see \ref{AppIBorel} \emph{infra}). As already pointed out by D.V. and G.V. Chudnovsky (\cite{ChudnovskysAcad85}), it also allows one to recover the theorem of Serre and Faltings asserting that two elliptic curves over $\Q$ are $\Q$-isogeneous if their $a_p$ invariants coincide for almost every prime $p$ (see \ref{AppIIIsog} \emph{infra}).
 
 Theorem \ref{ArAlgPM} may be understood as an arithmetic criterion for the algebraicity of the smooth formal curve $\hat{C}_\Q$ in $\PP_\Q^N.$ It asserts that it is algebraic when it admits  suitable integral and analytic parameterizations (encoded by $\phi$ and $\gamma$ respectively) and when the analytic parameterization $\gamma$ of $\hat{C}_\C$ has a large enough range in comparison to its integral formal parametrization $\phi$ (this is the intuitive meaning of condition   (\ref{ArAmplNaive})).
 
 One may actually show that these sufficient conditions for the algebraicity of $\hat{C}_\Q$  are actually necessary. The existence, when $\hat{C}_\Q$ is algebraic, of the formal integral parameterization $\phi$ is known as Eisenstein Lemma (see \cite{DworkvdPoorten92} for a sharp version and references), and indeed goes back to Eisenstein's last memoir (\cite{Eisenstein52}). Once some formal parametrization $\phi$ of an algebraic formal curve  $\hat{C}_\Q$ is chosen, the existence of $(V,O)$ and $\gamma$ satisfying conditions (\ref{ArAmplNaive}) and (\ref{inclalg}) is a straightforward application of potential theory on Riemann surfaces. (It follows from the fact that finite subsets in the compact Riemann surface defined as the normalization of the Zariski closure of $\hat{C}_\C$ in $\PP^N_\C$ are polar.)

\subsection{}\label{descriptionEpil}  In this last part, we use the formalism of  infinite dimensional Hermitian vector bundles and the properties of their $\theta$-invariants investigated in this monograph to give a natural proof of the Diophantine algebraization theorem \ref{ArAlgPM}. We shall actually establish some more general version of it, valid over an arbitrary number field $K$ instead of $\Q$ (see Theorem \ref{ArAlgK} \emph{infra}).

To provide some geometric background to this proof, in the next section (Section \ref{EstimatesAlgGeo}), we discuss similar algebraization statements, valid in the context of analytic and formal geometry, namely Chow's theorem and an algebraicity theorem \emph{à la} Andreotti-Hartshorne concerning ``pseudo-concave" formal surfaces.

These ``geometric" algebraization theorems admit simple parallel proofs based on the asymptotic behavior of the dimension of spaces of sections of large powers of some (ample) line bundle $L$ on some compact analytic manifold $M$ (resp. on some ``pseudo-concave" formal surface $\Vh$ over some base field $k$). The key point of these algebraization proofs will be the estimates, valid when the positive integer $D$ goes to infinity:
\begin{equation}\label{upperDan}
\dim_\C \Gamma(M, L^{\otimes D}) = O(D^{\dim M})
\end{equation}
\begin{equation}\label{upperDfor}
( \mbox{ resp.} \dim_k \Gamma (\Vh, L^{\otimes D}) = O(D^2) \;).
\end{equation}

In the remaining part of this chapter, we present Diophantine counterparts of these algebraization arguments, that will notably lead to a proof of Theorem \ref{ArAlgPM} and to its generalized version Theorem \ref{ArAlgK}.

Firstly, in Section \ref{AmpleTheta}, we present some basic results concerning the $\theta$-invariants of spaces of sections of large powers of ample line bundles over projective schemes over $\Spec \Z$ that will play the role, in  a Diophantine framework, of asymptotic control of the dimension of spaces of sections of ample line bundles on projective varieties over a field recalled in Proposition \ref{algbnd} and in (\ref{AsGeo}) and (\ref{lowerbencore}). 

The content of Section  \ref{AmpleTheta} is of interest, independently of the algebraization results established in this Section, for the relations it establishes between $\theta$-invariants and higher dimensional Arakelov geometry. It relies  only on the basic formalism of $\theta$-invariants of finite rank Hermitian vector bundles described in Chapter \ref{thetaone}. (We refer the reader to the recent work of Charles \cite{Charles2017}, Section 4, for another illustration of the fact that this basic formalism constitutes a remarkably flexible tool when dealing with classical problems of  Arakelov geometry.)

Then, after some preliminary concerning smooth formal relative curve (Section \ref{SmForCur}) and compact Riemann surfaces with boundary (Section \ref{boundaryGreen}), in Section \ref{FormAnOK} we introduce  the class of ``Diophantine spaces" that we will work with: the \emph{smooth formal-analytic surfaces} $\cVt$ over $\Spec \OK$. 

They are defined by a smooth formal relative curve $\cVh$ over $\Spec \OK$, admitting $\Spec \OK$ as scheme of definition, and by a family of compact Riemann surfaces with boundary $(V_\sigma)_{\sKC}$ (indexed by the field embeddings $\sKC$), that are ``glued together over the archimedean places of $K$" by means of immersions $i_\sigma: \cVh_\sigma \hlra V_\sigma$. The family $(V_\sigma, i_\sigma)_\sKC$ is also required to be compatible with complex conjugation.

For instance, the data in paragraph \ref{SimpleArAlG} define such a formal-analytic surface over $\Spec \Z$, provided $V$ may be equipped with some ``real structure" compatible with $\phi$ and $\gamma$\footnote{Namely, with some antiholomorphic involution $j$ such that $j(O)=O$ and such that the formal isomorphism $\widehat{\phi}_\C^{-1} \circ \widehat{\gamma}: \Vh_O \lrasim  {\rm Spf}\, \C[[T]]$ satisfies: $(\widehat{\phi}_\C^{-1} \circ \widehat{\gamma}) \circ \hat{j}_O = \overline{\widehat{\phi}_\C^{-1} \circ \widehat{\gamma}}.$} : we define $\cVh$ as ${\rm Spf\,} \Z[[X]]$ and $V_\sigma$ as $V$ (for $\sigma$ the unique embedding of $\Q$ in $\C$); the isomorphisms (\ref{defChat}) and (\ref{jform})
provide the additional glueing data:
$$i_\sigma := \widehat{\gamma}_O^{-1}  \circ \widehat{\phi}_\C : \cVh_\C \simeq {\rm Spf}\, \C[[T]] \lrasim \Vh_O.$$

There is a natural notion of \emph{formal-analytic vector bundle} $\cEt$ over such a formal-analytic surface $\cVt$ over $\Spec \OK$: $\cEt$ is defined by some vector bundle $\cEh$ on the formal scheme $\Vh$ and some $\C$-analytic vector bundles $\cE_\sigma$ on the Riemann surfaces $V_\sigma$, and by some suitable glueing data. 

Moreover, to any such formal-analytic vector bundle $\cEt$ is naturally attached the \emph{Hilbertisable pro-vector bun\-dle $\Gamma_{L^2}(\tilde{\cV}; \tilde{\cE})$} of its ``global sections": it is defined by the pro-vector bundle $\Gamma(\cVh; \cEh)$ over $\Spec \OK$ of formal sections of $\cEh$ and by the Hilbertisable spaces of $L^2$-analytic sections $\Gamma_{L^2}(V_\sigma, \cE_\sigma)$ of the analytic vector bundles $\cE_\sigma$ on the Riemann surfaces with boundary $V_\sigma$; the space  $\Gamma_{L^2}(V_\sigma, \cE_\sigma)$ may be embedded in $\Gamma(\cVh; \cEh)_\sigma$
by means of the glueing data alluded to above. 

When the Riemann surfaces $V_{\sigma}$ are equipped with some volume forms $\nu:=(\nu_\sigma)_{\sKC}$ and the vector bundles $\cE_\sigma$ with some Hermitian metric $(\Vert.\Vert_\sigma)_{\sKC}$ over these Riemann surfaces --- and therefore $\cEbh := (\cEt, (\Vert.\Vert_\sigma)_{\sKC}$ becomes a formal-analytic Hermitian vector bundle over $\cVt$ --- the Hilbertisable spaces  $\Gamma_{L^2}(V_\sigma, \cE_\sigma)$ becomes endowed with some canonical $L^2$-norms, and the Hilbertisable pro-vector bundle $\Gamma_{L^2}(\tilde{\cV}; \tilde{\cE})$ becomes a well-defined \emph{pro-Hermitian vector bundle} $\Gamma_{L^2}(\tilde{\cV}, \nu; \cEbh)$.

It turns out that the Hilbertisable pro-vector bundles $\Gamma_{L^2}(\tilde{\cV}; \tilde{\cE})$ are $\theta$-finite, under a suitable ``arithmetic pseudo-concavity" condition on $\cVh$, of which assumption  (\ref{ArAmplNaive}) in Theorem \ref{ArAlgPM} is a reformulation in the framework of paragraph \ref{SimpleArAlG}. Moreover, under this assumption, the $\theta$-invariants $h^0_\theta(\Gamma_{L^2}(\tilde{\cV}, \nu; \cLbh^{\otimes D}))$ associated to the powers of some formal-analytic Hermitian line bundle $\cLbh$ over $\cVt$ satisfy the asymptotic estimates
 \begin{equation}\label{upperDAr}
h^0_\theta(\Gamma_{L^2}(\tilde{\cV}, \nu; \cLbh^{\otimes D}))= O(D^2),
\end{equation}
similar to the  estimates (\ref{upperDan}) and (\ref{upperDfor}) valid in analytic and formal geometry.

The proof of these bounds on $\theta$-invariants in Section \ref{ArPsConcaveFin} relies on the basic properties of $\theta$-invariants associated to pro-Hermitian vector bundles established in Chapter \ref{SectionSum}, combined with some analytic estimates, concerning analytic sections of vector bundles on compact Riemann surfaces with boundary and their Green's functions, that have been presented in the preliminary section \ref{boundaryGreen}. These estimates are avatars of the Schwarz lemma used in traditional Diophantine approximation arguments, and of the First Main Theorem in Nevalinna theory (compare \cite{Bost01}, 4.3.2 and 4.3.3). It is striking that the Schwarz Lemma plays also a key role in the proof of the asymptotic estimates (\ref{upperDan}), valid in the complex analytic setting (see Appendix, \ref{ProofHolMult}).    

Having (\ref{upperDAr}) at one's disposal, one may establish an algebraicity property concerning the image of any $\OK$-morphism $\tilde{f} : \Vt \lra \PP^N_\OK$ from $\Vt$ to the projective space by mimicking the proofs of Chow's and Andreotti-Hartshorne's theorems in Section  \ref{EstimatesAlgGeo} (see Theorem \ref{ArAlgK}). 

We conclude this chapter by discussing the applications of Theorem \ref{ArAlgPM} to Borel's theorem (Subsection \ref{AppIBorel}) and to isogenies of elliptic curves over $\Q$ (Section \ref{AppIIIsog}).

        \section[Sections of line bundles and algebraization]{Dimension of spaces of sections of line bundles and algebraization in analytic and formal geometry}\label{EstimatesAlgGeo}
        
    \subsection{A simple proof of Chow's theorem}\label{ProofChow}
    
    The classical theorem of Chow (\cite{Chow49}) asserts that any closed complex analytic subsets of the complex projective space is a closed algebraic subset.
    In this paragraph, 
   we present a proof of Chow's that makes conspicuous
   its relation with estimates on the dimension of spaces of sections of line bundles over analytic and algebraic varieties.
 
 For later comparison with Diophantine situations --- where our understanding often lags behind the one we have of geometric situations, concerning algebraic varieties over some field --- we have tried to formulate our arguments in as elementary terms as possible.

  For simplicity, we will consider only the ``smooth case" of Chow's theorem, concerning closed complex submanifolds of the projective space:

\begin{theorem}[Chow]\label{Chowsmooth} Any closed complex analytic submanifold $V$ of the complex projective space $\PP^N(\C)$ is algebraic. In other words, it is the set of complex points of some closed subscheme of $\PP^N_\C$.
 \end{theorem}
 
 The proof of Theorem \ref{Chowsmooth} will rely on the following two Propositions.

\begin{proposition}\label{algbnd} For any closed complex analytic algebraic subset $Z$ in some complex projective space $\PP^N(\C)$, there exists some algebraic line bundle $L_{\rm alg}$ over $Z$ and some $c$ in $\R^\ast_+$ such that, for any positive integer $D$, the dimension of the vector space $\Gamma(Z, L_{\rm alg}^{\otimes D})$ of (algebraic) regular sections of $L_{\rm alg}^{\otimes D}$ over $M$ 
  satisfies:
 \begin{equation}\label{ineq:algbnd} \dim_{\C} \Gamma(Z, L_{\rm alg}^{\otimes D}) \geq c. D^{\dim Z}.
 \end{equation}
\end{proposition}

\begin{proposition}\label{holbnd}
  
For any complex analytic  line bundle $L_{\rm an}$ over a compact complex manifold $M$, there exists $C$ in $\R^\ast_{+}$ such that, for any 
  positive integer $D$, the dimension of the vector space
  $\Gamma(M, L^{\otimes D})$ of analytic sections of $L_{\rm an}^{\otimes D}$ over $M$ 
  satisfies:
   \begin{equation}\label{ineq:holbnd} {\dim}_{\C} \,  \Gamma(M,L_{\rm an}^{\otimes D}) \leq C. D^{\dim M}.
   \end{equation}
\end{proposition}

The key point in the derivation of Theorem \ref{Chowsmooth} is  the following lemma:

\begin{lemma}\label{keyChow}
 With the notation of Theorem \ref{Chowsmooth}, if $V$ is a non-empty complex manifold of dimension $d$, then its Zariski closure $\overline{V}^{\rm Zar}$ in $\PP^N(\C)$ also has dimension $d$. 
\end{lemma}

The lower bound 
$$\dim \overline{V}^{\rm Zar} \geq d$$
on the dimension of $\overline{V}^{\rm Zar}$ is straightforward\footnote{The regular points $(\overline{V}^{\rm Zar})_{\rm reg}$ of $\overline{V}^{\rm Zar}$ are Zariski open and dense in $\overline{V}^{\rm Zar}$. Consequently the intersection $(\overline{V}^{\rm Zar})_{\rm reg} \cap V$ is non-empty, and contains some point $P$ where the local dimension of $\overline{V}^{\rm Zar}$ is $\dim \overline{V}^{\rm Zar}$. Besides, $(\overline{V}^{\rm Zar})_{\rm reg}$ constitutes a complex manifold. The germ of $\overline{V}^{\rm Zar}$ at $P$ contains the germ of $V$ at $P$, and therefore the dimension of the former is at least the dimension of the latter.}. The actual content of Lemma \ref{keyChow} is the opposite inequality:
\begin{equation}\label{keyineqChow}
\dim \overline{V}^{\rm Zar} \leq d.
\end{equation}

Let us take Propositions \ref{algbnd} and \ref{holbnd} for granted, and let us derive Lemma \ref{keyChow}.

Let $V$ be a non-empty closed complex analytic submanifold in $\PP^N(\C)$, of dimension $d$. Let us denote by $Z:= \overline{V}^{\rm Zar}$ its Zariski closure in $\PP^N(\C)$, and let $L_{\rm alg}$ be an algebraic line bundle over $Z$ that satisfies the conclusion of Proposition \ref{algbnd}.

By restriction to $V$ and analytification, this line bundle defines a complex analytic line bundle $L_{\rm an}$ on $V$. Moreover, for any positive integer $D$, a regular section $s$ of $L_{\rm alg}^{\otimes D}$ over $Z$ defines an analytic section $s_{\mid V}$ of $L^{\otimes D}_{\rm an}$ over $V$.

Observe that, since $V$ is Zariski dense in $Z$, the linear map
\begin{equation*}
\begin{array}{rcl}
 \Gamma(Z, L_{\rm alg}^{\otimes D})  & \lra  & \Gamma(V, L_{\rm an}^{\otimes D})   \\
s  & \longmapsto  & s_{\mid V}    
\end{array}
\end{equation*}
is \emph{injective}. Therefore, for every positive integer $D,$
\begin{equation}\label{dimdim}\dim_\C \Gamma(Z, L_{\rm alg}^{\otimes D}) \leq \dim_\C\Gamma(V, L_{\rm an}^{\otimes D}).\end{equation} 

Applied to the compact complex manifold $M=V$, Proposition \ref{holbnd} establishes the existence of $C$ such that, for every positive integer $D$,
\begin{equation}\label{dimC}\dim_\C\Gamma(V, L_{\rm an}^{\otimes D}) \leq C. D^d.
\end{equation}
Besides, by the very choice of $L_{\rm alg}$, there exists $c>0$ such that, for every positive integer $D$,
\begin{equation}\label{dimc}
\dim_\C \Gamma(Z, L_{\rm alg}^{\otimes D}) \geq c. D^{\dim Z}.
\end{equation}
 
 From (\ref{dimdim})-(\ref{dimc}), we derive the required inequality (\ref{keyineqChow})
 by letting $D$ grow to infinity.
 
 \medskip
 
 To complete the proof of Chow's Theorem \ref{Chowsmooth} --- that asserts that $V= \overline{V}^{\rm Zar}$ --- one may assume that $V$ is non-empty and connected. The $\overline{V}^{\rm Zar}$ is (the set of $\C$-points of) a closed integral subscheme of $\PP^N_(\C)$ of dimension $d:= \dim V$, and the equality $V= \overline{V}^{\rm Zar}$ follows by a standard connectedness argument. We recall it in paragraph \ref{ChowConnect} below for the sake of completeness.
 
\subsection{Upper and lower bounds on the dimension of spaces of sections of line bundles}\label{CommentsDim} Let us formulate a few comments
about the lower (resp. upper) bounds on the dimension of spaces of sections of line bundles in the algebraic (resp. analytic) framework of Proposition \ref{algbnd} (resp. \ref{holbnd}).

\smallskip

Concerning the ``algebraic bounds" in Proposition \ref{algbnd}, recall that the lower bound (\ref{ineq:algbnd}) hold \emph{for any ample line bundle $L$ over $Z$} --- for instance for $L:= \cO(1)_{\mid Z}$ ---  provided the integer $D$ is large enough. Actually, in this situation, a more precise version of this lower bound holds: \emph{when $D$ is large enough, the dimension
$$P_L(D) := \dim_{\C} \Gamma(Z, L^{\otimes D})$$
is a polynomial in $D$, of degree $d:= \dim Z$ and of leading coefficient}
$$\frac{1}{d !} \deg_L Z := \frac{1}{d !} \deg c_1(L)^d.[Z].$$
This follows form the basic properties of ample line bundles and numerical intersections (see for instance \cite{stacks-project}, Section 32.43, notably Tag 0BEW).

In particular, when $D$ goes to $+\infty$,
\begin{equation}\label{AsGeo}
\dim_\C \Gamma(Z, L^{\otimes D}) = \frac{1}{d !} \deg_L Z. D^d + O(D^{d-1}).
\end{equation}

The weaker version in Proposition \ref{algbnd} --- that suffices for the derivation of Chow's theorem --- follows from very elementary considerations. Indeed, to prove it, we may observe that, by Noether's normalization theorem, there exists some linear projection
$$p: \PP^N_\C \dashrightarrow \PP^d_\C$$
that is defined on $Z$ and such that $p_{\mid Z}: Z \lra \PP^d_\C$ is a \emph{surjective} morphism. Then, for any positive integer $D$, the pull-back along $p_{\mid Z}$ defines a \emph{injective} linear map:
$$p^\ast_{\mid Z} : \Gamma (\PP^d_\C, \cO(D)) \hlra \Gamma(Z, p^\ast_{\mid Z} \cO(D)).$$

As $\Gamma (\PP^d_\C, \cO(D))$ clearly contains\footnote{actually, is equal to --- but we do not need this more precise fact.} the homogenous component of degree $D$ of $\C[X_0, \ldots, X_N]$, this leads to the lower bound:
$$\dim_\C  \Gamma(Z, p^\ast_{\mid Z} \cO(D)) \geq \dim_\C \C[X_0, \ldots, X_N]_D = \binom{D+d}{d}
\geq \frac{1}{d!} D^d.$$

This shows that the conclusion of Proposition \ref{algbnd} is satisfied when $L_{\rm alg} = p_{\mid Z}^\ast \cO(1).$ (This line bundle is easily seen to be isomorphic to the restriction to $Z$ of the line bundle $\cO(1)$ on $\PP^N_\C$.)

\medskip
 Let us now turn to the upper bounds of the dimension of spaces of analytic sections of powers of some line bundles on a compact analytic manifold that form the content of Proposition \ref{holbnd}. 
 
 Such estimates on dimensions of spaces of analytic sections go back to Serre's expos\'e \cite{Serre53} in S\'eminaire Cartan; closely related results appear in Siegel's work on  meromorphic functions on compact complex manifolds (\cite{Siegel55}).

 In Appendix \ref{ApUpBd}, we present a modernized version of Serre's arguments, which are based on the use of Schwarz Lemma. (The introduction of some Hermitian metric on the line bundle $L$ allows one to present them in an especially elementary form.) We also give an alternative proof of Proposition \ref{holbnd} when $M$ is a K\"ahler manifold. As any complex submanifold of some projective space is clearly K\"ahler,  having at one's disposal Proposition \ref{holbnd} under this additional assumption  is actually sufficient for completing the proof of Theorem \ref{Chowsmooth} presented below. 

\medskip 

Let us finally point out that the above line of arguments for establishing the algebraicity of some analytically defined complex manifolds originates in Poincar\'e's memoir \cite{Poincare02} on abelian functions\footnote{ See Section 2, p. 53--56. We refer the reader to Thimm's historical report in the Weierstrass Festband \cite{Thimm66} for the relations of Poincaré's work to the earlier results and conjectures of Riemann and Weierstrass on abelian functions and theta functions.}. Besides, as shown by Andreotti (\cite{Andreotti63}), the estimates in Proposition \ref{holbnd} hold, not only when $V$ is a compact complex manifold, but more generally when $V$, possibly non-compact, satisfies a suitable pseudo-concavity condition. Accordingly, the arguments in the proof of Lemma \ref{keyChow} allows one to extend it to an algebraicity statement concerning pseudo-concave complex manifolds embedded in a complex projective space (\cite{Andreotti63}, Th\'eor\`eme 6). 

{\small  \subsection{Completing the proof of Chow theorem: connectedness arguments}\label{ChowConnect}

To complete the proof of Chow's Theorem \ref{Chowsmooth} starting for the key Lemma \ref{keyChow},  we shall rely on some basic topological properties of complex algebraic varieties. 
    
    Let 
  $X$ be a quasi-projective complex scheme, which, without loss of 
  generality, we shall assume reduced. Its set of complex points 
  $X(\C)$ is naturally endowed with the \emph{analytic topology}: for 
  any embedding of $\C$-schemes $i: X \hookrightarrow \PP^N_{C}$, it 
  is the topology induced on $X(\C)$ by the topology of the analytic 
  manifold $\PP^N(\C)$ via the associated inclusion
  $i: X(\C) \hookrightarrow \PP^N(\C).$
  
   Specifically, to complete our proof of Theorem \ref{Chowsmooth}, we will use the following compatibility between the Zariski and the analytic topology:

\begin{proposition}\label{density}
 For any  
 subset $U$ of $X$, open and dense in the Zariski topology, its set 
 of complex points $U(\C)$ is dense in $X(\C)$ equipped with the analytic 
 topology.  \qed 
 \end{proposition}
 
 \begin{proposition}\label{connectivity}
If $X$ is connected (for instance, irreducible) in the Zariski 
topology, then $X(\C)$ is connected in the analytic 
 topology. \qed
 \end{proposition}
 
These propositions may be established by a ``reduction to curves", relying on some basic facts of algebraic geometry (Bertini's Theorems) and some elementary knowledge of complex analytic geometry.
 In \cite{Shafarevich77}, Section VII.2,  the reader may find alternative elementary derivations of these Propositions, in a more analytic vein.

Let $V$ be a closed complex analytic submanifold of the complex projective space $\PP^N(\C)$.  To complete the proof of Theorem \ref{Chowsmooth} --- that asserts that $V$ is algebraic, or equivalently, that $V$ coincides with its closure $\overline{V}^{\rm Zar}$  in the Zariski topology on $\PP^N(\C)$ --- one may clearly assume that $V$ is non-empty and connected.  Then $\overline{V}^{\rm Zar}$ is easily seen to be (the set of complex points of) some closed integral subscheme of $\PP^N(\C)$. According to Lemma \ref{keyChow}, its dimension is $d:=\dim V$. Moreover the set $\overline{V}^{\rm Zar}_{\rm reg}$ of regular points of $\overline{V}^{\rm Zar}$ satisfies the following properties:

(i) The complement $\overline{V}^{\rm Zar}_{\rm sing}:= \overline{V}^{\rm Zar} \setminus \overline{V}^{\rm Zar}_{\rm reg}$ is a closed algebraic subset of $\PP^N(\C)$ strictly contained in $\overline{V}^{\rm Zar}$, and therefore cannot contain $V$. In other words, $V \cap \overline{V}^{\rm Zar}_{\rm reg}$ is \emph{not empty}.

(ii) The intersection $V \cap \overline{V}^{\rm Zar}_{\rm reg}$ is (obviously) \emph{closed} in $\overline{V}^{\rm Zar}_{\rm reg}$ equipped with the analytic topology.

(iii) This intersection is also open in $V$ (in the analytic topology), and therefore a complex analytic submanifold of dimension $d$ of $\PP^N(\C)$. As $\overline{V}^{\rm Zar}_{\rm reg}$ also is complex analytic submanifold of dimension $d$ of $\PP^N_\C,$ the intersection $V \cap \overline{V}^{\rm Zar}_{\rm reg}$ is \emph{open} in $\overline{V}^{\rm Zar}_{\rm reg}$ equipped with the analytic topology.

(iv) According to Proposition \ref{connectivity}, $\overline{V}^{\rm Zar}_{\rm reg}$  is \emph{connected} in the analytic topology.

From the above observations, we derive the equality:
$$V \cap \overline{V}^{\rm Zar}_{\rm reg} = \overline{V}^{\rm Zar}_{\rm reg}.$$
Finally, since $\overline{V}^{\rm Zar}_{\rm reg}$ is dense in $\overline{V}^{\rm Zar}$ equipped with the analytic topology 
(by Proposition \ref{density}), this entails the inclusion $\oli{V}^{\rm Zar} \subset V,$ and 
therefore the equality $\oli{V}^{\rm Zar} = V.$
}

  \subsection{Algebraization of smooth formal surfaces over a field}\label{algformalpseudoconcave}  
  
  Let $k$ be a field and let $C$ be a smooth projective geometrically connected curve over $k$.
  
  Let us consider a formal noetherian scheme $\cVh$ over $k$ admitting $C$   as scheme of definition. Let us assume that $\cVh$ is formally smooth over $k$, of dimension 2, and let $N_C \cVh$ denote the normal bundle of $C$ in $\cVh$ --- it is a line bundle over $C$.
  
  When $N_C \cVh$ is ample --- that is, when
  \begin{equation}\label{Nample}
  \deg_C N_C \cVh > 0
\end{equation}
 --- the formal surface $\cVh$ may be seen as an analogue, in formal geometry, of the pseudo-concave complex analytic manifolds alluded to above.
 
 Actually the dimensions of spaces of sections of vector bundles over $\cVh$ satisfy the following upper bounds:

\begin{proposition}\label{formalbounds}Let us assume that the ampleness condition (\ref{Nample}) holds.

 Then, for any vector bundle $\Eh$ on $\cVh,$ the $k$-vector space 
$\Gamma(\cVh, \Eh)$ is finite dimensional.
Moreover, for any line bundle $\Lh$ over $\cVh,$ when the positive integer $D$ goes to infinity, we have:
\begin{equation}\label{eq:formalbounds}
\dim_k \Gamma(\cVh, \Lh^{\otimes D}) = O(D^2).
\end{equation}
\end{proposition}

\proof Let $\cI$ be the ideal sheaf in $\cO_\cVh$ that defines $C$ as a closed subscheme of the formal scheme $\cVh$. It is an Ideal of definition of $\cVh$. Moreover, as $\cVh$ is smooth, $C$ defines a Cartier divisor in $\cVh$, and $\cI$ may be identified the invertible sheaf $\cO_{\cVh}(-C)$. Its restriction $\cO_{\cVh}(-C)_{\mid C}$ to $C$ is the dual $(N_C \cVh)^\vee$ of the normal bundle $N_C \cVh$.

For every $n \in \N$, let $\cV_n$ denote the $n$-th infinitesimal neighborhood of $C$ in $\cVh$, namely the closed subscheme of $\cVh$ defined by the ideal sheaf $\cI^{n+1}$. In other words,
$$\vert \cV_n \vert := \vert \cVh \vert = \vert C \vert \quad\mbox{and}\quad \cO_{\cV_n} := \cO_{\cVh}/ \cI^{n+1}.$$
It will be convenient to let $\cV_{-1} := \emptyset$.

The $\cV_n$'s define an inductive system of $k$-schemes:
$$\cV_{0} = C \hra \cV_1 \hra \ldots \hra \cV_n \hra \cV_{n+1} \ldots,$$
from which we derive an identification of topologically ringed spaces:
\begin{equation}\label{limV}
\cVh \lrasim \varinjlim_{n} {\cV}_n.
 \end{equation}
 For any coherent sheaf $\cFh$ on $\cVh$, from (\ref{limV}), we deduce an isomorphism of (topological) $k$-vector spaces:
\begin{equation}\label{limGamma}
\Gamma(\cVh, \cFh) \lrasim \varprojlim_{n} \Gamma({\cV}_n, \cFh_{{\cV}_n}). 
\end{equation}

For every $n \in \N$, the short exact sequence of $\cO_{\cVh}$-modules
\begin{equation*}
0 \lra \cI^{n}/\cI^{n+1} \lra \cO_{\cVh}/\cI^{n+1} \lra \cO_{\cVh}/\cI^{n}  \lra 0
\end{equation*}
may be written
\begin{equation}\label{increment}
0 \lra i_\ast (N_C \cVh)^{\vee \otimes n} \lra \cO_{\cV_n} \lra  \cO_{\cV_{n-1}} \lra 0,
\end{equation}
where $i: C \hra \cVh$ denotes the inclusion morphism.

When $\Eh$ is a locally free coherent sheaf on $\cVh,$ this exact sequences of $\cO_{\cVh}$-modules becomes, after tensoring with $\Eh,$
$$ 0 \lra i_\ast (\Eh_{\mid C} \otimes N_C \cVh)^{\vee \otimes n} \lra \Eh_{\mid \cV_n} \lra  \Eh_{\mid \cV_{n-1}} \lra 0,$$
and induces a short exact sequence of finite dimensional $k$-vector spaces 
\begin{equation}\label{incrementGamma}
0 \lra \Gamma( C, \Eh_{\mid C}\otimes N_C \cVh^{\vee \otimes n}) \lra \Gamma( \cV_n, \Eh_{\mid\cV_n}) \lra
\Gamma(\cV_{n-1}, \Eh_{\mid \cV_{n-1}}).
\end{equation}
This implies the inequality:
$$\dim_k \Gamma( \cV_n, \Eh_{\mid\cV_n}) \leq 
\dim_k \Gamma( C, \Eh_{\mid C}\otimes N_C \cVh^{\vee \otimes n}) + \dim_k \Gamma(\cV_{n-1}, \Eh_{\mid \cV_{n-1}}).$$ 

The ampleness of the line bundle $N_C\cVh$ implies that, when $n$ is large enough, every section of $\Eh_{\mid C}\otimes N_C \cVh^{\vee \otimes n}$ over $C$ vanishes, and therefore 
$$\dim_k \Gamma( C, \Eh_{\mid C}\otimes N_C \cVh^{\vee \otimes n}) = 0 \quad\mbox{ and }\quad \dim_k \Gamma( \cV_n, \Eh_{\mid\cV_n}) \leq \dim_k \Gamma(\cV_{n-1}, \Eh_{\mid \cV_{n-1}}).$$
The sequence of non negative integers $(\dim_k \Gamma( \cV_n, \Eh_{\mid\cV_n}) )_{n \in \N}$ is therefore stationary, and consequently, when $n \in \N$ is large enough, the injective restriction morphism
$$\Gamma( \cV_n, \Eh_{\mid\cV_n}) \lra
\Gamma(\cV_{n-1}, \Eh_{\mid \cV_{n-1}})$$
is actually an isomorphism. This shows that, when $n \in \N$ is large enough, the inclusion morphism $\cV_n \hra \cVh$ induces an isomorphism
$$\Gamma(\cVh, \Eh) \lrasim \Gamma( \cV_n, \Eh_{\mid\cV_n}),$$
and 
that $\Gamma(\cVh, \Eh)$ is a finite dimensional $k$-vector space, of dimension at most
$$\sum_{n \in \N} \dim_k \Gamma( C, \Eh_{\mid C}\otimes N_C \cVh^{\vee \otimes n}).$$

This notably applies to $\Eh = \Lh^{\otimes D}$, when $\Lh$ is some line bundle over $\cVh$ and $D$ some positive integer. This establishes the upper bounds:
$$\dim_k \Gamma(\cVh, \Lh^{\otimes D}) \leq \sum_{n \in \N} \delta(D,n),$$
where, for any $(D,n) \in \N^2$,
$$\delta(D,n) := \dim_k \Gamma( C, \Lh_{\mid C}^{\otimes D}\otimes N_C \cVh^{\vee \otimes n}).$$ 

The dimension $\delta(D,n)$ may be easily bounded from above by means of Riemann inequality (\cf (\ref{RiemannGeom}) \emph{supra}).  Indeed, if we let
 $$a:= \deg_{C} N_C \cVh \quad\mbox{ and }\quad a':= \deg_C \Lh_{\mid C},$$
   For any $(D,n) \in \N^2,$
 we immediately obtain from the latter:
 $$\delta(D,n) \leq \left( 1 + \deg_C \Lh_{\mid C}^{\otimes D}\otimes N_C \cVh^{\vee \otimes n}\right)^+ =(- an + a'D +1)^+.$$

Finally, we obtain:
$$\dim_k \Gamma(\cVh, \Lh^{\otimes D}) \leq \sum_{n \in \N} (- an + a'D +1)^+.$$

According to the ampleness condition (\ref{Nample}), $a$ is positive. Therefore, 
when $a'\leq 0,$ the last sum stays bounded (by 1) when $D$ goes to $+\infty.$ Moreover, when $a'>0,$ we have:
$$\sum_{n \in \N} (- an + a'D +1)^+ = \int_0^{+\infty} (- ax + a'D +1)^+ \, dx + O(D) = (a'^2/2a)\, D^2 + O(D)$$
when $D$ goes to $+\infty.$
 \qed

Let us now consider some quasi-projective $k$-scheme $X$ and a morphism of formal schemes over $k$
$$f: \cVh \lra X.$$
Such a morphism is nothing but a morphism of $k$-locally ringed spaces form $(\vert \cVh \vert, \cO_{\cVh}) := (\vert C \vert, \cO_{\cVh})$ to $(\vert \cX \vert, \cO_\cX)$.) One easily see that there exists a smallest closed subscheme $Z$ of $X$ such that $f$ factorizes through the inclusion $Z \hra X.$\footnote{Indeed, for any closed subscheme $Y$ of $X$, $f$ factorizes through $Y$ if and only if  the inverse image $f^\ast\cI_Y$ of its ideal sheaf vanishes. If $\{Y_\alpha\}$ is the family of such closed subschemes of $X$, the sum $\cI:= \sum_{\alpha} \cI_\alpha$ of their ideal sheaves is a quasi-coherent (hence coherent) ideal sheaf in $\cO_X$ that defines the smallest closed subscheme of $X$ through which $f$ factorizes.} We shall call it \emph{the Zariski closure of the image of $f$}, and denote by $\overline{\im f}$. 

Observe that this construction satisfies the following two properties:

(i) If $U$ is some open subscheme of $X$ such that $f$ factorizes through the inclusion $U \hra X$ --- in other words, if the image of the underlying continuous map of topological spaces $f: \vert \cVh \vert = \vert C \vert \lra \vert X \vert$ is contained in $\vert U \vert$ --- and if $f^U$ denotes $f$ seen as a morphism from $\cVh$ to $U$, then we have:
\begin{equation}\label{restU}
\overline{\im f^U} = \overline{\im f} \cap U.
\end{equation}

(ii) The scheme $\overline{\im f}$ is an integral scheme, since $\cVh$ is an integral formal scheme.

An argument similar to the proof of the key algebraicity lemma (Lemma \ref{keyChow}) in paragraph \ref{ProofChow} allows one to deduce the following algebraization result from Proposition \ref{formalbounds}:

\begin{theorem}\label{formalAlgebr} Let $\cVh$ be a smooth formal surface over $k$, with scheme of definition $C$, as above.

If the ampleness condition 
$$\deg_C N_C \cVh > 0$$
holds, then, for any $k$-morphism 
$f: \cVh \lra X$
from $\cVh$ to some quasi-projective $k$-scheme $X$, we have:
\begin{equation}\label{imf2}
\dim \overline{\im f} \leq 2.
\end{equation}
\end{theorem}
 
 The conclusion (\ref{imf2}) of Theorem \ref{formalAlgebr} may be intuitively understood as asserting that ``the image of $f$ is algebraic".
 
\proof With no restriction of generality, we may assume that $X$ is projective (use (i) above), and then, by replacing $X$ by $\overline{\im f}$, that
\begin{equation}\label{imfX}
X =\overline{\im f}. 
\end{equation}

In particular, $X$ is an integral projective $k$-scheme. Let 
$$d := \dim X = \dim \overline{\im f}.$$

If $\cO_X(1)$ denotes the restriction to $X$ of the line bundle $\cO(1)$ in some projective embedding of $X$, as discussed in paragraph \ref{CommentsDim}, we know that there exists some positive constant $c$ such that, for any $D \in \N,$
\begin{equation}\label{lowerbencore}
\dim_k \Gamma(X, \cO_X(D)) \geq c. D^d.
\end{equation}

For any $D\in \N$ and any $s\in \Gamma(X, \cO_X(D)) \setminus \{0\},$ the morphism $f$ does \emph{not} factorizes through the inclusion $\div s \hra X$ (by (\ref{imfX})). This shows that, if we let $L := f^\ast \cO_X(1),$ then, for any $D \in \N,$ the pull-back map
$$
\begin{array}{rcl}
\Gamma(X, \cO_X(D)) 
& \lra & \Gamma(\cVh, L^{\otimes D}) \\
f  &  \longmapsto & f^\ast s
\end{array}
$$
is injective, and therefore 
 \begin{equation}\label{injineq}
\dim_k \Gamma(X, \cO_X(D)) \leq \dim_k  \Gamma(\cVh, L^{\otimes D}).
\end{equation} 

From   (\ref{injineq}), (\ref{lowerbencore}), and (\ref{eq:formalbounds}), we derive the required inequality $d \leq 2$ by letting $D$ grow to infinity.
 \qed 
   
  Observe that, if the restriction $f_{\mid C}$ of $f$ to $C$ is a constant morphism (with value some $k$-rational point $P$ in $X$), then the line bundle $\cO_X(1)$ may be trivialized on some open neighborhood of $P$, and therefore $L:= f^\ast(\cO_X(1))$ may be trivialized on $\cVh$. Consequently, in this situation, $\dim_k  \Gamma(\cVh, L^{\otimes D})$ stays constant (actually equals to 1) when $D$ goes to infinity, and the above arguments shows that $d=0$.
  
  In other words, we have:
   
\begin{scholium}
 With the notation of Theorem \ref{formalAlgebr},  one of the following three exclusive assertions holds:
\begin{enumerate}
\item the morphism $f$ is constant (with values some $k$-rational point of $X$);
\item $\overline{\im f} = f(C)$ is a projective curve;
\item $f(C)$ is an integral projective curve, and $\overline{\im f}$ is an integral surface (containing $f(C)$). \qed
\end{enumerate}
\end{scholium}

\section{Arithmetically ample Hermitian line bundles and $\theta$-invariants}\label{AmpleTheta}

This section is devoted to some counterparts, in the framework of Arakelov geometry, of  Proposition \ref{algbnd} \emph{supra}, that asserts the existence, on any projective variety, of some line bundle such that the dimension of the space of sections of its large powers grows ``fast enough".

As discussed in paragraph \ref{CommentsDim}, Proposition \ref{algbnd} is a weak version of a more general and precise result concerning the asymptotic expression of the dimension  of the space of sections of  large powers of some ample line bundle over a projective variety in terms of its Hilbert polynomial and the intersection theoretic interpretation of its leading term. In the Diophantine situation considered in this chapter, a counterpart of this ``precise result" may be established by using Gillet-Soul\'e's arithmetic intersection theory and the associated theory of heights, and Zhang's theory of ample Hermitian line bundles (\cf \cite{BostGilletSoule94} and \cite{Zhang95}). 

In this section, we first  present this counterpart (see Theorem \ref{AmpleThetaAs} \emph{infra}): it describes the asymptotic behavior of the $\theta$-invariants of the Hermitian vector bundles of sections of large powers of some arithmetic ample Hermitian line bundle over some integral scheme, projective and flat over $\Spec \Z.$  This theorem is of independent interest, and its proof illustrates how the basic formalism of $\theta$-invariants (described in Sections \ref{thetaone} and \ref{thetatwo}) naturally combines with some classical developments of Arakelov geometry.

Then we present a weak version (Corollary \ref{AmpleThetaPM}) of Theorem \ref{AmpleThetaAs}, that appears as the proper counterpart of the ``weak" Proposition \ref{algbnd} and will be sufficient for the applications to algebraization in a Diophantine framework
derived in the next section. We finally give a self-contained proof of this corollary, that relies on some basic results of algebraic geometry and some elementary analytic estimates.

\subsection{Sections of Hermitian vector bundles and John's norms}

We denote by $K$ a number field and by $\OK$ its ring of integers.

Let $\pi: \cX \lra \Spec \OK$ be a reduced scheme, flat and projective over $\Spec \OK$.

Let $\cEb := (\cE, \Vert. \Vert)$ be some Hermitian vector bundle over $\cX$. In other words, $\cE$ is a vector bundle (that is, a locally free coherent sheaf) over $\cX$, and $\Vert .\Vert$ is a continuous Hermitian metric, invariant under complex conjugation, on the complex analytic vector bundle $\cE^{\an}_\C$ over 
$$\cX(\C) = \coprod_{\sKC} \cX_\sigma (\C)$$
deduced from $\cE$ by base change to $\C$ and analytification.

The direct image $\pi_\ast \cE$ of $\cE$ under $\pi$ is a vector bundle over $\Spec \OK$. (It is the sheaf over $\Spec \OK$ attached to the finitely generated torsion free, hence projective, $\OK$-module $\Gamma(\cX, \cE).$) Any field embedding $\sKC$ makes $\C$ a flat $\OK$-algebra, and the base change morphism
$$(\pi_\ast \cL)_\sigma := \Gamma (\cX, \cE) \otimes_{\OK, \sigma} \C \lra \Gamma (\cX_\sigma, \cE_\sigma)$$
is therefore an isomorphism of (finite dimensional) complex vector spaces. 

In particular $(\pi_\ast \cL)_\sigma$ may be identified to a subspace of the space $\Gamma_{\rm cont} (\cX_\sigma(\C)^\an, \cE_\sigma^\an)$ of continuous sections of the analytic vector bundle $\cE^{\an}_\C$ over $\cX_\sigma(\C)$.\footnote{Actually, according's to Serre's GAGA Theorem, this subspace is precisely the space $\Gamma_{\an} (\cX_\sigma(\C)^\an, \cE_\sigma^\an)$ of $\C$-analytic sections of $\cE^{\an}_\C$ over $\cX_\sigma(\C)$. However we shall not use this fact.}
Accordingly it may be endowed with the $L^\infty$-norm $\Vert .\Vert_{L^\infty, \sigma}$ defined by the equality
$$\Vert s \Vert_{L^\infty, \sigma} := \max_{x \in \cX_\sigma(\C)} \Vert s(x) \Vert,$$
for every $s \in (\pi_\ast \cE)_\sigma \simeq \Gamma (\cX_\sigma, \cE_\sigma)$.

We may also introduce the John's Hermitian norm $\Vert . \Vert_{J, \sigma}$ on $(\pi_\ast \cL)_\sigma$ associated to the $L^\infty$-norm $\Vert .\Vert_{L^\infty, \sigma}$ (see Appendix \ref{ApJohn}). The family of norms $(\Vert .\Vert_{L^\infty, \sigma})_\sKC$  and $(\Vert . \Vert_{J, \sigma})_\sKC$ are clearly invariant under complex conjugation, and we may consider the Hermitian vector bundle over $\Spec \OK$
$$\pi_\ast \cEb_J := (\pi_\ast \cE, (\Vert . \Vert_{J, \sigma})_\sKC).$$

Let us denote by 
$$\pi_\Z := \cX \lra \Spec \Z$$
the (unique) morphism from $\cX$ to $\Spec \Z$. It coincides with the composition $\pi_{\OK/\Z}\circ \pi$ of $\pi$ and of the morphism
$$\pi_{\OK/\Z} : \Spec \OK \lra \Spec \Z.$$
The direct image $\pi_{\Z \ast} \cE$ of $\cE$ under $\pi_\Z$ is a vector bundle over $\Spec \Z$, associated to the free $\Z$-module of finite rank $\Gamma(\cX, \cE)$, and may be identified  with $\pi_{\OK/\Z \ast}( \pi_\ast \cE)$. 

The complex vector space
$$\Gamma(\cX, \cE) \otimes_\Z \C \simeq \Gamma(\cX(\C), \cL_C) \simeq \bigoplus_{\sKC} \Gamma (\cX_\sigma, \cE_\sigma)$$ is naturally endowed with the 
$L^\infty$-norm $\Vert.\Vert_{L^\infty}$, defined by
$$\Vert s \Vert_{L^\infty} := \max_{\sKC} \Vert s_{\mid \cX_\sigma(\C)} \Vert_{L^\infty, \sigma},$$
and we may introduce the associated John's Hermitian norm $\Vert. \Vert_J$. It is invariant under complex conjugation, and allows us to define the hermtian vector bundle over $\Spec \Z$
$$\pi_{\Z\ast} \cEb_J := (\pi_{\Z\ast} \cE, \Vert . \Vert).$$

Observe that, from the compatibility of John's norms with finite products (see Appendix \ref{ApJohn}, paragraph \ref{CompaJohnProd}), we derive that the identification
$$\pi_{\Z \ast} \cE \simeq  \pi_{\OK/\Z \ast}( \pi_\ast \cE)$$
``extends" to an identification of Hermitian vector bundles over $\Spec \Z$: 
$$\pi_{\Z \ast} \cEb_J \simeq  \pi_{\OK/\Z \ast} ( \pi_\ast \cEb_J).$$
(Recall that the direct image under $\pi_{\OK/\Z}$ of some Hermitian vector bundle over $\Spec \OK$ has been defined in Section \ref{Dircan}.)

Finally, we may consider the ``characteristic" $\chi_{\Vert.\Vert_{L^\infty}}( \pi_{\Z \ast} \cE)$ of the free $\Z$-module $\Gamma(\cX, \cE)$ equip\-ped with the norm $\Vert.\Vert_{L^\infty}$, as defined in Appendix  \ref{ApJohn}, paragraph \ref{defChi}. According to the estimates (\ref{chidegaJ}) and (\ref{ineqvn}), we have:
\begin{equation}\label{compchideg}
- (n/2) \log n + n \log 2 \leq \chi_{\Vert.\Vert_{L^\infty}}( \pi_{\Z \ast} \cE) - \dega \pi_{\Z \ast} \cEb_J \leq (n/2) (1 + \log 2\pi),
\end{equation}
where $$n:= \rk_\Z \pi_{\Z \ast} \cE = [K:\Q] \dim_K \Gamma(\cX_K, \cE_K).$$

\subsection{Asymptotic of the $\theta$-invariants associated to sections of powers of arithmetically ample line bundles}

 Let $\cX$ be some reduced scheme, projective and flat over $\Spec \OK$, and let $\cLb:=(\cL, \Vert. \Vert)$ be some Hermitian line bundle over $\cX$, such that the metric $\Vert.\Vert$ is $C^\infty$.
 
 To any integral subscheme $Z$ of $\cX$, by means of the arithmetic intersection theory of Gillet--Soul\'e, one may attach its  \emph{height  with respect to $\cLb$}
 $$h_{\cLb}(Z) := \dega (\hat{c}_1(\cLb)^{\dim Z} \mid Z) \in \R$$
 (see for instance \cite{BostGilletSoule94}, 3.1.1).
 
 Recall that the Hermitian line bundle $\cLb$ is said to be \emph{arithmetically ample}, in the sense of Zhang (\cf \cite{Zhang95}), when the first Chern form $c_1(\cLb)$ is semi-positive on $\cX(\C) = \coprod_{\sKC} \cX_\sigma(\C)$ and when the height $h_{\cLb}(Z)$ of every integral subscheme $Z$ of $\cX$ is positive. (This implies that the line bundle $\cL$ on $\cX$ is ample.)

In particular, if 
$$d := \dim \cX = \dim \cX_K + 1$$ denotes the Krull dimension of $\cX$, the height
$$h_{\cLb}(\cX) := \dega ( \hat{c}_1(\cLb)^d \mid \cX)$$ of $\cX$ with respect to $\cLb$  is positive when $\cLb$ is arithmetically ample (in \cite{Zhang95}, this height is denoted by $\hat{c}_1(\cLb)^d$.)

The following theorem summarizes the basic results on arithmetic ampleness established in  \cite{Zhang95}.

\begin{theorem}[Zhang] \label{recallArAmp} 1) If the line bundle $\cL_K$ on $\cX_K$ is ample and if the first Chern form $c_1(\cLb)$ is semi-positive,  
then, when the positive integer $D$ goes to infinity, we have:
\begin{equation}\label{ZhangAmp1}
\chi_{\Vert.\Vert_{L^\infty}} \left(\pi_{\Z \ast} \cL^{\otimes D}\right) = \frac{1}{d !} h_{\cLb}(\cX) D^d + o(D^d).
\end{equation}

2) If moreover $\cLb$ is arithmetically ample, 
there exists $D_0\in \N$ and  $\eta \in R^\ast_+$ such that, for any integer $D \geq D_0$, there exists a $\Z$-basis of $\pi_{\Z \ast} \cL^{\otimes D}$ consisting of elements $s$ such that
$$\Vert s \Vert_{L^\infty} \leq e^{-\eta D}.$$
\end{theorem}

The first assertion in Theorem \ref{recallArAmp} is actually an extension to a possibly singular scheme $\cX$ of some more precise asymptotic expression for the Arakelov degree of $\pi_{\Z\ast} \cL^{\otimes D}$ equipped with some $L^2$-norm, that is a consequence of the work of Gillet and Soul\'e on higher dimensional Arakelov geometry, and of Bismut, Lebeau, and Vasserot on analytic torsion.

The following theorem may be seen as an arithmetic counterpart of the asymptotic expression (\ref{AsGeo}) in the geometric case, complemented by the vanishing of the higher cohomology groups $H^i(Z, L^{\otimes D})$ for $i>0,$ when $D$ is large enough.

\begin{theorem}\label{AmpleThetaAs} If the Hermitian line bundle $\cLb$ is arithmetically ample on $X$, then there exists $\kappa \in \R^\ast_+$ such that
\begin{equation}\label{eq:AmpleThetaAs0}
\hot(\pi_\ast \cLb^{\otimes D}_J ) = \frac{1}{d !} h_{\cLb}(\cX) D^d + o(D^d) \quad \mbox{when $D \lra + \infty$,}
 \end{equation}
 and
 \begin{equation}\label{eq:AmpleThetaAs1}
\hut(\pi_{\Z\ast} \cLb^{\otimes D}_J ) \leq e^{-e^{\kappa D}} \quad \mbox{for any large enough positive integer $D$.}
 \end{equation}
\end{theorem}
\begin{proof} Observe that
\begin{equation}\label{pfthampl1}
n(D) := \rk_\Z \pi_{\Z \ast} \cL^{\otimes D} = [K:\Q] \dim_K \Gamma(\cX_k, \cL_K^{\otimes D}) = O(D^{\dim \cX_K}) = O(D^{d-1}).
\end{equation}

Let $D_0$ and $\eta$ be as in the second part of Theorem \ref{recallArAmp}. For any $D\geq D_0,$ $\pi_{\Z\ast} \cL^{\otimes D}$ admits a $\Z$-basis consisting of elements $s$ whose John's norms satisfy
$$\Vert s \Vert_J \leq (2n(D))^{1/2} \Vert s \Vert_{L^\infty} \leq (2n(D))^{1/2} e^{- \eta D}.$$  
Therefore then last of the successive minima of the Euclidean lattice $\pi_{\Z\ast} \cLb^{\otimes D}_J$ satisfies:
\begin{equation}\label{pfthampl2}
\lambda_{n(D)}(\pi_{\Z\ast} \cLb^{\otimes D}_J ) \leq (2n(D))^{1/2} e^{- \eta D}.
\end{equation}

The existence of $\kappa >0$ such that (\ref{eq:AmpleThetaAs1}) holds follows from (\ref{pfthampl1}), (\ref{pfthampl2}), and from the upper bound on $\hut$ in terms of the last of the successive minima established in Corollary \ref{Grbdhut}.

According to (\ref{compchideg}) applied to $\cEb := \cLb^{\otimes D}$ and to (\ref{pfthampl1}), we also have:
$$\vert \dega \pi_{\Z\ast} \cLb^{\otimes D}_J  - \chi_{\Vert.\Vert_{L^\infty}} \left(\pi_{\Z \ast} \cL^{\otimes D}\right) \vert = O(n(D) \log n(D)) = O(D^{d-1} \log D).$$
Together with the first assertion in Theorem \ref{recallArAmp}, this shows that
\begin{equation}\label{pfthampl3}
\dega \pi_{\Z\ast} \cLb^{\otimes D}_J = \frac{1}{d !} h_{\cLb}(\cX) D^d + o(D^d) \quad \mbox{when $D \lra + \infty$.}
\end{equation}

Finally, the Poisson-Riemann-Roch formula
$$\hot( \pi_{\Z\ast} \cLb^{\otimes D}_J  ) -\hut (\pi_{\Z\ast} \cLb^{\otimes D}_J ) = \dega \pi_{\Z\ast} \cLb^{\otimes D}_J$$
and the asymptotic expressions (\ref{pfthampl3}) and (\ref{eq:AmpleThetaAs1}) imply (\ref{eq:AmpleThetaAs0}).
\end{proof}

\subsection{Lower bounds on the invariant $\hot$ associated to sections of $\cO(D)$} 

For the proof of the Diophantine  algebraization criteria (Theorems \ref{ArAlgPM} and \ref{ArAlgK}), we will rely on the following weaker variant of Theorem \ref{AmpleThetaAs}:

\begin{corollary}\label{AmpleThetaPM} Let $N$ be a positive integer and let $\cX$ be a closed integral subscheme of $\PP^N_{\OK}$, flat over $\Spec \OK$. Let $\pi: \cX \ra \Spec \OK$ denote its structural morphism.

There exists some Hermitian line bundle $\cLb$ over $\cX$ and two sequences $(\Eb_D)_{D \in \N}$, of Hermitian vector bundles over $\Spec \OK$, and $(\iota_D)_{D \in \N}$, of injective morphisms of $\OK$-modules
$$\iota_D: E_D \hlra \pi_\ast \cL^{\otimes D},$$
such that, for some $D_0 \in \N$ and some $c \in \R_+^\ast,$ the following conditions are satisfied for every integer $D \geq D_0$ and every field  embedding $\sKC$:
\begin{equation}\label{comphermunif}
\Vert s \Vert_{L^\infty, \sigma} \leq \Vert s \Vert_{\Eb_{D},\sigma} \quad \mbox{ for every  section $s\in (\pi_\ast \cL^{\otimes D})_\sigma \simeq \Gamma(\cX_\sigma, \cL_\sigma^{\otimes D})$;}
\end{equation}
and
\begin{equation}\label{lowerDioph}
\hot(\Eb_D) \geq c. D^{\dim \cX}.
\end{equation}
 \end{corollary}
 
 This statement may be seen as a Diophantine counterpart of the geometric lower bounds in Proposition \ref{algbnd}, concerning the dimension of spaces of sections of powers of (ample) line bundles on projective varieties over some field. It is a straightforward consequence of Theorem  
 \ref{AmpleThetaAs}, which actually shows that Corollary   \ref{AmpleThetaPM} holds for any arithmetically ample Hermitian line bundle $\cLb$ over $\cX$, by letting
 $$\Eb_D := \pi_\ast \cLb^{\otimes D}_J$$ 
 for every positive integer $D$.
      
      However, like its geometric counterpart in Proposition \ref{algbnd}, it also admits a direct 
       proof, using Noether's normalization and some basic results concerning sections  of the line bundles $\cO(D)$ on the projective spaces (compare paragraph  \ref{CommentsDim}).
       
    This direct proof will use some elementary properties of natural Hermitian norms on symmetric powers of Hermitian vector spaces and on spaces of homogeneous polynomials that we now recall.
       
       Let $k$ be some natural integer. For any $\delta \in \R,$ let us equip 
       $$\C^{k+1 \vee} := \bigoplus_{0\leq i \leq k} \C X_i$$
       with the Hermitian norm $\Vert.\Vert_\delta$ such that $(X_i)_{0\leq i \leq k}$ becomes an orthogonal basis and 
       $$\Vert X_0 \Vert_\delta = \ldots =\Vert X_k \Vert_\delta = e^{- \delta}.$$ 
       Then, for any $D \in \N,$ we may equip 
       $$\C[X_0, \ldots, X_k]_D \simeq S^D(\C^{k+1 \vee})$$
       with the  Hermitian norm $\Vert.\Vert_{D,\delta}$ deduced from $\Vert.\Vert_\delta$ by taking its $D$-th symmetric power; by definition, it is the quotient norm of the Hermitian norm on $(\C^{k+1 \vee})^{\otimes D}$ deduced from $\Vert.\Vert_\delta$ by tensor product. 
       
       Let $$M_D := \{ I:= (i_0, \ldots, i_k) \in \N^{k+1}
       \mid  \vert I \vert := i_0 + \ldots i_k = D \},$$ and, as usual, let
       $$X^I := X_0^{i_0}\ldots X_k^{i_k} \quad\mbox{and}\quad I! := i_0! \ldots i_k ! \quad\mbox{for any $I=(i_0, \ldots,i_k) \in \N^{k+1}.$}$$
A straightforward computation shows that
$(X^I)_{I \in M_D}$ is an orthogonal basis of $\C[X_0, \ldots, X_k]_D$ equipped with $\Vert.\Vert_{D, \delta}$ and that 
\begin{equation}\label{XIdelta}
\Vert X^I \Vert_{D,\delta}^2 = \frac{I!}{D!} e^{-2 \delta D} \leq e^{-2 \delta D}.
\end{equation}

Besides, we may equip the canonical quotient bundle $\cO(1)$ on $\PP^k_\C:= \PP(\C^{k+1 \vee})$ with the Hermitian norm deduced from the norm $\Vert.\Vert_\delta$ on  $\C^{k+1 \vee}$. For any $D$ in $\N$, this Hermitian norm induces, by tensor product, some Hermitian norm $\Vert.\Vert_{\cO(D), \delta}$ on $\cO(D)$ over $\PP^k_\C.$

For any $P$ in $\C[X_0, \ldots, X_k]_D$, identified to some section in $\Gamma(\PP^k_\C, \cO(D))$, and any point $$\mathbf{x} := (x_0: \ldots : x_k)$$ in the projective space $\PP^k(\C),$ the norm of the value $P(\mathbf{x}) \in\cO(D)_{\mathbf{x}}$ of $P$ at  $\mathbf{x}$ satisfies, as a straightforward consequence of the definitions:
\begin{equation*}
\Vert P(\mathbf{x}) \Vert_{\cO(D), \delta} = e^{-\delta D} \frac{\vert P(x_0, \ldots, x_k) \vert}{(\sum_{i= 0}^k \vert x_i \vert^2)^{D/2}}.
\end{equation*}

Observe that, if 
$$P = \sum_{ I \in M_D} a_I X^I,$$
then, according to (\ref{XIdelta}), 
$$\Vert P \Vert^2_{D, \delta} = \sum_{I \in M_D} \vert a_I \vert^2 \Vert X^I \Vert^2_{D, \delta} = \sum_{I \in M_D} \vert a_I \vert^2 \frac{I!}{D!} e^{-2 \delta D}.$$
Therefore, by Cauchy-Schwarz inequality,
$$ \vert P(x_0, \ldots, x_k) \vert^2 = \left\vert \sum_{ I \in M_D} a_I x^I \right\vert^2 \leq  
 \Vert P \Vert^2_{D, \delta} \sum_{I \in M_D} \vert x^{2I} \vert \frac{D!}{I!} e^{2\delta D} 
 = \Vert P \Vert^2_{D, \delta} \, e^{2 \delta D} (\sum_{i=0}^n \vert x_i \vert^2)^D.$$
This shows that
\begin{equation}\label{boundDdelta}
\Vert P(\mathbf{x}) \Vert_{\cO(D), \delta} \leq \Vert P \Vert_{D,\delta}.
\end{equation}

 \proof[Direct proof of Corollary \ref{AmpleThetaPM}] Let $d:= \dim \cX$. It is a positive integer and 
      $$\dim \cX_K = d-1.$$
After possibly  composing the embedding $\cX \hra \PP^N_\OK$ by the automorphism of $\PP^N_\OK$ defined by some element $g \in SL_{k+1}(\OK),$ we may assume that $\cX_K$ is disjoint from the linear subspace
$$\PP^{N-d}_K = \bigcap_{i=0}^{d-1} \div X_i$$
defined by the vanishing of the first $d$ homogeneous coordinates $X_0, \ldots, X_{d-1}$\footnote{Indeed, by properness of $\PP_K^{N-d}$, the condition $g.\cX_K \cap \PP_K^{N-d} =\emptyset$ on $g$ defines some open subscheme in $SL_{N+1,K}$. By Noether's normalization, this open subscheme is not empty, and therefore meets $SL_{N+1}(\OK)$, which is Zariski dense in $SL_{N+1,K}$.}.  Then the linear projection
$$
\begin{array}{rrcl}
p:  &   \PP^N_{\OK} \setminus \PP^{N-d}_{\OK} & \lra &\PP_{\OK}^{d-1}    \\
& (x_0: \ldots :x_N)  & \longmapsto   & (x_0: \ldots x_{d-1})  
\end{array}
$$
has a well defined restriction
$$p_{\mid \cX_K}: \cX_K \lra \PP^{d-1}_K,$$
which is a finite surjective morphism.
     
     For any $D \in \N,$ we define
     $$E_D := \Gamma(\PP^{d-1}_\OK, \cO(D)) \simeq \OK[X_0, \ldots, X_{d-1}]_D.$$
    The pull-back  $p^\ast\cO(1)$ may be identified with the restriction $\cO(1)_{\mid  \PP^N_\OK \setminus \PP^{N-d}_\OK},$ and the pull-back map
    $$p^\ast: E_D \lra \Gamma(\PP^N_\OK \setminus \PP^{N-d}_\OK, p^\ast \cO(D)) \simeq \Gamma(\PP^N_\OK \setminus \PP^{N-d}_\OK, \cO(D))$$  
    has its image contained in   $\Gamma(\PP^N_{\OK}, \cO(D))$ and may actually   
 be identified with the inclusion morphism 
$$\OK[X_0, \ldots, X_{d-1}]_D \hlra \OK[X_0, \ldots, X_{N}]_D.$$

Besides, let us define
$$\cL := \cO(1)_{\mid \cX}.$$
By composing $p^\ast$ and the map ``restriction to $\cX$", we define a morphism of $\OK$-modules.
$$i_D : E_D \stackrel{p^\ast}{\lra} \Gamma(\PP^N_\OK, \cO(D)) \stackrel{\mbox{\tiny restriction}}{\lra} \Gamma(\cX, \cO(D)) \simeq \Gamma(\cX, \cL^{\otimes D}).$$

The pull-back $p_{\mid \cX_K}^\ast \cO(1)$ coincides with $\cO(1)_{\cX_K} \simeq \cL_K$ and therefore the map
$$i_{D, K} : E_{D,K} = \Gamma (\PP^{d-1}_K, \cO(D)) \lra \Gamma(\cX_K, \cL^{\otimes D}_K)$$
may be identified with the pull-back by the finite surjective morphism $p_{\mid \cX_K}$, and is therefore injective.

To complete the proof of Corollary \ref{AmpleThetaPM}, we are left to show that there exist Hermitian metrics on $\cL$ and on the $E_D$'s such that the conditions (\ref{comphermunif}) and (\ref{lowerDioph}) are satisfied.

To achieve this, let us choose $\delta \in \R^\ast_+,$ and let us use the above construction of the metrics $\Vert.\Vert_{\cO(1), \delta}$ (on the line bundle $\cO(1)$ over $\PP_\C^k$) and $\Vert.\Vert_{D, \delta}$ (on  $\Gamma(\PP^k_\C, \cO(D))$) in the special case $k= d-1.$

For every embedding $\sKC,$ we have an identification
$\cL_\sigma \simeq p^\ast_{\sigma \mid \cX_\sigma} \cO(1)$
of line bundles over $\cX_\sigma$. We define $\cLb$ as $\cL$ equipped with the metric defined on $\cX_\sigma$ as the pull-back of $\Vert.\Vert_{\cO(1), \delta}$ by $p_\sigma$. Besides, we define $\Eb_D$  as $E_D$ equipped with the metric $\Vert.\Vert_{\Eb_D, \sigma} := \Vert.\Vert_{D, \delta}$ for every embedding $\sKC.$

With these choices of Hermitian structures, the validity of (\ref{comphermunif}) follow from the norm estimates (\ref{boundDdelta}).

To establish (\ref{lowerDioph}), let us observe that $\Eb_D$ may be written as the direct sum of Hermitian line bundles
$$\Eb_D = \bigoplus_{I \in M_D} \oli{\OK X^I}$$
and that, according to (\ref{XIdelta}), 
$$\dega \oli{\OK X^I}  = - [K:\Q] (1/2) \log(\frac{I!}{D!} e^{-2 \delta D}) \geq [K:\Q] \delta D.$$
Therefore
$$\dega \Eb_D \geq [K:\Q] \delta D \,\rk E_D,$$
and, by the ``Riemann inequality" (\ref{ThetaRIK}), 
\begin{equation*}\begin{split}\hot(\Eb_D) & \geq \dega \Eb_D - \frac{1}{2} \log \vert \Delta_{K} \vert   \cdot \rk E_D\\
& \geq ([K:\Q] \,\delta D - \frac{1}{2} \log \vert \Delta_{K})\, \rk E_D.
\end{split}
\end{equation*}
Since
$$\rk E_D = \binom{D+d-1}{d-1}
\geq \frac{1}{(d-1)!} D^{(d-1)},$$
this proves that, for any given $c < [K:\Q]\, \delta /(d-1)!,$ we have: 
$$\hot(\Eb_D) \geq c. D^d$$
when $D$ is large enough. 
\qed

   \section{Pointed smooth formal curves}\label{SmForCur}
   
   Let us consider some noetherian affine base scheme $S:= \Spec A.$
   
  \subsection{Definitions} We define a \emph{pointed smooth formal curve over $S$} as the data $({\cCh}, \pi, \cP)$ of some noetherian affine formal scheme\footnote{See EGA I (\cite{EGAI}), Section 1.10.} $\cCh = \Spf B$, of a morphism of (formal) schemes 
   $$\pi : \cCh \lra S,$$
   and of some section 
   $$\cP : S \lra \cCh$$
   of $\pi$ such that the following conditions are satisfied:
   
   $\mathbf{PSFC_1 :}$ \emph{The ideal $I$ of $B$ defined as the kernel of the morphism of rings
   $$\cP^\ast :   B= \Gamma(\cCh, \cO_{\cCh}) \lra A = \Gamma(S, \cO_S)$$ 
   is an ideal of definition of the} (noetherian adic) \emph{ring $B$}.
    
   $\mathbf{PSFC_2 :}$ \emph{The morphism of topological rings
   $$\pi^\ast:  A = \Gamma(S, \cO_S) \lra B= \Gamma(\cCh, \cO_{\cCh})$$
   makes the} (noetherian adic) \emph{ring $B$ a formally smooth\footnote{See EGA ${\rm{IV}_1}$ (\cite{EGAIV1}), Chapitre $\mathbf{0}$, 19.3.1.} 
   algebra over $A$} (equipped with the discrete topology).
   
    $\mathbf{PSFC_3 :}$ \emph{The fiber of $\pi$ over any point of $S$ with value in a field is one-dimensional}. 
    
    \smallskip
    
    When   $\mathbf{PSFC_1}$ is satisfied, we may introduce $\cI := I^{\Delta}$ the Ideal of definition of $\cCh$ corresponding to $I$. For any $n \in \N,$ we may consider the affine $S$-schemes of finite type 
    $$C_n := (\vert \cCh\vert, \cO_{\cCh}/\cI^{n+1}) \simeq \Spec B/I^{n+1}$$
    and $\cCh$ may be identified with their direct limit in the category of formal schemes:
    $$\cCh \simeq \varinjlim_n C_n.$$
    Moreover the morphisms $\pi_{\vert X_0}$ and $\cP$ define isomorphisms, inverse to each other, between $C_0$ and $S$.
    
    Still assuming that $\mathbf{PSFC_1}$ holds, Conditions $\mathbf{PSFC_2}$ and $\mathbf{PSFC_3}$ are equivalent to the following one:
    
    $\mathbf{PSFC'_{2-3} :}$ \emph{The coherent sheaf $\cI/\cI^2$ over $C_0$ ($\simeq \cP)$ is a line bundle and, for any positive integer $n$, the canonical morphism of $\cO_{C_0}$-Modules\footnote{For every local section $s$ of $\cI$, of class $[s]$ in $\cI/\cI^2$, $\phi_n$ maps $[s]^{\otimes n}$ to the class of $s^n$ in $\cI^n/\cI^{n+1}$.}
    $$\phi_n : (\cI/\cI^2)^{\otimes n} \lra \cI^n/\cI^{n+1}$$
    is an isomorphism.}
    
    This follows from EGA ${\rm{IV}_1}$ (\cite{EGAIV1}),  Chapitre $\mathbf{0}$, 19.5.4, equivalence of a) and b).
    
    When Conditions $\mathbf{PSFC_{1-3}}$ hold, the line bundle $\cI/\cI^2$ (or its pull-back $\cP^\ast \cI/\cI^2$ over $S$) may be called the conormal bundle of $\cP$ in $\cCh$. We will also consider its dual, 
    $$N_\cP \cCh := \cP^\ast (\cI/\cI^2)^\vee,$$
    the \emph{normal bundle of $\cP$ in $\cCh$.}
    
   \subsection{Examples and remarks}\label{ExRemSFC} {\bf (i)} Observe that, if 
    $f : C \lra S$
    is a smooth relative curve over $S$ (namely, a smooth morphism of schemes of relative dimension $1$), any section $P: S \lra C$ of $f$ becomes, after possibly replacing  $C$ by an open subscheme\footnote{This is not needed if $f$, or equivalently $C$, is separated.}, a closed immersion. We may therefore consider the formal completion $\cCh$ of $C$ along the image of $P$. By completion, from $f$ and $\cP$, one deduces morphisms of formal schemes
    $$\pi: \cCh \lra S \quad \mbox{and} \quad \cP: S \lra \cCh,$$
    and it is straightforward that $(\cCh, \pi, \cP)$ defines a pointed smooth formal curve over $S$. If $T_f$ denote the relative tangent bundle of $f$, we have a canonical isomorphism:
    $$N_\cP \cCh \lrasim \cP^\ast T_f.$$
    
    \smallskip
    
  {\bf (ii)}    For any line bundles $\cL$ over $S$, we may consider the affine $S$-scheme
    $${\V}_S(\cL^\vee) := {\rm Spec}_S\,  {\rm Sym}_{\cO_S} \cL^{\vee}.$$
 (It is defined as the spectrum of the quasi-coherent $\cO_S$-algebra
    $${\rm Sym}_{\cO_S} \cL^{\vee} := \bigoplus_{n \in \N} \cL^{\vee \otimes n},$$
    and may be thought as the ``total space" of the line bundle $\cL$.) It is smooth over $S$, of relative dimension $1$, and equipped with the ``zero-section" $\epsilon$, defined by the augmentation ideal 
    $$\cI_\epsilon := 0 \oplus \bigoplus_{n \in \N_{>0}} \cL^{\vee \otimes n} = \cL^\vee . {\rm Sym}_{\cO_S} \cL^{\vee}$$
    of the symmetric algebra ${\rm Sym}_{\cO_S} \cL^{\vee}.$ 
    
    Applied to $C := {\V}_S(\cL^\vee)$ and $P:= \epsilon$, the previous construction defines a pointed smooth formal curve $(\widehat{\V}_S(\cL^\vee), \pi, \epsilon)$.
     By construction, we have a canonical isomorphism
    $$N_{\epsilon} \widehat{\V}_S(\cL^\vee) \lrasim \cL.$$
    
    In concrete terms, the $S$-formal scheme $\widehat{\V}_S(\cL^\vee)$ may be described as follows in terms of the projective $A$-module of rank one $L^\vee := \Gamma(S, \cL^\vee)$ : it is the affine formal scheme defined by the $A$-algebra 
    $$B:= \widehat{\rm Sym}_A L^\vee,$$
    defined as the completion of the symmetric algebra
    $${\rm Sym}_A L^\vee := \bigoplus_{n \in\N} L^{\vee \otimes_A n}$$
    with respect to the adic topology defined by the augmentation ideal
    $$I:= L^\vee. {\rm Sym}_A L^\vee = 0 \oplus \bigoplus_{n \in\N_{>0}} L^{\vee \otimes_A n}.$$
    As a topological $A$-module, 
    $$B \lrasim \prod_{n \in\N} L^{\vee \otimes_A n},$$
    where the right-hand side is equipped with product topology of the discrete topology on each of the factors  $L^{\vee \otimes_A n}$.
    
    It turns out that any pointed smooth formal curve $(\cCh, \pi, \cP)$ over $S$ is isomorphic (non-canonically in general) to $(\widehat{\V}_S(\cL^\vee), \pi, \epsilon)$ where $\cL:= N_{\cP} \cCh$ (see EGA ${\rm{IV}_1}$ (\cite{EGAIV1}),  Chapitre $\mathbf{0}$, 19.5.4, equivalence of a) and c)).
    
   When $S$ is the spectrum $\Spec k$ of a field $k$, a pointed smooth formal curve over $S$ is often called \emph{a smooth formal germ of curve over $k$}, and the previous classification result boils down to the classical fact that any of them is isomorphic to $\Spf k[[T]].$
    
     \smallskip
   
{\bf (iii)} Let us finally observe that, for any morphism of affine noetherian schemes $S' \lra S$, we may ``base change" any pointed  smooth formal curve $(\cCh, \pi, \cP)$ over $S$ to $S'$, and define  a  pointed  smooth formal curve $(\cCh_{S'}, \pi_{S'}, \cP_{S'})$ over $S'$.
     
      \section[Green's functions, capacitary metrics and Schwarz lemma]{Green's functions, capacitary metrics and Schwarz lemma on compact Riemann surfaces with boundary}\label{boundaryGreen}
      
 \subsection{Compact Riemann surfaces with boundary}\label{DefCRB}     A \emph{Riemann surface with boundary} is the data $(V,V^+)$ of some Riemann surface $V^+$ (without boundary) and of some closed $C^\infty$ submanifold with boundary (of codimension 0) $V$ of $V^+$.
 
 The interior $\mathring{V}$ of $V$ in $V^+$ is then a Riemann surface (without boundary), its boundary 
 $$\partial V := V \setminus \mathring{V}$$
 is a closed $1$-dimensional real $C^\infty$ submanifold of $V^+,$
 and we may consider the open and closed immersions:
      $$\mathring{V} \stackrel{i_V}{\hlra} V \stackrel{j_V}{\hlra} V^+.$$
      
      Actually we shall consider that shrinking the ``ambient" Riemann surface $V^+$ to some smaller open neighborhood $V^{'+}$ of $V$ does not change the surface with boundary $(V, V^+)$. In other words, $V^+$ should be understood as a germ of Riemann surface around $V$. 
      
      An alternative definition, formally more satisfactory, would be to define the Riemann surface with boundary associated to $(V,V^+)$ as the ringed space $(V, \cO^{\an}_V)$,  where $\cO^{\an}_V$ denotes the inverse image $j_V^{-1} \cO^{\an}_{V^+}$ of the sheaf of $\C$-analytic functions $\cO^{\an}_{V^+}$ on $V^+.$ 

           An \emph{analytic vector bundle} over the Riemann surface with boundary $V$ is, by definition, a germ of analytic vector bundle along $V$ in $V^+$. In other words, by its very definition,  an analytic vector bundle over $V$ extends to an analytic vector bundle on any small enough open neighborhood of $V$ in $V^+$. (Such analytic vector bundles precisely correspond to locall free sheaves of finite rank over the ringed space $(V, \cO^{\an}_V)$.)
      
 A $C^\infty$ Hermitian metric $\Vert. \Vert$ on some analytic vector bundle  $E$ on $V$ is, by definition, a $C^\infty$ Hermitian metric on the $C^\infty$ vector bundle $E_{\mid V}$ on the $C^\infty$ manifold (with boundary) $V$. The pair $(E, \Vert. \Vert)$ will the be called a \emph{Hermitian analytic vector bundle} over $V$. 
      
       Clearly, usual tensor operations (such as direct sums, or tensor products) make sense  for analytic vector bundles and Hermitian analytic vector bundles over Riemann surfaces with boundary.

       A \emph{volume form} $\nu$ on the Riemann surface with boundary $V$ is a $C^\infty$ $2$-form on the $C^\infty$ manifold (with boundary) $V$ which is everywhere positive. 
       
       \subsection{Green's functions and capacitary metrics} In this paragraph, we recall some well-known facts concerning Green's functions on compact Riemann surfaces with boundary. For a detailed discussion and references on this topic, we refer the reader to \cite{Bost99}, Section 3.1 and Appendix. We shall actually not need the ``refined" theory (that allows domains with rough boundaries) presented in \emph{loc. cit.}, and the results below may also be obtained as special cases of some basic results concerning elliptic boundary problems (see for instance \cite{Taylor2011}, Sections 5.1-2).
       
         Let $V$ be a \emph{connected compact} Riemann surface with boundary $V$ such that $\partial V$ is \emph{non-empty.}
       
       For any point $O$ in $\mathring{V},$ one defines the \emph{Green's function $g_{V,O}$ of $O$ in $V$} as the unique function $$g_{V,O}: V \setminus\{O\} \lra \R$$
       that satisfies the following conditions:
       
      $\mathbf{Gr_1 :}$ \emph{$g_{V,O}$ is continuous on $V\setminus\{O\}$ and $g_{V,O\mid \partial V} = 0.$}
      
      $\mathbf{Gr_2 :}$ \emph{$g_{V,O}$ is harmonic on $\mathring{V}\setminus\{O\}$.}
      
      $\mathbf{Gr_3 :}$ \emph{$g_{V,O}$ admits a logarithmic singularity at $O$;} namely, if $z: U \hlra U$ denotes some analytic chart of domain some open neighborhood $U$ of $O$ in  $\mathring{V}$, we have:
      \begin{equation*}
g_{V,O}(P) = \log \vert z(P) -z(O) \vert^{-1} + O(1) \quad \mbox{when $P\ra O$}.
\end{equation*}

When conditions $\mathbf{Gr_1}$ and $\mathbf{Gr_2}$ are satisfied, the function of $P$
$$g_{V,O}(P) - \log \vert z(P) -z(O) \vert^{-1}$$
is harmonic on $U\setminus\{O\}$ and stays bounded when $P$ goes to $O$, hence extends to some harmonic function $h$ on $U$. By construction, we have, for every $P \in U \setminus\{O\},$
     \begin{equation*}
g_{V,O}(P) = \Vert \partial/\partial z \Vert^{\rm cap}_{V,O} + h(P).
\end{equation*}
In particular, $g_{V,O}$ defines some locally $L^1$ function on $\mathring{V}$, hence a current on $\mathring{V}$. Together with the Poincar\'e-Lelong equation on $\C$
 $$(i/2\pi) \,\partial \overline{\partial} \log \vert . \vert^2  =\delta_0,$$
 these observations show that, conditions $\mathbf{Gr_2}$ and $\mathbf{Gr_3}$ imply:
 
   $\mathbf{Gr_{2-3} :}$ \emph{$g_{V,O}$ defines a current on $\mathring{V}$ which satisfies the equation of currents:}
 \begin{equation}\label{Greendd}
       i\, \partial \overline{\partial} g_{V,O} = - \pi\, \delta_O 
\end{equation}

Conversely, using that any current in the kernel of $\partial \overline{\partial}$ is actually some $C^\infty$ (harmonic) function, one easily see that $\mathbf{Gr_{2-3}}$ implies $\mathbf{Gr_2}$ and $\mathbf{Gr_3}$.

  With the above notation, the  \emph{capacitary metric} $\Vert . \Vert_{V,O}^{\rm cap}$ on the tangent space $T_O V$ of $V$ at $O$ is the Hermitian metric on this complex line defined by
  \begin{equation}\label{capmet1}
\Vert \partial/\partial z \Vert^{\rm cap}_{V,O} := e^{-h(P)}.
\end{equation}
In other words,
\begin{equation}\label{capmet2}
\log \Vert \partial/\partial z \Vert^{\rm cap}_{V,O} = \lim_{P \ra O} \left[ \log \vert z(P) -z(O) \vert^{-1} -
g_{V,O}(P)\right].
\end{equation}
The relation (\ref{capmet2}) makes clear that the definition (\ref{capmet1}) of the capacitary metric is actually independent of the choice of the local coordinate $z$. 

Let us finally recall that $g_{V,O}$ is positive on $\mathring{V} \setminus\{O\}$ (this is a straightforward consequence of the maximum principle for harmonic functions on Riemann surfaces). Moreover, $g_{V,O}$ is $C^\infty$ up to the boundary --- that is $C^\infty$ on the surface with boundary $V\setminus\{O\}$ --- and its differential $dg_{V,O}$ does not vanish on $\partial V$  (see for instance \cite{Taylor2011}, 5.1-2).

       \subsection{Examples}\label{ExampGreen}
       
       (i) For any $R\in \R_+,$ we let
       $$\mathring{D}(0, R) := \{ z \in \C \mid \vert z \vert < R \}$$
       and
       $$\overline{D}(0,R) := \{ z \in \C \mid \vert z \vert \leq  R \}.$$

      When $R>0,$ $\overline{D}(0,R)$ is a compact connected Riemann surface with non-empty boundary. It is straightforward that, for any $z \in \overline{D}(0,R)\setminus \{0\},$
      $$g_{\overline{D}(0,R), 0}(z) = \log \frac{R}{\vert z \vert}.$$
      Consequently, we have: 
      $$\Vert \partial / \partial z \Vert_{\overline{D}(0,R), 0}^{\rm cap} = 1/R.$$
    \smallskip
      
      (ii) Let $V \hra \PP^1(\C)$ be a closed submanifold with boundary (of codimension $0$) of the complex projective line. Let us assume that $V$ is connected and contains the point $\infty := (0:1)$.
      
      On $\PP^1_\C := {\rm Proj} (\C X_0 \oplus \C X_1)$, we may consider the rational function $z:= X_1/X_0$ --- it defines the usual identification $\PP^1_\C \setminus \{\infty\} \lrasim \A^1_\C$ --- and its inverse $t:= X_0/X_1$, which defines a local coordinate on $\PP^1(\C) \setminus \{0\}$, that vanishes at $\infty$.
      
      By the very definition of the Green's function and the capacitary metric attached to the point $\infty$ of $\mathring{V},$ we have, when $P\in \A^1(\C)$ goes to $\infty$:
\begin{align*}
 g_{V, \infty}(P) & = \log \vert t(P) \vert^{-1} - \log \Vert \partial/\partial t \Vert^{\rm cap}_{V,\infty} + o(1) \\
 & = \log \vert z(P) \vert - \log \Vert \partial/\partial t \Vert^{\rm cap}_{V,\infty} + o(1).
\end{align*}
 
 This shows that $g_{V,\infty}$ coincides with the classical Green's function of the compact set $K: = \A^1(\C) \setminus \mathring{V}$, and  that $-\log \Vert \partial/\partial t \Vert^{\rm cap}_{V,\infty}$ is the so-called Robin constant of $K$ and $\Vert \partial/\partial t \Vert^{\rm cap}_{V,\infty}$ its two-dimensional capacity $c(K)$ (see for instance \cite{Ransford95}, Chapter 5).
 
 \smallskip
      
      (iii) Let $F:=\{a_1, \ldots, a_n\}$ be a finite subset of $\mathring{D}(0,R)\setminus \{0\}$ and let 
      $$\boldsymbol{\epsilon} := (\epsilon_1, \ldots, \epsilon_n)$$ be an element of $\R_+^{\ast n}$ such that
      $$ \epsilon_i < R -\vert a_i \vert \quad \mbox{ for any $i \in \{1, \ldots, n\}$}$$
      and
      $$ \epsilon_i + \epsilon_j <\vert a_i - a_j \vert \quad \mbox{ for any $(i,j) \in \{1, \ldots, n\}^2$ such that $i \neq j$.}$$
      
      Then the disks $\overline{D}(a_i, \epsilon_i)$ are contained in $\overline{D}(0,R)$ and pairwise disjoint, and 
      $$V_{\boldsymbol{\epsilon}} := \overline{D}(0,R) \setminus \bigcup_{1 \leq i \leq n} \mathring{D}(a_i, \epsilon_i)$$
      is a compact Riemann surface with boundary.
      
      When $\boldsymbol{\epsilon}$ goes to $(0,\ldots,0)$ in $\R^n,$ the surface $V_{\boldsymbol{\epsilon}}$ shrinks to  $\overline{D}(0,R) \setminus F$. Moreover, the fact that the finite subset $F$ is $\emph{polar}$ in $\C$ implies that, for any $z \in \overline{D}(0,R)\setminus \{0\},$
      $$\lim_{\boldsymbol{\epsilon}\ra (0,\ldots,0)} g_{V_{\boldsymbol{\epsilon}}, 0} (z) = g_{\overline{D}(0,R), 0}(z) = \log \frac{R}{\vert z \vert}$$
    and 
  \begin{equation}\label{caplim}
\lim_{\boldsymbol{\epsilon}\ra (0,\ldots,0)} \Vert \partial / \partial z \Vert_{V_{\boldsymbol{\epsilon}}, 0}^{\rm cap}
= \Vert \partial / \partial z \Vert_{\overline{D}(0,R), 0}^{\rm cap} =1/R.
\end{equation}
  (See \cite{Ransford95}, notably the properties of the capacity discussed in Chapter 5.)

  \subsection{Hilbert spaces of analytic sections and Fr\'echet spaces of formal sections}\label{HilFre}
  
  Let   $V$ be a compact Riemann surface with boundary, $\nu$ a volume form on $V$, and $(E, \Vert.\Vert)$ some Hermitian analytic vector bundle over $\mathring{V}$, namely some analytic vector bundle $E$ over $\mathring{V}$, equipped with some Hermitian metric $\Vert.\Vert$, $C^\infty$ on $\mathring{V}$.
  
  The space of analytic sections $s$ of $E$ over $\mathring{V}$ such 
  $$\Vert s \Vert_{L^2} := \int_{\mathring{V}} \Vert s \Vert^2 \, \nu < + \infty,$$
when equipped with the $L^2$-norm $\Vert . \Vert_{L^2}$, defines a Hilbert space, that we shall denote by 
$$\Gamma_{L^2} (V, \nu; E, \Vert. \Vert).$$

Let us choose a point $O$ in $\mathring{V}$. We may consider its successive infinitesimal neighborhoods\footnote{By definition, for any $n\in \N$, $O_n$ is defined by the sheaf of ideals $\cI_O^{n+}$, where $\cI_O$ denotes the ideal sheaf of $O$. It is convenient to extend this definition to $n=-1,$ so that $O_{-1} = \emptyset$.} $O_n$ in $V$, and the formal completion 
$$\hat{V}_O :=  \varinjlim_n O_n$$
of $V$ at $O$. The space of sections of $E$ over $\hat{V}_O$,
$$\Gamma(\hV_O, E) \lrasim \varprojlim_n E_{\mid O_n},$$
is equipped  with a natural Fr\'echet space topology, as projective limit of the finite dimensional complex vector spaces $E_{\mid O_n}$ endowed with their natural topology of (separated) topological vector spaces.

We may finally introduce the restriction map
$$\begin{array}{crcl}
\hat{\eta}: & \Gamma_{L^2} (V, \nu; E, \Vert. \Vert)   &  \lra  & \Gamma(\hV_O, E)  \\
 & s & \longmapsto  & s_{\mid \hV_O},
 \end{array}
$$
which maps some $L^2$ holomorphic section $s$ of $E$ over $\mathring{V}$ to its ``jet" $s_{\mid \hV_O}$ at the point $O$.

\begin{proposition}\label{propeta} Let us assume that $V$ is connected, and that its boundary $\partial V$ is non-empty.
 Then the linear map $\hat{\eta}$ is injective, and continuous  when $\Gamma_{L^2} (V, \nu; E, \Vert. \Vert)$ is equipped with its Hilbert space topology and $\Gamma(\hV_O, E)$ with is natural Fr\'echet space topology.
 
 If moreover, the vector bundle $E$ is the restriction to $\mathring{V}$ of some analytic vector bundle $\tilde{E}$ on $V$ and if the Hermitian metric $\Vert.\Vert$ extends to some $C^\infty$ metric on $\tilde{E}$ over $V$, then the image of $\hat{\eta}$ is dense in $\Gamma(\hV_O, E)$.
\end{proposition}

\begin{proof}
 The injectivity of $\hat{\eta}$ follows from the connectedness of $\mathring{V}$ by analytic continuation.
 
 The continuity of $\hat{\eta}$ is equivalent to the continuity of the maps 
 $$\begin{array}{crcl}
{\eta}_n: & \Gamma_{L^2} (V, \nu; E, \Vert. \Vert)   &  \lra  & E_{\mid O_n} \\
 & s & \longmapsto  & s_{\mid O_n},
 \end{array}
$$
when $n \in \N.$ This directly follows for Cauchy estimates.

The density in $\Gamma(\hV_O, E)$ of the image of $\hat{\eta}$ is equivalent to the surjectivity of the maps $\eta_n$. This surjectivity follows from the surjectivity of the composite maps
$$\Gamma(V^+, \tilde{E}) \hlra \Gamma_{L^2} (V, \nu; E, \Vert. \Vert)   \stackrel{\eta_n}{\lra} E_{O_n},$$
which in turn is a consequence of the fact that $V$ is a Stein compact subset (that is, admits a basis of open Stein neighborhoods) in $V^+$, since any connected non-compact Riemann surface is a Stein manifold (theorem of Behnke-Stein; see also \cite{GuenotNarasimhan76}, Chapitre V, Th\'eor\`eme 1 for a simple proof).
\end{proof}

Let us observe that, for any $n \in \N$, if $s$ denotes some analytic section of $E$ on some open neighborhood of  $O$ in $\mathring{V}$ that vanishes at order at least $n$ at $O$ --- in other words, that satisfies 
$${\eta}_{n-1}(s) := s_{\mid O_{n-1}} = 0$$ --- then its jet of order $n$ at $O$, 
$${\eta}_{n}(s) := s_{\mid O_{n}}$$ may be identified with some element of $T_O^{\vee \otimes n} \otimes E_O$.

Let us also recall that, if $\Lb := (L, \Vert .\Vert)$ denotes some Hermitian analytic line bundle\footnote{that is some Hermitian analytic vector bundle of rank one} over $\mathring{V}$, then its \emph{first Chern form} $c_1(\Lb)$ is a real $(1,1)$-form defined by the equality
\begin{equation*}
c_1(\Lb) =\frac{1}{2\pi i} \partial\overline{\partial} \log \Vert s \Vert^2 
\end{equation*}
 for any local non-vanishing analytic section of $L$. More generally, if $s$ denotes a non-zero meromorphic section of $L$ over some connected open subset $U$ of $\mathring{V}$ and $\div s$ its divisor, the function $\log \Vert s \Vert$ is locally $L^1$ on $U$ and satisfies the following relation of currents on $U$:
\begin{equation}\label{PL}
\frac{i}{\pi} \partial\overline{\partial} \log \Vert s \Vert =  \delta_{\div s} - c_1(\Lb)
\end{equation}
(Poincar\'e-Lelong equation).

\subsection{Poisson-Jensen formula and Schwarz lemma over a compact Riemann surface with boundary}\label{SchwarzLemma}

Let $V$ be a connected compact Riemann surface with non-empty boundary, and let $O$ be some point of $\mathring{V}$.

Let $\eta$ be a positive real number and let $\rho: \R \lra \R_+$ be $C^\infty$ function with compact support such that
\begin{equation}\label{rhoeta}
{\rm supp\,} \rho \subset ]0,\eta[ \quad\mbox{and}\quad \int_\R \rho(t) = 1.
\end{equation} 
Let us denote by $\chi$ the ``second primitive" of $\rho$, namely the $C^\infty$ function from $\R$ to $\R$ defined by
$$\chi(0) = \chi'(0) =0 \quad\mbox{and}\quad \chi'' = \rho.$$
It is a non-decreasing convex function with values in $\R_+$. Moreover, there exists $c\in \R$ such that\footnote{A straigtforward computation shows that $c=-\int_0^{+\infty} t \rho(t) dt \in ]-\eta, 0[.$} 
\begin{equation}\label{chiinf}
\mbox{for any $t \in [\eta, +\infty[$}, \quad \chi(t) = t + c.
\end{equation}

Like $g_{V,O}$, the function $\chi \circ g_{V,O}$ has a logarithmic singularity at $O$ and is $C^\infty$ on $\mathring{V} \setminus \{O\}.$ It defines a current (of degree $0$) on $V^+$ with compact support in $\mathring{V}$ and with singular support $\{O\}$. 

\begin{proposition}\label{muchi}
 The following equality of currents holds on $\mathring{V}$:
 \begin{equation}\label{muchi1}
\frac{i}{\pi} \partial \overline{\partial}(\chi \circ g_{V,O}) = - \delta_O + \mu_\chi,
\end{equation}
where
\begin{equation}\label{muchi2}
\mu_\chi := (\chi''\circ g_{V,O}) \frac{i}{\pi} \partial g_{V,O} \wedge \overline{\partial} g_{V,O}
\end{equation}
is a $C^\infty$ positive $(1,1)$-form with compact support in $\mathring{V} \setminus \{O\}$ and satisfies
\begin{equation}\label{muchi3}
\int_{\mathring{V}} \mu_\chi =1.
\end{equation}
\end{proposition}

\begin{proof}
 On the open neighborhood $U:= g_{V,O}^{-1}(]\eta, +\infty])$ of $O$ in $\mathring{V}$, we have:
 $$\chi \circ g_{V,O} = g_{V,O} + C,$$
 and therefore
 $$\frac{i}{\pi} \partial \overline{\partial} (\chi \circ g_{V,O}) =- \delta_O.$$
 Moreover, on $\mathring{V}\setminus \{O\},$ the Green's function $g_{V,O}$ is $C^\infty$ and $\partial \overline{\partial} g_{V,O} =0$, and therefore:
 $$\overline{\partial} (\chi \circ g_{V,O}) = (\chi' \circ g_{V,O} )\,  \overline{\partial} g$$
 and
 \begin{align*}
 \partial \overline{\partial}(\chi \circ g_{V,O}) & = \partial (\chi' \circ g_{V,O}) \wedge \overline{\partial} g_{V,O} + 
 (\chi' \circ g_{V,O})\,  \partial \overline{\partial} g_{V,O} \\
 & = (\chi'' \circ g_{V,O})\,  \partial g_{V,O} \wedge \overline{\partial} g_{V,O}. 
 \end{align*}
 
To complete the proof of the proposition, we are left to establish the equality (\ref{muchi3}). It follows from (\ref{muchi1}) and from the fact that, since $\chi \circ g_{V,O}$ has compact support in $\mathring{V}$, we have:
$$\int_{\mathring{V}} \partial \overline{\partial} (\chi \circ g_{V,O}) =
\int_{\mathring{V}} d\,\overline{\partial}(\chi \circ g_{V,O})= 0$$
by Stokes formula.
\end{proof}

\begin{theorem}\label{PJ} Let $\Lb:= (L, \Vert.\Vert)$ be some Hermitian analytic line bundle $L$ over $\mathring{V}$, and  
let $s$ be a non-zero analytic section of $L$ on $\mathring{V}$ and $n:= v_O(s)$ its vanishing order at the point $O$.  

Let us denote by $\eta_n (s):= s_{\mid O_n}$ the jet of order $n$ of $s$ at $O$ ---  a non-zero element of the complex line $T_O^{\vee \otimes n} \otimes L_0$ --- and by $\Vert.\Vert^{\rm cap}_{M,O,\Lb}$ the Hermitian metric on $T_O^{\vee \otimes n} \otimes L_0$ deduced from the metrics $\Vert.\Vert^{\rm cap}_{M,O}$ on $T_OM$ and $\Vert.\Vert$ on $L_O$ by duality and tensor product.

Then we have:
\begin{multline}\label{eqPJ}
\log \Vert {\eta}_n(s)\Vert^{\rm cap}_{M,O,\Lb} = \int_{\mathring{V}} \log \Vert s \Vert\, \mu_\chi - 
 \int_{\mathring{V}}  (\chi\circ g_{V,O} ) \, \delta_{\div s - nO}
 + n \int_{\mathring{V}} g_{V,O}\,  \mu_\chi + \int_{\mathring{V}} (\chi\circ g_{V,O}) \,  c_1(\Lb).
\end{multline}
\end{theorem}

Let ${\bf{1}}_{\cO(O)}$ denote the tautological section, with divisor $O$, of the analytic line bundle $\cO(O)$ over $V$. Applied to the line bundle $\Lb := (\cO(O), \Vert {\bf{1}}_{\cO(O)} \Vert := e^{-\chi \circ g_{V,O}})$ and its section $s:=  {\bf{1}}_{\cO(O)},$ Theorem \ref{PJ} shows that the penultimate integral in the right-hand side of (\ref{eqPJ}) takes the value:
\begin{equation}\label{crho}
\int_{\mathring{V}} g_{V,O}\,  \mu_\chi  = -c = \int_0^{+\infty} t \rho(t) \, dt.
\end{equation}

The equality (\ref{eqPJ}) may be seen as  an avatar of the classical formulas of Poisson and Jensen. 

To clarify the relation between Theorem \ref{PJ} and these  formulas, let us observe that, using the Green fucntion $g_{V,O}$, we may construct a family of Riemann surfaces with boundary $V_\epsilon$ as follows.

Let us recall that the Green's function $g_{V,O}$ is $C^\infty$ on $V\setminus\{O\}$, positive on $\mathring{V}\setminus\{O\}$, and vanishes on $\partial V$ and that its differential $dg_{V,O}$ does not vanish on $\partial V$. Moreover $g: V\setminus\{O\} \ra [0, +\infty[$ is proper. This implies the existence of $\epsilon_0 >$ such that any element of $[0, \epsilon_0[$ is a regular value of $g$. Therefore, for any $\epsilon \in [0, \epsilon_0[$, $$V_\epsilon := g^{-1}([\epsilon, +\infty]) $$ is a connected Riemann surface with non-empty boundary; moreover $\mathring{V}_\epsilon$ contains $O$, and admits 
$$\partial V_\epsilon := g^{-1}(\epsilon)$$ as boundary. Finally, it is straighforward that
$$g_{V_\epsilon,O} = g_{V,O} - \epsilon \quad\mbox{on $V_\epsilon$}$$
and
\begin{equation}\label{capepsilon}
\Vert .\Vert_{V_\epsilon, O}^{\rm cap} = e^\epsilon \; \Vert .\Vert_{V, O}^{\rm cap}. 
\end{equation}
For any $\epsilon \in ]0, \epsilon_0[$, the ``limit case" of (\ref{eqPJ}) when $\rho$ is replaced by $\delta_\epsilon$ --- and accordingly $\chi(t)$ by $(t-\epsilon)^+$ --- becomes the following formula : 
\begin{equation}\label{PJeps}
\log \Vert {\eta}_n(s)\Vert^{\rm cap}_{V,O,\Lb} = \int_{\partial V_\epsilon} \log \Vert s \Vert\, d^c g_{V,O} - 
 \int_{{V}_\epsilon}  (g_{V,O}-\epsilon) \, \delta_{\div s - nO}
 + n \epsilon + \int_{{V}_\epsilon} (g_{V,O} -\epsilon) \;  c_1(\Lb),
\end{equation}
where $d^c := (i/2\pi) (\partial -\overline{\partial}).$

This relation may also be written as the following
Poisson-Jensen formula on the Riemann surface with boundary $V_\epsilon$:
\begin{equation}\label{PJepsbis}
\log \Vert {\eta}_n(s)\Vert^{\rm cap}_{V_\epsilon,O,\Lb} = \int_{\partial V_\epsilon} \log \Vert s \Vert\, d^c g_{V_\epsilon,O} - 
 \int_{{V}_\epsilon}  g_{V_\epsilon,O} \, \delta_{\div s - nO}
 + \int_{{V}_\epsilon} g_{V_\epsilon,O}  \,  c_1(\Lb).
\end{equation}

We leave the details of the derivation of (\ref{PJeps}) and (\ref{PJepsbis}) to the reader. We simply observe that, conversely, when $\rm{supp}\, \rho \subset \, ]0, \epsilon_0[,$ we recover (\ref{eqPJ}) from (\ref{PJeps}) by mutliplying both sides by $\rho(\epsilon)$ and integrating.

\begin{proof}[Proof of Theorem \ref{PJ}] The current $\log \Vert s \Vert + n g_{V,O}$ is clearly $C^\infty$ on $\mathring{V} \setminus \{O\}$. Actually the logarithmic singularities of $\log \Vert s \Vert$ and $n g_{V,O}$ at $O$ cancel, and $\log \Vert s \Vert + n g_{V,O}$ is  $C^\infty$ on some open neighborhood of $O$. Moreover, by the very definition (\ref{capmet2}) of the capacitary metric $\Vert.\Vert^{\rm cap}_{M,O}$, its value at $O$ is $\log \Vert {\eta}_n(s)\Vert^{\rm cap}_{M,O,\Lb}$. In other words,
\begin{equation}\label{capappar}
 \int_{\mathring{V}} (\log \Vert s \Vert + n g_{V,O})\, \delta_O = \log \Vert {\eta}_n(s)\Vert^{\rm cap}_{M,O,\Lb}.
\end{equation}

The current $\chi \circ g_{V,O}$ has a compact support in $\mathring{V}$ and its singular support and the one of $\log \Vert s \Vert + n g_{V,O}$ are disjoint, so that:
\begin{equation}\label{greenstokes}
\int_{\mathring{V}} (\chi\circ g_{V,O})  \frac{i}{\pi} \partial\overline{\partial}(\log \Vert s \Vert + n g_{V,O}) =
 \int_{\mathring{V}} (\log \Vert s \Vert + n g_{V,O}) \,  \frac{i}{\pi} \partial\overline{\partial}(\chi\circ g_{V,O})
\end{equation}
(Green-Stokes formula). 

Besides, by the defining property (\ref{Greendd}) of the Green's function $g_{V,O}$ and the Poincar\'e-Lelong equation (\ref{PL}), the following equality of currents on $\mathring{V}$ holds: 
\begin{equation}
\frac{i}{\pi} \partial\overline{\partial}(\log \Vert s \Vert + n g_{V,O}) = \delta_{\div s - nO} - c_1(\Lb).
\end{equation}
By means of this relation and of the expression (\ref{muchi1}) for $\frac{i}{\pi} \partial \overline{\partial}(\chi \circ g)$, the Green-Stokes formula (\ref{greenstokes}) becomes:
$$\int_{\mathring{V}} (\chi\circ g_{V,O}). (\delta_{\div s - nO} - c_1(\Lb)) 
= \int_{\mathring{V}} (\log \Vert s \Vert + n g_{V,O}) (-\delta_O + \mu_\chi).$$

Taking (\ref{capappar}) into account, this is precisely the formula (\ref{eqPJ}) to be proved. \end{proof}

From the Poisson-Jensen formula (\ref{eqPJ}), one easily derives the following version of the Schwarz lemma over a compact Riemann surface with boundary.

\begin{corollary}\label{SchV} Let $\Lb$ be as in Theorem \ref{PJ}. 
For any $n \in \N$ and any analytic section $s$ of $L$ over $\mathring{V}$ of vanishing order $v_O(s)$ at $O$ at least $n,$ we have:
\begin{equation}\label{eqSchV}
\Vert \eta_n (s)\Vert^{\rm cap}_{M,O,\Lb} \leq e^{\alpha(\Lb) + n \eta} \Vert s \Vert_{L^2(\mu_\chi, \Lb)},
\end{equation}
where the $L^2$-norm $\Vert s \Vert_{L^2(\mu_\chi, \Lb)}$ is defined by
$$\Vert s \Vert_{L^2(\mu_\chi, \Lb)}^2 := \int_{\mathring{V}} \Vert s \Vert^2 \, \mu_\chi$$
and
$$\alpha(\Lb) := \int_{\mathring{V}} (\chi\circ g_{V,O}) \,  c_1(\Lb).$$
\end{corollary}

\begin{proof}When $v_O(s) > n,$ the jet $j_n s$ vanishes and (\ref{eqPJ}) is clear. When $v_O(s) =n,$ it follows from the Poisson-Jensen formula (\ref{PJ}), together with the following elementary estimates:
\begin{equation}\label{ScV1}
\int_{\mathring{V}} \log \Vert s \Vert \,\mu_\chi \leq \log  \Vert s \Vert_{L^2(\mu_\chi, \Lb)},
\end{equation}
\begin{equation}\label{ScV2}
\int_{\mathring{V}}  (\chi\circ g_{V,O})\,  \delta_{\div s - nO} \geq 0,
\end{equation}
and
\begin{equation}\label{ScV3}
\int_{\mathring{V}} g_{V,O}\,  \mu_\chi  \leq \eta.
\end{equation}

Indeed (\ref{ScV1}) follows from Jensen's inequality applied to the concave function $\log$, since $\int_{\mathring{V}} \mu_\chi  =1.$ The lower bound (\ref{ScV2}) follows from the non-negativity of $\chi \circ g$. Finally, as the support of $\chi''$ is contained in $]0,\eta[,$ the expression (\ref{muchi2}) for $\mu_\chi$ shows that its support is contained in $g_{V,O}^{-1}([0,\eta])$. As $\int_{\mathring{V}} \mu_\chi  =1,$ this implies (\ref{ScV3}). The estimate (\ref{ScV3}) also follows from the relation (\ref{crho}).
 \end{proof}
 
 For later reference, let us state a straightforward consequence of Corollary \ref{SchV}.
 
 \begin{scholium}\label{SchoSchw} Let $V$ be a compact connected Riemann surface with (non-empty) boundary, $O$ a point in $\mathring{V}$ and $\nu$ a volume form on $V$.
 
 1) Let $\Eb:= (E, \Vert.\Vert)$ be some Hermitian analytic vector bundle over $V$. For any $\eta >0,$ there exists $C_\eta$ in $\R_+$ such that the following condition holds: for any $n\in \N$ and any analytic section $s$ of $E$ over $\mathring{V}$ of vanishing order at $O$ at least $n$, the capacitary norm of its jet $\eta_n(s)$ of  order $n$ at $O$\footnote{As observed at the end of \ref{HilFre}, it may be seen as an element $T_O^{\vee \otimes n} \otimes E_O$. The capacitary norm 
 $\Vert . \Vert^{\rm cap}_{M,O,\Eb}$ is the norm deduced from the norm $\Vert.\Vert^{\rm cap}_{M,O}$ on $T_OM$ and $\Vert.\Vert$ on $L_O$ by duality and tensor product.}
satisfies the  upper bound
 \begin{equation}\label{SchSchwE}
\Vert \eta_n (s)\Vert^{\rm cap}_{M,O,\Eb} \leq C_\eta e^{n \eta} \Vert s \Vert_{L^2(\nu, \Eb)},
\end{equation}
where
$$\Vert s \Vert^2_{L^2(\nu, \Eb)} := \int_{\mathring{V}} \Vert s\Vert^2 \, \nu \,(\in [0, +\infty]).$$

2) Let $\Lb:= (L, \Vert.\Vert)$ be some Hermitian analytic vector bundle over $V$. For any $\eta >0,$ there exists $C_\eta$ in $\R_+$ such that the following condition holds: for any $(D, n)\in \N_{>0}\times \N$ and any analytic section $s$ of $L^{\otimes D}$ over $\mathring{V}$ of vanishing order at $O$ at least $n$, the capacitary norm of its jet $\eta_n(s)$ of  order $n$ at $O$ satisfies the  upper-bound
\begin{equation}\label{SchSchwLD}
\Vert \eta_n (s)\Vert^{\rm cap}_{M,O,\Lb^{\otimes D}} \leq C_\eta^{D+1} e^{n \eta} \Vert s \Vert_{L^2(\nu, \Lb^{\otimes D})}.
\end{equation}
\end{scholium}

\begin{proof} Let us first establish 2). 

For any $\eta >0,$ we may choose $\rho$ such that Condition (\ref{rhoeta}) is satisfied. Then Corollary  \ref{SchV}, applied to $\Lb^{\otimes D}$ instead of $\Lb$, shows that, with eyh notation of 2), we have:
$$ 
\Vert \eta_n (s)\Vert^{\rm cap}_{M,O,\Lb^{\otimes D}} \leq C_\eta^{D} e^{n \eta} \Vert s \Vert_{L^2(\mu_\chi, \Lb^{\otimes D})}, 
 $$
 where 
 $$C_\eta := e^{\alpha(\cLb)}.$$
 (Indeed, $c_1(\Lb^{\otimes D}) =D c_1(\Lb),$ and therefore $\alpha(\Lb^{\otimes D}) =D \alpha(\Lb).$)
 
 After possibly increasing $C_\eta,$ we may also assume that
 $\mu_\chi \leq C_\eta^2 \nu.$
 Then 
 $$\Vert s \Vert_{L^2(\mu_\chi, \Lb^{\otimes D})} \leq C_\eta \Vert s \Vert_{L^2(\nu, \Lb^{\otimes D})},$$ 
 and (\ref{SchSchwLD}) follows.
 
 Assertion 1) when $\rk E =1$ follows from 2) with $D=1$. One reduces to this situation, thanks to the following observations:
\begin{enumerate}
\item 
the validity of 1) for some Hermitian analytic vector bundles $\Eb_1,\ldots,\Eb_N$ over $V$ is equivalent to its validity for the direct sum
 $\Eb_1\oplus\cdots \oplus\Eb_N$;
 \smallskip
 \item  for any analytic vector bundle $E$ over $V$ and any two $C^\infty$ Hermitian metrics $\Vert.\Vert_1$ and $\Vert.\Vert_2$ on $E$, the validity of 1) for $(E,\Vert.\Vert_1)$ and  $(E,\Vert.\Vert_2)$ are equivalent;
 \smallskip
 \item any analytic vector bundle $E$ over $V$ may be trivialized, and is therefore isomorphic to a direct sum $L_1 \oplus\cdots \oplus L_{\rk E}$ of analytic line bundles over $V$. 
\end{enumerate}
  (Assertions (1) and (2) are straightforward; (3) is  a consequence of the Theorem of Behnke-Stein; see for instance \cite{GuenotNarasimhan76}, Chapitre V, Th\'eor\`eme 3.) 
\end{proof}

  \section{Smooth formal-analytic surfaces over $\Spec \OK$}\label{FormAnOK}
  
  \subsection{Basic definitions}\label{SmForSurf}
  
  We shall define a \emph{smooth formal-analytic surface} over $\Spec \OK$ as a pair
  $$\tilde{\cV} := (\cVh, (V_\sigma, P_\sigma, i_\sigma)_\sKC)$$
  where:
  \begin{itemize}
\item $\tilde{\cV}$ is a pointed smooth formal curve over $\Spec \OK$; we shall denote by 
$\pi: \tilde{\cV} \lra \Spec \OK$
its structural morphism, and by $\cP$ the canonical section\footnote{If ${\cVh}_{\rm red}$ denotes $\vert {\cVh} \vert$ equipped with its structure of reduced scheme  (or equivalently the subscheme of ${\cVh}$ defined by its largest Ideal of definition), the map $\pi$ defines an isomorphism $\pi_{\mid {\cVh}_{\rm red}}:  {\cVh}_{\rm red} \lrasim \Spec \OK$, and $\cP$ coincides with the inverse of $\pi_{{\cVh}_{\rm red}}$.} of $\pi$; 
\smallskip
\item for every field embedding $\sKC$, $V_\sigma$ is a compact connected Riemann surface with non-empty boundary,  $P_\sigma$ is a point in its interior $\mathring{V}_\sigma$, and $i_\sigma$ is an isomorphism
$$i_\sigma : \cVh_\sigma \lrasim \widehat{V_\sigma}_{,P_\sigma}$$ between the smooth formal germs of  complex curves defined by the base change 
$$ \cVh_\sigma  := \cVh \otimes_{\OK, \sigma}\C$$
of $\cVh$ from $\OK$ to $\C$ through the embedding $\sigma: \OK \hra \C$ and by the formal completion $\widehat{V_\sigma}_{,P_\sigma}$ of $V_\sigma$ at the point $P_\sigma$.
\end{itemize}

These data are moreover assumed to be compatible with complex conjugation. Namely, we are given a family $(j_\sigma)_\sKC$ of  \emph{antiholomorphic} isomorphisms of Riemann surfaces with boundary
$$j_\sigma: V_\sigma \lrasim V_{\overline{\sigma}}$$
such that 
$$j_{\overline{\sigma}} = j_\sigma^{-1},$$
$$j_\sigma(P_\sigma)= P_{\overline{\sigma}},$$
and the following diagram is commutative:
\begin{equation}\label{compV}\begin{CD}
 \cVh_\sigma  := \cVh \otimes_{\OK, \sigma}\C @>i_\sigma>> \widehat{V_\sigma}_{,P_\sigma} \\
@V{id \otimes \overline{.}}VV      @VV{\widehat{j_\sigma}_{,P_\sigma}}V \\
\ \cVh_{\overline{\sigma}}  := \cVh \otimes_{\OK, \overline{\sigma}}\C @>{i_{\overline{\sigma}}}>> \widehat{V_{\overline{\sigma}}}_{,P_{\overline{\sigma}}}.
\end{CD}
\end{equation} 

It is convenient to see the isomorphisms $j_\sigma$ as identifications $V_{\overline{\sigma}} \simeq V_\sigma^{c.c.}$ of the Riemann surfaces $V_{\overline{\sigma}}$ with the complex conjugates of the Riemann surfaces $V_\sigma,$ and to think of them as defining some real structure on $$V_\C := \coprod_\sKC V_\sigma.$$

We shall define a \emph{vector bundle} $\tilde{\cE}$ over the smooth formal-analytic surface $\tilde{\cV}$ as a pair 
$$\tilde{\cE}:= (\cEh, (\cE_\sigma, \phi_\sigma)_\sKC)$$
where:
\begin{itemize}
\item $\cEh$ is a vector bundle (that is, a locally free coherent sheaf)   over $\cVh$; 
\smallskip
\item for every field embedding $\sKC$, $\cE_\sigma$  is an analytic vector bundle over $V_\sigma$ and  $\phi_\sigma$ is an isomorphism of formal vector bundles over $\cVh_\sigma$:
$$\phi_\sigma : \cEh_\sigma := \cEh \hat{\otimes}_{\OK, \sigma} \C \lrasim i_\sigma^\ast {\cE_\sigma}_{\mid \widehat{V_\sigma}_{,P_\sigma}}.$$
\end{itemize}
(Recall that we are given an isomorphism $\cVh_\sigma\stackrel{i_\sigma}{\lrasim} \widehat{V_\sigma}_{,P_\sigma}$.)

The family $(\cE_\sigma, \phi_\sigma)_\sKC$ is moreover assumed to be compatible with complex conjugation. Namely, we are given a family $(J_\sigma)_\sKC$ of isomorphisms 
$$J_\sigma : \cE_\sigma \lrasim \cE_{\overline{\sigma}}^{c.c.}$$
over $V_\sigma \simeq V_{\overline{\sigma}}^{c.c.}$ that satisfy the relations
$$J_{\overline{\sigma}} = J_\sigma^{c.c. \, -1}$$
and make the diagram
\begin{equation}\label{compE}\begin{CD}
 \cEh_\sigma  := \cEh \otimes_{\OK, \sigma}\C @>\phi_\sigma>> i_\sigma^\ast\widehat{V_\sigma}_{,P_\sigma} \\
@V{id \otimes \overline{.}}VV      @VV{\widehat{J_\sigma}}V \\
\ \cEh_{\overline{\sigma}}  := \cEh \otimes_{\OK, \overline{\sigma}}\C @>{\phi_{\overline{\sigma}}}>> i^\ast_{\overline{\sigma}}\widehat{V_{\overline{\sigma}}}_{,P_{\overline{\sigma}}}.
\end{CD}
\end{equation} 
commutative.

The family $(J_\sigma)$ may be understood as defining some ``real structure" on the $\C$-analytic vector bundle over $V_\C$, the restriction of which to $V_\sigma$ is $\cE_\sigma$ for any field embedding $\sKC$.  In practice, the definition of the isomorphisms $J_\sigma$ is often clear from the context and therefore not explicitly given.

Finally, a \emph{Hermitian vector bundle} $\cEbt$ over $\tilde{\cV}$ is defined as a pair 
$$\tilde{\cEb}:= (\cEh, (\cE_\sigma, \phi_\sigma, \Vert.\Vert_\sigma)_\sKC)$$
where:
\begin{itemize}
\item  $\tilde{\cE}:= (\cEh, (\cE_\sigma, \phi_\sigma)_\sKC)$ is a vector bundle over $\cVt$, as above; 
\smallskip
\item $(\Vert. \Vert_\sigma)_\sKC$ is a family of $C^\infty$ Hermitian metrics on the $C^\infty$ vector bundles $\cE_{\sigma \mid V_\sigma}$ on the $C^\infty$ manifolds with boundary $V_\sigma$, that is invariant under complex conjugation\footnote{Namely, the isomorphisms $J_\sigma$ are isometries when $\cE_{\sigma\mid V_\sigma}$ and $\cE_{\overline{\sigma}_{\mid V_{\overline{\sigma}}}}$ (which coincides with $\cE^{c.c.}_{\overline{\sigma}}$ as $C^\infty$ vector bundle) are equipped with the $C^\infty$ metrics $\Vert.\Vert_\sigma$ and $\Vert.\Vert_{\overline{\sigma}}$.}.
\end{itemize}

Then, for every $\sKC,$ the pair $(\cE_\sigma, \Vert.\Vert_\sigma)$ is a Hermitian analytic vector bundle over $V_\sigma$ in the sense of \ref{DefCRB}.

 \subsection{The pro-Hermitian vector bundle $\Gamma_{L^2}(\tilde{\cV}, \nu; \cEbh)$}\label{GammaL2}
  
  Let us consider a smooth formal-analytic surface over $\Spec \OK$
   $$\tilde{\cV} := (\cVh, (V_\sigma, P_\sigma, i_\sigma)_\sKC)$$
   and some Hermitian vector bundle over $\cVt$,
   $$\tilde{\cE}:= (\cEh, (\cE_\sigma, \phi_\sigma)_\sKC).$$
  
  Let us moreover assume that, for every embedding $\sKC,$ the Riemann surface $V_\sigma$ is connected and its boundary $\partial V_\sigma$ is non-empty.
  
  Finally, let $\nu$ be a $C^\infty$ volume form on $V_\C$, invariant under complex conjugation. 
  For every embedding $\sKC,$ we shall denote its restriction to $V_\sigma$ by
  $$\nu_\sigma:= \nu_{\vert V_\sigma}.$$
  (The family $(\nu_\sigma)_\sKC$ is invariant under complex conjugation; namely, for every field embedding $\sKC,$ we have
  \begin{equation}\label{jnu}
j_\sigma^\ast \nu_{\overline{\sigma}} = \nu_\sigma.
\end{equation}
Conversely, any family $(\nu_\sigma)_\sKC$ satisfying the conditions (\ref{jnu}) defines a volume form $\nu$ invariant under conjugation on $V_\C$.)

  We may consider the $\OK$-module $\Gamma (\cVh, \cEh)$
  of global sections of the locally free coherent sheaf $\cEh$ on the formal scheme $\cVh$ over $\OK$. By construction, it is endowed with a structure of topological module over $\OK$ (endowed with the discrete topology).
  
  Indeed, by the very definition of $\cVh$ as a pointed smooth formal curve over $\Spec \OK,$ we have:
  $$\cVh = \Spf B$$
  where $B$ is some noetherian adic topological $\OK$-algebra. Actually, as recalled in \ref{ExRemSFC} {\bf (ii)}, if we denote
  $$N:= \Gamma (\Spec \OK, N_\cP \cVh),$$
  there exists an isomorphism of  topological algebras 
  \begin{equation}\label{Bisom}
  B \lrasim \widehat{\rm Sym}_A N^\vee \simeq  \prod_{n \in\N} N^{\vee \otimes_{A} n},
  \end{equation}
where $\widehat{\rm Sym}_A N^\vee$ is equipped with the $I$-adic topology associated to the ideal $I:= N^\vee. B$. Moreover, the vector bundle $\cE$ over $\cVh$ may be written (canonically) as the coherent sheaf $M^\Delta$ over $\Spf B$ associated to some finitely generated projective $B$-module $M$ equipped with the $I$-adic topology. The $\OK$-module of sections $\Gamma (\cVh, \cEh)$ may be identified to $M$, and as such, is a topological $\OK$-module.
   
   Observe that, as a topological $\OK$-module, $M$ is a direct summand of the topological $\OK$-module $B^{\oplus N}$ for some $N \in \N.$ The isomorphism (\ref{Bisom}) shows that $B$ is an object of $CTC_\OK$ (see (\ref{CTCbasics}). Consequently, the topological $\OK$-module $B^{\oplus N}$ and all its closed submodules also are objects of $CTC_\OK$ (\cf Proposition (\ref{subCTC}). This establishes:

\begin{lemma}\label{GammaCTC} The topological $\OK$-module $\Gamma(\cVt, \cEt)$ is an object of $CTC_\OK$. \qed
 \end{lemma}
 
 For any field embedding $\sKC,$ we may form the completed tensor product
 $$\Gamma(\cVt, \cEt)_\sigma := \Gamma(\cVt, \cEt) \hat{\otimes}_{K,\sigma} \C.$$
 It may be identified with the space $\Gamma (\cVh_\sigma, \cEh_\sigma)$
of sections of the vector bundle $\cEh_\sigma$ over the smooth formal germ of complex curve $\cVh_\sigma$ deduced from $\cEh$ and $\cVh$ by the base change $\sigma.$

Moreover the isomorphisms
$$i_\sigma :\cVh_\sigma \lrasim \Vh_{\sigma,O_\sigma} \quad\mbox{ and }\quad \phi_\sigma: \cEh_\sigma \lrasim
 i_\sigma^\ast {\cE_\sigma}_{\mid \widehat{V_\sigma}_{,O_\sigma}}$$
 provide an identification
 \begin{equation}\label{GammaGamma}
\Gamma (\cVh_\sigma, \cEh_\sigma) \simeq \Gamma(\widehat{V_\sigma}_{,P_\sigma}, \cE_\sigma)
\end{equation}
compatible with the canonical topology of Fr\'echet space on these two complex vector spaces (\cf  \ref{topologicalcomplexpro} and \ref{HilFre} \emph{supra}).

Finally, we may consider the Hilbert space of analytic sections $\Gamma_{L^2}(V_\sigma, \nu_\sigma; \cE_\sigma, \Vert.\Vert_\sigma)$ of the Hermitian analytic vector bundle $(\cE_\sigma, \Vert.\Vert_\sigma)$ on the Riemann surface $\mathring{V}_\sigma$ equipped with the volume form $\nu_\sigma$, and the evaluation maps:
$$\begin{array}{crcl}
\hat{\eta}_\sigma: & \Gamma_{L^2} (V_\sigma, \nu_\sigma; \cE_\sigma, \Vert. \Vert_\sigma)   &  \lra  & \Gamma(\widehat{V_\sigma}_{,P_\sigma}, \cE_\sigma)  \\
 & s & \longmapsto  & s_{\mid \widehat{V_\sigma}_{,P_\sigma}}.
 \end{array}
$$

According to Proposition \ref{propeta}, these maps  
are injective, and continuous with dense image.
Moreover the construction of the maps  $\hat{\eta}_\sigma$ is clearly compatible with the complex conjugation isomorphisms
$j_\sigma: V_\sigma \lrasim V_{\overline{\sigma}}^{c.c.}$.

Together with Lemma \ref{GammaCTC}, this proves:

\begin{proposition}\label{defGammatilde} The pair
$$\Gamma_{L^2}(\tilde{\cV}, \nu; \cEbh):= (\Gamma(\cVt, \cEt), ( \Gamma_{L^2} (V_\sigma, \nu_\sigma; \cE_\sigma, \Vert. \Vert_\sigma), \hat{\eta}_\sigma)_\sKC)$$
defines a pro-Hermitian vector bundle over $\Spec \OK.$ \qed
\end{proposition}
  
 As every compact surface with boundary $V_\sigma$ is compact, any two volume forms (resp., any two Hermitian metric over $\cE_\sigma$) over $V_\sigma$ are comparable. This implies that the \emph{Hilbertisable} pro-vector bundle over $\Spec \OK$ defined by $\Gamma_{L^2}(\tilde{\cV}, \nu; \cEbh)$ is  independent of the choice of the volume form $\nu$ and of the Hermitian metrics $(\Vert . \Vert_\sigma)_{\sKC}$ on the complex vector bundles $(\cE_\sigma)_\sKC$. We will denote it by 
  $\Gamma_{L^2}(\tilde{\cV}; \tilde{\cE}).$ 
  
  \smallskip
  
  With the above notation, the ideal $I$ is the largest ideal of definition of the adic algebra $B$, since $B/I \simeq \OK$ is reduced. The corresponding coherent Ideal $\cI := I^\Delta$ of $\cO_{\Vh}$ is its largest Ideal of definition, and the corresponding reduced subscheme $\cVh_{\rm red}$ is the image of the closed immersion $\cP$. We shall often denote it by $\cP.$
  
  If $\cP_i$ denotes the $i$-th infinitesimal neighborhood of $\cP$ in $\cVh$ --- namely the $\OK$-subscheme of $\cVh$ defined by the Ideal $\cI^{i+1}$ --- we may describe $\cVh$ as the inductive limit (in the category of formal schemes)
  \begin{equation}\label{IndV}
\cVh = \varinjlim_i \cP_i
\end{equation}
of the inductive system of schemes finite and flat over $\Spec \OK$:
$$
\cP_0 := \cP \hlra \cP_1 \hlra \ldots \hlra \cP_i \hlra \cP_{i+1} \hlra \ldots \quad .$$

Actualy, when we identify $B$ with ${\widehat{\rm Sym}}_{\OK} N^\vee$, the $\OK$-schemes $\cP_i$ becomes naturally isomorphic to 
$$\Spec ({\widehat{\rm Sym}}_{\OK}) N^\vee/ I^{i+1} = \Spec ({\rm Sym}_{\OK} N^\vee)/ I^{i+1},$$
and we have an isomorphism of $\OK$-modules:
$$({\rm Sym}_{\OK} N^\vee)/ I^{i+1} \simeq \bigoplus_{k=0}^i  N^{\vee \otimes k}.$$

The pro-Hermitian vector bundle $\Gamma_{L^2}(\tilde{\cV}, \nu; \cEbh)$ over $\Spec \OK$ admits a natural description as a projective of (finite rank) of some projective system $\Eh_\bullet$ of Hermitian vector bundles that reflects the description of $\cVh$ as the inductive limit (\ref{IndV}).

Indeed, for every $i \in \N,$ we have an exact sequence of $\OK$-modules
$$0 \lra U_k := \Gamma(\cVh, \cI^{i+1}) \hlra \Gamma(\cVh, \cEh) \stackrel{\eta_i}{\lra} \Gamma(\cV_i, \cEh_{\mid \cV_i}) \lra 0.$$
Each module $U_i$ is open in $\Gamma(\cVh, \cEh)$, and the sequence $(U_i)_{i \in \N}$is non-decreasing and constitutes a fundamental system of neighborhood of zero in $\Gamma(\cVh, \cEh)$. 
Therefore we may apply the general construction of a pro-Hermitian vector bundle as a projective limit of Hermitian vector bundles discussed in paragraph \ref{Consprojective}.

We obtain that 
$\Ebh :=\Gamma_{L^2}(\tilde{\cV}, \nu; \cEbh)$r
may be identified with the projective limit $\varprojlim_i \Eb_i$, where the projective system
$$\Eb_{\bullet} : \Eb_0 \stackrel{q_0}{\longleftarrow}\Eb_1 \stackrel{q_1}{\longleftarrow}\dots \stackrel{q_{i-1}}{\longleftarrow}\Eb_i \stackrel{q_i}{\longleftarrow} \Eb_{i+1} \stackrel{q_{i+1}}{\longleftarrow} \dots$$
is defined by the  Hermitian vector bundles
$\Eb_i := \Eb_{U_i}$ over $\Spec \OK$  and  the surjective admissible quotient morphisms $q_i:=\Eb_{U_{i+1}}\lra \Eb_{U_i}$.
  
In concrete terms, we have canonical isomorphisms
$$E_i := \Eh/U_i = \Gamma((\cVh, \cEh) / \Gamma(\cVh, \cI^{i+1}) \lrasim \Gamma(\cV_i, \cEh_{\mid \cV_i})$$
and, for every embedding $\sKC,$
$$E_{i,\sigma} \lrasim  \Gamma(\cV_i, \cEh_{\mid \cV_i})_\sigma \lrasim \Gamma(\cV_{i,\sigma}, \cEh_{\mid \cV_i, \sigma}) \lrasim \cE_{\sigma\mid P_{\sigma i}}.$$
(We denote by $P_{\sigma i}$ the $i$-th infinitesimal neigborhood of $P_\sigma$ in $V_\sigma$. The last isomorphism is induced by the isomrophisms $i_\sigma$ and $\phi_\sigma$.)

Accordingly, $\Eb_i$ may be identified with
$$(\Gamma(\cV_i, \cEh_{\mid \cV_i}), (\Vert.\Vert_{i,\sigma})_\sKC),$$
where $\Vert.\Vert_{i,\sigma}$ is the Hermitian norm on $\Gamma(V_{\sigma, P_\sigma}, \cE_\sigma)$ which makes the jet map
$$\eta_{i,\sigma}: \Gamma_{L^2}(\mathring{V}_\sigma, \nu_\sigma; \cE_\sigma, \Vert.\Vert_\sigma) \lra  \cE_{\sigma\mid P_{\sigma i}}$$
a co-isometry\footnote{In other words, $\Vert.\Vert_{i,\sigma}$ is the quotient norm of $\Vert.\Vert_{L^2(\nu_\sigma; \cE_\sigma, \Vert.\Vert_\sigma)}$ \emph{via} the continuous surjective map $\eta_{i,\sigma}.$} when $\Gamma_{L^2}(\mathring{V}_\sigma, \nu_\sigma; \cE_\sigma, \Vert.\Vert_\sigma)$ is equipped with its $L^2$-norm $\Vert.\Vert_{L^2(\nu_\sigma; \cE_\sigma, \Vert.\Vert_\sigma)}$.

\subsection{Morphisms from a smooth formal-analytic surface to some $\OK$-scheme}\label{MorForAnScheme}
  Let $\cX$ be a separated scheme of finite type over $\Spec \OK$. For any $\tilde{\cV} := (\cVh, (V_\sigma, P_\sigma, i_\sigma)_\sKC)$ as above, we define a \emph{$\OK$-morphism from $\tilde{\cV}$ to $\cX$} 
  $${f} : \tilde{\cV} \lra \cX$$ as a pair
  $${f} := (\hat{f}, (f_\sigma)_{\sKC}),$$
  where $$\hat{f}: \cVh \lra \cX$$
  is a morphism of formal schemes overs $\OK$, and where, for every embedding $\sKC$, 
  $$f_\sigma: V^+_\sigma \lra \cX_\sigma(\C)$$
  is a $\C$-analytic morphism, or equivalently, a map of $\C$-locally ringed  spaces from  $(V_\sigma, \cO^\an_{\mid V_\sigma})$ to the $\C$-scheme $\cX_\sigma$. 
  
  These maps are moreover assumed to satisfy the following compatibility conditions: for every embedding $\sKC$, the diagram of ringed spaces
\begin{equation}\label{compMorVX}\begin{CD}
\quad\quad\cVh_\sigma\;\;\;\;\;\;\; @>{\hat{f}_\sigma}>> \cX_\sigma \\
@V{i_\sigma}V{\simeq\;\;\;\;\;\;\;\;\;}V      @VV{=}V \\
\quad\quad\quad\widehat{V_\sigma}_{, P_\sigma} \stackrel{\kappa_\sigma}{\hlra} V_\sigma @>f_\sigma>> \cX_\sigma 
\end{CD}
\end{equation} 
is commutative. This implies that the family $(f_\sigma)_\sKC$ is compatible with complex conjugation.

Observe that, for any vector bundle $\cE$ over $\cX,$ we may define the \emph{inverse image $\cE$ by $f$} as the following vector bundle over $\cVt$:
$$f^\ast \cE := (\hat{f}^\ast \cE, (f_\sigma, \cE_{\cX_\sigma})_\sKC),$$
where the isomorphism $\phi_\sigma$ is deduced form the base change isomorphism\footnote{We denote by $\hat{f}_\sigma: \cVh_\sigma \lra \cX_\sigma$ the morphism deduced from $\hat{f}: \cVh \lra \cX$ by the base change $\sigma: \OK \hlra \C.$}
$$\hat{f}^\ast_\sigma \cE \otimes_{\OK,\sigma} \C \lrasim \widehat{f_\sigma}^\ast \cE_\sigma$$
and the equality of morphisms
$$\hat{f}_\sigma = f_\sigma \circ \kappa_\sigma \circ i_\sigma : \widehat{V_\sigma}_{, P_\sigma}  \lra \cX_\sigma.$$
Namely, it is defined as the composition:
$$\phi_\sigma : (\hat{f}^\ast \cE)_\sigma = \hat{f}^\ast_\sigma \cE \otimes_{\OK,\sigma} \C \simeq \widehat{f_\sigma}^\ast \cE_\sigma 
\simeq i_\sigma^\ast \kappa_\sigma^\ast f_\sigma^\ast \cE_\sigma = i_\sigma^\ast (f_\sigma^\ast \cE_\sigma)_{\mid  \widehat{V_\sigma}_{, P_\sigma} }.$$

Observe also that, if $\cX_K$ is reduced and if $\cEb := (\cE, \Vert.\Vert)$ is some Hermitian vector bundle over $\cX$, defined by some vector bundle $\cE$ on $\cX$ and some $C^\infty$ metric $\Vert.\Vert$ on $\cE^{\rm an}_\C$ over $\cX(\C)= \coprod_\sKC \cX_\sigma(\C)$ (invariant under complex conjugation), then, for every embedding $\sKC,$ we may equip the analytic vector bundle $f_\sigma^\ast \cE_\sigma$ over $V_\sigma$ with the pull-back of the Hermitian metric $\Vert.\Vert$ on $\cX_\sigma(\C)$ by $f_\sigma$. The vector bundle $f^\ast \cE$, endowed with these Hermitian metrics, defines some Hermitian vector bundle $f^\ast \cEb$ over $\cVt$, the \emph{inverse image of $\cEb$ by} $f$.

\subsection{Schematic image and algebraicity}\label{AlgAr}

In paragraph \ref{algformalpseudoconcave}, we considered the Zariski closure $\overline{\im f}$ of the image of 	 a morphism $f: \cVh \lra X$ from a formal surface $\cVh$ over some field $k$ to some quasi-projective $k$-scheme $X$. The dimension bound
$$\dim \overline{\im f} \leq 2$$
provided a formal definition of the algebraicity of the image of $f$.

It is possible to formulate a similar algebraicity property concerning a morphism
$$f: \cVt \lra \cX$$
from a formal-analytic surface $\cVt$ to some $\OK$-scheme as above.

Actually, with the notation of the previous subsection \ref{MorForAnScheme}, to each of the morphisms of locally ringed spaces
$\hat{f}: \cVh \lra \cX,$
$\hat{f}_K: \cVh_K \lra \cX_K,$
$\hat{f}_\sigma: \cVh_\sigma  
\lra \cX_\sigma,$
and
$f_\sigma: V_\sigma \lra \cX_\sigma,$
we may attach the Zariski closure of its image $\overline{\im\hat{f}}$, 
$\overline{\im\hat{f}_K}$, $\overline{\im\hat{f}_\sigma}$,  and $\overline{\im f_\sigma}$, namely the smallest closed subscheme of the range of the morphism such that it factorizes through this closed subscheme.

The formation of the Zariski closure of the image of these morphisms  is compatible with the replacement of their range scheme by some open subscheme through which the morphism factorizes (as in  \ref{algformalpseudoconcave}, Property (i) and (\ref{restU})). Moreover, each of these Zariski closures is an integral scheme. For instance, we have:

\begin{lemma}\label{CompShrink}
If $U$ is some open subscheme of $\cX$ such that $\hat{f}$ factorizes through the inclusion $U \hra \cX$ --- \emph{in other words, if the image of the continuous map of topological spaces from $\vert \cVh \vert = \cP$ to $\vert \cX \vert$ defined by $\hat{f}$ is contained in $\vert U \vert$} --- and if $\hat{f}^U$ denotes $\hat{f}$ seen as a morphism from $\cVh$ to $U$, then we have:  
$$\overline{\im \hat{f}^U} = \overline{\im \hat{f}} \cap U.$$ \qed
\end{lemma}

Actually, $\overline{\im\hat{f}}$ contains the $\OK$-point $\hat{f}\circ \cP$ of $\cX$ and $\overline{\im\hat{f}_\sigma}$ the associated $K$-rational point of $\cX_K$. This implies that the integral scheme $\overline{\im\hat{f}}$ is  flat over $\Spec \OK$ and that $\overline{\im\hat{f}_K}$ is geometrically integral over $K$.

\begin{proposition}\label{propclosureimage}
We the above notation, we have:
\begin{equation}\label{closureU1}
\overline{\im\hat{f}_K} = \left(\overline{\im\hat{f}}\right)_K,
\end{equation} 
\begin{equation}\label{closureU2}
\overline{\im\hat{f}_\sigma} = \left(\overline{\im\hat{f_K}}\right)_\sigma,
\end{equation} 
and
\begin{equation}\label{closureU3}
\overline{\im{f}_\sigma} = \overline{\im\hat{f}_\sigma}.
\end{equation} 
\end{proposition}
        
\begin{proof} By the compatibility of the formation of the Zariski closure of the image of a morphism with the ``shrinking" of the range scheme alluded to above, we may assume that $\cX$ is a closed subscheme of $\PP_\OK^N$. 

Then the graded ideal in 
$$\OK[X_0,\cdots,X_N] \simeq \bigoplus_{D \in \N} \Gamma(\PP^N_\OK, \cO(D))$$
that defines $\overline{\im \hat{f}}$ is the direct sum 
$\bigoplus_{D \in \N} \ker \hat{\eta}_D$ where $\hat{\eta}_D$ denotes the evaluation map:
$$
\begin{array}{crcl}
\hat{\eta}_{D} : & \Gamma(\PP^N_\OK, \cO(D))  & \lra    &   \Gamma(\cVh, \hat{f}^\ast \cO(D))\\
& s & \longmapsto    & \hat{f}^\ast s.  \end{array}
$$
Similarly, $\overline{\im \hat{f}_K}$ is defined by the graded ideal $\bigoplus_{D \in \N} \ker \hat{\eta}_{D,K}$ in $K[X_0,\ldots,X_N]$ defined by the evaluation maps
$$
\begin{array}{crcl}
\hat{\eta}_{D,K} : & \Gamma(\PP^N_K, \cO(D))  & \lra    &   \Gamma(\cVh_K, \hat{f}^\ast \cO(D))\\
& s & \longmapsto    & \hat{f}_K^\ast s.  \end{array}
$$
Since $\hat{\eta}_{D,K}$ is deduced from $\hat{\eta}_{D}$ by the base change $\Spec K \hra \Spec \OK$ and $\cVh$ is flat over $\Spec \OK,$  we have:
$$\ker \hat{\eta}_{D} = \hat{\eta}_{D,K} \cap \Gamma(\PP^N_\OK, \cO(D)) \quad \mbox{ for every $D \in \N$.}$$
This establishes (\ref{closureU1}).

The Zariski closure $\overline{\im\hat{f}_\sigma}$ is defined by the graded ideal $\bigoplus_{D \in \N} \ker \hat{\eta}_{D,\sigma}$ in $\C[X_0,\ldots,X_N]$ defined by the evaluation maps
$$
\begin{array}{crcl}
\hat{\eta}_{D,\sigma} : & \Gamma(\PP^N_\C, \cO(D))  & \lra    &   \Gamma(\cVh_\sigma, \hat{f}_\sigma^\ast \cO(D))\\
& s & \longmapsto    & \hat{f}_\sigma^\ast s.  \end{array}
$$
The morphism $\hat{\eta}_{D,\sigma}$ may be identified with the morphism deduced from $\hat{\eta}_{D,K}$ by the field extension $\sigma: K \hlra \C$, and  (\ref{closureU2}) follows.

Finally $\overline{\im{f}_\sigma} $ is defined by the graded ideal $\bigoplus_{D \in \N} \ker \hat{\eta}^{\rm an}_{D,\sigma}$ in $\C[X_0,\ldots,X_N]$ defined by the evaluation maps
$$
\begin{array}{crcl}
\hat{\eta}_{D,\sigma}^{\rm an} : & \Gamma(\PP^N_\C, \cO(D))  & \lra    &   \Gamma(V_\sigma, {f}_\sigma^\ast \cO(D))\\
& s & \longmapsto    & {f}_\sigma^\ast s.  \end{array}
$$

As the Riemann surface  $V_\sigma$ is connected, by analytic continuation, we have:
$$ \ker \hat{\eta}_{D,\sigma}^{\rm an} = \ker \hat{\eta}_{D,\sigma} \quad \mbox{ for every $D \in \N$.}$$
This establishes (\ref{closureU3}).
\end{proof}

\begin{corollary} The following four conditions are equivalent:
\begin{equation}\label{algAr1}
\dim \overline{\im\hat{f}} \leq 2,
\end{equation}
\begin{equation}\label{algAr2}
\dim \overline{\im\hat{f}_K} \leq 1,
\end{equation}
\begin{equation}\label{algAr3}
\dim \overline{\im\hat{f}_\sigma} \leq 1,
\end{equation}
and
\begin{equation}\label{algAr4}
\dim \overline{\im f_\sigma} \leq 1.
\end{equation}
\end{corollary}
\qed 

When Conditions (\ref{algAr1})-(\ref{algAr4}) are satisfied, we shall say that \emph{the image of $f$ is algebraic.} 

Actually, when this holds, either the morphism $\hat{f}: \cVh \lra \cX$ is constant (that is, factorizes through the structural  morphism $\pi: \cVh \lra \Spec \OK$) and the morphisms $f_\sigma$ also for every embedding $\sKC$, or the equality holds in the relations (\ref{algAr1})-(\ref{algAr4}).     

When the latter holds, we may consider the normalization $\nu: C \lra \overline{\im\hat{f}_K}$ of the projective curve $\overline{\im\hat{f}_K}$. Then, by the compatibility of normalization and completion, the morphism $\hat{f}_K$  factorizes as
$\hat{f}_K = \nu \circ \hat{g}$ 
for some uniquely determined $K$-morphism
$\hat{g}: \cVh_K \lra C.$
The image $Q:=\hat{g}(\cP_K)$ of the $K$-point $\cP_K =(\cVh_K)_{\rm red}$ defines a $K$-rational point of the 
smooth integral curve $C$, and $\hat{g}$ defines a finite morphism of smooth pointed curves over $K$:
   $\hat{g}: \cVh_K \lra \widehat{C}_{Q}.$
   
   In brief, $\hat{f}_K$ may be written as the composition of the (ramified) morphism of smooth formal germs of curves over $K$,
    $$\hat{g}: \cVh_K \lra \widehat{C}_{Q},$$
    and of the finite morphism of $K$-schemes
    $$\nu: C \lra \cX_K.$$
   On may actually show that, for every embedding $\sKC,$ the morphism of complex formal germs
   $$\hat{g}_\sigma: \cVh_\sigma \lra \widehat{C}_{Q, \sigma}$$ 
   ``extends" to some analytic map
   $$g_\sigma: V_\sigma^+ \lra \cX_\sigma(\C)$$  
   and that
   $$f_\sigma = \nu_\sigma \circ g_\sigma.$$
   We leave this to the interested reader.
     
  \section{Arithmetic pseudo-concavity and finiteness}\label{ArPsConcaveFin}
 
\subsection{}Let $\cVt := (\cVh, (V_\sigma, P_\sigma, i_\sigma)_\sKC)$ be some smooth formal-analytic surface over $\Spec \OK$. As before, we shall denote by $\pi$ the structural morphism of $\cVh$, and by $\cP$ the canonical section of $\pi.$

  The normal bundle $N_{\cP}{\cVh}$ of $\cP$ in the smooth formal surface $\cVh$ is a line bundle over $\Spec \OK$. For any field embedding $\sKC$, the complex line $N_{\cP}{\cVh}$ may be identified with the tangent line $T_{O_\sigma}$ at the point $O_\sigma$ of the Riemann surface $V_\sigma,$ and may therefore be equipped with the capacitary metric $\Vert. \Vert^{\rm cap}_{V_\sigma, P_\sigma}.$ 
  
  This construction defines some Hermitian line bundle over $\Spec \OK$:
  $$\overline{N}_{\cP}\tilde{\cV} := (N_{\cP}{\cVh}, (\Vert. \Vert^{\rm cap}_{V_\sigma, P_\sigma})_\sKC).$$
  
\begin{theorem}\label{finitewhenNArample}
 Let us keep the above notation, and let us assume that
 \begin{equation}\label{NArample}
\dega \overline{N}_{\cP}\tilde{\cV} > 0.
\end{equation}

Then,
for any  vector bundle $\tilde{\cE}$ over $\cVt$, the Hilbertisable pro-vector bundle $\Gamma_{L^2}(\tilde{\cV}; \tilde{\cE})$ is $\theta$-finite.  

Moreover, for any Hermitian line bundle $\cLbh$ over $\cVt$ and any conjugation invariant volume form on $V_\C$, when the integer $D$ goes to infinity, we have:
\begin{equation}\label{hotD2}
\hot( \Gamma_{L^2}(\tilde{\cV}, \nu; {\cLbh} \,^{\otimes D})) = O(D^2).  
\end{equation}
\end{theorem}

Condition (\ref{NArample}) is an arithmetic avatar of the pseudo-concavity condition  (\ref{Nample}) considered in the geometric framework of smooth formal surfaces over a field in paragraph \ref{algformalpseudoconcave}, and Theorem \ref{finitewhenNArample} is an arithmetic counterpart of Proposition \ref{formalbounds}. The basic properties of the $\theta$-invariant $\hot$  developed in Chapters \ref{thetainfinite} and \ref{SectionSum}, together with the Schwarz Lemma on compact Riemann surfaces with boundary established in Section   \ref{SchwarzLemma} will allows us to give a proof of this Theorem formally similar to the one of Proposition \ref{formalbounds}.

\subsection{}We shall also rely on some elementary properties of the $\theta$-invariants of (finite rank) Hermitian vector bundles which we now recall in a form suitable for the proof of Theorem \ref{finitewhenNArample}.

From Proposition \ref{scholiumhot}, it follows that the theta-invariant $\hot(\Lb)$ of some Hermitian line bundle $\Lb$ over $\Spec \OK$ admits the following upper bound in terms of its Arakelov degree:
\begin{equation}\label{hotchi}
\hot(\Lb) \leq \chi (\dega \Lb),
\end{equation}
 where 
 $$
\begin{array}{ccll}
\chi(x)  & :=   & 1 + x & \mbox{ if $x \geq 0,$}  \\
  & :=  & \exp x  &  \mbox{ if $x \leq 0.$}
 \end{array}
$$

Actually $\hot(\Lb)$ satisfies much stronger estimates when $\dega \Lb < 0$. A technical advantage of the above choice for $\chi$ is that it is a convex function. As a consequence, for any $x \in \R$ and any $a\in \R^\ast_+,$ we have:
\begin{equation}\label{chiconvex}
\chi(x) \leq \frac{1}{a} \int_{x-a/2}^{x+a/2} \chi(t) \, dt.
\end{equation}

Let $\Eb$ be some Hermitian vector bundle over $\Spec \OK,$ of rank $N$, and let ${\bf \xi}:= (\xi_1, \ldots, \xi_N)$ be some $N$-tuple of elements of $E^\vee$ that constitutes a $K$-basis of $E_K^\vee.$  The morphism
$${\bf \xi}: E \lra \OK^{\oplus N}$$ is injective and defines an element of $\Hom_\OK^{\leq \lambda}(\Eb, \cOb_{\Spec \OK}^{\oplus N})$  provided
$$\lambda \geq M := \max_{\sKC} \left(\sum_{i=1}^n \Vert \xi_i \Vert_{\Eb^\vee, \sigma}\right).$$
Therefore, for any Hermitian line bundle $\Lb$ over $\Spec \OK,$ the tensor product $  Id_L \otimes{\bf \xi}$ defines an injective morphism in $\Hom_\OK^{\leq 1}( \Lb  \otimes \Eb, \Lb \otimes \cOb_{\Spec \OK}(\log M)^{\oplus N})$. Together with the monotonicity of $\hot$ and (\ref{hotchi}), this implies the following upper-bound: 
\begin{equation}\label{hotchiVect}
\hot(\Lb \otimes \Eb) \leq \rk E. \hot(\Lb \otimes \cOb_{\Spec \OK}(\log M)) \leq \rk E . \chi (\dega \Lb + \log M).
\end{equation}

\subsection{Proof of Theorem \ref{finitewhenNArample}-I}
Let us consider some Hermitian vector bundle over $\cVt,$
$$\tilde{\cEb}:= (\cEh, (\cE_\sigma, \phi_\sigma, \Vert.\Vert_\sigma)_\sKC),$$
and let $\nu$ be some volume form on $V_\C$ invariant under complex conjugation.

As explained at the end of Section \ref{GammaL2},  we may describe the pro-Hermitian vector bundle 
$$\Ebh := \Gamma_{L^2}(\tilde{\cV}, \nu; \cEbh)$$
as the projective limit $\varprojlim_i \Eb_i$ of the system of admissible surjective morphisms
$$\Eb_{\bullet} : \Eb_0 \stackrel{q_0}{\longleftarrow}\Eb_1 \stackrel{q_1}{\longleftarrow}\dots \stackrel{q_{i-1}}{\longleftarrow}\Eb_i \stackrel{q_i}{\longleftarrow} \Eb_{i+1} \stackrel{q_{i+1}}{\longleftarrow} \dots$$
where
$$\Eb_i := (\Gamma(\cP_i, \cEh_{\mid \cP_i}), (\Vert. \Vert_{i,\sigma})_\sKC),$$
and the Hermitian norm $\Vert. \Vert_{i,\sigma}$ is defined has the quotient of the $L^2$-norm on $\Gamma_{L^2}(V_\sigma, \nu_\sigma; \cE_\sigma, \Vert.\Vert_\sigma)$ by the $i$-th jet map
$$\eta_{i, \sigma}: \Gamma_{L^2}(V_\sigma, \nu_\sigma; \cE_\sigma, \Vert.\Vert_\sigma) \lra  \cE_{\sigma\mid P_{\sigma i}}.$$
The morphisms $q_i$ are the obvious restriction morphisms, induced by the inclusion $\cP_i \hlra \cP_{i+1}$.

As in Section \ref{SectionMainTheo}, for every $i \in \N,$ we may consider the consider the Hermitian vector bundle over $\Spec \OK$:
$$\overline{\ker q_i}:=(\ker q_i, (\Vert.\Vert_{i,\sigma})_\sKC).$$

Recall that, if $\cI$ denotes the Ideal of $\cP$ in $\cO_{\cVh}$, there exists canonical isomorphisms of $\cO_\cVh$-Modules:
\begin{equation}\label{IsomA}
\cO_{\cP_i} \lrasim \cO_{\cVh}/\cI^{i+1} \quad \mbox{for every $i\in \N$}
\end{equation}
and
\begin{equation}\label{IsomB}
\cI^i/\cI^{i+1} \lrasim (N_\cP \cVh)^{\vee{\otimes i}}.
\end{equation}
(Indeed,  (\ref{IsomA}) holds by the very definition of $\cP_i$; (\ref{IsomB}) holds because of the smoothness assumption on $\cVh$, and shows that the quotient $\cI^i/\cI^{i+1}$ is supported by $\cP.$) From these isomorphisms, we deduce short exact sequences of $\cO_\cVh$-modules
$$0 \lra \cI^i/\cI^{i+1} \lra \cO_{\cP_i} \lra \cO_{\cP_{i-1}} \lra 0,$$
and, after taking  tensor products with $\cEh$ and spaces of global sections, short exact sequences of $\OK$-modules:
$$0 \lra \Gamma(\Spec \OK,  (N_\cP \cVh)^{\vee{\otimes i}} \otimes \cP^\ast \cEh)
\hlra \Gamma(\cP_i, \cEh_{\mid \cP_i}) \stackrel{q_{i-1}}{\lra} \Gamma(\cP_i, \cEh_{\mid \cP_{i-1}}) \lra 0.$$
This shows that, for every $i \geq 1,$ the $\OK$-module $\ker q_{i-1}$ may be identified with (the global sections of) $(N_\cP \cVh^\vee)^{\otimes i} \otimes \cP^\ast \cEh.$

Observe that, for every $\sKC,$ the complex vector space
$$(\cP^\ast \cE)_\sigma \simeq \cE_{\sigma, P_\sigma}$$
is endowed with the Hermitian norm $\Vert. \Vert_{\sigma, P_\sigma}$, restriction over $P_\sigma$ of the Hermitian metric $\Vert.\Vert_\sigma$ over $\cE_\sigma$. We shall denote by $\cP^\ast \cEb$ the Hermitian vector bundle $(\cP^\ast \cE, (\Vert. \Vert_{\sigma, P_\sigma})_{\sKC})$ over $\Spec \OK.$

Then $(N_\cP \cVh)^{\vee{\otimes i}} \otimes \cP^\ast \cEh$ appears as the underlying $\OK$-module of the Hermitian vector bundle $(\overline{N}_\cP \cVh)^{\vee{\otimes i}} \otimes \cP^\ast \cEb$. We shall denote by
$\Vert.\Vert^{\rm cap}_{i, \sigma}$ the corresponding metric on 
$$ [(N_\cP \cVh)^{\vee{\otimes i}} \otimes \cP^\ast \cEh]_\sigma \simeq (T_{P_\sigma}V_\sigma)^{\vee\otimes i} \otimes \cE_{\sigma,P_\sigma}.$$ By construction, it is the metric
deduced from the metrics $\Vert.\Vert^{\rm ca
p}_{V_\sigma, P_\sigma}$ on  $T_{P_\sigma}V_\sigma$ and $\Vert.\Vert_\sigma$ on $\cE_\sigma$ by duality and tensor product.  

Under the identification 
$$\ker q_{i-1} \simeq  (N_\cP \cVh^\vee)^{\otimes i} \otimes \cP^\ast \cEh,$$
the Hermitian structures on $\overline{\ker q_{i-1}}$ and on $(\overline{N}_\cP \cVh)^{\vee{\otimes i}} \otimes \cP^\ast \cEb$ do \emph{not} coincide in general. Indeed, the Hermitian norms $\Vert.\Vert_{i,\sigma}$ on $\overline{\ker q_{i-1}}$ are defined as quotients of $L^2$-norms, while the Hermitians norms $\Vert.\Vert^{\rm cap}_{i, \sigma}$ on  $(\overline{N}_\cP \cVh)^{\vee{\otimes i}} \otimes \cP^\ast \cEb$ are defined in terms of capacitary metrics.

However, in paragraph \ref{SchwarzLemma}, we have established a version of Schwarz Lemma on compact Riemann surfaces with boundary that allows one to compare these norms. Indeed, as a reformulation of the estimates (\ref{SchSchwE}) in Scholium \ref{SchoSchw} applied to the compact connected Riemann surfaces with boundary $(V_\sigma)_\sKC$, we obtain the following comparison estimates:

\begin{lemma}\label{CapQuotEst}
For any $\eta >0,$ there exists $C_\eta \in \R_+$ such that, for any field embedding $\sKC$ and any $i \in \N$:
\begin{equation}\label{capquot}
\Vert .\Vert_{i,\sigma}^{\rm cap} \leq C_\eta e^{\eta i} \Vert.\Vert_{i,\sigma}.
\end{equation}
\end{lemma}
\qed

From the estimates (\ref{capquot}) and the monotonicity of $\hot$ (Proposition \ref{ineqmortheta}, 1)), we derive that, for any integer $i \geq 1,$
\begin{equation}\label{capquothot}
\hot(\overline{\ker q_{i-1}}) \leq \hot( (\overline{N}_\cP \cVh)^{\vee{\otimes i}} \otimes \cP^\ast \cEb \otimes \cO(\eta\, i + \log C_\eta)).
\end{equation}

Moreover, the upper-bound (\ref{hotchiVect}) on the $\theta$-invariant of the tensor product of some fixed Hermitian vector bundle and some Hermitian line bundle shows the existence of some $M >0$, depending only on $\cP^\ast \cEb$, such that
 \begin{equation*}
\hot( (\overline{N}_\cP \cVh)^{\vee{\otimes i}} \otimes \cP^\ast \cEb \otimes \cO(\eta i + \log C_\eta))
\leq \rk \cE . \chi( - i \dega \overline{N}_\cP \cVh + [K:\Q] (\eta \, i + \log C_\eta) + \log M).
\end{equation*}

When the arithmetic ampleness condition (\ref{NArample}) is satisfied, we may choose $\eta$ in the interval $]0, [K:\Q]^{-1}\dega \overline{N}_\cP \cVh[$. Then
$$a: = \dega \overline{N}_\cP \cVh - [K:\Q] \eta$$
is positive, and, if we let:
$$b := [K:\Q] \log C_\eta + \log M,$$
the previous estimates imply that, for every integer $i\geq 1,$
$$\hot(\overline{\ker q_{i-1}}) \leq  \rk \cE \, \chi(-a i +b).$$

Finally, we obtain:
$$\sum_{n=0}^{+\infty}\hot(\overline{\ker q_n}) \leq \rk \cE  \sum_{i =1}^{\infty} \chi( -ai +b) < +\infty.
$$
 This shows that the projective system $\Eb_\bullet$ is summable.
 
 More generally, the estimates (\ref{capquot}) show that, for any $\delta \in \R,$  
 \begin{equation}
\hot(\overline{\ker q_{i-1}}\otimes \cO(\delta)) \leq \hot( (\overline{N}_\cP \cVh)^{\vee{\otimes i}} \otimes \cP^\ast \cEb \otimes \cO(\eta\, i + \log C_\eta + \delta)).
\end{equation}
Replacing (\ref{capquothot}) by this ``twisted version", we now obtain that
$$\hot(\overline{\ker q_{i-1}}\otimes \cO(\delta)) \leq  \rk \cE \, \chi(-a i +b+ [K:\Q] \delta),$$
and consequently,
$$\sum_{n=0}^{+\infty}\hot(\overline{\ker q_n}\otimes \cO(\delta)) < +\infty.$$

This establishes that $\Eb_\bullet\otimes \cO(\delta)$ is summable for every $\delta \in \R$, and finally that $\Gamma_{L^2}(\tilde{\cV}, \nu; \cEbh) \simeq \varprojlim_i \Eb_i$ is $\theta$-summable.

\subsection{Proof of Theorem \ref{finitewhenNArample}-II}
Let us know consider some Hermitian line bundle $\tilde{\cLb}$ over $\cVt$. For any $D \in \N,$ we may apply the previous construction to 
$\tilde{\cEb} := \tilde{\cLb}^{\otimes D}.$ In this way, we get the upper bound:
 \begin{equation}\label{boundhotLD}\hot( \Gamma_{L^2}(\tilde{\cV}, \nu; \cEbh)) \leq \hot(\Eb_0) + \sum_{n =0}^{+\infty}\hot(\overline{\ker q_n}) = \sum_{i=0}^{+\infty}\hot((N_\cP \cVh^\vee)^{\otimes i} \otimes \cP^\ast \cLh^{\otimes D}, (\Vert.\Vert_{D,i,\sigma})_\sKC),
 \end{equation}
where $\Vert.\Vert_{D,i,\sigma}$ denotes the metric on 
$$[(N_\cP \cVh^\vee)^{\otimes i} \otimes \cP^\ast \cLh^{\otimes D}]_\sigma 
\simeq (T_{P_\sigma}V_\sigma)^{\vee\otimes i} \otimes \cL^{\otimes D}_{\sigma,P_\sigma}$$ 
deduced (by ``subquotient" as above) from the $L^2$-norm on $\Gamma_{L^2}(V_\sigma, \nu_\sigma; \cL^{\otimes D}_\sigma, \Vert.\Vert_\sigma)$.

By using the estimates (\ref{SchSchwLD}) in Scholium \ref{SchoSchw}, we now obtain the following variant of Lemma \ref{CapQuotEst}:
 
\begin{lemma}\label{CapQuotEstLD}
For any $\eta >0,$ there exists $C_\eta \in \R_+$ such that, for any field embedding $\sKC$ and any $(D,i) \in \N_{>0} \times \N$:
\begin{equation}\label{capquotLD}
\Vert .\Vert_{i,\sigma}^{\rm cap} \leq C^{D+1}_\eta e^{\eta i} \Vert.\Vert_{i,\sigma}.
\end{equation}
\end{lemma}
\qed

The estimates (\ref{capquotLD}) and the monotonicity of $\hot$now imply that, for any $i \in \N,$
\begin{multline}\label{LD1}
\hot((N_\cP \cVh^\vee)^{\otimes i} \otimes \cP^\ast \cLh^{\otimes D}, (\Vert.\Vert_{D,i,\sigma})_\sKC)  \\ \leq
\hot( (\overline{N}_\cP \cVh)^{\vee{\otimes i}} \otimes \cP^\ast \cLb^{\otimes D} \otimes \cO(\eta\, i + (D+1)\log C_\eta)).
\end{multline}

Besides, the upper-bound (\ref{hotchi}) on the $\theta$-invariant of an Hermitian line bundle in terms of its Arakelov degree show that
\begin{multline}\label{LD2}
\hot( (\overline{N}_\cP \cVh)^{\vee{\otimes i}} \otimes \cP^\ast \cLb^{\otimes D} \otimes \cO(\eta\, i + (D+1)\log C_\eta)) \\ \leq
\chi( - i \dega \overline{N}_\cP \cVh + D \dega \cP^\ast \cLb + [K:\Q] (\eta \, i + (D+1)\log C_\eta)).
\end{multline}

When the arithmetic ampleness condition (\ref{NArample}) is satisfied and $\eta$ belongs to the interval $]0, [K:\Q]^{-1}\dega \overline{N}_\cP \cVh[$, we obtain from (\ref{LD1}) and (\ref{LD2}):
\begin{equation}\label{LD3}
\hot((N_\cP \cVh^\vee)^{\otimes i} \otimes \cP^\ast \cLh^{\otimes D}, (\Vert.\Vert_{D,i,\sigma})_\sKC) \leq \chi (-ai + a'D +b),
\end{equation}
where
$$a: = \dega \overline{N}_\cP \cVh - [K:\Q] \eta >0,$$
$$a' := \dega \cP^\ast \cLb + [K:\Q] \log C_\eta,$$
and
$$b := [K:\Q] \log C_\eta.$$

The asymptotic bound (\ref{hotD2})
$$\hot( \Gamma_{L^2}(\tilde{\cV}, \nu; {\cLbh} \,^{\otimes D})) = O(D^2)$$
now follows from (\ref{boundhotLD}) and (\ref{LD3}), combined with the following elementary lemma:

\begin{lemma}
 For any $(a,a',b) \in \R^\ast_+ \times \R^2$ and any $D \in \N$, the series 
 $\sum_{n \in \N} \chi( -an + a'D+b)$ is convergent. Moreover, when $D$ goes to $+\infty,$
 \begin{equation*}
\sum_{n \in \N} \chi(-a n + a'D+b) = O(D^2).
\end{equation*}
\end{lemma}
\begin{proof} According to the convexity estimate (\ref{chiconvex}), we have:
\begin{equation}\label{chiint}
\sum_{n \in \N} \chi( -an + a'D+b) \leq \frac{1}{a} \int_{-\infty}^{a/2+ a'D +b} \chi (t) \, dt.
\end{equation}
This establishes the convergence of the series. Moreover, when $a'\leq 0,$ the right-hand side of (\ref{chiint}) stays bounded when $D$ goes to $+\infty$. When $a' >0,$ it is bounded from above by
$$\frac{1}{a} + \frac{1}{a} \int_0^{a/2+ a'D +b} (1+t) \, dt = \frac{a'^2}{2a} D^2  + O(D).$$
 \end{proof}
 
  \section{Arithmetic pseudo-concavity and algebraization}\label{ArPsConcaveAlg}

 \subsection{A Diophantine algebraization theorem}\label{DiophantAlg} 
 
   We may now establish the Diophantine algebraization theorem that constitutes the main result of this chapter.  It may be seen as some Diophantine counterpart of the algebraization theorem concerning ``pseudo-concave" formal surfaces discussed in paragraph \ref{algformalpseudoconcave} (\cf Theorem \ref{formalAlgebr}). 
  
  Indeed, as explained in paragraph \ref{descriptionEpil}, its proof  will be  similat to the one of Theorem \ref{formalAlgebr}: it will combine the upper-bound for the theta-invariants of pro-Hermitian vector bundles of sections of line bundles over a formal-analytic surface (Theorem \ref{NArample}) and the lower-bound for the $\theta$-invariants of Hermitian vector bundles of arithmetically ample Hermitian line bundles on projective schemes over $\Spec \OK$ (Theorem \ref{AmpleThetaAs} and Corollary \ref{AmpleThetaPM}).

\begin{theorem}\label{ArAlgK} Let $\cVt := (\cVh, (V_\sigma, O_\sigma, i_\sigma)_\sKC)$ be some smooth formal-analytic surface over $\Spec \OK$ as above.  

 When the ``arithmetic pseudo-concavity" condition (\ref{NArample})
 $$
\dega \overline{N}_{\cP}\tilde{\cV} > 0.
$$
 is satisfied,  
for any $\OK$-morphism
$$f := (\hat{f}, (f_\sigma)_{\sKC}) : \cVt \lra \cX$$
 with range some quasi-projective $\OK$-scheme $\cX$, the image of $f$ is algebraic.
\end{theorem}

Recall that the meaning of the algebraicity of the image of $f$ has been discussed in paragraph \ref{AlgAr}.

\proof[Proof of Theorem  \ref{ArAlgK}]  We have to show that the closed integral subscheme $\oli{\im \hat{f}}$ satisfies 
\begin{equation}\label{dimleq2}
\dim \oli{\im \hat{f}} \leq 2.
\end{equation}

With no restriction of generality, we may assume that $\cX$ is projective (this follows from Lemma \ref{CompShrink}), and moreover, by replacing $\cX$ by $\oli{\im f}$, that
\begin{equation}\label{Ximf}
\cX = \oli{\im \hat{f}}.
\end{equation}

In particular, $\cX$ is an integral projective scheme over $\Spec \OK$; it is clearly flat over $\Spec \OK$ (indeed, it contains the $\OK$-point $\hat{f}(\cP)$). Let
$$d := \dim \cX = \dim \oli{\im \hat{f}}.$$

Let us apply Corollary \ref{AmpleThetaPM} to $\cX$. We shall keep the notations of this corollary, and denote by $\cLb$ 
and by $(\Eb_D, i_D)_{D \in \N}$ a Hermitian line bundle over $\cX$ and a sequence of Hermitian vector bundles over $\Spec \OK$ and of injections 
$$i_D : E_D \hlra \pi_\ast \cL^{\otimes D}$$ satisfying the conclusions (\ref{comphermunif}) and (\ref{lowerDioph}) of Corollary \ref{AmpleThetaPM}. Notably, there exists $c>0$ such that, for any large enough integer $D$,
\begin{equation}\label{recalllowhot}
\hot(\Eb_D) \geq c. D^{\dim \cX}.
\end{equation}

Besides, let us choose a family
$$\nu:= (\nu_{\sigma})_\sKC$$
of positive $C^\infty$ volume forms on the Riemann surfaces $V_\sigma$ that is invariant under complex conjugation and satisfies the normalization conditions 
\begin{equation}\label{normnu}
\int_{V_\sigma} \nu_\sigma = 1 \quad\mbox{ for every embedding $\sKC$.}
\end{equation}

For any $D \in \N,$ we may consider the proHermitian vector bundle over $\Spec \OK$
$$\Gamma_{L^2}(\cVt, \nu; f^\ast \cL^{\otimes D}) := (\Gamma(\cVh, \hat{f}^\ast \cL^{\otimes D}), (\Gamma_{L^2}(V_\sigma, \nu_\sigma; \cLb_\sigma)_{\sKC})).$$

According to Theorem \ref{finitewhenNArample}, the ``arithmetic pseudo-concavity" of $\cVt$ implies that the pro-Hermitian vector bundle $\Gamma_{L^2}(\cVt, \nu; f^\ast \cL^{\otimes D})$ is $\theta$-finite and that, when $D$ goes to infinity,
\begin{equation}\label{hotD2fast}
\hot(\Gamma_{L^2}(\cVt, \nu; f^\ast \cL^{\otimes D})) = O(D^2).
\end{equation}

Besides, we may consider the pull-back maps
$$
\begin{array}{rcl}
 \hat{\phi}_D :\Gamma(\cX, \cL^{\otimes D}) & \lra   &  \Gamma(\cVh, \hat{f}^\ast \cL^{\otimes D})  \\
  s & \longmapsto  & \hat{f}^\ast s.   
\end{array}
$$
For any $D \in \N$ and any $s\in \Gamma (\cX, \cL^{\otimes D}) \setminus \{0\}$, the morphism $\hat{f}$ does \emph{not} factorizes through the inclusion $\div s \hra \cX$ (by (\ref{Ximf}). This shows that the maps $\hat{\phi}_D$ are injective.

We may also consider the maps 
$$
\begin{array}{rcl}
 {\phi}_{D,\sigma} :\Gamma(\cX_\sigma, \cL_\sigma^{\otimes D}) & \lra   &  \Gamma_{L^2}(V_\sigma, f_\sigma^\ast \cL_\sigma^{\otimes D})  \\
  s & \longmapsto  & f_\sigma^\ast s.   
\end{array}
$$
Clearly, they satisfy:
\begin{equation}\label{boundphiDsigma}
\Vert \phi_{D,\sigma} s\Vert_{L^2(V_\sigma,\nu_\sigma)} \leq \Vert \phi_{D,\sigma} s\Vert_{L^\infty(V_\sigma)} \leq \Vert s \Vert_{L^\infty(\cX_\sigma)}. 
\end{equation}
(The first inequality holds because of the normalization conditions (\ref{normnu}).)

 Moreover the maps $\hat{\phi}_{D}$ and ${\phi}_{D,\sigma}$ are compatible --- the base change
 $$ \hat{\phi}_{D, \sigma} :\Gamma(\cX, \cL^{\otimes D})_\sigma  \lra     \Gamma(\cVh, \hat{f}^\ast \cL^{\otimes D})_\sigma$$
 may be identified with the composition of $\phi_{D,\sigma}$ and of the restriction map
 $$\Gamma_{L^2}(V_\sigma, f_\sigma^\ast \cL_\sigma^{\otimes D}) \lra (f_\sigma^\ast \cL_\sigma^{\otimes D})_{\mid  \widehat{V_\sigma}_{, P_\sigma} }$$ 
 --- and therefore define a morphism of Hilbertisable vector bundles 
$$\phi_D: \pi_\ast \cL^{\otimes D} \lra \Gamma_{L^2} (\cVt, f^\ast \cL^{\otimes D}).$$

The norm estimates (\ref{boundphiDsigma}) and (\ref{comphermunif}) show that the composite morphism
$$\phi_D \circ i_D : \Eb_D \lra \Gamma_{L^2}(\cVt, \nu; f^\ast \cLb^{\otimes D})$$
belongs to
$\Hom_{\OK}^{\leq 1} (\Eb_D,\Gamma_{L^2}(\cVt, \nu; f^\ast \cLb^{\otimes D})).$
Moreover, like $\hat{\phi}_D,$ the composite map $\hat{\phi}_D \circ i_D$ is injective. Consequently, we obtain:
\begin{equation}\label{hotEhotGamma}
\hot(\Eb_D) \leq \hot(\Gamma_{L^2}(\cVt, \nu; f^\ast \cLb^{\otimes D})) 
\end{equation}
(\cf Proposition \ref{ineqmorthetapro}).

From  (\ref{recalllowhot}), 
(\ref{hotD2fast}),
and (\ref{hotEhotGamma}), we derive the requested inequality $$\dim \cX \leq 2$$ by letting $D$ go to infinity.
\qed

\subsection{Proof of Theorem \ref{ArAlgPM}}.

Let us now explain how, from Theorem \ref{ArAlgK}, one simply derives the ``naive" algebraicity criterion stated at the beginning of this chapter as Theorem \ref{ArAlgPM}. 

Let us go back to the notation of this theorem, introduced in paragraph \ref{SimpleArAlG}, and let us choose some imaginary quadratic field $K$. We shall denote by $\{\sigma_0, \overline{\sigma}_0 \}$ the field embeddings of $K$ in $\C$.

From the data of $\phi \in \Z[[X]]^N$, of the pointed Riemann surface with boundary $V$, and of the map $j:V \lra \PP^N(\C),$ we may construct a smooth formal-analytic surface $\cVt$ over $\Spec \OK$ by letting
$$\cVh := {\rm Spf\,} \OK[[X]],$$
$$V_{\sigma_0} := V \quad\mbox{ and } \quad V_{\overline{\sigma}_0} := V^{c.c.},$$
and
$$i_{\sigma_0} := \widehat{\gamma_O}^{-1} :\hV_{\sigma_0}\simeq {\rm Spf} \, \C[[T]] \lrasim \hV_O \quad\mbox{and}\quad 
i_{\overline{\sigma}_0} := \widehat{\overline{\gamma}_O}^{-1} :\hV_{\overline{\sigma}_0}\simeq {\rm Spf \,} \C[[T]] \lrasim \hV_O^{c.c.}.$$

Then one defines a morphism
$$f := (\hat{f},(f_{\sigma_0}, f_{\overline{\sigma}_0}): \cVt \lra \PP^N_{\OK}$$
by
$$\hat{f} := \phi, \quad f_{\sigma_0}:= \gamma, \quad \mbox{and} \quad f_{\overline{\sigma}_0} := \overline{\gamma}.$$

By the very definition of the Arakelov degree, we have
$$\dega \overline{N}_{\cP}\tilde{\cV} = - 2 \log \Vert D\gamma(O)^{-1} \phi'(0) \Vert^{\rm cap}_{V,O}.$$
This is positive precisely when the condition (\ref{ArAmplNaive}) in Theorem \ref{ArAlgPM} is satisfied. According to Theorem \ref{ArAlgK}, this condition implies the algebraicity of the image of $f$. In particular $Y:= \overline{\im \hat{f}_K}$ is a closed integral scheme of dimension 1 in $\PP^N_K$ such that 
\begin{equation}\label{CK}
\hat{C}_K \subset \hat{Y}_P \quad \mbox{ and  } \quad
\gamma(V) \subset Y_{\sigma_0}(\C).
\end{equation} 

The fact that the series $\phi_1, \ldots, \phi_N$ that define $\hat{f}$ belongs to $\Q[[X]]$ implies, by a straightforward descent argument, that $Y$ may be written $X_K$ for some closed integral curve $X$ in $\PP^N_\Q$. From (\ref{CK}), it follows that the curves $X$ satisfies  the conclusion  (\ref{inclalg}) of Theorem \ref{ArAlgPM}, namely:
$$\hat{C}_\Q \subset \Xh_P \quad\mbox{and}\quad \gamma(V) \subset X(\C).$$

\subsection{Borel's rationality criterion}\label{AppIBorel}
As a first application of Theorem \ref{ArAlgK}, let us explain how one may derive Borel's rationality criterion (Theorem \ref{thBorel}, \emph{supra}) from its naive variant Theorem \ref{ArAlgPM}. 

We shall actually establish Borel's criterion in the following more general form:

\begin{theorem}\label{thBoreless}
 Let $f$ be some formal series in $\Z[[X]].$
 
 If the complex radius of convergence of $f$ is positive, and if $f$ extends, as a $\C$-analytic function, to $\mathring{D}(0,R) \setminus F$ for some $R \in ] 1, +\infty[$ and some finite subset $F$ of  $\mathring{D}(0,R) \setminus \{0\}$,  then $f$ is Taylor expansion at $0$ of some rational function in $\Q(X)$. 
\end{theorem}

Theorem \ref{thBorel} is the special case of Theorem \ref{thBoreless} where the singularities  $F$ of $f$ are assumed to be (at worst) poles. Theorem \ref{thBoreless} is actually a consequence of the main result of P\'olya \cite{Polya1928} (see notably Section 11). 

Let us start with some simple observation.

\begin{lemma}\label{Borelnaive} Theorem \ref{thBoreless} holds when $f$ satisfies the additional assumption that the  $f$ has a pole at each point of $F$ and when $F$ is contained in the algebraic closure $\Qbar$ of $\Q$ in $\C$. 
 \end{lemma}

\begin{proof}
 When this additional assumption holds, we may find some non-zero polynomial $P \in \Z[X]$ which admits every point of $F$ as a zero. 
 
 Then, if $N$ denotes the maximum order of the poles of $f$ in the closed unit disk $\Db(0;1)$,  the product $P^N f$ is a formal series $\sum_{n \in \N} a_n X^n$ in $Z[[X]]$ that is the expansion of some holomorphic function on some open disk $\mathring{D}(0,R')$ of radius  $R'>1$. 
 This implies that $\sum_{n \in \N} \vert a_n \vert < + \infty,$ and therefore, as the $a_n's$ are integers, that $a_n$ vanishes for $n$ large enough. 

This shows that $P^N f$ is a polynomial $Q$ in $\Z[X]$, and finally that $f = Q/P^N$ belongs to $\Q(X)$.
\end{proof}

Let us now consider $f$ as in the statement of Theorem \ref{thBoreless}, and let us choose $R' \in ]1,R[$ such that $F \subset \mathring{D}(0,R')$.

According \ref{ExampGreen}, Example (iii), for any $\epsilon >0$ small enough, 
$$V_\epsilon := \Db(0,R') \setminus \bigcup_{a\in F} \mathring{D}(a ; \epsilon)$$
is a compact Riemann surface with boundary, containing $0$ in its interior, such that
\begin{equation}\label{smallcap}
\Vert \partial /\partial z \Vert_{V_\epsilon, 0}^{\rm cap} < 1.
\end{equation}
By construction, $f$ defines an analytic function --- that we shall denote by $f^{\rm an}$ --- on some open neighborhood of $V_\epsilon$ in $\C$.

We may apply Theorem \ref{ArAlgPM} to the following data:
$$V := V_\epsilon \quad \mbox{ and } \quad O:= 0,$$
$$\phi := (X, f) \in \Z[[X]]^2,$$
and 
$$\gamma: V \lra \C^2 \quad\mbox{ defined by $\gamma(z) := (z, f^{\an}(z)).$}$$

The condition (\ref{jform})
$$\hat{\gamma}_O : \Vh_0 \lrasim \hat{C}_\C$$
is clearly satisfied. Moreover
$$\phi'(0) = \partial/\partial x_1 + f'(0) \partial/\partial x_2 = D\gamma(0) \partial/\partial t.$$
Therefore
$$D\gamma(0)^{-1} \phi'(0) = \partial/\partial t$$
and, according to (\ref{smallcap}), the condition (\ref{ArAmplNaive}) 
$$\Vert D\gamma(O)^{-1} \phi'(0) \Vert^{\rm cap}_{V,O} < 1$$
is satisfied. 

The conclusion of Theorem \ref{ArAlgPM} asserts that  the image of $\gamma$ is algebraic, and is actually contained in (the complex points of) some algebraic curve defined over $\Q$. It may be rephrased as the existence of some closed integral curve $C$ in $\A^1_\Q \times \PP^1_\Q$ such that 
$$\gamma(V) \subset C(\C).$$
By analytic continuation, we have
$$\gamma(\mathring{D}(0,R) \setminus F) \subset C(\C).$$
In other words, the graph of $f^{\an}: \mathring{D}(0,R) \setminus F) \lra \C$ is contained in the affine complex curve $C(\C) \cap \A^2(\C).$ 

This implies that the singularities of $f$ at the points of $F$ are (at worst) poles, and that, if some point of $F$ is actually  a pole of $f$, it belongs to $\Qbar$. 

Indeed  the projection ${\rm pr}_1 : C_\C \lra \A^1_\C
$ is a finite morphism, and by Riemann's theorem, its analytic section $\gamma$ over $\mathring{D}(0,R) \setminus \{0\}$ extends analytically to $\mathring{D}(0,R)$. This proves that $f^{\an}$ is meromorphic on $\mathring{D}(0,R)$. Moreover the poles of this function are contained in ${\rm pr}_1 (C(\C) \cap (\C \times \{\infty\})),$ and therefore belong to $\Qbar$ since $C$ is defined over $\Q$.

Finally we can conclude the proof of Theorem \ref{thBoreless} by applying Lemma \ref{Borelnaive}.

\section{The isogeny theorem for elliptic curves over $\Q$}\label{AppIIIsog}

In this final section, we want to demonstrate that the arithmetic algebraization criterion established in Theorem \ref{ArAlgK} admits significant Diophantine applications, in spite of its unsophisticated formulation when compared to the more general and flexible criteria established in \cite{Andre89}, Chapter VIII or in \cite{Bost01}. Namely, we  use  Theorem \ref{ArAlgK} to derive to derive a classical isogeny criterion,  due to Serre and Faltings, for elliptic curves over $\Q$.

The relevance of algebraicity theorems such as Theorems  \ref{ArAlgPM} and \ref{ArAlgK} for the construction of isogenies between elliptic curves is one of the striking discoveries presented in the seminal article \cite{ChudnovskysAcad85} by D.V. and G.V. Chudnovsky. Like the one in \emph{loc. cit.}, the construction of isogenies presented in this section will rely on some famous result of Honda (\cite{Honda68Formal}) concerning formal groups of elliptic curves over $\Q$ and of their models over $\Z$. (Stronger form of the algebraicity criterion, such as the ones referred to above, would allow one to rely on less precise results concerning these formal groups than Honda's; see for instance \cite{Bost01}, Corollary 2.5. We shall actually not use the full strength of Honda results, and avoid any reference to the theory of minimal and N\'eron models of elliptic curves.)   

In this section, we assume some basic knowledge of the geometry and arithmetic of elliptic curves, say at the level of Robert's Lecture Notes  \cite{Robert73} or Silverman's textbook \cite{Silverman92}.

\subsection{The isogeny theorem} Let $E$ be some elliptic curve over $\Q$. If $N$ denotes some ``sufficiently divisible" positive integer, there exists some elliptic curve over $\Z[1/N]$
\begin{equation*}
\pi: \cE \lra \Spec \Z[1/N] 
\end{equation*}
that is a model of $E$ --- that is, an elliptic curve whose generic fiber is isomorphic to $E_\Q$\footnote{The structure of $\cE$ as an elliptic curve over $\Z[1/N]$ is determined by its $\Z[1/N]$-scheme structure and its zero section $\epsilon$. Formally, a specific isomorphism $\iota : \cE_\Q \lrasim E$ is part of the data that define a model of $E$ over $\Spec \Z[1/N]$, and $\iota$ maps $\epsilon_\Q$ to the zero element $0_E$ in the abelian group $E(\Q)$.}. Moreover any two such models 
$$ \pi: \cE \lra \Spec \Z[1/N] \quad \mbox{ and } \quad \pi': \cE' \lra \Spec \Z[1/N']$$
of $E$ become isomorphic over some open subscheme $\Spec \Z[1/ \tilde{N}]$ (where $\tilde{N}$ denotes some common positive multiple of $N$ and $N'$; actually, one may take $\tilde{N} := {\rm lcm} (N, N')$).

For any prime number $p$ such that $(p,N) =1,$ we may consider the elliptic curve $\cE_{\F_p}$ over the prime field $\F_p$ and its $a_p$ invariant:
$$a_p(\cE_{\F_p}) := p+1 - \vert \cE_{\F_p} (\F_p) \vert =  p+1 - \vert \cE(\F_p) \vert.$$
The integer $a_p(\cE_{\F_p})$ is defined, and independent of the chosen model $\cE$ of $E$, for any large enough prime number $p$, and is denoted by $a_p(E)$.

Our aim in this section is to prove the following theorem on elliptic curves over $\Q$, due to Serre under some additional hypothesis of bad reduction (\cite{Serre68}, IV.2.3) and to Faltings in general (\cite{Faltings83}, §5, Corollary 2; Faltings actually establishes a considerably more general result, concerning abelian varieties over arbitrary number fields). 

\begin{theorem}\label{IsogenQ}
Let $E$ and $E'$ be two elliptic curves over $\Q$. If, for any large enough prime $p,$
$$a_p(E) = a_p(E'),$$
then $E$ and $E'$ are isogenous over $\Q$. 
\end{theorem}
 
 The converse implication is classically known to hold: it follows from the fact that two isogenous elliptic curves over $\F_p$ have the same $a_p$ invariant.

\subsection{Honda's theorem} Let $E$ be some elliptic curve over $\Q$ and let $$\pi:\cE \lra \Spec \Z[1/N]$$ be a model of $E$ as above. We may consider the formal completion $\cEh$ of $\cE$ along its zero section $\epsilon$. It is a pointed smooth formal curve over $\Spec \Z[1/N]$. Moreover, the group scheme structure of $\cE$ induces on $\cEh$ a structure of formal group scheme over $\Spec \Z[1/N].$ 

We shall use Honda's results on the formal groups associated to elliptic curves over $\Q$ and to their models in the following form:

\begin{theorem}[\cite{Honda68Formal}, Section 4, Theorem 4 and Corollary 2]\label{HondaTh} Let $E_1$ and $E_2$ be two elliptic curves over $\Q$, and let $\cE_1$ and $\cE_2$ be elliptic curves over $\Spec \Z[1/N_0]$ (for some positive integers $N_0$) that are models of $E_1$ and $E_2$ respectively.

If, for any large enough prime $p$, 
$$a_p(E_1) =a_p(E_2),$$
there exists some multiple $N$ of $N_0$ such that $\cEh_{1, \Z[1/N]}$ and $\cEh_{2, \Z[1/N]}$ are isomorphic, as formal groups over $\Spec \Z[1/N].$
 \end{theorem}
 
 Actually Honda establishes a more precise result, concerning normalized isomorphisms of the formal groups of the N\'eron models of $E_1$ and $E_2$. His proof is a beautiful application of the basic results concerning minimal models of elliptic curves and of some classical theorems of Lazard, Lubin and Tate  concerning one parameter formal group laws (see \cite{Lazard55}, \cite{Lubin64}, and \cite{LubinTate65})\footnote{A reader familiar with the content of \cite{Robert73} or \cite{Silverman92} should have no difficulty in getting acquainted with the content of the classical articles \cite{Lazard55}, \cite{Lubin64} and \cite{LubinTate65}, and then reading Honda's proof of Theorem \ref{HondaTh} in \cite{Honda68Formal}.}. Let us also indicate that a motivation behind Honda's results in \ref{HondaTh} --- that actually relate the $L$-functions and the formal groups 	attached to elliptic curves over $\Q$ --- was the conjecture of Shimura-Taniyama-Weil on the modularity of elliptic curves over $\Q$ (see also \cite{Honda70}), and that this circle of idea had also been explored by Cartier (\cite{Cartier71}). 
 
 \subsection{Formal and analytic groups associated to elliptic curves}
 
 Besides Honda's Theorem recalled above, we shall also rely on some elementary results concerning the formal and analytic groups associated to some elliptic cuve over $\Q$ and its base changes to local fields. These results, in contrast to Honda's Theorem, actually admit straightforward generalizations concerning commutative algebraic groups over (local) fields of characteristic zero. For the sake of simplicity, we will state them for elliptic curves only.
 
 Let $E$ be some elliptic curve over some field $k$ of characteristic zero, and let $\omega$ be a non-zero element in $\Omega^1(E/k) := \Gamma(E, \Omega^1_{E/k}).$
 
 The formal completion $\hE$ of $E$ at its zero element $O_E$ is a formal group scheme of dimension $1$ over $k$. Moreover, if $\G_{a,k} \simeq \Spec k[X]$ denotes the additive group over $k$ and $\widehat{\G}_{a,k}\simeq \Spf k[[X]]$ the associated formal group (deduced from $\G_{a,k}$ by completion at $0$), there exists a unique isomorphism of formal groups over $k$
 $$\Exph_{E,\omega} : \widehat{\G}_{a,k} \lrasim \Eh$$
which satisfies the normalization condition:
$$ \Exph_{E,\omega}^\ast \omega_{\mid \Eh} = dX.$$

The construction of $\Exph_{E,\omega}$ is compatible with extensions of the base field $k$: if $k'$ is an extension of $k$, the base change $\Exph_{E,\omega, k'}$  of  $\Exph_{E,\omega}$ coincides with $\Exph_{E_{k'},\omega_{k'}} : \widehat{\G}_{a,k'} \lrasim \Eh_{k'}$.

Moreover, when $k$ is a local field --- like the $p$-adic field $\Q_p$, or $\R$, or $\C$ ---  the formal isomorphism is actually analytic. Let us spell out these analyticity properties in elementary terms.

When $k=\Q_p$ and $E$ is embedded in some projective space $\PP^N_{\Q_p}$, say with $O_E$ mapped to the origin 
$$O:=(0,\ldots,0) \in \A^N_{\Q_p} \subset \PP^N_{\Q_p},$$
the map 
$$\Exph_{E,\omega} : \widehat{\G}_{a,\Q_p} \simeq  \Spf \Q_p[[X]] \lra \Eh\,  \hlra  \A^N_{\Q_p} $$
may be described by some $N$-tuple of formal series $(\exp^1_{E,\omega}, \ldots, \exp^N_{E,\omega})$ 
in $\Q_p[[X]]$, with vanishing constant terms. Then the analyticity of $\Exph_{E,\omega}$ may be expressed as the fact that the series $\exp^1_{E,\omega}\ldots, \exp^N_{E,\omega}$ have positive $p$-adic radii of convergence. 
 
 When $k=\C$, a stronger version of this analyticity holds. Namely there exists a (unique) $\C$-analytic map 
 $$\Exp_{E,\omega} := \Ga(\C) (:= \C) \lra E(\C),$$
 the formal germ of which at $0$ is $\Exph_{E,\omega}.$  The fact that, over archimedean places, the formal isomorphism $\Exp_{E,\omega}$ not only defines some $\C$-analytic map on some open neighborhood of the origin, but indeed  extends analytically to the whole complex line $\Ga(\C)$ (and not only to some analytic neighborhood of the origin) will play a key role in our proof of the isogeny theorem.

 Actually the map $\Exp_{E,\omega}$ is a surjective morphism of complex analytic Lie groups, and satisfies:
 $$\Exp_{E,\omega}^\ast \omega = dz.$$
 Its kernel is the lattice of periods of $\omega$,
 $$\Lambda := \left\{ \int_\gamma \omega \; ; \gamma \in H_1(E(\C), \Z) \right\},$$
 and $\Exp_{E,\omega}$ defines an isomorphism of complex analytic Lie groups:
 $$\C / \Lambda \lrasim E(\C).$$
 
 For instance, when $E$ is the complex elliptic curve in $\PP^2_\C$ of equation
 $$X_0 X_2^2 = 4 X_1^3 -g_2 X_0^2 X_1 -g_3 X_0^3,$$
 with origin $O_E:= (0:0:1),$
 and when
 $\omega := dx/y$ (where $x:= X_1/X_0$ and $y:=X_2/X_0$),
 the map $\Exp_{E,\omega}$ may be expressed in terms of the Weierstrass function $\wp_\Lambda$ associated to the lattice of periods $\Lambda$ of $\omega$. Namely, for any $z \in \C\setminus\Lambda,$ we have:
 $$\wp_\Lambda(z) := z^{-2} + \sum_{\lambda \in \Lambda\setminus\{0\}} [(z-\lambda)^{-2} - \lambda^{-2}]$$
 and 
 $$\Exp_{E,\omega}(z) = (1: \wp_\Lambda(z): \wp_\Lambda'(z)).$$
  
\subsection{Proof of Theorem \ref{IsogenQ} -- I. Construction of formal morphisms} In the proof of Theorem \ref{IsogenQ}, we shall use the following notation.

For any $n \in \Z,$ we denote by  
 $$\lambda_n : \Ahm^1_\Z:= \Spf \Z[[T]] \lra \Ahm^1_\Z$$ the morphism of formal schemes (over $\Spec \Z$)
 attached to the ``multiplication by $n$". Formally it is defined by:
 $$\lambda_n^\ast T = n\,T.$$
 Similarly, for any $q \in \Q$, we may consider the ``multiplication by $q$"
 $$\lambda_q := \widehat{\G}_{a,\Q} = \Spf \Q[[X]] \lra \widehat{\G}_{a,\Q} = \Spf \Q[[X]].$$ 
 These morphisms $\lambda_q$ are precisely the endomorphisms of the formal group $\widehat{\G}_{a,\Q}$ over $\Q$.

 \subsubsection{} Let $E$ and $E'$ be two elliptic curves over $\Q$ satisfying the assumptions of Theorem \ref{IsogenQ}.
 
 We may assume that both of them are embedded in some projective space, say $\PP^2_\Q$, and that their origins $O_E$ and $O_{E'}$ are mapped to 
 $$0 = (0,0) \in \A^2_\Q \hra \PP^2_\Q.$$
 We may take as models of $E$ and $E'$ their closures in $\PP^2_{\Z[1/N]}$, for $N$ sufficiently divisible. We will denote them by 
 $$\pi: \cE \lra \Spec \Z[1/N] \quad \mbox{ and } \quad  \pi': \cE' \lra \Spec \Z[1/N].$$
By construction $\cE$ and $\cE'$ are embedded in $\PP^2_{\Z[1/N]}$, with their zero sections $\epsilon_{\cE}$ and $\epsilon_{\cE'}$ sent to the section $O= (0,0)$ of $\A^2_{\Z[1/N]} (\hra \PP^2_{\Z[1/N]})$ over $\Z[1/N].$

After possibly replacing $N$ by some positive multiple, we may also assume that the following two conditions are satisfied:

(1) There exists some uniformizing parameter $t$ at the origin $O_E$ of $E$ and some Zariski neighborhood $U$ of the image of $\epsilon_\cE$ in $\cE$ such that $t$ belongs to $\cO_\cE(U)$ and its differential $dt \in \Gamma(U, \Omega^1_{\cE/\Z[1/N]})$ does not vanish on $U$.
 
 Then the morphism 
 $$t: U \lra \A^1_{\Z[1/N]}$$
 is \'etale and defines an isomorphism of pointed smooth formal curves over $\Spec \Z[1/N]$:
 $$\hat{t} : \cEh \lrasim \Spf \Z[1/N][[T]]=: \widehat{\A}^1_{\Z[1/N]}.$$
 This isomorphism induces an isomorphism of formal curves over $\Q$:
 $$\hat{t}_\Q : \Eh \lrasim \Spf \Q[[T]]=: \widehat{\A}^1_{\Q}.$$
 
 (2) The exists an isomorphism of formal groups over $\Spec \Z[1/N]$
 $$\widehat{\phi}: \cEh \lrasim \cEh'.$$

 Assertion (1) follows from basic principles of algebraic geometry, and (2) from Honda's Theorem \ref{HondaTh}. From now on, we assume that $t$ and $\widehat{\phi}$ as above have been chosen.
 We also choose differential forms 
 $$\omega \in \Omega^1(E/\Q) \setminus \{0\} \quad \mbox{ and } \quad \omega' \in \Omega^1(E'/\Q) \setminus \{0\}.$$
 
 By evaluating $\omega$ and $dt$ at the point $E$ of $E$, we get non-zero elements $\omega_{O_E}$ and $dt_{O_E}$ in the fiber at $O_E$ of $\Omega^1_{E/\Q}$. For some $\alpha \in \Q^\ast,$ we have:
 $$\omega_{O_E} = \alpha\, dt_{O_E}.$$
 In other words, the differential $D\Exph_{E,\omega}(0)$ of $\Exph_{E,\omega}$ maps $\partial/\partial z$ to $\alpha^{-1} \partial/\partial t$.
 
 The isomorphism of formal groups $\widehat{\phi}$ defines an isomorphism 
 $$\widehat{\phi}_\Q: \Eh \lrasim \Eh'$$
 between the formal groups of $E$ and $E'$, and there exists a unique $\beta \in \Q^\ast$ such that the following diagram of (isomorphisms of) formal groups is commutative:
 \begin{equation*}
 \begin{CD}
\Eh @>{\widehat{\phi}_\Q}>> \Eh'\\
@V{\Exph_{E,\omega}}VV      @VV{\Exph_{E',\omega'}}V \\
 \widehat{\G}_{a,\Q} @>\lambda_\beta>>  \widehat{\G}_{a,\Q}.
\end{CD}
\end{equation*} 
  
 \subsubsection{}  We may define the following isomorphisms of formal curves over $Q$:
  $$\widehat{\psi}_\Q := \widehat{t}^{-1}_\Q : \widehat{\A}^1_{\Q} \lrasim \Eh$$
  and 
   $$\widehat{\psi}'_\Q := \widehat{\phi}_\Q \circ \widehat{t}^{-1}_\Q : \widehat{\A}^1_{\Q} \lrasim \Eh',$$
   and, for any $a \in \N,$
     $$\widehat{\psi}_{a,\Q} := \widehat{\psi}_\Q \circ \lambda_{N^a}: \widehat{\A}^1_{\Q} \lrasim \Eh$$
  and 
   $$\widehat{\psi}'_{a,\Q} := \widehat{\phi}_\Q \circ \widehat{\psi}_{a,\Q}  =  \widehat{\psi}'_\Q \circ \lambda_{N^a}: \widehat{\A}^1_{\Q} \lrasim \Eh'.$$
   
 For any $a \in \N,$ we may consider the commutative diagram  of isomorphisms of formal curves over $\Q$:
 \begin{equation}\label{comphipsia}
 \begin{CD}
\Ahm^1_\Q @>{\psi_{a,\Q}}>> 
 \Eh @<{\Exph_{E,\omega}}<< \widehat{\G}_{a,\Q} \\ 
@V{\rm Id}VV  @V{\widehat{\phi}_\Q}VV      @VV{\lambda_\beta}V \\
\Ahm^1_\Q @>{\psi'_{a,\Q}}>>  \Eh' @<{\Exph_{E',\omega'}}<<  \widehat{\G}_{a,\Q}.
\end{CD}
\end{equation}  
 
\begin{lemma}\label{integralform}
 1) For any $a \in \N,$ $\widehat{\psi}_{a,\Q}$ and $\widehat{\psi}'_{a,\Q}$ extend to isomorphisms of formal curves over $\Spec \Z[1/N]$:
 $$\widehat{\psi}_{a,\Z[1/N]} : \widehat{\A}^1_{\Z[1/N]} \lrasim \cEh \;(\hra \PP^2_{\Z[1/N]})$$
  and 
   $$\widehat{\psi}'_{a,\Z[1/N]} = {\widehat{\phi}} \circ \widehat{\psi}_{a,\Z[1/N]} : \widehat{\A}^1_{\Z[1/N]} \lrasim \cEh' \;(\hra \PP^2_{\Z[1/N]}).$$
   2) There exists $a_0 \in \N$ such that, for any integer $a \geq a_0,$ $\widehat{\psi}_{a,\Q}$ and $\widehat{\psi}'_{a,\Q}$ extends to morphisms of (formal) schemes over $\Spec \Z$:
     $$\widehat{\psi}_{a} : \widehat{\A}^1_{\Z} \lra \A^2_{\Z} \quad \mbox{and} \quad
\widehat{\psi}'_{a}  : \widehat{\A}^1_{\Z} \lra \A^2_{\Z}.$$
\end{lemma}
 
 \begin{proof} Assertion 1) is straightforward. Indeed, all morphisms in the  diagram 
 $$\widehat{\psi}_{a,\Z[1/N]} \stackrel{\lambda_{N^a}}{\lra} \widehat{\psi}_{a,\Z[1/N]} 
 \stackrel{\widehat t}{\longleftarrow} \cEh \stackrel{\widehat{\phi}}{\lra} \cEh'$$ are isomorphisms of smooth formal curves over $\Spec \Z$, and  $\widehat{\psi}_{a,\Z[1/N]}$ (resp. $\widehat{\psi}'_{a,\Z[1/N]}$) may be defined as $\widehat{t}^{-1} \circ \lambda_{N^a}$ (resp. $\widehat{\phi} \circ \widehat{t}^{-1} \circ \lambda_{N^a}$).
 
 Assertion 2) will follow from 1) and from the analyticity of $\widehat{t}_\Q$, $\Exp_{E,\omega},$ and
 $\Exp_{E',\omega'}$ over the $p$-adic field $\Q_p$, for the primes $p$ dividing $N$.
 
 Recall that the $\Q$-morphism 
 $t_\Q: U_\Q \lra \A^1_\Q$
 is \'etale and maps $O_E$ to $0$. This implies, not only that it defines an isomorphism of formal curves over $\Q$
 $$\widehat{\psi}_\Q = \widehat{t}^{-1}_\Q : \widehat{\A}^1_{\Q} \lrasim \widehat{U}_{\Q, O_E} = \Eh_Q (\hra \A^2_\Q),$$
 but also that, for any prime number $p$, it defines an isomorphism between the germ of $\Q_p$-analytic curves of $\A^1(\Q_p)$ and $E(\Q_p)$ at $0$ and $O_E$. 
 
 In elementary terms, this means that the power series $(\psi^1, \psi^2)$ with vanishing constant terms in $\Q[[X]]$, defined as the two components of $\widehat{\psi}_\Q$, have positive $p$-adic radii of convergence. 
 In other words, for every prime $p$, there exists $a_0(p) \in \N$ such that, for any integer $a \geq a_0(p)$, the coefficients of $\psi^1(p^a X)$ and $\psi^2(p^a X)$ are $p$-adic integers.
 
 Since $\psi^1$ and $\psi^2$ belong to $\Z[1/N][[X]]$ by 1), this implies that, for any integer
 $a \geq a_0 := \max_{p \mid N} a_0(p),$
 the series $\psi^1(N^a X)$ and $\psi^2(N^a X)$ belong to $\Z[[X]]$. This precisely means that $\widehat{\psi}_{a,\Q}$  extends to a morphism $\widehat{\psi}_{a} : \widehat{\A}^1_{\Z} \lra \A^2_{\Z}.$
 
 Similarly, the $\Q_p$-analyticity of 
 $$\widehat{\phi}_\Q = \widehat{\Exp}_{E',\omega'} \circ \lambda_\beta \circ \widehat{\Exp}_{E,\omega}^{-1}
$$
 implies the one of $\widehat{\psi}'_\Q := \widehat{\phi}_\Q \circ \widehat{t}^{-1}_\Q $ and allows one to prove that $\widehat{\psi}'_{a,\Q}$  extends to a morphism $\widehat{\psi}'_{a} : \widehat{\A}^1_{\Z} \lra \A^2_{\Z}$ when $a$ is large enough.
  \end{proof}

 \subsection{Proof of Theorem \ref{IsogenQ} -- II. Algebraization}
 \subsubsection{}
For any $a \in \N,$ 
from the isomorphisms in the first line of the diagram (\ref{comphipsia}), we deduce an isomorphism:
$$i_a := \Exph_{E,\omega}^{-1} \circ\, \psi_{a,\Q}: \Ahm^1_\Q \lrasim \widehat{\G}_{a,\Q} = \Ahm^1_\Q.$$

By construction, 
\begin{equation}\label{psicomp}
\Exph_{E,\omega} \circ\,  i_a = \psi_{a,\Q}.
\end{equation}
Moreover, from the commutativity of (\ref{comphipsia}), we deduce:
\begin{equation}\label{psiprcomp}
\Exph_{E',\omega'} \circ \, \lambda_\beta \circ i_a = \widehat{\phi}_\Q \circ \psi_{a,\Q} = \widehat{\psi}'_{a,\Q}.
\end{equation}

 For any $a \in \N$ and $R \in \R_+^\ast,$ we may define the  smooth formal-analytic surface over $\Spec \Z$
  $$\cVt_{a,R}:= (\Ahm^1_\Z, \overline{D}(0, R), i_{a,\C}),$$
 where $i_{a,\C}$ denotes the isomorphism from $\Ahm^1_\C$ to $\widehat{\overline{D}(0, R)}_0 = 
\Ahm^1_\C$ deduced from $i_a$ by the base field extension $\Q \hra \C.$ In other words,
$$i_{a, \C} = \Exp_{E, \omega, \C}^{-1} \circ \;\widehat{t}_\C^{-1} \circ \lambda_{N^a, \C}.$$

Observe that a small enough open analytic neighborhood $V$ of $0$ in $\C\, (= \A^1(\C))$ is contained in the inverse image $\Exp_{E_\C, \omega_\C}^{-1}(U(\C))$. Then the composite map 
$$t_\C \circ \Exp_{E_\C, \omega_\C} : V \lra \C$$ is analytic and \'etale, and maps $0$ to $0$. In intuitive terms, it describe the ``Weierstrass uniformization" $\Exp_{E_\C, \omega_\C}$ of $E(\C)$ in terms of the ``local algebraic coordinate" $t_\C$ on the Zariski open neighborghood $U(\C)$ of $O_{E_\C}$ in $E(\C).$ (The function $t_\C \circ \Exp_{E_\C, \omega_\C}$ is actually an elliptic function, in the classical sense, attached to the lattice of periods of $(E_\C, \omega_\C).$)

This remark shows that, when $N=1$, the formal-analytic surface $\cVt_{a,R}$ may be understood as obtained by ``glueing" the formal scheme $\cEh$ (identified $\Ahm^1_\Z$ though $\widehat{t}$) and the complex disk of radius $R$ in its Lie algebra (identified to $C$ via $\langle \omega_\C, .\rangle$) by means of the complex uniformization of $E_\C$. 
When $N>1,$ this description is still valid over $\Spec \Z[1/N]$. 

 \subsubsection{} The Hermitian line bundle $\overline{N}_\cP \cVt_{a,R}$ is easily described. Indeed the normal bundle 
${N}_\cP \cVh_{a,R}$ is the $\Z$-module $\Z \,\partial/\partial T$ \footnote{The largest Ideal of definition $\cI$ of $\cVh_{a,R} := \Spf \Z[[T]]$ is defined by the ideal $I := T\Z[[T]]$ and ${N}_\cP \cVh_{a,R}$ by the module $(I/I^2)^\vee$. It is a free $\Z$-module with generator the ``inverse" $\partial/\partial T$ of the class $dT$ of $T$ in $I/I^2.$}. The differential 
$$Di_{a,\C}(O)= D\!\Exp_{E, \omega, \C}(0)^{-1} \circ D\,\widehat{t}(O_{E_\C})^{-1} \circ D\lambda_{N^a, \C}(0)$$
 of $i_{a,\C}$ maps the generator $\partial/\partial T$ of  ${N}_\cP \cVh_{a,R, \C}$ to the  vector $N^a \alpha\,\partial/\partial z$ in the tangent space at the origin of the disk $\overline{D}(0,R)$. Moreover 
 $$\Vert \partial/\partial z\Vert^{\rm cap}_{\overline{D}(0,R),0} = 1/R.$$
 (\cf \ref{ExampGreen}, Example (i)). These observations establish the following Lemma:

\begin{lemma}\label{NPcalcul} 
 The capacitary metric of the generator $\partial /\partial T$ of ${N}_\cP \cVh_{a,R} = N_0 \Ahm^1_\Z$ is given by
 $$\Vert \partial /\partial T \Vert^{\rm cap}_{\cVt_{a,R}} = N^a \vert \alpha \vert /R.$$
 The Arakelov degree of $\overline{N}_\cP \cVt_{a,R}$ is given by:
 \begin{equation}\label{degaNR}\dega \overline{N}_\cP \cVt_{a,R} = \log R - a \log N - \log \vert \alpha \vert.
 \end{equation}
\end{lemma}
 \qed
 
 \subsubsection{} Let us now assume that $a \geq a_0$, so that the formal morphisms 
 $$\widehat{\psi}_a: \widehat{\A}^1_{\Z} \lra \A^2_\Z \hra \PP^2_\Z \quad \mbox{and} \quad
 \widehat{\psi}'_a: \widehat{\A}^1_{\Z} \lra \A^2_\Z \hra \PP^2_\Z$$
 are defined. Since the formal-analytic surface $\cVt_{a,R}$ is defined by the ``glueing map" $i_{a,\C}$,  the relations (\ref{psicomp}) and (\ref{psiprcomp}), after extending the base field from $\Q$ to $\C$, immediately imply:

\begin{lemma}\label{psian} With the above notation, for any $R \in \R^\ast_+,$ if we let
$$\psi_a^{\rm an} := \Exp_{E_\C, \omega_\C \mid \overline{D}(0,R)} : \overline{D}(0,R) \lra E(\C) \hlra \PP^2(\C)$$
and 
 $${\psi'_a}^{\rm an} := \Exp_{E'_\C, \omega'_\C}(\beta \, .)_{\mid \overline{D}(0,R)} : \overline{D}(0,R) \lra E'(\C) \hlra \PP^2(\C),$$
then 
$\psi_a := (\widehat{\psi}_a, \psi_a^{\rm an})$ and  $\psi'_a := (\widehat{\psi}'_a, {\psi'_a}^{\rm an})$
are morphisms \emph{(as defined in \ref{MorForAnScheme})} of the formal-analytic surface $\cVt_{a,R}$ to $\PP^2_\Z$. \qed
\end{lemma}

We may consider the product morphism
$$f_a := (\psi_a, \psi'_a) : \cVt_{a,R} \lra \PP^2_\Z \times \PP^2_\Z$$
defined by 
$$\widehat{f}_a:= (\widehat{\psi}_a, \widehat{\psi}'_a): \Ahm^1_\Z \lra \PP^2_\Z \times \PP^2_\Z$$
and 
$$f^{\rm an}_a := ({\psi_a}^{\rm an},{\psi'_a}^{\rm an}): 
\overline{D}(0,R) \lra E(\C) \times E'(\C) \hlra \PP^2(\C) \times \PP^2(\C).$$

For any given $a \geq a_0,$ let us choose $R \in \R^\ast_+$ such that
$$\log R > a \log N + \log \vert \alpha \vert.$$ 
Then, according to the expression (\ref{degaNR}) for $\dega \overline{N}_\cP \cVt_{a,R},$
the 
``arithmetic pseudo-concavity" condition (\ref{NArample})
 $$
\dega \overline{N}_{\cP}\tilde{\cV} > 0.
$$
 is satisfied. Therefore, according to Theorem \ref{ArAlgPM}, the image of $f_a$ is algebraic. 
 
 Let us consider
 $C:= \overline{\im f_\Q}.$
 It is an integral closed subscheme of dimension 1 of $\PP^2_\Q \times \PP^2_\Q,$ containing $(O_E, O_{E'})$. According to Proposition  \ref{propclosureimage}, by extending the base field from $\Q$ to $\C$, it becomes
 $C_\C = \overline{\im f^{\rm an}_a},$
 the Zariski closure of
 $$\left\{ (\Exp_{E_\C, \omega_\C}(z), \Exp_{E'_\C, \omega'_\C}(\beta z)); z \in \overline{D}(0,R) \right\}.$$
 By analytic continuation, the complex curve $C_\C$ is also the Zariski closure of the map
$$(\Exp_{E_\C, \omega_\C}(z), \Exp_{E'_\C, \omega'_\C}(\beta z)): \Ga(\C) \lra E(\C) \times E'(\C),$$
which is a morphism of $\C$-analytic Lie groups. This shows that $C(\C)$ is a subgroup of $E(\C) \times E'(\C)$ and therefore that $C_\C$ is an algebraic subgroup of $(E \times E')_\C$, and thus an elliptic curve of origin $(O_{E,\C}, O_{E',\C})$).

This implies that $C$ is an elliptic curve (of origin $(O_E, O_{E'})$), and that the projections from $C$ to $E$ and $E'$ are dominant (since the projections from $C(\C)$ to $E(\C)$ and $E'(\C)$ are surjective), and finally that $E,$ $C,$ and $E'$ are isogenous.  
  
\appendix

\medskip

\chapter{Large deviations and Cram\'er's theorem}\label{Append:LD}

In this Appendix, we present some basic results in the theory of large deviations in a form suited to the application to Euclidean lattices discussed in Section \ref{asymptoticho}.

In particular, we formulate a general version of Cram\'er's theory of large deviation (Theorem \ref{theoCramer}). Recall that, 
a measurable function
$$H: \cE \lra \R$$ being given 
on some probability space $(\cE, \cT, \mu)$, this theory describes the asymptotic behavior, when the positive integer $n$ goes to infinity, of the values of the function $$H_n: \cE^n \lra \R$$
defined by 
\begin{equation}\label{Fn}
H_n (e_1, \ldots, e_n) :=  H(e_1) + \ldots + H(e_n).
\end{equation}
This description is formulated in terms of the integral 
$$\log \int_X e^{\xi H} \, d\mu,$$ 
as a function of $\xi \in \R$ with values in $]-\infty, +\infty]$, and of its Legendre-Fenchel transform.

Actually, for the application to Euclidean lattices of the theory of large deviations, we rely on some extension of Cram\'er's theorem covering the situation where $\mu$ is an arbitrary $\sigma$-finite positive measure, assuming that $H$ is non-negative.\footnote{Indeed, in Section \ref{asymptoticho}, to investigate the properties of the ``asymptotic invariants" $\hont(\Eb, t)$ attached to some Euclidean lattice $\Eb := (E, \Vert.\Vert)$, 
we consider the situation where $\cE$ is the free $\Z$-module $E$ underlying  $\Eb$, where $\mu$ is the counting measure $\sum_{e\in E} \delta_e$, and where $H$ is the squared Euclidean norm $\Vert. \Vert^2.$}

Such extensions of Cram\'er's theorem are possibly known to some experts, but for lack of  references, in Sections \ref{Cramerinfinite} and  \ref{ReforComp}, we   formulate and establish suitable versions of Cram\'er's theorem (Theorems \ref{theoCramermeasurepos} and \ref{thermo}) covering the case where $\mu(\cE)$ is possibly $+\infty$.
We achieve this by a simple reduction to the case where $\mu$ is a probability measure.

In the first two sections (\ref{NotPrem} and \ref{sub:Lanford}) of this Appendix, we also extend  various preliminary results concerning the asymptotic behavior of the values of $H_n$ on $\cE^n$ when $n$ goes to infinity to this more general setting where $\mu(\cE)$ is possibly $+\infty$. Here again, for lack of suitable references, we have included some details concerning these ``well-known" results. Then, in Section \ref{CramerTH}, we recall diverse  forms of the ``classical" Cram\'er's theorem, concerning the situation where $\mu$ is a probability measure.

Let us finally point out that the generalized version of Cram\'er's theorem presented  in this Appendix has close relations to the so-called \emph{canonical ensembles} in statistical mechanics and with the existence of thermodynamic limits. We do not discuss this in details, but we simply refer the reader to \cite{Khinchin49} and \cite{Schroedinger62} for presentations of statistical thermodynamics that emphasize the points of contact between statistical mechanics and limit theorems in probability\footnote{See also \cite{Lanford73} and \cite{Ellis85} for related mathematical results and references.}, and we indicate that the notation in Section \ref{ReforComp} has been chosen to express these relations.

The first paragraph \ref{AppendixScholium} of the final Section \ref{ReforComp} of this Appendix summarizes some of its main results in a form suitable for their applications in Section \ref{asymptoticho}. Section \ref{ReforComp} has been written to be accessible without a knowledge of the formalism previously developed in this Appendix, and could be read immediately after this introduction.

\section{Notation and preliminaries}\label{NotPrem}
\subsection{Notation} In this Appendix, we consider a measure space $(\cE, \cT,\mu)$ defined by a set $\cE$, a $\sigma$-algebra $\cT$ over $\cE$, and   a non-zero $\sigma$-finite  non-negative measure $\mu : \cT \lra [0,+\infty]$.

Besides, we consider a $\cT$-measurable function $$H : \cE \lra \R.$$
We shall denote by ${\inf}_\mu H$ (resp. ${\sup}_\mu H$) its essential infimum (resp. supremum) with respect to the measure $\mu.$

For every positive integer $n$, we may introduce the $n$-th power of the measure space $(\cE,\cT, \mu)$ --- it is defined as the product $\cE^n$ of $n$ copies of $\cE,$ equipped with the $\sigma$-algebra
$$\cT^{\otimes n} := \cT \otimes \cdots \otimes \cT  \,\,\, \mbox{($n$-times)}$$
over  $\cE^n$ and with the product $\sigma$-finite measure
$$\mu^{\otimes n} := \mu \otimes \cdots \otimes \mu  \,\,\, \mbox{($n$-times)}$$
on $\cT^{\otimes n}$ --- and we may consider the fonction $H_n: \cE^n \lra \R$ defined by (\ref{Fn}).

\subsection{Log-Laplace transform} Recall that a function $p:\R \lra ]-\infty, + \infty]$ is lower semi-continuous and convex if and only if 
$$\Gr^{\geq}(p) := \{ (x,y) \in \R  ^2 \mid y \geq p(x) \}$$ is a closed convex subset of $\R^2.$
One easily shows that this holds precisely when there exists an interval $I \subset \R$ such that the following conditions are satisfied:
\begin{enumerate}
\item $p_{\mid I}$ is real valued, continuous and convex;
\item $p_{\R \setminus I} \equiv +\infty;$
\item if $\mathring{I} \neq \emptyset$ and if $a := \inf I \notin I \cup \{-\infty \},$ then $\lim_{x \ra a_+} = + \infty;$

\noindent if $\mathring{I} \neq \emptyset$ and if $b := \sup I \notin I \cup \{+\infty \},$ then $\lim_{x \ra b_-} = + \infty.$
\end{enumerate}

If $p:\R \lra ]-\infty, + \infty]$ is non-increasing and convex, then there exists a unique $c \in [-\infty,+\infty]$ such that,  
for any $x \in \R,$
$$x < c \Longrightarrow p(x) \in \R \,\,\text{ and }\,\, x > c \Longrightarrow p(x) = +\infty,$$
and $p$ is lower semi-continuous if and only if 
$$\lim_{x \ra c_-} p(x) = p(c).$$

Similar remarks apply \emph{mutatis mutandis} to non-decreasing convex functions, and to concave functions from $\R$ to $[-\infty, + \infty[.$

 The following proposition is a straightforward consequence of the basic results of Lebesgue integration theory:

\begin{proposition}\label{defp} For every $\xi \in \R,$ the integral $\int_X \exp(\xi H)\, d\mu$ belongs to $]0, +\infty]$ and the function 
$$ \ell : \R \lra ]-\infty, +\infty]$$
defined by 
$$\ell(\xi) := \log \int_X e^{\xi H} \, d\mu$$
is lower semi-continuous and convex.

Moreover the restriction $\ell_{\mid \mathring{I}}$ of $\ell$ to the interior $\mathring{I}$ of the interval $I := \ell^{-1}(\R)$ is real analytic. Unless $H$ is constant $\mu$-almost everywhere, it satisfies $\ell''_{\mid \mathring{I}} > 0$ and defines an increasing real analytic diffeomorphism  $\ell': \mathring{I} \lrasim \ell'(\mathring{I})$ between the open intervals  $\mathring{I}$ and $\ell'(\mathring{I})$ in $\R.$ \qed
 \end{proposition}
 
 The function $\ell$ is appears in the litterature  under various names. It is sometimes called the \emph{log-Laplace transform} of the Borel measure $H_\ast \mu$ on $\R$, that is nothing but the ``law" of $H$ when $\mu$ is a probability measure. In this situation, it is also called the \emph{logarithmic moment generating function} of $H_\ast \mu$ (in
\cite{Stroock2011}, for instance).

\subsection{The functions $A^>_n$ and $A^<_n$, and $s_+$ and $s_-$}\label{AAss} To every $x \in \R,$ we may also attach the sequences $(A_{+,n}(x))_{n \geq 1}$ and $(A_{+,n}(x))_{n \geq 1}$ in $[0,+\infty]$ defined by:
 $$ A^>_{n}(x)=\mu^{\otimes n}(H_n^{-1}([nx, +\infty[)$$
 and 
 $$ A^<_{n}(x)=\mu^{\otimes n}(H_n^{-1}(\,]\!\!-\infty, nx])).$$
 
We shall be interested in situations where these sequences have an exponential asymptotic behavior, and accordingly, 
for every $x \in \R,$ we consider the following elements of $[-\infty, +\infty]$:
\begin{equation}\label{splus}
s_+(x) := \sup_{n \geq 1} \frac{1}{n} \log A^>_{n}(x) 
\end{equation}
and 
 \begin{equation}\label{smoins}
s_-(x) := \sup_{n \geq 1} \frac{1}{n} \log A^<_{n}(x).
\end{equation}

Clearly, for any positive integer $n$, the function  $ A^>_{n}: \R \lra [0,+\infty]$ is non-increasing and the function $ A^<_{n}: \R \lra [0,+\infty]$
is non-decreasing. Accordingly, the functions $s_+$ and $s_-: \R \lra [-\infty, +\infty]$ are respectively non-increasing and non-decreasing.

Observe also, that for any $x\in \R$, the following alternative holds: either (i) the function $H$ is $<x$ $\mu$-almost everywhere on $\cE$, and then $A^>_n(x) = 0$ for every positive integer $n$; or (ii) the function $H$ takes values $\geq x$ on some subset of positive $\mu$-measure, and then $A^>_n(x) > 0$ for every positive integer $n$. 

When $x={\sup}_\mu H,$ either case may occur, but we always have:
$$A^>_n({\sup}_\mu H) = \mu^{\otimes n}(H^{-1}({\sup}_\mu H)^n) = \mu(H^{-1}({\sup}_\mu H))^n,$$
and consequently:
$$s_+({\sup}_\mu H) = \log \mu(H^{-1}({\sup}_\mu H)).$$

A similar alternative holds for the sequence $(A^<_n(x))_{n \geq 1},$ and
$$s_-({\inf}_\mu H) = \log \mu(H^{-1}({\inf}_\mu H)).$$

\section{Lanford's inequalities}\label{sub:Lanford}
 In the present general framework (that does not require the finiteness of the measure $\mu$ and of the log-Laplace transform $\ell$), the following inequalities play a crucial role in the study of the asymptotic behavior of the functions $A^>_n$ and $A^<_n$ when $n$ goes to $+\infty$:

\begin{lemma}\label{Lanfordsub}
For any $(x,y)$ in $\R^2$ and any $(p,q)$ in $\N_{\geq 1}^2,$ we have:
\begin{equation}\label{subplus}
 A^>_p(x) . A^>_q(y) \leq A^>_{p+q}((px+qy)/(p+q))
\end{equation}
and
\begin{equation}\label{submoins}
 A^<_p(x) . A^<_q(y) \leq A^<_{p+q}((px+qy)/(p+q)).
\end{equation}
 \end{lemma}
In (\ref{subplus}) and (\ref{submoins}), we use the standard measure theoretical convention: $$0.(+ \infty) = (+\infty).0 =0.$$

Together with 
the  subadditivity and  concavity arguments leading to the proofs of Proposition \ref{trichotomy} and Theorem \ref{existlim} \emph{infra},
this lemma originates in Lanford's work 
\cite{Lanford73} on the rigorous derivation of ``thermodynamic limits" in statistical mechanics.

\begin{proof} Observe that the following inclusions hold between $\cT^{\otimes n}$-measurable subsets of $\cE^n$:
\begin{equation}\label{incplus}
 H_p^{-1}([px, +\infty[)\times H_{q}^{-1}([qy, +\infty[) \subset H_{p+q}^{-1}([px+qy, +\infty[) 
\end{equation}
and
\begin{equation}\label{incmoins}
H_p^{-1}(]-\infty, px]) \times H_{q}^{-1}(]-\infty, qy]) \subset H_{p+q}^{-1}(]-\infty, px+qy]).
\end{equation}

By applying the measure $\mu^{\otimes p+q}$ to both sides of (\ref{incplus}) (resp., of (\ref{incmoins})), we get the estimate  (\ref{subplus}) (resp. the estimate (\ref{submoins})).
\end{proof}
 
When $y=x,$ Lemma \ref{Lanfordsub} asserts that
\begin{equation}\label{subpluseasy}
 A^>_p(x) A^>_p(x) \leq A^>_{p+q}(x)
\end{equation}
and
\begin{equation}\label{submoinseasy}
 A^<_p(x) A^<_q(x) \leq A^<_{p+q}(x).
\end{equation}
 In other words, the sequences $(\log A^>_n(x))_{n \geq 1}$ and $(\log A^<_n(x))_{n \geq 1}$ are superadditive.
 
Actually, combined with the alternative in \ref{AAss} above concerning the vanishing or the non-vanishing of the $A^>_n(x)$ and with the  Lemma of Fekete on superadditive sequences (Lemma \ref{lemFekete}), the estimates (\ref{subpluseasy}) easily imply the following:

\begin{proposition}\label{trichotomy}
 For every $x \in \R$, precisely one of the following three conditions is satisfied:

${\bf O}^> :$ for every positive  integer $n$, $A^>_n(x) = 0$;

${\bf F}^> :$ for every positive integer $n$, $A^>_n(x) \in ]0, +\infty[;$ then the sequence 
 $((\log A^>_n(x))/n)_{n\geq 1}$ admits a limit:
\begin{equation}
\lim_{n \ra +\infty} \frac{1}{n} \log A^>_{n}(x) = \sup_{n \geq 1} \frac{1}{n} \log A^>_{n}(x) =:s_+(x) \in ]-\infty, +\infty].
\end{equation}
 
 ${\bf I}^> :$   for some positive integer $n_0,$ $A_{n_0}^>(x) = + \infty.$ Then, for every integer $n \geq n_0,$ $A_{n}^>(x) = + \infty.$
 
 A similar trichotomy holds with the sequence $(A^>_n(x))_{n\geq 1}$ replaced by $(A^<_n(x))_{n\geq 1}$, and $s_+(x)$ by $s_-(x),$ and defines conditions ${\bf O}^<$, ${\bf F}^<$ and ${\bf I}^<$.\qed
\end{proposition}

We may consider the subsets $I_{{\bf I}^>}$,  $I_{{\bf F}^>}$ and $I_{{\bf O}^>}$   of $\R$ defined by each the conditions  ${{\bf F}^>}$,  ${{\bf I}^>}$ and ${{\bf O}^>}$. Clearly they constitute disjoint consecutive intervals, and we have: 
$$\R = I_{{\bf I}^>} \cup I_{{\bf F}^>}     \cup I_{{\bf O}^>}$$
and
$$s_+^{-1}(-\infty) = I_{{\bf O}^>} = \{ x \in \R \mid \mu(H^{-1}([x, +\infty[)) = 0\}.$$

Similarly, we define disjoint consecutive intervals $I_{{\bf O}^<}$, $I_{{\bf F}^<}$ and $I_{{\bf I}^<}$ such that 
$$\R= I_{{\bf O}^<} \cup I_{{\bf F}^<} \cup I_{{\bf I}^<}$$
and we have:
$$s_-^{-1}(-\infty) = I_{{\bf O}^<} = \{ x \in \R \mid \mu(H^{-1}(]-\infty, x])) = 0\}.$$

\begin{lemma} For any two points $x$ and $y$ in  $ s_+^{-1}(]-\infty, +\infty]) = I_{{\bf I}^>} \cup I_{{\bf F}^>}$ (resp., in $ s_-^{-1}(]-\infty, +\infty]) = I_{{\bf I}^<} \cup I_{{\bf F}^<}$) and any two positive rational number $\alpha$ and $\beta$ such that $\alpha + \beta = 1,$ we have:
\begin{equation}\label{concaverat}
\alpha s_+(x) + \beta s_+(y) \leq s_+ (\alpha x + \beta y)
\end{equation}      
$$(\text{resp., } \alpha s_-(x) + \beta s_-(y) \leq s_- (\alpha x + \beta y)).$$
 \end{lemma}
 \begin{proof}
 Consider positive a positive integer $n$ such that $p:= n\alpha$ and $q:=n\beta$ are integers. Clearly $p$ and $q$ satisfy $p+q =n,$ and from (\ref{subplus}), by applying $(1/n) \log$, we derive:
 $$\alpha \frac{1}{p} \log A^>_p(x) + \beta \frac{1}{q} A^>_q(y) \leq  \frac{1}{n} \log A^>_n (\alpha x + \beta y).$$
 The estimate (\ref{concaverat}) follows by taking the limit when $n$ goes to infinity.
\end{proof}

The inequality (\ref{concaverat}) implies that, if $s_+^{-1}(+\infty)$ is not empty, then $ s_+^{-1}(]-\infty, +\infty[)$ contains at most one point. 
When one investigates the asymptotic behavior of the numbers $(A_n^>(x))_{n\geq 1}$, it is sensible to exclude this case and to assume that the following condition is satisfied:

$\mathbf{B}^> :$ \emph{ For every $x \in \R,$ $s_+(x) < +\infty.$}
 
\noindent Clearly, this condition implies that $I_{{\bf I}^>}$ is empty.

A similar discussion applies to  $s_-^{-1}(+\infty)$ and l$ s_-^{-1}(]-\infty, +\infty[)$, and leads one to introduce the condition:
 
$\mathbf{B}^< :$ \emph{ For every $x \in \R,$ $s_-(x) < +\infty.$} 

\begin{lemma}\label{BBis}
 The conditions $\mathbf{B}^>$ and $\mathbf{B}^<$ are  satisfied if $\mu(\cE) < +\infty.$
 
 More generally, if there exists $\xi$ in $\R_+$ (resp., in $\R_-$) such that $\ell(\xi) < +\infty,$ then $\mathbf{B}^>$ (resp., $\mathbf{B}^<$) is  satisfied.
 \end{lemma}

\begin{proof}
 When $\mu(\cE) < +\infty,$ $s_+(x)$ and $s_-(x)$ are  bounded from above by $\log \mu(\cE)$ for every $x \in \R.$
 
 Let us assume that $\xi$ is an element of $\R_+$ such that $p(\xi) < +\infty.$ Then, for every positive integer $n$, we have:
 $$\int_{\cE^n} e^{\xi H_n} \, d\mu^{\otimes n} = \left(\int_\cE e^{\xi H} \, d\mu \right)^n = e^{n p(\xi)} < +\infty.$$
 Besides, for every $x \in \R$ and every $P \in H_n^{-1}([nx, +\infty[),$
 $$e^{\xi H_n(P)} \geq e^{n \xi x}.$$
 Therefore
 $$\mu^{\otimes n}(H_n^{-1}([nx, +\infty[) e^{n \xi x} \leq e^{n p(\xi)}$$
 and
 $$\frac{1}{n}\log A^>_n(x) \leq p(\xi) -\xi x.$$
 This shows that
 $$s_+(x) \leq p(\xi) -\xi x < + \infty.$$
 
 The proof of $\mathbf{B}^<$ when there exists $\xi$ in $\R_-$ such that $p(\xi) < +\infty$ is similar.
\end{proof}
 
 The following theorem summarizes and completes some of the results  obtained so far concerning the asymptotic behavior of the functions 
 $A^>_n$ and $A^<_n$:

\begin{theorem}\label{existlim}
 1) If Condition $\mathbf{B}^>$ is satisfied, then, for every $x \in \R$, 
 either :
 \begin{itemize}
\item $\mu^{\otimes n}(H_n^{-1}([nx, +\infty[))$ $=0$ for every positive integer 
   $n$ and $s_+(x) = -\infty$, or
   \item  the sequence $(\log \mu^{\otimes n}(H_n^{-1}([nx, +\infty[)))_{n\geq 1}$ lies in $\R$, is superadditive, and satisfies:
   \begin{equation}\label{existlimplus}\lim_{n \ra + \infty} \frac{1}{n} \log \mu^{\otimes n}(H_n^{-1}([nx, +\infty[)) = s_+(x) \in \R.
\end{equation}
\end{itemize}

 The function $s_+ : \R \lra [-\infty, +\infty[$  is   non-increasing and concave.
 
  Moreover, for any $x \in \R,$
$$x < {\sup}_\mu H \Longrightarrow s_+(x) \in \R \,\,\text{ and }\,\, x >{\sup}_\mu H \Longrightarrow s_+(x)= -\infty, $$
 and $$s_+({\sup}_\mu H) = \log \mu(H^{-1}({\sup}_\mu H)).$$
 
 2) Symetrically, if Condition $\mathbf{B}^<$ is satisfied, then, for every $x \in \R$, either 
 \begin{itemize}
\item $\mu^{\otimes n}(H_n^{-1}(]-\infty, nx])) =0$ for every positive integer $n$ and $s_-(x) = -\infty$, or 
\item the sequence $(\log \mu^{\otimes n}(H_n^{-1}(]-\infty, nx])))_{n\geq 1}$ lies in $\R$, is superadditive, and  satisfies:
   \begin{equation}\label{existlimmoins}
   \lim_{n \ra + \infty} \frac{1}{n} \log \mu^{\otimes n}(H_n^{-1}(]-\infty,nx])) = s_-(x) \in \R.
\end{equation}
\end{itemize}
 
 The function $s_- : \R \lra [-\infty, +\infty[$  is   non-decreasing and concave. 
 
 Moreover, for any $x \in \R,$
$$x < {\inf}_\mu H \Longrightarrow s_-(x) = -\infty \,\,\text{ and }\,\, x >{\inf}_\mu H \Longrightarrow s_-(x) \in \R, $$
and $$s_-({\inf}_\mu H) = \log \mu(H^{-1}({\inf}_\mu H)).$$
\end{theorem}

\begin{proof}
 At this stage, to complete the proof of 1), we simply need to observe that the inequality (\ref{concaverat}) holds  for any two points $x$ and $y$ in $s_+^{-1}(]-\infty, +\infty[)$, not only when the coefficient $(\alpha, \beta)$ are positive rational number such that $\alpha + \beta =1$, but more generally for any two positive real numbers $(\alpha, \beta)$ such  $\alpha + \beta =1$ : as $s_+$ is non-increasing, this follows from a straightforward approximation argument. This establishes  the concavity of $s_+$.
 
 The proof of 2) is similar, or follows from 1) applied to $-H$ instead of $H$. 
\end{proof}
 
\section{Cram\'er's theorem}\label{CramerTH}

In this section, we assume that \emph{the measure $\mu$ is a probability measure}. Observe that it implies that
$$\ell(0) = 0.$$
In particular, the interval $I:=\ell^{-1}(\R)$ is non-empty and contains $0$. Moreover, the conditions $\mathbf{B}^>$ and $\mathbf{B}^<$ are  satisfied.

\subsection{}The following theorem is the formulation, in measure theoretic language, of a general version of Cramer's theorem concerning the ``empirical means" $$\overline{X}_n := \frac{1}{n} \sum_{i=1}^n X_i$$
attached to a sequence $(X_i)_{i\geq 1}$ of independent and identically distributed real-valued random variables. 
We refer the reader to the article \cite{CerfPetit2011} of Cerf and Petit for  a short and elegant proof.\footnote{The function $s$ in  \emph{loc. cit.} corresponds to the function $s_+$ defined above. The assertions concerning $s_-$ in Theorem \ref{theoCramer} follow from the ones concerning $s_+$ applied to the function  $-H$ instead of $H$.}

\begin{theorem}\label{theoCramer} 
For every $x \in \R,$ we have:
\begin{equation}\label{Cramersplus}
s_+(x) = \inf_{\xi \in \R_+} (\ell(\xi) - \xi x)
\end{equation}
and 
\begin{equation}\label{Cramersmoins}
s_-(x) = \inf_{\xi \in \R_-} (\ell(\xi) - \xi x).
\end{equation}

Moreover, for any $\xi \in \R_+$ (resp., for any $\xi \in \R_-$), we have:
\begin{equation}\label{Cramerpplus}
 \ell(\xi) = \sup_{x \in \R} (\xi x + s_+(x))
 \end{equation}
 \begin{equation}\label{Cramerpmoins}
\mbox{\emph{(}resp.  $\ell(\xi) = \sup_{x \in \R} (\xi x + s_-(x))$\emph{)}}. \qedhere
 \end{equation}
\end{theorem}

Observe also that
functions $s_+$ and $s_-$ take their values in $[-\infty, 0]$. Moreover, according to (\ref{Cramerpplus}) and (\ref{Cramerpmoins}), they are upper semi-continuous and satisfy:
\begin{equation}\label{ApA} \lim_{x \ra -\infty} s_+(x) = 0, 
\end{equation} 
$$\lim_{x \ra ({\sup}_\mu H)_-} s_+(x) = s_+({\sup}_\mu H) = \log \mu(H^{-1}({\sup}_\mu H))$$
and 
$$\lim_{x \ra ({\inf}_\mu H)_+} s_-(x) = s_-({\inf}_\mu H) = \log \mu(H^{-1}({\inf}_\mu H)),$$
\begin{equation}\label{ApB}
\lim_{x \ra +\infty} s_-(x) = 0.
\end{equation}

Under the terminology ``Cram\'er's Theorem", one usually means the conjunction of the existence of the limits (\ref{existlimplus}) and (\ref{existlimmoins}) for every $x\in \R$ that is established  in Theorem \ref{existlim}, together with the expressions (\ref{Cramersplus}) and (\ref{Cramersmoins}) of these limits in terms of the log-Laplace transform $\ell$ stated in Theorem \ref{theoCramer}. 

   The reader will also refer to \cite{CerfPetit2011} for additional references about large deviations and Cram\'er's theorem,  notably   concerning  the successive contributions, starting from Cram\'er's seminal article (\cite{Cramer38}), which have led to the present simple and general formulation of Cram\'er's theorem.  
  
Let us also indicate that Chernoff's version of Cram\'er's theorem (\cite{Chernoff52}, Theorem 1) would be enough to derive Theorem \ref{theoCramermeasurepos} below.

\subsection{} The  version  of Cram\'er's theorem formed by Theorems \ref{existlim} and  \ref{theoCramer}, when $\mu$ is a probability measure,  may be supplemented by the following observations, intended to clarify its relation with other presentations of Cram\'er's theory of large deviations (see for instance
\cite{Stroock2011}, Section 1.3).

Let us define  $m_+ \in [-\infty, +\infty]$ and $m_- \in [-\infty, +\infty]$ as the ``right and left derivatives"
\begin{equation*}\label{mplus}
m_+ := \inf_{\xi \in \R_+^\ast} \ell(\xi)/\xi = \lim_{\xi \ra 0_+} \ell(\xi)/\xi =: \ell'_r (0).
\end{equation*}
and
\begin{equation*}\label{mmoins}
m_- := \sup_{\xi \in \R_-^\ast} \ell(\xi)/\xi = \lim_{\xi \ra 0_-} \ell(\xi)/\xi =: \ell'_l (0).
\end{equation*}
of the lower semi-continuous convex function $\ell$ at $0$.

 We clearly have: 
$$m_+ \leq m_-$$
and,  from Theorem \ref{theoCramer}, we easily get:

\begin{corollary}\label{corCramer} For any $x\in \R,$ we have:
\begin{equation}\label{ApC}
s_+(x) = 0 \Longleftrightarrow x \leq m_+
\end{equation}
and 
\begin{equation}\label{ApD}
 s_-(x) = 0 \Longleftrightarrow x \geq m_-.
 \end{equation}
Moreover the function
\begin{equation}\label{ApE}
s := \min (s_-,s_+) : \R \lra [-\infty, +\infty]
\end{equation}
is upper semi-continuous and concave, and the functions $\ell$ and $-s$ may be deduced from each other by Legendre-Fenchel duality:
\begin{equation*}\label{Cramers}
\text{ for every $x \in \R,$ } s(x) = \inf_{\xi\in\R} (\ell(\xi) - \xi x)
\end{equation*}
and 
\begin{equation*}\label{Cramerp}
\text{ for every $\xi \in \R,$ } \ell(\xi) = \sup_{x \in \R} (\xi x + s(x)).
\end{equation*}
\qed 
\end{corollary}

Observe finally that, if we denote by $H^+$ and $H^-$ the positive and negative parts\footnote{namely the $\R_+$-valued functions on $\cE$ defined by $H^+ := \max (H,0)$ and $H^-:= \max(-H,0).$}
  of $H$, then the following conditions are equivalent:
\begin{enumerate}
\item $m_+ < +\infty,$
\item $I \cap \R_+^\ast \neq \emptyset,$
\item for some $\epsilon \in \R^\ast_+,$ the function $\exp(\epsilon H_+)$ is $\mu$-integrable,
\end{enumerate}
and the following ones as well:
\begin{enumerate}
\item $m_- > -\infty,$
\item $I \cap \R_-^\ast \neq \emptyset,$
\item for some $\epsilon \in \R^\ast_+,$ the function $\exp(\epsilon H_-)$ is $\mu$-integrable.
\end{enumerate}
In particular, if $m_+ < +\infty$ and $m_- > -\infty$, then $I$ contains some neighborhood of $0$ in $\R,$ the function $H$ is $\mu$-integrable, and
$$m_+ = m_-= \ell'(0) = \int_E H \, d\mu.$$

More generally, if  $m_+ < +\infty$ (resp., $m_- > -\infty$), the  integral
$$\int_E H \, d\mu := \int_E H^+ \, d\mu - \int_E H^- \, d\mu$$
is well-defined in $[-\infty, +\infty [$ (resp., in $]-\infty, +\infty]$) and is easily seen to be equal to $m_+$ (resp., to $m_-$).

\section{An extension of Cram\'er's Theorem concerning general positive measures}\label{Cramerinfinite}

In this section, the measure $\mu$ is allowed to have a total mass $\mu(E)$ different of $1$. 

When $\mu(\cE)$ is finite, one may apply the results in the previous section with $\mu$ replaced by the probability measure
$$\mu_0 := \mu(\cE)^{-1} \mu.$$
In this way, one easily sees that Theorem  \ref{theoCramer}  still holds \emph{ne varietur}, and that its consequences remain valid with minor modifications (in \eqref{ApA}-\eqref{ApB} and in \eqref{ApC}-\eqref{ApE}, $0$ has to be replaced by $\log \mu(\cE)$).

When the total mass $\mu(\cE)$ is $+\infty,$ then for any $x \in \R,$
$$A_1^>(x) + A_1^<(x)= +\infty$$
and therefore $s_+(x) = +\infty$ or $s_-(x) = + \infty.$ Therefore the conditions ${\bf B}^>$ and ${\bf B}^<$ cannot be simultaneously satisfied, and only one of the two functions $s_+$ and $s_-$ may have an interesting behavior. 

In this section, we shall focus on $s_-$, and we will show that the results in Theorem \ref{theoCramer} concerning $s_-$  admit a sensible generalization when $H$ is bounded from below --- say when ${\inf}_\mu H$ is non-negative.\footnote{The interested reader will have no difficulty in extending the results of this section to the more general situation where ${\inf}_\mu H > -\infty.$}

\subsection{} In this paragraph, we assume that \emph{the interval $I:= \ell^{-1}(\R)$ is not empty}, and we choose some element $\xi_0$ in $I$.

Then the integral
$$\int_\cE e^{\xi_0 H} d\mu = e^{\ell(\xi_0)}$$
belongs to $]0,+\infty[$ and the measure $$\mu_0 := \left( \int_\cE e^{\xi_0 H} d\mu \right)^{-1}  e^{\xi_0 H }\, \mu$$
is a probability measure on the measurable space $(\cE,\cT)$. Therefore, we may apply the constructions and the results of the previous sections to the probability space $(\cE,\cT, \mu_0)$ and to the function $H$. 

Let us denote by $\ell_0$, $s_{0 +},$ and $s_{0 -}$ the functions $\ell$, $s_+$, and $s_-$ associated to these new data. The following lemma is a straightforward consequence of the definitions:

\begin{lemma}\label{mumuo}
 For every non-negative integer $n,$ and every $(\xi, x)$ in $\R^2,$ we have:
 \begin{equation}\label{L1}
\mu_0^{\otimes n} = e^{\xi_0 H_n - n p(\xi_0)}\; \mu^{\otimes n},
\end{equation}
 \begin{equation}\label{L2}
\ell_0(\xi) = \ell(\xi + \xi_0) - \ell(\xi_0),
\end{equation}
and
\begin{equation}\label{L3}
\frac{1}{n} \log \mu_{0}^{\otimes n} (H_n^{-1}(]-\infty, nx ])) = \frac{1}{n} \log \int_{H_n^{-1}(]-\infty, nx ])} e^{\xi_0 H_n} \, d\mu^{\otimes n} - \ell(\xi_0). 
\end{equation}
\qed
\end{lemma}

By means of the relation (\ref{L1})-(\ref{L3}),  the assertions involving $s_{0-}$ and $p_0$ in  Theorems  \ref{existlim} and   \ref{theoCramer} applied to  the probability space $(E,\cT, \mu_0)$ and to the function $H$ may be reformulated as follows, without reference to the measures $\mu_0^{\otimes n}$:

\begin{corollary}\label{corxio}
Let $\xi_0$ be a real number such that $\ell(\xi_0) < + \infty.$

For every $x$ in $\R$, we have:
\begin{equation}\label{L4}
 \lim_{n \ra + \infty} \frac{1}{n} \log \int_{H_n^{-1}(]-\infty, nx ])} e^{\xi_0 H_n} 
  \, d\mu^{\otimes n} 
 = \inf_{\xi \in \R_-} (\ell(\xi + \xi_0) - \xi x). 
\end{equation}

Moreover, for every $\xi \in \R_-,$ 
\begin{equation}\label{L5}
\ell(\xi + \xi_0) = \sup_{x \in \R} \left(\xi x +  \lim_{n \ra + \infty} \frac{1}{n} \log \int_{H_n^{-1}(]-\infty, nx ])} e^{\xi_0 H_n} \, d\mu^{\otimes n}\right). 
\end{equation} \qed
\end{corollary}

Observe that every term of sequence in the left-hand side of \eqref{L4}, and consequently its limit, belongs to $[-\infty, \ell(\xi_0)]$. 
\subsection{An application} In this paragraph, we  apply the previous results to the situation where 
$H$ assumes only non-negative values on $\cE$.

We shall use the non-negativity of $H$ through the following simple observation:

\begin{lemma} Let us assume that $H(\cE) \subset \R_+$, and let us consider a positive integer $n$ and some elements $x$ of $\R_+$ and $\eta$ of $\R^\ast_+.$

Then $H_n^{-1} (]-\infty, nx])$ is empty if $x<0$ and equals  $H_n^{-1}([0,nx])$ if $x\geq 0.$ Therefore, every $P \in H_n^{-1} (]-\infty, nx])$ satisfies:
$$e^{-\eta n x} \leq e^{-\eta H_n(P)} \leq 1,$$
and consequently:
\begin{equation}\label{L6}
\frac{1}{n} \log  A_n^<(x) - \eta x
\leq
\frac{1}{n} \log \int_{H_n^{-1}(]-\infty, nx ])} e^{-\eta H_n} \, d\mu^{\otimes n} 
\leq \frac{1}{n} \log  A_n^<(x). 
\end{equation}\qed
\end{lemma}

We may now formulate the main result of this Appendix:

\begin{theorem}\label{theoCramermeasurepos} Let us assume that the following conditions are satisfied: 
\begin{equation*}%
H(\cE) \subset \R_+
\end{equation*}
and 
\begin{equation}\label{ellfiniteonRminus}
\text{ for any $\xi \in \R^\ast_-,$} \int_\cE e^{\xi H}\, d\mu < + \infty.
\end{equation}

Then, for every $x$ in $\R$,  we have:
 \begin{equation}\label{equCramerinfinites}
s_-(x)  = \inf_{\xi \in \R^\ast_-} (p(\xi) - \xi x).
\end{equation}
In particular, Condition $\mathbf{B}^<$ is satisfied and  the concave function $s_-: \R \lra [-\infty , +\infty[$ is upper-semicontinuous.

Moreover, for every $\xi \in \R^\ast_-,$ we have: 
\begin{equation}\label{equCramerinfinitep}
\ell(\xi) = \sup_{x \in \R} (\xi x +  s_-(x)).
\end{equation}
\end{theorem}

\begin{proof} Let us consider  $x$ in $\R$ and $\eta$ in $\R^\ast_+.$

We may apply Corollary \ref{corxio} with $\xi_0 = -\eta,$ and observe that, when $n$ goes to $+\infty,$ the inequalities  
\eqref{L6} become
 \begin{equation}\label{L7}
s_-(x) - \eta x \leq  \lim_{n \ra + \infty} \frac{1}{n} \log \int_{H_n^{-1}(]-\infty, nx ])} e^{-\eta H_n} 
  \, d\mu^{\otimes n} 
 \leq s_-(x).
\end{equation}

According to \eqref{L4}, the middle term in \eqref{L7} 
is
$$ \inf_{\xi \in \R_-} (\ell(\xi - \eta) - \xi x) = \inf_{\xi \in ]-\infty, - \eta]} (\ell(\xi) - \xi x) - \eta x,$$
so that \eqref{L7} may be also be written:
$$s_-(x) - \eta x  \leq \inf_{\xi \in ]-\infty, - \eta]} (\ell(\xi) - \xi x) - \eta x \leq s_-(x).
$$
By taking the limit of these inequalities when $\eta$ goes to $0+$, we obtain \eqref{equCramerinfinites}.

Let $\xi$ be an element of $\R_-^\ast.$ The relation \eqref{L5} for $\xi_0:= -\eta,$ together with \eqref{L7}, shows that:
\begin{equation*}
\sup_{x \in \R} (\xi x + s_-(x) -\eta x) \leq \ell(\xi - \eta) \leq \sup_{x\in \R} (\xi x + s_-(x)).
\end{equation*}
Since $\ell$ is continuous on $\R_-^\ast,$ \eqref{equCramerinfinitep} follows.
 \end{proof}

\section{Reformulation and complements}\label{ReforComp}
In this section, for the convenience of the reader, we reformulate some of the results previously established in this Appendix without making reference to the formalism introduced in the previous sections. 

To emphasize the possible thermodynamical interpretation of these results, we will introduce  some new notation, related to the previous one through the following formulas:
$$\Psi(\beta) = \ell (-\beta) \;\; \mbox{ and } \;\; S(x) = s_-(x).$$
 
\subsection{A scholium}\label{AppendixScholium} Let us recall that we consider a measure space $(\cE, \cT,\mu)$ defined by a set $\cE$, a $\sigma$-algebra $\cT$ over $\cE$, and   a non-zero $\sigma$-finite  non-negative measure $\mu : \cT \lra [0,+\infty]$. For every positive integer $n,$ we  also consider the product $\mu^{\otimes n}$ of $n$ copies of the measure $\mu$ on $\cE^n$ equipped with the $\sigma$-algebra $\cT^{\otimes n}$.

Besides, we consider a $\cT$-measurable function $$H : \cE \lra \R_+.$$

For every $\beta \in \R_+^\ast$ such that $e^{-\beta H}$ is $\mu$-integrable, the integral $\int_\cE e^{-\beta H} d\mu$ is a positive real number, and we let:
\begin{equation}\label{Psidef}
\Psi(\beta) := \log \int_\cE e^{-\beta H} d\mu \;\; (\in \R).
\end{equation}

Under this integrability assumption, we also define:
$$F(\beta) := -\beta^{-1} \Psi(\beta).$$
Equivalently, $F(\beta)$ may be defined by the relation:
$$e^{-\beta F(\beta)} = \int_\cE e^{-\beta H} d\mu.$$

The $\mu$-integrability of the function $e^{-\beta H}$ for every $\beta \in \R^\ast_+$ is easily seen to be equivalent to the  {\bf s}ub-{\bf e}xponential growth of the function 
$$N: \R_+ \lra [0, +\infty]$$
defined by $$N(x) := \mu (H^{-1}[0,x]),$$ 
namely to the condition:

{\bf SE} \emph{For every $x \in \R_+,$ $N(x)$ is finite, and, when $x$ goes to} $+\infty,$ 
$$\log N(x) = o(x).$$

We shall also consider   the essential infimum $${\inf}_\mu H :=\inf \{ x \in \R_+ \mid N(x) > 0 \}$$ of $H$ with respect to the measure space $(\cE,\cT, \mu)$.

\begin{theorem}\label{thermo}
Let us keep the above notation and let us assume that Condition {\bf SE} is satisfied and that $\mu( \cE) = + \infty.$ 

1) For every $x$ in $]{\inf}_\mu H, +\infty[,$ the limit 
\begin{equation*}
S(x) := \lim_{n\ra + \infty} \frac{1}{n} \log \mu^{\otimes n}\left(\left\{ (e_1, \ldots, e_n) \in \cE^n \mid H(e_1) + \ldots + H(e_n) \leq n x \right\}\right)
\end{equation*}
exists in $\R$.

The function $S: \, ]{\inf}_\mu H, +\infty[ \lra \R$ is non-decreasing and concave, and satisfies
\begin{equation}\label{pourNernst}
\lim_{ x \ra  ({\inf}_\mu H)_+} S(x) = \log \mu(H^{-1}(\inf_\mu H)) \;\; (\in [-\infty, + \infty[).
\end{equation}

2)  The function $\Psi: \R_+^\ast \lra \R$ is real analytic and convex. Its derivative up to a sign 
$$U:=-\Psi'$$
satisfies, for every $\beta \in \R^\ast_+,$
\begin{equation}\label{UH}
U(\beta) = \frac{\int_\cE H \, e^{-\beta H} \, d\mu}{\int_\cE e^{-\beta H} \, d\mu}
\end{equation}
and defines a decreasing real analytic diffeomorphism:
\begin{equation}\label{Udiff}
U: \R_+^\ast \lrasim \;]{\inf}_\mu H, +\infty[.
\end{equation}

3) The functions $-S(- .)$ and $\Psi$ are Legendre-Fenchel transforms of each other. 

Namely, for every $x \in ]{\inf}_\mu H, +\infty[,$
\begin{equation}\label{LegendreS}
S(x) = \inf_{\beta \in \R^\ast_+} (\beta x + \Psi(\beta)),
\end{equation}
and, for every $\beta \in \R^\ast_+,$
\begin{equation}\label{LegendrePsi}
\Psi(\beta) = \sup_{x \in ]{\inf}_\mu H, +\infty[} (S(x) -\beta x).
\end{equation}
\end{theorem}
\begin{proof} Assertion 1) follows from Theorem \ref{theoCramermeasurepos} and from Theorem \ref{existlim}, 2).

Observe that $H$ is not $\mu$-almost everywhere constant --- otherwise the conditions {\bf SE} and $\mu(\cE) = +\infty$ could not be both satisfied. Therefore  $\Psi '' >0$ on $\R_+^\ast$  by Proposition \ref{defp}.  

The expression (\ref{UH}) \emph{à la Gibbs}   for $U:= -\Psi'$ is a straightforward consequence of Lebesgue integration theory. To complete the proof of 2), we are thus left to show that
\begin{equation}\label{U0}
\lim_{\beta \ra 0_+} U(\beta) = +\infty
\end{equation}
and
\begin{equation}\label{Uinfty}
\lim_{\beta \ra + \infty} U(\beta) = {\inf}_\mu H.
\end{equation}

Since the function $\Psi$ satisfies 
$$\lim_{\beta \ra 0+} \Psi(\beta) = \int_\cE d\mu = +\infty,$$
its derivative cannot stay bounded near zero. This proves (\ref{U0}). 

To prove (\ref{Uinfty}), simply observe that, according to (\ref{UH}), for any $\beta \in \R^\ast_+,$
$$U(\beta) \geq {\inf}_\mu H$$
and that, for any $\eta > {\inf}_\mu H$ and for any $\beta \in \R^\ast_+,$
\begin{equation*}
\begin{split}
 \Psi(\beta) & \geq \log \int_\cE e^{-\beta H} \, d\mu \\
 & \geq \log \int_{H^{-1}([0,\eta])}  e^{-\beta \eta} \, d\mu \\
 & \geq \log N(\eta)  - \beta \eta,
\end{split}
\end{equation*}
 so that $\lim_{\beta \ra +\infty} \Psi'(\beta) \geq -\eta.$
 
 Assertion 3) directly follows from Theorem \ref{theoCramermeasurepos}.
\end{proof}
 
 The following corollary is now a consequence of the elementary theory of Legendre-Fenchel transforms of convex functions of one real variable\footnote{Basically it follows from the fact that the Legendre-Fenchel transform $g$ of real analytic function $f$ with positive second derivative is itself real analytic with positive second derivative, as it may be expressed by the classical Legendre duality relation : $g(\xi) + f(x) = x \xi$ where $\xi = f'(x)$ or, equivalently, $x=g'(\xi)$.}:  
 
\begin{corollary}\label{gentildual}
 The function $S$ is increasing and real analytic on $]{\inf}_\mu H, +\infty[$. Moreover its derivative establishes a decreasing real analytic diffeomorphism
 $$S': \;]{\inf}_\mu H, +\infty[ \lrasim \R^\ast_+$$
 inverse of the diffeomorphism (\ref{Udiff}).
 
 For any $x \in \;]{\inf}_\mu H, +\infty[$, the infimum in the right-hand side of (\ref{LegendreS}) is attained  at a unique $\beta \in \R^\ast_+$, namely 
 \begin{equation}\label{betax}
\beta = S'(x).
\end{equation}

Dually, for any $\beta \in \R^\ast_+$, the supremum in the right-hand side of (\ref{LegendrePsi}) is attained  at a unique any $x \in \;]{\inf}_\mu H, +\infty[$, namely
\begin{equation}\label{xbeta}
x = U(\beta).
\end{equation}
\end{corollary}

Observe that, when $x \in \;]{\inf}_\mu H, +\infty[$ and $\beta \in \R^\ast_+$ satisfy the equivalent relations (\ref{betax}) and (\ref{xbeta}), then
\begin{equation}\label{SPsi}
S(x) = \beta x + \Psi(\beta),
\end{equation}
or equivalently:
$$F(\beta) := -\beta^{-1} \Psi(\beta) = U(\beta) -\beta^{-1} S(x).$$
(``Expression of the free energy $F$ in terms of the energy $U$, the temperature $\beta^{-1}$, and the entropy $S$").

Observe also that, according to (\ref{pourNernst}),  the following two conditions are equivalent:
\begin{equation*}
\mu(\{e \in \cE \mid H(e) = {\inf}_\mu H \}) = 1
\end{equation*}
and 
\begin{equation*}
\lim_{\beta \ra +\infty} S(U^{-1}(\beta)) = 0
\end{equation*}
(``Nernst's principle").

\subsection{Products and thermal equilibrium}\label{prodtherm}

The formalism summarized in the previous paragraph --- that attaches functions $\Psi$ and  $S$  to a measure space  $(\cE, \cT,\mu)$ and to a non-negative function $H$ on $\cE$ satisfying {\bf SE} --- satisfies a simple but remarkable compatibility with finite products, that we want to discuss briefly.

Assume that, for any element $i$ in some finite set $I$, we are given a measure space $(\cE_i, \cT_i, \mu_i)$ and a measurable function $H_i: \cE_i \lra \R_+$ as in the previous paragraph \ref{AppendixScholium} above.

Then we may form the product measure space $(\cE, \cT, \mu)$ defined by the set
$\cE := \prod_{i \in I} \cE_i$ equipped with the $\sigma$-algebra $\cT := \bigotimes_{i \in I} \cT_i$ and the product measure $\mu := \bigotimes_{i \in I} \mu_i$. 

We may also define a measurable function
$$H: \cE \lra \R_+$$
by the formula
$$H := \sum_{i \in I} {\rm pr}_i^\ast H_i,$$
where ${\rm pr}_i: \cE \lra \cE_i$ denotes the projection on the $i$-th factor.

Let us assume that, for every $i\in I$,  $(\cE_i, \cT_i, \mu_i)$ and $H_i$ satisfy the condition {\bf SE}, or equivalently that the functions $e^{-\beta H_i}$ is $\mu_i$-integrable for every $\beta \in\R^\ast_+.$ 

Then $(\cE, \cT, \mu)$ and $H$ are easily seen to satisfy {\bf SE} also, as a consequence of Fubini's Theorem.
Actually Fubini's Theorem shows that the function $\Psi: \R^\ast_+ \lra \R$ attached to the above data, defined as in \ref{AppendixScholium} by the formula 
$$\Psi(\beta) := \log \int_\cE e^{-\beta H} d\mu,$$
and the ``partial functions" $\Psi_i$, $i \in I,$ attached to each measure space $(\cE_i, \cT_i, \mu_i)$ equipped with the function $H_i$ by the similar formula
$$\Psi_i(\beta) := \log \int_{\cE_i} e^{-\beta H_i} d\mu_i,$$
satisfy the additivity relation:
\begin{equation}\label{addPsi}
\Psi = \sum_{i \in I} \Psi_i.
\end{equation}

Let us also assume that $\mu_i(\cE_i) = +\infty$ for every $i\in I.$ Then we also have $\mu(\cE) = +\infty$, and we may apply Theorem \ref{thermo}  and Corollary \ref{gentildual} to the data $(\cE_i, \cT_i, \mu_i, H_i)$, $i \in I,$ and $(\cE, \cT, \mu, H)$.

Notably, we may define some concave functions
$$S_i: \, ]{\inf}_{\mu_i} H_i, +\infty[ \lra \R, \;\;\; \mbox{ for $i \in I$,}$$
and 
$$S: \, ]{\inf}_\mu H, +\infty[ \lra \R.$$

Observe also that, as a straightforward consequence of the definitions, we have:
$${\inf}_\mu H = \sum_{i \in I} {\inf}_{\mu_i} H_i.$$

The following proposition may be seen as a mathematical interpretation of the second law of thermodynamics:

\begin{proposition}\label{Prop:secondlaw} 1) For each $i \in I,$ let $x_i$ be a real number in $]{\inf}_{\mu_i} H_i, + \infty [.$

Then the following inequality is satisfied:
 \begin{equation}\label{Secondlaw}
\sum_{i \in I} S_i(x_i) \leq S(\sum_{i \in I} x_i).
\end{equation}
 
 Moreover equality holds in (\ref{Secondlaw}) if and only if the positive real numbers $S'(x_i),$ $i \in I,$ are all equal. When this holds, if $\beta$ denotes their common value, we also have: 
 $$\beta = S'(\sum_{i \in I} x_i).$$ 
 
 2) Conversely, for any $x \in ] {\inf}_\mu H, +\infty[,$ there exists a unique family $(x_i)_{i\in I} \in \prod_{i \in I} ]{\inf}_{\mu_i} H_i, + \infty [$ such that $$x= \sum_{i \in I} x_i \;\; {\mbox  and }\;\; S(x) = \sum_{i \in I} S_i(x_i).$$
 Indeed, if $\beta = S'(x),$ it is given by $$(x_i)_{i \in I} = (U_i(\beta))_{i \in I},$$
where $U_i = -\Psi'_i.$ \end{proposition}

\begin{proof} Let $(x_i)_{i\in I}$ be an element of  $\prod_{i \in I} ]{\inf}_{\mu_i} H_i, + \infty [$  According to Corollary \ref{gentildual}, for every $i \in I,$ 
\begin{equation}
S(x_i) = \inf _{\beta > 0} (\beta x_i + \Psi_i(\beta)).
\end{equation}
Moreover, the infimum is attained for a unique positive $\beta,$ namely $S'(x_i)$.

Similarly, for $x := \sum_{i \in I} x_i,$
\begin{equation}
S(x) = \inf _{\beta > 0} (\beta x + \Psi(\beta)),
\end{equation}
and the infimum is attained for a unique positive $\beta$, namely $S'(x)$.

Besides, the additivity relation (\ref{addPsi}) shows that, for every $\beta$ in $\R^\ast_+,$
$$\beta x + \Psi(\beta) = \sum_{i \in I} \left(\beta x_i + \Psi_i(\beta)\right).$$

Part 1) of the proposition  
directly follows from these observations.
Part 2) follows from Part 1) and from the relation $\Psi'= \sum_{i \in I} \Psi'_i.$
 \end{proof}

 \subsection{An example: Gaussian integrals and Maxwell velocity distribution}\label{Maxwell}
 
 In this paragraph, we discuss a simple but significant instance of the formalism summarized in paragraph \ref{AppendixScholium} and in Theorem \ref{thermo}. This example may be seen as a mathematical counterpart of Maxwell's statistical approach to the theory of ideal gases. It is also included for comparison with the application in Section \ref{asymptoticho} of the above formalism to Euclidean lattices --- the present example appears as a ``classical limit" of the discussion of Section \ref{asymptoticho}.
 
 Let $V$ be a finite dimensional real vector space equipped with some Euclidean norm $\Vert. \Vert.$
 
 We shall denote by $\lambda$ the Lebesgue measur on $V$ attached to this Euclidean norm. It may be defined as the unique translation invariant Radon measure on $V$ such that
 \begin{equation*}
\int_V  e^{-\pi \Vert x \Vert^2} \, d\lambda(x) = 1.
\end{equation*}
(Compare \ref{PoissonSchwartz} and equation (\ref{gaussint}).)

We may apply the formalism of this appendix to the measure space $(V, \cB, \lambda)$, defined by $V$ equipped with the Borel $\sigma$-algebra $\cB$ and with the Lebesgue measure $\lambda$, and to the function
$$H := (1/2m) \Vert. \Vert^2$$
where $m$ denotes some positive real number.

Then, for every $\beta$ in $\R^\ast_+,$ we have:
$$\int_v e^{- \beta \Vert p \Vert^2/2m} \, d\lambda(p) = (2 \pi m/\beta)^{\dim V / 2}.$$
Therefore
\begin{equation}\label{PsiMaxw}\Psi(\beta) =  (\dim V/2)  \, \log (2 \pi m/\beta)
\end{equation}
and 
\begin{equation}\label{UMaxw}
U(\beta) = -\Psi'(\beta) = \dim V/(2\beta).
\end{equation}
 
The equations $(\ref{betax})$ and $(\ref{xbeta})$, that relates the ``energy"
$x$ and the ``inverse temperature" $\beta$, takes the following form, for any $x\in \,]{\inf}_\lambda H, +\infty[ = \R^\ast_+$ and any $\beta \in  \R^\ast_+$:
\begin{equation}\label{betaxMaxw}
\beta x = \dim V/2.
\end{equation}

The function $S(x)$ may be computed directly from its definition. 

Indeed, for any $x \in \R^\ast_+$ and any positive integer $n$, we have:
\begin{equation}\label{S1}
\lambda^{\otimes n} \left(\{(e_1, \ldots, e_n) \in V^n \mid (1/2m)(\Vert e_1\Vert^2+\ldots+\Vert e_n\Vert^2) \leq n x \}\right)
= v_{n\dim V} (2mnx)^{n (\dim V)/2}.
\end{equation} 
Here $v_{n\dim V}$ denotes the volume of the unit ball in the Euclidean space of dimension $n\dim V$. It is given by:
\begin{equation}\label{S2} v_{n\dim V} = \frac{\pi^{n (\dim V)/2}}{\Gamma(1 + {n (\dim V)/2})}.
\end{equation} 

From (\ref{S1}) and (\ref{S2}), by a simple application of Stirling's formula, we get:
$$S(x) = \lim_{n \ra + \infty} \frac{1}{n} \log\left[ v_{n\dim V} (2mnx)^{n (\dim V)/2} \right] = (\dim V/2) [1 + \log(4 \pi m x /\dim V)].$$
In particular, $$S'(x) = \dim V/(2x)$$ and we recover (\ref{betaxMaxw}).

Conversely,  combined with the expression (\ref{PsiMaxw}) for the function $\Psi,$ Part 3) of Theorem \ref{thermo} allows one to recover the asymptotic behaviour of the volume $v_n$ of the $n$-dimensional unit ball, in the form:
$$v_n^{1/n} \sim \sqrt{2e\pi/n} \;\; \mbox{ when $n \ra + \infty.$} $$ 

Finally, observe that when $m = (2\pi)^{-1}$ --- the  case relevant for the comparison with the application to Euclidean lattices in Section  \ref{asymptoticho} --- the expressions for $\Psi$ and $S$ take the following simpler forms: 
$$\Psi(\beta) = (\dim V/2)  \, \log (1/\beta)$$
and 
$$S(x) = (\dim V/2) [1 + \log(2x /\dim V)].$$

\subsection{Relations with probability measures of maximal entropy}\label{maxentrop}

In this paragraph, we keep the notation recalled in paragraph \ref{AppendixScholium}, and we assume that the hypotheses of Theorem \ref{thermo} are satisfied --- namely we assume that the growth condition {\bf SE} holds and that $\mu(\cE) = +\infty.$

Let $\cC$ be the space of \emph{probability measures on $(\cE, \cT)$ absolutely continuous with respect to $\mu$.}

By sending such a measure $\nu$ to its Radon-Nikodym derivative $f = d\mu/d\nu,$ one establishes a bijection from $\cC$ to the convex subset 
$$\left\{ f \in L^1(\cE, \mu) \mid \mbox{$f \geq 0$ $\mu$-a.e. and $\int_\cE f d\mu =1$} \right\}$$
of $L^1(\cE, \mu)$.

To any measure $\nu = f \mu$ in $\cC,$ we may attach its ``energy":
$$\epsilon(\nu) := \int_\cE H \, d\nu = \int_\cE H f \,d\mu \in [0, +\infty].$$

\begin{lemma}\label{Gibbsestimate}
 Let $f: \cE \lra \R_+$ and $g:\cE \lra \Rpa$ be two $\cT$-measurable functions.
 
 1) For every $x \in \cE,$
 \begin{equation}\label{eq:Gibbsestimate}
f(x) \log \frac{f(x)}{g(x)} \geq f(x) -g(x).
\end{equation}
Moreover, equality holds in (\ref{eq:Gibbsestimate}) if and only if $f(x) = g(x)$.

2) If $g$ is $\mu$-integrable, then the negative part $(f \log(f/g))^-$ of $f \log(f/g)$ is $\mu$-integrable, and therefore $$\int_\cE f\log(f/g)\, d\mu$$ is well-defined in $]-\infty, +\infty].$

3) If $f$ and $g$ belongs to $\cC$, then $$\int_\cE f\log(f/g)\, d\mu$$ belongs to $[0, +\infty]$ and vanishes if and only if $f = g$ $\mu$-almost everywhere.
\end{lemma}
 
\begin{proof} For any $t \in \R_+,$ we have:
$$t \log t \geq t-1,$$
and equality holds if and only if $t=1$. Applied to $t =f(x)/g(x),$ this  implies 1). 

From 1), assertions 2) and 3) immediately follow.
\end{proof}

\begin{proposition}\label{prop:maxentrop} Let $\nu = f \mu$ be an element of $\cC$.

1) If $\epsilon(\nu) < + \infty,$ then $(f \log f)^-$ is $\mu$-integrable, and therefore the ``information theoretic entropy" of $\nu$ with respect to $\mu$
$$I(\nu\mid \mu) := - \int_{\cE} \log (d\mu/ d\nu) d\nu = - \int_{\cE} f \log f \, d\mu$$
is well defined in $[-\infty, +\infty[$.

2) Let $u$ and $\beta$ be two positive real numbers such that $u =U(\beta).$

If $\epsilon(\nu) = u,$ then
\begin{equation}\label{IS}
I(\nu \mid \mu) \leq S(u).
\end{equation}
Moreover the equality is achieved in (\ref{IS}) for a unique measure $\nu$ of $\cC$ in $\epsilon^{-1}(u),$ namely for the measure
$$\nu_\beta := Z(\beta)^{-1}  e^{-\beta H} \mu,$$
where $$Z(\beta) := \int_\cE  e^{-\beta H}\,  d\mu.$$
 \end{proposition}
 
 In substance, the content of Proposition \ref{prop:maxentrop} goes back to the seminal work of Boltzmann and Gibbs on statistical mechanics\footnote{See  for instance Boltzmann's memoirs \cite{Boltzmann72} and  \cite{Boltzmann77}, Chapter V.  Our presentation is a straightforward generalization, in the framework of general measure theory, of the ``axiomatic" approach of Gibbs in  \cite{Gibbs05}, Chapter XI (Lemma \ref{Gibbsestimate} above notably appears in \emph{loc. cit.}, p. 130).}.  
 Similar results play also a central role in information theory and in statistics (see for instance \cite{Kullback97}, notably Chapter 3). We refer the reader to \cite{Georgii03} (notably § 3.4) for additional informations and references.

\begin{proof} Let $\nu$ be an element of $\cC$ such that $u:= \epsilon(\nu)$ is finite, and let $\beta:= U_\Eb^{-1}(u).$

We may consider the measurable function 
$$g_\beta := Z(\beta)^{-1} e^{- \beta H}.$$
 It is everywhere positive on $\cE$, and satisfies:
 \begin{equation}
\int_\cE g_\beta\, d\mu = 1
\end{equation}
and
 \begin{equation*}
\log g_\beta = -\log Z(\beta) - \beta H = -\Psi(\beta) -\beta H.
\end{equation*}

In particular, we have:
\begin{align}
f \log f & = f \log(f/g_\beta) + f \log g_\beta  \notag \\
& =f \log (f/g_\beta)  -\Psi(\beta) f -\beta Hf. \label{flogf}
\end{align}

Also recall that
$$\int_\cE f \,d\mu = 1$$
and
$$\int_\cE Hf d\mu = u = U_\Eb(\beta) < +\infty.$$
Consequently, using (\ref{SPsi}), we get:
\begin{equation}\label{voilalentropie}
\int_\cE \left(\Psi(\beta) f  + \beta Hf\right) d\mu = \Psi(\beta) + \beta u = S (u).
\end{equation}

According to Lemma \ref{Gibbsestimate}, 2), the negative part $(f \log (f/g_\beta))^-$ of  $f \log (f/g_\beta)$ is $\mu$-integrable. Together with (\ref{flogf}), this shows that the negative part $(f \log f)^-$ of $f \log f$ is $\mu$-integrable. This establishes 1).

Moreover, according to Lemma \ref{Gibbsestimate},  3), the integral
$\int_\cE f\, \log (f/g_\beta) \, d\mu$ 
is non-negative and vanishes if and only if $f=  g_\beta$ $\mu$-almost everywhere. Together with (\ref{flogf}) and (\ref{voilalentropie}), this establishes 2).
\end{proof}

\medskip

\chapter[Continuity of linear forms on prodiscrete modules]{Non-complete discrete valuation rings and continuity of linear forms on prodiscrete modules}\label{prodiscretemod}

\medskip

This Appendix is devoted to a discussion of the results of ``automatic continuity" of Specker (\cite{Specker50}) and Enochs (\cite{Enochs64}), in a setting adapted to their application to the categories $CTC_A$ in Section \ref{CTCbasics}.

\section{Preliminary: maximal ideals, discrete valuation rings, and  completions} Let $R$ be a ring, and let $\fm$ be a maximal ideal of $R$ such that the local ring
$$R_{(\fm)} := (R\setminus \fm)^{-1} R$$
is a discrete valuation ring. Its maximal ideal is $\fm_{(\fm)} := (R\setminus \fm)^{-1} \fm.$ We shall denote by 
$$\iota: R \lra R_{(\fm)}$$
the canonical morphism of rings, which send an element $x$ of $R$ to $\iota(x) := x/1.$ 

We may consider the $\fm$-adic completion of $R$ at $\fm$:
$$\hR_{\fm} := \varprojlim_n R/{\fm}^n.$$
It is a local ring, of maximal ideal $\hat{\fm} :=  \varprojlim_n \fm/{\fm}^n.$

As $\fm$ is a maximal ideal, the localization $R_{(\fm)}$ may be identified with a subring of $\hR_{\fm}$, so that the canonical morphism from $R$ to $\hR_{\fm}$ becomes the composition
  $R \stackrel{\iota}{\lra} R_{(\fm)} \hlra \hR_{\fm}.$
Moreover the $\fm$-adic (resp. $\fm_{(\fm)}$-adic, resp. $\hat{\fm}$-adic) filtrations on $R$ (resp. $R_{(\fm)}$,  resp. $\hR_{\fm}$)
are strictly compatible.

In particular, $\hR_{\fm}$ may be identified with the $\fm_{(\fm)}$-adic completion of $R_{(\fm)}$, and therefore is a complete discrete valuation ring.

We shall denote by $v$ the $\fm$-adic\footnote{or, more correctly $\fm_{(\fm)}$-adic  or $\hat{\fm}$-adic.} valuation on $R_{(\fm)}$ and $\hR_{\fm}$ and by $\vert . \vert = e^{-v}$ the associated absolute value. For simplicity, we shall also denote by $v$ and $\vert. \vert$  their composition with the morphism 
$\iota: R \lra \hR_{\fm}.$

Let us consider: 
$$\cL= \{ (\lambda_i)_{i \in \N}\in \hR_{\fm}^\N \mid \lim_{i \ra + \infty} \vert \lambda_i \vert = 0 \}.$$
For any $\bm{\lambda} = (\lambda_i)_{i \in \N}$ in $\cL,$  
we may define a $R$-linear map
$$\Sigma_{\bm{\lambda}}:  R^\N \lra \hR_{\fm}$$
by letting:
\begin{equation*}
\Sigma_{\bm{\lambda}}((x_i)_{i\in \N}) := \sum_{i \in \N} \lambda_i \iota(x_i).
\end{equation*}

\begin{proposition}\label{propSigma} Let $\bm{\lambda}:=(\lambda_i)_{i \in \N}$ be an element of  $\cL$ such that $I:=\{ i \in \N \mid \lambda_i \neq 0 \}$ is infinite. 

If we define
$$n:=\min_{i \in \N} v(\lambda_i),$$
then we have:
\begin{equation}\label{Sigman}
\Sigma_{\bm{\lambda}}(R^{\N}) = \hat{\fm}^{n}. 
\end{equation}
\end{proposition}

\begin{proof}[Proof of Proposition \ref{propSigma}] Let us choose a bijection $\psi : \N \lrasim I$. The integer $n$ and the image of $\Sigma_{\bm{\lambda}}$ are clearly unchanged if we replace the sequence $\bm{\lambda}:=(\lambda_i)_{i \in \N}$ by $(\lambda_{\psi(i)})_{i \in \N}$. Therefore, to establish Proposition \ref{propSigma}, we may assume that $\bm{\lambda}$ belongs to $(R\setminus \{0\})^\N$ and that the sequence of non-negative integers
$n_i := v(\lambda_i)$ $(i \in \N)$ is such that $n:= \min_{i \in \N} n_i = n_0.$

For any $(x_i)_{i \in \N}$ in $R^\N$ and any $i$ in $\N$, the product $\lambda_i  \iota(x_i)$ belongs to $\hat{\fm}^{n_i},$ and \emph{a fortiori} to $\hat{\fm}^n$. This implies that $\Sigma_{\bm{\lambda}}(R^{\N})$ is contained in $\hat{\fm}^{n}$.

Conversely, let $\alpha$ be an element of $\hat{\fm}^{n}$. Observe that, for any $k \in \N,$ $\lambda_k \iota(R)$ is dense in $\lambda_k \hR_\fm = \hat{\fm}^{n_k}.$ Therefore we may inductively construct a sequence $(x_i)_{i \in \N}$ such that, for any $k\in \N,$
$$v(\alpha - \sum_{i= 0}^k \lambda_i \iota(x_i)) \geq n_{k+1}.$$ Then $$\alpha = \sum_{i \in \N} \lambda_i \iota(x_i) =\Sigma_{\bm{\lambda}}((x_i)_{i \in \N}).$$\end{proof}

Observe that, for any natural integer $k$, by applying Proposition \ref{propSigma} to the sequence $(\lambda_{i+k})_{i \in \N}$, that still belongs to $\cL$, we obtain:

\begin{corollary} Under the assumption of Proposition \ref{propSigma}, the 
map $\Sigma_{\bm{\lambda}}$ is continuous and open from $R^{\N}$, equipped with the product topology of the discrete topology on each of the factors $R$, onto $\hat{\fm}^{n}$ equipped with the $\hat{\fm}$-adic topology.

Actually,   for any $k \in \N$, if  we define $n_k := \min_{i \in \N_{\geq k}} v(\lambda_i),$
then
\begin{equation*}\label{Sigmank}
 \Sigma_{\bm{\lambda}}(\{0\}^k \oplus R^{\N_{\geq k}}) = \hat{\fm}^{n_k}, 
\end{equation*}
while 
$\lim_{k \ra + \infty} n_k = + \infty.$ \qed 
\end{corollary}

\section{Continuity of linear forms on prodiscrete modules}

We keep the notation of the previous section, and we consider a topological module $M$ over the ring $R$ equipped with the discrete topology.

The topological $R$-module $M$ is complete and prodiscrete, with a countable basis of neighborhoods of $0$ precisely if it isomorphic (as a topological module) to the projective limit $\varprojlim_k M_k$ of some projective system 
$$M_0 \stackrel{q_0}{\longleftarrow}M_1 \stackrel{q_1}{\longleftarrow} M_2 \stackrel{q_2}{\longleftarrow}\dots$$
of discrete $R$-modules. (The maps $q_i$ may actually be assumed surjective.)

Then, if we denote by $$p_k: M \lra M_k$$ the canonical projection maps, a sequence $(m_i)_{i\in \N} \in M^\N$ converges to zero in $M$ if and only if, for every $k \in \N,$ the projection $p_k(m_i)$ vanishes for $i$ large enough (depending on $k$).

Moreover, for any such sequence $(m_i)_{i\in \N}$ in $M^\N$ that converges to zero and for any sequence $(r_i)_{i \in \N}$, the series 
$\sum_{i \in \N} r_i m_i$ converges in $M$.  

\begin{proposition}\label{contuitautom}

Let $M$ be a complete prodiscrete topological $R$-module, with a countable basis of neighborhoods of $0$, and let $\phi : M \lra \hR_{\fm}$ be a $R$-linear map.

If $\phi$ is not continuous when $M$ is equipped with its prodiscrete topology and $\hR_{\fm}$ with the discrete topology, then there exists $k \in \N$ such that $\phi(M) = \hat{\fm}^k$.
\end{proposition}

\begin{proof}
Let us denote by $\pi$ an element of $\fm \setminus \fm^2$. (The set $\fm \setminus \fm^2$ is indeed not empty, since  $\fm_{(\fm)} \neq \fm^2_{(\fm)}$ and therefore $\fm \neq \fm^2$.) 

If $\phi$ is not continuous, then there exists a sequence $(m_i)_{i \in \N}$ in $M^\N$ which converges to $0$ in $M$ and such that $(\phi(m_i))_{i \in \N}$
has an infinity of non-zero terms. 

Then the  family $$\bm(\lambda) = (\lambda_i)_{i \in \N}:=(\iota(\pi)^i\phi(m_i))_{i \in \N}$$ satisfies the assumptions of Proposition \ref{propSigma}, and therefore, for some non-negative integer $n$, $\Sigma_{\bm{\lambda}}(R^\N) = \hat{\fm}^n.$ 

We are going to prove that any element of $\Sigma_{\bm{\lambda}}(R^\N)$ belongs to the image of $\phi$. This will show that $\phi(M)$ contains $\hat{\fm}^n$. Since $\phi(M)$ is a $R$-submodule of $\hR_{\fm}$, this will complete the proof.

To achieve, let $(x_i)_{i \in \N}$ be any sequence in $R^\N.$
We may form  the following sum,  convergent  in $M$:
 $$m := \sum_{i \in \N} \pi ^i x_i\, m_i.$$
 For any non-negative integer $k$, we have:
 \begin{equation}\label{sumconvm}
m = \sum_{0 \leq i \leq k} \pi ^i x_i \,m_i + \pi^{k+1} r_k,
\end{equation}
 where $r_k$ is defined as the convergent sum in $M$:
 $$r_k := \sum_{i \in \N} \pi^i x_{i+k+1}\, m_{i+k+1}.$$ 
By applying $\phi$ to the relation (\ref{sumconvm}), we see that, for every $k \in \N,$
$$\phi(m) - \sum_{0 \leq i \leq k} \iota(\pi ^i x_i) \phi(m_i) \in \fm^{k+1}.$$ 
 Finally, 
 $$\Sigma_{\bm{\lambda}}((x_i)_{i \in \N}) = \sum_{i \in \N} \iota(\pi)^i \phi(m_i) \iota(x_i) = \phi(m).$$
\end{proof}

 Observe that, when $R$ is a domain, the morphism $\iota: R \lra R_{(\fm)}$ is injective and $\iota(R)$ contains $\hat{\fm}^k$ for some $k \in \N$ if and only if $R$ is $\fm$-adically complete.  Therefore, from Proposition \ref{contuitautom}, we immediately derive:

 \begin{corollary}\label{contuitautombis} Let $M$ be a complete prodiscrete topological $R$-module, with a countable basis of neighborhoods of $0$. 
If $R$ is a domain and is not $\fm$-adically complete,  then 
 any $R$-linear map $\phi : M \lra R$ is continuous when $M$ is equipped with its prodiscrete topology and $R$ with the discrete topology.\qed

 \end{corollary}
 
\medskip

\chapter{Measures on countable sets and their projective limits}\label{measurespro}

\medskip

This Appendix is devoted to various results concerning measure theory on the Polish spaces defined as projective limits of countable systems of countable discrete sets that are used in the proofs of  Section  \ref{SectionSum}.

\section{Finite measures on countable sets}

Let $D$ be a (possibly finite) countable set.
 The real vector space $\cM^b(D)$ of real bounded measures on $D$ (equipped with the $\sigma$-algebra $\cP(D)$ of all subsets of $D$)   may be identified with the vector space $l^1(D)$:
 \begin{equation}\label{Ml1}
\begin{array}{rcl}
 \cM^b(D) & \lrasim   & l^1(D)   \\
 \mu & \longmapsto   & (\mu(\{x\})_{x \in D}. 
\end{array}
 \end{equation}
 This isomorphism maps the cone $\cM^b_+(D)$ of positive bounded measures onto
 $$l^1_+(D) = l^1(D) \cap \R_+^D.$$
 Moreover the total mass of some measure $\mu \in  \cM^b(D)$ coincides with the $l^1$-norm of its image by the isomorphism (\ref{Ml1}):
 $$\Vert \mu \Vert  = \sum_{x \in D} \vert \mu(\{x\}) \vert.$$
 
 On the real vector space $\cM^b(D)$, or equivalently on $l^1(D)$, we may consider the following separated locally convex topologies:
 
 (i) the topology of vague convergence of measures in $\cM^b(D)$, or equivalently the topology on $l^1(D)$ induced by the topology of pointwise convergence on $\R^D,$ or the $\sigma(l^1(D), \R^{(D)})$-topology on $l^1(D);$
 
 (ii) the topology of narrow convergence of measures in $\cM^b(D)$,  that is the $\sigma(l^1(D), l^{\infty}(D))$-topology on $l^1(D);$
 
 (iii) the topology defined by the ``total mass norm"  on $\cM^b(D)$ (which coincides with  the $l^1$-norm on $l^1(D)$).
 
 The first of these topologies is strictly finer than the second one, and the second one strictly finer than the third one. 
 The following proposition compares the induced topologies on the cone $\cM^b_+(D)$, and notably asserts that the topology of narrow convergence and the topology of norm convergence on  $\cM^b_+(D)$ coincide:
 
 \begin{proposition}\label{convnarrow} Let $(\mu_i)$ be a sequence, or more generally a net, of elements of $\cM^b_+(D)$ which converges vaguely to an element $\mu$ of $\cM^b_+(D)$.
 Then the following conditions are equivalent:
 
  \emph{(1)} $(\mu_i)$  converges to $\mu$  in the topology of narrow convergence on $\cM^b(D)$;
 
 \emph{(2)} $\lim_i\mu_i({D}) = \mu(D);$
 
   \emph{(3)}  the total mass $\mu_i(D)$ stays bounded (for $i$ large enough); moreover, for every $\epsilon \in \R^\ast_+,$ there exists a finite subset $F$ of $D$ such that, for $i$ large enough,
   $\mu_i(D \setminus F) < \epsilon$;
 
  \emph{(4)}  $\lim_i \Vert \mu_i - \mu \Vert = 0.$

\end{proposition}

This is well-known, and the implications $(1) \Rightarrow (2)  \Rightarrow (3)  \Rightarrow (4)  \Rightarrow (1)$ are indeed easily established. 

In this monograph, we shall say that \emph{a sequence $(\mu_i)$ in $\cM^b_+(D)$ converges to some measure $\mu \in \cM^b_+(D)$} when the above equivalent conditions are satisfied.

The following convergence criterion plays a central role in our study of the $\theta$-invariants of infinite dimensional Hermitian vector bundles.

\begin{proposition}\label{critconvmu}
Let $(\mu_i)_{i \in \N}$ be a sequence in $\cM^b_+(D)$. If there exists $(t_i)_{i\in \N}$ in $l^1(\N)$ such that, for every $i\in \N,$
\begin{equation}\label{ineqtmu}
\mu_{i+1} \leq e^{t_i} \mu_i, 
\end{equation}
then the sequence $(\mu_i)_{i \in \N}$ converges to some $\mu \in \cM^b_+(D)$.
 \end{proposition}

\begin{proof} When the condition (\ref{ineqtmu}) is satisfied by $t_i=0$ for every $i \in \N$ --- that is, when
$\mu_{i+1} \leq \mu_i $ 
for every $i\in \N$ 
--- then, to any subset $X$ of $D$, we may associate the limit
$$\mu(X) := \lim_{i \rightarrow + \infty} \mu_i(X)$$
of the non-increasing sequence $( \mu_i(X))_{i\in \N}$ in $\R_+.$ It is straightforward that this defines an element $\mu$ of $\cM^b_+(D)$ and that  $(\mu_i)_{i \in \N}$ converges to  $\mu$.

One reduces the general case of the Proposition to this special case by letting, for every $i \in \N$,
$$\nu_i := e^{-\sum_{0\leq j < i} t_j}  \mu_i.$$
Indeed, $(\nu_i)_{i \in \N}$ is then a sequence in $\cM^b_+(D)$ such that
$\nu_{i+1} \leq \nu_i $ 
for every $i\in \N$, and therefore admits a limit $\nu$ in $\cM^b_+(D)$. Consequently, $(\mu_i)_{i \in \N}$ converges to 
\begin{equation*}\mu := e^{\sum_{j\in \N} t_j} \nu. \qedhere
\end{equation*} 
\end{proof}

Recall that the construction which associates, to some countable set $D$, the cone $\cM^b_+(D)$ equipped with the topology of narrow convergence (or equivalently, of norm convergence) is functorial. 

Namely, for any map 
$f: D \lra D'$
between two countable sets $D$ and $D'$, we may defined a $\R$-linear map
$$f_\ast: \cM^b(D) \lra \cM^b(D')$$
by letting, for any $\mu \in \cM^b(D)$ and any $X' \subset D',$
$$f_\ast \mu(X'):= \mu(f^{-1}(X')).$$
The map $f_\ast$ maps $ \cM^b_+(D)$ to $\cM^b_+(D')$ and is continuous when $\cM^b(D)$ and $\cM^b(D')$ are equipped both with the topology of narrow convergence or of norm convergence. Moreover, if $f': D' \lra D''$ is a second map, with range some countable set $D''$, we have:
$$(f'\circ f)_\ast = f'_\ast \circ f_\ast.$$
\section{Finite measures on projective limits of countable sets}

Consider a projective system
\begin{equation}\label{Dproj}
D_{\bullet} : D_0 \stackrel{q_0}{\longleftarrow}D_1 \stackrel{q_1}{\longleftarrow}\dots \stackrel{q_{i-1}}{\longleftarrow}D_i \stackrel{q_i}{\longleftarrow} D_{i+1} \stackrel{q_{i+1}}{\longleftarrow} \dots
\end{equation}
of countable sets. We may equip each of them with the discrete topology, and the projective limit of this system
$$\hD := \varprojlim_i D_i$$
with the associated projective limit topology. It is a Polish space, and we shall denote by $\cM^b(\hD)$ the real vector space of real bounded Borel measures on $\hD$ and by $\cM^b_+(\hD)$ the cone in $\cM^b(\hD)$ of positive bounded Borel measures on $\hD$. 

Recall that every measure in $\cM^b_+(\hD)$ is regular and tight. Moreover the cone  $\cM^b_+(\hD)$ may be identified with the projective limit of the projective system 
$$\cM^b_+(D_{\bullet}) : \cM^b_+(D_0) \stackrel{q_{0\ast}}{\longleftarrow}\cM^b_+(D_1) \stackrel{q_{1\ast}}{\longleftarrow}\dots \stackrel{q_{i-1, \ast}}{\longleftarrow}\cM^b_+(D_i) \stackrel{q_{i\ast}}{\longleftarrow} \cM^b_+(D_{i+1}) \stackrel{q_{i+1,\ast}}{\longleftarrow} \dots$$
deduced from (\ref{Dproj}) by application of the functor $\cM_+^b$. Indeed, if $$p_i: \hD \lra D_i$$ denotes, for any $i \in \N$, the canonical map from $\hD$ to $D_i$, then the map which sends a measure  in $\cM^b_+(\hD)$ to the family  of its direct images by the $p_i$'s defines a bijection 
\begin{equation}\label{Kolmogorov}
 \begin{array}{rcl}
 \cM^b_+(\hD) & \lrasim  &  \varprojlim_i  \cM^b_+(D_i) (\subset \prod_{i \in \N} \cM^b_+(D_i) ) \\
 \mu & \longmapsto  & (p_{i\ast} \mu)_{i \in \N},   
\end{array}
\end{equation}
according to a classical theorem of Kolmogorov (\cite{Kolmogorov33},  Section III.4; see also \cite{BourbakiINT9}, Section 4.3).

For every $(i,j) \in \N^2$ such that $i \leq j,$ we let:
$$p_{ij}:= q_i \circ q_{i+1} \circ \cdots \circ q_{j-1}: D_j \lra D_i.$$

\begin{proposition}\label{limitmuDh}
Let $(\gamma_i)_{i\in \N}$ be an element of $\prod_{i\in \N} \cM^b_+(D_i)$ and let $(\lambda_i)_{i \in \N} \in l^1(\N)$ such that, for every $j\in \N$,
\begin{equation}\label{ineqlambdagamma}
 q_{j\ast} \gamma_{j+1} \leq e^{\lambda_j}\gamma_j.
\end{equation}

Then, for every $i \in \N,$ the sequence $(p_{ij\ast} \gamma_j)_{j \in \N_{\geq i}}$ converges to some limit $\mu_i$ in $\cM^b_+(D_i)$. Moreover there exists a unique measure $\mu \in \cM^b_+(\hD)$ such that, for any $i \in \N,$
$$\mu_i = p_{i\ast} \mu.$$

\end{proposition}

In particular, the sequence 
$$(\gamma_j(D_j))_{j \in \N} = (p_{0,j\ast} \gamma_j(D_0))_{j \in \N}$$
converges in $\R_+$ and, for any $i \in N$,
\begin{equation}\label{mulim}
\mu(\hD) = \mu_i(D_i) = \lim_{j \ra + \infty} \gamma_j(D_j).
\end{equation} 

\begin{proof} From (\ref{ineqlambdagamma}), by application of $p_{ij\ast},$ we derive that, for any $(i,j) \in \N^2$ such that $i \leq j,$ 
$$p_{i,j+1 \ast} \gamma_{j+1}  \leq e^{\lambda_j}\,p_{ij,\ast} \gamma_j.$$
The convergence of the sequence $(p_{ij\ast} \gamma_j)_{j \in \N_{\geq i}}$ to some limit $\mu_i$ in  $\cM^b_+(D_i)$ therefore follows from Proposition \ref{critconvmu}.  Moreover the continuity of the maps $q_{i\ast}$ with respect to the norm topology shows that, for every $i \in \N,$
$$\mu_i = q_{i\ast} \mu_{i+1}.$$
The existence and the unicity of $\mu$ then follows from Kolmogorov's theorem (\ref{Kolmogorov}).
\end{proof}

The following Proposition provides an alternative interpretation of the convergence condition on the sequence $(\gamma_{i})_{i \in \N}$ which appears in Proposition \ref{limitmuDh}. (We refer the reader to \cite{Schwartz73}, Appendix, or to
\cite{Billingsley99}, Chapter 1, for basic results concerning the topology of narrow convergence --- also called topology of weak convergence in \cite{Billingsley99} ---  on the space of bounded measures on the Polish space $\hD$.)

\begin{proposition}\label{limitgammatilde}
Let $(\gamma_i)_{i\in \N}$ be a family of measures in $\prod_{i\in \N} \cM^b_+(D_i)$ and, for every $i \in \N,$ let $\tilde{\gamma}_i$ be a measure in $\cM^b_+(\hD)$ such that
\begin{equation}\label{pig}
p_{i\ast} \tilde{\gamma}_i = \gamma_i.
\end{equation}
Then the following two conditions are equivalent:

$\mathbf{Conv_1:}$ For every $i \in \N,$ the sequence $(p_{ij\ast} \gamma_j)_{j \in \N_{\geq i}}$ converges to some measure $\mu_i$ in $\cM_+^b(D_i).$

$\mathbf{Conv_2:}$ The sequence $(\tilde{\gamma}_j)_{j \in \N}$ convergence to some measure $\mu$ in the topology of narrow convergence on $\cM_+^b(\hD).$

When these conditions hold, the measure $\mu$ is the unique element of $\cM^b_+(\hD)$ such that $p_{i\ast}\mu = \mu_i$ for any $i\in \N.$
\end{proposition}

 When the conditions $\mathbf{Conv_1}$ and $\mathbf{Conv_2}$ of Proposition \ref{limitgammatilde} are satisfied, we shall say that \emph{the sequence $(\gamma_i)_{i\in \N}$  satisfies condition $\mathbf{Conv}$} and that $\mu$ is the \emph{limit measure} associated to $(\gamma_i)_{i\in \N}$. 
 
 Observe that, when the maps $q_i: D_{i+1} \ra D_i$, or equivalently the maps $p_i:\hD \ra D_i$, are all surjective then for any sequence $(\gamma_i)_{i\in \N}$ in $\prod_{i\in \N} \cM^b_+(D_i)$ as above, there exists a sequence $(\tilde{\gamma}_i)_{i\in \N}$ in $\cM^b_+(\hD)^\N$ such that the conditions (\ref{pig}) hold. (Indeed, if for any $x \in D_i$, we denote by $\tilde{x}$ a point in $p_i^{-1}(x)$, we may simply define $\tilde{\gamma}_i$ as $\sum_{x \in D_i} \gamma_i(x) \delta_{\tilde{x}}.$)
 
 \begin{proof} Observe that, for any $(i,j) \in \N^2$ such that $i\leq j,$ we have
 $$p_{ij\ast}\gamma_{j}= p_{ij\ast}p_{j\ast}\tilde{\gamma}_{j}= p_{i\ast}\tilde{\gamma}_{j}.$$
 
 The continuity of the maps $p_{i\ast}:\cM^b(\hD) \ra \cM^b(D_{i})$ (for the topology of narrow convergence) therefore establishes the implication $\mathbf{Conv_1} \Longrightarrow \mathbf{Conv_2}$, with $\mu_{i}= p_{i\ast}\mu$ for every $i\in \N.$  The unicity of the measure $\mu$ satisfying these conditions follows from the injectivity assertion in Kolmogorov's theorem (\ref{Kolmogorov}).
 
 Conversely, let us assume that $\mathbf{Conv_1}$ is satisfied. The continuity of the maps $p_{ii'\ast}:\cM^b(D_{i'}) \ra \cM^b(D_{i})$, $(0\leq i \leq i')$ in the topology of narrow convergence shows that the measures $\mu_{i}$ satisfy the conditions
 $$p_{ii'\ast}\mu_{i'}= \mu_{i}.$$
 According to Kolmogorov's theorem (\ref{Kolmogorov}), there exists a (unique) measure $\mu \in \cM^b_{+}(\hD)$ such that 
 $\mu_{i}= p_{i\ast}\mu$ for every $i\in \N.$ 
 
 To complete the proof, we are left to show that $(\tilde{\gamma}_{i})_{i \in \N}$ converges to $\mu$ in the topology of narrow convergence. Recall that this means that, for every function $f$ in the space $C^b(\hD, \R)$ of bounded continuous real valued functions on $\hD$, we have:
 \begin{equation}\label{narrowdef}
 \lim_{j \ra + \infty}\int_{\hD} f \,d\tilde{\gamma}_{j} = \int_{\hD} f \,d\mu.
\end{equation}
Also recall that, if $d$ denotes a metric on $\hD$ that defines its topology, this condition is satisfied as soon as it is satisfied by any $f$ in the subspace $C_{u}^b(\hD, \R)$ of of bounded  real valued functions on $\hD$ which are \emph{uniformly} continuous with respect to $d$ (see for instance \cite{Billingsley99}, Theorem 2.1\footnote{In \emph{loc. cit.}, only the narrow (a.k.a. weak) convergence of \emph{probability} measures is considered. The case of \emph{bounded positive} measures easily reduces to this one.}).

In turn, condition  (\ref{narrowdef}) is satisfied for every $f$ in $C_{u}^b(\hD, \R)$ if it is satisfied for every $f$ in some subset $\cE$ of $C_{u}^b(\hD, \R)$ dense in the topology of uniform convergence.

To prove that this last condition is indeed fulfilled, choose a decreasing sequence $(\epsilon_{i})_{i\in \N}$ in $\R_{+}^{\ast\N}$ such that 
$\lim_{i \ra + \infty}\epsilon_{i} = 0$ and equip $\hD$ with a metric $d$ such that, for any $x\in \hD$ and any $i\in \N$:
$$\{ x' \in \hD \mid d(x',x) \leq \epsilon_{i} \} = p_{i}^{-1}(p_{i}(x)).$$ 
(For instance, we may choose $\epsilon_{-1}$ in $]\epsilon_{0}, + \infty[$ and consider the metric $d$ defined by $$d(x,y) := \epsilon_{k-1}, \mbox{ where } k:= \min\{i \in \N \mid p_{i}(x) \neq p_{j}(y) \}$$
for any $(x,y) \in \hD^2$ such that $x \neq y.$) Any such metric defines the topology of $\hD$. Moreover, for every $i\in \N,$ the image of 
$$p_{i}^\ast := . \circ p_{i}: l^\infty (D_{i}) \lra  C^b(\hD, \R)$$
belongs to the subspace $C_{u}^b(\hD, \R)$ of functions uniformly continuous with respect to $d$, and the union
$$\cE := \bigcup_{i \in \N}  p_{i}^\ast(l^\infty (D_{i}))$$
is dense in $C_{u}^b(\hD, \R)$ equipped with the topology of uniform convergence.

Finally, if $f$ is an element of $\cE,$ of the form $p_{i}^\ast \phi$ for some $i \in \N$ and some $\phi \in l^\infty(D_{i}),$ then, for any $j\in \N_{\geq i}$,
$$\int_{\hD}f \, d\tilde{\gamma}_{j} = \int_{\hD}p_{i}^\ast \phi \, d\tilde{\gamma}_{j} = \int_{D_{i}} \phi\, dp_{i\ast}\tilde{\gamma}_{j}= \int_{D_{i}} \phi\, dp_{ij\ast}{\gamma}_{j}.$$
When $j$ goes to $+\infty,$ this converges to
\begin{equation*}
\int_{D_{i}}\phi \, d\mu_{i}= \int_{\hD}p_{i}^\ast \phi \, d\mu = \int_{\hD} f \, d\mu. 
\end{equation*}
Consequently, condition (\ref{narrowdef}) is satisfied by $f$.
\end{proof}
 
 \begin{proposition}\label{limitgammaC}
 Let $(\gamma_i)_{i\in \N}$ be a family of measures in $\prod_{i\in \N} \cM^b_+(D_i)$.
 
 If the sequence $(\gamma_i(D_{i)}))_{i\in \N}$ converges in $\R_+$ and if $\cC$ is a countable subset of $\hD$ such that, for any $x \in \cC,$ the sequence $(\gamma_j(\{p_j(x)\}))_{j\in \N}$ admits a limit $\gamma(x)$ in $\R_+$ and if
 \begin{equation}\label{sumgammabis}
 \sum_{x\in \cC} \gamma(x) = \lim_{i \ra + \infty} \gamma_i(D_i),
 \end{equation}
 then the sequence $(\gamma_i)_{i\in \N}$ satisfies condition $\mathbf{Conv}$ and the associated limit measure is
 $$\mu:= \sum_{x \in \cC} \gamma(x) \delta_x.$$ 
 \end{proposition}

 \begin{proof} This will follow from the following
 
 \begin{lemma}\label{FC} For every $\epsilon \in \R^\ast_{+},$ there exists $i(\epsilon) \in \N$ and a finite subset $F_{\epsilon}\subset \cC$ such that
 \begin{equation}\label{LA}\mu(\cC \setminus F_{\epsilon}) < \epsilon
 \end{equation}
 and, for any $j\in \N_{\geq i(\epsilon)}$, 
 $p_{j\mid F_{\epsilon}}: F_{\epsilon} \lra D_{j}$ is injective
 and 
 \begin{equation}\label{LB}\gamma_{j}(D_{j}\setminus p_{j}(F_{\epsilon})) < \epsilon.
 \end{equation}
\end{lemma}

Indeed, taking this lemma for granted, we have, for any $\epsilon >0$ and any $j\in \N_{\geq i(\epsilon)},$
\begin{equation}\label{alpha}
\Vert \gamma_{j}- 1_{p_{j}(F_{\epsilon})} \gamma_{j}\Vert < \epsilon.
\end{equation}
Moreover, for any $x$ in the finite set $F_{\epsilon},$
$$\lim_{j \ra +\infty} \gamma_{j}(p_{j}\{x\}) -\mu(\{x\})= \lim_{j \ra +\infty} \gamma_{j}(p_{j}\{x\}) -\gamma(x) = 0.$$
Therefore 
 \begin{equation}\label{beta}
\lim_{j \ra + \infty} \Vert 1_{p_{j}(F_{\epsilon})}\gamma_{j} -p_{j\ast}(1_{F_{\epsilon}}\mu)\Vert =0.
\end{equation}
Finally,
\begin{equation}\label{gamma}
\Vert 1_{F_{\epsilon}}\mu -\mu\Vert = \mu (\cC \setminus F_{\epsilon}) \leq \epsilon.
\end{equation}

For any $(i,j) \in \N^2$ such that $i\leq j,$ we have:
\begin{equation*}
\begin{split}
\Vert p_{ij\ast}\gamma_{j} - p_{i\ast}\mu \Vert & \leq \Vert p_{ij\ast}(\gamma_{j}-1_{p_{j}(F_{\epsilon})} \gamma_{j}) \Vert
+ \Vert p_{ij\ast}(1_{p_{j}(F_{\epsilon})} \gamma_{j} -p_{j\ast}(1_{F_{\epsilon}}\mu))\Vert + \Vert p_{i\ast}(1_{F_{\epsilon}}\mu -\mu)\Vert \\
 & \leq \Vert \gamma_{j}-1_{p_{j}(F_{\epsilon})} \gamma_{j} \Vert
+ \Vert 1_{p_{j}(F_{\epsilon})} \gamma_{j} -p_{j\ast}(1_{F_{\epsilon}}\mu)\Vert + \Vert 1_{F_{\epsilon}}\mu -\mu\Vert
\end{split}
\end{equation*}

From (\ref{alpha}-\ref{gamma}), we deduce that,  when $j$ goes to $+\infty$, the upper limit of the last sum is $\leq 2 \epsilon.$ As $\epsilon \in \R^\ast_{+}$ is arbitrary, this shows that
\begin{equation*}\lim_{j \ra +\infty} \Vert p_{ij\ast}\gamma_{j} - p_{i\ast}\mu \Vert = 0. \qedhere
\end{equation*}

\end{proof}

\begin{proof}[Proof of Lemma \ref{FC}] We first choose a finite subset $F_{\epsilon}$ of $\cC$ so large that
\begin{equation}\label{La}
\sum_{x \in \cC \setminus F_{\epsilon}} \gamma(x) \leq \epsilon/2.
\end{equation}
Then, since $F_{\epsilon}$ is finite, if the integer $j$ is large enough --- say $j \geq i(\epsilon)$ --- the map
 $p_{j\mid F_{\epsilon}}: F_{\epsilon} \lra D_{j}$ is injective and 
\begin{equation}\label{Lb}
\sum_{x \in F_{\epsilon}} \vert \gamma_{j}(\{p_{j}(x)\}) - \gamma(x)\vert \leq \epsilon/4.
\end{equation}
According to (\ref{sumgammabis}), we may also assume that, when $j\geq i(\epsilon),$
\begin{equation}\label{Lc}
\gamma_{j}(D_{j}) \leq \sum_{x \in \cC} \gamma(x) + \epsilon/8.
\end{equation}

Then the estimate (\ref{LA}) follows from (\ref{La}). To establish (\ref{LB}), write
$$\gamma_{j}(D_{j} \setminus p_{j}(F_{\epsilon})) = \gamma_{j}(D_{j}) - \gamma_{j}(p_{j}(F_{\epsilon}))$$
and observe that
$$\gamma_{j}(p_{j}(F_{\epsilon})) \geq \sum_{x \in F_{\epsilon}} \gamma(x) - \sum_{x \in F_{\epsilon}}  \vert \gamma_{j}(\{p_{j}(x)\}) - \gamma(x)\vert = 
\sum_{x \in \cC} \gamma(x)  - \sum_{x \in \cC \setminus F_{\epsilon}} \gamma(x)- \sum_{x \in F_{\epsilon}}  \vert \gamma_{j}(\{p_{j}(x)\}) - \gamma(x)\vert.$$
Together with the estimates (\ref{La}-\ref{Lb}), this shows that
\begin{equation*}\gamma_{j}(D_{j} \setminus p_{j}(F_{\epsilon})) \leq \epsilon /2 + \epsilon/4 + \epsilon/8. \qedhere
\end{equation*}
\end{proof}

\medskip

\chapter{Exact categories}\label{ExactCat}

Exact categories play a key role in Quillen's foundational work \cite{Quillen73} on higher algebraic $K$-theory, who introduced the terminology. As shown notably by Deligne and Keller, exact categories admit a formalism of derived categories that includes the formalism of derived categories of abelian categories, and is especially convenient for applications.  

This Appendix gathers the basic definitions and summarizes, without proof, some basic properties of exact categories in the sense of Quillen, for the commodity of the reader of Section \ref{proexact}. We refer the reader to the expositions by  Keller  \cite{Keller1996} and B\"uhler  \cite{Buehler2010} for the relevant proofs and for references.

We only emphasize that the historical development of the formalism of exact categories has been rather intricate. Notably, the axioms defining exact categories in \cite{Quillen73} may be significantly streamlined, as shown by the works of Yoneda and Keller (see \cite{Buehler2010}, Section 2). Besides, an important forerunner of Quillen's notion of exact category has been the theory of ``abelian categories" developed by Heller \cite{Heller58}. These categories are precisely the exact categories, in Quilllen's sense, the underlying additive category of which are idempotent complete (\cf  \cite{Buehler2010}, Appendix B), and Heller already showed that these idempotent complete exact categories did constitute a convenient framework for homological algebra.  

\section{Definitions and basic properties} 

\subsection{Definitions}
Let $\cA$ be an additive category. A \emph{kernel-cokernel} pair $(i,p)$ in $\cA$ is a pair of composable morphisms
$$A' \stackrel{i}{\lra} A  \stackrel{p}{\lra} A''$$
in $\cA$ such that $i$ is a kernel of $p$ and $p$ is a cokernel of $i$.

Let $\cE$ be a class of kernel-cokernel pairs in $\cA$.
An \emph{allowable monomorphism} (with respect to $\cE$) is a morphism $i$ in $\cA$ for which there exists a morphism $p$ in $\cA$ such that $(i,p)$ belongs to $\cE$.
An \emph{allowable epimorphism} (with respect to $\cE$) is a morphism $p$ in $\cA$ for which there exists a morphism $i$ in $\cA$ such that $(i,p)$ belongs to $\cE$.

An \emph{exact structure} on $\cA$ is a class $\cE$ of kernel-cokernel pairs which is closed under isomorphisms and satisfies the following conditions:

\emph{$\mathbf{E_0 :}$ For every object $A$ of $\cA,$ the identity morphism $1_A$ is an allowable monomorphism.}

\emph{$\mathbf{E_0^{\rm op} :}$ For every object $A$ of $\cA,$ the identity morphism $1_A$ is an allowable epimorphism.}

\emph{$\mathbf{E_1:}$ The class of allowable monomorphisms is closed under composition.}

\emph{$\mathbf{E_1^{\rm op} :}$ The class of allowable epimorphisms is closed under composition.}

\emph{$\mathbf{E_2:}$ The push-out of an allowable monomorphism along an arbitrary morphism exists and yields an allowable monomorphism.}

\emph{$\mathbf{E_2^{\rm op} :}$ The pull-back of an allowable epimorphism along an arbitrary morphism exists and yields an allowable epimorphism.}

An \emph{exact category} is a pair $(\cA,\cE)$ consisting of an additive category and of some exact structure $\cE$ in $\cA$. Elements $(i,p)$ of $\cE$ are often displayed as
$$0 \lra A' \stackrel{i}{\lra} A  \stackrel{p}{\lra} A'' \lra 0$$
and called \emph{allowable short exact sequences} of the exact category, or simply \emph{short exact sequences} when no confusion may arise.

\subsection{Some properties of allowable monomorphisms and epimorphisms}

In any exact category $\cA$ as above, the following properties are satisfied (see \cite{Buehler2010}, Section 2):

(i) Any split exact sequence in $\cA$ --- namely any diagram in $\cA$ isomorphic to a diagram of the form
$$0 \lra M' \stackrel{(1_{M'},0)}{\lra} M' \oplus M'' \stackrel{\pr_2}{\lra} M'' \lra 0$$
for some objects $M'$ and $M''$ in $\cA$ --- is a short exact sequence.

(ii) The direct sum of two exact sequences is a short exact sequence.

(iii) The pull-back of an allowable monomorphism along an admissible epimorphism yields an allowable monomorphism. 

Dually, the push-out of an allowable epimorphism along an allowable monomorphism yields an allowable epimorphism.
 
 (iv) Suppose that $i: A \lra B$ is a morphism in $\cA$ admitting a cokernel. If there exists a morphism $j: B \lra C$ in $\cA$ such that $ji: A \lra C$ is an allowable monomorphism, then $i$ is an allowable monomorphism. 
 
 Dually, let $p: B' \lra A'$ be a morphism in $\cA$ admitting a kernel. If there exists a morphism $q:C' \lra B'$ in $\cA$ such that $pq: C' \lra A'$ is an allowable epimorphism, then $p$ is an allowable epic.
 
 In Quillen's definition of exact categories (\cite{Quillen73}, § 2), Property (iv) appears as an extra axiom besides (a simple variant of) axioms $\mathbf{E_{0-2}^{\rm (op)}}$.
 
 \subsection{Allowable morphisms and exactness}\label{Amorex} 
 
 In some exact category $\cA$, a morphism $f: A \lra B$ is said to be an \emph{allowable morphism} if it admits a factorization 
 \begin{equation}\label{fact}
 f = m\, e : A \stackrel{e}{\lra} I \stackrel{m}{\lra} B
 \end{equation}
 where $e$ is an allowable epimorphism and $m$ and allowable monomorphism. The factorization (\ref{fact}) is then unique, up to unique isomorphism.
 
 A diagram of allowable morphisms 
 $$A' \stackrel{f}{\lra} A \stackrel{f'}{\lra}A''$$
 is called \emph{exact} (at $A$) when the factorizations 
 $$ f = m\, e : A' \stackrel{e}{\lra} I \stackrel{m}{\lra} A \quad \mbox{ and }
  f' = m'\, e' : A \stackrel{e'}{\lra} I'\stackrel{m'}{\lra} A''$$
  of $f$ and $f'$ as products of allowable monomorphisms and epimorphisms are such that the diagram
  $$0 \lra I   \stackrel{m}{\lra} A \stackrel{e'}{\lra} I' \lra 0$$
  is a short exact sequence. 
 
 \section{The derived category of an exact category} 

\subsection{Definitions and Notation}\label{DefNot}

Let us recall that an additive category $\cA$ is said to be \emph{idempotent complete} when it satisfies the following equivalent conditions (see for instance \cite{Buehler2010}, 6.1-2):

$\mathbf{IC} :$ \emph{For every object $A$ of $\cA$ and any endomorphism $p: A \lra A$ which is idempotent --- \emph{that is, which satisfies $p^2 =p$} --- there is decomposition 
$$\phi : A \lrasim K \oplus I$$
of $A$ in $\cA$ such that}
$$\phi \, p \, \phi^{-1} = 
\begin{bmatrix}
 0 & 0 \\ 0 & {\rm Id}_A
\end{bmatrix}.$$
  
$\mathbf{IC'} :$ \emph{Every idempotent endomorphism in $\cA$ has a kernel}.

For any additive category $\cA$, we shall denote the additive category of complexes and chain maps in $\cA$ by $\mathbf{Ch}(\cA)$, and the homotopy category of $\cA$ by $\mathbf{K}(\cA)$. 

The category $\mathbf{K}(\cA)$ has the same objects as  $\mathbf{Ch}(\cA)$, and its morphisms are the chain maps, modulo the chain maps homotopic to zero. It is equipped with a natural structure of triangulated category, with suspension functor the ``shift" functor $\cdot[1]$, and with exact triangles the mapping cones of chain maps in  $\mathbf{Ch}(\cA)$, up to isomorphisms in $\mathbf{K}(\cA)$.

\subsection{The derived category $\mathbf{D}(\cA)$}\label{Dex} We finally review the definition and some basic properties of the derived category of some exact category. We refer the reader to \cite{Keller1996} and \cite{Buehler2010}, Section 10, for details of this construction.

  Let $\cA$ be some exact category. 

\subsubsection{}\label{acyclicquasiiso} A complex
\begin{equation}\label{compA}
(A^\bullet, d^\bullet):  \quad \cdots \lra A^{n-1} \stackrel{d^{n-1}}{\lra} A^n  \stackrel{d^{n}}{\lra} A^{n+1} \lra \cdots
\end{equation}
in $\mathbf{Ch}(\cA)$ is called \emph{acyclic}, or \emph{exact}, when the morphisms $(d^n)_{n \in \Z}$ are admissible and when, for every $n \in \Z,$ the diagram 
$$A^{n-1} \stackrel{d^{n-1}}{\lra} A^n  \stackrel{d^{n}}{\lra} A^{n+1}$$
is exact (at $A^n$).

Equivalently, the complex  (\ref{compA}) is acyclic when it may be obtained by ``splicing together" a sequence of short exact sequences in $\cA$:
$$0 \lra Z^n A \lra A^n \lra Z^{n+1} A \lra 0, \quad n \in \Z.$$

The mapping cone of a chain map between acyclic complexes is acyclic, and the class $\mathbf{Ac}(\cA)$ of acyclic complexes in $\cA$ is a triangulated subcategory of $\mathbf{K}(\cA)$.

A chain map in $\mathbf{Ch}(\cA)$ is called a \emph{quasi-isomorphism} when its mapping cone is homotopy equivalent to an acyclic complex. 

\subsubsection{}\label{defder}  The \emph{derived category} $\mathbf{D}(\cA)$ of the exact category is defined as the Verdier quotient 
$$\mathbf{D}(\cA) := \mathbf{K}(\cA)/\mathbf{Ac}(\cA).$$

If $\ast$ is an element of $\{+, -, b\}$, one may also define the full subcategory $\mathbf{D}^\ast(\cA)$ of $\mathbf{D}(\cA)$ formed by the complexes which are acyclic in degree $n$ for all $n<<0$, resp. all $n>>0,$ resp. all $n>>0$ and all $n<<0.$  It may be identified with the Verdier quotient $\mathbf{K}^\ast(\cA)/\mathbf{Ac}^\ast(\cA)$ where $\mathbf{K}^\ast(\cA)$ (resp. $\mathbf{Ac}^\ast(\cA)$) is the full subcategory of $\mathbf{K}(\cA)$ (resp. $\mathbf{Ac}(\cA)$) defined by complexes satisfying the boundedness condition $\ast.$

\subsubsection{}\label{Karoub}  When the exact category $\cA$ is \emph{idempotent complete}, the following properties hold:

(i) a retract in $\mathbf{K}(\cA)$ of an acyclic complex is acyclic;

(ii) the class of acyclic complexes is closed under isomorphisms in  $\mathbf{K}(\cA)$;  notably, 
any null homotopic complex in $\mathbf{Ch}(\cA)$ is acyclic;

(iii) a chain map in is a $\mathbf{Ch}(\cA)$ becomes an isomorphism in $\mathbf{D}(\cA)$ if and only if  it is a quasi-isomorphism, and if and only if  its mapping cone is acyclic.

\chapter[Holomorphic sections of line bundles over compact complex manifolds]{Upper bounds on the dimension of spaces of holomorphic sections of line bundles over compact complex manifolds}\label{ApUpBd}

\medskip

\section{Spaces of sections of analytic line bundles and multiplicity bounds} In this Appendix, we present two proofs of Proposition \ref{holbnd}, that we rephrase as follows:

\begin{proposition}\label{holbndBis}
  
For any analytic line bundle $L$ over a compact complex manifold $M$ of
 complex  dimension $n$, there exists $C$ in $\R^\ast_{+}$ such that, for any 
  positive integer $D$, the dimension of the vector space
  $\Gamma(M,L^{\otimes D})$ of analytic sections of $L^{\otimes D}$ 
  satisfies:
  $$\dim_{\C} \Gamma(M,L^{\otimes D}) \leq C. D^n.$$

 \end{proposition}
 
 The first proof follows the arguments of Serre (\cite{Serre53}) and Siegel (\cite{Siegel55}), and relies on the use of Schwarz Lemma,
after introducing suitable coordinate charts on $M$.

The second proof will assume that the manifold $M$ is K\"ahler and relies on some simple computations of intersection numbers.

To establish Proposition \ref{holbndBis}, we may clearly --- and we shall --- assume that $M$ is connected. For any $R\in \R_+,$ we denote
$$\B^n(R) := \{(z_1,\ldots,z_n) \in \C^n \mid \vert z_1 \vert^2 + \ldots + \vert z_n \vert^2 < R^2 \}.$$

By compactness, there exists a finite family of analytic charts on $M$
$$\varphi_{\alpha}:U_{\alpha}\stackrel{\sim}{\longrightarrow} \B^n(1),
\;\;\; 1 \leq \alpha \leq N$$
--- the $U_{\alpha}$ are open subsets of $M,$ and the $\varphi_{\alpha}$ 
analytic diffeomorphisms --- such that 
$$X = \bigcup_{1\leq \alpha \leq A} U_{\alpha}.$$
For every $\alpha$ in $\{1,\ldots,A\},$ we may consider the ``center
of the chart $\varphi_{\alpha}$":
$$P_{\alpha}:=\varphi_{\alpha}^{-1}(0).$$

Proposition \ref{holbndBis} will be a straightforward consequence of the following result, of independent interest:

\begin{proposition}\label{holbmult}
    With the above notation, for every holomorphic line bundle $L$
    over $M$, there exists $c$ in $\R_{+}$ such that, for any
    non-negative integer $D$ and any holomorphic section $s$ of
    $L^{\otimes D}$ over $M$, if
    $$\mult_{P_{\alpha}} s > c.D \;\;\;\mbox{ for every $\alpha$ in
    }\{1,\ldots,A\},$$
    then $s$ vanishes everywhere on $M.$
\end{proposition}

When $M$ is K\"ahler, the following stronger variant of Proposition \ref{holbmult} holds:

\begin{proposition}\label{holbmultK}
    Let us assume that the connected compact complex $M$ is K\"ahler.  
    
    For every point $P$ of $M$ and every holomorphic line bundle $L$
    over $M$, there exists $c$ in $\R_{+}$ such that, for any
    non-negative integer $D$ and any holomorphic section $s$ of
    $L^{\otimes D}$ over $M$, if
    $$\mult_{P} s > c.D,$$
    then $s$ vanishes everywhere on $M.$
\end{proposition}

\begin{proof}[Proof of Proposition \ref{holbndBis} from Proposition \ref{holbmult}] Again, we may introduce charts $(\varphi_{\alpha})_{1\leq
     \alpha \leq A}$ and their centers $(P_{\alpha})_{1\leq \alpha
     \leq A}$ as above. If  $M_{P_{\alpha},i}$ denotes the infinitesimal
      neighborhood of order $i$ of $P_{\alpha}$ in $M$, the space
      $\Gamma(M_{P_{\alpha},i}, L^{\otimes D})$ of
      sections of $L^{\otimes D}$ over $M_{P_{\alpha},i}$ --- that is,
      the space of jets of order $i$ of sections of $L^{\otimes D}$
      at $P_{\alpha}$ --- Proposition \ref{holbmult} precisely asserts
      the existence of $c$ in $\R^\ast_{+}$ such
      that the natural evaluation map:
      $$\eta_{D}^i: \Gamma(M, L^{\otimes D}) \longrightarrow
      \bigoplus_{1\leq \alpha \leq A} \Gamma(M_{P_{\alpha},i}, L^{\otimes
      D})$$
      is injective if $i > \lfloor cD\rfloor.$
      
      Consequently, if we let $i(D):=\lfloor cD\rfloor +1,$ we get
      $$\dim_{\C} \Gamma(M,L^{\otimes D}) \leq \sum_{1\leq \alpha \leq A}
      \dim_{\C} \Gamma(M_{P_{\alpha},i}, L^{\otimes
      D}) = A \binom{n+i(D)}{n}.$$
      
      Since, when $D$ goes to $+\infty,$
      $$\binom{n+i(D)}{n} \sim \frac{c^n}{n!} D^n,$$
      this establishes the announced upper bound.
      \end{proof}

\section[Proof of Proposition E.1.2]{Proof of Proposition \ref{holbmult}}\label{ProofHolMult} We will rely on the following higher
dimensional version of the well-known Schwarz Lemma, combined with
compactness arguments:

\begin{lemma}\label{Schwarz}
    Let $R$ be a positive real number and $i$ a non-negative integer. 
    For any holomorphic function $f$ on $\B^n(R)$ such that 
    $$\mult_{0}f \geq i$$
    and for any $z$ in $\B^n(R),$ we have:
    $$|f(z)| \leq  \left(\frac{\|z\|}{R}\right)^i \sup_{w \in
    \B^n(R)}|f(w)|.$$
    \end{lemma}

    \Proof{} Let $f$ be such a holomorphic function satisfying $\mult_{0}f \geq
    i$, and let $z$ be a point in $\B^n(R).$
    
    When $z=0,$ the inequality is immediate. Let us
    therefore assume $z$ non-zero, and let us introduce the function
    $g_{z}$ of
    one complex $t$ defined by
    $$g_{z}(t):= t^{-i} f(t .z/\|z\|).$$
    It is \emph{a priori} defined and analytic on the punctured disk
    $D(0,R)\setminus\{0\}.$ According to our assumption on
    $\mult_{0}f,$ it actually extends to an analytic function on
    $D(0,R).$ Consequently, using the maximum modulus principle and the very definition
    of $g_{z},$ we obtain, for every $t$
    in $D(0,R)$:
    $$|g_{z}(t)|\leq \limsup_{|w|\rightarrow R}|g_{z}(w)|
    \leq R^{-i}\sup_{w \in
    \B^n(R)}|f(w)|.$$
    The required upper-bound on $|f(z)|$ follows from this inequality 
    applied to $t=\Vert z\Vert.$
    \qed 

\medskip
Let us choose a continuous Hermitian metric
$\left\|.\right\|_{L}$ on $L,$ and, for every $\alpha$ in
$\{1,\ldots,A\},$ a non-vanishing holomorphic section $s_{\alpha}$ of 
$L$ over $U_{\alpha}$\footnote{The existence of such a section is equivalent 
to the triviality of the holomorphic line bundle $L$ over
$U_{\alpha}$. This follows from the isomorphism
$$\varphi_{\alpha}^\ast: H^1(U_{\alpha}, \cO^\ast) \simeq H^1(\B^n,
\cO^\ast),$$
 and from the vanishing of $H^1(\B^n,\cO^\ast),$ itself a
 straightfoward consequence of the vanishing of $H^1(\B^n,\cO)$ and of
 $H^2(\B^n, \Z)$, thanks to the short exact sequence of sheaves on $M$:
 $$0\longrightarrow Z \longrightarrow \cO \stackrel{\exp (2\pi
 i.)}{\longrightarrow} \cO^\ast \longrightarrow 0$$
 and to the associated long exact sequence of cohomology groups.
 
 Observe also that, for a given line bundle $L$, it is straightforward that such
 $s_{\alpha}$ exist if the domains $U_{\alpha}$ of the charts
 $\varphi_{\alpha}$ are  small enough. This observation makes
 particularly elementary the proof of  
 vanishing statement in
 Proposition \ref{holbmult} for one specific line bundle $L$ and for 
 a suitable choice of the $P_{\alpha}$ only --- this weaker version is
 sufficient to derive  Corollary \ref{holbnd} below and Chow's Theorem.}.
 
 By compactness of $M$ again, there exist $r$ in $]0,1[$ such that
 $$M= \bigcup_{1\leq \alpha \leq A} U_{\alpha}(r),$$
 where $U_{\alpha}:=\varphi_{\alpha}^{-1}(\B^n(r)).$
 
Finally, let us choose a real number $r'$ in the interval $]r,1[.$
 
Let $D$ be a non-negative integer, and $s$ a holomorphic section of 
$L^{\otimes D}$ over $L$. We shall denote 
$\left\|.\right\|_{L^{\otimes D}}$ the continous Hermitian metric on 
$L^{\otimes D}$ defined as the $D$-th tensor power of the metric 
$\left\|.\right\|_{L}$ on $L$.

For any $\alpha$ in $\{1,\ldots, A\}$, we may write
$$S_{U_{\alpha}} = f_{\alpha}\circ \varphi_{\alpha}. 
s_{\alpha}^{\otimes D},$$
where $\varphi_{\alpha}$ denotes a holomorphic function on $\Bn$. 
The following estimates are straightforward:
for any $P$ in $U_{\alpha}(r),$
\begin{equation}\label{etdeun}
    \left\|s(P)\right\|_{L^{\otimes D}} \leq \sup_{Q \in 
    U_{\alpha}(r)} \left\|s_{\alpha}(Q)\right\|_{L}^D 
    |f_{\alpha}(\phi_{\alpha}(P))|,
    \end{equation}
and for any $w$ in  $\Bn (r')$,
\begin{equation}\label{etdedeux}
  |f_{\alpha}(w)| \leq \sup_{Q \in 
    U_{\alpha}(r')} \left\|s_{\alpha}(Q)\right\|_{L}^{-D} 
    \left\|s_{\alpha}(\varphi_{\alpha}^{-1}(w))\right\|_{L^{\otimes D}}.
    \end{equation}
    
    Moreover, if $\mult_{0}f_{\alpha}$ --- or equivalently 
    $\mult_{P_{\alpha}} s$ --- is at least $i,$ the Schwarz Lemma 
    \ref{Schwarz} shows that
    \begin{equation*} 
|f_{\alpha}(\phi_{\alpha}(P))| \leq \left(\frac{r}{r'}\right)^i
\sup_{w \in \Bn(r)}  |f_{\alpha}(w)|.
\end{equation*}
Combining this inequality with (\ref{etdeun}) and (\ref{etdedeux}), 
we then have:
$$\left\|s(P)\right\|_{L^{\otimes D}} \leq
\left(\frac{r}{r'}\right)^i
\left(\frac{\sup_{Q \in 
    U_{\alpha}(r)} \left\|s_{\alpha}(Q)\right\|_{L}}{
    \inf_{Q \in 
    U_{\alpha}(r')} \left\|s_{\alpha}(Q)\right\|_{L}}
    \right)^D
    \sup_{Q \in U_{\alpha}(r')} \left\|s(Q)\right\|_{L^{\otimes D}}.$$
    
    Finally we obtain that, if
$\mult_{P_{\alpha}} s \geq i$ for every $\alpha$ in
    $\{1,\ldots,A\},$
    then
    \begin{equation}\label{lm}
	\max_{P \in M} \left\|s(P)\right\|_{L^{\otimes D}}
	\leq \lambda^i M^D \max_{P \in M} 
	\left\|s(P)\right\|_{L^{\otimes D}},
	\end{equation}
	where
	$\lambda := r/r'$ 	and
	$$ M:= \max_{1\leq \alpha \leq A} \frac{\sup_{Q \in 
    U_{\alpha}(r)} \left\|s_{\alpha}(Q)\right\|_{L}}{
    \inf_{Q \in 
    U_{\alpha}(r')} \left\|s_{\alpha}(Q)\right\|_{L}}.$$
    
    By their very definition, $\lambda$ belongs to $]0,1[$, and $M$ 
    to $[1,+\infty[$, and
    $$c:=\frac{\log M}{\log \lambda^{-1}}$$ is a well-defined non-negative real number.
    The upper-bound (\ref{lm}) implies the vanishing of $s$ when
    $\lambda^i M^D < 1,$
    that is, when $ i > c. D.$

   \section[Proof of Proposition E.1.3]{Proof of Proposition \ref{holbmultK}}
   
   Let $M$ be a compact connected complex manifold of dimension $n$, equipped with some K\"ahler form $\omega$.
   
   Let $P$ be a point in $M$ and let
   $$\nu: \tilde{M} \lra M$$ 
   denote the blow-up of $P$ in $M$, and 
   $$E:= \nu^{-1}(P)$$
   its exceptional divisor.

\begin{lemma}\label{exlambda} There exists a $C^\infty$ Hermitian metric $\Vert.\Vert$ on the line bundle $\cO(-E)$ over $\tilde{M}$ and a  positive real number $\lambda$ such the first Chern form
$$\alpha:= c_1(\cO(-E), \Vert.\Vert)$$
satisfies:
\begin{equation}\label{alphaomega}
\alpha + \lambda \, \nu^\ast \omega \geq 0.
\end{equation}
 \end{lemma}
 
 \proof  Let us consider  the blow-up 
 $$\nu_0: \widetilde{\A_\C^n} \lra \A_\C^n$$
at the origin $0$ of the $n$-dimensional complex affine space $\A_\C^n$.

The variety $\widetilde{\A_\C^n}$ may be identified with the smooth subscheme of $\A^n_\C \times \PP^{n-1}_\C$ defined by
$$((z_1,\ldots,z_n), (w_1:\ldots:w_n)) \in \tilde{\A_\C^n} \Longleftrightarrow \mbox{ for every $1\leq i < j \leq n, $} 
\begin{vmatrix} z_i & w_i \\z_j & w_j \end{vmatrix} =0.$$
The map $\nu_0$ coincides with the restriction of the first projection 
${\rm pr}_1 :  \A^n_\C \times \PP^{n-1}_\C \lra   \PP^{n-1}_\C,$
and the exceptional divisor $E:= \nu_0^{-1}(0)$ with $\{0\} \times \PP^{n-1}_\C.$  Moreover, the line bundle $\cO_{\widetilde{\A_\C^n}} (-E)$ is canonically isomorphic to the pull-back ${\rm pr}_{2\mid \widetilde{\A_\C^n}}^\ast \cO_{\PP^{n-1}_\C}(1)$ of the tautological line bundle  $\cO_{\PP^{n-1}_\C} (1)$ by the restriction of the second projection
 ${\rm pr}_2 :  \A^n_\C \times \PP^{n-1}_\C \lra \A^n_\C.$
 (The canonical isomorphism maps ${\rm pr}_{2\mid \widetilde{\A_\C^n}}X_i$ to the pull-back under ${\rm pr}_{1\mid \widetilde{\A_\C^n}}$ of the $i$-th coordinate on $\A^n_\C$, for every $i \in \{1, \ldots, n\}.$)

 Consequently, like $\cO_{\PP^{n-1}_\C(1)}$, the line bundle  $\cO_{\widetilde{\A_\C^n}} (-E)$ may be endowed with some $C^\infty$ Hermitian metric $\Vert.\Vert_0$ such that
 $c_1(\cO_{\widetilde{\A_\C^n}} (-E), \Vert.\Vert_0) \geq 0$
 over $\widetilde{\A_\C^n}$. 
 
 Let finally $\Vert.\Vert_1$ be a $C^\infty$ metric on $\cO_{\widetilde{\A_\C^n}} (-E)$ which coincides with $\Vert.\Vert_0$ over 
 $\nu_0^{-1}(\B^n(1/2))$ and satisfies
 $\Vert \mathbf{1}_{\cO_{\widetilde{\A_\C^n}}(-E)} \Vert_1 = 1$
 over $\nu_O^{-1}(\A_\C^n \setminus \B^n(3/4)).$
 
 Let us choose a $\C$-analytic chart $\varphi: U \lrasim \B^n(1)$ on $M$,
 such that $\varphi^{-1}(O)=P.$ By means of $\varphi,$ we may transport the metric $\Vert.\Vert_1$ on $\cO_{\widetilde{\A_\C^n}} (-E)$ over ${\nu_0^{-1}(\B^n(1))}$ to a metric on $\cO_{\tilde{M}}(-E)$ over $\nu^{-1}(U)$.  The metric $\Vert . \Vert$ satisfies 
 \begin{equation}\label{eq:11}
\Vert \mathbf{1}_{\cO_{\tilde{M}}(-E)} \Vert = 1
\end{equation}
over $U \setminus \varphi^{-1}(\B^n(3/4))$, and therefore may be extended to a metric (still denoted $\Vert.\Vert$)  over $\tilde{M}$ by requiring (\ref{eq:11}) to hold over $\tilde{M} \setminus
\varphi^{-1}(\B^n(3/4))$.

By construction, the first Chern form $c_1(\cO_{\tilde{M}}(-E), \Vert. \Vert)$ is $\geq 0$ on $\tilde{M} \setminus \nu^{-1}(K)$, where 
$$K := \varphi^{-1} (\overline{\B^n}(3/4) \setminus \B^n(1/2)).$$
Since $\nu$ is an isomorphism over $M\setminus \{P\}$ --- that contains the compact $K$ --- and $\omega$ is everywhere $>0$, the condition (\ref{alphaomega}) holds for any $\lambda$ large enough. 
 \qed
 
 Let us consider the cohomology class $[\omega]$ of $\omega$. For any line bundle $L$ over $M$, we may consider its
 first Chern class $c_1(L)$ in the real cohomology group $H^2(M, \R)$ and form the intersection number:
 $$c_1(L). [\omega]^{d-1}
  \in H^{2n}(M, \R) \simeq \R.$$
 
We may now establish Proposition \ref{holbmultK} in the following more precise form:

\begin{lemma} Let $\lambda$ be as in Lemma \ref{exlambda}. For any analytic line bundle $L$ over $M$, any positive integer $D$, and any analytic section $s$ of $L^{\otimes D}$ over $M$ that does not vanish identically, the order of vanishing $\mult_P s $ of $s$ at $P$ satisfies the following upper bound:
\begin{equation}\label{upmultP}
\mult_P s \leq \frac{c_1(L).[\omega]^{d-1}}{\lambda^{d-1}} D.
\end{equation}
\end{lemma}
         
 \proof  Consider some section $s$ in $\Gamma(M, L^{\otimes D})\setminus \{0\}$ and its divisor $\div s $ on $M$. It is effective, and $\mult_P$ is the multiplicity of the exceptional divisor $E$ in its inverse image $\nu^\ast \div s$ on $\tilde{M}$. 
 
  The divisor on $\tilde{M}$
  $$Z := \nu^\ast \div s - \mult_P s. E$$
  is therefore effective. Together with  (\ref{alphaomega}), this shows that
  \begin{equation}\label{intpos}
  \int_{\tilde{M}} (\alpha + \lambda \nu^\ast)^{d-1} \delta_Z \geq 0.
  \end{equation}
  
If $[E]$ and $[Z]$ denote the cohomology classes in $H^2(\tilde{M}, \R)$ of the divisors $E$ and $Z$, the following equality of cohomology classes hold:
$$[\alpha + \lambda \nu^\ast \omega] = -[E]+ \lambda \nu^\ast [\omega] \mbox{ and }[\delta_\Z] = \nu^\ast c_1(L) -\mult_P s .[E],$$
and the integral in (\ref{intpos}) may be written as an intersection number:
\begin{equation}\label{bigint}
\int_{\tilde{M}} (\alpha + \lambda \nu^\ast)^{d-1} \delta_Z  = (-[E]+ \lambda \nu^\ast [\omega])^{d-1}.(D \,\nu^\ast c_1(L) -\mult_P s .[E]).
\end{equation}

Since $\nu$ is a birational morphism, and 
$$\nu_\ast [E]^i =0 \mbox{ if $i<d$}$$
 and $$(-1)^{d-1}[E]^{d} = c_1(\cO(_E))^{d-1}.[E]= 1,$$
the intersection number in the left-hand side of (\ref{bigint}) also equals:
$$D \, c_1(L).[\omega]^{d-1} - \lambda^{d-1} \mult_P S.$$

Together with (\ref{intpos}), this establishes the upper-bound (\ref{upmultP}).
\qed 

    The previous proof is a counterpart, in the framework of K\"ahler geometry, of the derivation of the basic properties of the Seshadri constants of nef line bundles on algebraic varieties 
     (see \cite{Lazarsfeld04}, Chapter 5, notably Proposition 5.1.9, and \cite{Bost04}, Lemma 2.3).
 
 \medskip

 \chapter{John ellipsoids and finite dimensional normed spaces}\label{ApJohn}
 
 \medskip

 \section{John ellipsoids and John Euclidean norms} Let $E$ be a finite dimensional real vector space, equipped with some norm $\Vert .\Vert,$ and let 
 $$B:= \{ x \in E \mid \Vert x \Vert \leq 1 \}$$
 be its closed unit ball.
 
 A simple compactness argument shows that, among the ellipsoids of center 0 in $E$ contained in $B$, there is one of maximal volume. As shown by John 
\cite{John48}, this ellipsoid $J$ satisfies 
$$B \subset (\dim_\R E)^{1/2} J.$$
Moreover this ellipsoid is unique\footnote{This unicity result appears to have been observed independently by Loewner (in a dual form, see \cite{Busemann50}, p. 159-160), Danzer, Laugwitz and Lenz (\cite{DanzerLaugwitzLenz57}) and Zaguskin (\cite{Zaguskin58}).}. It is called the \emph{John ellipsoid} of $B$.

These results may be translated in terms of Euclidean norms as follows:

\begin{proposition}\label{prop:John}
Among the Euclidean norms  $\vert . \vert$ on $E$ such that $\Vert .\Vert \leq \vert. \vert,$ there exists a unique one --- the \emph{John Euclidean norm} $\Vert.\Vert_J$ attached to $\Vert.\Vert$ --- for which the associated Euclidean norm on $\wedge^{\dim E} E$ is minimal. The norm $\Vert .\Vert_J$ satisfies:
$$\Vert.\Vert \leq \Vert .\Vert_J \leq (\dim_\R)^{1/2} \Vert.\Vert.$$ 
\end{proposition}

\section{Properties of the John norm} Let us indicate a few properties of the John norm that are direct consequences of Proposition \ref{prop:John}.

\subsection{Invariance under automorphisms}

Let $K := {\rm Isom}(E, \Vert. \Vert)$ be the set of elements of $GL_\R(E)$ that preserve the norm $\Vert.\Vert.$ It is a compact subgroup of $GL_\R(E)$.
\emph{The Euclidean norm $\Vert. \Vert_J$ is  fixed under the action of $K$.}

This follows from the unicity assertion in Proposition  \ref{prop:John} and from the fact that the action of the compact group $K$ on $\wedge^{\dim E} E$ factors through $\{1, -1\} \hra \R^\ast = GL_\R (\wedge^{\dim E} E)$.

\subsection{The John norm attached to a normed complex vector spaces} Assume that $E$ is a complex vector space and that $\Vert.\Vert$ is a norm on this complex vector space. Then \emph{the John norm $\Vert.\Vert_J$} associated to $(E,\Vert.\Vert)$ seen as a normed real vector space \emph{is a Hermitian norm on $E$.}

Indeed, in this situation, the group  ${\rm Isom}(E, \Vert. \Vert)$ contains the operation $[i]$ of multiplication by $i$ on $E$, and therefore $\Vert .\Vert_J$ is invariant under the action of $[i]$, hence a Hermitian norm on $E$.

\subsection{Compatibility with finite products}\label{CompaJohnProd} Let $(E_1, \Vert.\Vert_1), \ldots, (E_N, \Vert.\Vert_N)$ be a finite family of finite dimensional normed vector spaces (over $\R$ or over $\C$). Let us consider the direct sum
$$E:= E_1 \oplus \ldots \oplus E_N$$
and the norm $\Vert .\Vert$ on $E$ defined by:
$$\Vert(e_1, \ldots, e_n)\Vert := \max_{1\leq i \leq n} \Vert e_i \Vert_i \quad \mbox{ for every $(e_1, \ldots, e_n) \in  E_1 \oplus \ldots \oplus E_N.$}$$

Then \emph{the John norms $\Vert.\Vert_J$ and $\Vert.\Vert_{1,J}, \ldots, \Vert.\Vert_{N,J}$ associated to $\Vert.\Vert$ and $\Vert.\Vert_{1}, \ldots, \Vert.\Vert_{N}$, satisfy the relation:}
\begin{equation}\label{JohnProd}
\Vert(e_1, \ldots, e_n)\Vert^2_J := \sum_{1\leq i \leq n} \Vert e_i \Vert_{i,J}^2.
\end{equation}

Indeed the group  ${\rm Isom}(E, \Vert. \Vert)$ contains the group $\{1,-1\}^N$ acting diagonally on $E_1 \oplus \ldots \oplus E_N$. Therefore the (Euclidean or hermitain) norm $\Vert \Vert$ is invariant by this subgroup, and therefore may be written:
$$\Vert(e_1, \ldots, e_n)\Vert_J := \left(\sum_{1\leq i \leq n} \Vert e_i \Vert_{J}^2\right)^{1/2}.$$
The expression (\ref{JohnProd}) readily follows from this observation.

\section{Application to lattices in normed real vector spaces}\label{defChi}

Let $\Gamma$ be a free $\Z$-module of finite rank, and let $\Vert. \Vert$ be some norm on the finite dimensional real vector space $\Gamma_\R := \Gamma \otimes_\Z \R.$

Let $\covol_{\Vert. \Vert} \Gamma$ denote the covolume of $\Gamma$  in $\Gamma_\R$ with respect to the Lebesgue measure\footnote{A Lebesgue measure on $\Gamma_\R$ is a non-zero, translation invariant Radon measure on $\Gamma_\R.$} on $\Gamma_R$ that gives the volume $1$ to the unit ball
$$B := \{ x \in \Gamma_\R \mid \Vert x \Vert \leq 1 \}.$$
In other words, for any Lebesgue measure $\lambda$ on $\Gamma_\R$ and for any Borel subset $\Delta$ of $\Gamma_\R$ that is a fundamental domain for the action of $\Gamma$,
\begin{equation}\label{covolnorm}
\covol_{\Vert. \Vert} \Gamma = \frac{\lambda (\Delta)}{\lambda(B)}.
\end{equation}

To the pair $(\Gamma, \Vert.\Vert),$ we may associate the real number
\begin{equation*}
\chi_{\Vert .\Vert} (\Gamma) := -\log \covol_{\Vert. \Vert} \Gamma .
\end{equation*}

When the norm $\Vert.\Vert$ is Euclidean, the pair $(\Gamma, \Vert.\Vert)$ define an Euclidean lattice $\oli{\Gamma}$, and from the expression (\ref{covolnorm}), applied to the Lebesgue measure $\lambda := \lambda_{\oli{\Gamma}}$ 
that gives the volume $1$ to the unit cube in the Euclidean space $(\Gamma_\R, \Vert.\Vert)$ (\cf Section \ref{PoissonSchwartz}), shows that:
\begin{equation*}
\covol_{\Vert .\Vert} \Gamma = v_{\rk \Gamma}^{-1}. \covol \oli{\Gamma},
\end{equation*}
where $v_{\rk E}$ denotes the volume of the unit ball in dimension $\rk E$.
Therefore
\begin{equation}\label{chidega}
\chi_{\Vert .\Vert}( \Gamma) = \dega \oli{\Gamma} + \log v_{\rk \Gamma}.
\end{equation}

In general, we may consider the John Euclidean norm $\Vert.\Vert_J$ on $\Gamma_\R$ associated to $\Vert.\Vert$ and form the Euclidean lattice
$$\oli{\Gamma}_J := (\Gamma, \Vert.\Vert_J).$$ 
If $B$ and $J$ denote the unit balls in $\Gamma_\R$ associated to $\Vert.\Vert$ and $\Vert.\Vert_J$ respectively, we have:
$$ J \subset B \subset (\rk \Gamma)^{1/2} J.$$
From these inclusions, we derive the following estimates, for any Lebesgue measure $\lambda$ on $\Gamma_\R$:
$$\lambda(J) \leq \lambda (B) \leq (\rk \Gamma)^{\rk \Gamma /2} \lambda (J).$$
Consequently, we have:
$$\chi_{\Vert . \Vert_J} (\Gamma) \leq  \chi_{\Vert.\Vert} (\Gamma) \leq (\rk \Gamma/2) \log \rk \Gamma + \chi_{\Vert.\Vert} (\Gamma).$$
Using the relation (\ref{chidega}), these estimates may also be written:
\begin{equation}\label{chidegaJ}
\dega \oli{\Gamma}_J + \log v_{\rk \Gamma} \leq \chi_{\Vert. \Vert}(\Gamma) \leq \dega \oli{\Gamma}_J + \log v_{\rk \Gamma} +  
(\rk \Gamma/2) \log \rk \Gamma.
\end{equation}

In practice, when using these estimates, it is useful to notice that, as a consequence of Stirling's formula, when $n$ goes to infinity:
\begin{equation}\label{asympvn}
\log v_n = \log\frac{\pi^{n/2}}{\Gamma(1+n/2)} = -(n/2) \log n + (n/2) (1+ \log 2\pi) + O(\log n).
\end{equation}
Actually, for every positive integer $n$, we have:
\begin{equation}\label{ineqvn} -(n/2) \log n + n \log 2 \leq \log v_n \leq -(n/2) \log n + (n/2) (1+ \log 2\pi).
\end{equation}



\begin{thebibliography}{{Bou}74}

\bibitem[Ami75]{amice75}
Y.~Amice.
\newblock {\em Les nombres {$p$}-adiques}.
\newblock Presses Universitaires de France, Paris, 1975.

\bibitem[And63]{Andreotti63}
A.~Andreotti.
\newblock Th\'eor\`emes de d\'ependance alg\'ebrique sur les espaces complexes
  pseudo-concaves.
\newblock {\em Bull. Soc. Math. France}, 91:1--38, 1963.

\bibitem[And89]{Andre89}
Y.~Andr{\'e}.
\newblock {\em ${G}$-functions and geometry}.
\newblock Friedr. Vieweg \& Sohn, Braunschweig, 1989.

\bibitem[And04]{Andre04}
Y.~Andr\'e.
\newblock Sur la conjecture des {$p$}-courbures de {G}rothendieck-{K}atz et un
  probl\`eme de {D}work.
\newblock In {\em Geometric aspects of {D}work theory. {V}ol. {I}, {II}}, pages
  55--112. Walter de Gruyter, Berlin, 2004.

\bibitem[{Ban}31]{Banach31}
S.~{Banach}.
\newblock {\"Uber metrische Gruppen.}
\newblock {\em {Stud. Math.}}, 3:101--113, 1931.

\bibitem[Ban91]{Banaszczyk91}
W.~Banaszczyk.
\newblock {\em Additive subgroups of topological vector spaces}, volume 1466 of
  {\em Lecture Notes in Mathematics}.
\newblock Springer-Verlag, Berlin, 1991.

\bibitem[Ban93]{Banaszczyk93}
W.~Banaszczyk.
\newblock New bounds in some transference theorems in the geometry of numbers.
\newblock {\em Math. Ann.}, 296(4):625--635, 1993.

\bibitem[BCL09]{BostChambert-Loir07}
J.-B. Bost and A.~Chambert-Loir.
\newblock Analytic curves in algebraic varieties over number fields.
\newblock In {\em Algebra, arithmetic, and geometry: in honor of {Y}u. {I}.
  {M}anin. {V}ol. {I}}, volume 269 of {\em Progr. Math.}, pages 69--124.
  Birkh\"auser, Boston, MA, 2009.

\bibitem[BGS94]{BostGilletSoule94}
J.-B. Bost, H.~Gillet, and C.~Soul{\'e}.
\newblock Heights of projective varieties and positive {G}reen forms.
\newblock {\em J. Amer. Math. Soc.}, 7(4):903--1027, 1994.

\bibitem[Bil99]{Billingsley99}
P.~Billingsley.
\newblock {\em Convergence of probability measures}.
\newblock Wiley Series in Probability and Statistics: Probability and
  Statistics. John Wiley \& Sons, Inc., New York, second edition, 1999.

\bibitem[BK10]{BK10}
J.-B. Bost and K.~K{\"u}nnemann.
\newblock Hermitian vector bundles and extension groups on arithmetic schemes.
  {I}. {G}eometry of numbers.
\newblock {\em Adv. Math.}, 223(3):987--1106, 2010.

\bibitem[Bol72]{Boltzmann72}
L.~Boltzmann.
\newblock {Weitere Studien \"uber W\"armegleichgewicht unter Gasmolek\"ulen}.
\newblock {\em Wien. Ber.}, 66:275--370, 1872.

\bibitem[Bol77]{Boltzmann77}
L.~Boltzmann.
\newblock {\"Uber die Beziehungen zwischen dem zweiten Hauptsatz der
  W\"armetheorie und der Warscheinlichkeitsrechnung respektive den S\"atzen
  \"uber das W\"armegleichgewicht}.
\newblock {\em Wien. Ber.}, 76:373--435, 1877.

\bibitem[Bor94]{Borel94}
E.~Borel.
\newblock Sur une application d'un th\'eor\`eme de {M}. {H}adamard.
\newblock {\em Bulletin des sciences math\'ematiques}, 18:22--25, 1894.

\bibitem[Bos96]{Bost96Bourbaki}
J.-B. Bost.
\newblock P\'eriodes et isog\'enies des vari\'et\'es ab\'eliennes sur les corps
  de nombres (d'apr\`es {D}. {M}asser et {G}. {W}\"ustholz). {S}\'eminaire
  {B}ourbaki 1994/95, {E}xp.\ no.\ 795.
\newblock {\em Ast\'erisque}, 237:115--161, 1996.

\bibitem[Bos99]{Bost99}
J.-B. Bost.
\newblock Potential theory and {L}efschetz theorems for arithmetic surfaces.
\newblock {\em Ann. Scient. \'{E}c. Norm. Sup.}, 32:241--312, 1999.

\bibitem[Bos01]{Bost01}
J.-B. Bost.
\newblock Algebraic leaves of algebraic foliations over number fields.
\newblock {\em Publ. Math. I.H.E.S.}, 93:161--221, 2001.

\bibitem[Bos04]{Bost04}
J.-B. Bost.
\newblock {Germs of analytic varieties in algebraic varieties: canonical
  metrics and arithmetic algebraization theorems.}
\newblock In {\em {A. Adolphson et al. (ed.), Geometric aspects of Dwork
  theory}}, volume~II, pages 371--418. Walter de Gruyter, Berlin, 2004.

\bibitem[Bos06]{Bost06}
J.-B. Bost.
\newblock Evaluation maps, slopes, and algebraicity criteria.
\newblock In {\em International {C}ongress of {M}athematicians. {V}ol. {II}},
  pages 537--562. Eur. Math. Soc., Z\"urich, 2006.

\bibitem[Bou65]{BourbakiAC7}
N.~Bourbaki.
\newblock {\em \'{E}l\'ements de math\'ematique. {F}asc. {XXXI}. {A}lg\`ebre
  commutative. {C}hapitre 7: {D}iviseurs}.
\newblock Actualit\'es Scientifiques et Industrielles, No. 1314. Hermann,
  Paris, 1965.

\bibitem[Bou69]{BourbakiINT9}
N.~Bourbaki.
\newblock {\em \'{E}l\'ements de math\'ematique. {F}asc. {XXXV}. {L}ivre {VI}:
  {I}nt\'egration. {C}hapitre {IX}: {I}nt\'egration sur les espaces
  topologiques s\'epar\'es}.
\newblock Actualit\'es Scientifiques et Industrielles, No. 1343. Hermann,
  Paris, 1969.

\bibitem[{Bou}74]{BourbakiTGII74}
N.~{Bourbaki}.
\newblock {\em {\'El\'ements de math\'ematique. Topologie g\'en\'erale. Chap. 5
  \`a 10. Nouvelle \'edition}}.
\newblock Hermann, Paris, 1974.

\bibitem[B{\"u}h10]{Buehler2010}
Th. B{\"u}hler.
\newblock Exact categories.
\newblock {\em Expo. Math.}, 28(1):1--69, 2010.

\bibitem[Bus50]{Busemann50}
H.~Busemann.
\newblock The foundations of {M}inkowskian geometry.
\newblock {\em Comment. Math. Helv.}, 24:156--187, 1950.

\bibitem[Can80]{Cantor1980}
D.~G. Cantor.
\newblock On an extension of the definition of transfinite diameter and some
  applications.
\newblock {\em Jour. Reine Angew. Math.}, 316:160--207, 1980.

\bibitem[Car71]{Cartier71}
P.~Cartier.
\newblock Groupes formels, fonctions automorphes et fonctions zeta des courbes
  elliptiques.
\newblock In {\em Actes du {C}ongr\`es {I}nternational des {M}ath\'ematiciens
  ({N}ice, 1970), {T}ome 2}, pages 291--299. Gauthier-Villars, Paris, 1971.

\bibitem[CC85a]{ChudnovskysGroth85}
D.~V. Chudnovsky and G.~V. Chudnovsky.
\newblock Applications of {P}ad\'e approximations to the {G}rothendieck
  conjecture on linear differential equations.
\newblock In {\em Number theory, Semin. New York 1983-84}, volume 1135 of {\em
  Lectures Notes in Mathematics}, pages 52--100. Springer, Berlin, 1985.

\bibitem[CC85b]{ChudnovskysAcad85}
D.~V. Chudnovsky and G.~V. Chudnovsky.
\newblock Pad\'e approximations and {D}iophantine geometry.
\newblock {\em Proc. Nat. Acad. Sci. U.S.A.}, 82(8):2212--2216, 1985.

\bibitem[Cha17]{Charles2017}
F.~Charles.
\newblock Arithmetic ampleness and an arithmetic {B}ertini theorem.
\newblock arXiv: 1703.02481 [math.AG], 2017.

\bibitem[Che52]{Chernoff52}
H.~Chernoff.
\newblock A measure of asymptotic efficiency for tests of a hypothesis based on
  the sum of observations.
\newblock {\em Ann. Math. Statistics}, 23:493--507, 1952.

\bibitem[Cho49]{Chow49}
W.-L. Chow.
\newblock On compact complex analytic varieties.
\newblock {\em Amer. J. Math.}, 71:893--914, 1949.

\bibitem[CL02]{Chambert01}
A.~Chambert-Loir.
\newblock Th\' eor\`emes d'alg\'ebricit\' e en g\'eom\'etrie diophantienne.
  {S}\'eminaire {B}ourbaki, {V}ol.\ 2000/2001, {E}xpos\'e 886.
\newblock {\em Ast\'erisque}, 282:175--209, 2002.

\bibitem[CP11]{CerfPetit2011}
R.~Cerf and P.~Petit.
\newblock A short proof of {C}ram\'er's theorem in {$\mathbb R$}.
\newblock {\em Amer. Math. Monthly}, 118(10):925--931, 2011.

\bibitem[{Cra}38]{Cramer38}
H.~{Cram\'er}.
\newblock {Sur un nouveau th\'eor\`eme-limite de la th\'eorie des
  probabilit\'es.}
\newblock In {\em Conf\'erences internationales de sciences mathématiques
  (Universit\'e de Gen\`eve 1937). Th\'eorie des probabilit\'es. III: Les
  sommes et les fonctions de variables al\'eatoires}, volume 736 of {\em
  Actualit\'es scientifiques et industrielles}, pages 5--23. Hermann, Paris,
  1938.

\bibitem[CS93]{ConwaySloane1993}
J.~H. Conway and N.~J.~A. Sloane.
\newblock {\em Sphere packings, lattices and groups}, volume 290 of {\em
  Grundlehren der Mathematischen Wissenschaften}.
\newblock Springer-Verlag, New York, second edition, 1993.
\newblock With additional contributions by E. Bannai, R. E. Borcherds, J.
  Leech, S. P. Norton, A. M. Odlyzko, R. A. Parker, L. Queen and B. B. Venkov.

\bibitem[Die50]{DieudonneCrelle50}
J.~Dieudonn{\'e}.
\newblock Matrices semi-finies et espaces localement lin\'eairement compacts.
\newblock {\em J. Reine Angew. Math.}, 188:162--166, 1950.

\bibitem[DLL57]{DanzerLaugwitzLenz57}
L.~Danzer, D.~Laugwitz, and H.~Lenz.
\newblock {\"U}ber das {L}\"ownersche {E}llipsoid und sein {A}nalogon unter den
  einem {E}ik\"orper einbeschriebenen {E}llipsoiden.
\newblock {\em Arch. Math. (Basel)}, 8:214--219, 1957.

\bibitem[Dri06]{Drinfeld2006}
V.~Drinfeld.
\newblock Infinite-dimensional vector bundles in algebraic geometry: an
  introduction.
\newblock In {\em The unity of mathematics}, volume 244 of {\em Progr. Math.},
  pages 263--304. Birkh\"auser Boston, Boston, MA, 2006.

\bibitem[DvdP92]{DworkvdPoorten92}
B.~M. Dwork and A.~J. van~der Poorten.
\newblock The {E}isenstein constant.
\newblock {\em Duke Math. J.}, 65(1):23--43, 1992.

\bibitem[Dwo60]{Dwork60}
B.~Dwork.
\newblock On the rationality of the zeta function of an algebraic variety.
\newblock {\em Amer. J. Math.}, 82:631--648, 1960.

\bibitem[Eis52]{Eisenstein52}
G.~Eisenstein.
\newblock Eine allgemeine {E}igenschaft der {R}eihen-{E}ntwicklungen aller
  algebraischen {F}unktionen.
\newblock {\em Bericht der K{\"o}nigl. Preuss. Akademie der Wissenschaften zu
  Berlin}, pages 441--443, 1852.

\bibitem[Ell85]{Ellis85}
R.~S. Ellis.
\newblock {\em Entropy, large deviations, and statistical mechanics}, volume
  271 of {\em Grundlehren der Mathematischen Wissenschaften}.
\newblock Springer-Verlag, New York, 1985.

\bibitem[Eno64]{Enochs64}
E.~E. Enochs.
\newblock A note on reflexive modules.
\newblock {\em Pacific J. Math.}, 14:879--881, 1964.

\bibitem[Fal83]{Faltings83}
G.~Faltings.
\newblock Endlichkeitss\"atze f\"ur abelsche {V}ariet\"aten \"uber
  {Z}ahlk\"orpern.
\newblock {\em Invent. Math.}, 73(3):349--366, 1983.

\bibitem[Fek23]{Fekete23}
M.~Fekete.
\newblock \"{U}ber die {V}erteilung der {W}urzeln bei gewissen algebraischen
  {G}leichungen mit ganzzahligen {K}oeffizienten.
\newblock {\em Math. Z.}, 17(1):228--249, 1923.

\bibitem[Gas99]{Gasbarri99}
C.~Gasbarri.
\newblock Hermitian vector bundles of rank two and adjoint systems on
  arithmetic surfaces.
\newblock {\em Topology}, 38(6):1161--1174, 1999.

\bibitem[GD71]{EGAI}
A.~Grothendieck and J.~A. Dieudonn{\'e}.
\newblock {\em El\'ements de g\'eom\'etrie alg\'ebrique. {I}}, volume 166 of
  {\em Grundlehren der Mathematischen Wissenschaften}.
\newblock Springer-Verlag, Berlin, 1971.

\bibitem[Geo03]{Georgii03}
H.-O. Georgii.
\newblock Probabilistic aspects of entropy.
\newblock In {\em Entropy}, Princeton Ser. Appl. Math., pages 37--54. Princeton
  Univ. Press, Princeton, NJ, 2003.

\bibitem[Gib60]{Gibbs05}
J.~W. Gibbs.
\newblock {\em Elementary principles in statistical mechanics: developed with
  especial reference to the rational foundation of thermodynamics}.
\newblock Dover publications, Inc., New York, 1960.
\newblock Republication of the work originally published by Yale University
  Press in 1902.

\bibitem[GMS91]{GilletMazurSoule1991}
H.~Gillet, B.~Mazur, and C.~Soul{\'e}.
\newblock A note on a classical theorem of {B}lichfeldt.
\newblock {\em Bull. London Math. Soc.}, 23(2):131--132, 1991.

\bibitem[GN76]{GuenotNarasimhan76}
J.~Guenot and R.~Narasimhan.
\newblock {\em Introduction \`a la th\'eorie des surfaces de {R}iemann}.
\newblock L'Enseignement Math\'ematique, Universit\'e de Gen\`eve, Geneva,
  1976.
\newblock Extrait de l'Enseignement Math. (2) {{\bf{2}}1} (1975), no. 2-4,
  123--328, Monographie No. 23 de L'Enseignement Math\'ematique.

\bibitem[Gra01]{Graftieaux2001a}
Ph. Graftieaux.
\newblock Formal groups and the isogeny theorem.
\newblock {\em Duke Math. J.}, 106(1):81--121, 2001.

\bibitem[Gro64]{EGAIV1}
A.~Grothendieck.
\newblock \'{E}l\'ements de g\'eom\'etrie alg\'ebrique. {IV}. \'{E}tude locale
  des sch\'emas et des morphismes de sch\'emas. {I}.
\newblock {\em Inst. Hautes \'Etudes Sci. Publ. Math.}, (20):259, 1964.

\bibitem[Gro01]{Groenewegen2001}
R.~P. Groenewegen.
\newblock An arithmetic analogue of {C}lifford's theorem.
\newblock {\em J. Th\'eor. Nombres Bordeaux}, 13(1):143--156, 2001.
\newblock 21st Journ{\'e}es Arithm{\'e}tiques (Rome, 2001).

\bibitem[GS91]{GilletSoule1991}
H.~Gillet and C.~Soul{\'e}.
\newblock On the number of lattice points in convex symmetric bodies and their
  duals.
\newblock {\em Israel J. Math.}, 74(2-3):347--357, 1991.

\bibitem[GS09]{GilletSoule2009}
H.~Gillet and C.~Soul{\'e}.
\newblock Erratum for: ``{O}n the number of lattice points in convex symmetric
  bodies and their duals'', {I}srael {J}ournal of {M}athematics {\bf 74}
  (1991), 347--357.
\newblock {\em Israel J. Math.}, 171:443--444, 2009.

\bibitem[Has12]{Hasse2012}
H.~Hasse.
\newblock {\em {Die mathematischen Tageb\"ucher von Helmut Hasse 1923--1935.}}
\newblock G\"ottingen: Universit\"atsverlag G\"ottingen, 2012.
\newblock Herausgegeben und kommentiert von Franz Lemmermeyer und Peter
  Roquette.

\bibitem[{Hec}17]{Hecke17}
E.~{Hecke}.
\newblock {\"Uber die Zetafunktion beliebiger algebraischer Zahlk\"orper.}
\newblock {\em {Nachr. Ges. Wiss. G\"ottingen, Math.-Phys. Kl.}}, 1917:77--89,
  1917.

\bibitem[{Hec}19]{Hecke19}
E.~{Hecke}.
\newblock {Reziprozit\"atsgesetz und Gau\ss 'sche Summen in quadratischen
  Zahlk\"orpern.}
\newblock {\em {Nachr. Ges. Wiss. G\"ottingen, Math.-Phys. Kl.}},
  1919:265--278, 1919.

\bibitem[{Hec}23]{Hecke23}
E.~{Hecke}.
\newblock {\em {Vorlesungen \"uber die Theorie der algebraischen Zahlen.}}
\newblock Akademische Verlagsgesellschaft, Leipzig, 1923.

\bibitem[Hel58]{Heller58}
A.~Heller.
\newblock Homological algebra in abelian categories.
\newblock {\em Ann. of Math. (2)}, 68:484--525, 1958.

\bibitem[{Hil}00]{Hilbert1900}
D.~{Hilbert}.
\newblock {Mathematische Probleme. Vortrag, gehalten auf dem internationalen
  Mathematiker-Congress zu Paris 1900.}
\newblock {\em {Nachr. Ges. Wiss. G\"ottingen, Math.-Phys. Kl.}},
  1900:253--297, 1900.

\bibitem[Hon68]{Honda68Formal}
T.~Honda.
\newblock Formal groups and zeta-functions.
\newblock {\em Osaka J. Math.}, 5:199--213, 1968.

\bibitem[Hon70]{Honda70}
T.~Honda.
\newblock On the theory of commutative formal groups.
\newblock {\em J. Math. Soc. Japan}, 22:213--246, 1970.

\bibitem[Joh48]{John48}
F.~John.
\newblock Extremum problems with inequalities as subsidiary conditions.
\newblock In {\em Studies and {E}ssays {P}resented to {R}. {C}ourant on his
  60th {B}irthday, {J}anuary 8, 1948}, pages 187--204. Interscience Publishers,
  Inc., New York, N. Y., 1948.

\bibitem[Kap52]{Kaplansky52}
I.~Kaplansky.
\newblock Modules over {D}edekind rings and valuation rings.
\newblock {\em Trans. Amer. Math. Soc.}, 72:327--340, 1952.

\bibitem[Kel96]{Keller1996}
Bernhard Keller.
\newblock Derived categories and their uses.
\newblock In {\em Handbook of algebra, {V}ol.\ 1}, volume~1 of {\em Handb.
  Algebr.}, pages 671--701. Elsevier/North-Holland, Amsterdam, 1996.

\bibitem[Khi49]{Khinchin49}
A.~I. Khinchin.
\newblock {\em Mathematical {F}oundations of {S}tatistical {M}echanics}.
\newblock Dover Publications, Inc., New York, N. Y., 1949.
\newblock Translated by G. Gamow.

\bibitem[{Kol}33]{Kolmogorov33}
A.~{Kolmogoroff}.
\newblock {\em {Grundbegriffe der Wahrscheinlichkeitsrechnung}}, volume~3 of
  {\em Ergebnisse der Math. und ihrer Grenzgebiete. 2}.
\newblock Julius Springer, Berlin, 1933.

\bibitem[K{\"o}t49]{Koethe49}
G.~K{\"o}the.
\newblock Eine axiomatische {K}ennzeichnung der linearen {R}\"aume vom {T}ypus
  {$\omega$}.
\newblock {\em Math. Ann.}, 120:634--649, 1949.

\bibitem[Kul97]{Kullback97}
S.~Kullback.
\newblock {\em Information theory and statistics}.
\newblock Dover Publications, Inc., Mineola, NY, 1997.
\newblock Reprint of the second (1968) edition.

\bibitem[Lan73]{Lanford73}
O.~E. Lanford.
\newblock Entropy and equilibrium states in classical statistical mechanics.
\newblock In A.~Lenard, editor, {\em Statistical mechanics and mathematical
  problems (Battelle Seattle 1971 Rencontres)}, pages 1--113. Lecture Notes in
  Physics, Vol. 20. Springer-Verlag, Berlin Heidelberg, 1973.

\bibitem[Laz55]{Lazard55}
M.~Lazard.
\newblock Sur les groupes de {L}ie formels \`a un param\`etre.
\newblock {\em Bull. Soc. Math. France}, 83:251--274, 1955.

\bibitem[Laz04]{Lazarsfeld04}
R.~Lazarsfeld.
\newblock {\em {Positivity in algebraic geometry. I. Classical setting: line
  bundles and linear series.}}
\newblock {Ergebnisse der Mathematik und ihrer Grenzgebiete 48. Berlin:
  Springer}, 2004.

\bibitem[Lef42]{Lefschetz42}
S.~Lefschetz.
\newblock {\em Algebraic {T}opology}, volume~27 of {\em American Mathematical
  Society Colloquium Publications}.
\newblock American Mathematical Society, New York, 1942.

\bibitem[LT65]{LubinTate65}
J.~Lubin and J.~Tate.
\newblock Formal complex multiplication in local fields.
\newblock {\em Ann. of Math. (2)}, 81:380--387, 1965.

\bibitem[Lub64]{Lubin64}
Jonathan Lubin.
\newblock One-parameter formal {L}ie groups over {${\mathfrak p}$}-adic integer
  rings.
\newblock {\em Ann. of Math. (2)}, 80:464--484, 1964.

\bibitem[Man85]{Manin85}
Yu.~I. Manin.
\newblock New dimensions in geometry.
\newblock In {\em Workshop {B}onn 1984 ({B}onn, 1984)}, volume 1111 of {\em
  Lecture Notes in Math.}, pages 59--101. Springer, Berlin, 1985.

\bibitem[MO90]{MazoOdlyzko90}
J.~E. Mazo and A.~M. Odlyzko.
\newblock Lattice points in high-dimensional spheres.
\newblock {\em Monatsh. Math.}, 110(1):47--61, 1990.

\bibitem[Mor95]{Morishita95}
M.~Morishita.
\newblock Integral representations of unramified {G}alois groups and matrix
  divisors over number fields.
\newblock {\em Osaka J. Math.}, 32(3):565--576, 1995.

\bibitem[MR58]{MacbeathRogers58}
A.~M. Macbeath and C.~A. Rogers.
\newblock Siegel's mean value theorem in the geometry of numbers.
\newblock {\em Proc. Cambridge Philos. Soc.}, 54:139--151, 1958.

\bibitem[MR07]{MicciancioRegev2007}
D.~Micciancio and O.~Regev.
\newblock Worst-case to average-case reductions based on {G}aussian measures.
\newblock {\em SIAM J. Comput.}, 37(1):267--302, 2007.

\bibitem[{Neu}92]{Neukirch92}
J.~{Neukirch}.
\newblock {\em {Algebraische Zahlentheorie.}}
\newblock Berlin etc.: Springer-Verlag, 1992.

\bibitem[Poi02]{Poincare02}
H.~Poincar\'e.
\newblock Sur les fonctions ab\'eliennes.
\newblock {\em Acta Math.}, 26:43--98, 1902.

\bibitem[P{\'o}l28]{Polya1928}
G.~P{\'o}lya.
\newblock {\"U}ber gewisse notwendige {D}eterminantenkriterien f\"ur die
  {F}ortsetzbarkeit einer {P}otenzreihe.
\newblock {\em Math. Ann.}, 99:687--706, 1928.

\bibitem[Qui]{QuillenNotebooks}
D.~Quillen.
\newblock {\emph{Quillen Notebooks 1968--2003}, edited by G.Luke and G. Segal}.
\newblock Published online by the Clay Mathematics Institute.
  {\tt{http://www.claymath.org/publications/quillen-notebooks}}.

\bibitem[Qui73]{Quillen73}
D.~Quillen.
\newblock Higher algebraic {$K$}-theory. {I}.
\newblock In {\em Algebraic {$K$}-theory, {I}: {H}igher {$K$}-theories ({P}roc.
  {C}onf., {B}attelle {M}emorial {I}nst., {S}eattle, {W}ash., 1972)}, volume
  341 of {\em Lectures Notes in Mathematics}, pages 85--147. Springer, Berlin,
  1973.

\bibitem[Ran95]{Ransford95}
T.~Ransford.
\newblock {\em Potential theory in the complex plane}.
\newblock Cambridge University Press, Cambridge, 1995.

\bibitem[RG71]{GrusonRaynaud71}
M.~Raynaud and L.~Gruson.
\newblock Crit\`eres de platitude et de projectivit\'e. {T}echniques de
  ``platification'' d'un module.
\newblock {\em Invent. Math.}, 13:1--89, 1971.

\bibitem[Rob73]{Robert73}
A.~Robert.
\newblock {\em Elliptic curves}.
\newblock Lecture Notes in Mathematics, Vol. 326. Springer-Verlag, Berlin-New
  York, 1973.
\newblock Notes from postgraduate lectures given in Lausanne 1971/72.

\bibitem[Roe93]{Roessler93}
D.~Roessler.
\newblock The {R}iemann-{R}och theorem for arithmetic curves.
\newblock Diplomarbeit, ETH Z\"urich, 1993.

\bibitem[{Sch}31]{Schmidt31}
F.~K. {Schmidt}.
\newblock {Analytische Zahlentheorie in K\"orpern der Charakteristik $p$.}
\newblock {\em {Math. Z.}}, 33:1--32, 1931.

\bibitem[Sch62]{Schroedinger62}
E.~Schr{\"o}dinger.
\newblock {\em Statistical thermodynamics}.
\newblock A course of seminar lectures delivered in January-March 1944, at the
  School of Theoretical Physics, Dublin Institute for Advanced Studies. Second
  edition, reprinted. Cambridge University Press, New York, 1962.

\bibitem[Sch73]{Schwartz73}
L.~Schwartz.
\newblock {\em Radon measures on arbitrary topological spaces and cylindrical
  measures}, volume~6 of {\em Tata Institute of Fundamental Research Studies in
  Mathematics}.
\newblock Oxford University Press, London, 1973.

\bibitem[Ser54]{Serre53}
J.~P. Serre.
\newblock Fonctions automorphes: quelques majorations dans le cas o\`u
  ${X}/{G}$ est compact.
\newblock {\em S\'eminaire H. Cartan}, 6, 1953--1954.
\newblock Expos\'e 2.

\bibitem[Ser68]{Serre68}
J.-P. Serre.
\newblock {\em Abelian $l$-adic representations and elliptic curves}.
\newblock W. A. Benjamin, Inc., New York-Amsterdam, 1968.

\bibitem[Sha77]{Shafarevich77}
I.R. Shafarevich.
\newblock {\em {Basic algebraic geometry. 2nd ed.}}
\newblock {Springer Study Edition. Berlin-Heidelberg-New York:
  Springer-Verlag}, 1977.

\bibitem[Sie45]{Siegel45}
C.~L. Siegel.
\newblock A mean value theorem in geometry of numbers.
\newblock {\em Ann. of Math. (2)}, 46:340--347, 1945.

\bibitem[Sie55]{Siegel55}
C.~L. Siegel.
\newblock Meromorphe {F}unktionen auf kompakten analytischen
  {M}annigfaltigkeiten.
\newblock {\em Nachr. Akad. Wiss. G\"ottingen. Math.-Phys. Kl. IIa.},
  1955:71--77, 1955.

\bibitem[Sil92]{Silverman92}
J.~H. Silverman.
\newblock {\em The arithmetic of elliptic curves}, volume 106 of {\em Graduate
  Texts in Mathematics}.
\newblock Springer-Verlag, New York, 1992.
\newblock Corrected reprint of the 1986 original.

\bibitem[Sou97]{Soule97}
C.~Soul{\'e}.
\newblock Hermitian vector bundles on arithmetic varieties.
\newblock In {\em Algebraic geometry---Santa Cruz 1995}, pages 383--419. Amer.
  Math. Soc., Providence, RI, 1997.

\bibitem[Spe50]{Specker50}
E.~Specker.
\newblock Additive {G}ruppen von {F}olgen ganzer {Z}ahlen.
\newblock {\em Portugaliae Math.}, 9:131--140, 1950.

\bibitem[{Sta}17]{stacks-project}
The {Stacks Project Authors}.
\newblock {\emph{Stacks Project}}.
\newblock {\tt{http://stacks.math.columbia.edu}}, 2017.

\bibitem[{Str}11]{Stroock2011}
D.~W. {Stroock}.
\newblock {\em {Probability theory. An analytic view. 2nd ed.}}
\newblock Cambridge: Cambridge University Press, 2nd ed. edition, 2011.

\bibitem[Szp85]{Szpiro85}
L.~Szpiro.
\newblock Degr\'es, intersections, hauteurs.
\newblock {\em Ast\'erisque}, (127):11--28, 1985.

\bibitem[Tat67]{Tate67}
J.~T. Tate.
\newblock Fourier analysis in number fields, and {H}ecke's zeta-functions.
\newblock In {\em Algebraic {N}umber {T}heory ({P}roc. {I}nstructional {C}onf.,
  {B}righton, 1965)}, pages 305--347. Thompson, Washington, D.C., 1967.
\newblock Unaltered reproduction of Tate's doctoral thesis (Princeton, May
  1950).

\bibitem[Tay11]{Taylor2011}
M.~E. Taylor.
\newblock {\em Partial differential equations {I}. {B}asic theory}, volume 115
  of {\em Applied Mathematical Sciences}.
\newblock Springer, New York, second edition, 2011.

\bibitem[Thi66]{Thimm66}
W.~Thimm.
\newblock Der {W}eierstra\ss sche {S}atz der algebraischen {A}bh\"angigkeit von
  {A}belschen {F}unktionen und seine {V}erallgemeinerungen.
\newblock In {\em Festschr. {G}ed\"achtnisfeier {K}. {W}eierstrass}, pages
  123--154. Westdeutscher Verlag, Cologne, 1966.

\bibitem[TLX14]{TianLiuXu2014}
C.~Tian, M.~Liu, and G.~Xu.
\newblock Measure inequalities and the transference theorem in the geometry of
  numbers.
\newblock {\em Proc. Amer. Math. Soc.}, 142(1):47--57, 2014.

\bibitem[Toe09]{Toeplitz1909}
O.~Toeplitz.
\newblock \"{U}ber die {A}ufl\"{o}sung unendlichvieler linearer {G}leichungen
  mit unendlichvielen {U}nbekannten.
\newblock {\em Palermo Rend.}, 28:88--96, 1909.

\bibitem[vdGS00]{vanderGeerSchoof2000}
G.~van~der Geer and R.~Schoof.
\newblock Effectivity of {A}rakelov divisors and the theta divisor of a number
  field.
\newblock {\em Selecta Math. (N.S.)}, 6(4):377--398, 2000.

\bibitem[{Wei}39]{Weil39}
A.~{Weil}.
\newblock {Sur l'analogie entre les corps de nombres alg\'ebriques et les corps
  de fonctions alg\'ebriques.}
\newblock {\em {Rev. Sci.}}, 77:104--106, 1939.

\bibitem[Wei46]{Weil46}
A.~Weil.
\newblock Sur quelques r\'esultats de {S}iegel.
\newblock {\em Summa Brasil. Math.}, 1:21--39, 1946.

\bibitem[Wei82]{Weil82}
A.~Weil.
\newblock {\em Adeles and algebraic groups}, volume~23 of {\em Progress in
  Mathematics}.
\newblock Birkh\"auser, Boston, Mass., 1982.
\newblock (I.A.S. Lectures 1959-1960). With appendices by M. Demazure and
  Takashi Ono.

\bibitem[Zag58]{Zaguskin58}
V.~L. Zaguskin.
\newblock Circumscribed and inscribed ellipsoids of extremal volume.
\newblock {\em Uspehi Mat. Nauk}, 13(6 (84)):89--93, 1958.

\bibitem[Zha95]{Zhang95}
S.~Zhang.
\newblock Positive line bundles on arithmetic varieties.
\newblock {\em J. Amer. Math. Soc.}, 8(1):187--221, 1995.

\end{thebibliography}
\end{document}